\documentclass[12pt]{amsbook}
\usepackage[all,cmtip]{xy}

\usepackage{a4wide}
\usepackage{pxfonts}
\usepackage{wasysym}
\usepackage{amssymb,amsxtra,mathrsfs}
\usepackage{eucal}
\usepackage{setspace,upgreek}
\usepackage{mathtools}
\usepackage{url}
\usepackage[hidelinks]{hyperref}

\newcommand{\mf}[1]{\mathfrak{#1}}
\newcommand{\mc}[1]{\mathcal{#1}}

\newcommand{\ms}[1]{\mathsf{#1}}

\newcommand{\pf}[1]{\pmb{\mathfrak{#1}}}
\newcommand{\pc}[1]{\pmb{\mathcal{#1}}}
\newcommand{\ps}[1]{\pmb{\mathsf{#1}}}

\newcommand{\s}[3]
{\ms{S}_{\ms{F}_{\ps{#1}_{#2}^{{#3}}}}^{G}}

\newcommand{\sm}[1]{\overset{\mf{#1}}{\simeq}}

\newcommand{\af}[1]{\overset{\mf{#1}}{\Bumpeq}}

\newcommand{\dm}[1]{\overset{\mf{#1}}{\eqsim}}

\newcommand{\tm}[1]{\overset{\mf{#1}}{\thickapprox}}

\newcommand{\ov}{\overline}
\newcommand{\lr}[2]{\langle\, #1,\,#2\,\rangle}
\newcommand{\fp}[5]{\lr{#1}{#2,#3,\{#4_{j},#5_{j}\}_{j\in\{m,w\}}}}
\newcommand{\up}{\upharpoonright}
\newcommand{\ep}{\upvarepsilon} 
\newcommand{\dv}{\divideontimes}
\newcommand{\vm}{\vardiamondsuit}

\newcommand{\vc}{\varclubsuit}

\newcommand{\ze}{\mathsf{0}} 
\newcommand{\un}{\mathsf{1}} 
\newcommand{\N}{\mathbb{N}}
\newcommand{\K}{\mathbb{K}}
\newcommand{\R}{\mathbb{R}}
\newcommand{\C}{\mathbb{C}}
\newcommand{\Z}{\mathbb{Z}}
\newtheorem{theorem}{Theorem}
\newtheorem{definition}[theorem]{Definition}
\newtheorem{convention}[theorem]{Convention}
\newtheorem{hypothesis}{Hypothesis}
\newtheorem{assumption}{Assumption}
\newtheorem{main}[theorem]{Main Theorem}
\newtheorem*{dyn}{Thermal nature}
\newtheorem*{equivar}{Equivariance}
\newtheorem*{thermal}{Phase transition via dynamical symmetry breaking}
\newtheorem*{integ}{Integrality}
\newtheorem*{stb}{Equivariant stability}
\theoremstyle{definition}
\newtheorem{lemma}[theorem]{Lemma}
\newtheorem{proposition}[theorem]{Proposition}
\newtheorem{corollary}[theorem]{Corollary}
\theoremstyle{remark}

\newtheorem{remark}[theorem]{Remark}

\numberwithin{assumption}{part}

\numberwithin{section}{part}
\numberwithin{theorem}{section}
\numberwithin{equation}{section}

\begin{document}

\title{Natural transformations associated with a locally compact group and universality of the global Terrell law}

\author{Benedetto Silvestri}

\email{bndsilvestri@proton.me}

\address{440 W Kelly Ave Apt 6 Jackson, WY 83001-8544 USA (via express couriers signature required)}

\urladdr{\url{https://arxiv.org/a/silvestri_b_1}}

\urladdr{https://orcid.org/0000-0002-8042-4932}

\begin{abstract}
Via the construction of a functor from $\mathsf{C}_{u}(H)$ to an auxiliary category 
we associate, with any triplet $(G,F,\uprho)$,
two natural transformations,
$\mathfrak{m}_{\star}$ 
between functors from the opposite of $\mathsf{C}_{u}(H)$ to $\mathsf{Fct}(H,\mathsf{set})$ 
and 
$\mathfrak{v}_{\natural}$ between functors from the 
opposite of $\mathsf{C}_{u}^{0}(H)$ to $\mathsf{Fct}(H,\mathsf{set})$. 
$G$ and $F$ are locally compact groups, $\uprho:F\to Aut_{\ms{Gr}}(G)$ is a group morphism
such that the map $(g,f)\mapsto\uprho_{f}(g)$ is continuous, 
$H$ is the external topological semidirect product of $G$ and $F$ relative to $\uprho$, 
a groupoid when seen as a category, 
$\mathsf{C}_{u}^{0}(H)$ is a subcategory of $\mathsf{C}_{u}(H)$ 
a subcategory of the 
category of $C^{\ast}-$dynamical systems with symmetry group $H$
and equivariant morphisms, 
and $\mathsf{Fct}(H,\mathsf{set})$ is the category of functors from $H$ to $\mathsf{set}$ and natural transformations.
For any object $\mathfrak{A}$ of $\mathsf{C}_{u}^{0}(H)$ 
to assemble $\mathfrak{m}_{\star}^{\mathfrak{A}}$  
we exploit the Chern-Connes characters generated 
by JLO cocycles $\Upphi$ on the unitization of certain 
$C^{\ast}-$crossed products relative to $\mathfrak{A}$, 
while to construct $\mathfrak{v}_{\natural}^{\mathfrak{A}}$ we exert
the states of the $C^{\ast}-$algebra underlying $\mathfrak{A}$ 
associated in a convenient manner with the $0-$dimensional components of the 
$\Upphi$'s.
We interpret $\mathsf{C}_{u}(H)$ as the category of fissioning systems and their transformations,
we use $\mathfrak{m}_{\star}^{\mathfrak{A}}$ and $\mathfrak{v}_{\natural}^{\mathfrak{A}}$ 
to define the nucleon phases and the fragment states resulting 
next the fission processes of the fissioning system $\mathfrak{A}$ occur
and we formulate in a $C^{\ast}-$algebraic 
framework a new nucleon phase hypothesis originally advanced by Mouze and Ythier.
We apply the naturality of $\mathfrak{m}_{\star}$ and $\mathfrak{v}_{\natural}$ to
establish the universality of the global nucleon masses and the global Terrell law, 
stated as invariance of the light and heavy nucleon core masses
and invariance of the prompt-neutron yield 
under contravariant action of $\mathsf{C}_{u}^{0}(H)$ and under action of $H$
over the field of fission processes.
Then under the nucleon phase hypothesis 
we exhibit the stability of the nucleon core masses at values $82$ and $126$ 
and the invariance of the Terrell law, as a particular case of the global ones.
\end{abstract}
\maketitle

\flushbottom
\tableofcontents
\section{Introduction}
\label{introI}


Mouze and Ythier advanced
the nucleon phase hypothesis 
in order to explain, among other occurrences, 
the presence of two fixed values $82$ and $126$ 
in the following Terrell law 
\begin{equation}
\label{06091115}
\ov{\nu}=0.08\,(A_{L}-82)+0.1\,(A_{H}-126),
\end{equation}
describing 
the mean value of the prompt-neutron yield of two asymmetric fragments 
resulting next the neutron fission of 
$U^{233}$, $U^{235}$, $Pu^{239}$ 
and the spontaneous fission of $Cf^{252}$, 
in function of the mass numbers $A_{L}$ and $A_{H}$ of the light 
and heavy fragments respectively,
\cite{mhy1,mhy2,ric}.
Basically they assumed that under the particular circumstances occurring in the fission process,
two nucleon cores come into existence, 
one of mass number $82$ and the other of mass number $126$.
Thus by mean of the closure of the shells at these values, 
their presence in the above equality can be justified.
Since the evenience of closed shells in the presence 
of the nucleon phase in which the cores happen to appear,
it has been claimed by Ricci 
that 
``the organization of matter in closed shells should be a universal law''
see \cite[IIa]{ric}. 
\par
We contend instead that the nucleon phase itself is universal 
and we grant the tenet that universal in physics 
has to be understood natural in category theory.
The present is the first of a series of three parts where, 
by introducing the concept of the 
\emph{category $\mf{G}(G,F,\uprho)$ of nucleon systems} 
and the structure of \emph{nucleon-fragment doublet},
we state and resolve the equivariant and invariant forms 
of the \emph{universality claim}, later described,
for the positions $\mf{K}=\ms{C}_{u}(H)$ and $\mf{H}=\ms{C}_{u}^{0}(H)$.
Here $\ms{C}_{u}^{0}(H)$ is a strict subcategory of 
the category $\ms{C}_{u}(H)$ 
of fissioning systems and their transformations,
a subcategory
of the category of $C^{\ast}-$dynamical systems with symmetry group $H$
and equivariant morphisms.
In short the first asserts that 
the global fragment state is originated via
the global nucleon phase and both are
equivariant under contravariant action of 
$\ms{C}_{u}^{0}(H)$ and under action of $H$.
The second establishes
that the global nucleon and fragment masses, and the global Terrell law 
are invariant
under contravariant action of $\ms{C}_{u}^{0}(H)$ and under action of $H$
over the field of fission processes.
As a consequence we obtain 
the stability of the masses of the nucleon cores 
and the invariance of the prompt-neutron yield 
under essential contravariant action of 
$\ms{C}_{u}^{0}(H)$ and under action of $H$.
Then provided a new nucleon phase hypothesis
the stability of the nucleon masses at the values $82$ and $126$ 
and the invariance of the Terrell law 
under action of the above transformations,
follow.
The meaning of the term 'essential' will be clarified in course of the introduction.
We anticipate that the introduction of the 
global nucleon and fragment masses, and the global Terrell law 
- generalizing and extending the nucleon and fragment masses, and the Terrell law respectively -
as extensions serve 
to switch 'essential' into 'true' contravariance.
\par
By balancing the language 
in order to allocate the mathematical and physical descriptions,
we organize this introduction as follows.
Firstly we describe features of the category $\mf{G}(G,F,\uprho)$ and their physical interpretation, 
then we exert them to expose the forms of the universality claim.
Next we outline the concept of nucleon-fragment doublet $\mc{T}$ 
on an arbitrary category
and exhibit the $\mc{T}-$resolution of the forms of the universality claim. 
Then the resolution of the claim is established 
as $\mc{T}_{\bullet}-$resolution, 
where $\mc{T}_{\bullet}$ is the canonical nucleon-fragment doublet 
on $\ms{C}_{u}(H)$ constructued in the main theorem of the work.
\par
A nucleon-fragment doublet provides to nucleon phases and fragment states 
the following list of properties, next called initial list,  
whose meaning will be later clarified:
\begin{enumerate}
\item
thermal nature of the nucleon phases,
\label{07081617} 
\item
noncommutative geometric and thermal origination of fragment states,
\label{07081618}
\item
phase transition via symmetry breaking,
\label{07081619} 
\item
contravariance under action of the category $\Upxi(\pf{D},\mc{L})$
and covariance under action of the symmetry group $H$. 
\label{07062056}
\end{enumerate}
Let us start with some notation.
\par
For any set $A$ let $\mathscr{P}(A)$ be the power set of $A$.
If $\mc{A}$ is a $C^{\ast}-$algebra let $\ms{E}_{\mc{A}}$ be the set of states of $\mc{A}$,
$\mc{A}_{ob}$ be the set of selfadjoint elements of $\mc{A}$
and $\mc{A}_{+}\coloneqq\{a^{\ast}a\,\vert\, a\in\mc{A}\}$.
If $A$ is a category and $x,y$ are objects of $A$, 
then let $Mor_{A}(x,y)$ denote the set of morphisms of $A$
whose domain or source is $x$ and whose codomain or target is $y$, 
let $Mor_{A}$ denote the set of morphisms of $A$.
Let $Aut_{A}(x)$ be the group of invertible morphisms of $A$ whose source and target is $x$.
Let $^{op}$ mean opposite category, and for any morphism $T$ of a category  
let $d(T)$ and $c(T)$ denote the domain and codomain of $T$ respectively.
Let $\ms{set}$ be the category of sets and maps in a fixed universe,
moreover all the following categories are implicitly assumed to have the object set a 
subset of the fixed universe.
Let $\ms{gr}:Mor_{\ms{set}}\to Mor_{\ms{set}}$ be such that $d(\ms{gr}(f))=d(f)$ 
and so defined $f\mapsto(x \mapsto(x,f(x)))$;
by abuse of language let $\prod_{x\in X}C$ be $\prod_{x\in X}C_{x}$ where $C_{x}=C$ for all $x\in X$,
with $X$ and $C$ sets.
Let $\ms{Gr}$ be the category of groups and group morphisms with map composition, 
and $\ms{Ab}$ the subcategory of abelian groups.
If $G$ and $F$ are locally compact groups, and
$\uprho:F\to Aut_{\ms{Gr}}(G)$ is a group homomorphism
such that the map $(g,f)\mapsto\uprho_{f}(g)$
on $G\times F$ at values in $G$, is continuous,
then let $H=G\rtimes_{\uprho}F$ be the external topological semidirect product of $G$ and $F$ relative to $\uprho$.
Here $Aut_{\ms{Gr}}(G)$ is the group of group automorphisms of the group underlying $G$.
Let $j_{1}$ and $j_{2}$ be the canonical injections of $G$ and $F$ into $H$ respectively.
For any locally compact group $K$
let $\ms{C}(K)$ be the category of unital $C^{\ast}-$dynamical systems with symmetry group $K$
and surjective equivariant morphisms. 
For all categories $X,Y$ let $\ms{Fct}(X,Y)$ be the category of functors from $X$ to $Y$
and natural transformations as morphisms.
For any functor $F$ let $F_{o}$ and $F_{m}$ denote 
the object map and morphism map respectively.
In particular we consider the group $H$ as the category whose unique object is its unit $\un$,
so $\ms{Fct}(H,\ms{set})$ can be considered as the category of set representations of $H$
and intertwinings as morphisms.
Let a field contravariant (respectively covariant) under action of $X$ 
be a functor from $X^{op}$ (respectively $X$) to $\ms{set}$,
moreover we add via $L$ if $L$ is its morphism map.
If $Z$ is a field contravariant (respectively covariant) under action of $X$ 
then let a subfield $J$ of $Z$ be any field
contravariant (respectively covariant) under action of $X$ 
such that $J_{o}(a)\subseteq Z_{o}(a)$ and $J_{m}(t)=Z_{m}(t)\up J_{o}(c(t))$
(respectively $J_{m}(t)=Z_{m}(t)\up J_{o}(d(t))$), for any object $a$ and morphism $t$
of $X$.
Notice that for any field say $N$ covariant under action of the group $H$ via $R$,
$R$ is a representation of $H$ on the set $N(\un)$.
If $M$ is a field contravariant under action of $X$, 
we call $x$ a section of $M$ contravariant under action of $X$ via $f$ and $g$,
if $x$ is a natural transformation between functors from $X^{op}$ and $\ms{set}$
such that 
$f$ is the morphism map of the target functor of $x$,
$g$ is the morphism map of the source functor of $x$,
and $M$ is the target functor of $x$.
Similarly for covariant sections by replacing $X^{op}$ by $X$.
We call $x$ a section contravariant (covariant) under action of $X$ via $f$ and $g$,
a section $x$ of $M$ contravariant (covariant) under action of $X$ via $f$ and $g$,
for some field $M$ contravariant (covariant) under action of $X$.
We call $x$ a section invariant under contravariant (covariant) action of $X$ via $g$,
if $x$ is a section of $M$ contravariant (covariant) under action of $X$ via $f$ and $g$, 
the object map of $M$ is a constant with constant value equal some set $Z$
and $f$ is the constant map with constant value equal the identity map on $Z$.
The term equivariance will be used to mean covariance or contravariance.
Let $\mf{K}$ and $\mf{D}$ be categories, 
$\mc{F}\in\ms{Fct}(\mf{K},\mf{D})$ and $C$ be an object of $\mf{D}$.
Let $(\cdot)_{\dagger}^{\mc{F},C}$ be the map on $Mor_{\mf{K}}$ such that 
for any object $X,Y$ of $\mf{K}$ and $T\in Mor_{\mf{K}}(X,Y)$ 
we have
\begin{equation*}
\begin{aligned}
T_{\dagger}^{\mc{F},C}:Mor_{\mf{D}}(\mc{F}_{o}(Y),C)&\to Mor_{\mf{D}}(\mc{F}_{o}(X),C),
\\
g&\mapsto g\circ\mc{F}_{m}(T).
\end{aligned}
\end{equation*}
We call $(\cdot)_{\dagger}^{\mc{F},C}$ \emph{the conjugate of $\mc{F}_{m}$},
and call dual of the restriction of an action on the automorphism group of an object of a category 
the composition of conjugation with inversion. 
Notice that we do not refer to the object $C$
since we use this notation only in the following cases
from which it will be clear by the context which object $C$ is concerned.
(1) $\mf{D}=\ms{Ab}$ and $C$ the field $\R$ of real numbers;
(2) $\mf{K}=\ms{CA}^{\ast}$ the category of 
$C^{\ast}-$algebras and $\ast-$homomorphisms,
$\mf{D}=\ms{BS}$ the category
of complex Banach spaces and linear bounded maps,
$\mc{F}$ the forgetful functor from $\ms{CA}^{\ast}$ to $\ms{BS}$
and $C$ the field $\C$ of complex numbers.
Moreover in this case by abuse of language we shall speak of 
the conjugate of a $\ast-$homomorphism, by meaning its image 
via the conjugate of $\mc{F}_{m}$. 
\clearpage
The first step is the introduction of the category $\mf{G}(G,F,\uprho)$ of nucleon systems
whose data determine what follows.
\begin{equation}
\label{07101147}
\begin{aligned}
(\ms{A},\Pr_{1})
&\in\ms{Fct}(\mf{G}(G,F,\uprho),\ms{Ab}),
\\
(\pf{T},\Pr_{2})
&\in\ms{Fct}(\mf{G}(G,F,\uprho)^{op},\ms{set}),
\\
(\uppsi^{\mc{M}},\mf{b}^{\mc{M}}\circ\ms{inv})
&\in 
Mor_{\ms{Gr}}(H,Aut_{\mf{G}(G,F,\uprho)}(\mc{M})),\,
\forall\mc{M}\in\mf{G}(G,F,\uprho),
\\
\ms{I}_{\mc{M}}&:
\pf{T}_{\mc{M}}\to\ms{set},\,
\forall\mc{M}\in\mf{G}(G,F,\uprho),
\\
\upbeta_{c}
&\in
\prod_{\mc{M}\in\mf{G}(G,F,\uprho)}
\prod_{\mf{Q}\in\pf{T}_{\mc{M}}}
\ms{I}_{\mc{M}}^{\mf{Q}},
\\
\mf{a}&:
\prod_{\mc{M}\in\mf{G}(G,F,\uprho)}
\prod_{\mf{Q}\in\pf{T}_{\mc{M}}}
\prod_{\beta\in\ms{I}_{\mc{M}}^{\mf{Q}}}
\ms{C}(H),
\\
\mf{e}&:
\prod_{\mc{M}\in\mf{G}(G,F,\uprho)}
\prod_{\mf{Q}\in\pf{T}_{\mc{M}}}
\prod_{\beta\in\ms{I}_{\mc{M}}^{\mf{Q}}}
\ms{C}(\R),
\\
\ps{\upvarphi}
&\in
\prod_{\mc{M}\in\mf{G}(G,F,\uprho)}
\prod_{\mf{Q}\in\pf{T}_{\mc{M}}}
\prod_{\beta\in\ms{I}_{\mc{M}}^{\mf{Q}}}
\ms{E}_{\mc{A}(\mc{M})_{\beta}^{\mf{Q}}},
\\
\ms{V}
&\in
\prod_{\mc{M}\in\mf{G}(G,F,\uprho)}
\prod_{\mf{Q}\in\pf{T}_{\mc{M}}}
\prod_{\beta\in\ms{P}_{\mc{M}}^{\mf{Q}}}
\prod_{l\in H}
Mor_{\ms{CA}^{\ast}}(\mc{A}(\mc{M})_{\beta}^{\mf{Q}},
\mc{A}(\mc{M})_{\beta}^{\mf{b}^{\mc{M}}(l)\mf{Q}});
\end{aligned}
\end{equation}
where
$(\mf{a}_{\mc{M}})_{\alpha}^{\mf{T}}=
\lr{\mc{A}(\mc{M})_{\alpha}^{\mf{T}},H}{(\upeta^{\mc{M}})_{\alpha}^{\mf{T}}}$, 
$(\mf{e}_{\mc{M}})_{\alpha}^{\mf{T}}=
\lr{\mc{A}(\mc{M})_{\alpha}^{\mf{T}},\R}{(\ep^{\mc{M}})_{\alpha}^{\mf{T}}}$,
for any object $\mc{M}$ of $\mf{G}(G,F,\uprho)$,
$\mf{T}\in\pf{T}_{\mc{M}}$ and $\alpha\in\ms{I}_{\mc{M}}^{\mf{T}}$.
\par
Let $\mc{N}$ be an object of $\mf{G}(G,F,\uprho)$,
$\mf{T}\in\pf{T}_{\mc{N}}$ and $\alpha\in\ms{I}_{\mc{N}}^{\mf{T}}$.
Thus the following properties hold
\begin{enumerate}
\item
$(\ps{\upvarphi}^{\mc{N}})_{\alpha}^{\mf{T}}$ is invariant under action of $G$, i.e.
\begin{equation*}
(\ps{\upvarphi}^{\mc{N}})_{\alpha}^{\mf{T}}\circ(\upeta^{\mc{N}})_{\alpha}^{\mf{T}}(j_{1}(g))
=
(\ps{\upvarphi}^{\mc{N}})_{\alpha}^{\mf{T}},\forall g\in G,
\end{equation*}
\item
$(\ps{\upvarphi}^{\mc{N}})_{\alpha}^{\mf{T}}$ is an $(\ep^{\mc{N}})_{\alpha}^{\mf{T}}-$KMS state.
\end{enumerate}
Note that if $\alpha\in\R^{+}$ then 
$(\ps{\upvarphi}^{\mc{N}})_{\alpha}^{\mf{T}}$
is a $\alpha-$KMS state, i.e. a state of thermal equilibrium 
at the inverse temperature $\alpha$, 
with respect to the dynamics
$(\ep^{\mc{N}})_{\alpha}^{\mf{T}}(-\alpha^{-1}(\cdot))$. 
Next set
\begin{equation*}
\begin{aligned}
\ms{F}(\mc{N})_{\alpha}^{\mf{T}}
&\coloneqq
\{h\in F
\,\vert\,
(\ps{\upvarphi}^{\mc{N}})_{\alpha}^{\mf{T}}
\circ
(\upeta^{\mc{N}})_{\alpha}^{\mf{T}}
(j_{2}(h))
=
(\ps{\upvarphi}^{\mc{N}})_{\alpha}^{\mf{T}}
\},
\\
\ms{S}_{\alpha}^{\mf{T}}(\mc{N})
&\coloneqq
G\rtimes_{\uprho}
\ms{F}(\mc{N})_{\alpha}^{\mf{T}},
\\
\mc{B}_{\alpha}^{\mf{T}}(\mc{N})
&\coloneqq
\mc{A}(\mc{N})_{\alpha}^{\mf{T}}
\rtimes_{(\upeta^{\mc{N}})_{\alpha}^{\mf{T}}}
\ms{S}_{\alpha}^{\mf{T}}(\mc{N}),
\\
\ms{P}
&\in
\prod_{\mc{M}\in\mf{G}(G,F,\uprho)}
\prod_{\mf{Q}\in\pf{T}_{\mc{M}}}
\mathscr{P}(\ms{I}_{\mc{M}}^{\mf{Q}}),
\\
\ms{P}_{\mc{N}}^{\mf{T}}
&\coloneqq
\{\beta\in\ms{I}_{\mc{N}}^{\mf{T}}
\,\vert\,
\ms{F}(\mc{N})_{\beta}^{\mf{T}}
\supseteq
\ms{F}(\mc{N})_{(\upbeta_{c}^{\mc{N}})^{\mf{T}}}^{\mf{T}}\},
\\
\ms{A}_{\mc{N}}^{\ast}
&\coloneqq Mor_{\ms{Ab}}(\ms{A}_{\mc{N}},\R).
\end{aligned}
\end{equation*}
A fundamental set valued map is $Rep^{\mc{N}}:\ms{A}_{\mc{N}}^{\ast}\to\ms{set}$.
Let $\mf{c}\in\ms{A}_{\mc{N}}^{\ast}$ then
\footnote{Actually 
$\mc{B}_{\mf{r}}$
depends also by a Haar measure on
$\ms{S}_{\alpha_{\mf{r}}}^{\mf{T}_{\mf{r}}}(\mc{N})$
but here we do not need such a generality.}
\begin{equation}
\label{08121751}
\begin{aligned}
Rep^{\mc{N}}(\mf{c})
&\ni\mf{r}\mapsto
\mf{T}_{\mf{r}}
\in
\pf{T}_{\mc{N}},
\\
Rep^{\mc{N}}(\mf{c})
&\ni\mf{r}\mapsto
\alpha_{\mf{r}}
\in
\ms{P}_{\mc{N}}^{\mf{T}_{\mf{r}}},
\\
Rep^{\mc{N}}(\mf{c})
&\ni\mf{r}\mapsto
\mc{A}_{\mf{r}}
\coloneqq
\mc{A}(\mc{N})_{\alpha_{\mf{r}}}^{\mf{T}_{\mf{r}}},
\\
Rep^{\mc{N}}(\mf{c})
&\ni\mf{r}\mapsto
\ps{\upvarphi}_{\mf{r}}
\coloneqq
(\ps{\upvarphi}^{\mc{N}})_{\alpha_{\mf{r}}}^{\mf{T}_{\mf{r}}},
\\
Rep^{\mc{N}}(\mf{c})
&\ni\mf{r}\mapsto
\mc{B}_{\mf{r}}
\coloneqq
\mc{B}_{\alpha_{\mf{r}}}^{\mf{T}_{\mf{r}}}(\mc{N}),
\\
Rep^{\mc{N}}(\mf{c})
&\ni\mf{r}\mapsto
\ep_{\mf{r}}
\coloneqq
(\ep^{\mc{N}})_{\alpha_{\mf{r}}}^{\mf{T}_{\mf{r}}},
\\
Rep^{\mc{N}}(\mf{c})
&\ni\mf{r}
\mapsto
\ms{V}_{\mf{r}}
\coloneqq
\ms{V}(\mc{N})_{\alpha_{\mf{r}}}^{\mf{T}_{\mf{r}}},
\\
Rep^{\mc{N}}(\mf{c})
&\ni\mf{r}
\mapsto
\mf{u}_{\mf{r}}
\in Mor_{\ms{Ab}}(\ms{A}_{\mc{N}},\ms{K}_{0}(\mc{B}_{\mf{r}}^{+}));
\end{aligned}
\end{equation}
and the maps
\begin{equation}
\begin{aligned}
\label{08121752}
Rep^{\mc{N}}(\mf{c})
&\ni\mf{r}
\mapsto
\Upphi^{\mf{r}},
\\
Rep^{\mc{N}}(\mf{c})
&\ni\mf{r}\mapsto
\uprho_{\mf{r}}\in\ms{E}_{\mc{A}_{\mf{r}}};
\end{aligned}
\end{equation}
such that 
\emph
{$\Upphi^{\mf{r}}$ is an entire normalized even cocycle on the 
unitization $\mc{B}_{\mf{r}}^{+}$ of $\mc{B}_{\mf{r}}$,
$\mf{c}$ factorizes through the real part, via the standard
duality, of the character generated by $\Upphi^{\mf{r}}$ and the map $\mf{u}_{\mf{r}}$; 
while $\uprho_{\mf{r}}$ is a $\ps{\upvarphi}_{\mf{r}}-$normal state
associated in an appropriate way to the $0-$dimensional component
of $\Upphi^{\mf{r}}$},
where the standard duality is between the 
entire cyclic cohomology and the $K_{0}-$theory 
$\ms{K}_{0}(\mc{B}_{\mf{r}}^{+})$ of $\mc{B}_{\mf{r}}^{+}$.
Let us define
\begin{equation*}
\mf{N}^{\mc{N}}(\mf{c})
\coloneqq
\{
\uprho_{\mf{r}}
\,\vert\,
\mf{r}\in Rep^{\mc{N}}(\mf{c})
\}.
\end{equation*}
\par
Next we interpret the previous data in a physical context as follows.
\begin{enumerate}
\item
$\mc{N}$ is a physical system called nucleon system whose set of observables is $\ms{A}_{\mc{N}}$
and whose set of states called nucleon phases is $\ms{A}_{\mc{N}}^{\ast}$;
\item
$\mc{L}(a)$ is the nucleon system generated by the fissioning system $a$,
for any functor $\mc{L}$ from a category $\mf{H}$ to $\mf{G}(G,F,\uprho)$
and $a\in\mf{H}$;
\item
for any $\mf{T}\in\pf{T}_{\mc{N}}$ and $\alpha\in\ms{P}_{\mc{N}}^{\mf{T}}$
there exists a physical system 
$\mc{O}(\mc{N})_{\alpha}^{\mf{T}}$ called fragment system such that 
\begin{enumerate}
\item
$\mc{A}(\mc{N})_{\alpha}^{\mf{T}}$ is its observable algebra 
and $\ms{E}_{\mc{A}(\mc{N})_{\alpha}^{\mf{T}}}$ is its state space. 
If $\alpha>0$ then $\mc{O}(\mc{N})_{\alpha}^{\mf{T}}$
evolves in time through the dynamics
$(\ep^{\mc{N}})_{\alpha}^{\mf{T}}(-\alpha^{-1}(\cdot))$,
\item
$\mf{T}$ is an operation performable over the states of $\mc{O}(\mc{N})_{\alpha}^{\mf{T}}$;
\end{enumerate}
\item
for all $\mf{c}\in\ms{A}_{\mc{N}}^{\ast}$ 
and $\mf{r}\in Rep^{\mc{N}}(\mf{c})$ 
the following properties hold 
\begin{enumerate}
\item
if the operation $\mf{T}_{\mf{r}}$ is performed on 
$\mc{O}(\mc{N})_{\alpha_{\mf{r}}}^{\mf{T}_{\mf{r}}}$
when occurring in the state $\ps{\upvarphi}_{\mf{r}}$,
then $\mc{O}(\mc{N})_{\alpha_{\mf{r}}}^{\mf{T}_{\mf{r}}}$
will occur in the state $\uprho_{\mf{r}}$,
\label{07071636}
\item
if $\mc{O}(\mc{N})_{\alpha_{\mf{r}}}^{\mf{T}_{\mf{r}}}$
occurs in the state $\uprho_{\mf{r}}$, 
then $\mc{N}$ occurred in the phase $\mf{c}$;
\label{07071637}
\end{enumerate}
\item
the information about the operations in $\pf{T}_{\mc{N}}$
are sufficient to organize $\ms{A}_{\mc{N}}$ as a group but not as an algebra.
\end{enumerate}
This interpretation is consistent with the fact that $\upomega-$normal states 
are usually considered perturbations of the state $\upomega$.
Notice that
\emph
{$\uprho_{r}$ is a state of the fragment system $\mc{O}(\mc{N})_{\alpha_{\mf{r}}}^{\mf{T}_{\mf{r}}}$,
moreover if $\alpha>0$ then $\ps{\upvarphi}_{\alpha}$ is a state of thermal equilibrium at inverse
temperature $\alpha$ with respect to the dynamical system underlying $\mc{O}(\mc{N})_{\alpha}^{\mf{T}}$}.
\par
We refer to \eqref{07071637} by saying that
\emph
{the fragment state $\uprho_{\mf{r}}$ is \textbf{originated} via the nucleon phase $\mf{c}$},
in particular $\mf{N}^{\mc{N}}(\mf{c})$ is the set of all the fragment states originated via the 
nucleon phase $\mf{c}$. 
\par
Referring only to $\mf{c}$ the sequence (\ref{07071636},\ref{07071637}) is summarized by saying that 
\emph
{$\mf{c}$ is the phase of the nucleon system $\mc{N}$ occurring by performing the operation $\mf{T}_{\mf{r}}$
on the state of thermal equilibrium $\ps{\upvarphi}_{\mf{r}}$}.
Now the meaning of \eqref{07081617} and \eqref{07081618} 
in the initial list appear clear.
Note that in general $\mf{N}^{\mc{N}}(\mf{c})$ contains more than one element 
and its elements may belong to state spaces of different $C^{\ast}-$algebras.
\par
If $\mc{L}$ is a functor from a category $\mf{C}$ to $\mf{G}(G,F,\uprho)$,
then for any map $M$ whose domain is a subset of the object set of $\mf{G}(G,F,\uprho)$,
we let $M_{x}$ denote $M_{\mc{L}(x)}$ for any object $x$ of $\mf{C}$ such that $\mc{L}(x)\in Dom(M)$.
Now we can expose the two forms of the universality claim, 
then we exploit the rest of the introduction to sketch their solutions.
We shall describe the equivariant form of the claim firstly in mathematical (M) and then in physical (P) terms.
\par
Equivariant form of the universality claim (M).
\emph
{There exist a category $\mf{K}$, a subcategory $\mf{H}$ of $\mf{K}$,
a functor $\mc{L}$ from $\mf{K}$ to $\mf{G}(G,F,\uprho)$,
$\mf{m}$ and $\mc{W}$ such that 
\begin{equation*}
\begin{aligned}
\mf{m}
&\in Mor_{\ms{Fct}(\mf{K}^{op},\ms{set})},
\\
\ms{gr}\circ\mc{W}
&\in Mor_{\ms{Fct}(\mf{H}^{op},\ms{set})}.
\end{aligned}
\end{equation*}
Let $\pf{U}$ and $\pf{D}$
be the object maps of the source functors of 
$\mf{m}$ and $\ms{gr}\circ\mc{W}$
respectively
and let $a\in\mf{K}$ and $b\in\mf{H}$.
Thus
\begin{equation*}
\begin{aligned}
(\un\mapsto\mf{m}^{a})
&\in Mor_{\ms{Fct}(H,\ms{set})},
\\
(\un\mapsto\ms{gr}(\mc{W}^{b}))
&\in Mor_{\ms{Fct}(H,\ms{set})};
\end{aligned}
\end{equation*}
$\pf{U}_{a}\subseteq\pf{T}_{a}$
and
$\pf{D}_{b}\subseteq\pf{U}_{b}$.
Moreover 
for all 
$\mf{Q}\in\pf{U}_{a}$, $\beta\in\ms{P}_{a}^{\mf{Q}}$
\begin{equation*}
\mf{m}^{a}(\mf{Q},\beta)\in\ms{A}_{a}^{\ast},
\end{equation*}
and all 
$\mf{T}\in\pf{D}_{b}$, $\alpha\in\ms{P}_{b}^{\mf{T}}$
\begin{equation}
\label{07121138ante}
\begin{aligned}
\exists\,\mf{r}&\in Rep^{b}(\mf{m}^{b}(\mf{T},\alpha))
\\
\mc{W}^{b}(\mf{T},\alpha)&=\uprho_{\mf{r}},\,
\mf{T}_{\mf{r}}=\mf{T},
\alpha_{\mf{r}}=\alpha.
\end{aligned}
\end{equation}}
We say that
$\mc{L}$, $\mf{m}$ and $\mc{W}$ 
satisfy
the equivariant form of the universality claim (M)
w.r.t. $\mf{K}$ and $\mf{H}$.
Let $\mf{m}$ and $\mc{W}$ be called the global nucleon phase and global fragment state 
w.r.t. $\mc{L}$ respectively.
\par
Equivariant form of the universality claim (P).
\emph
{There exist a category $\mf{K}$ of fissioning systems and their transformations, 
a subcategory $\mf{H}$ of $\mf{K}$,
a functor $\mc{L}$ from $\mf{K}$ to $\mf{G}(G,F,\uprho)$,
a section $\mf{m}$ contravariant under action of $\mf{K}$,
and a map $\mc{W}$ such that $\ms{gr}\circ\mc{W}$ is a section 
contravariant under action of $\mf{H}$.
Let $\pf{U}$ and $\pf{D}$ be the object maps of the source functors of 
$\mf{m}$ and $\ms{gr}\circ\mc{W}$ respectively
and let $a$ in $\mf{K}$ and $b$ in $\mf{H}$.
Thus
$\un\mapsto\mf{m}^{a}$ and $\un\mapsto\ms{gr}(\mc{W}^{b})$ are covariant under action of $H$,
$\pf{U}_{a}\subseteq\pf{T}_{a}$
and
$\pf{D}_{b}\subseteq\pf{U}_{b}$.
Moreover 
for all $\mf{Q}\in\pf{U}_{a}$, $\beta\in\ms{P}_{a}^{\mf{Q}}$,
and for all $\mf{T}\in\pf{D}_{b}$, $\alpha\in\ms{P}_{b}^{\mf{T}}$
\begin{enumerate}
\item
$\mf{m}^{a}(\mf{Q},\beta)$ 
is a phase of the nucleon system $\mc{L}(a)$;
\item
$\mc{W}^{b}(\mf{T},\alpha)$ is a fragment state originated via the nucleon phase $\mf{m}^{b}(\mf{T},\alpha)$
through the relation \eqref{07121138ante};
\item
$\mc{W}^{b}(\mf{T},\alpha)$ is a state of the fragment system 
$\mc{O}(b)_{\alpha}^{\mf{T}}$
whose operator algebra is $\mc{A}(b)_{\alpha}^{\mf{T}}$
and if $\alpha>0$ then the system
evolves in time through $(\ep^{b})_{\alpha}^{\mf{T}}(-\alpha^{-1}(\cdot))$.
$\mf{m}^{b}(\mf{T},\alpha)$ is the phase of the nucleon system $\mc{L}(b)$, 
occurring by performing the operation $\mf{T}$ 
on the state of thermal equilibrium $(\ps{\upvarphi}^{b})_{\alpha}^{\mf{T}}$ 
at inverse temperature $\alpha$
with respect to the dynamical system underlying $\mc{O}(b)_{\alpha}^{\mf{T}}$.
\end{enumerate}}
Invariant form of the universality claim.
\emph
{There exist a category $\mf{K}$, a subcategory $\mf{H}$ of $\mf{K}$,
$\mc{L}$, $\mf{m}$ and $\mc{W}$ 
satisfying
the equivariant form of the universality claim (M)
w.r.t. $\mf{K}$ and $\mf{H}$;
moreover there exist 
$\ms{N}_{as}$, $\mc{P}$, $\ms{R}$, 
$\{\muup_{j},\uplambda_{j}\}_{j\in\{m,w\}}$ and $\uptheta$
with the following properties.
$\ms{N}_{as}$ is a set valued map defined on the object set of $\mf{H}$,
$\mc{P},\ms{R}\in\ms{Fct}(\mf{H}^{op},\ms{set})$,
$\mc{P}(b)$ equals the power set of 
$\ms{N}_{as}(b)$ the set of fission processes whose underlying fissioning system is $b$,
for all $b\in\mf{H}$.
$\ms{R}$ is the unique object of $\ms{Fct}(\mf{H}^{op},\ms{set})$
whose object map 
is the constant map with constant value equal to the power set of $\R$,
and whose morphism map is 
the constant map with constant value equal to the identity map on the power set of $\R$.
For all $j\in\{m,w\}$
\begin{equation*}
\muup_{j},\,
\uplambda_{j},\,
\uptheta
\in
Mor_{\ms{Fct}(\mf{H}^{op},\ms{set})}(\mc{P},\ms{R}).
\end{equation*}
Let $b\in\mf{H}$.
Thus
\begin{equation*}
(\un\mapsto(\muup_{j})_{b}),\,
(\un\mapsto(\uplambda_{j})_{b}),\,
(\un\mapsto\uptheta_{b})
\in
Mor_{\ms{Fct}(H,\ms{set})},
\end{equation*}
and there exist maps 
\begin{enumerate}
\item 
$\ms{N}_{as}(b)\ni y\mapsto\mf{T}_{y}\in\pf{D}_{b}$,
\item
$\ms{N}_{as}(b)\ni y\mapsto\alpha_{y}\in\ms{P}_{b}^{\mf{T}_{y}}$,
\item
$\ms{N}_{as}(b)\ni y\mapsto\ms{f}^{y}\in\prod_{j\in\{m,w\}}\ms{A}_{b}$,
\item
$\ms{N}_{as}(b)\ni y\mapsto N^{y}\in\prod_{j\in\{m,w\}}(\mc{A}(b)_{\alpha_{y}}^{\mf{T}_{y}})_{+}$,
\end{enumerate}
such that for any $Y\in\mc{P}(b)$ and $j\in\{m,w\}$ we have
\begin{equation*}
\begin{aligned}
(\muup_{j})_{b}(Y) 
&=\{
\mf{m}^{b}(\mf{T}_{y},\alpha_{y})(\ms{f}_{j}^{y})\, 
\vert\, y\in Y\},
\\
(\uplambda_{j})_{b}(Y) 
&=\{
\mc{W}^{b}(\mf{T}_{y},\alpha_{y})(N_{j}^{y})\,
\vert\, y\in Y\},
\\
\uptheta_{b}(Y) 
&=
\bigl\{
0.08
(\mc{W}^{b}(\mf{T}_{y},\alpha_{y})(N_{m}^{y})
-
\mf{m}^{b}(\mf{T}_{y},\alpha_{y})(\ms{f}_{m}^{y}))+
\\
&0.1
(\mc{W}^{b}(\mf{T}_{y},\alpha_{y})(N_{w}^{y})
-
\mf{m}^{b}(\mf{T}_{y},\alpha_{y})(\ms{f}_{w}^{y}))\,
\vert\, y\in Y\bigr\}.
\end{aligned}
\end{equation*}}
Here we let 
$\pf{D}$
be the object map of the source functor of 
$\ms{gr}\circ\mc{W}$.
Let 
$\muup_{m}$ and $\muup_{w}$
be called global light and global heavy nucleon masses
w.r.t. $\mc{L}$ respectively,
let 
$\uplambda_{m}$ and $\uplambda_{w}$
be called
global light and global heavy fragment masses
w.r.t. $\mc{L}$ respectively,
and 
$\uptheta$
be called global Terrell law w.r.t. $\mc{L}$. 
\par
Next we introduce the crucial structure of the entire work
and explain how it resolves the diverse forms of the universality claim,
but let us start with some convention.
$\mf{G}(G,F,\uprho)$ acts covariantly on the field $\ms{A}$ of nucleon observables 
via $\Pr_{1}$
and contravariantly on the field $\ms{A}^{\ast}$ of nucleon phases 
via the conjugate of $\Pr_{1}$;
while $\mf{G}(G,F,\uprho)$ acts contravariantly on the field $\pf{T}$ of operations 
via $\Pr_{2}$.
Therefore for any object $\mc{N}$ of $\mf{G}(G,F,\uprho)$, 
the group $H$ acts on $\ms{A}_{\mc{N}}$ via $\uppsi^{\mc{N}}$ 
and on $\ms{A}_{\mc{N}}^{\ast}$ via duality, 
while $H$ acts on $\pf{T}_{\mc{N}}$ via $\mf{b}^{\mc{N}}$.
Let $l\in H$, $\ms{f}\in\ms{A}_{\mc{N}}$, $\mf{c}\in\ms{A}_{\mc{N}}^{\ast}$
$\mf{T}\in\pf{T}_{\mc{N}}$, $\alpha\in\ms{P}_{\mc{N}}^{\mf{T}}$,
$a\in\mc{A}(\mc{N})_{\alpha}^{\mf{T}}$ 
and $\upomega$ continuous functional on $\mc{A}(\mc{N})_{\alpha}^{\mf{T}}$.
We let 
$\ms{f}^{l}=\uppsi^{\mc{N}}(l)(\ms{f})$, 
$\mf{c}^{l}=\mf{c}\circ\uppsi^{\mc{N}}(l^{-1})$,
$\mf{T}^{l}=\mf{b}^{\mc{N}}(l)(\mf{T})$, 
$a^{l}=\ms{V}(\mc{N})_{\alpha}^{\mf{T}}(l)(a)$, 
and 
$\upomega^{l}=\upomega\circ\ms{V}(\mc{N})_{\alpha}^{\mf{T}^{l}}(l^{-1})$.
Notice that the last two
are an element and a continuous functional on 
$\mc{A}(\mc{N})_{\alpha}^{\mf{T}^{l}}$
respectively.
For any functor $\mc{L}$ from a category $\mc{K}$ to $\mf{G}(G,F,\uprho)$,
let $\mc{L}_{i}=\Pr_{i}\circ\mc{L}_{m}$ for any $i\in\{1,2\}$.
Hence $(\ms{A}\circ\mc{L}_{o},\mc{L}_{1})$ is a functor from $\mc{K}$ to $\ms{Ab}$,
while $(\pf{T}\circ\mc{L}_{o},\mc{L}_{2})$ is a contravariant field 
under action of $\mc{K}$.
For any $T\in Mor_{\mc{K}}$, 
$\ms{h}\in\ms{A}_{d(T)}$, 
$\mf{h}\in\ms{A}_{c(T)}^{\ast}$ 
and $\mf{Q}\in\pf{T}_{c(T)}$
we let $\ms{h}^{T}=\mc{L}_{1}(T)(\ms{h})$, $\mf{h}^{T}=\mf{h}\circ\mc{L}_{1}(T)$, 
and $\mf{Q}^{T}=\mc{L}_{2}(T)(\mf{Q})$. 
\par
A \textbf{nucleon-fragment doublet $\mc{T}$} on a category $\mf{C}$
is a tuple
$\lr{S,J,\mc{Z},\mc{S}}{\mc{L},\mf{m},\mc{W},R,\pf{D},\pf{U}}$ 
satisfying the following properties
\footnote{
the complete and detailed definition is given in 
Def. \ref{06161650},
here for simplicity we assume that $R$ is the constant map with constant value equal to $H$.}.
\begin{enumerate}
\item
$\pf{U}$ and $\pf{D}$ are maps defined on subsets 
$Dom(\pf{U})$
and 
$Dom(\pf{D})$
of the object set of $\mf{G}(G,F,\uprho)$
respectively,
such that 
$Dom(\pf{D})\subseteq Dom(\pf{U})$,
$\pf{U}_{\mc{N}}\subseteq\pf{T}_{\mc{N}}$
and
$\pf{D}_{\mc{M}}\subseteq\pf{U}_{\mc{M}}$
for all $\mc{N}\in Dom(\pf{U})$
and $\mc{M}\in Dom(\pf{D})$;
\item
$(\pf{U},\Pr_{2})$ and $(\pf{D},\Pr_{2})$ 
are fields contravariant 
under action of $Dom(\pf{U})$ and $Dom(\pf{D})^{0}$ 
respectively;
\item
the third item in \eqref{07101147} 
holds true by replacing in all the occurrances 
$\mf{G}(G,F,\uprho)$ by $Dom(\pf{U})$
and by $Dom(\pf{D})^{0}$;
\item
$\mc{L}$ is a functor from the category $\mf{C}$ to $\mf{G}(G,F,\uprho)$; 
\label{07080602st1}
\item
$S$ is a field of nucleon phases contravariant under action of $\Uptheta(\pf{U},\mc{L})$
via the conjugate of $\mc{L}_{1}$;
\label{07080602st2}
\item
each fiber of $S$ is covariant under action of $H$ via the dual of $\uppsi\circ\mc{L}_{o}$;
\label{07080602st3}
\item
$\mf{m}$ is a section of maps valued in $S-$valued maps,
\emph{contravariant} 
under action of $\Uptheta(\pf{U},\mc{L})$ via $S_{m}$ and $\mc{L}_{2}$,
whose values induce covariant sections under action of $H$ via the dual of 
$\uppsi\circ\mc{L}_{o}$ and via $\mf{b}\circ\mc{L}_{o}$;
\label{07080602st4}
\item
$\mc{Z}$ is a field contravariant under action of $\Upxi(\pf{D},\mc{L})$;
\item
$\mc{S}$ is a map defined on the morphism set of $\Upxi(\pf{D},\mc{L})$
satisfying properties consistent with $\mc{Z}$
and such that 
$\mc{S}(T,\mf{T},\alpha)$ is a $\ast-$homomorphism 
from $\mc{A}(b)_{\alpha}^{\mf{T}^{T}}$ to $\mc{A}(a)_{\alpha}^{\mf{T}}$,
for any
$a,b\in\Upxi(\pf{D},\mc{L})$, 
$T\in Mor_{\Upxi(\pf{D},\mc{L})}(b,a)$,
$\mf{T}\in\pf{D}_{a}$
and
$\alpha\in\ms{P}_{a}^{\mf{T}}$;
\item
$J$, a subfield of $\mc{Z}$,
is a field of disjoint union over operations of set of maps 
-
with values fragment states originated via the nucleon phases 
determined by $\mf{m}$
- 
contravariant under action of $\Upxi(\pf{D},\mc{L})$
via $J_{m}$ where $J_{1}=\mc{L}_{2}$ and $J_{2}$ is a map induced by the conjugate of 
the evaluation of $\mc{S}$;
\label{07080602st5}
\item
each fiber of $J$ is covariant under action of $H$
via a map induced by the dual of $\ms{V}\circ\mc{L}_{o}$;
\label{07080602st6}
\item
$\mc{W}$ is a map such that
$\ms{gr}\circ\mc{W}$ is a section of $J$,
\emph{contravariant} under action of $\Upxi(\pf{D},\mc{L})$ via $J_{m}$ and $\mc{L}_{2}$,
whose values
induce sections covariant under action of $H$ 
via the dual of $\ms{V}\circ\mc{L}_{o}$ and via $\mf{b}\circ\mc{L}_{o}$.
\label{07080602st7}
\end{enumerate}
Here 
$\Uptheta(\pf{U},\mc{L})$ 
is the subcategory of $\mf{C}$ inverse images via $\mc{L}$
of $Dom(\pf{U})$ the full subcategy of $\mf{G}(G,F,\uprho)$ whose object set is the domain of $\pf{U}$;
while 
$\Upxi(\pf{D},\mc{L})$ 
is the subcategory of $\mf{C}$ inverse images via $\mc{L}$
of $Dom(\pf{D})^{0}$ a strict subcategy of $\mf{G}(G,F,\uprho)$ whose object set is the domain of $\pf{D}$.
In particular $\Upxi(\pf{D},\mc{L})$ is a strict subcategory of $\Uptheta(\pf{U},\mc{L})$. 
$J_{1}$ is the component of $J_{m}$ acting on operations,
and $J_{2}$ is the component of $J_{m}$ acting on maps of fragment states.
\par
The symmetries above described in \eqref{07080602st4} and \eqref{07080602st7}
imply what we said in \eqref{07062056} in the initial list, 
we shall return later to these properties.
For any object $a$ of $\Upxi(\pf{D},\mc{L})$ and any $\mf{T}\in\pf{D}_{a}$,
$\mf{m}^{a}(\mf{T})$ and $\mc{W}^{a}(\mf{T})$ are maps defined on 
$\ms{P}_{a}^{\mf{T}}$ 
such that for all $\alpha\in\ms{P}_{a}^{\mf{T}}$
we have that 
\begin{enumerate}
\item
$\mf{m}^{a}(\mf{T},\alpha)$ 
is a nucleon phase of the nucleon system $\mc{L}(a)$
generated by the fissioning system $a$;
\item
\begin{equation}
\label{07121138}
\begin{aligned}
\exists\,\mf{r}&\in Rep^{a}(\mf{m}^{a}(\mf{T},\alpha))
\\
\mc{W}^{a}(\mf{T},\alpha)&=\uprho_{\mf{r}},\,
\mf{T}_{\mf{r}}=\mf{T},
\alpha_{\mf{r}}=\alpha;
\end{aligned}
\end{equation}
in particular 
\begin{enumerate}
\item
$\mc{W}^{a}(\mf{T},\alpha)\in\mf{N}^{a}(\mf{m}^{a}(\mf{T},\alpha))$
namely
$\mc{W}^{a}(\mf{T},\alpha)$ is a fragment state originated via $\mf{m}^{a}(\mf{T},\alpha)$,
\item
$\mc{W}^{a}(\mf{T},\alpha)$ is a state of the fragment system 
$\mc{O}(a)_{\alpha}^{\mc{T}}$
whose operator algebra is $\mc{A}(a)_{\alpha}^{\mf{T}}$
and if $\alpha>0$ then the system
evolves in time through $(\ep^{a})_{\alpha}^{\mf{T}}(-\alpha^{-1}(\cdot))$.
$\mf{m}^{a}(\mf{T},\alpha)$ is the phase of the nucleon system $\mc{L}(a)$, 
occurring by performing the operation $\mf{T}$ 
on the state of thermal equilibrium $(\ps{\upvarphi}^{a})_{\alpha}^{\mf{T}}$ at inverse temperature $\alpha$
with respect to the dynamical system underlying $\mc{O}(a)_{\alpha}^{\mf{T}}$.
\end{enumerate}
\end{enumerate}
Since the definition of the map $\ms{P}$ we have that
$\mf{m}(\mf{T},\alpha)$ and $\mc{W}(\mf{T},\alpha)$ 
occur by performing the operation $\mf{T}$
on $(\ps{\upvarphi}^{a})_{\alpha}^{\mf{T}}$
only for those 
$\alpha\in\ms{I}_{a}^{\mf{T}}$ 
such that 
\emph
{the symmetry group of the state of thermal equilibrium 
$(\ps{\upvarphi}^{a})_{\alpha}^{\mf{T}}$
is larger than
the symmetry group of the state of thermal equilibrium 
$(\ps{\upvarphi}^{a})_{(\beta_{c})_{a}^{\mf{T}}}^{\mf{T}}$},
justifying \eqref{07081619} in the initial list. 
\par
Let us call $\mf{m}$ and $\mc{W}$ 
the $\mc{T}-$nucleon phase and $\mc{T}-$fragment state respectively.
For any
$a,b\in\Upxi(\pf{D},\mc{L})$, $\mf{T}\in\pf{D}_{a}$, 
$\alpha\in\ms{P}_{a}^{\mf{T}}$,
$T\in Mor_{\Upxi(\pf{D},\mc{L})}(b,a)$
continuous functional $\upomega$ on $\mc{A}(a)_{\alpha}^{\mf{T}}$ 
and element $B$ of $\mc{A}(b)_{\alpha}^{\mf{T}^{T}}$
we let 
$\upomega^{T}=\upomega\circ\mc{S}(T,\mf{T},\alpha)$ and $B^{T}=\mc{S}(T,\mf{T},\alpha)B$.
\par
The characteristics core of $\mf{m}$ and $\mc{W}$
consists of \eqref{07121138} 
relating the two natural transformations, 
and the properties of equivariance
described in \eqref{07080602st4} and \eqref{07080602st7},
namely for any 
$d,e\in\Uptheta(\pf{U},\mc{L})$, 
$\mf{Q}\in\pf{U}_{d}$, $\beta\in\ms{P}_{d}^{\mf{Q}}$,
$Q\in Mor_{\Uptheta(\pf{U},\mc{L})}(e,d)$,
and 
$a,b\in\Upxi(\pf{D},\mc{L})$, 
$\mf{T}\in\pf{D}_{a}$, $\alpha\in\ms{P}_{a}^{\mf{T}}$,
$T\in Mor_{\Upxi(\pf{D},\mc{L})}(b,a)$,
$l\in H$,
we have
\begin{equation}
\label{07081634}
\begin{aligned}
\mf{m}^{e}(\mf{Q}^{Q},\beta)&=\mf{m}^{d}(\mf{Q},\beta)^{Q},
\\
\mf{m}^{d}(\mf{Q}^{l},\beta)&=\mf{m}^{d}(\mf{Q},\beta)^{l},
\\
\mc{W}^{b}(\mf{T}^{T},\alpha)&=\mc{W}^{a}(\mf{T},\alpha)^{T},
\\
\mc{W}^{a}(\mf{T}^{l},\alpha)&=\mc{W}^{a}(\mf{T},\alpha)^{l}.
\end{aligned}
\end{equation}
From the above equalities we deduce \eqref{07062056} in the initial list. 
More in general it is clear by the very definition that
nucleon-fragment doublets generate resolutions of the claim,
more specifically \eqref{07121138} and \eqref{07081634}
determine the $\mc{T}-$resolution of the 
equivariant form of the universality claim (M).
It is worthwhile emphasizing that the concept of nucleon-fragment doublet
is flexible enough to model diverse situations with symmetries $H$ and $\Upxi(\pf{D},\mc{L})$, 
in which an origination mechanism of the type above described occurs.
Thus the terms nucleon phase, fragment state, fissioning system
and later prompt-neutron yield,
have to be understood in a broad sense. 
\par
A nucleon-fragment doublet on $\mf{C}$ can be related with 
what we call an extended $\mf{C}-$equivariant stability, 
and the first main result of the entire work is the construction
of the \textbf{canonical extended $\ms{C}_{u}(H)-$equivariant stability $\mc{E}_{\bullet}$}.
\emph
{The main difficulty is to construct $\mf{m}$ and $\mc{W}$ satisfying 
\eqref{07121138} and \eqref{07081634}}.
Let $\mc{T}_{\bullet}$ denote the nucleon-fragment doublet on $\ms{C}_{u}(H)$
related to $\mc{E}_{\bullet}$, in such a case 
$\Uptheta(\pf{U},\mc{L})=\ms{C}_{u}(H)$
and
$\Upxi(\pf{D},\mc{L})=\ms{C}_{u}^{0}(H)$.
\par
Let us call global nucleon phase and global fragment state 
the $\mc{T}_{\bullet}-$nucleon phase and $\mc{T}_{\bullet}-$fragment state 
respectively,
then 
\textbf{\eqref{07121138} and \eqref{07081634}
applied to $\mc{T}_{\bullet}$ resolve the 
equivariant form of the universality claim (M).}
\par
In the second main result of this work
- 
under a suitable hypothesis
ensuring the functoriality of the source and target functors of the second of the below
morphisms 
-
we encode \eqref{07081634} in the following result 
\begin{equation}
\label{07121303}
\begin{aligned}
\mf{m}_{\star}
\text{ morphism of }
&\ms{Fct}(\ms{C}_{u}(H)^{op},\ms{Fct}(H,\ms{set})),
\\
\mf{v}_{\natural}
\text{ morphism of }
&\ms{Fct}(\ms{C}_{u}^{0}(H)^{op},\ms{Fct}(H,\ms{set})),
\end{aligned}
\end{equation}
related each other, modulo quotient, by the relation in \eqref{07121138},
establishing in this way the \textbf{compact equivariant form of the universality claim}.
\par
Finally let us examine how the concept of nucleon-fragment doublet 
enables also to resolve the invariant form of the universality claim.
To this end let us define a Terrell-type law corresponding to $\mc{T}$,
exhibiting invariances as a result of the above equivariances.
For any $a\in\Upxi(\pf{D},\mc{L})$
let the set of fission processes occurring to the fissioning system $a$
be the set $\ms{N}_{as}^{\mc{T}}(a)$ of the $4-$tuples
\begin{equation*}
\mf{n}=\lr{a}{\mf{T},\alpha,\{\ms{f}_{j},N_{j}\}_{j\in\{m,w\}}},
\end{equation*}
satisfying properties related to the binary fission phenomenon
and such that $\mf{T},\alpha$ are as above, 
$\ms{f}_{m}$
and 
$\ms{f}_{w}$ are observables of the nucleon system $\mc{L}(a)$,
while 
$N_{m}$ and $N_{w}$ are positive observables, 
namely elements of $(\mc{A}(a)_{\alpha}^{\mf{T}})_{+}$.
Let the set $\ms{N}_{as}^{\mc{T}}$ of fission processes 
be the union of the family $\{\ms{N}_{as}^{\mc{T}}(a)\}_{a\in\Upxi(\pf{D},\mc{L})}$. 
Since the semantics outlined we have that
\begin{enumerate}
\item
$a$ is the fissioning system for which the fission process $\mf{n}$ occurs;
\item
$\mc{L}(a)$ is the nucleon system generated by the fissioning system $a$;
\item
$\mc{O}(a)_{\alpha}^{\mf{T}}$ 
is the fragment system 
whose observable algebra 
is $\mc{A}(a)_{\alpha}^{\mf{T}}$ and 
if $\alpha>0$ evolving in time through
$(\ep^{a})_{\alpha}^{\mf{T}}(-\alpha^{-1}(\cdot))$;
\item
$(\ps{\upvarphi}^{a})_{\alpha}^{\mf{T}}$ is a state of $\mc{O}(a)_{\alpha}^{\mf{T}}$ 
and if $\alpha>0$ then
it is a state of thermal equilibrium at inverse temperature $\alpha$
with respect to the underlying dynamical system of 
$\mc{O}(a)_{\alpha}^{\mf{T}}$;
\item 
$\mf{m}^{a}(\mf{T},\alpha)$ 
is the phase of $\mc{L}(a)$, 
occurring by performing $\mf{T}$ on $(\ps{\upvarphi}^{a})_{\alpha}^{\mf{T}}$;
\item 
$\mc{W}^{a}(\mf{T},\alpha)$ 
is the state of $\mc{O}(a)_{\alpha}^{\mf{T}}$ 
\emph{originated} via $\mf{m}^{a}(\mf{T},\alpha)$;
\item
for $j\in\{m,w\}$
\begin{enumerate}
\item
$\mf{m}^{a}(\mf{T},\alpha)(\ms{f}_{j})$ is the mean value in $\mf{m}^{a}(\mf{T},\alpha)$ 
of $\ms{f}_{j}$, 
\item
$\mc{W}^{a}(\mf{T},\alpha)(N_{j})$ is the mean value in $\mc{W}^{a}(\mf{T},\alpha)$ 
of $N_{j}$.
\end{enumerate}
\end{enumerate}
In order to decode the binary fission let us 
give the following interpretation of the data of $\mf{n}$.
Let $\#_{m}=$light and $\#_{w}=$heavy, then for all $j\in\{m,w\}$
\begin{enumerate}
\item
$\mf{T}$ is the operation realizing the fission process $\mf{n}$ whenever performed on 
$(\ps{\upvarphi}^{a})_{\alpha}^{\mf{T}}$;
\item
$N_{j}$ is the observable of $\mc{O}(a)_{\alpha}^{\mf{T}}$ relative to
the mass of the $\#_{j}$ fragment;
\item
$\ms{f}_{j}$ is the observable of $\mc{L}(a)$ relative to
the mass of the $\#_{j}$ nucleon core.
\end{enumerate}
Next to $\ms{N}_{as}^{\mc{T}}$ the action of $H$ naturally extends, since for all $l\in H$ we can set 
\begin{equation*}
\mf{k}^{\mc{T}}(l)(\mf{n})=
\lr{a}{\mf{T}^{l},\alpha,\{\ms{f}_{j}^{l},N_{j}^{l}\}_{j\in\{m,w\}}},
\end{equation*}
later simply denoted by $\mf{n}^{l}$,
as well ``essentially'' extends the contravariant action of $\Upxi(\pf{D},\mc{L})$,
meaning that for any object $b$ of $\Upxi(\pf{D},\mc{L})$
and any morphism $T$ of $\Upxi(\pf{D},\mc{L})$ from $b$ to $a$
we have
\begin{equation*}
\mf{n}^{(T,\mf{x})}
=\lr{b}{\mf{T}^{T},\alpha,\{\ms{f}_{j}',N_{j}'\}_{j\in\{m,w\}}},
\end{equation*}
where $\mf{x}=\{\ms{f}_{j}',N_{j}'\}_{j\in\{m,w\}}$ such that
$(\ms{f}_{j}')^{T}=\ms{f}_{j}$ and $(N_{j}')^{T}=N_{j}$.
Later we shall see how to modify this map in order to have a ``true'' contravariant action.
\par
For all $j\in\{m,w\}$ define 
$\upkappa_{j}^{\mc{T}}$
and
$\upzeta_{j}^{\mc{T}}$
the functions
on 
$\ms{N}_{as}^{\mc{T}}$
mapping any $\mf{n}$
into 
\begin{equation*}
\begin{aligned}
\upkappa_{j}^{\mc{T}}(\mf{n})
&\coloneqq
\mc{W}^{a}(\mf{T},\alpha)(N_{j}),
\\
\upzeta_{j}^{\mc{T}}(\mf{n})
&\coloneqq
\mf{m}^{a}(\mf{T},\alpha)(\ms{f}_{j});
\end{aligned}
\end{equation*}
called restricted $\mc{T}-\#_{j}$ fragment mass 
and restricted $\mc{T}-\#_{j}$ nucleon mass, 
collectively
called restricted 
$\mc{T}-$fragment masses
and
$\mc{T}-$nucleon masses
respectively.
Next the restricted $\mc{T}-$Terrell law is the function 
$\nuup^{\mc{T}}$ 
on 
$\ms{N}_{as}^{\mc{T}}$
mapping any $\mf{n}$
into the mean value $\nuup^{\mc{T}}(\mf{n})$ 
of the prompt-neutron yield in $\mc{W}^{a}(\mf{T},\alpha)$, 
said also the 
\emph{prompt-neutron yield of the fission process $\mf{n}$},
where
\begin{equation*}
\nuup^{\mc{T}}
\coloneqq
0.08(\upkappa_{m}^{\mc{T}}-\upzeta_{m}^{\mc{T}})
+
0.1(\upkappa_{w}^{\mc{T}}-\upzeta_{w}^{\mc{T}}).
\end{equation*}
As a result of \eqref{07081634} we obtain 
under essential contravariant action of $\Upxi(\pf{D},\mc{L})$ and action of $H$ 
the invariance of the restricted $\mc{T}-$fragment and nucleon masses,
i.e. for all $j\in\{m,w\}$ 
\begin{equation}
\label{06082144}
\begin{aligned}
\upkappa_{j}^{\mc{T}}(\mf{n}^{(T,\mf{x})})
&=
\upkappa_{j}^{\mc{T}}(\mf{n}),
\\
\upkappa_{j}^{\mc{T}}(\mf{n}^{l})
&=
\upkappa_{j}^{\mc{T}}(\mf{n}),
\\
\upzeta_{j}^{\mc{T}}(\mf{n}^{(T,\mf{x})})
&=
\upzeta_{j}^{\mc{T}}(\mf{n}),
\\
\upzeta_{j}^{\mc{T}}(\mf{n}^{l})
&=
\upzeta_{j}^{\mc{T}}(\mf{n});
\end{aligned}
\end{equation}
therefore the invariance of $\nuup^{\mc{T}}$ follows
\begin{equation}
\label{06082142}
\begin{aligned}
\nuup^{\mc{T}}(\mf{n}^{(T,\mf{x})})
&=
\nuup^{\mc{T}}(\mf{n}),
\\
\nuup^{\mc{T}}(\mf{n}^{l})
&=
\nuup^{\mc{T}}(\mf{n}).
\end{aligned}
\end{equation}
\eqref{06082144} and \eqref{06082142}
constitute 
\emph
{the invariance of the restricted $\mc{T}-$fragment and nucleon masses, 
and the restricted $\mc{T}-$Terrell law
respectively.
The properties of invariance of the restricted global fragment and nucleon masses, 
and restricted global Terrell law follow for $\mc{T}=\mc{T}_{\bullet}$.}
\par
Now in order to have a true contravariant action and then to resolve the invariant form of the 
universality claim let us procede as follows.
Let $\mc{P}^{\mc{T}}=(\mc{P}_{o}^{\mc{T}},\mc{P}_{m}^{\mc{T}})$
be the couple of maps 
defined on the object set and morphism set of $\Upxi(\pf{D},\mc{L})$ respectively,
such that 
$\mc{P}_{o}^{\mc{T}}(a)$ is the power set of $\ms{N}_{as}^{\mc{T}}(a)$,
while $\mc{P}_{m}^{\mc{T}}(T)$ is the map on $\mc{P}_{o}^{\mc{T}}(a)$ 
preserving the union and
mapping any $\{\mf{n}\}$ into the set of the $\mf{n}^{(T,\mf{x})}$ 
where $\mf{x}=\{\ms{f}_{j}',N_{j}'\}_{j\in\{m,w\}}$ such that
$(\ms{f}_{j}')^{T}=\ms{f}_{j}$ and $(N_{j}')^{T}=N_{j}$.
\par
We call 
\emph
{$\mc{T}-\#_{j}$ nucleon mass and $\mc{T}-\#_{j}$ fragment mass} 
the maps
$\muup_{j}^{\mc{T}}$
and
$\uplambda_{j}^{\mc{T}}$
defined on the object set of $\Upxi(\pf{D},\mc{L})$ such that 
$(\muup_{j}^{\mc{T}})_{a}$ 
and
$(\uplambda_{j}^{\mc{T}})_{a}$ 
are the $\mathscr{P}(\R)-$valued extensions to $\mc{P}_{o}^{\mc{T}}(a)$ of the 
restrictions of 
$\upzeta_{j}^{\mc{T}}$ 
and
$\upkappa_{j}^{\mc{T}}$ 
respectively
to $\ms{N}_{as}^{\mc{T}}(a)$,
with $j\in\{m,w\}$.
Moreover we call 
\emph{$\mc{T}$-Terrell law} the map $\uptheta^{\mc{T}}$ 
defined on the object set of $\Upxi(\pf{D},\mc{L})$ such that 
$\uptheta_{a}^{\mc{T}}$ 
is the $\mathscr{P}(\R)-$valued extension to $\mc{P}_{o}^{\mc{T}}(a)$ of the 
restriction of $\upnu^{\mc{T}}$ 
to $\ms{N}_{as}^{\mc{T}}(a)$.
\par
Now $\mc{P}^{\mc{T}}$ results to be a functor from $\Upxi(\pf{D},\mc{L})^{op}$ to $\ms{set}$,
since \eqref{07081634}. 
Let $\ms{R}^{\mc{T}}$ denote the unique functor from $\Upxi(\pf{D},\mc{L})^{op}$ to $\ms{set}$
such that its object map is the constant map with constant value equal the power set $\mathscr{P}(\R)$ of $\R$ 
and its morphism map is the constant map with constant value equal $Id_{\mathscr{P}(\R)}$.
Let $\ms{R}^{H}$ denote the unique functor from $H$ to $\ms{set}$ such that 
$\ms{R}_{o}^{H}=(\un\mapsto\mathscr{P}(\R))$ and $\ms{R}_{m}^{H}$ is the constant map 
with constant value equal to $Id_{\mathscr{P}(\R)}$,
Finally let $\uptau_{a}^{\mc{T}}$ the map on $H$ 
such that $\uptau_{a}^{\mc{T}}(l)$ is the extension of
$\mf{k}^{\mc{T}}(l)$
to $\mc{P}_{o}^{\mc{T}}(a)$ for all $l\in H$,
and set $\ms{Q}_{a}^{\mc{T}}=(\un\mapsto\mc{P}_{o}^{\mc{T}}(a),\uptau_{a}^{\mc{T}})$.
\emph
{Thus as a result of \eqref{07081634} the 
universality of the $\mc{T}-$nucleon masses and $\mc{T}-$Terrell law
follows for all $j\in\{m,w\}$ in terms of 
\begin{enumerate}
\item
Invariance of the $\mc{T}-$nucleon and fragment masses and $\mc{T}-$Terrell law 
under contravariant action of $\Upxi(\pf{D},\mc{L})$
\begin{equation}
\label{07311318}
\muup_{j}^{\mc{T}},\,
\uplambda_{j}^{\mc{T}},\,
\uptheta^{\mc{T}}
\in
Mor_{\ms{Fct}(\Upxi(\pf{D},\mc{L})^{op},\ms{set})}(\mc{P}^{\mc{T}},\ms{R}^{\mc{T}}).
\end{equation}
\item
Invariance of the $\mc{T}-$nucleon and fragment masses and $\mc{T}-$Terrell law 
under action of $H$.
For all $\ms{a}\in\Upxi(\pf{D},\mc{L})$ 
\begin{equation}
\label{07311319}
(\un\mapsto(\muup_{j})_{\ms{a}}^{\mc{T}}),\,
(\un\mapsto(\uplambda_{j})_{\ms{a}}^{\mc{T}}),\,
(\un\mapsto\uptheta_{\ms{a}}^{\mc{T}})
\in
Mor_{\ms{Fct}(H,\ms{set})}(\ms{Q}_{\ms{a}}^{\mc{T}},\ms{R}^{H}).
\end{equation}
\end{enumerate}}
In other words 
\begin{equation*}
\begin{aligned}
(\muup_{j}^{\mc{T}})_{d(T)}
\circ
\mc{P}_{m}^{\mc{T}}(T)
&=
(\muup_{j}^{\mc{T}})_{c(T)},
\\
(\muup_{j}^{\mc{T}})_{\ms{a}}
\circ
\uptau_{\ms{a}}^{\mc{T}}(l)
&=
(\muup_{j}^{\mc{T}})_{\ms{a}},\,
\forall l\in H 
\end{aligned}
\end{equation*}
similarly for $\uplambda_{j}^{\mc{T}}$,
and
\begin{equation*}
\begin{aligned}
\uptheta_{d(T)}^{\mc{T}}
\circ
\mc{P}_{m}^{\mc{T}}(T)
&=
\uptheta_{c(T)}^{\mc{T}},
\\
\uptheta_{\ms{a}}^{\mc{T}}
\circ
\uptau_{\ms{a}}^{\mc{T}}(l)
&=
\uptheta_{\ms{a}}^{\mc{T}},\,
\forall l\in H. 
\end{aligned}
\end{equation*}
Here the contravariance of $\mc{P}^{\mc{T}}$ under action of $\Upxi(\pf{D},\mc{L})$ 
means that
for all $T,S\in Mor_{\Upxi(\pf{D},\mc{L})}$ such that $d(T)=c(S)$ 
\begin{equation*}
\begin{aligned}
\mc{P}_{m}^{\mc{T}}(T)
\mc{P}_{o}^{\mc{T}}(c(T))
&\subseteq
\mc{P}_{o}^{\mc{T}}(d(T)),
\\
\mc{P}_{m}^{\mc{T}}(T\circ S)
&=
\mc{P}_{m}^{\mc{T}}(S)
\circ
\mc{P}_{m}^{\mc{T}}(T).
\end{aligned}
\end{equation*}
Thus
\eqref{07311318} and \eqref{07311319}
jointly \eqref{07121138} and \eqref{07081634}
represent the $\mc{T}-$resolution of the 
invariant form of the universality claim.
\par
If we call
global $\#_{j}$ nucleon mass,
global $\#_{j}$ fragment mass
and 
global Terrell law
the maps
$\muup_{j}^{\mc{T}_{\bullet}}$,
$\uplambda_{j}^{\mc{T}_{\bullet}}$
and 
$\uptheta^{\mc{T}_{\bullet}}$
respectively, 
then
\textbf{the universality of the global nucleon and fragment masses
and the universality of the global Terrell law
follow since 
\eqref{07311318} and \eqref{07311319}
for $\mc{T}=\mc{T}_{\bullet}$,
resolving jointly \eqref{07121138} and \eqref{07081634}
the invariant form of the universality claim}.
This is the third main result of this work.
Incidentally the invariant form of the universality claim 
results as a consequence of its equivariant form.
\par
Finally let the $\#_{j}$ nucleon and fragment masses, and Terrell law 
denote the restrictions of 
$\upzeta_{j}^{\mc{T}_{\bullet}}$,
$\upkappa_{j}^{\mc{T}_{\bullet}}$
and
$\nuup^{\mc{T}_{\bullet}}$
respectively
to the set of all the fission processes $\mf{n}$ satisfying the 
revised nucleon phase hypothesis
requiring 
$\upzeta_{m}^{\mc{T}_{\bullet}}(\mf{n})=82$,
$\upzeta_{w}^{\mc{T}_{\bullet}}(\mf{n})=126$
and that $H$ would contain as a subgroup 
the direct product of the universal covering group of the Poincar\'{e} 
group with the gauge group of the standard model.
Thus under the revised nucleon phase hypothesis 
as a result of \eqref{06082144} and \eqref{06082142} we can state what follows.
The light and heavy nucleon masses are invariant with constant values $82$ and $126$
and the prompt-neutron yield \eqref{06091115} is invariant 
under essential contravariant action 
over the field of fission processes
of suitable perturbations of \emph{fissioning systems},
and under action over the field of fission processes of relativistic
transformations of \emph{reference frames}.
\par
In conclusion it is worthwhile remarking that 
the universality of the global nucleon masses and the universality of the global Terrell law
- 
in the special form of the stability of the values $82$ and $126$ occurring in the Terrell law
and the invariance of the Terrell law itself under the above transformations
mainly the relativistic ones
-
is an experimentally testable property that can be provided 
in order to furnish an indirect empirical evidence of the existence and universality of the 
global nucleon phase $\mf{m}$.
\par
Let us outline the main content of the three parts.
\par
In part \ref{07301106}
we introduce the category $\mf{G}(G,F,\uprho)$ of nucleon systems,
set its physical interpretation,
define the concepts of extended $\mf{C}-$equivariant stability,
nucleon-fragment doublet $\mc{T}$ on a category $\mf{C}$,
define the $\mc{T}-$nucleon phase, the $\mc{T}-$fragment state
and the $\mc{T}-$resolution 
of the equivariant form of the universality claim.
Then in the first main result of this work 
we construct, and in this way establish the existence of, 
the canonical extended $\ms{C}_{u}(H)-$equivariant stability
and the canonical nucleon-fragment doublet $\mc{T}_{\bullet}$ on $\ms{C}_{u}(H)$.
As a result we resolve the equivariant form of the universality claim,
by applying its $\mc{T}-$resolution to $\mc{T}_{\bullet}$. 
\par
In part \ref{07301107} 
we establish in a more coincise and elegant form  
the equivariant form of the universality claim.
Namely in the second main result of this work 
we encode the equivariances \eqref{07081634} 
applied for $\mc{T}=\mc{T}_{\bullet}$
in a unique fashion into the existence of 
natural transformations $\mf{m}_{\star}$ and $\mf{v}_{\natural}$ satisfying \eqref{07121303}.
Then we define the $\mc{T}-$nucleon and $\mc{T}-$fragment light and heavy masses, 
and the $\mc{T}-$Terrell law, establish 
the $\mc{T}-$resolution of the invariant form of the universality claim,
and apply it to $\mc{T}_{\bullet}$ in order to establish 
in the third main result of this work 
the invariant form of the universality claim,
in particular the universality of the global Terrell law.
\section{Terminology and preliminaries}
\label{not1}
\subsection{Sets and topologies}
In all three parts 
we consider the Zermelo-Fraenkel theory together the axiom of 
universes stating that for any set there exists a universe 
containing it as an element \cite[p. $10$]{ks}.
We consider fixed a universe $\mc{U}$ and a universe $\mc{U}_{0}$ such that $\mc{U}\in\mc{U}_{0}$.
See \cite[Expose I Appendice]{sga4} for the definition and properties of universes, 
see also \cite[p. $22$]{mcl} and \cite[$\S1.1$]{bor},
in particular if $V$ is a universe, then $y\in V$ implies $y\subset V$,
hence $\mc{U}\subset\mc{U}_{0}$.
Let $\ms{set}$ be the category of sets belonging to the universe $\mc{U}$, 
functions as morphisms with map composition.
Whenever we refer to a set unless the contrary is stated,
we mean an element of $\mc{U}$, moreover
\textbf{for any structure $S$} (e.g. the structure of topological space, topological algebra, etc.)
\textbf{whenever we refer to ``the set of the $S$'s'', we always mean the subset of 
those elements of $\mc{U}$ satisfying the axioms of $S$}.
Let $\ms{Gr}$ be the category of groups and group morphisms with map composition
and let $\ms{Ab}$ be the full subcategory of $\ms{Gr}$ of abelian groups.
\par
If $A$, $B$ and $C$ are categories,
then we 
let $Obj(A)$ denote the set of objects of $A$, we let $a\in A$ denote $a\in Obj(A)$. 
For any $x,y\in A$ let $Mor_{A}(x,y)$ be the set of morphisms of $A$ from $x$ to $y$, 
let $\un_{x}$ be the identity morphism of $x$, 
while $Inv_{A}(x,y)=\{f\in Mor_{A}(x,y)\,\vert\,(\exists g\in Mor_{A}(y,x))(f\circ g=\un_{y},g\circ f=\un_{x}))\}$ 
denotes the, possibly empty if $x\neq y$,
set of invertible morphisms from $x$ to $y$;
set $Aut_{A}(y)=Inv_{A}(y,y)$, for any $y\in A$.
\par
Given two sets $A,B$ let $\mathscr{P}(A)$ and $\mathscr{P}_{\omega}(A)$  
denote the set of subsets and finite subsets of $A$ respectively.
For any map $f:A\to B$ let $f^{-1}:\mathscr{P}(B)\to\mathscr{P}(A)$ 
such that $f^{-1}(Y)\coloneqq\{a\in A\,\vert\,f(a)\in Y\}$ for any $Y\subseteq B$.
Given any set $x$, often and only if it will not cause confusion, 
we use the convention to denote $\{x\}$ by $x$,
so for example if $f:A\to B$ and $b\in B$, then
$f^{-1}(b)$ stands for $f^{-1}(\{b\})$. 
Let $\ms{ev}_{(\cdot)}$ denote the evaluation map, i.e. if $F:A\to B$ is any map and $a\in A$, 
then $\ms{ev}_{a}(F)\coloneqq F(a)$.
If $\ms{x}\in\prod_{a\in A}Mor_{\ms{set}}(B_{a},C_{a})$, 
with $A$ object of $\ms{set}$ and $B_{a},C_{a}$ objects of $\ms{set}$ for all $a\in A$, 
then we let $\ms{x}(a,b)$ denote $\ms{x}(a)(b)$ for any $b\in B_{a}$.
If $Y(a,b)=Y\in\mc{U}_{0}$ for all $a\in A$ and $b\in B_{a}$, then we let 
$\prod_{a\in A}\prod_{b\in B_{a}}Y$
denote
$\prod_{a\in A}\prod_{b\in B_{a}}Y(a,b)$.
If $f:X\to A$ and $g:X\to B$ then by abuse of the standard language 
we denote by $f\times g$ the map on $X$ with values in 
$A\times B$ such that $(f\times g)(x)\coloneqq(f(x),g(x))$ for all $x\in X$.
\par
Set
$\N_{0}\coloneqq\N-\{0\}$ and $\R_{0}\coloneqq\R-\{0\}$, 
while $\widetilde{\R}\coloneqq\R\cup\{\infty\}$
provided by the topology of one-point compactification.
If $A$ is any set then $Id_{A}$ is the identity map on $A$
we often use the convention to remove the index $A$ if it is 
clear the set involved.
If $S$ is a topological space, then $Cl(S)$, $Op(S)$ and $Comp(S)$ denote the sets of closed, 
open and compact subsets of $S$ respectively, while $\mc{B}(S)$ denotes the $\sigma-$field of Borel subsets of $S$.
If $T$ is a locally compact space and $E$ is a Hausdorff locally convex
space, let $\mc{C}_{c}(T,E)$ denote the linear space of continuous $E-$valued 
maps $f$ on $T$ with compact support $supp(f)$, 
where $supp(f)\coloneqq\ov{f^{-1}(E-\{\ze\})}$, set $\mc{C}_{c}(T)\coloneqq\mc{C}_{c}(T,\C)$
provided by the inductive limit topology
of uniform convergence over compact subsets of $T$.
\par
Let $X$ and $T$ be a locally compact group and locally compact space respectively, 
then 
$\mc{H}(X)$ is the set of Haar measures on $X$
and
$\mc{M}(T)$ is the set of Radon measures on $T$.
Let $\mu\in\mc{M}(T)$
and $S$ be a locally compact subspace of $T$, set
$\mu_{S}:\mc{C}_{c}(S)\to\C$ such that
$\mu_{S}(f)\coloneqq
\mu(\widetilde{f})$,
where $\widetilde{f}$ is the $\ze-$extension on $T$ of 
$f\in\mc{C}_{c}(S)$, thus $\mu_{S}\in\mc{M}(S)$.
Let $T,S$ be two locally compact spaces, $\mu\in\mc{M}(T)$
and $\ep:T\to S$ be $\mu-$proper, thus $\ep(\mu)$ denotes
the image of $\mu$ under $\ep$ as defined in
\cite[Ch. $5$, $\S 6$, $n^{\circ}1$, Def. $1$]{IntBourb}.
By construction 
$\ep(\mu)\in\mc{M}(S)$
such that for all $f\in\mc{C}_{c}(S)$
\begin{equation}
\label{09191048}
\int f\,d\ep(\mu)
=
\int f\circ\ep\,d\mu.
\end{equation}
\par
Let $X$ be a locally compact group and $s\in X$,
set
$L_{s},R_{s}:X\to X$, 
such that
$L_{s}(x)\coloneqq s\cdot x$
and
$R_{s}(x)\coloneqq x\cdot s$,
while $L_{s}^{\ast},R_{s}^{\ast}:\C^{X}\to\C^{X}$
such that
$L_{s}^{\ast}(h)\coloneqq h\circ L_{s^{-1}}$
and
$R_{s}^{\ast}(h)\coloneqq h\circ R_{s^{-1}}$
respectively.
Whenever we will deal with different groups, it will be clear by the 
context to which group the maps $R$ and $L$ are referring to.
By definition $\mc{H}(X)$ is the set of left-invariant $\mu\in\mc{M}(X)$,
i.e. $\mu\in\mc{M}(X)$ such that $\mu\circ L_{s}^{\ast}\up\mc{C}_{c}(X)=\mu$, for all $s\in X$.
\par
If $X$ and $Y$ are two topological linear spaces over 
$\K\in\{\R,\C\}$, 
$\mc{L}(X,Y)$ denotes the
linear space of continuous linear maps from $X$ to $Y$,
set $\mc{L}(X)\coloneqq\mc{L}(X,X)$ and $X^{\ast}\coloneqq\mc{L}(X,\K)$.
$\mc{L}_{s}(X,Y)$ is the topological linear space
whose underlying linear space is $\mc{L}(X,Y)$ provided by the topology of pointwise convergence,
while
$\mc{L}_{w}(X,Y)$ is the locally convex linear space
whose underlying linear space is $\mc{L}(X,Y)$ provided
by the topology generated by the following set
of seminorms
$\{q_{(\upphi,x)}\mid(\upphi,x)\in Y^{\ast}\times X\}$,
where $q_{(\upphi,x)}(A)\doteq|\upphi(Ax)|$.
All the normed spaces in this work are assumed to be over the complex field.
In case $X$ is a normed space we assume $\mc{L}(X)$ to be provided by the topology generated by the usual $\sup-$norm.
If $\ms{X}$ is any structure including as a substructure the one of normed space say $\ms{X}_{0}$, 
for example the normed space underlying any normed algebra,
we let $\mc{L}(\ms{X})$ denote the normed space $\mc{L}(\ms{X}_{0})$
whenever it does not cause conflict of notations.
Therefore we never shall use this convention
in case $\ms{X}$ is an Hilbert $C^{\ast}-$module, 
where $\mc{L}(\ms{X})$ always denotes the set of all adjointable operators on $\ms{X}$, as we shall see in the second part.
If $A$ and $B$ are two linear operators in $X$, we set $[A,B]\coloneqq AB-BA$, where the composition and sum are to be 
understood in the context of possibly unbounded operators, i.e. defined on the intersection of the corresponding domains. 
If $X,Y$ are Hilbert spaces and $U\in\mc{L}(X,Y)$ is unitary then $\ms{ad}(U)\in\mc{L}(\mc{L}(X),\mc{L}(Y))$ 
denotes the isometry defined by $\ms{ad}(U)(a)\coloneqq UaU^{-1}$, for all $a\in\mc{L}(X)$.
Let $HS$ denote the set of Hilbert spaces. 
Finally unless the contrary is not stated any convention established in one part has to be understood valid for the
remaining of that part and for the following parts. 
\subsection{$C^{\ast}-$algebras, $C^{\ast}-$dynamical systems and their crossed products}
For any normed algebra $\mc{D}$ let
$R_{(\cdot)}^{\mc{D}}:\mc{D}\to\mc{L}(\mc{D})$ 
and
$L_{(\cdot)}^{\mc{D}}:\mc{D}\to\mc{L}(\mc{D})$ 
denote the 
right and left 
multiplication map on $\mc{D}$ respectively, i.e. 
$R_{a}(b)=ba$ and $L_{a}(b)=ab$ 
for any $a,b\in\mc{D}$, often we remove the index $\mc{D}$.
Let $\mc{A}$ be a $C^{\ast}-$algebra,
let $\mc{A}_{ob}\coloneqq\{a\in\mc{A}\,\vert\, a=a^{\ast}\}$
and $\mc{A}_{+}\coloneqq\{a^{\ast}a\,\vert\, a\in\mc{A}\}$.
$\ms{E}_{\mc{A}}$ denotes the set of states of $\mc{A}$,
if $\mc{A}$ is a von Neumann algebra $\ms{N}_{\mc{A}}$ denotes the set of its normal states.
If $\mc{{B}}$ is a $C^{\ast}-$algebra  
$Hom^{\ast}(\mc{A},\mc{B})$ is the set of $\ast-$homomorphisms
defined on $\mc{A}$ and at values in $\mc{B}$,
$Isom^{\ast}(\mc{A},\mc{B})$ is the subset of bijective elements of $Hom^{\ast}(\mc{A},\mc{B})$
and $Aut^{\ast}(\mc{A})=Isom^{\ast}(\mc{A},\mc{A})$.
Let $\ms{CA}^{\ast}$ denote the category of $C^{\ast}-$algebras and $\ast-$homomorphisms
with map composition as law of morphism composition,
easily we deduce that 
$Isom^{\ast}(\mc{A},\mc{B})=Inv_{\ms{CA}^{\ast}}(\mc{A},\mc{B})$,
so
$Aut^{\ast}(\mc{A})=Aut_{\ms{CA}^{\ast}}(\mc{A})$.
If in addition $\mc{A}$ and $\mc{B}$ are unital, then 
$Hom_{\un}^{\ast}(\mc{A},\mc{B})$ is the set of unit preserving
$\ast-$homomorphisms defined on $\mc{A}$ and at values in $\mc{B}$.
Let $\mc{A}^{+}$ denote the $\ast-$algebra whose underlying linear space
is $\mc{A}\times\C$, the involution and product are 
defined by 
$(a,\lambda)^{\ast}=(a^{\ast},\ov{\lambda})$, 
and
$(a,\lambda)\cdot(b,\mu)\coloneqq
(a\cdot b+\lambda b+\mu a,\lambda\mu)$, 
for all
$(a,\lambda),(b,\mu)\in\mc{A}\times\C$.
(\cite[Ch. $2$, $\S 7$, $n^{\circ}2$, $I$ and 
$\S 10$, $n^{\circ}1$, $III$]{nai}).
$(\ze,1)$ is the unity of $\mc{A}^{+}$, while 
the map $\upphi:a\mapsto(a,0)$ is an injective $\ast-$isomorphism
of $\mc{A}$ into $\mc{A}^{+}$, we often shall identify $\mc{A}$
with its image in $\mc{A}^{+}$ under $\upphi$. 
Under this identification 
$\mc{A}$ is a two side ideal of $\mc{A}^{+}$, 
so we have that
$L_{x},R_{x}\up\mc{A}$
are linear endomorphisms of the vector space underlying $\mc{A}$
for any $x\in\mc{A}^{+}$.
$\mc{A}^{+}$ is a $C^{\ast}-$algebra
if provided by
the following norm extending the one on $\mc{A}$, 
$\|x\|\coloneqq\|L_{x}\up\mc{A}\|$, well-set 
since $L_{(a,\lambda)}\up\mc{A}=L_{a}+\lambda\cdot\in\mc{L}(\mc{A})$
(\cite[Ch. $3$, $\S 16$, $n^{\circ}1$, $III$]{nai}).
Set $\tilde{A}$ the smallest unital subalgebra of $\mc{A}^{+}$
containing $\mc{A}$, so 
$\tilde{\mc{A}}=\upphi(\mc{A})$, if $\mc{A}$ has the identity,
$\tilde{\mc{A}}=\mc{A}^{+}$ otherwise.
Let $\mc{B}$ a $\ast-$algebra and 
$\upalpha\in Mor_{\ms{CA}^{\ast}}(\mc{A},\mc{B})$, 
set $\upalpha^{+}:\mc{A}^{+}\ni(a,\lambda)\mapsto(\upalpha(a),\lambda)\in\mc{B}^{+}$, thus
\begin{equation}
\label{10291116}
\upalpha^{+}\in Hom_{\un}^{\ast}(\mc{A}^{+},\mc{B}^{+}),
\end{equation}
while in case $\mc{B}$ has the identity, we set
\begin{equation}
\label{10281544}
\begin{cases}
\tilde{\upalpha}:\mc{A}^{+}\to\mc{B},
\\
(a,\lambda)\mapsto\upalpha(a)+\un_{\mc{B}}\lambda,
\end{cases}
\end{equation}
thus $\tilde{\upalpha}\in Hom_{\un}^{\ast}(\mc{A}^{+},\mc{B})$,
called the $\ast-$homomorphism of $\mc{A}^{+}$ induced by $\upalpha$
($\ast-$representation in case 
$\mc{B}$ equals $\mc{L}(\mf{H})$ for some Hilbert space $\mf{H}$).
\par
Let $Aut_{s}^{\ast}(\mc{A})$ denote the topological group of
$\ast-$automorphisms of $\mc{A}$
provided by the topology of pointwise convergence.
$Rep(\mc{A})$ denotes the set of $\ast-$representations of $\mc{A}$,
while $Rep_{c}(\mc{A})$ 
denotes the set of cyclic $\ast-$representations of $\mc{A}$.
If 
$\pf{H}=\lr{\mf{H},\uppi}{\Upomega}$
and
$\pf{K}=\lr{\mf{K},\upxi}{\Uppsi}$
are in $Rep_{c}(\mc{A})$ 
then we call them unitarily equivalent if 
their underlying representations of $\mc{A}$ are unitarily equivalent,
say through the unitary operator $U:\mf{H}\to\mf{K}$, and
$U\Upomega=\Uppsi$.
If $\uppi$ is a nondegenerate representation of $\mc{A}$ on $\mf{H}$, 
then we denote by $\ms{N}_{\uppi}$ and call $\uppi-$normal states its elements, the set of the $\upphi\circ\uppi$ where 
$\upphi\in\ms{N}_{\mc{L}(\mf{H})}$. Since $\uppi$ is nondegenerate $\ms{N}_{\uppi}\subset\ms{E}_{\mc{A}}$. 
If $\uppsi$ is a state of $\mc{A}$ then $\uppsi-$normal means $\uppi-$normal, where $\lr{\mf{H},\uppi}{\Omega}$ 
is a cyclic representation associated with $\uppsi$. 
If $T\in Mor_{\ms{CA}^{\ast}}(\mc{A},\mc{B})$, set 
$T_{\dagger}:\mc{B}^{\ast}\to\mc{A}^{\ast}$
such that 
$T_{\dagger}(\upomega)=\omega\circ T$,
if
$\uplambda\in Inv_{\ms{CA}^{\ast}}(\mc{A},\mc{B})$,
then set $\uplambda^{\ast}=(\uplambda^{-1})_{\dagger}$.
For any map 
$\ms{f}:\ms{D}\subseteq\mc{A}\to\mc{A}$
define the map
$\ms{ad}(\uplambda)(\ms{f}):
\uplambda(\ms{D})\to\mc{A}$
such that for all $b\in\uplambda(\ms{D})$
\begin{equation}
\label{09081203}
\ms{ad}(\uplambda)(\ms{f})(b)
\coloneqq
(\uplambda\circ\ms{f}\circ\uplambda^{-1})(b).
\end{equation}
Let $U$ be any set and
$\upgamma:U\to Aut_{\ms{CA}^{\ast}}(\mc{A})$,
define
\begin{equation*}
\ms{E}_{\mc{A}}^{U}(\upgamma)
\coloneqq\{
\uppsi\in\ms{E}_{\mc{A}}
\mid
(\forall u\in U)(\uppsi\circ\upgamma(u)=\uppsi)\}.
\end{equation*}
Let
$\mf{A}=\lr{\mc{A}}{H,\upsigma}$ 
and
$\mf{B}=\lr{\mc{B}}{H,\uptheta}$ 
be
$C^{\ast}-$dynamical systems,
here called simply dynamical systems
or dynamical systems with group symmetry $H$.
$T$ 
is an
$(\mf{A},\mf{B})-$equivariant morphism,
or equivalently,
$(\upsigma,\uptheta)-$equivariant morphism
if $T\in Mor_{\ms{CA}^{\ast}}(\mc{A},\mc{B})$ and
$T\circ\upsigma(h)=\uptheta(h)\circ T$ for all $h\in H$.
$\lr{\mf{H},\uppi}{W}$ 
is a (nondegenerate) covariant representation of $\mf{A}$
if $\lr{\mf{H}}{\uppi}$ is a (nondegenerate) 
$\ast-$representation
of $\mc{A}$, $W$ is a strongly continuous unitary representation
of $H$ on $\mf{H}$ such that for all $h\in H$
\begin{equation*} 
\uppi\circ\upsigma(h)
=
\ms{ad}(W(h))\circ\uppi.
\end{equation*}
We denote by $Cov(\mf{A})$ the set of nondegenerate covariant representations of $\mf{A}$.
$\lr{\pf{H}}{W}$ is a cyclic covariant representation of $\mf{A}$ if $\pf{H}=\lr{\mf{H},\uppi}{\Upomega}$ 
is a cyclic representation of $\mc{A}$, $\lr{\mf{H},\uppi}{W}$ is a covariant representation of $\mf{A}$ and 
$W(H)\{\Upomega\}=\{\Upomega\}$.
\par
Let
$\upvarphi\in\ms{E}_{\mc{A}}^{H}(\upsigma)$
and 
$\pf{H}=\lr{\mf{H},\uppi}{\Upomega}$
be a cyclic representation of $\mc{A}$ associated
with 
$\upvarphi$.
Set
$\ms{W}_{\pf{H}}^{\upsigma}:
H\to\mc{L}(\mf{H})$ 
such that for all $h\in H$ and $a\in\mc{A}$
\begin{equation}
\label{09101627}
\ms{W}_{\pf{H}}^{\upsigma}(h)
\uppi(a)\Upomega
=\uppi(\upsigma(h)a)\Upomega.
\end{equation}
Then 
$\lr{\pf{H}}{\ms{W}_{\pf{H}}^{\upsigma}}$
is a cyclic covariant representation of $\mf{A}$
called the cyclic covariant representation of $\mf{A}$ induced by $\pf{H}$.
We convein to remove the index
$\upsigma$ whenever it does not cause confusion.
\par
Set
$\mc{U}(\mc{A})\coloneqq
\{U\in\mc{A}\mid U^{\ast}=U^{-1}\}$
provided by the group structure inherited
by the product on $\mc{A}$,
and $\mc{U}(\mf{H})\coloneqq\mc{U}(\mc{L}(\mf{H}))$
for any Hilbert space $\mf{H}$.
We say that 
$\upsigma$ is 
inner 
if there exists a group morphism 
$\ms{v}:H\to\mc{U}(\mc{A})$ such that 
$\upsigma=\ms{ad}\circ \ms{v}$,
in such a case we say that 
$\upsigma$ is inner implemented by 
$\ms{v}$,
or that
$\ms{v}$ implements unitarily
$\upsigma$.
$\mf{A}$ is said inner implemented by $\ms{v}$
if $\upsigma$ is so. 
If $\mc{A}$ is a von Neumann
algebra, it can be always considered
in standard form since \cite[$III.2.2.26$]{bla2}
and \cite[Def. $9.1.18$]{tak2}, 
called its canonical standard form,
then
by \cite[Thm $9.1.15$]{tak2} 
we deduce that
$\upsigma$ is inner, 
moreover
said $\mc{U}_{st}(\mc{A})$ the subgroup of $\mc{U}(\mc{A})$
whose elements $u$ satisfy
$uJu^{\ast}=J$ and 
$u L^{2}(\mc{A})_{+} =L^{2}(\mc{A})_{+}$, see 
\cite[Def. $9.1.18$]{tak2} for the notations,
there exists a 
unique group action
$V:H\to\mc{U}_{st}(\mc{A})$ 
inner implementing $\upsigma$.
\par 
Let $\mu\in\mc{H}(H)$ 
and let $\mc{C}_{c}^{\mu}(H,\mc{A})$ denote the $\ast-$algebra whose underlying linear space is 
$\mc{C}_{c}(H,\mc{A})$, while the product and involution are respectively $\ast^{\mu}$ and $\ast$ such that for all 
$f,g\in\mc{C}_{c}(H,\mc{A})$ and $s\in H$
(\cite[eqs.$2.16-2.17$]{will})
\begin{equation}
\label{09191116}
\begin{aligned}
(f\ast^{\mu}g)(s)
&\coloneqq
\int f(r)\upsigma(r)\bigl(g(r^{-1}s)\bigr)\,d\mu(r),
\\
f^{\ast}(s)
&\coloneqq
\Delta_{H}(s^{-1})\upsigma(s)(f(s^{-1})^{\ast}),
\end{aligned}
\end{equation}
where the integration is w.r.t. to the norm topology on $\mc{A}$.
Denote by $\mc{A}\rtimes_{\upsigma}^{\mu}H$
the $C^{\ast}-$crossed product of $\mc{A}$ by $H$ associated 
to $\mu$, see \cite[Lemma $2.27$]{will}.
It is defined as the $C^{\ast}-$algebra completion 
of the normed $\ast-$algebra
$
\pc{C}_{c}^{\mu}(H,\mc{A}) 
\coloneqq
\lr{\mc{C}_{c}^{\mu}(H,\mc{A})}
{\|\cdot\|^{\mu}}$,
where
$\|\cdot\|^{\mu}$ is the $\mu-$universal norm 
such that for any $f\in\mc{C}_{c}(H,\mc{A})$
(\cite[eq. $2.22$ and Lm. $2.31$]{will})
\begin{equation*}
\|f\|^{\mu}
=
\sup\bigl\{
\|(\uppi\rtimes^{\mu}\ms{u})(f)\|_{\mc{L}(\mf{H})}
\mid
\lr{\mf{H},\uppi}{\ms{u}}
\in Cov(\mf{A})
\bigr\},
\end{equation*}
where (\cite[Prp. $2.23$]{will})
\begin{equation}
\label{05171307}
(\uppi\rtimes^{\mu}\ms{u})(f)
\coloneqq
\int\uppi(f(s))\ms{u}(s)\,d\mu(s).
\end{equation}
Here the integral is valued in the locally convex space $\mc{L}_{s}(\mf{H})$, 
i.e. it is an element of $\mc{L}(\mf{H})$ and the integration is w.r.t. the strong operator topology on $\mc{L}(\mf{H})$.
Its existence is ensured by Rmk. \ref{09171300} and by the fact that 
the product is continuous as a map on
$\mc{L}_{s}(\mf{H})\times \mc{L}_{s}(\mf{H})_{1}$
with values in $\mc{L}_{s}(\mf{H})$,
where
$\mc{L}_{s}(\mf{H})_{1}$
is the topological subspace of
$\mc{L}_{s}(\mf{H})$ 
of its elements with norm less or equal to $1$.
We call the following
$L_{\mu}^{1}-$norm 
\begin{equation*}
\mc{C}_{c}(H,\mc{A})
\ni 
f\mapsto
\int\|f(s)\|\,d\mu(s).
\end{equation*}
\par
It is worthwhile a remark about notations.
It is well-known that the Haar measure is unique up to a constant
factor \cite[Th. $1$ $\S1$ $n^{\circ}2$]{IntBourb}, 
nevertheless in this work, 
at difference with the standard usage, 
we prefer to mention expressly which Haar measure we use in (\ref{09191116},\ref{05171307})
and a fortiori in $\mc{A}\rtimes_{\upsigma}^{\mu}H$.
\par
Since \cite[Prp. $2.39.$]{will}
if
$\lr{\mf{H},\uppi}{W}$ is a 
covariant representation 
of $\mf{A}$, 
then
$\uppi\rtimes^{\mu}W$
extends uniquely by continuity
to a 
$\ast-$representation
of $\mc{A}\rtimes_{\upsigma}^{\mu}H$
which is nondegenerate
if it is so $\lr{\mf{H},\uppi}{W}$.
\par
Let $G$ and $F$ be two topological groups, 
$\uprho:F\to Aut_{\ms{Gr}}(G)$ a group homomorphism
such that the map $(g,f)\mapsto\uprho_{f}(g)$
on $G\times F$ at values in $G$, is continuous, 
where $Aut_{\ms{Gr}}(G)$ is the group of automorphisms of the group underlying $G$.
Thus we let
$G\rtimes_{\uprho}F$ 
denote
the 
external topological semi-direct product of 
$G$ and $F$ relative to $\uprho$, see
\cite[$III.19$]{BourGT}.
By definition for all
$(g_{1},h_{1}),(g_{2},h_{2})\in G\rtimes_{\uprho}F$
\begin{equation*}
(g_{1},h_{1})\cdot_{\uprho}(g_{2},h_{2}) 
\coloneqq
(g_{1}\cdot\uprho_{h_{1}}(g_{2}),h_{1}\cdot h_{2}).
\end{equation*}
Moreover $j_{1}:G\to G\rtimes_{\uprho}F$ and 
$j_{2}:F\to G\rtimes_{\uprho}F$ will be the 
(continuous) canonical injections.
\par
Since \cite[Prp. $14$, $I.66$]{BourGT}
$G\rtimes_{\uprho}F$ 
is locally compact if and only if 
$G$ and $F$ are locally compact,
so in this case we can 
consider a dynamical system 
$\mf{A}=\lr{\mc{A}}{G\rtimes_{\uprho}F,\upsigma}$.
In addition let 
$F_{0}$ be a topological subgroup of $F$, 
$\upxi:\R\to G$ be a continuous group
homomorphism
and $l\in G\rtimes_{\uprho}F$,
set 
\begin{equation*}
\begin{aligned}
\uptau_{\upsigma}
&\coloneqq\upsigma\circ j_{1}
\\
\upgamma_{\upsigma}
&\coloneqq\upsigma\circ j_{2}
\\
\ms{S}_{F_{0}}^{G}
&\coloneqq
G\rtimes_{\uprho}F_{0}
\\
\upsigma_{F_{0}}
&\coloneqq
\upsigma\up
\ms{S}_{F_{0}}^{G}
\\
\uptau_{\upsigma}^{(l,\upxi)}
&\coloneqq
\uptau_{\upsigma}
\circ
\ms{ad}(l)
\circ
j_{1}
\circ
\upxi.
\end{aligned}
\end{equation*}
Here and thereafter we use the convention to denote $\uprho\up F_{0}$ simply by $\uprho$ anytime it is clear by the context 
w.r.t. which subset $F_{0}$ of $F$ the restriction has been performed.
$\uptau_{\upsigma}^{(l,\upxi)}$ is well-set,
since $j_{1}(G)$ is a normal subgroup of
$G\rtimes_{\uprho}F$,
moreover we have for all $h\in G\rtimes_{\uprho}F$ 
\begin{equation}
\label{09071855}
\uptau_{\upsigma}^{(l\cdot_{\uprho}h,\upxi)}
=
\ms{ad}(\upsigma(l))
\circ
\uptau_{\upsigma}^{(h,\upxi)}.
\end{equation}
Whenever it is clear by the context which 
dynamical system is involved we convein
to remove the index $\upsigma$.
If $F_{0}$ is locally compact,
for instance closed in $F$ 
by \cite[Prp. $13$, $I.66$]{BourGT},
then
$\ms{S}_{F_{0}}^{G}$
is locally compact
since \cite[Prp. $14$, $I.66$]{BourGT}.
Therefore
$\lr{\mc{A}}{\ms{S}_{F_{0}}^{G},\upsigma_{F_{0}}}$
is a dynamical system,
and the crossed product
$\mc{A}\rtimes_{\upsigma}^{\mu}
\ms{S}_{F_{0}}^{G}$
is well-set
for any $\mu\in\mc{H}(\ms{S}_{F_{0}}^{G})$.
For any domain $D$ in $\C$, i.e. open simply connected subset
of $\C$, 
denote by $H(D)$ the set of analytic maps on $D$ 
with values in $\C$, while by 
$H_{w}(D,\mc{A})$ and $H_{s}(D,\mc{A})$
the set of weak and strong analytic maps
on $D$ with values in $\mc{A}$.
Let $f:D\to\mc{A}$, then
by definition 
$f\in H_{w}(D,\mc{A})$ 
if 
$\upphi\circ f\in H(D)$ 
for all $\upphi\in\mc{A}^{\ast}$,
while
$f\in H_{s}(D,\mc{A})$ 
if $f$ is $\C-$derivable, thus differentiable.
Thus $H_{s}(D,\mc{A})\subseteq H_{w}(D,\mc{A})$.
Let
$\lr{\mc{A}}{\R,\upeta}$
be a dynamical system. 
Set $\updelta_{\upeta}$ to be the infinitesimal generator 
of the strongly continuous semigroup $\upeta\up\R^{+}$ acting 
on $\mc{A}$, i.e. $a\in Dom(\updelta_{\upeta})$ iff
the following limit exists in the norm topology of $\mc{A}$
\begin{equation}
\label{09081154}
\updelta_{\upeta}(a)
\coloneqq
\lim_{t\to 0,\,t\neq 0}
\frac{\upeta(t)(a)-a}{t}.
\end{equation}
$\mc{A}_{\upeta}$ is the $\ast-$subalgebra 
of $\mc{A}$ of the entire analytic elements
of $\upeta$, see \cite[Def. $2.5.20.$]{bra},
and
$\ov{\upeta}:\C\to\mc{A}^{\mc{A}_{\upeta}}$
denote the unique entire analytic extension of $\upeta$,
i.e. 
for any $a\in\mc{A}_{\upeta}$
the map
$\C\ni z\mapsto\ov{\upeta}(z)(a)\in\mc{A}$
belongs to
$H_{w}(\C,\mc{A})$,
so to $H_{s}(\C,\mc{A})$ since \cite[Prp. $2.5.21.$]{bra},
and it is the necessarily unique analytic extension of 
$\R\ni t\mapsto\upeta(t)(a)\in\mc{A}$.
The uniqueness follows by \cite[Def. $2.5.20.$]{bra}
and by
the uniqueness of entire analytic extension of numerical maps,
see \cite[Cor. of Thm. $10.18$]{Rudin2}.
We use the definition of $\beta-KMS$ given in
\cite[Def. $5.3.1$]{bra}, namely
for any $\beta\in\R$
the set $\ms{K}_{\beta}^{\upeta}$ 
of the $\beta-KMS$ states with respect to the dynamical system $\lr{\mc{A},\R}{\upeta}$ 
is the set of the states $\upomega\in\ms{E}_{\mc{A}}$
such that 
there exists an $\upeta-$invariant, norm dense 
$\ast-$subalgebra 
$\ms{D}$ of $\mc{A}$ 
such that 
$\upomega(a\ov{\upeta}(i\beta)b)=\upomega(ba)$
for all $a,b\in\ms{D}$.
If we replace $\ms{D}$ by $\mc{A}$ we obtain 
an equivalent statement, see \cite[Prp. $8.12.3$]{ped}.
Let the set of $\upeta-KMS$ states denote $\ms{K}_{-1}^{\upeta}$,
finally
$\ms{K}_{\infty}^{\upeta}$ is defined in \cite[Def. $5.3.18$]{bra}.
\subsection{Borelian functional calculus of possibly unbounded scalar type spectral operators in Banach spaces}
\label{07231648}
Let $S$ be a set, 
denote by $B(S)$ the Banach space 
of bounded complex maps on $S$ with
the norm $\|f\|=\sup_{s}|f(s)|$.
Let $\mc{B}$ be a field of subsets of $S$, 
a complex map defined on $S$ is $\mc{B}-$measurable if 
$f^{-1}(A)\in\mc{B}$, for any Borel set $A$ of $\C$.
Denote by
$TM(\mc{B})$ the closure in $B(S)$ of the linear subspace 
$\mc{I}(\mc{B})$
generated
by the set $\{\chi_{\delta}\mid\delta\in\mc{B}\}$, 
where $\chi_{\delta}$ is the characteristic map of the set $\delta$.
$TM(\mc{B})$ as a normed subspace of $B(S)$ is a Banach space, 
moreover the space of all bounded $\mc{B}-$measurable maps 
on $S$ is contained in $TM(\mc{B})$.
Let $X$ be a Banach space, 
a map $F:\mc{B}\to X$ 
is defined to be additive if 
$F(\varnothing)=\ze$ and
$F(\cup_{i=1}^{n}\sigma_{i})=\sum_{i=1}^{n}F(\sigma_{i})$,
for any $n\in\N$ and any family
$\{\sigma_{i}\}_{i=1}^{n}\subset\mc{B}$ of disjoint sets,
while $F$ is said to be bounded if 
$\sup_{\delta\in\mc{B}}\|F(\delta)\|<\infty$.
Let $F:\mc{B}\to X$ be a bounded additive map, 
thus we can define 
$\ms{I}^{F}:\mc{I}(\mc{B})\to X$ 
such that
\begin{equation}
\ms{I}^{F}
(\sum_{i=1}^{n}\lambda_{i}\chi_{\sigma_{i}})
\coloneqq
\sum_{i=1}^{n}\lambda_{i}F(\sigma_{i}),
\end{equation}
for any $n\in\N$, any family
$\{\sigma_{i}\}_{i=1}^{n}\subset\mc{B}$ 
of disjoint sets
and $\{\lambda_{i}\}_{i=1}^{n}\subset\C$. 
$\ms{I}^{F}$ is a well defined linear bounded operator, so
admits the linear extension by continuity to the space
$TM(\mc{B})$, we shall denote this extension again by 
$\ms{I}^{F}$ (\cite[$10.1$]{ds}).
If $Y$ is a Banach space and $\Psi\in\mc{L}(X,Y)$, then 
\begin{equation}
\label{11211647}
\Psi\circ\ms{I}^{F}=\ms{I}^{\Psi\circ F}.
\end{equation}
\par
Assume now in addition that $\mc{B}$ is a $\sigma-$field, 
let $G$ be a complex Banach space, 
$\Pr(G)$ the subset of $P\in\mc{L}(G)$ such that $PP=P$ and $\un_{G}$ and $\ze_{G}$
be the identity and the zero operator on $G$ respectively,
we convein to remove the index $G$ whenever it is clear by the 
context which space is involved.
$E$ is defined to be 
a countably additive spectral measure in $G$ on $\mc{B}$
if for all $\{\alpha_{n}\}_{n\in\N}\subset\mc{B}$ 
disjoint sets
and
$\delta_{1},\delta_{2}\in\mc{B}$
we have
\begin{enumerate}
\item
$E(\mc{B})\subseteq\Pr(G)$,
\item
$E(\delta_{1}\cap\delta_{2})=E(\delta_{1}) E(\delta_{2})$,
\item
$E(\delta_{1}\cup\delta_{2})
=
E(\delta_{1})+E(\delta_{2})-E(\delta_{1})E(\delta_{2})$,
\item
$E(S)=\un$,
\item
$E(\varnothing)=\ze$,
\item
$E(\bigcup_{n\in\N}\alpha_{n})=\sum_{n=1}^{\infty}E(\alpha_{n})$
w.r.t. the weak operator topology on $\mc{L}(G)$.
\label{weak}
\end{enumerate}
It results that the convergence in \eqref{weak} 
holds w.r.t. the strong operator topology on $G$ 
and $E$ is bounded (\cite[Cor $15.2.4$]{ds}),
hence $\ms{I}^{E}:\mc{B}\to\mc{L}(G)$ 
is well defined.
In case $S$ is a topological space and $\mc{B}$ is generated as $\sigma-$field by a basis of the topology on $S$ containing $S$, 
then $\mc{B}(S)\subseteq\mc{B}$ so we can set (\cite[Ch. $2$, pg. $126$]{ber})
\begin{equation*}
supp(E)
\coloneqq
\bigcap_{\{\delta\in\mc{B}\mid 
E(\delta)=\un,\,\delta\in Cl(S)\}}
\delta,
\end{equation*}
in particular $supp(E)$ is closed. 
By using the same argument to prove \cite[eq. $(1.5)$]{SilInt} we obtain
\begin{equation}
\label{supp1}
E(supp(E))=\un.
\end{equation}
We are now able to define the functional calculus 
(shortly $f.c.$)
associated
to a countably additive spectral measure.
Let 
\begin{itemize}
\item
$S$ be a set and $\mc{B}$ a $\sigma-$field of subsets of $S$;
\item
$E$ be a countably additive spectral
measure in $G$ on $\mc{B}$;
\item
$f$ a $\mc{B}-$measurable map;
\item
$f_{\sigma}\coloneqq f\chi_{\sigma}$, for all $\sigma\in\mc{B}$;
\item
$\delta_{n}\coloneqq[-n,n]$ 
and 
$f_{n}\coloneqq f_{|f|^{-1}(\delta_{n})}$, for all $n\in\N$.
\end{itemize}
Thus $f_{n}\in TM(\mc{B})$ being $\mc{B}-$measurable and
bounded, hence we are able to set 
(\cite[Def. $18.2.10$]{ds} and \cite[Def. $1.3.$]{SilInt})
\begin{equation}
\label{fc}
\begin{cases}
Dom(f(E))
\coloneqq
\{x\in G\mid
\lim_{n\in\N}\ms{I}^{E}(f_{n})x
\text{ exists}\},
\\
f(E)x
\coloneqq
\lim_{n\in\N}\ms{I}^{E}(f_{n})x,
\forall x\in Dom(f(E)).
\end{cases}
\end{equation}
The map $f\mapsto f(E)$ is called the functional calculus
of the spectral measure $E$ and $f(E)$ is a densely defined
closed linear operator in $G$ (\cite[Thm. $17.2.11$]{ds}).
It is easy to show that $f(E)E(\sigma)=(f\chi_{\sigma})(E)$,
for all $\sigma\in\mc{B}$.
In case $S$ is a topological space and
$\mc{B}$ is generated as $\sigma-$field by a basis of the topology
on $S$ containing $S$, then since \eqref{supp1} 
we obtain for any $f,g:S\to\C$ $\mc{B}-$measurable maps
\begin{equation}
\label{11271207}
f\up supp(E)
=
g\up supp(E)
\Rightarrow
f(E)=g(E).
\end{equation}
\par
A linear possibly unbounded operator $R$ in $G$ 
is called a scalar type spectral operator in $G$ if there exists
a countably additive spectral measure $F$ in $G$ on 
$\mc{B}(\C)$, 
such that $R=\imath(F)$, where $\imath$ is the identity map
on $\C$. We call $F$ a resolution of the identity (shortly $r.o.i$) of $R$.
There exists a unique $r.o.i$ of $R$ (\cite[Cor. $18.2.14$]{ds}) denoted by $\ms{E}_{R}$, 
moreover (\cite[Lemmas $18.2.13$ and $18.2.25$]{ds})
\begin{equation}
\label{18225}
sp(R)=supp(\ms{E}_{R}).
\end{equation}
Denote by $Bor(\C)$ the set of the Borelian maps, i.e.
$\mc{B}(\C)-$ measurable maps, thus for any $g\in Bor(\C)$
we convein to denote the operator $g(\ms{E}_{R})$
by $g(R)$ and call
the map $Bor(\C)\ni g\mapsto g(R)$ the Borelian functional calculus of $R$ (\cite[Def. $18.2.15$]{ds}).
\par
If $E$ is a countable additive spectral measure in $G$ on a $\sigma-$field $\mc{B}$ of subsets of a set $S$,
then for any $\mc{B}-$measurable map $f$ the operator $f(E)$ is a scalar type spectral operator in $G$ whose
resolution of the identity is (\cite[Thm. $18.2.17$]{ds})
\begin{equation}
\label{18217}
\ms{E}_{f(E)}
=
E\circ f^{-1}.
\end{equation}
\subsection{Joint RI constructed from $\{E_{x}\}_{x\in X}$}
\label{jri}
The material, and with some adaptation the notations, of the present section \ref{jri} 
arises from \cite[Ch. $2$, $\S 1.3$]{ber}.
Let $S$ be a set, $\mc{B}$ a $\sigma-$field of subsets of $S$,
$\mf{H}$ be a Hilbert space, then $E$ is defined to be a RI in $\mf{H}$ on $\mc{B}$
if it is a countably additive spectral measure in $\mf{H}$ on $\mc{B}$ 
such that $E(\sigma)$ is selfadjoint, for all $\delta\in\mc{B}$.
If $S$ is a topological space, a RI in $\mf{H}$ on $\mc{B}(S)$ is called a Borel RI in $\mf{H}$ on $S$.
Let $X$ be a set, $\mathscr{P}_{\omega}(X)$ the set of its finite parts,
$\ms{R}\doteq\{R_{x}\}_{x\in X}$,
where $R_{x}$ is a complete separable metric space for all $x\in X$.
$\{E_{x}\}_{x\in X}$ is defined to be a 
family of commuting Borel RI's in $\mf{H}$ on $\ms{R}$,
if $E_{x}$ is a Borel RI in $\mf{H}$ on $R_{x}$, for all $x\in X$,
and
$[E_{y}(\sigma_{y}),E_{z}(\sigma_{z})]=\ze$,
for all $y,z\in X$ and 
$\sigma_{q}\in\mc{B}(R_{q})$, $q\in\{y,z\}$.
We convein to avoid mentioning $\ms{R}$, whenever $R_{x}=\C$,
for all $x\in X$.
For any $Y\subseteq X$ let $R_{Y}$ denote $\prod_{x\in Y}R_{x}$ and if $Q=\{x_{1},\dots,x_{p}\}\subseteq X$, 
$R_{x_{1},\dots,x_{p}}$ denotes $R_{Q}$.
Set
\begin{equation*} 
\begin{cases}
\mc{C}(R_{X})
\coloneqq
\bigcup
\{\mc{C}(Q,\delta)
\mid
Q\in\mathscr{P}_{\omega}(X),
\delta\in\mc{B}(R_{Q})
\},
\\
\mc{C}(Q,\delta)
\coloneqq
\{
\uplambda\in R_{X}
\mid
\uplambda\up Q\in\delta
\},\,
\forall 
Q\in
\mathscr{P}_{\omega}(X),
\delta\in
\mc{B}(R_{Q}).
\end{cases}
\end{equation*}
$\mc{C}(R_{X})$ is a field of subsets of $R_{X}$,
called the field of cylindrical sets, 
$\mc{C}_{\sigma}(R_{X})$ denotes the $\sigma-$field generated
by $\mc{C}(R_{X})$.
Let $A\in\mathscr{P}(R_{X})$,
we recall that $\mc{B}(R_{A})$ is the $\sigma-$field
generated
by the set $\Psi(\prod_{x\in A}\mc{B}(R_{x}))$, 
where 
$\Psi:\prod_{x\in A}\mc{B}(R_{x})\to\mathscr{P}(R_{A})$
is such that 
$\Psi(\ps{\delta})\coloneqq\prod_{x\in A}\ps{\delta}_{x}$,
for all $\ps{\delta}\in\prod_{x\in A}\mc{B}(R_{x})$.
Let 
$\mc{E}\doteq
\{E_{x}\}_{x\in X}$ be a family of commuting Borel RI's in $\mf{H}$ on $\ms{R}$, 
then there exists a unique 
RI
in $\mf{H}$ on $\mc{C}_{\sigma}(R_{X})$, 
denoted by 
$\varprod_{x\in X}E_{x}$
and called the joint RI constructed from $\mc{E}$,
such that for all $Q\in\mathscr{P}_{\omega}(X)$ and 
$\ps{\delta}\in\prod_{y\in Q}\mc{B}(R_{y})$ 
we have (\cite[Ch. $2$, $\S 1.3$, Thm. $1.3$]{ber})
\begin{equation} 
\label{prod}
(\varprod_{x\in X}E_{x})
(\mc{C}(Q,
\prod_{y\in Q}\ps{\delta}_{y}))
=
\prod_{y\in Q}E_{y}(\ps{\delta}_{y}),
\end{equation}
where 
$\prod_{y\in Q}E_{x}(\ps{\delta}_{y})
\coloneqq
E_{x_{1}}(\ps{\delta}_{x_{1}})
\circ\dots\circ
E_{x_{p}}(\ps{\delta}_{x_{p}})$, if 
$Q=\{x_{1},\dots, x_{p}\}$, well-defined since the commutativity
property.
Let $f:R_{X}\to S$ be a 
$(\mc{C}_{\sigma}(R_{X}),\mc{B})$ measurable map, 
then
$\ms{E}_{\mc{E}}^{f}$ 
denotes
$(\varprod_{x\in X}E_{x})\circ f^{-1}$ defined on $\mc{B}$,
and 
$\ms{E}_{\mc{E}}$
denotes $\ms{E}_{\mc{E}}^{Id}$ in case $S=R_{X}$ and $\mc{B}=\mc{C}_{\sigma}(R_{X})$. 
\subsection{Categories}
We recall that $\mc{U}$ and $\mc{U}_{0}$
are fixed universes such that $\mc{U}\in\mc{U}_{0}$,
$\ms{set}$ is the category of sets belonging to the universe $\mc{U}$, 
functions as morphisms with map composition.
Whenever we refer to a set unless the contrary is stated,
we mean an element of $\mc{U}$, moreover
for any structure $S$
whenever we refer to ``the set of the $S$'s'', we always mean the subset of 
those elements of $\mc{U}$ satisfying the axioms of $S$.
We recall that we let $\ms{Gr}$ be the category of groups and group morphisms with map composition, 
and $\ms{Ab}$ be the full subcategory of $\ms{Gr}$ of abelian groups.
For any set $X$ and any group $G$ by abuse of language 
we let $Mor_{\ms{set}}(X,G)$ denote also the object of $\ms{Gr}$ 
with pointwise composition, inversion and identity.
If the contrary is not stated, any diagram involving maps between sets of $\mc{U}$ has to be understood in $\ms{set}$.
Let $A$, $B$ and $C$ be categories. 
Let $Obj(A)$ denote the set of objects of $A$, we let $a\in A$ denote $a\in Obj(A)$. 
For any $x,y\in A$ let $Mor_{A}(x,y)$ be the set of morphisms of $A$ from $x$ to $y$, 
let $\un_{x}$ be the identity morphism of $x$, 
while $Inv_{A}(x,y)=\{f\in Mor_{A}(x,y)\,\vert\,(\exists g\in Mor_{A}(y,x))(f\circ g=\un_{y},g\circ f=\un_{x}))\}$ 
denotes the set of invertible morphisms from $x$ to $y$.
Set $Mor_{A}=\bigcup\{Mor_{A}(x,y)\,\vert\, x,y\in A\}$ and 
$Inv_{A}=\bigcup\{Inv_{A}(x,y)\,\vert\, x,y\in A\}$, while
$Aut_{A}(y)=Inv_{A}(y,y)$, for $y\in A$, and $Aut_{A}=\bigcup\{Aut_{A}(x)\,\vert\, x\in A\}$.
We shall consider $Aut_{A}(y)$ with its natural group structure.
For any $T\in Mor_{A}(x,y)$ we set $d(T)=x$ called domain or source of $T$
and $c(T)=y$ called codomain or target of $T$,
the composition on $Mor_{A}$ is always denoted by $\circ$.
Let $A^{op}$ denote the opposite category of $A$, \cite[p. $33$]{mcl}.
Let $\ms{Fct}(A,B)$ denote the category of functors from $A$ to $B$ and natural transformations provided by pointwise 
composition,
see \cite[$1.3$]{ks}, \cite[p.$40$]{mcl}, \cite[p. $10$]{bor}.
Let us identify any $F\in\ms{Fct}(A,B)$ with the couple $(F_{o},F_{m})$, where $F_{o}:Obj(A)\to Obj(B)$ 
called the object map of $F$ 
while $F_{m}:Mor_{A}\to Mor_{B}$, called the morphism map of $F$ is such that 
$F_{m}^{x,y}:Mor_{A}(x,y)\to Mor_{B}(F_{o}(x),F_{o}(y))$ 
where $F_{m}^{x,y}=F_{m}\up Mor_{A}(x,y)$ for all $x,y\in A$.
For any object $x$ and morphism $t$ of $A$
we always let $F(x)$ and $F(t)$ denote $F_{o}(x)$ and $F_{m}(t)$ respectively,
while we shall never use this convention for the maps $F_{o}$ and $F_{m}$.
Let $\circ:\ms{Fct}(B,C)\times\ms{Fct}(A,B)\to\ms{Fct}(A,C)$ denote the standard composition of functors 
\cite[Def. $1.2.10$]{ks},
\cite[p. $14$]{mcl}, for any $\upsigma\in\ms{Fct}(A,B)$, the identity morphism 
$\un_{\upsigma}$ of $\upsigma$ in the category 
$\ms{Fct}(A,B)$ 
is such that $\un_{\upsigma}(M)=\un_{\upsigma_{o}(M)}$, for all $M\in A$.
Let $\beta\ast\alpha\in Mor_{\ms{Fct}(A,C)}(H\circ F,K\circ G)$ be the Godement product between the 
natural transformations $\beta$ and $\alpha$, where $H,K\in\ms{Fct}(B,C)$ and $F,G\in\ms{Fct}(A,B)$ while 
$\beta\in Mor_{\ms{Fct}(B,C)}(H,K)$ and $\alpha\in Mor_{\ms{Fct}(A,B)}(F,G)$, see \cite[Prp. $1.3.4$]{bor} 
(or \cite[p. $42$]{mcl} where it is used the symbol $\circ$ instead of $\ast$). The product $\ast$ is associative, in addition
we have the following rule for all $\gamma\in Mor_{\ms{Fct}(A,B)}(H,L)$, $\alpha\in Mor_{\ms{Fct}(A,B)}(F,H)$, 
and $\delta\in Mor_{\ms{Fct}(B,C)}(K,M)$, $\beta\in Mor_{\ms{Fct}(B,C)}(G,K)$, see \cite[Prp. $1.3.5$]{bor}
\begin{equation}
\label{bor135}
(\delta\ast\gamma)\circ(\beta\ast\alpha)=(\delta\circ\beta)\ast(\gamma\circ\alpha).
\end{equation}
We have 
\begin{equation}
\label{20061403}
\beta\ast\un_{F}=\beta\circ F_{o},
\end{equation}
moreover $\un_{G}\ast\un_{F}=\un_{G\circ F}$, and $\un_{F}\circ\un_{F}=\un_{F}$.
For any two categories $\mf{D},\mf{F}$, any $\ms{a},\ms{b},\ms{c}\in\ms{Fct}(\mf{D},\mf{F})$, any
$\ms{T}\in\prod_{O\in\mf{D}}Mor_{\mf{F}}(\ms{a}(O),\ms{b}(O))$ and 
$\ms{S}\in\prod_{O\in\mf{D}}Mor_{\mf{F}}(\ms{b}(O),\ms{c}(O))$ 
we set $S\circ T\in\prod_{O\in\mf{D}}Mor_{\mf{F}}(\ms{a}(O),\ms{c}(O))$ such that $(S\circ T)(M)\coloneqq S(M)\circ T(M)$ 
for all
$M\in\mf{D}$.
If $S$ is a nonempty subset of $Obj(A)$ by abuse of language we indicate with 
$S$ also the full subcategory of $A$ whose object set is $S$.
Let $A$ be a subcategory of $B$ 
let $\ms{I}_{A\to B}\in\ms{Fct}(A,B)$ be the identity functor defined in the obvious way,
set $Id_{A}=\ms{I}_{A\to A}$. 
Let $M,P$ be categories, $\mc{F}$ be a functor from $M$ to $P$ and $N$ be a subcategory of $P$,
thus we let $\mc{F}^{-1}(N)$ be the possibly empty subcategory of $M$ such that 
\begin{equation}
\label{06131929}
\begin{aligned}
Obj(\mc{F}^{-1}(N))&\coloneqq\mc{F}_{o}^{-1}(Obj(N))
\\
Mor_{\mc{F}^{-1}(N)}&\coloneqq\mc{F}_{m}^{-1}(Mor_{N}).
\end{aligned}
\end{equation}
Clearly 
$\mc{F}\circ\ms{I}_{\mc{F}^{-1}(N)\to M}$ is a functor from $\mc{F}^{-1}(N)$ to $P$, which by abuse 
of language we also consider a functor from $\mc{F}^{-1}(N)$ to $N$.
We have the following equivalence of categories see \cite[(1.3.5)]{ks}
\begin{equation}
\label{06231849}
\ms{Fct}(\mf{K},\mf{D})^{op}
\simeq
\ms{Fct}(\mf{K}^{op},\mf{D}^{op});
\end{equation}
where the morphism map of the equivalence is the identity map while
the object map is 
$F\mapsto op^{\mf{D}}\circ F\circ op^{\mf{K}^{op}}$,
with $op^{X}$ is the contravariant functor from any category $X$ to $X^{op}$
whose object and morphism maps are the identity maps.
By abuse of language in what follows we shall not mention in the above map the functors
$op^{\mf{D}}$ and $op^{\mf{K}^{op}}$,
so as a matter of fact considering the equivalence \eqref{06231849} as an identity.
Finally 
let $\ms{BS}$ be the category of complex Banach spaces, linear bounded maps and map composition,
and let us recall that
$\ms{CA}^{\ast}$
is the category of $C^{\ast}-$algebras, $\ast-$homomorphisms and map composition.
\subsection{Multiplier algebra and the category of dynamical systems}
Let $Sd(\mf{H})$ be the set of the possibly unbounded selfadjoint 
operators in a Hilbert space $\mf{H}$, while let $U(\mf{H},\mf{K})$
be the set of the unitary operators from $\mf{H}$ to the Hilbert space $\mf{K}$.
Let $\mc{A}$ be a $C^{\ast}-$algebra, $\ms{X}$ be a Hilbert $\mc{A}-$module, \cite[Def. $2.8$]{rw}, 
$\mc{L}_{\mc{A}}(\ms{X})$ or simply $\mc{L}(\ms{X})$ the $C^{\ast}-$algebra of adjointable 
maps on $\ms{X}$, \cite[Def. $2.17$ and Prp. $2.21$]{rw},
and let $\mc{K}(\ms{X})$ be the algebra of compact operators on $\ms{X}$,
\cite[Def. $2.24$]{rw}.
Let $\mc{A}_{\mc{A}}$ denote the Hilbert $\mc{A}-$module 
associated with $\mc{A}$, \cite[Exm. $2.10$]{rw},
let $\mc{K}(\mc{A})$ denote $\mc{K}(\mc{A}_{\mc{A}})$,
while let
$\ms{M}(\mc{A})\coloneqq\mc{L}(\mc{A}_{\mc{A}})$
and
$\mf{i}^{\mc{A}}$ 
denote the multiplier algebra of $\mc{A}$
and 
the canonical
embedding of $\mc{A}$ into $\ms{M}(\mc{A})$
respectively, where
$\mf{i}^{\mc{A}}(a)(b)=ab$ for all $a,b\in\mc{A}$,
\cite[Def. $2.48$ and Exm. $2.43$]{rw}.
Let $\mc{B}$ be a $C^{\ast}-$algebra 
and $\upalpha:\mc{B}\to\mc{L}(\ms{X})$ be 
a $\ast-$homomorphism, 
$\upalpha$
is said nondegenerate if
$span\{\upalpha(b)x\mid b\in\mc{B},x\in\ms{X}\}$
is dense in $\ms{X}$, \cite[Def. $2.49$]{rw}.
$\mf{i}^{\mc{B}}$ is nondegenerate and injective
since any $C^{\ast}-$algebra admits an approximate
identity.
As a consequence of \cite[Prp. $2.50$]{rw} we have that if $\upalpha$ is nondegenerate then there exists a unique
$\ast-$homomorphism $\upalpha^{-}:\ms{M}(\mc{B})\to\mc{L}(\ms{X})$ such that 
\begin{equation}
\label{01081715}
\upalpha^{-}\circ\mf{i}^\mc{B}=\upalpha.
\end{equation}
$\upalpha^{-}$ is nondegenerate since it is so $\upalpha$,
thus $\upalpha^{-}(\un)=\un$. 
Since the double 
conjugate Hilbert space of any Hilbert space $\mf{H}$ equals $\mf{H}$,
we deduce by \cite[Exm. $2.27$]{rw} 
that $\mf{H}$ is a 
$\mc{K}(\ov{\mf{H}})-$Hilbert module with the same norm,
where $\ov{\mf{H}}$ is the
conjugate Hilbert space of $\mf{H}$, and $\mc{K}(\mf{K})$ is the
$C^{\ast}-$algebra of the compact operators on $\mf{K}$, for any
Hilbert space $\mf{K}$.
Therefore since \eqref{01081715} it follows that
if $\mc{R}$ is a nondegenerate representation of $\mc{B}$,
then $\mc{R}^{-}$ is
the unique extension of $\mc{R}$ to a representation of
$\ms{M}(\mc{B})$ such that 
\begin{equation}
\label{01081712}
\mc{R}^{-}\circ\mf{i}^{\mc{B}}=\mc{R}.
\end{equation}
Let $\pf{H}=\lr{\mf{H},\mc{R}}{\Upomega}$ be a cyclic 
representation of $\mc{B}$, 
thus
for all $c\in\ms{M}(\mc{B})$ and $b\in\mc{B}$
we deduce by \cite[proof of Prp. $2.50$]{rw} 
\begin{equation}
\label{12181442}
\mc{R}^{-}(c)\,\mc{R}(b)\Upomega=
\mc{R}((\mf{i}^{\mc{B}})^{-1}(c\, \mf{i}^{\mc{B}}(b)))\,\Upomega.
\end{equation}
$\pf{H}^{-}\doteq\lr{\mf{H},\mc{R}^{-}}{\Upomega}$
is a cyclic representation of $\ms{M}(\mc{B})$ since 
\eqref{01081712}
and since $\pf{H}$ is cyclic.
By the uniqueness of the extension to $\ms{M}(\mc{B})$
it follows that
if
$\lr{\mf{K},\mc{S}}{\Upomega}$
is a cyclic representation 
of $\ms{M}(\mc{B})$, 
then
$\lr{\mf{K},\mc{S}\circ \mf{i}^{\mc{B}}}{\Upomega}$
is a cyclic one of $\mc{B}$ such that 
\begin{equation}
\label{12181437}
(\mc{S}\circ \mf{i}^{\mc{B}})^{-}=\mc{S}.
\end{equation}
Let $\mf{A}=\lr{\mc{A}}{H,\upsigma}$ be a 
$C^{\ast}-$dynamical system, $\mu\in\mc{H}(H)$, 
and
$\mc{R}$ is 
a nondegenerate
$\ast-$representation
of 
$\mc{A}\rtimes_{\upsigma}^{\mu}H$
then there exists 
a nondegenerate
covariant representation 
$\lr{\mf{H},\uppi}{W}$ 
of $\mf{A}$
such that 
$\mc{R}=\uppi\rtimes^{\mu}W$.
In particular 
said 
$\mc{B}=\mc{A}\rtimes_{\upsigma}^{\mu}H$
\begin{equation}
\label{12191452}
\mc{R}^{-}\circ\mf{j}_{\mc{A}}^{\mc{B}}=\uppi,
\end{equation}
where $\mf{j}_{\mc{A}}^{\mc{B}}$ 
is the canonical embedding of $\mc{A}$
into $\ms{M}(\mc{B})$
such that 
$\mf{j}_{\mc{A}}^{\mc{B}}(a)(f)(l)=af(l)$,
for all $a\in\mc{A}$, $f\in\mc{C}_{c}(H,\mc{A})$ and $l\in H$.
$\mf{j}_{\mc{A}}^{\mc{B}}$ is nondegenerate since
$\mc{C}_{c}(H,\mc{A})$ provided by the $\sup-$norm equals $\mc{A}\hat{\otimes}\mc{C}_{c}(H)$
and since $\mf{i}^{\mc{A}}$ is nondegenerate.
We call 
$\lr{\mf{H},\uppi}{W}$ 
the covariant representation
of $\mf{A}$
associated with $\mc{R}$.
\par
The set of dynamical systems with symmetry group $H$, equivariant morphisms and map composition is a category denoted by 
$\ms{C}_{0}(H)$, see \cite[pg.$26$]{ght}.
For any $(\upsigma,\uptheta)-$equivariant morphism $T$ the map $f\mapsto T\circ f$
defined on $\pc{C}_{c}^{\mu}(H,\mc{A})$
extends uniquely by continuity
to a $\ast-$homomorphism $\ms{c}_{\upmu}(T)$
from
$\mc{A}\rtimes_{\upsigma}^{\upmu}H$
to
$\mc{B}\rtimes_{\uptheta}^{\upmu}H$,
moreover
$\ms{c}_{\upmu}(S\circ T)=
\ms{c}_{\upmu}(S)\circ\ms{c}_{\upmu}(T)$
for any dynamical system $\mf{C}$ and $(\mf{B},\mf{C})-$equivariant
morphism $S$.
Hence the functions
$\lr{\mc{A},H}{\upsigma}\mapsto\mc{A}\rtimes_{\upsigma}^{\upmu}H$
and 
$T\mapsto\ms{c}_{\upmu}(T)$
determine a functor from $\ms{C}_{0}(H)$ to the category of 
$C^{\ast}-$algebras and $\ast-$homomorphisms,
see \cite[pg.$26$]{ght} or \cite[Cor. $2.48$]{will}.
If $H_{0}$ is a locally compact subgroup of $H$ then any
$(\upsigma,\uptheta)-$equivariant morphism
is $(\upsigma\up H_{0},\uptheta\up H_{0})-$equivariant,
hence the map 
$\lr{\mc{A},H}{\upsigma}\mapsto\lr{\mc{A},H_{0}}{\upsigma\up H_{0}}$ and the identity map on $ Mor_{\ms{C}_{0}(H)}$
determine a functor from 
$\ms{C}_{0}(H)$ to $\ms{C}_{0}(H_{0})$.
In particular for any $\upnu\in\mc{H}(H_{0})$
\begin{equation}
\label{12010348}
\ms{c}_{\upnu}(T)\in Mor_{\ms{CA}^{\ast}}
(\mc{A}\rtimes_{\upsigma}^{\upnu}H_{0},
\mc{B}\rtimes_{\uptheta}^{\upnu}H_{0}).
\end{equation}
\subsection{$K_{0}-$theory for $C^{\ast}-$algebras}
In this section $\mc{A}$ denotes a $C^{\ast}-$algebra.
\paragraph{The $\ast-$algebra $\mathbb{M}_{\infty}(\mc{A})$}
Let $n\in\N_{0}$,
set $a^{n}\coloneqq\overbrace{a\cdot...\cdot a}^{n-\text{times}}$
and
$\mathbb{P}(\mc{A})\coloneqq\{p\in\mc{A}\mid p=p^{\ast}=p^{2}\}$.
Denote by $\mathbb{M}_{n}(\mc{A})$ the set
of the square matrices of order $n$ at elements
in $\mc{A}$, which is a $\ast-$algebra 
providing it by the operations
such that for all $i,j\in\{1,...,n\}$ and $t\in\C$
\begin{enumerate}
\item
$(M+N)_{ij}
\coloneqq
M_{ij}+N_{ij}$,
\item
$(t\cdot M)_{ij}
\coloneqq
t\cdot M_{ij}$,
\item
$(M\cdot N)_{ij}
\coloneqq
\sum_{s=1}^{n}
M_{is}\cdot N_{sj}$.
\item
$(M^{\ast})_{ij}=M_{ji}^{\ast}$.
\end{enumerate}
Let $\mc{H}$ the Hilbert space in which 
$\mc{A}$ acts faithfully, e.g. through the universal
representation, see \cite[Rmk. $4.5.8$]{kad}, so 
we can identify $\mc{A}$ as $C^{\ast}-$algebra
with its faithful image acting on $\mc{H}$
provided with the standard operator norm.
Next let $\mc{H}_{k}=\mc{H}$ for any $k=1,\dots,n$,
set
\begin{equation*}
\begin{cases}
\ms{o}:
\mathbb{M}_{n}(\mc{A})
\to
\mc{L}(\bigoplus_{s=1}^{n}\mc{H}_{s}),
\\
\ms{o}(M)(v)_{k}
\coloneqq
\sum_{j=1}^{n}M_{kj}v_{j},
\forall
M\in\mathbb{M}_{n}(\mc{A}),
v\in\bigoplus_{s=1}^{n}\mc{H}_{s},
k=1,\dots,n,
\end{cases}
\end{equation*}
it is possible to show 
(\cite[pg. $147$]{kad})
that
$\ms{o}$ is a $\ast-$isomorphism.
Hence
the following map 
define a norm 
\begin{equation*}
\|\cdot\|_{\mathbb{M}_{n}(\mc{A})}:
\mathbb{M}_{n}(\mc{A})
\ni
M\mapsto
\|\ms{o}(M)\|,
\end{equation*}
on $\mathbb{M}_{n}(\mc{A})$ for which 
it is a $C^{\ast}-$algebra since
isometric to 
$\mc{L}(\bigoplus_{s=1}^{n}\mc{H}_{s})$.
\par
Let $\mc{A}\odot\mathbb{M}_{n}(\C)$
be the algebraic tensor product of
the $\ast-$algebras $\mc{A}$ and
$\mathbb{M}_{n}(\C)$, which is a $\ast-$algebra
providing the product to be defined by 
$(a\otimes b)\cdot(c\otimes d)
\coloneqq
(a\cdot c)\otimes(b\cdot d)$.
For any $i,j\in\{1,\dots,n\}$, 
set
$e_{ij}\in\mathbb{M}_{n}(\C)$
such that 
$(e_{ij})_{rs}\coloneqq\delta_{ir}\delta_{js}$,
for all 
$r,s\in\{1,\dots,n\}$.
It is possible to show (\cite[Prp. $T.5.20$]{weg}) that
the following map
\begin{equation*}
\begin{cases}
\mc{A}\odot\mathbb{M}_{n}(\C)
\to
\mathbb{M}_{n}(\mc{A}),
\\
\sum_{i,j=1}^{n}a_{ij}\otimes e_{ij}
\mapsto
A,
\\
A_{rs}=a_{rs},\,\forall r,s\in\{1,...,n\},
\end{cases}
\end{equation*}
is a well-defined $\ast-$isomorphism, 
called 
the standard isomorphism 
between 
$\mc{A}\odot\mathbb{M}_{n}(\C)$
and
$\mathbb{M}_{n}(\mc{A})$.
Moreover since 
\cite[Ex. $11.1.5$]{kad}, the fact that 
$\ms{o}$ is an isometry of $\mathbb{M}_{n}(\mc{A})$
onto its image through $\ms{o}$
and that the universal representation of $\mc{A}$
is faithful, 
we can state that the previous map 
is an isometry
between 
$\mathbb{M}_{n}(\mc{A})$,
and
the spatial tensor product 
(see \cite[pg. $847$]{kad})
$\mc{A}\otimes\mathbb{M}_{n}(\C)$
of the $C^{\ast}$ algebras
$\mc{A}$ and $\mathbb{M}_{n}(\C)$.
\par
Set
$\Phi_{nm}^{\mc{A}}:\mathbb{M}_{n}(\mc{A})\to\mathbb{M}_{m}(\mc{A})$
such that
for any $m\in\N_{0}$, $m\geq n$
and $A\in\mathbb{M}_{n}(\mc{A})$
\begin{equation*}
\begin{cases}
\Phi_{nm}^{\mc{A}}(A)_{ij}
\coloneqq A_{ij},\, i,j\in\{1,...,n\},
\\
\Phi_{nm}^{\mc{A}}(A)_{ij}
\coloneqq\ze,\,i,j\in\{n+1,...,m\},
\end{cases}
\end{equation*}
clearly we have
$\Phi_{mk}
\circ
\Phi_{nm}
=
\Phi_{nk}$,
for all $k,m,n\in\N_{0}$
such that
$k\geq m\geq n$,
moreover 
$\Phi_{nm}$ is a $\ast-$isometry into its image.
We call 
\begin{equation*}
\mathbb{M}_{\infty}(\mc{A})
\coloneqq
\varinjlim(\mathbb{M}_{s}(\mc{A}),\Phi_{rs}^{\mc{A}}),
\end{equation*}
the \emph
{normed inductive limit} 
of the system 
$\{(\mathbb{M}_{n},\Phi_{nm})\}_{n,m\in\N_{0},m\geq n}$
and it is the normed $\ast-$algebra constructed as follows.
Let
$U\doteq\bigcup_{n\in\N_{0}}\mathbb{M}_{n}(\mc{A})$
and $\lambda:U\to\N_{0}$ such that
$a\in\mathbb{M}_{\lambda_{a}}(\mc{A})$,
for all $a\in U$.
Define $\simeq$ on $U$ such that
for any $a,b\in U$
\begin{equation*}
a\simeq b
\Leftrightarrow
(\exists p\in\N_{0})
(p\geq\lambda_{a},
p\geq\lambda_{b},
\Phi_{\lambda_{a}p}^{\mc{A}}(a)=\Phi_{\lambda_{b}p}^{\mc{A}}(b)),
\end{equation*}
denote by $\langle a\rangle$
the equivalence set of $a\mod\simeq$.
Next the following
\begin{equation*}
\|\langle a\rangle\|_{\infty}
\coloneqq\|a\|_{\mathbb{M}_{\lambda(a)}(\mc{A})}
\end{equation*}
is a well-dedifed norm on $U\slash\simeq$,
indeed let $a,b\in U$ such that $a\simeq b$, 
then there exists a $p\in\N_{0}$ such that
$
\Phi_{\lambda_{a}p}^{\mc{A}}
(a)
=
\Phi_{\lambda_{b}p}^{\mc{A}}(b))
$,
so
$\|a\|_{\mathbb{M}_{\lambda(a)}(\mc{A})}=
\|b\|_{\mathbb{M}_{\lambda(b)}(\mc{A})}$ since 
$\Phi_{\lambda_{a}p}^{\mc{A}}$ and 
$\Phi_{\lambda_{b}p}^{\mc{A}}$
are isometries.
Set by abuse of language 
\begin{equation*}
\begin{cases}
\mathbb{M}_{\infty}(\mc{A})
\coloneqq
\bigcup_{n\in\N_{0}}\mathbb{M}_{n}(\mc{A})\slash\simeq,
\\
f_{n}^{\mc{A}}:\mathbb{M}_{n}(\mc{A})\to
\mathbb{M}_{\infty}(\mc{A}),
a\mapsto\langle a\rangle,\,n\in\N_{0}.
\end{cases}
\end{equation*}
We obtain
\begin{equation}
\label{10251307}
\begin{cases}
\mathbb{M}_{\infty}(\mc{A})
=
\bigcup_{n\in\N_{0}}
f_{n}(\mathbb{M}_{n}(\mc{A})),
\\
f_{m}\circ\Phi_{nm}=f_{n},m\geq n.
\end{cases}
\end{equation}
Set $\ze_{\infty}^{\mc{A}}\coloneqq f_{n}(\ze_{n}^{\mc{A}})$ 
for some $n\in\N_{0}$,
where $\ze_{n}^{\mc{A}}$ is the matrix of order $n$ at elements
in $\mc{A}$ all equal to $\ze$,
well-defined by the second equality in \eqref{10251307}.
Let $a,b\in\mathbb{M}_{\infty}(\mc{A})$, 
then since the second equality in \eqref{10251307}
there exist $p\in\N_{0}$ and $a,b\in\mathbb{M}_{p}(\mc{A})$
such that 
$x=f_{p}(a)$ and $y=f_{p}(b)$,
the following
\begin{equation*}
\begin{aligned}
x+y
&
\coloneqq f_{p}(a+b),
\\
x\cdot y
&\coloneqq
f_{p}(a\cdot b),
\\
t\cdot x
&\coloneqq
f_{p}(t\cdot a),\,t\in\C
\\
x^{\ast}
&\coloneqq
f_{p}(a^{\ast}),
\end{aligned}
\end{equation*}
are well-defined operations making
$\lr{\mathbb{M}_{\infty}(\mc{A}),+,\cdot,\ze_{\infty}^{\mc{A}}}
{\cdot,\ast,\|\cdot\|_{\infty}}$
a (non unital) normed $\ast-$algebra, 
denoted again by $\mathbb{M}_{\infty}(\mc{A})$.
$\|\cdot\|_{\infty}$ will be called the standard
norm in $\mathbb{M}_{\infty}(\mc{A})$.
For any $n\in\N_{0}$, since $\mc{A}$ is identifiable
with a two side ideal of $\mc{A}^{+}$, 
we deduce that
$\mathbb{M}_{n}(\mc{A})$ and $\mathbb{M}_{\infty}(\mc{A})$ 
is identifiable with a two side ideal of 
$\mathbb{M}_{n}(\mc{A}^{+})$
and
$\mathbb{M}_{\infty}(\mc{A}^{+})$
respectively, 
thus we can consider 
$a\equiv b\mod\mathbb{M}_{\infty}(\mc{A})$,
for any  $a,b\in\mathbb{M}_{\infty}(\mc{A}^{+})$.
Moreover for any $m\in\N_{0}$, $m\geq n$
\begin{equation}
\label{10271904}
\begin{cases}
f_{n}^{\mc{A}}=f_{n}^{\mc{A}^{+}}\up\mathbb{M}_{n}(\mc{A}),
\\
\Phi_{nm}^{\mc{A}}=\Phi_{nm}^{\mc{A}^{+}}\up\mathbb{M}_{n}(\mc{A}).
\end{cases}
\end{equation}
\paragraph{The group $K(\mc{A})$}
Let $n\in\N_{0}$, define 
$diag_{n}^{\mc{A}}:
\mathbb{M}_{n}(\mc{A})
\times
\mathbb{M}_{n}(\mc{A})
\to
\mathbb{M}_{2n}(\mc{A})$
such that for any $a,b\in\mathbb{M}_{n}(\mc{A})$
\begin{equation*}
diag_{n}(a,b)
\coloneqq
\begin{pmatrix}
a&\ze_{n}\\
\ze_{n}&b.
\end{pmatrix}
\end{equation*}
Define in $\mathbb{P}(\mc{A})$ the following equivalences
(\cite[Def. $4.2.1$, Prp. $4.6.3$]{bla}).
Let $p,q\in\mathbb{P}(\mc{A})$ thus
\begin{description}
\item[algebraic equivalence]
$p\sim q$ iff there exists $x,y\in\mc{A}$ such that 
$xy=p$ and $yx=q$,
\item[similarity]
$p\sim_{s}q$ iff there exists an invertible element $z$ 
in $\tilde{A}$ such that $zpz^{-1}=q$,
\item[homotopy]
$p\sim_{h}q$ iff there exists a continuous path 
of projections in $\mc{A}$ from $p$ to $q$,
i.e. a norm-continuous map $f:[0,1]\to\mc{A}$ such that 
$f(t)\in\mathbb{P}(\mc{A})$, for all $t\in[0,1]$
and
$f(0)=p$ and $f(1)=q$.
\end{description}
Denote by $[e]$ the equivalence class
of $e\mod\sim$.
For any 
$e,f\in\mathbb{P}(\mathbb{M}_{\infty}(\mc{A}))$
it follows (\cite[Ch. $II$, Sez. $4$]{bla})
\begin{equation}
\label{10271922}
e\sim f
\Leftrightarrow
e\sim_{s}f
\Leftrightarrow
e\sim_{h}f.
\end{equation}
Set
\begin{equation}
\label{10271921}
\begin{cases}
V(\mc{A})
\coloneqq
\lr{\mathbb{P}(\mathbb{M}_{\infty}(\mc{A}))\slash\sim}{+},
\\
[e]+[f]\coloneqq[e'+f'],\,
\forall e,f\in\mathbb{P}(\mathbb{M}_{\infty}(\mc{A})),
\\
e'\in[e], f'\in[f], e'\perp f',
\end{cases}
\end{equation}
(\cite[Def. $5.1.2$ and comments following it]{bla}),
where $a\perp b$ iff $a\cdot b=\ze$ for any 
$a,b\in\mathbb{P}(\mc{A})$.
The operation $+$ is indipendent by the choice of $e'$ and $f'$
which there exist as showed below.
Below we convein to denote 
$\ze_{\infty}^{\mc{A}^{+}}$ by $\ze_{\infty}$
$f_{n}^{\mc{A}^{+}}$ by $f_{n}$,
$\Phi_{np}^{\mc{A}^{+}}$ 
by 
$\Phi_{np}$
and $diag_{n}^{\mc{A}^{+}}$ by $diag_{n}$,
for any $n,p\in\N_{0}$, $p\geq n$. 
Let
$e,f\in\mathbb{P}(\mathbb{M}_{\infty}(\mc{A}))$
then there exist $m\in\N_{0}$, $a,b\in\mathbb{M}_{m}(\mc{A})$ 
such that $e=f_{m}(a)$ and $f=f_{m}(b)$.
Note that $\Phi_{m(2m)}(c)=diag_{m}(c,\ze_{m})$, $c\in\{a,b\}$,
thus 
$e=f_{2m}(diag_{m}(a,\ze_{m}))$
and
$f=f_{2m}(diag_{m}(b,\ze_{m}))$,
since (\ref{10271904}-\ref{10251307}).
Define $g\in\mathbb{M}_{2m}(\mc{A}^{+})$ as
\begin{equation*}
\begin{cases}
g\coloneqq
\begin{pmatrix}
\ze_{m}&\un_{m}
\\
\un_{m}&\ze_{m}
\end{pmatrix},
\\
z\coloneqq f_{2m}(g),
\\
f'\coloneqq f_{2m}\bigl(diag_{m}(\ze_{m},b)\bigr).
\end{cases}
\end{equation*}
Thus $z=z^{\ast}=z^{-1}$, since $g=g^{\ast}=g^{-1}$, and
\begin{equation*}
\begin{aligned}
zfz^{-1}&=zfz
\\
&=f_{2m}\bigl(g\, diag_{m}(b,\ze_{m})\, g\bigr)=f',
\end{aligned}
\end{equation*}
moreover $f'\in\mathbb{P}(\mathbb{M}_{2m}(\mc{A}^{+}))$
thus
$f'\in[f]$ since \eqref{10271922}.
Finally
$f'\cdot e=f_{2m}\bigl(
diag_{m}(\ze_{m},b)\cdot diag_{m}(a,\ze_{m})\bigr)
=\ze_{\infty}$
and the third sentence of \eqref{10271921}
follows.
Define the relation $\thickapprox$ on 
$V(\mc{A})\times V(\mc{A})$
such that for all 
$p,q,r,s\in\mathbb{P}(\mathbb{M}_{\infty}(\mc{A}))$
\begin{equation*}
([p],[q])\thickapprox([r],[s])
\Leftrightarrow
(\exists z\in\mathbb{P}(\mathbb{M}_{\infty}(\mc{A})))
([p]+[s]+[z]
=
[q]+[r]+[z]),
\end{equation*}
$\thickapprox$ is an equivalence relation,
let us denote by $[([p],[q])]$ the equivalence class
of $([p],[q])\mod\thickapprox$. 
So we can define
for any 
$([p],[q]),([a],[b])\in V(\mc{A})\times V(\mc{A})$
and
for some $([r],[r])\in V(\mc{A})\times V(\mc{A})$
\begin{equation*}
\begin{cases}
K_{00}(\mc{A})
\coloneqq
(V(\mc{A})\times V(\mc{A}))\slash\thickapprox,
\\
[([p],[q])]+[([a],[b])]
\coloneqq
[([p]+[a],[q]+[b])],
\\
\ze
\coloneqq
[([r],[r])],
\\
\ms{K}_{00}(\mc{A})
\coloneqq
\lr{K_{00}(\mc{A})}{+,\ze}.
\end{cases}
\end{equation*}
It is possible to show that 
$\ms{K}_{00}(\mc{A})$ is a well-defined
commutative group where if we denote by $-[([p],[q])]$
the inverse of $[([p],[q])]$, we have 
$-[([p],[q])]=[([q],[p])]$
(\cite[Appendix $G$]{weg}).
Finally set 
\begin{equation*}
\begin{aligned}
K_{0}(\mc{A})
&\coloneqq
\bigl\{[([p],[q])]\in
K_{00}(\mc{A}^{+})
\mid
p\equiv q\mod\mathbb{M}_{\infty}(\mc{A})
\bigr\},
\\
\ms{K}_{0}(\mc{A})
&\coloneqq
\lr{K_{0}(\mc{A})}
{+\up(K_{0}(\mc{A})\times K_{0}(\mc{A})),\ze},
\end{aligned}
\end{equation*}
then $\ms{K}_{0}(\mc{A})$ is a subgroup of 
$\ms{K}_{00}(\mc{A}^{+})$
(\cite[Ch. $III$, Sec. $5.5$]{bla}).
$V$ is the object part of a functor from the category of $\ast-$algebras and $\ast-$homomorphisms
to the category $Sg$ of commutative semigroups and semigroup morphisms,
while $\ms{K}_{00}$ and $\ms{K}_{0}$ are the object part of functors 
from the category of $\ast-$algebras and $\ast-$homomorphisms
to the category $\ms{Ab}$.
More exactly for any couple of $\ast-$algebras $\mc{A}$ and $\mc{B}$ and $\ast-$homomorphism $\upalpha:\mc{A}\to\mc{B}$, 
we have that
\begin{equation}
\label{10291120}
\begin{cases}
\upalpha_{\#}\in Mor_{Sg}(V(\mc{A}),V(\mc{B})),
\\
\upalpha_{\ast}\in 
Mor_{\ms{Ab}}(\ms{K}_{00}(\mc{A}),\ms{K}_{00}(\mc{B})),
\\
\upalpha_{\star}\in 
Mor_{\ms{Ab}}(\ms{K}_{0}(\mc{A}),\ms{K}_{0}(\mc{B})),
\end{cases}
\end{equation}
and
$(\upbeta\circ\upalpha)_{\bullet}
=\upbeta_{\bullet}\circ\upalpha_{\bullet}$,
for any 
$\ast-$algebra $\mc{C}$,
$\ast-$homomorphism $\beta:\mc{B}\to\mc{C}$
and
$\bullet\in\{\#,\ast,\star\}$.
\par
Here for any 
$a\in\bigcup_{n\in\N_{0}}\mathbb{M}_{n}(\mc{A})$
such that 
$\langle a\rangle\in\mathbb{P}(\mathbb{M}_{\infty}(\mc{A}))$,
and
$r,s\in\mathbb{P}(\mathbb{M}_{\infty}(\mc{A}^{+}))$
such that 
$r\equiv s\mod\mathbb{M}_{\infty}(\mc{A})$
\begin{equation*}
\begin{cases}
\upalpha_{\#}([\langle a\rangle])
=
[\langle\upalpha\circ a\rangle],
\\
\upalpha_{\star}
\bigl([([r],[s])]\bigr)
=
\Bigl[\bigl((\upalpha^{+})_{\#}([r]),
(\upalpha^{+})_{\#}([s])\bigr)\Bigr],
\end{cases}
\end{equation*}
while
for any $p,q\in\mathbb{P}(\mathbb{M}_{\infty}(\mc{A}))$
\begin{equation}
\label{10281552}
\upalpha_{\ast}
\bigl([([p],[q])]\bigr)
=
\Bigl[\bigl(\upalpha_{\#}([p]),
\upalpha_{\#}([q])\bigr)\Bigr],
\end{equation}
(\cite[Prp. $6.1.3$ and Prp. $6.2.4$]{weg}).
If $\mc{A}$ is
unital then $\ms{K}_{00}(\mc{A})$ is isomorphic as a group
to $\ms{K}_{0}(\mc{A})$, so we shall use the 
convention to identify $\ms{K}_{0}(\mc{A})$ 
with
$\ms{K}_{00}(\mc{A})$ whenever $\mc{A}$ is unital,
called here as in \cite{bla} 
the standard picture of $\ms{K}_{0}(\mc{A})$
(\cite[Prp. $5.5.5$]{bla}).
Let $H$ be a locally compact group, $\upmu\in\mc{H}(H)$,
$\mf{A}=\lr{\mc{A},H}{\upsigma}$ and 
$\mf{B}=\lr{\mc{B},H}{\uptheta}$ be two objects of $\ms{C}_{0}(H)$
and $T$ be a $(\mf{A},\mf{B})-$equivariant morphism.
Set
\begin{equation*}
\ms{k}_{\upmu}(T)\coloneqq(\ms{c}_{\upmu}(T)^{+})_{\ast},
\end{equation*}
thus 
$\ms{k}_{\upmu}(T):
\ms{K}_{0}((\mc{A}\rtimes_{\upeta}^{\upmu}H)^{+})
\to
\ms{K}_{0}((\mc{B}\rtimes_{\uptheta}^{\upmu}H)^{+})$
is a group morphism,
in addition the functions 
$\lr{\mc{A},H}{\upeta}\mapsto
\ms{K}_{0}((\mc{A}\rtimes_{\upeta}^{\upmu}H)^{+})$
and
$T\mapsto\ms{k}_{\upmu}(T)$ define
a functor from $\ms{C}_{0}(H)$ to $\ms{Ab}$.
In particular 
for any 
locally compact subgroup $H_{0}$ of $H$ and $\upnu\in\mc{H}(H_{0})$
since the discussion prior \eqref{12010348} we have
\begin{equation}
\label{11291802}
\ms{k}_{\upnu}(T):
\ms{K}_{0}((\mc{A}\rtimes_{\upeta}^{\upnu}H_{0})^{+})
\to
\ms{K}_{0}((\mc{B}\rtimes_{\uptheta}^{\upnu}H_{0})^{+}).
\end{equation}
\subsection{The Chern-Connes character}
Let $\mc{A}$ and $\mc{B}$ be unital $\C^{\ast}-$algebras.
$\lr{\cdot}{\cdot}_{\mc{A}}:
\ms{K}_{0}(\mc{A})\times H_{\ep}^{ev}(\mc{A})\to\C$
denote
the pairing as defined in 
\cite[$IV.7.\delta$, Thm. $21$]{connes}.
If
$T:\mc{A}\to\mc{B}$ is a unit preserving
$\ast-$homomorphism, 
$\upphi=\{\upphi_{2n}\}_{n\in\N}$ 
is an entire even normalized cocycle on $\mc{B}$
and $e$ an idempotent in
$\mathbb{M}_{\infty}(\mc{A})$,
then since \cite[$IV.7.\delta$, Thm. $21$, Lemma $20$]{connes},
we obtain 
\begin{equation}
\label{12011443}
\lr{(T)_{\ast}([e])}{[\upphi]}_{\mc{B}}
=
\lr{[e]}{T_{\dagger}([\upphi])}_{\mc{A}}.
\end{equation}
Here 
$T_{\dagger}:H_{\ep}^{ev}(\mc{B})\to H_{\ep}^{ev}(\mc{A})$ 
denotes the map
$[\{\upphi_{2n}\}_{n\in\N}]
\mapsto
[\{\upphi_{2n}\circ T^{[2n]}\}_{n\in\N}]$, $[\upphi]$ 
is the class in $H_{\ep}^{ev}(\mc{B})$ corresponding to $\upphi$, and 
$T^{[N]}:\mc{A}^{N}\to\mc{B}^{N}$ such that 
$\Pr_{i}\circ T^{[N]}=T\circ\Pr_{i}$ for all $i\in\{1,\dots N\}$.
Let $\ms{ch}(\mc{A},\uppi,\ms{D},\Upgamma)\in H_{\ep}^{ev}(\mc{A})$, 
also denoted by 
$\ms{ch}\lr{\mc{R}}{\ms{D},\Upgamma}$
where $\mc{R}=\lr{\mc{A}}{\uppi}$, 
denote the Chern-Connes character, 
\cite[$IV.8.\delta$, Def. $17$]{connes}, of the even $\theta-$summable 
$K-$cycle $(\mc{A},\uppi,\ms{D},\Upgamma)$,
\cite[$IV.2.\gamma$, Def. $11$ and $IV.8$, Def. $1$]{connes}.
It is well-known that 
the class corresponding to
the JLO cocycle \cite{jlo}, 
associated with any even $\theta-$summable $K-$cycle $(\mc{A},\uppi,\ms{D},\Upgamma)$
via \cite[$IV.8.\epsilon$, Thm. $21$]{connes},
belongs to $\ms{ch}(\mc{A},\uppi,\ms{D},\Upgamma)$,
\cite[$IV.8.\epsilon$, Thm. $22$]{connes}.
Hence one can use the JLO cocycle for computing
the pairing
$\lr{\ms{x}}{\ms{ch}(\mc{A},\uppi,\ms{D},\Upgamma)}_{\mc{A}}$.
An important result concerning the Connes character is that
$\lr{\ms{x}}{\ms{ch}(\mc{A},\uppi,\ms{D},\Upgamma)}_{\mc{A}}\in\Z$
for all $\ms{x}\in\ms{K}_{0}(\mc{A})$,
\cite[$IV.8.\delta$, Thm. $19$ and Prp. $18$]{connes}.
In this work whenever we refer to an even $\theta-$summable $K-$cycle $(\mc{A},\uppi,\ms{D},\Upgamma)$, 
we assume that 
$\ms{D}\neq\ze$ which is automatically true in case $\dim\mf{H}=\infty$, with $\mf{H}$ the Hilbert space 
where $\uppi$ acts.
\part{Preliminaries in Equivariance}
\label{07301047}
In the present part we fix more specific terminology and collect properties of equivariance we employ 
in the paper. In our knowledge all the results in this part are \textbf{new}, however they are relatively 
not difficult to obtain by applying general well-known facts. 
Because of this we use again the term Preliminaries to title this part.
\par
In section \ref{equivfc} we establish that the Borelian functional calculus 
in a Banach space is equivariant under an isometric
action. 
Then we apply this result to ensure the equivariance
in case the spectral measure involved is the 
joint resolution of identity constructed by a family of 
commuting Borel resolutions of the identity in a Hilbert space.
Then we provide a sufficient condition to ensure selfadjointness.
In section \ref{kmseq} we display the equivariance of the $KMS-$states
under the conjugate of surjective equivariant morphisms.
Section \ref{05061338} is devoted to state
the equivariance under action of $H$ of representations 
of $C^{\ast}-$crossed products for different symmetry groups.
In section \ref{multiplier}
we prove for any state of a $C^{\ast}-$algebra the existence of a canonical extention to the multiplier algebra,
then for any unital dynamical system, with $\mc{A}$ its underlying algebra,  
the result is used to associate a state of $\mc{A}$ with any state of the crossed 
product $\mc{A}\rtimes_{\upsigma}H$.
\section{Equivariance}
\label{equiv}
\subsection{Equivariance of the functional calculus in Banach spaces}
\label{equivfc}
We prove in Thm. \ref{11251210}  
that the Borelian functional calculus 
in a Banach space is equivariant under an isometric
action. We apply this result to ensure 
in Thm. \ref{11252000}
the equivariance
when the spectral measure $\mc{E}$ involved is 
the joint $RI$ constructed by a family of 
commuting Borel $RI$'s in a Hilbert space.
In Thm. \ref{11241732} we provide a sufficient condition
concerning the Borelian map $f$ to ensure the selfadjointness of the 
operator $f(\mc{E})$.
In this section we assume fixed a topological space $S$,
a Banach space $G$
and a Hilbert space $\mf{H}$.
\begin{lemma}
\label{01251721}
Let $\ms{E}$ be a countably
spectral measure in $G$ on a $\sigma-$field
$\mc{R}$ of subsets of a set $Z$, and $\sigma,\delta\in\mc{R}$
such that $\delta\supseteq\sigma$ and $\ms{E}(\sigma)=\un$.
Then $\ms{E}(\delta)=\un$.
\end{lemma}
\begin{proof}
$Z=\sigma\cup\complement\sigma$ implies 
$\un=\ms{E}(\sigma)+\ms{E}(\complement\sigma)$,
so $\ms{E}(\complement\sigma)=\ze$.
Moreover
$\delta=
\delta\cap(\sigma\cup\complement\sigma)
=\sigma\cup(\delta\cap\complement\sigma)$
hence
$\ms{E}(\delta)=
\ms{E}(\sigma)+
\ms{E}(\delta)
\ms{E}(\complement\sigma)
=\un$.
\end{proof}
\begin{proposition}
\label{23112100}
Let $\ms{E}$ be a countably additive spectral measure in $G$ on a $\sigma-$field $\mc{R}$ of subsets of $S$ such that 
$\mc{R}$ is generated as $\sigma-$field by a basis of the topology on $S$ containing $S$. 
Thus
\begin{equation*}
supp(\ms{E})
=
\bigcap_{\{\sigma\in\mc{R}\mid\,\ms{E}(\sigma)=\un\}}
\ov{\sigma}.
\end{equation*}
\end{proposition}
\begin{proof}
The inclusion $\supseteq$ follows since \eqref{supp1} and since $supp(\ms{E})$ is closed.
Let $\sigma\in\mc{R}$ such that $\ms{E}(\sigma)=\un$,
then $\ms{E}(\ov{\sigma})=\un$ since Lemma \ref{01251721}
and $\ov{\sigma}\in\mc{R}$, thus
$\{\delta\in\mc{R}\mid\ms{E}(\delta)=1,\,\delta\in Cl(S)\}
\supseteq
\{\ov{\sigma}\mid\sigma\in\mc{R},\,\ms{E}(\sigma)=1\}$
and the inclusion $\subseteq$ follows.
\end{proof}
\begin{proposition}
\label{11232010}
Let $Z$ be a set, 
$\mc{R}$ be a $\sigma-$field of subsets of $Z$,
$\ms{E}$ a countably additive spectral measure on $\mc{R}$
and $f:Z\to S$ be a $(\mc{R},\mc{B}(S))-$measurable map.
Thus
\begin{enumerate}
\item
$supp(\ms{E}\circ f^{-1})\subseteq
\bigcap_{\{\sigma\in\mc{R}\mid\ms{E}(\sigma)=\un\}}
\ov{f(\sigma)}$;
\label{st1}
\item
if $Z$ is a topological space and $\mc{R}$ is generated as $\sigma-$field by a basis of the topology on $Z$ containing $Z$, 
then $supp(\ms{E}\circ f^{-1})\subseteq\ov{f(supp(\ms{E}))}$;
\label{st2}
\item
if in addition to the conditions in st. \eqref{st2} $f$ is continuous, then $supp(\ms{E}\circ f^{-1})=\ov{f(supp(\ms{E}))}$.
\label{st3}
\end{enumerate}
\end{proposition}
\begin{proof}
By Prp. \ref{23112100} we have
\begin{equation}
\label{01251743}
supp(\ms{E}\circ f^{-1})
=
\bigcap_{\{\delta\in\mc{B}(S)\mid\ms{E}(f^{-1}(\delta))=\un\}}
\ov{\delta}.
\end{equation}
Let $\sigma\in\mc{R}$ thus $\ov{f(\sigma)}\in\mc{B}(S)$ and   $f^{-1}(\ov{f(\sigma)})\supseteq\sigma$,
so $\ms{E}(\sigma)=\un$ implies $\ms{E}(f^{-1}(\ov{f(\sigma)})=\un$ since Lemma \ref{01251721}, and st.\eqref{st1} follows by 
\eqref{01251743}.
St.\eqref{st2} follows by st.\eqref{st1} and \eqref{supp1}. 
Let $f$ be continuous, set $\mc{B}_{f}\doteq
\{\delta\in\mc{B}(S)\mid\ms{E}(f^{-1}(\delta))=\un\}$,
$\mc{R}_{f}\doteq f^{-1}(\mc{B}_{f})$,
and
$A_{\sigma}\doteq\ov{f(\sigma)}$
and
$B_{\delta}\doteq\ov{\delta}$, 
for all $\sigma\in\mc{R}_{f}$ and $\delta\in\mc{B}_{f}$.
Thus
$A_{f^{-1}(\delta)}\subseteq B_{\delta}$, 
since 
$f(f^{-1}(\delta))\subseteq\delta$, for all $\delta\in\mc{B}_{f}$,
therefore
$\bigcap_{\{\delta\in\mc{B}_{f}\}}
A_{f^{-1}(\delta)}
\subseteq
\bigcap_{\{\delta\in\mc{B}_{f}\}}
B_{\delta}
$, i.e.
\begin{equation}
\label{star1}
\bigcap_{\{\sigma\in\mc{R}_{f}\}}\ov{f(\sigma)}
\subseteq
\bigcap_{\{\delta\in\mc{B}_{f}\}}
\ov{\delta}.
\end{equation}
Moreover
\begin{equation}
\label{star2}
\begin{aligned}
f(supp(\ms{E}))
&
=  
f\bigl(
\bigcap_{\{\sigma\in\mc{R}\mid\ms{E}(\sigma)=\un,
\sigma\in Cl(Z)\}}
\sigma
\bigr)
\\
&\subseteq
\bigcap_{\{\sigma\in\mc{R}\mid\ms{E}(\sigma)=\un,
\sigma\in Cl(Z)\}}
f(\sigma)
\\
&\subseteq
\bigcap_{\{\sigma\in\mc{R}\mid\ms{E}(\sigma)=\un
\}}
f(\ov{\sigma})
\\
&\subseteq
\bigcap_{\{\sigma\in\mc{R}_{f}\}}
f(\ov{\sigma})
\\
&\subseteq
\bigcap_{\{\sigma\in\mc{R}_{f}\}}
\ov{f(\sigma)}.
\end{aligned}
\end{equation}
Here the second inclusion follows by Lemma \ref{01251721},
the forth by the continuity of the map $f$.
Finally \eqref{star2}, \eqref{star1} 
and \eqref{01251743}
imply 
$f(supp(\ms{E}))\subseteq supp(\ms{E}\circ f^{-1})$
and st.\eqref{st3} follows by st.\eqref{st2}.
\end{proof}
\begin{corollary}
\label{11241219}
Let $\ms{R}\doteq\{R_{x}\}_{x\in X}$ 
be a family of complete separable metric
spaces, 
$\mc{E}\doteq\{E_{x}\}_{x\in X}$ a family of 
commuting Borel RI's in $\mf{H}$ on $\ms{R}$,
and $f:R_{X}\to S$ be 
$(\mc{C}_{\sigma}(R_{X}),\mc{B}(S))-$measurable.
Then for any $Q\in\mathscr{P}_{\omega}(X)$
\begin{equation*}
supp(\ms{E}_{\mc{E}}^{f})
\subseteq
\ov{f(\mc{C}(Q,\prod_{x\in Q}supp(E_{x})))}.
\end{equation*}
\end{corollary}
\begin{proof}
By Prp. \ref{11232010}\eqref{st1} and \eqref{prod}.
\end{proof}
\begin{remark}
\label{11241723}
Let $\mc{B}$ be a $\sigma-$field of subsets of a set $Z$
and $E$ a RI in $\mf{H}$ on $\mc{B}$.
By following the argument in the proof of \cite[Thm. $12.2.6.(d)$]{ds}
with the help of \cite[Thm. $18.2.11.(i)$]{ds},
we obtain $h(E)^{\ast}=h^{\ast}(E)$,
for all $\mc{B}-$measurable map $h$,
where $h^{\ast}$ is the complex conjugate map of $h$. 
\end{remark}
\begin{lemma}
\label{star}
Let $\mc{B}$ be a $\sigma-$field of subsets of a set $Z$,
$E$ a countably additive spectral measure in $G$ on $\mc{B}$
and $\sigma\in\mc{B}$
such that $E(\sigma)=\un$.
Let $f$ and $g$ be two $\mc{B}-$measurable maps,
thus $f\up\sigma=g\up\sigma$
implies
$f(E)=g(E)$.
\end{lemma}
\begin{proof}
Let $\delta\in\mc{B}$, 
then
$(f\chi_{\delta})(E)=f(E)E(\delta)$
since 
\cite[Thm. $18.2.11(f)$]{ds} and the fact that
$\chi_{\delta}(E)=E(\delta)$.
Moreover 
$f\up\delta=g\up\delta$ 
implies
$f\chi_{\delta}=g\chi_{\delta}$
and the statement follows. 
\end{proof}
The following result yields sufficient conditions to ensure
the selfadjointness of the operator $f(\ms{E}_{\mc{E}})$.
\begin{remark}
\label{32st1}
Let $\ms{R}\doteq\{R_{x}\}_{x\in X}$ be a family of complete separable metric spaces, $\mc{E}\doteq\{E_{x}\}_{x\in X}$ 
a family of commuting Borel RI's in $\mf{H}$ on $\ms{R}$, and $f:R_{X}\to\C$ be $\mc{C}_{\sigma}(R_{X})-$measurable, then
since \eqref{18217} we obtain that
$\ms{E}_{\mc{E}}^{f}=\ms{E}_{f(\ms{E}_{\mc{E}})}$
i.e.
$\ms{E}_{\mc{E}}^{f}$ 
is the resolution of the identity of 
$f(\ms{E}_{\mc{E}})$.
\end{remark}
\begin{theorem}
[Selfadjointness]
\label{11241732}
Let $\ms{R}\doteq\{R_{x}\}_{x\in X}$ be a family of complete separable metric spaces, $\mc{E}\doteq\{E_{x}\}_{x\in X}$ a family of 
commuting Borel RI's in $\mf{H}$ on $\ms{R}$, and $f:R_{X}\to\C$ be $\mc{C}_{\sigma}(R_{X})-$measurable. Then
\begin{enumerate}
\item
for all $Q\in\mathscr{P}_{\omega}(X)$ we have
\begin{equation*}
sp(f(\ms{E}_{\mc{E}}))
\subseteq 
\ov{f(\mc{C}(Q,\prod_{x\in Q}supp(E_{x})))},
\end{equation*}
\label{32st2}
\item
if there exists an 
$A\in\mathscr{P}_{\omega}(X)$
such that 
$\ov{f(\mc{C}(A,\prod_{x\in A}supp(E_{x})))}\subseteq\R$
then
$f(\ms{E}_{\mc{E}})$ is selfadjoint.
\label{32st3}
\end{enumerate}
\end{theorem}
\begin{proof}
Since Rmk. \ref{32st1} and \eqref{18225}
\begin{equation}
\label{01261450}
sp(f(\ms{E}_{\mc{E}}))=supp(\ms{E}_{\mc{E}}^{f}),
\end{equation}
then st.\eqref{32st2} follows since Cor. \ref{11241219}.
Since Rmk. \ref{32st1} we have
\begin{equation}
\label{01260848}
f(\ms{E}_{\mc{E}})=\imath(\ms{E}_{\mc{E}}^{f}),
\end{equation}
moreover since \eqref{01261450} and \eqref{supp1}
\begin{equation}
\label{01261453}
\ms{E}_{\mc{E}}^{f}(sp(f(\ms{E}_{\mc{E}})))=\un.
\end{equation}
Let $h$ be the $0-$extension to $\C$ of $\imath\up sp(f(\ms{E}_{\mc{E}}))$, then since
\eqref{01261453}, Lemma \ref{star} and \eqref{01260848} 
we obtain
\begin{equation}
\label{01260850}
f(\ms{E}_{\mc{E}})=h(\ms{E}_{\mc{E}}^{f}).
\end{equation}
If there exists an $A\in\mathscr{P}_{\omega}(X)$
such that 
$\ov{f(\mc{C}(A,\prod_{x\in A}supp(E_{x})))}\subseteq\R$
then
$h=h^{\ast}$ since st.\eqref{32st2},
thus st.\eqref{32st3} follows since Rmk. \ref{11241723}
and \eqref{01260850}.
\end{proof}
\begin{proposition}
\label{11270642}
Let $U_{r}:\R^{+}\to\mc{L}(\mf{H})$
be a strongly continuous semigroup of unitary operators
on $\mf{H}$, for $r\in\{1,2\}$, such that 
$[U_{1}(t),U_{2}(s)]=\ze$, for all $s,t\in\R^{+}$.
Denote by $A_{r}$ the infinitesimal generator of $U_{r}$,
then 
$[\ms{E}_{-iA_{1}}(\sigma),\ms{E}_{-iA_{2}}(\delta)]=\ze$, for all
$\sigma,\delta\in\mc{B}(\C)$.
\end{proposition}
\begin{proof}
In this proof set $E_{r}\doteq\ms{E}_{-iA_{r}}$, for $r\in\{1,2\}$. 
Since the Stone Thm., see for example 
\cite[Thm. $12.6.1.$]{ds} and its proof,
we have $U_{r}(t)=f_{t}(B_{r})$, where 
$B_{r}\doteq -iA_{r}$ and 
$f_{t}:\C\ni\lambda\mapsto\exp(it\lambda)$,
for all $t\in\R^{+}$.
Hence, since \eqref{18217},
$U_{r}(1)$ is a scalar type spectral operator
whose $r.o.i.$ is 
$\ms{E}_{U_{r}(1)}=E_{r}\circ f_{1}^{-1}$.
Let $g_{t}:\C-\{0\}\ni\lambda\mapsto\frac{-i}{t}\ln\lambda$,
for any $t\in\R^{+}-\{0\}$, then $g_{t}\circ f_{t}=Id$ so
\begin{equation}
\label{01261559}
E_{r}=\ms{E}_{U_{r}(1)}\circ g_{1}^{-1}.
\end{equation}
Moreover since \cite[Lemma $18.2.13$ and Cor. $18.2.4$]{ds} we deduce for all $\delta\in\mc{B}(\C)$ that 
$[U_{1}(1),\ms{E}_{U_{2}(1)}(\delta)]=\ze$, so by applying again \cite[Lemma $18.2.13$ and Cor. $18.2.4$]{ds}
we obtain $[\ms{E}_{U_{1}(1)}(\sigma),\ms{E}_{U_{2}(1)}(\delta)]=\ze$, for all $\sigma,\delta\in\mc{B}(\C)$
and the statement follows by \eqref{01261559}.
\end{proof}
\begin{proposition}
\label{11251148}
Let $E$ be a countably additive spectral measure in $G$ on 
a $\sigma-$field of subsets of a set $Z$, $F$ a Banach space and 
$\upalpha
\in
\mc{L}(\mc{L}_{w}(G),\mc{L}_{w}(F))$ 
morphism of
unital algebras.
Then
\begin{enumerate}
\item
$\upalpha\circ E$ is a countably additive spectral measure in $F$
on $\mc{B}$;
\label{11251148st1} 
\item
if $Z$ is a topological space and $\mc{B}$ is generated as a $\sigma-$field by a basis for the topology of $Z$ containing $Z$,
then $supp(\upalpha\circ E)\subseteq supp(E)$, in addition $supp(\upalpha\circ E)= supp(E)$ if $\upalpha^{-1}(\{\un\})=\{\un\}$.
\label{11251148st2} 
\end{enumerate}
\end{proposition}
\begin{proof}
Trivial.
\end{proof}
Now we can state the first main result of 
this part 
namely the equivariance of the general functional calculus, result that shall be 
applied to ensure in Thm. \ref{11252000} the covariance of the f.c. for of a commuting set of resolutions of the identity.
\begin{theorem}
[Equivariance of the functional calculus]
\label{11251210}
Let $E$ be a countably additive spectral measure in $G$ on 
a $\sigma-$field $\mc{B}$ of subsets of a set $Z$, 
$F$ a Banach space and $U:G\to F$ be a linear isometry.
Then
\begin{enumerate}
\item
$ad(U)\circ E$ is a countably additive spectral measure
in $F$ on $\mc{B}$;
\label{1210st1}
\item
if $Z$ is a topological space and $\mc{B}$ is generated as a 
$\sigma-$field by a basis for the topology of $Z$ containing $Z$,
then 
$supp(E)=supp(ad(U)\circ E)$;
\label{1210st2}
\item
for any $\mc{B}-$measurable map $f$ we have
\begin{equation}
\label{11251436}
\begin{cases}
Uf(E)U^{-1}
=
f(ad(U)\circ E),
\\
sp(Uf(E)U^{-1})=sp(f(E)),
\\
\ms{E}_{f(ad(U)\circ E)}=ad(U)\circ E\circ f^{-1}.
\end{cases}
\end{equation}
\end{enumerate}
\end{theorem}
\begin{proof}
St. \eqref{1210st1}\,\&\,\eqref{1210st2} follow since Prp. \ref{11251148}\eqref{11251148st1}\,\&\,\eqref{11251148st2}   
and since $ad(U)\in\mc{L}(\mc{L}_{w}(G),\mc{L}_{w}(F))$ and it is a morphism of unital algebras.
Let $f$ be $\mc{B}-$measurable and set
$E^{U}\doteq ad(U)\circ E$. Thus
\begin{equation*}
\begin{cases}
Dom(f(E^{U}))
=
\{y\in G\mid\exists\,\lim_{n\in\N}\ms{I}^{E^{U}}(f_{n})y\},
\\
f(E^{U})y=\lim_{n\in\N}\ms{I}^{E^{U}}(f_{n})y,\,
y\in Dom(f(E^{U})).
\end{cases}
\end{equation*}
Moreover $\ms{I}^{E^{U}}=ad(U)\circ\ms{I}$, since \eqref{11211647},
thus since $ad(U)^{-1}=ad(U^{-1})$
\begin{equation}
\label{10302053}
Dom(f(E^{U}))
=
U Dom(f(E)),
\end{equation}
in particular $Dom(f(E^{U})U)=Dom(f(E))$, and for any $y\in Dom(f(E))$
\begin{equation*}
f(E^{U})Uy=\lim_{n\in\N}U\,\ms{I}^{E}(f_{n})y=U\lim_{n\in\N}\ms{I}^{E}(f_{n})y=Uf(E)y,
\end{equation*}
so $f(E^{U})U=Uf(E)$ and the first equality in \eqref{11251436} follows.
Let $T$ be a scalar type spectral operator in $G$, since Prp. \ref{11251148}\eqref{11251148st2} we have
\begin{equation*}
supp(\ms{E}_{T}^{U})=supp(\ms{E}_{T}).
\end{equation*}
Moreover by the first equality in \eqref{11251436}
\begin{equation*}
UTU^{-1}=\imath(\ms{E}_{T}^{U}),
\end{equation*}
i.e. $UTU^{-1}$ is a scalar type spectral operator in $F$
such that
\begin{equation}
\label{01262044}
\ms{E}_{T}^{U}=\ms{E}_{UTU^{-1}},
\end{equation} 
hence
\begin{equation*}
supp(\ms{E}_{UTU^{-1}})=supp(\ms{E}_{T}),
\end{equation*}
therefore by \eqref{18225} we obtain
\begin{equation}
\label{01262119}
sp(UTU^{-1})=sp(T).
\end{equation}
Hence the second equality in \eqref{11251436} 
follows with the position $T=f(E)$, well-set since $f(E)$ is a scalar type spectral operator by \cite[Lemma $18.2.17$]{ds}.
With this position by \eqref{01262044} and the first equlity in \eqref{11251436} follows $\ms{E}_{f(E)}^{U}=\ms{E}_{f(ad(U)\circ E)}$,
then the third equality in \eqref{11251436} follows by \eqref{18217}.
\end{proof}
\begin{corollary}
\label{10301830}
Let $F$ be a Banach space $U:G\to F$ a linear isometry,
$T$ a scalar type spectral operator in $G$ and $f$ a Borelian
map. Thus
\begin{enumerate}
\item
$UTU^{-1}$ is a scalar type spectral operator in $F$ such that
$\ms{E}_{UTU^{-1}}=ad(U)\circ\ms{E}_{T}$;
\label{10301830st1}
\item
$sp(UTU^{-1})=sp(T)$;
\label{10301830st2}
\item
$f(UTU^{-1})=Uf(T)U^{-1}$;
\label{10301830st3}
\item
$\ms{E}_{f(UTU^{-1})}=ad(U)\circ\ms{E}_{T}\circ f^{-1}$.
\label{10301830st4}
\end{enumerate}
\end{corollary}
\begin{proof}
St.\eqref{10301830st1}\,\&\,\eqref{10301830st2} follows since \eqref{01262044}\,\&\,\eqref{01262119}, st.\eqref{10301830st3}\,\&\,\eqref{10301830st4}
follow by st.\eqref{10301830st1} and since the first and third equality in \eqref{11251436} applied for the position $E=\ms{E}_{T}$.
\end{proof}
\begin{definition}
\label{01262145}
Let $X$ be a set and $\ms{T}\doteq\{T_{x}\}_{x\in X}$ such that $T_{x}$ is a scalar type spectral operator in $G$ for all $x\in X$,
set $\mc{E}_{\ms{T}}\coloneqq\{\ms{E}_{T_{x}}\}_{x\in X}$.
\end{definition}
Now we can state the main result of this section 
\begin{theorem}
[Equivariance of the $f.c.$ associated with $\mc{E}_{\ms{T}}$]
\label{11252000}
Let $X$ be a set, $\mf{K}$ a Hilbert space, $U:\mf{H}\to\mf{K}$ a unitary operator, and $f$ a $\mc{C}_{\sigma}(\C^{X})-$measurable 
map. Moreover let $\ms{T}\doteq\{T_{x}\}_{x\in X}$ satisfy the following two properties:
$T_{x}$ is a scalar type spectral operator in $\mf{H}$ for all $x\in X$,
and
$\mc{E}_{\ms{T}}$ is a family of commuting Borel $RI$'s in $\mf{H}$.
Set $\ms{T}(U)\doteq\{UT_{x}U^{-1}\}_{x\in X}$, then
\begin{enumerate}
\item
$\ms{E}_{\mc{E}_{\ms{T}(U)}}
=
ad(U)\circ
\ms{E}_{\mc{E}_{\ms{T}}}
$;
\label{11252000st1}
\item
$f(\ms{E}_{\mc{E}_{\ms{T}(U)}})
=
Uf(\ms{E}_{\mc{E}_{\ms{T}}})U^{-1}$;
\label{11252000st2}
\item
$\ms{E}_{f(\ms{E}_{\mc{E}_{\ms{T}(U)}})}
=
ad(U)\circ\ms{E}_{\mc{E}_{\ms{T}}}\circ f^{-1}$.
\label{11252000st3}
\end{enumerate}
\end{theorem}
\begin{proof}
St.\eqref{11252000st1}
follows since Cor. \ref{10301830}\eqref{10301830st1}  
applied to any $T_{x}$, $x\in X$
and by the uniqueness in 
\cite[Thm. $1.3$ pg. $122$]{ber}.
St.\eqref{11252000st2} follows
since
st.\eqref{11252000st1}
and the first equality in \eqref{11251436}, 
while 
st.\eqref{11252000st3} follows
since the third equality in \eqref{11251436}
and 
st.\eqref{11252000st1}.
\end{proof}
\subsection{Equivariance of KMS-states}
\label{kmseq}
In Thm. \ref{11141627} and Cor. \ref{09081138} we prove the equivariance of the $KMS-$states 
under the conjugate of the action of surjective equivariant morphisms.
In what follows with approximate identity of a Banach $\ast-$algebra 
we mean a bounded order preserving approximate identity 
of positive elements bounded by $1$.
\begin{lemma}
\label{11191814pre}
Let $\mc{A}$ be a Banach $\ast-$algebra, $\mc{B}$ be $C^{\ast}-$algebra and $T$ be a
$\ast-$homomorphism 
from $\mc{A}$ to $\mc{B}$ such that $T(\mc{A})$ is norm dense.
If
$\{e_{\alpha}\}_{\alpha\in D}$ 
is an approximate identity of $\mc{A}$
then
$\{T(e_{\alpha})\}_{\alpha\in D}$
is an approximate identity of $\mc{B}$.
\end{lemma}
\begin{proof}
Let $\{e_{\alpha}\}_{\alpha\in D}$ 
be an approximate identity of $\mc{A}$.
Then 
$\|T(e_{\alpha})\|\leq 1$
and
$\alpha\leq\beta\Rightarrow T(e_{\alpha})\leq T(e_{\beta})$
for all $\alpha,\beta\in D$
since the positivity and the continuity of $T$ with $\|T\|\leq 1$, see \cite[Prp. $2.3.1$]{bra}.
Next
\begin{equation*}
\|T(e_{\alpha})T(a)-T(a)\|
=
\|T(e_{\alpha}a-a)\|
\leq
\|e_{\alpha}a-a\|,
\end{equation*}
for all $a\in\mc{A}$ and $\alpha\in D$,
so for all $b\in T(\mc{A})$
\begin{equation}
\label{11191838}
\lim_{\alpha\in D}
\|T(e_{\alpha})b-b\|=0.
\end{equation}
Let $\ep>0$ and $B\in\mc{B}$ thus there exists $b\in T(\mc{A})$
such that $\|b-B\|\leq\frac{\ep}{2}$,
and for all $\alpha\in D$
\begin{equation*}
\begin{aligned}
\|T(e_{\alpha})B-B\|
&\leq
\|T(e_{\alpha})(B-b)\|
+
\|T(e_{\alpha})b-b\|
+
\|b-B\|
\\
&\leq
2\|b-B\|+\|T(e_{\alpha})b-b\|.
\end{aligned}
\end{equation*}
Hence
$\limsup_{\alpha\in D}\|T(e_{\alpha})B-B\|\leq\ep$
and
$\liminf_{\alpha\in D}\|T(e_{\alpha})B-B\|\leq\ep$
for all $\ep>0$
since \cite[$IV.24(14)$, $IV.27(33)$, $IV.23(13)$]{BourGT} and \eqref{11191838}, 
then $\lim_{\alpha\in D}\|T(e_{\alpha})B-B\|=0$ since \cite[$IV.23$, Cor.$1$]{BourGT}
and the statement follows.
\end{proof}
\begin{remark}
\label{01201940}
Let us assume the hypotheses of Lemma \ref{11191814pre}.
If $\mc{A}$ is a $C^{\ast}-$algebra $T(\mc{A})$ is norm closed \cite[Thm. $9.5.12.(d)$]{palmer}, hence in this case the norm 
density of $T(\mc{A})$ is equivalent to the surjectivity of $T$.
While if $\mc{A}$ and $\mc{B}$ are unital algebras then $T$ is unit preserving since its continuity.
\end{remark}
\begin{lemma}
\label{11191814}
Let $\mc{A}$, $\mc{B}$ be $C^{\ast}-$algebras and $T$ be a surjective $\ast-$homomorphism from $\mc{A}$ to $\mc{B}$, 
then
$T_{\dagger}(\ms{E}_{\mc{B}})
\subseteq
\ms{E}_{\mc{A}}$.
\end{lemma}
\begin{proof}
Since $T_{\dagger}(\upomega)$ is positive for all $\upomega\in\ms{E}_{\mc{B}}$
the statement follows by Lemma \ref{11191814pre} and \cite[Prp. $2.3.11$]{bra}.
\end{proof}
\begin{lemma}
\label{11141536}
Let 
$\lr{\mc{A},H}{\upeta}$ and $\lr{\mc{B},H}{\uptheta}$
be dynamical systems 
and $T$ be a 
$(\upeta,\uptheta)-$equivariant 
morphism.
Then
\begin{enumerate}
\item
$T(\mc{A}_{\upeta})\subseteq\mc{B}_{\uptheta}$
and the for all $z\in\C$ the following diagram is commutative 
\begin{equation}
\label{11141536d}
\xymatrix{
\mc{B}_{\uptheta}
\ar[rr]^{\ov{\uptheta}(z)}
& &
\mc{B}
\\
& &
\\
\mc{A}_{\upeta}
\ar[uu]^{T}
\ar[rr]_{\ov{\upeta}(z)}
& &
\mc{A}
\ar[uu]_{T}
}
\end{equation}
\item
$T\circ\updelta_{\upeta}\subseteq\updelta_{\uptheta}\circ T$.
\end{enumerate}
\end{lemma}
\begin{proof}
Let $a\in\mc{A}_{\upeta}$ and $\ov{\upeta}^{a}$ denote the map $\C\ni z\mapsto\ov{\upeta}(z)(a)\in\mc{A}$, 
thus $\uptheta(t)(Ta)=T\upeta(t)(a)=(T\circ\ov{\upeta}^{a})(t)$ for all $t\in\R$.
Next $T\circ\ov{\upeta}^{a}$ is analytic being composition of two analytic maps, then $T(a)\in\mc{B}_{\uptheta}$
and by the uniqueness of the analytic extension, the commutativity of the diagram follows.
Let $b\in Dom(\updelta_{\upeta})$, so since the continuity of $T$
\begin{equation*}
T(\updelta_{\upeta}b)
=\lim_{t\to 0,t\neq 0}
\frac{T\upeta(t)b-Tb}{t}
=\lim_{t\to 0,t\neq 0}
\frac{\uptheta(t)Tb-Tb}{t}
=
\updelta_{\uptheta}(Tb).
\end{equation*}
\end{proof}
The next result states the equivariance of the $KMS-$states
under the conjugate of the action of surjective equivariant maps.
\begin{theorem}
\label{11141627}
Let 
$\lr{\mc{A},H}{\upeta}$
and
$\lr{\mc{B},H}{\uptheta}$
be dynamical systems 
and $T$ an $(\upeta,\uptheta)-$equivariant 
morphism such that $T$ is surjective or unit preserving in case $\mc{A}$ and $\mc{B}$ are unital. 
Then $T_{\dagger}(\ms{K}_{\beta}^{\uptheta})\subseteq\ms{K}_{\beta}^{\upeta}$ for all $\beta\in\widetilde{\R}$.
\end{theorem}
\begin{proof}
If $\beta\in\R$, $\upomega\in\ms{K}_{\beta}^{\uptheta}$
and $x,y\in\mc{A}_{\upeta}$ then
$\upomega(T(x)\ov{\uptheta}(i\beta)T(y))
=\upomega(T(y)T(x))$
since Lemma \ref{11141536}$(1)$,
thus the statement follows by \eqref{11141536d} and 
Lemma \ref{11191814}.
If $\beta=\infty$ the statement follows by the equivariance of
$T$, by \cite[Prp. $5.3.19(3)$]{bra} and Lemma \ref{11191814}.
\end{proof}
An alternative proof of Thm. \ref{11141627} for $\beta\in\R$
follows by Lemma \ref{11191814}, by the equivariance of $T$ and
by \cite[Prp. $5.3.7.(2)$]{bra}.
\begin{lemma}
\label{09061626}
Let $\lr{\mc{A}}{\R,\upeta}$ be a dynamical system
and
$\uplambda\in Aut_{\ms{CA}^{\ast}}(\mc{A})$.
Then
\begin{enumerate}
\item
$\uplambda(\mc{A}_{\upeta})
=
\mc{A}_{\ms{ad}(\uplambda)\circ\upeta}$;
\label{09061715a}
\item
$\ov{\ms{ad}(\uplambda)\circ\upeta} 
=
\ms{ad}(\uplambda)\circ\ov{\upeta}$;
\label{09061715b}
\item
$
\ms{ad}(\uplambda)(\updelta_{\upeta})
=
\updelta_{\ms{ad}(\uplambda)\circ\upeta}
$.
\label{09061715c}
\end{enumerate}
\end{lemma}
\begin{proof}
In this proof let $\upeta^{\uplambda}$ denote $\ms{ad}(\uplambda)\circ\upeta$,
which is clearly a strongly continuous one-parameter group acting on $\mc{A}$ by $\ast-$automorphisms.
$\uplambda$ is $\C-$differentiable since it is $\C-$linear and norm continuous, 
therefore by the chain rule of differentiable
maps and by \cite[Prp. $2.5.21.$]{bra}
we deduce that
$(z\mapsto\uplambda\circ\ov{\upeta}(z)(a))\in H_{s}(\C,\mc{A})$, 
for all $a\in\mc{A}_{\upeta}$.
Hence for all $c\in\uplambda(\mc{A}_{\upeta})$ 
\begin{equation}
\label{09071701}
\begin{aligned}
(z\mapsto\ms{ad}(\uplambda)\circ\ov{\upeta}(z)(c))
&\in 
H_{st}(\C,\mc{A}),
\\
\ms{ad}(\uplambda)\circ\ov{\upeta}\up\R
&=
\upeta^{\uplambda},
\end{aligned}
\end{equation}
and the inclusion $\subset$ of st.\eqref{09061715a} follows.
If we apply this result to the system 
$\lr{\mc{A}}{\R,\ms{ad}(\uplambda)\circ\upeta}$ 
and to the $\ast-$automorphism $\uplambda^{-1}$,
we deduce the remaining inclusion and st.\eqref{09061715a} 
follows.
St.\eqref{09061715b} 
follows by \eqref{09071701}, st.\eqref{09061715a}
and the uniqueness of the entire analytic extension of 
$\upeta^{\uplambda}$.
Let $b\in\uplambda(Dom(\updelta_{\upeta}))$ 
thus by \eqref{09081203}, \eqref{09081154}
and the norm continuity of $\uplambda$ we deduce that
\begin{equation*}
\ms{ad}(\uplambda)(\updelta_{\upeta})(b)
=
\lim_{t\up 0,\,t\neq 0}
\frac{\upeta^{\uplambda}(t)(b)-b}{t},
\end{equation*}
thus
\begin{equation}
\label{09081155}
\ms{ad}(\uplambda)
(\updelta_{\upeta})
\subset
\updelta_{\upeta^{\uplambda}}.
\end{equation}
By applying \eqref{09081155} we obtain
$\ms{ad}(\uplambda^{-1})
(\updelta_{\upeta^{\uplambda}})
\subset
\updelta_{\upeta}$, 
which implies 
\begin{equation*}
\updelta_{\upeta^{\uplambda}}
\subset
\ms{ad}(\uplambda)
(\updelta_{\upeta}),
\end{equation*}
and the statement follows.
\end{proof}
If
$\lr{\mc{A}}{\R,\upeta}$ is a dynamical system
and
$\uplambda\in Aut_{\ms{CA}^{\ast}}(\mc{A})$
then $\uplambda^{-1}$ is 
$(\ms{ad}(\uplambda)\circ\upeta,\upeta)-$equivariant, hence by Thm. \ref{11141627} we obtain the following result,
however we prefer to show it as a consequence of Lemma \ref{09061626}.
Cor. \ref{09081138} will be used via Cor. \ref{09081449}
to prove 
Thm. \ref{11260828} 
a step toward the construction
of the object part of the functor from $\ms{C}_{u}(H)$ and
$\mf{G}(G,F,\uprho)$.
\begin{corollary}
\label{09081138}
Let $\lr{\mc{A}}{\R,\upeta}$ be a dynamical system,
$\uplambda\in Aut_{\ms{CA}^{\ast}}(\mc{A})$ an $\beta\in\widetilde{\R}$,
then
$\uplambda^{\ast}
\bigl(\ms{K}_{\beta}^{\upeta}\bigr)
=
\ms{K}_{\beta}^{\ms{ad}(\uplambda)\circ\upeta}$.
\end{corollary}
\begin{proof}
In this proof we denote 
$\ms{ad}(\uplambda)\circ\upeta$
by
$\upeta^{\uplambda}$.
Let $\beta\in\R$,  
$\upomega\in\ms{K}_{\beta}^{\upeta}$
and
$\mc{D}_{\upeta}$ 
be 
a norm dense, $\upeta-$invariant, 
$\ast-$subalgebra of $\mc{A}_{\upeta}$ 
such that 
$
\upomega(a\,\ov{\upeta}(i\beta)(b))
=
\upomega(b\,a)$,
for all
$a,b\in\mc{D}_{\upeta}$.
Hence for all $c,d\in\uplambda(\mc{D}_{\upeta})$
\begin{equation*}
\upomega\bigl(\uplambda^{-1}(c)\,
\ov{\upeta}(i\beta)(\uplambda^{-1}(d))\bigr)
=
\uplambda^{\ast}(\upomega)(d\,c),
\end{equation*} 
thus
\begin{equation}
\label{09071112}
\uplambda^{\ast}(\upomega)
\bigl(c\,(\ms{ad}(\uplambda)\circ\ov{\upeta})(i\beta)(d)\bigr)
=
\uplambda^{\ast}(\upomega)(d\,c).
\end{equation} 
Moreover
$\uplambda(\mc{D}_{\upeta})$
is a norm dense, 
since $\uplambda$ is surjective and norm continuous,
$\upeta^{\uplambda}-$invariant, 
$\ast-$subalgebra of $\uplambda(\mc{A}_{\upeta})$,
hence the inclusion $\subset$ of the statement 
follows since 
\eqref{09071112} and
Lemma \ref{09061626}\eqref{09061715a}\,\&\,\eqref{09061715b}.
The remaining inclusion follows by the previous one applied
to the system 
$\lr{\mc{A}}{\R,\ms{ad}(\uplambda)\circ\upeta}$ 
and to the $\ast-$automorphism $\uplambda^{-1}$.
Let $\uppsi\in\ms{K}_{\infty}^{\upeta}$ 
and $b\in\uplambda(Dom(\updelta_{\upeta}))$. 
By definition
$
i\uppsi(a\updelta_{\upeta}(a))\geq 0
$, for all 
$a\in Dom(\updelta_{\upeta})$,
then
\begin{equation*}
i\uplambda^{\ast}(\uppsi)
\bigl(b\,\ms{ad}(\uplambda)(\updelta_{\upeta})(b)\bigr)
=
i\uppsi\bigl(\uplambda^{-1}(b)
\updelta_{\upeta}(\uplambda^{-1}b)
\bigr)\geq 0,
\end{equation*}
hence the inclusion $\subset$ of the statement follows since Lemma \ref{09061626}\eqref{09061715c}.
By applying this inclusion to to the system $\lr{\mc{A}}{\R,\ms{ad}(\uplambda)\circ\upeta}$ 
and to $\uplambda^{-1}$ we obtain 
$(\uplambda^{-1})^{\ast}\bigl(\ms{K}_{\beta}^{\ms{ad}(\uplambda)\circ\upeta}\bigr)\subset\ms{K}_{\beta}^{\upeta}$
which is the remaining inclusion and the statement follows.
\end{proof}
\subsection{Equivariance of representations of $C^{\ast}-$crossed products}
\label{05061338}
The main result of this section is Thm. \ref{09191747}
where we show the equivariance under action of $H$
of representations of $C^{\ast}-$crossed products for different symmetry groups.
In this section we assume fixed 
two locally compact topological groups $G$ and $F$,
a group homomorphism $\uprho:F\to Aut_{\ms{Gr}}(G)$ 
such that the map $(g,f)\mapsto\uprho_{f}(g)$
on $G\times F$ at values in $G$, is continuous,
moreover let $H$ denote $G\rtimes_{\uprho}F$.
\begin{definition}
\label{08191928}
Let $\mf{A}=\lr{\mc{A},H}{\upsigma}$ be a dynamical system
and
$\upomega\in\mc{A}^{\ast}$ set
\begin{equation*}
\ms{F}_{\upomega}(\mf{A})
\coloneqq
\{
h\in F\mid
\upomega\circ\upgamma_{\upsigma}(h)
=
\upomega
\},
\end{equation*}
and
\begin{equation*}
\mf{F}_{\upomega}(\mf{A})
\coloneqq
\{
F_{0}\text{ closed subgroup of }F
\mid
\upomega\in\ms{E}_{\mc{A}}^{F_{0}}(\upgamma_{\upsigma}\up F_{0})
\}.
\end{equation*}
\end{definition}
We convein to remove $(\mf{A})$ 
from $\ms{F}_{\upomega}(\mf{A})$ and $\mf{F}_{\upomega}(\mf{A})$
whenever it is clear 
which dynamical system is involved.
\begin{lemma}
\label{08191843}
Let $\lr{\mc{A},H}{\upsigma}$ be a dynamical system
and
$\upomega\in\mc{A}^{\ast}$, 
thus
$\ms{F}_{\upomega}=\max\mf{F}_{\upomega}$,
in particular 
$\ms{F}_{\upomega}=\bigcup_{F_{0}\in\mf{F}_{\upomega}}F_{0}$.
\end{lemma}
\begin{proof}
By construction 
$\ms{F}_{\upomega}\supseteq F_{0}$ for any 
$F_{0}\in\mf{F}_{\upomega}$, 
thus it is sufficient to show that 
$\ms{F}_{\upomega}\in\mf{F}_{\upomega}$.
Let $h\in F$ such that there exists a net
$\{h_{\alpha}\}_{\alpha\in D}$ in $F$ 
for which
$h=\lim_{\alpha\in D}h_{\alpha}$
and $\upomega\circ\upsigma(j_{2}(h_{\alpha}))=\upomega$
for all $\alpha\in D$.
Since the $\sigma(\mc{A},\mc{A}^{\ast})-$continuity
of $\upsigma$ and the continuity of $j_{2}$,
we have for all $A\in\mc{A}$
\begin{equation*}
\begin{aligned}
\upomega(\upgamma(h)A)
&=
\upomega(\upsigma(j_{2}(h))A)
\\
&=
\lim_{\alpha\in D}
\upomega(\upsigma(j_{2}(h_{\alpha}))A)
=
\upomega(A),
\end{aligned}
\end{equation*}
so $\ms{F}_{\upomega}$ is closed.
\end{proof}
\begin{remark}
\label{06091956}
Since Lemma \ref{08191843} 
the group 
$\ms{S}_{\ms{F}_{\upomega}}^{G}$
is locally compact, hence 
for any 
dynamical system
$\lr{\mc{A},H}{\upsigma}$
it is so also
$\lr{\mc{A},\ms{S}_{\ms{F}_{\upomega}}^{G}}
{\upsigma}$.
In particular it makes sense to consider the $C^{\ast}-$crossed
product
$\mc{A}\rtimes_{\upsigma}^{\upmu}\ms{S}_{\ms{F}_{\upomega}}^{G}$
for any 
$\upmu\in\mc{H}(\ms{S}_{\ms{F}_{\upomega}}^{G})$.
\end{remark}
\begin{lemma}
\label{11131734}
Let 
$\lr{\mc{A},H}{\upeta}$ and $\lr{\mc{B},H}{\uptheta}$
be dynamical systems, 
$\upomega\in\mc{B}^{\ast}$
and $T$ a $(\upeta,\uptheta)-$equivariant morphism
such that $T(\mc{A})$ is $\sigma(\mc{B},\mc{B}^{\ast})-$dense.
Then 
$\ms{F}_{\upomega}
=
\ms{F}_{T_{\dagger}(\upomega)}$
and
$\ms{S}_{\ms{F}_{\upomega}}^{G}
=
\ms{S}_{\ms{F}_{T_{\dagger}(\upomega)}}^{G}$.
\end{lemma}
\begin{proof}
\begin{equation*}
\begin{aligned}
\ms{F}_{T_{\dagger}(\upomega)}
&=
\{h\in F\mid
\upomega\circ T\circ\upeta(j_{2}(h)) 
=
\upomega\circ T\}
\\
&=
\{h\in F\mid
\upomega\circ\uptheta(j_{2}(h))\circ T
=
\upomega\circ T\}
\\
&\supseteq
\{h\in F\mid
\upomega\circ\uptheta(j_{2}(h))
=
\upomega
\}
=\ms{F}_{\upomega}.
\end{aligned}
\end{equation*}
The inclusion $\subseteq$ 
follows since the density hypothesis
and $\upomega\circ\uptheta(j_{2}(h)),\upomega\in\mc{B}^{\ast}$.
The second equality follows by the first one.
\end{proof}
For the remaining of the present section we assume fixed a $C^{\ast}-$dynamical system $\mf{A}=\lr{\mc{A},H}{\upsigma}$.
\begin{lemma}
\label{08191854}
If
$\upomega\in\ms{E}_{\mc{A}}^{G}(\uptau)$ 
and
$F_{0}\in\mf{F}_{\upomega}$, 
then
$\upomega\in
\ms{E}_{\mc{A}}^{\ms{S}_{F_{0}}^{G}} 
(\upsigma_{F_{0}})$.
\end{lemma}
\begin{proof}
The statement follows since $\upsigma$ is a group action
and since
$(g,f)=(g,\un)\cdot_{\rho}(\un,f)$,
for all $(g,f)\in G\times F$.
\end{proof}
Since Lemma \ref{08191843}\,\&\,\ref{08191854} we can state the following
\begin{corollary}
\label{08192221}
$\upomega\in\ms{E}_{\mc{A}}^{G}(\uptau)
\Rightarrow
\upomega\in
\ms{E}_{\mc{A}}^{\ms{S}_{\ms{F}_{\upomega}}^{G}} 
(\upsigma_{\ms{F}_{\upomega}})$.
\end{corollary}
\begin{corollary}
\label{09081449}
Let $l,h\in H$, 
$\beta\in\widetilde{\R}$
and $\upxi:\R\to G$ be a continuous group morphism.
Then
$\upsigma^{\ast}(l)
\bigl(\ms{K}_{\beta}^{\uptau^{(h,\upxi)}}\bigr)
=
\ms{K}_{\beta}^{\uptau^{(l\cdot_{\uprho}h,\upxi)}}$.
\end{corollary}
\begin{proof}
By Cor. \ref{09081138} and \eqref{09071855}.
\end{proof}
\begin{lemma}
\label{09102032}
Let $\uppsi\in\ms{E}_{\mc{A}}^{G}(\uptau)$ and 
$l\in H$,
then
\begin{enumerate}
\item
$\ms{F}_{\upsigma^{\ast}(l)(\uppsi)}
=
\ms{ad}(\Pr_{2}(l))(\ms{F}_{\uppsi})
$;
\label{09101724pre}
\item
$\ms{S}_{\ms{F}_{\upsigma^{\ast}(l)(\uppsi)}}^{G}
=
\ms{ad}(l)(\ms{S}_{\ms{F}_{\uppsi}}^{G})$.
\label{09101945pre}
\end{enumerate}
\end{lemma}
\begin{proof}
Since 
$(g,h)=(\un,h)\cdot_{\uprho}(\uprho(h^{-1})g,\un)$
for all $g\in G$ and $h\in F$, we have
\begin{equation*}
l=j_{2}(\Pr_{2}(l))
\cdot_{\uprho}
j_{1}\left(\uprho(\Pr_{2}(l^{-1}))\Pr_{1}(l)\right),
\end{equation*}
next 
$\uppsi\in\ms{E}_{\mc{A}}^{G}(\uptau)$
by construction, thus
\begin{equation}
\label{09101804}
\upsigma^{\ast}(l)(\uppsi) 
=
\upgamma^{\ast}(\Pr_{2}(l))(\uppsi).
\end{equation}
Hence
$\ms{ad}(\Pr_{2}(l))(\ms{F}_{\uppsi})\in
\mf{F}_{\upsigma^{\ast}(l)(\uppsi)}$,
so by Lemma \ref{08191843}
\begin{equation}
\label{09101855}
\ms{ad}(\Pr_{2}(l))(\ms{F}_{\uppsi})
\subset
\ms{F}_{\upsigma^{\ast}(l)(\uppsi)}.
\end{equation}
Now 
$\upsigma^{\ast}(l)(\uppsi)\in\ms{E}_{\mc{A}}^{G}(\uptau)$,
since $j_{1}(G)$ is a normal subgroup
of $H$,
hence \eqref{09101855} holds if we replace 
$\uppsi$ by $\upsigma^{\ast}(l)(\uppsi)$ 
and $l$ by $l^{-1}$
and obtain
\begin{equation*}
\ms{ad}(\Pr_{2}(l^{-1}))
(\ms{F}_{\upsigma^{\ast}(l)(\uppsi)})
\subset
\ms{F}_{\uppsi},
\end{equation*}
hence
$\ms{F}_{\upsigma^{\ast}(l)(\uppsi)}
\subset
\ms{ad}(\Pr_{2}(l))(\ms{F}_{\uppsi})$
and st.\eqref{09101724pre} follows by \eqref{09101855}.
st.\eqref{09101945pre} follows 
by
$(g,h)=j_{1}(g)\cdot_{\uprho}j_{2}(h)$
for all $g\in G$ and $h\in F$,
by st.\eqref{09101724pre}
and since
$j_{1}(G)$ is a normal subgroup
of $H$.
\end{proof}
\begin{proposition}
\label{09131203}
Let $V\in\mc{U}(\mc{A})$,
$\uppsi\in\ms{E}_{\mc{A}}$ and 
$\lr{\mf{H},\uppi}{\Upomega}$
be 
a cyclic representation of $\mc{A}$ associated
with 
$\uppsi$.
Then
$\lr{\mf{H},\uppi}{\uppi(V)\Upomega}$
is a cyclic representation of $\mc{A}$ 
associated with $\ms{ad}(V)^{\ast}(\uppsi)$.
\end{proposition}
\begin{proof}
$\uppsi\circ\ms{ad}(V^{-1})
=
\upomega_{\uppi(V)\Upomega}
\circ\uppi$.
Since 
$\mc{A}V\subset\mc{A}$
and
$\mc{A}V^{-1}\subset\mc{A}$
we have
$\mc{A}=\mc{A}V^{-1}V\subset\mc{A}V\subset\mc{A}$,
so
$\mc{A}V=\mc{A}$.
Thus $\uppi(V)\Upomega$
is cyclic for the set $\uppi(\mc{A})$.
\end{proof}
\begin{definition}
\label{09141932}
Let 
$\pf{H}
\coloneqq
\lr{\mf{H},\uppi}{\Upomega}$
be a cyclic representation of $\mc{A}$,
$\ms{v}$
a group morphism of 
$H$
into 
$\mc{U}(\mc{A})$,
and 
$l\in H$.
Set
$\pf{H}^{(\ms{v},l)}
\coloneqq
\lr{\mf{H},\uppi}
{\uppi(\ms{v}(l))\Upomega}$.
\end{definition}
\begin{remark}
\label{09151131}
Let $\uppsi\in\ms{E}_{\mc{A}}$
and $\mf{H}$
be 
a cyclic representation of $\mc{A}$ 
associated
with 
$\uppsi$.
Since Prp. \ref{09131203},
if $\mf{A}$ is inner implemented by $\ms{v}$,
then
$\pf{H}^{(\ms{v},l)}$ is a cyclic
representation of $\mc{A}$ associated with 
$\upsigma^{\ast}(l)(\uppsi)$.
\end{remark}

\begin{definition}
\label{08192224pre}
Let $A$ be a nonempty set,
$\ps{\upomega}:A\to\ms{E}_{\mc{A}}^{G}(\uptau)$
and
$\pf{H}:A\to Rep_{c}(\mc{A})$ 
be such that 
$\pf{H}_{\alpha}=
\lr{\mf{H}_{\alpha},\uppi_{\alpha}}{\Upomega_{\alpha}}$ 
is a cyclic representation of $\mc{A}$
associated with $\ps{\upomega}_{\alpha}$, for all $\alpha\in A$.
Thus 
we can
set for all $\alpha\in A$
\begin{equation*}
\ms{U}_{\pf{H},\alpha}^{\upsigma}
\coloneqq
\ms{W}_{\pf{H}}^{\upsigma_{\ms{F}_{\ps{\upomega}_{\alpha}}}}.
\end{equation*}
Whenever it is clear the dynamical system involved,
we convein to remove the index $\upsigma$,
moreover we remove the index $\alpha$ anytime
$A$ is a singleton.
\end{definition}
\begin{remark}

\label{11241735}
Since Lemma \ref{08191843}
$\ms{F}_{\ps{\upomega}_{\alpha}}$ is closed in $F$
for any $\alpha\in A$,
so
$\lr{\mc{A}}{\ms{S}_{\ms{F}_{\ps{\upomega}_{\alpha}}}^{G},
\upsigma_{\ms{F}_{\ps{\upomega}_{\alpha}}}}$
is a dynamical system, 
hence Def. \ref{08192224pre} is well-set by Cor. \ref{08192221}.
By construction 
$\ms{U}_{\pf{H},\alpha}^{\upsigma}:
\ms{S}_{\ms{F}_{\ps{\upomega}_{\alpha}}}^{G}
\to
\mc{L}(\mf{H}_{\alpha})$
such that 
for all $l\in\ms{S}_{\ms{F}_{\ps{\upomega}_{\alpha}}}^{G}$ 
and $a\in\mc{A}$
we have
\begin{equation*}
\ms{U}_{\pf{H},\alpha}^{\upsigma}(l)\,
\uppi_{\alpha}(a)\Upomega_{\alpha}
=
\uppi_{\alpha}(\upsigma(l)a)\Upomega_{\alpha}.
\end{equation*}
\end{remark}
\begin{corollary}
\label{09131225}
Let 
$\upvarphi\in\ms{E}_{\mc{A}}^{G}(\uptau)$
and $l\in H$,
thus
$\upsigma^{\ast}(l)(\upvarphi)\in\ms{E}_{\mc{A}}^{G}(\uptau)$.
Moreover let
$\pf{H}=\lr{\mf{H},\uppi}{\Upomega}$
be 
a cyclic representation of $\mc{A}$ 
associated
with 
$\upvarphi$.
If $\mf{A}$ is inner implemented by $\ms{v}$,
then
$\pf{H}^{(\ms{v},l)}$
is a cyclic representation of $\mc{A}$ associated
with 
$\upsigma^{\ast}(l)(\upvarphi)$
and
\begin{equation}
\label{09151818}
\ms{U}_{\pf{H}^{(\ms{v},l)}}
=
(\ms{ad}\circ\uppi_{\upvarphi}\circ \ms{v})(l)
\circ
\ms{U}_{\pf{H}}
\circ
\ms{ad}(l^{-1})\up
\ms{S}_{\ms{F}_{\upsigma^{\ast}(l)(\upvarphi)}}^{G}.
\end{equation}
\end{corollary}
\begin{proof}
The first sentence of the statement follows
since 
$l^{-1}\cdot_{\uprho}j_{1}(g)=
\ms{ad}(l^{-1})(j_{1}(g))\cdot_{\uprho}l^{-1}$,
for all $g\in G$
and since $j_{1}(G)$ is a normal subgroup of
$H$.
If $\mf{A}$ is inner, then the second
sentence of the statement follows by Rmk. \ref{09151131}.
$\ms{U}_{\pf{H}^{(\ms{v},l)}}$ is well-set,
since the first two sentences of the statement 
and Def. \ref{08192224pre}.
Moreover since Lemma \ref{09102032}\eqref{09101945pre} 
\begin{equation}
\label{09151941}
\ms{ad}(l)(\ms{S}_{\ms{F}_{\upvarphi}}^{G})
=
\ms{S}_{\ms{F}_{\upsigma^{\ast}(l)(\upvarphi)}}^{G}.
\end{equation}
If $\pf{H}=\lr{\mf{H},\uppi}{\Upomega}$
thus for all $h\in\ms{S}_{\ms{F}_{\upvarphi}}^{G}$ 
and $a\in\mc{A}$
we have
\begin{equation*}
\begin{aligned}
\ms{U}_{\pf{H}^{(\ms{v},l)}}
\bigl(\ms{ad}(l)h\bigr)
\,\uppi(a)\,\uppi(\ms{v}(l))\,\Upomega
&=
\uppi\bigl(\upsigma(\ms{ad}(l)h)a\bigr)\,
\uppi(\ms{v}(l))\,\Upomega
\\
&=
\uppi\bigl(\ms{v}(\ms{ad}(l)h)\,
a\,\ms{v}(\ms{ad}(l)h^{-1})
\ms{v}(l)\bigr)\,\Upomega
\\
&=
\uppi\bigl(\ms{v}(l)\,
\ms{v}(h)\,
\ms{v}(l^{-1})\,
a\,
\ms{v}(l)\,
\ms{v}(h^{-1})\bigr)\,
\Upomega
\\
&=
\uppi(\ms{v}(l))\,
\uppi\bigl(\upsigma(h)
\bigl(\ms{ad}(\ms{v}(l^{-1}))a\bigr)
\bigr)\,\Upomega
\\
&=
\uppi(\ms{v}(l))\,
\ms{U}_{\pf{H}}(h)\,
\uppi\bigl(\ms{ad}(\ms{v}(l^{-1}))a\bigr)\,\Upomega
\\
&=
\uppi(\ms{v}(l))\,
\ms{U}_{\pf{H}}(h)\,
\uppi(\ms{v}(l^{-1}))\,
\uppi(a)\,
\uppi(\ms{v}(l))\,\Upomega.
\end{aligned}
\end{equation*}
Moreover 
$\uppi(\ms{v}(l^{-1}))
=
\uppi(\ms{v}(l))^{-1}$,
hence \eqref{09151818}
follows by \eqref{09151941}
and the cyclicity of the 
representation $\pf{H}^{(\ms{v},l)}$.
\end{proof}
\begin{remark}
\label{09151937}
Under the hypotheses of Cor. \ref{09131225}
we have
for all $s\in\ms{S}_{\ms{F}_{\upsigma^{\ast}(l)(\upvarphi)}}^{G}$
\begin{equation*}
\ms{U}_{\pf{H}^{(\ms{v},l)}}(s)
=
\uppi(\ms{v}(l))\,
\ms{U}_{\pf{H}}\bigl(\ms{ad}(l^{-1})s\bigr)\,
\uppi(\ms{v}(l^{-1})).
\end{equation*}
\end{remark}
\begin{remark}
\label{09171300}
Let $T$ be a locally compact space $\mu\in\mc{M}(T)$, 
$\mf{H}$ a Hilbert space. Thus for any map $f:T\to \mc{L}_{s}(\mf{H})$
scalarly essentially $\mu-$integrable we have 
$\int f\,d\mu\in\mc{L}(\mf{H})$, since 
\cite[Ch. $6$, $\S 1$, $n^{\circ} 4$, Cor. $2$]{IntBourb}.
\end{remark}
\begin{lemma}
\label{09171301}
Let $G_{1}, G_{2}$ be two locally compact groups, 
$\eta:G_{1}\to G_{2}$ be an isomorphism of topological
groups and $\mu\in\mc{H}(G_{1})$. Then $\eta(\mu)$
is well-set and $\eta(\mu)\in\mc{H}(G_{2})$.
\end{lemma}
\begin{proof}
$\eta$ is $\mu-$misurable since it is continuous. Let 
$E$ be a Hausdorff locally convex space 
and 
$g\in\mc{C}_{c}(G_{2},E)$, thus 
\begin{equation}
\label{09191032}
\eta^{-1}(supp(g))
=
supp(g\circ\eta)
\in Comp(G_{1})
\end{equation}
Indeed $\eta^{-1}(supp(g))$ is compact 
$\eta^{-1}$ being continuous, moreover
$supp(g\circ\eta)
\supseteq
\eta^{-1}(supp(g))$.
Let $x\in supp(g\circ\eta)$
thus there exists a net $\{x_{\alpha}\}_{\alpha\in D}$
in $G_{1}$ such that $\lim_{\alpha\in D}x_{\alpha}=x$
and $g(\eta(x_{\alpha}))\neq\ze$, for all $\alpha\in D$.
But $\lim_{\alpha\in D}\eta(x_{\alpha})=\eta(x)$
so $\eta(x)\in supp(g)$, i.e. $x\in\eta^{-1}(supp(g))$
and \eqref{09191032} follows.
Let $f\in\mc{C}_{c}(G_{2})$, 
then by \eqref{09191032}
$f\circ\eta\in\mc{C}_{c}(G_{1})$ and $\eta(\mu)$ is well-defined.
Let $s\in G_{2}$ 
thus $L_{s}\circ\eta=\eta\circ L_{\eta^{-1}(s)}$ so
\begin{equation*}
L_{s}^{\ast}(f)\circ\eta
=
L_{\eta^{-1}(s)}^{\ast}(f\circ\eta),
\end{equation*}
then
\begin{equation*}
\begin{aligned}
\int L_{s}^{\ast}(f)\,d\eta(\mu)
&=
\int L_{s}^{\ast}(f)\circ\eta\,d\mu
\\
&=
\int
L_{\eta^{-1}(s)}^{\ast}(f\circ\eta)\,
d\mu
\\
&=
\int f\circ\eta\, d\mu
=
\int f\, d\eta(\mu),
\end{aligned}
\end{equation*}
where the third equality follows by the left invariance of 
$\mu$, while the first and forth ones follow by \eqref{09191048}.
\end{proof}
\begin{remark}
\label{09171400}
Let $E$ be a Hausdorff locally convex space, 
$T,S$ be two locally compact spaces, $\mu\in\mc{M}(T)$,
and $\ep:T\to S$ be $\mu-$proper.
Since 
\cite[Ch. $6$, $\S 1$, $n^{\circ}1$, pg. $4$
and
Ch. $5$, $\S 6$, $n^{\circ}2$, Thm. $1$]{IntBourb}
we have for any scalarly essentially 
$\ep(\mu)-$integrable map $f:S\to E$ 
that
$f\circ\ep$ is scalarly essentially $\mu-$integrable and
\begin{equation*}
\int f\circ\ep\,d\mu
=
\int f\,d\ep(\mu).
\end{equation*}
Moreover if $E=\mc{L}_{s}(\mf{H})$ for some Hilbert space $\mf{H}$,
then by Rmk. \ref{09171300}
\begin{equation*}
\int f\circ\ep\,d\mu
\in
\mc{L}(\mf{H}).
\end{equation*}
\end{remark}
\begin{lemma}
\label{10041115}
Let $X$ be a locally compact group, $\mu\in\mc{H}(X)$
and $Y$ be a locally compact subgroup of $X$.
Then $\mu_{Y}\in\mc{H}(Y)$ and $\Delta_{Y}=\Delta_{X}\up Y$.
\end{lemma}
\begin{proof}
Let $e\in\complement Y$ and $s\in Y$ then
$s^{-1}\cdot e, e\cdot s^{-1}\in\complement Y$, 
indeed if by absurdum
$s^{-1}\cdot e, e\cdot s^{-1}\in Y$, then 
$e=s\cdot (s^{-1}\cdot e)=(e\cdot s^{-1})\cdot s\in Y$.
Let $f\in\mc{C}_{c}(Y)$ then
\begin{equation*}
\widetilde{L_{s}^{\ast}(f)}
=
L_{s}^{\ast}(\widetilde{f})
\text{ and }
\widetilde{R_{s}^{\ast}(f)}
=
R_{s}^{\ast}(\widetilde{f}),
\end{equation*}
so
\begin{equation*}
\begin{aligned}
\mu_{Y}(L_{s}^{\ast}(f))
&=
\mu(\widetilde{L_{s}^{\ast}(f)})
\\
&=
\mu(L_{s}^{\ast}(\widetilde{f}))
\\
&=
\mu(\widetilde{f})
=
\mu_{Y}(f).
\end{aligned}
\end{equation*}
Thus $\mu_{Y}\in\mc{H}(Y)$.
Moreover
\begin{equation*}
\begin{aligned}
\Delta_{X}(s)\mu_{Y}\bigl(R_{s^{-1}}^{\ast}(f)\bigr)
&=
\Delta_{X}(s)\mu\bigl(\widetilde{R_{s^{-1}}^{\ast}(f)}\bigr)
\\
&=
\Delta_{X}(s)\mu\bigl(R_{s^{-1}}^{\ast}(\widetilde{f})\bigr)
\\
&=\mu(\widetilde{f})
=\mu_{Y}(f),
\end{aligned}
\end{equation*}
then since $\mu_{Y}\in\mc{H}(Y)$ and
the independence of the modular function 
by the Haar measure (\cite[Lemma $1.61$]{will}), 
we deduce that $\Delta_{Y}=\Delta_{X}\up Y$.
\end{proof}
Lemma \ref{09102032}\eqref{09101945pre} and Lemma \ref{09171301} allow to set the following
\begin{definition}
[Haar systems]
\label{09181000}
Let 
$A$ be a nonempty set and 
$\ps{\upomega}:A\to\ms{E}_{\mc{A}}^{G}(\uptau)$.
We define the set of 
Haar systems associated with $\ps{\upomega}$ and $\mf{A}$, 
denoted by $\mc{H}(\ps{\upomega},\mf{A})$,
the subset of the 
\begin{equation*}
\ps{\upmu}\in
\prod_{(\alpha,l)\in 
A\times H}
\mc{H}(
\ms{S}_{\ms{F}_{\upsigma^{\ast}(l)(\ps{\upomega}_{\alpha})}}^{G}
),
\end{equation*}
such that
for all $(\alpha,l)\in A\times H$
\begin{equation}
\label{10041104}
\ps{\upmu}_{(\alpha,l)}
= 
\ms{ad}(l)
(\ps{\upmu}_{(\alpha,\un)}).
\end{equation}
\end{definition}
\begin{definition}
\label{09191140}
Let $\upnu\in\mc{H}(H)$
and 
$\ps{\upomega}:A\to\ms{E}_{\mc{A}}^{G}(\uptau)$,
define for all 
$(\alpha,l)\in A\times H$
\begin{equation*}
\ps{\upnu}_{(\alpha,l)}
\coloneqq 
\ms{ad}(l)(\upnu_{\ms{S}_{\ms{F}_{\ps{\upomega}_{\alpha}}}^{G}}).
\end{equation*}
$\ps{\upnu}$
will be called
the 
Haar system generated by $\upnu$ and $\ps{\upomega}$.
\end{definition}
$\mc{H}(\ps{\upomega},\mf{A})$ is nonempty, indeed
\begin{proposition}
Let $\upnu\in\mc{H}(H)$
and 
$\ps{\upomega}:A\to\ms{E}_{\mc{A}}^{G}(\uptau)$,
then the Haar system generated by $\upnu$ and $\ps{\upomega}$
belongs to $\mc{H}(\ps{\upomega},\mf{A})$.
\end{proposition}
\begin{proof}
Since Lemma \ref{10041115}, Lemma \ref{09102032}\eqref{09101945pre} and Lemma \ref{09171301}.
\end{proof}
\begin{definition}
\label{11211142}
Let 
$A$ be a nonempty set,
$\ps{\upomega}:A\to\ms{E}_{\mc{A}}^{G}(\uptau)$
and
$\ps{\upmu}\in\mc{H}(\ps{\upomega},\mf{A})$.
For any $l\in H$ and $\alpha\in A$
set
\begin{equation*}
\begin{aligned}
\ms{B}_{\ps{\upmu}}^{\ps{\upomega},\alpha,l}(\mf{A})
&\coloneqq
\mc{A}
\rtimes_{\upsigma}^{\ps{\upmu}_{(\alpha,l)}}
\ms{S}_{\ms{F}_{\upsigma^{\ast}(l)(\ps{\upomega}_{\alpha})}}^{G}
\\
\ms{B}_{\ps{\upmu}}^{\ps{\upomega},\alpha,l,+}(\mf{A})
&\coloneqq
(\ms{B}_{\ps{\upmu}}^{\ps{\upomega},\alpha,l}(\mf{A}))^{+}.
\end{aligned}
\end{equation*}
We convein to remove $l$ when it equals the identity and 
to remove $\alpha$ when $A$ is a singleton.
Moreover 
whenever it is clear which dynamical system 
is involved, we convein to 
remove $\mf{A}$ and
to denote the map
$\upsigma^{\ast}(l)\circ\ps{\upomega}$
by 
$\ps{\upomega}^{l}$
for any $l\in H$. 
Let
$\pf{H}:A\to Rep_{c}(\mc{A})$ 
be such that 
$\pf{H}_{\alpha}=
\lr{\mf{H}_{\alpha},\uppi_{\alpha}}{\Upomega_{\alpha}}$ 
is a cyclic representation of $\mc{A}$
associated with $\ps{\upomega}_{\alpha}$, for all $\alpha\in A$.
Set for all $\alpha\in A$
\begin{enumerate}
\item
$\mf{R}_{\pf{H},\alpha}^{\ps{\upmu}}(\mf{A})
\coloneqq
\uppi_{\alpha}\rtimes^{\ps{\upmu}_{(\alpha,\un)}}
\ms{U}_{\pf{H}_{\alpha}}^{\upsigma}$
\item
$\pf{R}_{\pf{H},\alpha}^{\ps{\upmu}}(\mf{A})
\coloneqq
(\ms{B}_{\ps{\upmu}}^{\ps{\upomega},\alpha}(\mf{A}),
\mf{R}_{\pf{H},\alpha}^{\ps{\upmu}}(\mf{A}))$
\item
$\tilde{\pf{R}}_{\pf{H},\alpha}^{\ps{\upmu}}(\mf{A})
\coloneqq(\ms{B}_{\ps{\upmu}}^{\ps{\upomega},\alpha,+}(\mf{A}),
\tilde{\mf{R}}_{\pf{H},\alpha}^{\ps{\upmu}}(\mf{A}))$.
\end{enumerate}
If $\mf{A}$ is inner implemented 
by $\ms{v}$,
we can define 
$\pf{H}^{(\ms{v},l)}:I\to Rep_{c}(\mc{A})$,
such that
$\beta\mapsto\pf{H}_{\alpha}^{(\ms{v},l)}
\coloneqq
(\pf{H}_{\alpha})^{(\ms{v},l)}$
and
for all $\alpha\in A$
\begin{enumerate}
\item
$\ms{U}_{\pf{H},\ms{v},\alpha,l}^{\upsigma}
\coloneqq
\ms{U}_{\pf{H}_{\alpha}^{(\ms{v},l)}}^{\upsigma}$
\item
$\mf{R}_{\pf{H},\ms{v},\alpha,l}^{\ps{\upmu}}(\mf{A})
\coloneqq
\uppi_{\alpha}\rtimes^{\ps{\upmu}_{(\alpha,l)}}
\ms{U}_{\pf{H},\ms{v},\alpha,l}^{\upsigma}$
\item
$\pf{R}_{\pf{H},\ms{v},\alpha,l}^{\ps{\upmu}}(\mf{A})
\coloneqq
(\ms{B}_{\ps{\upmu}}^{\ps{\upomega},\alpha,l}(\mf{A}),
\mf{R}_{\pf{H},\ms{v},\alpha,l}^{\ps{\upmu}}(\mf{A}))$
\item
$\tilde{\pf{R}}_{\pf{H},\ms{v},\alpha,l}^{\ps{\upmu}}(\mf{A})
\coloneqq(\ms{B}_{\ps{\upmu}}^{\ps{\upomega},\alpha,l,+}(\mf{A}),
\tilde{\mf{R}}_{\pf{H},\ms{v},\alpha,l}^{\ps{\upmu}}(\mf{A}))$.
\end{enumerate}
We convein 
to remove $\mf{A}$ and the index $\upsigma$ 
whenever 
it is clear the dynamical system involved,
to remove the index $l$ 
whenever it equals the identity,
and the index $\ms{v}$ whenever 
it is uniquely determined by $\upsigma$,
and to remove $\alpha$ if it is a singleton.
\end{definition}
\begin{remark}
\label{10111026r}
Let $\alpha\in A$ and $l\in H$, and $\ps{\upomega}$ and $\ps{\upmu}$ as in Def. \ref{11211142}, then
$\ms{B}_{\ps{\upmu}}^{\ps{\upomega},\alpha,l}$ is well-set since Lemma \ref{08191843},
and $\ms{B}_{\ps{\upmu}}^{\ps{\upomega},\alpha,l,+}$ is the $C^{\ast}-$algebra obtained by adding the unit to
$\ms{B}_{\ps{\upmu}}^{\ps{\upomega},\alpha,l}$.
Moreover
\begin{itemize}
\item
$\mf{R}_{\pf{H},\ms{v},\alpha,l}^{\ps{\upmu}}$
extends uniquely to a $\ast-$representation of 
$\ms{B}_{\ps{\upmu}}^{\ps{\upomega},\alpha,l}$
in $\mf{H}_{\alpha}$;
\label{10121949b}
\item
$\tilde{\mf{R}}_{\pf{H},\ms{v},\alpha,l}^{\ps{\upmu}}$ is 
the $\ast-$representation of 
$\ms{B}_{\ps{\upmu}}^{\ps{\upomega},\alpha,l,+}$
in $\mf{H}_{\alpha}$
induced by 
$\mf{R}_{\pf{H},\ms{v},\alpha,l}^{\ps{\upmu}}$
according to \eqref{10281544}.
\end{itemize}
\end{remark}
Lemma \ref{09102032}\eqref{09101945pre} and \eqref{09191032} permit to set the following
\begin{definition}
\label{10041228}
Define 
\begin{equation*}
\ps{\upsigma}\in
\prod_{
(\uppsi,l)\in\ms{E}_{\mc{A}}^{G}(\uptau)\times H}
Mor_{\ms{set}}
\bigl(
\mc{C}_{c}(\ms{S}_{\ms{F}_{\uppsi}}^{G},\mc{A}),
\mc{C}_{c}(\ms{S}_{\ms{F}_{\upsigma^{\ast}(l)(\uppsi)}}^{G},
\mc{A})
\bigr),
\end{equation*}
such that
\begin{equation*}
\ps{\upsigma}^{(\uppsi,l)}(f)
\coloneqq
\upsigma(l)\circ f\circ\ms{ad}(l^{-1})
\up
\ms{S}_{\ms{F}_{\upsigma^{\ast}(l)(\uppsi)}}^{G},
\end{equation*}
for all
$(\uppsi,l)\in\ms{E}_{\mc{A}}^{G}
(\uptau)\times H$.
\end{definition}
\begin{proposition}
\label{12261620}
$supp(\ps{\upsigma}^{(\uppsi,l)}(f))
=
\ms{ad}(l)(supp(f))$,
for any 
$f\in\mc{C}_{c}(\ms{S}_{\ms{F}_{\uppsi}}^{G},\mc{A})$
and
$(\uppsi,l)\in\ms{E}_{\mc{A}}^{G}(\uptau)\times H$,
in particular
$\ps{\upsigma}^{(\uppsi,l)}$ is well-defined. 
\end{proposition}
\begin{proof}
Let $(\uppsi,l)\in\ms{E}_{\mc{A}}^{G}(\uptau)\times H$,
and $\phi^{l}$ and $H_{\phi}$ 
denote
$\upsigma^{\ast}(l)(\phi)$ and $\ms{S}_{\ms{F}_{\phi}}^{G}$ 
respectively for all $\phi\in\ms{E}_{\mc{A}}$.
If
$f\in\mc{C}_{c}(H_{\uppsi},\mc{A})$
then
$\ps{\upsigma}^{(\uppsi,l)}(f)$ is continuous since composition
of continuous maps, moreover $\upsigma(l)$ is linear thus
\begin{equation*}
supp(\ps{\upsigma}^{(\uppsi,l)}(f))
\subseteq
supp(f\circ ad(l^{-1})\up H_{\uppsi^{l}}).
\end{equation*}
Clearly
$ad(l)f^{-1}(\ze)=(f\circ ad(l^{-1})\up H_{\uppsi^{l}})^{-1}(\ze)$,
then since $ad(l)$ is injective 
it follows
$ad(l)\complement f^{-1}(\ze)
=
\complement(f\circ ad(l^{-1})\up H_{\uppsi^{l}})^{-1}(\ze)$,
moreover 
$\ov{ad(l)\complement f^{-1}(\ze)}=ad(l)supp(f)$
since $ad(l)$ is a homeomorphism,
thus
$supp(f\circ ad(l^{-1})\up H_{\uppsi^{l}})
=
\ms{ad}(l)supp(f)$.
Hence 
\begin{equation*}
supp(\ps{\upsigma}^{(\uppsi,l)}(f))\subseteq\ms{ad}(l)supp(f),
\end{equation*}
which applied to the position $\uppsi^{l}$, $l^{-1}$ and $\ps{\upsigma}^{(\uppsi,l)}(f)$, possible according to the first sentence 
in Cor. \ref{09131225}, yields
\begin{equation*}
supp\bigl(\ps{\upsigma}^{(\uppsi^{l},l^{-1})}
(\ps{\upsigma}^{(\uppsi,l)}(f))
\bigr)
\subseteq
\ms{ad}(l^{-1})
supp(\ps{\upsigma}^{(\uppsi,l)}(f)),
\end{equation*}
i.e
$\ms{ad}(l)supp(f)
\subseteq
supp(\ps{\upsigma}^{(\uppsi,l)}(f))$,
and the statement follows.
\end{proof}
\begin{lemma}
\label{09171541}
Let $\ps{\upomega}:A\to\ms{E}_{\mc{A}}^{G}(\uptau)$
and $\ps{\upmu}$ be a Haar system associated with
$\ps{\upomega}$ and $\mf{A}$.
Then for any $(\alpha,l)\in A\times H$
and $h\in H$
we have 
\begin{enumerate}
\item
$\ps{\upsigma}^{(\ps{\upomega}_{\alpha},l)}$ is a $\ast-$isomorphism of $\ast-$algebras 
from 
$\mc{C}_{c}^{\ps{\upmu}_{(\alpha,\un)}}
(\ms{S}_{\ms{F}_{\ps{\upomega}_{\alpha}}}^{G},\mc{A})$
onto
$\mc{C}_{c}^{\ps{\upmu}_{(\alpha,l)}}
(\ms{S}_{\ms{F}_{\upsigma^{\ast}(l)(\ps{\upomega}_{\alpha})}}^{G},\mc{A})$
continuous w.r.t. the inductive limit topologies;
\label{10041833a}
\item
$\ps{\upmu}^{h}$ is a Haar system associated with
$\ps{\upomega}^{h}$ and $\mf{A}$,
and
$\ps{\upsigma}^{(\ps{\upomega}_{\alpha}^{l},l^{-1})} 
=
(\ps{\upsigma}^{(\ps{\upomega}_{\alpha},l)})^{-1}$;
\label{10041833c}
\item
$\ps{\upsigma}^{(\ps{\upomega}_{\alpha},l)}$
is an isometry 
of 
$\pc{C}_{c}^{\ps{\upmu}_{(\alpha,\un)}}
(\ms{S}_{\ms{F}_{\ps{\upomega}_{\alpha}}}^{G},\mc{A})$
onto
$\pc{C}_{c}^{\ps{\upmu}_{(\alpha,l)}}
(\ms{S}_{\ms{F}_{\upsigma^{\ast}(l)(\ps{\upomega}_{\alpha})}}^{G},
\mc{A})$.
\label{10041833b}
\end{enumerate}
Here
$\ps{\upmu}^{h}_{(\beta,u)}
\coloneqq
\ps{\upmu}_{(\beta,u\cdot h)}$,
for all $(\beta,u)\in A\times H$
and recall that 
$\ps{\upomega}^{h}\coloneqq\upsigma^{\ast}(h)\circ\ps{\upomega}$.
\end{lemma}
\begin{proof}
Let $(\alpha,l)\in A\times H$
and
$f,f_{1},f_{2}\in
\mc{C}_{c}(\ms{S}_{\ms{F}_{\ps{\upomega}_{\alpha}}}^{G},\mc{A})$,
then for all 
$s\in
\ms{S}_{\ms{F}_{\upsigma^{\ast}(l)(\ps{\upomega}_{\alpha})}}^{G}$
\begin{equation*}
\begin{aligned}
\ps{\upsigma}^{(\ps{\upomega}_{\alpha},l)}
(f_{1}\ast^{\ps{\upmu}_{(\alpha,\un)}} f_{2})(s)
&=
\upsigma(l)
\bigl(
\int
f_{1}(r)\upsigma(r)\bigl(
f_{2}(r^{-1}\ms{ad}(l^{-1})(s))
\bigr)\,
d\ps{\upmu}_{(\alpha,\un)}(r)
\bigr)
\\
&=
\int
\upsigma(l)(f_{1}(r))
\upsigma(lr)\bigl(
f_{2}(r^{-1}\ms{ad}(l^{-1})(s))
\bigr)\,
d\ps{\upmu}_{(\alpha,\un)}(r)
\\
&=
\int
\upsigma(l)(f_{1}(\ms{ad}(l^{-1})r))
\upsigma(rl)\bigl(
f_{2}(\ms{ad}(l^{-1})(r^{-1}s))
\bigr)\,
d\ps{\upmu}_{(\alpha,l)}(r)
\\
&=
\int
\ps{\upsigma}^{(\ps{\upomega}_{\alpha},l)}(f_{1})(r)
\upsigma(r)
\ps{\upsigma}^{(\ps{\upomega}_{\alpha},l)}(f_{2})(r^{-1}s)
d\ps{\upmu}_{(\alpha,l)}(r)
\\
&=
\bigl(\ps{\upsigma}^{(\ps{\upomega}_{\alpha},l)}(f_{1})
\ast^{\ps{\upmu}_{(\alpha,l)}}
\ps{\upsigma}^{(\ps{\upomega}_{\alpha},l)}(f_{2})
\bigr)(s).
\end{aligned}
\end{equation*}
Here the second equality follows by the continuity of $\upsigma(l)$
in $\|\cdot\|_{\mc{A}}-$topology and by the fact the integration
is w.r.t. the same topology.
The third equality follows by the fact that 
$\ps{\upmu}_{(\alpha,\un)}
=
\ms{ad}(l^{-1})(\ps{\upmu}_{(\alpha,l)})$
and by Rmk. \ref{09171400}.
\begin{equation*}
\begin{aligned}
\ps{\upsigma}^{(\ps{\upomega}_{\alpha},l)}(f)^{\ast}(s)
&=
\Delta(s^{-1})
\upsigma(s)(\ps{\upsigma}^{
(\ps{\upomega}_{\alpha},l)}(f)(s^{-1})^{\ast})
\\
&=
\Delta(s^{-1})
\upsigma(sl)(
f(\ms{ad}(l^{-1})s^{-1})^{\ast})
\\
&=
\upsigma(l)
\Delta(\ms{ad}(l^{-1})s^{-1})
\upsigma(\ms{ad}(l^{-1})s)(
f(\ms{ad}(l^{-1})s^{-1})^{\ast})
\\
&=
\ps{\upsigma}^{(\ps{\upomega}_{\alpha},l)}(f^{\ast})(s).
\end{aligned}
\end{equation*}
Here $\Delta=\Delta_{H}$.
The first and last equality follow
by $\Delta_{Y}=\Delta\up Y$ for 
$Y\in\{
\ms{S}_{\ms{F}_{\ps{\upomega}_{\alpha}}}^{G},  
\ms{S}_{\ms{F}_{\upsigma^{\ast}(l)(\ps{\upomega}_{\alpha})}}^{G}  
\}$,
since Lemma \ref{10041115},
while the third one follows since
$\Delta$ is a group morphism.
Moreover since 
Lemma \ref{09102032}\eqref{09101945pre} 
it is easy to see that
$(\ps{\upsigma}^{(\ps{\upomega}_{\alpha},l)})^{-1}
=
\ps{\upsigma}^{(\upsigma^{\ast}(l)
(\ps{\upomega}_{\alpha}),l^{-1})}$,
and the first part of st.\eqref{10041833a} follows.
In the following we let
$H_{\alpha}=\s{\upomega}{\alpha}{}$
and
$H_{\alpha,l}=\ms{S}_{\ms{F}_{\upsigma^{\ast}(l)(\ps{\upomega}_{\alpha})}}^{G}$.
For any compact subset $K$ of $H_{\alpha}$ 
let
$\mc{C}(H_{\alpha};K,\mc{A})$
be the space of the $f\in\mc{C}_{c}(H_{\alpha},\mc{A})$
such that $supp(f)\subseteq K$ 
and $\imath_{K}$ be the identity map
embedding $\mc{C}(H_{\alpha};K,\mc{A})$ into 
$\mc{C}_{c}(H_{\alpha},\mc{A})$.
If we prove that for any compact $K$ of $H$ the map
$\ps{\upsigma}^{(\ps{\upomega}_{\alpha},l)}\circ\imath_{K}$
is continuous w.r.t. 
the topology of uniform convergence
on $\mc{C}(H;K,\mc{A})$ and the inductive limit topology on
$\mc{C}_{c}(H_{\alpha,l},\mc{A})$,
then the second part of st.\eqref{10041833a} will follows
since \cite[$II.27$ Prp. $5(ii)$]{BourTVS}.
Next since Prp. \ref{12261620} 
\begin{equation}
\label{12261646}
supp(\ps{\upsigma}^{(\ps{\upomega}_{\alpha},l)}(\imath_{K}(f)))
\subseteq ad(l)(K),\,\forall f\in\mc{C}(H_{\alpha};K,\mc{A})
\end{equation}
hence since \cite[Rmk $1.86$]{will} it remains only to show that
$\ps{\upsigma}^{(\ps{\upomega}_{\alpha},l)}(\imath_{K}(f_{\beta}))$
converges to $\ze$ uniformly on $H_{\alpha,l}$ for any net
$\{f_{\beta}\}_{\beta}$ in $\mc{C}(H_{\alpha};K,\mc{A})$
converging uniformly on $H_{\alpha}$ to $\ze$. But this follows since
\eqref{12261646} and since $\upsigma(l)$ is an isometry.
st.\eqref{10041833c} is easy to show.
st.\eqref{10041833b}
follows since st.\eqref{10041833a} and \cite[Cor. $2.47$]{will}.
\end{proof}
\begin{lemma}
\label{09191416}
Let $\mf{A}$ be inner implemented by $\ms{v}$,
$\uppsi\in\ms{E}_{\mc{A}}^{G}(\uptau)$ and
$\lr{\mf{H},\uppi}{\ms{U}}$ be a covariant representation
of 
$\lr{\mc{A}}{\ms{S}_{\ms{F}_{\uppsi}}^{G},
\upsigma_{\ms{F}_{\uppsi}}}$.
Then
$\lr{\mf{H},\uppi}{\ms{U}^{l}}$
is a covariant representation 
of $\lr{\mc{A}}{\ms{S}_{\ms{F}_{\upsigma^{\ast}(l)(\uppsi)}}^{G},
\upsigma_{\ms{F}_{\upsigma^{\ast}(l)(\uppsi)}}}$,
where 
$\ms{U}^{l}\doteq
\ms{ad}(\uppi(\ms{v}(l)))
\circ
\ms{U}
\circ
\ms{ad}(l^{-1})
\up
\ms{S}_{\ms{F}_{\upsigma^{\ast}(l)(\uppsi)}}^{G}$.
\end{lemma}
\begin{proof}
Since Lemma \ref{09102032}\eqref{09101945pre} $\ms{U}^{l}$ is well-defined, moreover it is a strongly continuous
unitary group action since it is so $\ms{U}$, since $\ms{ad}(V)$ is 
strongly continuous for any unitary operator $V$ on $\mf{H}$,
and since $\ms{ad}(\cdot)$ on $H$ is an isomorphism of topological groups. 
Let $l\in H$,
$h\in\ms{S}_{\ms{F}_{\upsigma^{\ast}(l)(\uppsi)}}^{G}$
and $a\in\mc{A}$
then 
\begin{equation*}
\begin{aligned}
\ms{U}^{l}(h)\uppi(a)\ms{U}^{l}(h^{-1})
&=
\\
\uppi(\ms{v}(l))
\ms{U}(\ms{ad}(l^{-1})h)
\uppi(\ms{v}(l^{-1}))
\uppi(a)
\uppi(\ms{v}(l))
\ms{U}(\ms{ad}(l^{-1})h^{-1})
\uppi(\ms{v}(l^{-1}))
&=
\\
\uppi(\ms{v}(l))
\ms{U}(\ms{ad}(l^{-1})h)
\uppi(\upsigma(l^{-1})a)
\ms{U}(\ms{ad}(l^{-1})h^{-1})
\uppi(\ms{v}(l^{-1}))
&=
\\
\uppi(\ms{v}(l))
\uppi(\upsigma(\ms{ad}(l^{-1})(h)l^{-1})a)
\uppi(\ms{v}(l^{-1}))
&=
\\
\uppi(\upsigma(l\ms{ad}(l^{-1})(h)l^{-1})a)
&=
\uppi(\upsigma(h)a).
\end{aligned}
\end{equation*}
\end{proof}
\begin{remark}
\label{12261805}
Let us prove directly
Lemma \ref{09171541}\eqref{10041833b}
in case $\mf{A}$ is inner implemented by
$\ms{v}$.
Let $\lr{\mf{H},\uppi}{\ms{U}}$ be a nondegenerate
covariant representation of 
$\lr{\mc{A},
\ms{S}_{\ms{F}_{\upsigma^{\ast}(l)(\ps{\upomega}_{\alpha})}}^{G}}
{\upsigma_{\ms{F}_{\upsigma^{\ast}(l)(\ps{\upomega}_{\alpha})}}}$
set
\begin{equation*}
\ms{U}^{l^{-1}}
\doteq
\ms{ad}(\uppi(\ms{v}(l^{-1})))\circ\ms{U}\circ\ms{ad}(l)
\up\ms{S}_{\ms{F}_{\ps{\upomega}_{\alpha}}}^{G},
\end{equation*}
then for any 
$f\in
\mc{C}_{c}(\ms{S}_{\ms{F}_{\ps{\upomega}_{\alpha}}}^{G},\mc{A})$
\begin{equation}
\label{10071234}
\begin{aligned}
\bigl\|
(\uppi\rtimes^{\ps{\upmu}_{\alpha,l}}\ms{U})
(\ps{\upsigma}^{(\ps{\upomega}_{\alpha},l)}(f))
\bigr\|
&=
\\
\bigl\|
\int\uppi
(\ps{\upsigma}^{(\ps{\upomega}_{\alpha},l)}(f)(s))
\ms{U}(s)\,
d\ps{\upmu}_{(\alpha,l)}(s)
\bigr\|
&=
\\
\bigl\|
\int\uppi\bigl(
(\upsigma(l)\circ f\circ\ms{ad}(l^{-1}))(s)
\bigr)
\ms{U}(s)\,
d\ps{\upmu}_{(\alpha,l)}(s)
\bigr\|
&=
\\
\bigl\|
\int\uppi\bigl(
(\upsigma(l)\circ f)(s)
\bigr)
\ms{U}(\ms{ad}(l)s)\,
d\ps{\upmu}_{(\alpha,\un)}(s)
\bigr\|
&=
\\
\bigl\|
\int
\uppi(\ms{v}(l))
\uppi(f(s))
\uppi(\ms{v}(l^{-1}))
\ms{U}(\ms{ad}(l)s)
\uppi(\ms{v}(l))
\uppi(\ms{v}(l^{-1}))
\,
d\ps{\upmu}_{(\alpha,\un)}(s)
\bigr\|
&=
\\
\bigl\|
\int
\ms{ad}(\uppi(\ms{v}(l)))
\bigl(
\uppi(f(s))
\ms{U}^{l^{-1}}(s)
\bigr)
\,
d\ps{\upmu}_{(\alpha,\un)}(s)
\bigr\|
&=
\\
\bigl\|
\ms{ad}(\uppi(\ms{v}(l)))
\bigl(
\int
\uppi(f(s))
\ms{U}^{l^{-1}}(s)
\,
d\ps{\upmu}_{(\alpha,\un)}(s)
\bigr)
\bigr\|
&=
\\
\bigl\|
(\uppi\rtimes^{\ps{\upmu}_{\alpha,\un}}
\ms{U}^{l^{-1}})(f)
\bigr\|
&\leq
\|f\|^{\ps{\upmu}_{\alpha,\un}}.
\end{aligned}
\end{equation}
Here the third equality follows by Rmk. \ref{09171400},
the sixth one since $\ms{ad}(V)\in\mc{L}(\mc{L}_{s}(\mf{H}))$,
for any $V\in\mc{U}(\mf{H})$ and the integration is 
w.r.t. the strong operator topology,
the seventh one follows 
by the fact that 
$\ms{ad}(V)$ is an isometry
since
$\ms{ad}(V)\in Aut_{\ms{CA}^{\ast}}(\mc{L}(\mf{H}))$
and the well-known fact that 
any $\ast-$homomorphism between $C^{\ast}-$algebras
is automatically continuous with norm less or equal 
to $1$.
Finally
the inequality follows since an application of
Lemma \ref{09191416} 
stating that 
$\lr{\mf{H},\uppi}{\ms{U}^{l^{-1}}}$ 
is a
nondegenerate
covariant representation of 
$\lr{\mc{A},\ms{S}_{\ms{F}_{\ps{\upomega}_{\alpha}}}^{G}}
{\upsigma_{\ms{F}_{\ps{\upomega}_{\alpha}}}}$.
Therefore
\begin{equation}
\label{10071148}
\|\ps{\upsigma}^{(\ps{\upomega}_{\alpha},l)}(f)\|
^{\ps{\upmu}_{(\alpha,l)}}
\leq
\|f\|^{\ps{\upmu}_{(\alpha,\un)}}.
\end{equation}
Let $h,s\in H$ and 
$\mf{g}\in
\mc{C}_{c}(
\ms{S}_{\ms{F}_{\ps{\upomega}_{\alpha}^{h}}}^{G},\mc{A})$,
since Lemma \ref{09171541}\eqref{10041833c} 
we can apply \eqref{10071148}
to $\ps{\upomega}^{h}$ and $\ps{\upmu}^{h}$ 
to obtain
\begin{equation*}
\|\ps{\upsigma}^{(\ps{\upomega}_{\alpha}^{h},s)}(\mf{g})\|
^{\ps{\upmu}_{(\alpha,s)}^{h}}
\leq
\|\mf{g}\|^{\ps{\upmu}_{(\alpha,\un)}^{h}}.
\end{equation*}
Next 
$\ps{\upmu}_{(\alpha,l^{-1})}^{l}=\ps{\upmu}_{(\alpha,\un)}$
and
$\ps{\upmu}_{(\alpha,\un)}^{l}=\ps{\upmu}_{(\alpha,l)}$
thus for any 
$g\in
\mc{C}_{c}(
\ms{S}_{\ms{F}_{\upsigma^{\ast}(l)
(\ps{\upomega}_{\alpha})}}^{G},\mc{A})$
\begin{equation*}
\|\ps{\upsigma}^{(\ps{\upomega}_{\alpha}^{l},l^{-1})}(g)\|
^{\ps{\upmu}_{(\alpha,\un)}}
\leq
\|g\|^{\ps{\upmu}_{(\alpha,l)}},
\end{equation*}
moreover $\ps{\upsigma}^{(\ps{\upomega}_{\alpha}^{l},l^{-1})}=(\ps{\upsigma}^{(\ps{\upomega}_{\alpha},l)})^{-1}$
since Lemma \ref{09171541}\eqref{10041833c}, therefore
\begin{equation}
\label{10071148b}
\|f\|^{\ps{\upmu}_{(\alpha,\un)}}\leq
\|\ps{\upsigma}^{(\ps{\upomega}_{\alpha},l)}(f)\|
^{\ps{\upmu}_{(\alpha,l)}}.
\end{equation}
Lemma \ref{09171541}\eqref{10041833b} follows by \eqref{10071148}\,\&\,\eqref{10071148b}.
\end{remark}
\begin{remark}
Let 
$\lr{\mf{H},\uppi}{\ms{V}^{l^{-1}}}$ 
be a
nondegenerate
covariant representation of 
$\lr{\mc{A},\ms{S}_{\ms{F}_{\ps{\upomega}_{\alpha}}}^{G}}
{\upsigma_{\ms{F}_{\ps{\upomega}_{\alpha}}}}$,
and
set
$\ms{V}
\coloneqq
\ms{ad}(\uppi(\ms{v}(l)))\circ\ms{V}^{l^{-1}}
\circ\ms{ad}(l^{-1})
\up\ms{S}_{\upsigma^{\ast}(l)(
\ms{F}_{\ps{\upomega}_{\alpha}})}^{G}$,
then
since Lemma \ref{09191416} 
$\lr{\mf{H},\uppi}{\ms{V}}$ is a nondegenerate
covariant representation of 
$\lr{\mc{A},
\ms{S}_{\ms{F}_{\upsigma^{\ast}(l)(\ps{\upomega}_{\alpha})}}^{G}}
{\upsigma_{\ms{F}_{\upsigma^{\ast}(l)(\ps{\upomega}_{\alpha})}}}$.
Hence we can reload the chain of equalities \eqref{10071234}
in the opposite sense by replacing $\ms{U}$ by $\ms{V}$ and 
$\ms{U}^{l^{-1}}$ by $\ms{V}^{l^{-1}}$, for obtaining
\eqref{10071148b}.
\end{remark}
\begin{corollary}
\label{09191329}
Let 
$\ps{\upomega}:A\to\ms{E}_{\mc{A}}^{G}(\uptau)$
and $\ps{\upmu}$ be a Haar system associated with
$\ps{\upomega}$ and $\mf{A}$.
Then for any 
$(\alpha,l)\in A\times H$
there exists a unique extension 
by continuity
of 
$\ps{\upsigma}^{(\alpha,l)}$
to $\ms{B}_{\ps{\upmu}}^{\ps{\upomega},\alpha}$,
denoted by the same symbol, such that
$\ps{\upsigma}^{(\ps{\upomega}_{\alpha},l)}
\in
Inv_{\ms{CA}^{\ast}}
(\ms{B}_{\ps{\upmu}}^{\ps{\upomega},\alpha},
\ms{B}_{\ps{\upmu}}^{\ps{\upomega},\alpha,l})$.
\end{corollary}
\begin{proof}
Since
Lemma \ref{09171541}\eqref{10041833b} 
there exists a unique extension by continuity 
of
$\ps{\upsigma}^{(\ps{\upomega}_{\alpha},l)}$
from
$\ms{B}_{\ps{\upmu}}^{\ps{\upomega},\alpha}$
to
$\ms{B}_{\ps{\upmu}}^{\ps{\upomega},\alpha,l}$,
while since the first sentence of 
Lemma \ref{09171541}\eqref{10041833c} 
and 
Lemma \ref{09171541}\eqref{10041833b} 
applied to $\ps{\upomega}^{l}$ and to $\ps{\upmu}^{l}$
there exists a unique extension by continuity 
of 
$\ps{\upsigma}^{(\ps{\upomega}_{\alpha}^{l},l^{-1})}$
from 
$\ms{B}_{\ps{\upmu}}^{\ps{\upomega},\alpha,l}$
to
$\ms{B}_{\ps{\upmu}}^{\ps{\upomega},\alpha}$.
Since the second sentence of 
Lemma \ref{09171541}\eqref{10041833c} 
such two extensions are one the inverse of the other.
Finally the extension 
$\ps{\upsigma}^{(\ps{\upomega}_{\alpha},l)}$
is a $\ast-$homomorphism since 
Lemma \ref{09171541}\eqref{10041833a} 
and the norm continuity of the product and 
involution on 
$\ms{B}_{\ps{\upmu}}^{\ps{\upomega},\alpha}$.
\end{proof}
\begin{corollary}
\label{01031207}
Let $l,h\in H$ and $\alpha\in A$, then 
$\ps{\upsigma}^{(\upsigma^{\ast}(l)(\ps{\upomega}_{\alpha}),h)}
\circ
\ps{\upsigma}^{(\ps{\upomega}_{\alpha},l)}
=
\ps{\upsigma}^{(\ps{\upomega}_{\alpha},h\cdot l)}$.
\end{corollary}
\begin{proof}
The equality holds if restricted to
$\mc{C}_{c}(\s{\upomega}{\alpha}{},\mc{A})$
hence the statement follows by
Cor. \ref{09191329}.
\end{proof}
We are now in the position to state the second main result of 
this part namely 
\begin{theorem}
[Equivariance of representations]
\label{09191747}
Let $\mf{A}$ be inner implemented by 
$\ms{v}$,
$\ps{\upomega}:A\to\ms{E}_{\mc{A}}^{G}(\uptau)$,
$\ps{\upmu}$ be a Haar system associated with $\ps{\upomega}$ and $\mf{A}$, while $\pf{H}:A\to Rep_{c}(\mc{A})$ such that 
$\pf{H}_{\alpha}
\coloneqq
\lr{\mf{H}_{\alpha},\uppi_{\alpha}}{\Upomega_{\alpha}}$ 
is a cyclic representation of $\mc{A}$ associated with
$\ps{\upomega}_{\alpha}$, for all $\alpha\in A$.
Then for all $\alpha\in A$ and $l\in H$
the following diagram is commutative
\begin{equation*}
\xymatrix{
\ms{B}_{\ps{\upmu}}^{\ps{\upomega},\alpha,l}
\ar[rr]^{\mf{R}_{\pf{H},\ms{v},\alpha,l}^{\ps{\upmu}}}
& &
\mc{L}(\mf{H}_{\alpha})
\\
& &
\\
\ms{B}_{\ps{\upmu}}^{\ps{\upomega},\alpha}
\ar[uu]^{\ps{\upsigma}^{(\ps{\upomega}_{\alpha},l)}}
\ar[rr]_{\mf{R}_{\pf{H},\alpha}^{\ps{\upmu}}}
& &
\mc{L}(\mf{H}_{\alpha})
\ar[uu]_{\ms{ad}(\uppi_{\alpha}(\ms{v}(l)))}}
\end{equation*}
\end{theorem}
\begin{proof}
By construction 
$\mc{C}_{c}(\ms{S}_{\ms{F}_{\ps{\upomega}_{\alpha}}}^{G},\mc{A})$
is $\|\cdot\|^{\ps{\upmu}_{(\alpha,\un)}}-$dense
in 
$\ms{B}_{\ps{\upmu}}^{\ps{\upomega},\alpha}$,
moreover 
all the maps involved
in the statement are 
norm continuous and linear,
since Cor. \ref{09191329}
and the well-known fact that any $\ast-$homomorphism
between $C^{\ast}-$algebras is automatically continuous.
Thus it is sufficient to show the statement for the 
maps
$\ps{\upsigma}^{(\ps{\upomega}_{\alpha},l)}$
and
$\mf{R}_{\pf{H},\alpha}^{\ps{\upmu}}$
restricted on
$\mc{C}_{c}(\ms{S}_{\ms{F}_{\ps{\upomega}_{\alpha}}}^{G},
\mc{A})$.
Let 
$f\in\mc{C}_{c}(\ms{S}_{\ms{F}_{\ps{\upomega}_{\alpha}}}^{G},
\mc{A})$,
then
\begin{equation*}
\begin{aligned}
\bigl((\uppi_{\alpha}
\rtimes^{\ps{\upmu}_{(\alpha,l)}}
\ms{U}_{\pf{H}_{\alpha}^{(\ms{v},l)}})
\circ
\ps{\upsigma}^{(\ps{\upomega}_{\alpha},l)}
\bigr)(f)
&=
\\
\int
\uppi_{\alpha}(\ps{\upsigma}^{(\ps{\upomega}_{\alpha},l)}(f)(s))
\ms{U}_{\pf{H}_{\alpha}^{(\ms{v},l)}}(s)\,
d\,\ps{\upmu}_{(\alpha,l)}(s)
&=
\\
\int
\uppi_{\alpha}
(\ps{\upsigma}^{(\ps{\upomega}_{\alpha},l)}(f)(\ms{ad}(l)s))
\ms{U}_{\pf{H}_{\alpha}^{(\ms{v},l)}}(\ms{ad}(l)s)\,
d\,\ps{\upmu}_{(\alpha,\un)}(s)
&=
\\
\int
\uppi_{\alpha}\bigl((\upsigma(l)\circ f)(s)\bigr)
\uppi_{\alpha}(\ms{v}(l))
\ms{U}_{\pf{H}_{\alpha}}(s)
\uppi_{\alpha}(\ms{v}(l^{-1}))
d\,\ps{\upmu}_{(\alpha,\un)}(s)
&=
\\
\int
\uppi_{\alpha}(\ms{v}(l))
\uppi_{\alpha}(f(s))
\ms{U}_{\pf{H}_{\alpha}}(s)
\uppi_{\alpha}(\ms{v}(l^{-1}))
d\,\ps{\upmu}_{(\alpha,\un)}(s)
&=
\\
\ms{ad}(\uppi_{\alpha}(\ms{v}(l)))
\bigl(\int
\uppi_{\alpha}(f(s))
\ms{U}_{\pf{H}_{\alpha}}(s)
d\,\ps{\upmu}_{(\alpha,\un)}(s)
\bigr)
&=
\\
\big(\ms{ad}(\uppi_{\alpha}(\ms{v}(l)))
\circ
(\uppi_{\alpha}\rtimes^{\ps{\upmu}_{(\alpha,\un)}}
\ms{U}_{\pf{H}_{\alpha}})\bigr)(f).
\end{aligned}
\end{equation*}
Here the second equality follows since Rmk. \ref{09171400},
the third by Cor. \ref{09131225}, the fifth by 
$\ms{ad}(V)\in\mc{L}(\mc{L}_{s}(\mf{H}_{\alpha}))$, for
any $V\in\mc{U}(\mf{H}_{\alpha})$ and since the integration
is w.r.t. the strong operator topology.
\end{proof}
\section{Extensions on the multiplier algebra}
\label{multiplier}
We prove in Lemma \ref{12171928} that there is a unique canonical manner to extend a state of a 
$C^{\ast}-$algebra 
to a state of its multiplier algebra. 
Then we apply this result in Cor. \ref{12181715} to associate a state of $\mc{A}$ with a state of 
the crossed product $\mc{A}\rtimes_{\upsigma}^{\upmu}H$, where $\lr{\mc{A},H}{\upsigma}$ is a dynamical system such that 
$\mc{A}$ is unital and $\upmu\in\mc{H}(H)$. Cor. \ref{12181715} is used via Lemma \ref{12201040} 
in constructing an equivariant stability in Thm. \ref{01151104}.
Lemma \ref{01081112} and Lemma \ref{01091430} 
are used to obtain Thm. \ref{01141808} one of the auxiliary results needed 
to prove in Cor. \ref{11271221} the existence of the 
object part of a functor from $\ms{C}_{u}(H)$ to $\mf{G}(G,F,\uprho)$.
We prove in a functional analytic setting the convergence formula \eqref{12201544}, and the 
extension results Lemma \ref{01091430}, Lemma \ref{01081112}, 
Cor. \ref{12202039} and from Lemma \ref{01081736} to Cor. \ref{01101930}. 
Finally we prove in a different way Lemma \ref{12201745} and the convergence in \eqref{01111919}.
\begin{convention}
In the present section 
we fix a $C^{\ast}-$algebra $\mc{B}$
and 
for any nonzero positive trace class operator $\rho$ on 
a Hilbert space $\mf{H}$,
we use the convention of denoting
with $\omega_{\rho}$ the state
on $\mc{L}(\mf{H})$ defined by $\omega_{\rho}(a)\coloneqq\frac{Tr(\rho\,a)}{Tr(\rho)}$, 
for all $a\in\mc{L}(\mf{H})$.
Moreover by abuse of language 
for any unit vector $\Upomega\in\mf{H}$
let $\omega_{\Upomega}$ denote $\omega_{P_{\Upomega}}$,
where $P_{\Upomega}$ is the projector associated with the
closed subspace generated by $\Omega$.
Let $\psi\in\ms{E}_{\mc{B}}$ then 
$\pf{H}=\lr{\mf{H},\mc{R}}{\rho}$
is called a representation of $\mc{B}$ relative to $\psi$,
if $\lr{\mf{H}}{\mc{R}}$ is a nondegenerate representation
of $\mc{B}$ and 
$\rho$
is a positive trace class operator on $\mf{H}$
such that 
$\psi=\omega_{\rho}\circ\mc{R}$,
define
$\pf{H}^{-}\coloneqq\lr{\mf{H},\mc{R}^{-}}{\rho}$.
\end{convention}
\begin{lemma}
\label{12171928}
Let $\psi\in\ms{E}_{\mc{B}}$, then
\begin{enumerate}
\item
$(\exists\,!\psi^{-}\in\ms{E}_{\ms{M}(\mc{B})})
(\psi^{-}\circ\mf{i}^{\mc{B}}=\psi)$,
\label{12171928st1}
\item
if
$\pf{H}$ 
is a
representation 
of $\mc{B}$
relative to $\psi$
then
$\pf{H}^{-}$
is a representation of $\ms{M}(\mc{B})$
relative to $\psi^{-}$,
\label{12171928st3}
\end{enumerate}
\end{lemma}
\begin{remark}
Since Lemma \ref{12171928}\eqref{12171928st3}
if
$\pf{H}$ 
is a
cyclic representation 
of $\mc{B}$
associated with $\psi$
then
$\pf{H}^{-}$
is a 
cyclic representation of $\ms{M}(\mc{B})$
associated with $\psi^{-}$.
\end{remark}
\begin{proof}
[Proof of Lemma \ref{12171928}]
In this proof let $\mf{i}$ denote $\mf{i}^{\mc{B}}$.
Let $\pf{H}=\lr{\mf{H},\mc{R}}{\rho}$ 
be a representation 
relative to $\psi$,
set $\phi=\omega_{\rho}\circ\mc{R}^{-}$,
thus $\phi$ is a state of $\ms{M}(\mc{B})$
since $\mc{R^{-}}$ is a representation
of $\ms{M}(\mc{B})$ such that 
$\mc{R}^{-}(\un)=\un$.
Moreover by construction $\pf{H}^{-}$ 
is a representation associated with $\phi$,
while
$\phi\circ\mf{i}=\psi$, hence 
the existence part of st.\eqref{12171928st1}
and st.\eqref{12171928st3} follow.
Let us prove the uniqueness part of st.\eqref{12171928st1}.
Let $\psi^{-}\in\ms{E}_{\ms{M}(\mc{B})}$ such that 
$\psi^{-}\circ\mf{i}=\psi$, and 
$\lr{\mf{K},\mc{S}}{\Upomega}$ 
be a cyclic associated with $\psi^{-}$,
set
\begin{equation*}
\begin{cases}
\mf{K}_{0}=
\ov{\mc{S}(\mf{i}(\mc{B}))\Upomega},
\\
\mc{S}_{\up}:\ms{M}(\mc{B})\ni c\mapsto
\mc{S}(c)\up\mf{K}_{0},
\end{cases}
\end{equation*}
where the closure is w.r.t. the norm topology.
So $\mc{S}_{\up}$ is a representation of $\ms{M}(\mc{B})$
on $\mf{K}_{0}$, since $\mc{S}(c)$ is norm continuous
for any $c\in\ms{M}(\mc{B})$
and
$\mf{i}(\mc{B})$ is an ideal of $\ms{M}(\mc{B})$.
Next $\Upomega\in\mf{K}_{0}$ since 
a standard argument, \cite[p. $56$]{bra}, applied
to the state 
$\psi=\omega_{\Upomega}\circ\mc{S}\circ\mf{i}$
and the representation $\mc{S}\circ\mf{i}$. 
Hence 
$\pf{K}_{\up}\doteq\lr{\mf{K}_{0},\mc{S}_{\up}}{\Upomega}$
is a cyclic representation of 
$\ms{M}(\mc{B})$
such that 
$\omega_{\Upomega}\circ\mc{S}_{\up}
=\psi^{-}$, i.e.
\begin{equation}
\label{12181417}
\pf{K}_{\up} 
\text{ is a cyclic associated with }
\psi^{-}.
\end{equation}
Next let 
\begin{equation*}
\pf{K}_{0}
\doteq
\lr{\mf{K}_{0},\mc{S}_{\up}\circ\mf{i}}{\Upomega},
\end{equation*}
then 
$\ov{(\mc{S}_{\up}\circ\mf{i})(\mc{B})\Upomega}
=\ov{\mc{S}_{\up}(\mf{i}(\mc{B}))\Upomega}
=\mf{K}_{0}$,
hence $\pf{K}_{0}$ is a cyclic representation
of $\mc{B}$ since $\mc{S}_{\up}$ is a representation
of $\ms{M}(\mc{B})$.
Moreover 
$\omega_{\Upomega}\circ\mc{S}_{\up}\circ\mf{i}
=\psi^{-}\circ\mf{i}=\psi$ the first equality
coming since \eqref{12181417},
thus
\begin{equation}
\label{12181425}
\pf{K}_{0} 
\text{ is a cyclic associated with }
\psi.
\end{equation}
Next let 
$\psi^{j}$ be a state of $\ms{M}(\mc{B})$
such that $\psi^{j}\circ\mf{i}=\psi$
and
$\pf{K}^{j}=\lr{\mf{K}^{j},\mc{S}^{j}}{\Upomega^{j}}$
be a cyclic associated with $\psi^{j}$, 
for any $j\in\{1,2\}$.
Thus
since \eqref{12181425}  
and
the uniqueness modulo unitary equivalence
of the $GNS$ construction for states on a $C^{\ast}-$algebra
there exists a unique unitary operator 
$U:\mf{K}_{0}^{1}\to\mf{K}_{0}^{2}$ such that 
\begin{equation}
\label{12181510}
\begin{cases}
U\Upomega^{1}=\Upomega^{2}
\\
\mc{S}_{\up}^{2}\circ\mf{i}=
\ms{ad}(U)\circ(\mc{S}_{\up}^{1}\circ\mf{i}).
\end{cases}
\end{equation}
Moreover 
$(\mc{S}_{\up}^{j}\circ\mf{i})^{-}=\mc{S}_{\up}^{j}$
since \eqref{12181437} applied to the cyclic $\pf{K}_{\up}^{j}$, 
therefore since \eqref{12181510} and \eqref{12181442}
applied to the cyclic $\pf{K}_{0}^{j}$ for any $j\in\{1,2\}$,
we obtain for all $c\in\ms{M}(\mc{B})$ and $b\in\mc{B}$
\begin{equation*}
 \begin{aligned}
\mc{S}_{\up}^{2}(c)
(\mc{S}_{\up}^{2}\circ\mf{i})(b)\Upomega^{2}
&=
(\mc{S}_{\up}^{2}\circ\mf{i})^{-}(c)
(\mc{S}_{\up}^{2}\circ\mf{i})(b)\Upomega^{2}
\\
&=
(\mc{S}_{\up}^{2}\circ\mf{i})(d)\Upomega^{2}
\\
&=
U(\mc{S}_{\up}^{1}\circ\mf{i})(d)\Upomega^{1}
\\
&=
U\mc{S}_{\up}^{1}(c)
(\mc{S}_{\up}^{1}\circ\mf{i})(b)\Upomega^{1}
\\
&=
U\mc{S}_{\up}^{1}(c)U^{-1}
(\mc{S}_{\up}^{2}\circ\mf{i})(b)\Upomega^{2},
\end{aligned}
\end{equation*}
where $d=\mf{i}^{-1}(c\mf{i}(b))$.
Next $\pf{K}_{0}^{2}$ is cyclic so
$\mc{S}_{\up}^{2}=\ms{ad}(U)\circ\mc{S}_{\up}^{1}$,
therefore
\begin{equation*}
\omega_{\Upomega^{2}}\circ\mc{S}_{\up}^{2}
=
\omega_{\Upomega^{1}}\circ\mc{S}_{\up}^{1},
\end{equation*}
thus 
$\psi_{1}^{-}=\psi_{2}^{-}$
since \eqref{12181417}.
\end{proof}
\begin{definition}
\label{30051132}
Let $\psi\in\ms{E}_{\mc{B}}$, then
we call the canonical extension of $\psi$ to $\ms{M}(\mc{B})$
the unique state $\psi^{-}$ 
such that
$\psi^{-}\circ\mf{i}^{\mc{B}}=\psi$.
\end{definition}
\begin{corollary}
\label{12181715}
Let $\mf{A}=\lr{\mc{A},H}{\upsigma}$ be a dynamical system
such that $\mc{A}$ is unital, $\upmu\in\mc{H}(H)$ 
and $\psi$ a state of
$\mc{B}=\mc{A}\rtimes_{\upsigma}^{\upmu}H$.
Thus
$\psi^{-}\circ\mf{j}_{\mc{A}}^{\mc{B}}\in\ms{E}_{\mc{A}}$
and
\begin{enumerate} 
\item
if 
$\lr{\mf{H},\mc{R}}{\rho}$ is
a representation relative to $\psi$
then
$\lr{\mf{H},\uppi}{\rho}$
is a representation relative to 
$\psi^{-}\circ\mf{j}_{\mc{A}}^{\mc{B}}$,
\label{12181715st1}
\item
if $\lr{\mf{H},\mc{R}}{\Upomega}$ is
a cyclic representation associated with $\psi$
then
$\lr{\mf{H}_{\Upomega}^{\uppi},\uppi^{\up}}{\Upomega}$
is a cyclic representation associated with
$\psi^{-}\circ\mf{j}_{\mc{A}}^{\mc{B}}$.
\label{12181715st2}
\end{enumerate}
Here
$\lr{\mf{H},\uppi}{U}$ 
is the covariant representation
of $\mf{A}$ associated with $\mc{R}$,
$\mf{H}_{\Upomega}^{\uppi}
\coloneqq\ov{\uppi(\mc{A})\Upomega}$
and
$\uppi^{\up}:\mc{A}\ni 
a\mapsto\uppi(a)\up\mf{H}_{\Upomega}^{\uppi}$.
\end{corollary}
\begin{proof}
Let 
$\mc{B}=\mc{A}\rtimes_{\upsigma}^{\upmu}H$. 
$\psi^{-}$ is a state of $\ms{M}(\mc{B})$
since Lemma \ref{12171928}\eqref{12171928st1},
while $\mf{j}_{\mc{A}}^{\mc{B}}(\un)=\un$ 
since $\mf{j}_{\mc{A}}^{\mc{B}}$ is nondegenerate,
thus
$(\psi^{-}\circ\mf{j}_{\mc{A}}^{\mc{B}})(\un)=\un$ 
moreover $\psi^{-}\circ\mf{j}_{\mc{A}}^{\mc{B}}$ is positive
hence it is a state of $\mc{A}$.
Let
$\lr{\mf{H},\mc{R}}{\rho}$ be
a representation relative to $\psi$
which there exists since $\psi$ is a state,
thus
$\uppi$ is nondegenerate since $\mc{R}$ it is so, 
moreover
\begin{equation}
\label{12201219}
\begin{aligned}
\psi^{-}\circ\mf{j}_{\mc{A}}^{\mc{B}}
&=
\omega_{\rho}\circ\mc{R}^{-}
\circ\mf{j}_{\mc{A}}^{\mc{B}}
\\
&=
\omega_{\rho}\circ\uppi,
\end{aligned}
\end{equation}
where the first equality follows 
since Lemma \ref{12171928}\eqref{12171928st3}
and the second one since \eqref{12191452}, so st.\eqref{12181715st1} follows.
Next if $\lr{\mf{H},\mc{R}}{\Upomega}$
is a cyclic representation associated
to $\psi$, then
$\psi^{-}\circ\mf{j}_{\mc{A}}^{\mc{B}}=\omega_{\Upomega}\circ\uppi$
since \eqref{12201219}
therefore
$\Omega\in\mf{H}_{\Upomega}^{\uppi}$
since the argument in \cite[p.56]{bra}
applied to the state 
$\psi^{-}\circ\mf{j}_{\mc{A}}^{\mc{B}}$
and the representation $\uppi$,
so st.\eqref{12181715st2} follows.
\end{proof}
The next Lemma \ref{01081112} together Lemma \ref{01091430} are 
important in showing Thm. \ref{01141808} one of the auxiliary 
results used in the proof of Cor. \ref{11271221} were we construct the object part of a functor 
from $\ms{C}_{u}(H)$ to $\mf{G}(G,F,\uprho)$.
\begin{lemma}
\label{01081112}
Let $\lr{\mc{A},H}{\upsigma}$ be a dynamical system, 
$\upmu\in\mc{H}(H)$, $\mc{B}=\mc{A}\rtimes_{\upsigma}^{\upmu}H$.
Then for any $a\in\mc{A}$ the following is a commutative diagram
\begin{equation}
\label{01081112I}
\xymatrix{
\ms{M}(\mc{B})
\ar[rr]^{\ms{ev}_{a}(\mf{i}^{\ms{M}(\mc{B})}\circ
\mf{j}_{\mc{A}}^{\mc{B}})}
& &
\ms{M}(\mc{B})
\\
& &
\\
\mc{B}
\ar[uu]^{\mf{i}^{\mc{B}}}
\ar[rr]_{\ms{ev}_{a}(\mf{j}_{\mc{A}}^{\mc{B}})}
& &
\mc{B}
\ar[uu]_{\mf{i}^{\mc{B}}}}
\end{equation}
in particular 
for all $f\in\mc{C}_{c}(H,\mc{A})$
the following diagram is commutative 
\begin{equation}
\label{01081112II}
\xymatrix{
\mc{B}
\ar[rr]^{\mf{j}_{\mc{A}}^{\mc{B}}(a)}
& &
\mc{B}
\\
& &
\\
\mc{B}
\ar[uu]^{\mf{i}^{\mc{B}}(f)}
\ar[uurr]_{\mf{i}^{\mc{B}}(\mf{j}_{\mc{A}}^{\mc{B}}(a)(f))}
& &}
\end{equation}
\end{lemma}
\begin{proof}
\eqref{01081112I}
trivially implies
\eqref{01081112II},
moreover all the maps involved in \eqref{01081112I} are norm continuous 
and $\mc{C}_{c}(H,\mc{A})$ is norm dense in $\mc{B}$
hence 
\eqref{01081112II}
implies
\eqref{01081112I},
thus let us proof
\eqref{01081112II}.
Let $f\in\mc{C}_{c}(H,\mc{A})$ 
and  
let here
$\mf{i}$ and $\mf{j}$
denote 
$\mf{i}^{\mc{B}}$ and $\mf{j}_{\mc{A}}^{\mc{B}}$
respectively,
thus for all $g\in\mc{C}_{c}(H,\mc{A})$ and $s\in H$
\begin{equation*}
\begin{aligned}
(\mf{j}(a)\circ\mf{i}(f))(g)(s)
&=
\mf{j}(a)(f\ast^{\upmu}g)(s)
\\
&=
a \int f(r)\upsigma(r)(g(r^{-1}s))d\,\upmu(r)
\\
&=
\int a f(r)\upsigma(r)(g(r^{-1}s))d\,\upmu(r)
\\
&=
(\mf{j}(a)(f)\ast^{\upmu}g)(s)
=
\mf{i}(\mf{j}(a)(f))(g)(s).
\end{aligned}
\end{equation*}
\end{proof}
\begin{lemma}
\label{12201417}
Let $\lr{\mc{A},H}{\upsigma}$ 
be a dynamical system, 
$\upmu\in\mc{H}(H)$ and 
$\{E_{\beta}\}_{\beta\in C}$
be an approximate identity of $\mc{A}$. 
Let $\mc{B}$ denote $\mc{A}\rtimes_{\upsigma}^{\upmu}H$
then $\lim_{\beta}\mf{i}^{\ms{M}(\mc{B})}(\mf{j}(E_{\beta}))=\un$
w.r.t. the topology on 
$\ms{M}(\ms{M}(\mc{B}))$ of simple convergence in $\mf{i}^{\mc{B}}(\mc{B})$, i.e. for all $m\in\mf{i}^{\mc{B}}(\mc{B})$
\begin{equation}
\label{12201544}
\lim_{\beta\in C}
\|\mf{j}_{\mc{A}}^{\mc{B}}(E_{\beta})m-m\|_{\ms{M}(\mc{B})}=0.
\end{equation}
\end{lemma}
\begin{proof}
Let 
$\mf{A}=\lr{\mc{A},H}{\upsigma}$, 
$\mf{j}=\mf{j}_{\mc{A}}^{\mc{B}}$,
$\mf{i}=\mf{i}^{\mc{B}}$
and 
$\|\cdot\|_{\mc{B}}$ be the universal norm on $\mc{B}$.
We have
$supp(\mf{i}^{\mc{A}}(E_{\beta})\circ f)
\subseteq supp(f)$
for any $f\in\mc{C}_{c}(H,\mc{A})$ and $\beta\in C$
since $E_{\beta}$ is linear
and $supp(f)\doteq\ov{\complement f^{-1}(\ze)}$, 
hence
\begin{equation}
\label{12261235}
\begin{aligned}
\sup_{l\in supp(f)}\|(E_{\beta}f-f)(l)\|
&=
\sup
\{\|(E_{\beta}f-f)(l)\|\,
\mid\,
l\in supp(\mf{i}^{\mc{A}}(E_{\beta})\circ f)\cup supp(f)
\}
\\
&=
\sup_{l\in H}\|(E_{\beta}f-f)(l)\|.
\end{aligned}
\end{equation}
Since the definition of the approximate identity 
we deduce that
$\lim_{\beta}\mf{i}^{\mc{A}}(E_{\beta})=\un$ 
w.r.t. the topology on $\ms{M}(\mc{A})$
of simple convergence in $\mc{A}$,
moreover 
$\|\mf{i}^{\mc{A}}(E_{\beta})\|_{\ms{M}(\mc{A})}
\leq 1$
since $\mf{i}^{\mc{A}}$ is an isometry into its range,
next the unit ball of $\ms{M}(\mc{A})$ 
is clearly a bounded subset of $\ms{M}(\mc{A})$ 
w.r.t. the topology
of simple convergence in $\mc{A}$,
hence it is equicontinuous 
according to \cite[$III.25$ Thm. $1$]{BourTVS}.
Therefore we can apply
\cite[$III.17$ Prp. $5(2,3)$]{BourTVS}
and deduce that
$\lim_{\beta}\mf{i}^{\mc{A}}(E_{\beta})=\un$ 
w.r.t. the topology on $\ms{M}(\mc{A})$
of uniform convergence in
compact subsets of $\mc{A}$,
i.e.
$\lim_{\beta}\sup_{a\in K}\|E_{\beta}a-a\|_{\mc{A}}=0$,
for any compact subset $K$ of $\mc{A}$.
Next
$f(Q)$ is a compact subset of $\mc{A}$
for any $f\in\mc{C}_{c}(H,\mc{A})$
and any compact subset $Q$ of $H$,
therefore
\begin{equation*}
\lim_{\beta}\sup_{l\in Q}
\left\|\bigl(\mf{j}(E_{\beta})(f)-f\bigr)(l)\right\|_{\mc{A}}
=
\lim_{\beta}\sup_{l\in Q}
\|E_{\beta}f(l)-f(l)\|_{\mc{A}}
=0.
\end{equation*}
In particular 
$\lim_{\beta}\sup_{l\in H}\left\|\bigl(\mf{j}(E_{\beta})(f)-f\bigr)(l)\right\|_{\mc{A}}=0$ 
since \eqref{12261235} and $supp(f)$ is compact, 
so $\lim_{\beta}\mf{j}(E_{\beta})(f)=f$ w.r.t. the $L_{\upmu}^{1}-$norm topology.
Next $\|\cdot\|_{\mc{B}}$ is majorized by the $L_{\upmu}^{1}-$norm, see \cite[Lemma 2.27]{will}, therefore 
\begin{equation}
\label{01111757}
\lim_{\beta}\|\mf{j}(E_{\beta})(f)-f\|_{\mc{B}}=0\qquad
\forall f\in
\mc{C}_{c}(H,\mc{A}).
\end{equation}
Next $\mf{i}$ is an isometry onto its range thus by \eqref{01111757}\,\&\,\eqref{01081112I} 
we obtain
\begin{equation}
\label{12202126}
\lim_{\beta}
\mf{i}^{\ms{M}(\mc{B})}(\mf{j}(E_{\beta}))
=\un,
\end{equation}
w.r.t. the topology on 
$\ms{M}(\ms{M}(\mc{B}))$
of simple convergence in 
$\mf{i}(\mc{C}_{c}(H,\mc{A}))$.
Next by construction, 
see \cite[Lemma 2.27]{will}, 
$\mc{C}_{c}(H,\mc{A})$ is dense in $\mc{B}$ 
w.r.t. the $\|\cdot\|_{\mc{B}}-$topology
hence 
$\mf{i}(\mc{C}_{c}(H,\mc{A}))$ 
is dense in 
$\mf{i}(\mc{B})$,
moreover 
$\|\mf{i}^{\ms{M}(\mc{B})}
(\mf{j}(E_{\beta}))\|_{\ms{M}(\ms{M}(\mc{B}))}
\leq 1$
for any $\beta\in C$ since 
$\mf{i}^{\ms{M}(\mc{B})}\circ\mf{j}$
is an isometry.
Next the unit ball of
$\ms{M}(\ms{M}(\mc{B}))$
is clearly a bounded subset 
w.r.t. the topology on 
$\ms{M}(\ms{M}(\mc{B}))$
of simple convergence in 
$\ms{M}(\mc{B})$
hence it is equicontinuous
according to \cite[$III.25$ Thm. $1$]{BourTVS}.
Therefore since \eqref{12202126} 
we can apply
\cite[$III.17$ Prp. $5(1,2)$]{BourTVS}
and the statement follows.
\end{proof}
\begin{lemma}
\label{12201745}
Let $\mc{A}$ and $\mc{B}$
be $C^{\ast}-$algebras, $\ms{X}$ be a Hilbert
$\mc{B}-$module and $\uppi:\mc{A}\to\mc{L}(\ms{X})$ 
a $\ast-$homomorphism.
Then $\uppi$ maps approximate identities of $\mc{A}$
into approximate identities of $\uppi(\mc{A})$,
if in addition $\uppi$ is nondegenerate
and $\{E_{\beta}\}_{\beta\in C}$
is an approximate identity of $\mc{A}$,
then
$\lim_{\beta}\uppi(E_{\alpha})=\un$
w.r.t. the topology on $\mc{L}(\ms{X})$ of simple convergence.
\end{lemma}
\begin{proof}
The first sentence of the statement follows since $\uppi$
is norm continuous with norm less or equal to $1$
and since it is order preserving.
If $\uppi$ is nondegenerate
then
$\ms{X}=\ov{span}\{\uppi(a)x\mid a\in\mc{A},x\in\ms{X}\}$
thus
$\lim\uppi(E_{\beta})=\un$ 
w.r.t. the topology on $\mc{L}(\ms{X})$ of simple convergence in a total subset of $\ms{X}$,
since the first sentence of the statement.
Next for any $\beta\in C$ the 
$\uppi(E_{\beta})$ lies in the unit ball of $\mc{L}(\ms{X})$
which is a bounded set 
w.r.t. the topology of simple convergence in $\ms{X}$
hence equicontinuous 
by \cite[$III.25$ Thm. $1$]{BourTVS}.
Therefore we can apply
\cite[$III.17$ Prp. $5(1,2)$]{BourTVS}
and deduce that
$\lim\uppi(E_{\beta})=\un$ 
w.r.t. the topology on $\mc{L}(\ms{X})$ of simple convergence in 
$\ms{X}$.
\end{proof}
\begin{remark}
\label{12201417bis}
We can deduce for any $a\in\mc{B}$
\begin{equation}
\label{01111919}
\lim_{\beta\in C}\|
\mf{j}_{\mc{A}}^{\mc{B}}(E_{\beta})a-a\|_{\mc{B}}=0,
\end{equation}
as an application of Lemma \ref{12201745}
to the Hilbert $\mc{B}-$module 
$\ms{X}=\mc{B}$ and to the nondegenerate
homomorphism $\uppi=\mf{j}_{\mc{A}}^{\mc{B}}$.
Thus since Lemma \ref{01081112} 
and since $\mf{i}^{\mc{B}}$
is an isometry into its range we obtain 
\eqref{12201544}. 
Viceversa we can use Lemma \ref{01081112} 
and \eqref{12201544} to obtain 
\eqref{01111919}.
\end{remark}
\begin{corollary}
\label{12202039}
Let $\mc{A}$ be a $C^{\ast}-$algebra, $(\mf{H},\mc{R})$
a nondegenerate representation of $\mc{A}$
and $\rho$ a nonzero positive trace class operator on $\mf{H}$.
Then $\omega_{\rho}\circ\mc{R}\in\ms{E}_{\mc{L}(\mf{H})}$.
\end{corollary}
\begin{proof}
Let $\{E_{\beta}\}_{\beta\in C}$ be an approximate identity 
of $\mc{A}$ 
then
$\|\mc{R}(E_{\beta})\|\leq 1$ for all $\beta\in C$
and
$\lim_{\beta}\mc{R}(E_{\beta})=\un$ weakly,
since Lemma \ref{12201745} and the strong operator 
topology is stronger than the weak operator one. 
Therefore
$\lim_{\beta}\mc{R}(E_{\beta})=\un$ $\sigma-$weakly
since \cite[Prp. $2.4.2$]{bra}, so
\begin{equation}
\label{12202054}
\lim_{\beta}\omega_{\rho}(\mc{R}(E_{\beta}))
=\omega_{\rho}(\un)=1,
\end{equation}
since $\omega_{\rho}$ is $\sigma-$weakly continuous,
see for example \cite[Thm. $2.4.21$]{bra}.
Next $\omega_{\rho}\circ\mc{R}$ is positive hence the statement
follows since \eqref{12202054} and \cite[Prp. $2.3.11$]{bra}.
\end{proof}
\begin{remark}
\label{12211131}
Under the hypotheses of Cor. \ref{12202039}
and
since 
$Tr(\rho)\,\omega_{\rho}=Tr\circ L_{\rho}$
we have that
$\|Tr\circ L_{\rho}\circ\mc{R}\|=Tr(\rho)$.
\end{remark}
\begin{lemma}
\label{01091430}
Let $\mc{A}$ be a $C^{\ast}-$algebra then the map
$\ms{M}(\mc{A})\ni 
u\mapsto\mf{i}^{\ms{M}(\mc{A})}(u)\up\mc{K}(\mc{A})$
is a $\ast-$isomorphism of $\ms{M}(\mc{A})$ onto
$\ms{M}(\mc{K}(\mc{A}))$.
\end{lemma}
\begin{proof}
$(\ms{M}(\mc{A}),Id\up\mc{K}(\mc{A}))$ is a maximal unitization
of $\mc{K}(\mc{A})$ since \cite[Cor. $2.54$]{rw}, 
hence by \cite[proof. of Thm. $2.47$]{rw} 
there exists a unique $\ast-$homomorphism 
$\Phi:\ms{M}(\mc{A})\to\ms{M}(\mc{K}(\mc{A}))$ such that
$\Phi\circ Id\up\mc{K}(\mc{A})=\mf{i}^{\mc{K}(\mc{A})}$,
moreover $\Phi$ is a $\ast-$isomorphism.
Next 
$\mf{i}^{\ms{M}(\mc{A})}(u)\up\mc{K}(\mc{A})
=\mf{i}^{\mc{K}(\mc{A})}(u)$ for any $u\in\mc{K}(\mc{A})$
therefore the statement follows since the above uniqueness.
\end{proof}
Let us end this section by proving useful results concerning the extension of suitable morphisms to multiplier algebras.
\begin{lemma}
\label{01081736}
Let $\mc{A}$, $\mc{B}$ and $\mc{C}$
be $C^{\ast}-$algebras, 
$\ms{X}$ and $\ms{Y}$ be a 
Hilbert $\mc{B}-$module 
and 
Hilbert $\mc{C}-$module 
respectively,
$\upbeta:\mc{L}(\ms{X})\to\mc{L}(\ms{Y})$
and
$\upalpha:\mc{A}\to\mc{L}(\ms{X})$
be $\ast-$homomorphisms such that 
$\upbeta$ is nondegenerate.
If
$\upalpha$ is surjective,
or 
$\upalpha(\mc{A})$ is strictly dense in $\mc{L}(\ms{X})$
and 
$\upbeta(\mc{K}(\ms{X}))\supseteq\mc{K}(\ms{Y})$,
then $\upbeta\circ\upalpha$ is nondegenerate
and 
$(\upbeta\circ\upalpha)^{-}
=\upbeta\circ\upalpha^{-}$.
\end{lemma}
\begin{proof}
Let $\updelta=\upbeta\circ\upalpha$ and $y\in\ms{Y}$.
Case $\upalpha$ surjective.
Of course
$\upbeta(\mc{L}(\ms{X}))=\updelta(\mc{A})$, so
\begin{equation}
\label{01211955}
\upbeta(\mc{L}(\ms{X}))\ms{Y}
\subseteq
\ov{span}\{\updelta(a)z\mid a\in\mc{A},z\in\ms{Y}\},
\end{equation}
and the first sentence of the statement follows since $\upbeta$ is 
nondegenerate.
Remaining case.
$\upbeta$ is strictly 
continuous since it is norm continuous and 
$\upbeta(\mc{K}(\ms{X}))\supseteq\mc{K}(\ms{Y})$.
Moreover 
$\upbeta$ is continuous w.r.t. the strict topology 
on $\mc{L}(\ms{X})$ 
and the strong operator topology on $\mc{L}(\ms{Y})$
since the strict topology on $\mc{L}(\ms{Y})$ is stronger than the $\ast-$strong topology (\cite[Prp. $C.7$]{rw})
so stronger than the strong operator topology.
Then since the hypothesis it follows that $\upbeta(\mc{L}(\ms{X}))\subseteq\ov{\updelta(\mc{A})}$ 
closure w.r.t. the strong operator topology on $\mc{L}(\ms{Y})$, so \eqref{01211955} follows. 
Then $\beta\circ\alpha$ is nondegenerate and $(\upbeta\circ\upalpha)^{-}\circ\mf{i}^{\mc{A}}=\upbeta\circ\upalpha$ since \eqref{01081715}.
Moreover $\upalpha$ is nondegenerate 
in both cases, indeed 
$\un\in\mc{L}(\ms{X})$, while the strict topology on $\mc{L}(\ms{X})$ 
is stronger than the strong operator topology. 
So 
$\upalpha^{-}\circ\mf{i}^{\mc{A}}=\upalpha$,
then the equality follows 
since the uniqueness of which in \eqref{01081715}.
\end{proof}
More in general we can state
\begin{proposition}
\label{01101756}
Let $\mc{A}$, $\mc{B}$ and $\mc{C}$
be $C^{\ast}-$algebras, 
$\ms{X}$ and $\ms{Y}$ be a 
Hilbert $\mc{B}-$module 
and 
Hilbert $\mc{C}-$module 
respectively,
$\upbeta:\mc{L}(\ms{X})\to\mc{L}(\ms{Y})$
and
$\upalpha:\mc{A}\to\mc{L}(\ms{X})$
be $\ast-$homomorphisms.
If $\un_{\ms{Y}}$ belongs to the strong operator 
closure of the set $(\upbeta\circ\upalpha)(\mc{A})$
and $\upalpha$ is nondegenerate, 
then $\upbeta\circ\upalpha$ is nondegenerate
and 
$(\upbeta\circ\upalpha)^{-}
=\upbeta\circ\upalpha^{-}$.
\end{proposition}
\begin{proof}
$\upalpha^{-}\circ\mf{i}^{\mc{A}}=\upalpha$
since
\eqref{01081715}, while
since the hypothesis there exists a net $\{a_{i}\}$ in $\mc{A}$
such that 
$y=\lim_{i}(\upbeta\circ\upalpha)(a_{i})y$
for all $y\in\ms{Y}$, thus $\upbeta\circ\upalpha$ is nondegenerate.
Therefore
$(\upbeta\circ\upalpha)^{-}\circ\mf{i}^{\mc{A}}
=
\upbeta\circ\upalpha$ 
and the equality in the statement follows
since the uniqueness of which in
\eqref{01081715}.
\end{proof}
\begin{remark}
\label{01101755}
Under the notations of Prp. \ref{01101756}
we obtain the same statement if $\upalpha$ is surjective
and $\upbeta$ is nondegenerate, indeed in such a case 
$\upbeta\circ\upalpha$ is nondegenerate.
Clearly we obtain the same statement if we only require
$\upalpha$ and $\upbeta\circ\upalpha$ to be nondegenerate.
\end{remark}
\begin{corollary}
\label{01101820}
Let $\mc{A}$, $\mc{B}$ and $\mc{C}$
be $C^{\ast}-$algebras, 
$\ms{Y}$ be a 
Hilbert $\mc{C}-$module 
$\upbeta:\mc{B}\to\mc{L}(\ms{Y})$
and
$\upalpha:\mc{A}\to\mc{B}$
be $\ast-$homomorphisms.
If $\un_{\ms{Y}}$ belongs to the strong operator 
closure of the set $(\upbeta\circ\upalpha)(\mc{A})$
and $\upalpha$ is surjective,
then $\upbeta\circ\upalpha$ is nondegenerate
and 
$(\upbeta\circ\upalpha)^{-}
=\upbeta^{-}\circ(\mf{i}^{\mc{B}}\circ\upalpha)^{-}$.
\end{corollary}
\begin{proof}
Since $\mf{i}^{\mc{B}}$
is nondegenerate 
and $\upalpha$ is surjective
then 
$\mf{i}^{\mc{B}}\circ\upalpha$
is nondegenerate,
moreover 
$\upbeta^{-}\circ\mf{i}^{\mc{B}}\circ\upalpha
=
\upbeta\circ\upalpha$.
Hence we can apply Prp. \ref{01101756} to the maps
$\upbeta^{-}$ and $\mf{i}^{\mc{B}}\circ\upalpha$,
and the statement follows.
\end{proof}
\begin{corollary}
\label{01101930}
Let $\mc{A}$, $\mc{B}$ and $\mc{C}$
be $C^{\ast}-$algebras, 
$\ms{Y}$ be a 
Hilbert $\mc{C}-$module 
$\upbeta:\mc{B}\to\mc{L}(\ms{Y})$
and
$\upalpha:\mc{A}\to\mc{B}$
be $\ast-$homomorphisms.
If $\upbeta\circ\upalpha$ is nondegenerate
and $\upalpha$ is surjective
then 
and 
$(\upbeta\circ\upalpha)^{-}
=\upbeta^{-}\circ(\mf{i}^{\mc{B}}\circ\upalpha)^{-}$.
\end{corollary}
\begin{proof}
Since $\mf{i}^{\mc{B}}$
is nondegenerate 
and 
$\upalpha$ is surjective
then 
$\mf{i}^{\mc{B}}\circ\upalpha$
is nondegenerate, 
moreover 
$\upbeta^{-}\circ\mf{i}^{\mc{B}}\circ\upalpha
=
\upbeta\circ\upalpha$.
Thus the statement follows since Rmk. \ref{01101755}.
\end{proof}
\part{Stability}
\label{07301106}
\section{Introduction}
\label{introII}
In this part we construct the category $\mf{G}(G,F,\uprho)$ of nucleon systems,
introduce its physical interpretation,
define nucleon-fragment doublets on a category $\mf{C}$ and extended $\mf{C}-$equivariant stabilities,
construct the canonical extended $\ms{C}_{u}(H)-$equivariant stability 
and the canonical nucleon-fragment doublet on $\ms{C}_{u}(H)$.
The part is organized as follows.
In section \ref{cat} we introduce the category $\mf{G}(G,F,\uprho)$, 
describe in section \ref{06171340} the concept of a state \emph{originated} via a phase
and introduce the physical interpretation of these data in section \ref{06171244}.
In section \ref{06241544} we define the main structure of the work namely 
the \emph{nucleon-fragment doublet on a category $\mf{C}$},
and introduce the auxiliary structures of equivariant stability and 
extended $\mf{C}-$equivariant stability.
Extended $\mf{C}-$equivariant stabilities are important because of 
with each of them easily we can associate a nucleon-fragment doublet 
on $\mf{C}$. 
Our tenet is to translate the physical concept of ``universality'' into the concept of 
``natural transformation'' in category theory.
Therefore our goal of resolving the universality claim 
justifies the introduction of nucleon-fragment doublets.
Indeed in this section we describe the symmetry properties of a doublet
and define the expanded equivariances of the 
$\mc{T}-$nucleon phase and $\mc{T}-$fragment state,
representing 
the $\mc{T}-$resolution of the equivariant form of the universality claim.
Then we describe the properties of invariance and the corresponding physical interpretation.
Section \ref{07101559} is dedicated to establish the existence of an equivariant stability.
In section \ref{func1} we construct the object part of a functor 
from $\ms{C}_{u}(H)$ to $\mf{G}(G,F,\uprho)$, where the former is a subcategory of $C^{\ast}-$dynamical systems
and equivariant morphisms, and in section \ref{func2} we complete the construction of the functor.
It is in section \ref{func2} that we state in our main theorem the existence of the 
canonical extended $\ms{C}_{u}(H)-$equivariant stability 
and the canonical nucleon-fragment doublet $\mc{T}_{\bullet}$ on $\ms{C}_{u}(H)$.
As a result we resolve the equivariant form of the universality claim
stated as $\mc{T}_{\bullet}-$resolution. 
In part \ref{07301107} 
we shall provide the compact equivariant form and the invariant form
of the universality claim, and as a result we will establish the universality of the global Terrell law.
\section{The category of nucleon systems}
\label{thermalinv}
Section \ref{cat} is dedicated to the definition of the category of nucleon systems,
whose physical interpretation is given in section \ref{06171244} by mean of the 
cardinal concept of origination of fragment states via a nucleon phase
developed in section \ref{06171340}.
In section \ref{06241544} we introduce the concept of
nucleon-fragment doublet on a category, the main structure of this work,
and state its symmetries. 
In the same section we define equivariant stabilities and 
extended $\mf{C}-$equivariant stabilities auxiliary objects to construct nucleon-fragment doublets.
In section \ref{07101559} we construct an equivariant stability. 
In the present section \ref{thermalinv} we assume fixed 
two locally compact topological groups $G$ and $F$,
a group homomorphism $\uprho:F\to Aut_{\ms{Gr}}(G)$ 
such that the map $(g,f)\mapsto\uprho_{f}(g)$
on $G\times F$ at values in $G$, is continuous,
moreover let $H$ denote $G\rtimes_{\uprho}F$.
\subsection{The category $\mf{G}(G,F,\uprho)$}
\label{cat}
\begin{definition}
[The category of dynamical systems]
\label{11211655}
For any locally compact group $V$
we define the category 
$\ms{C}(V)$ whose object set
is the set of the dynamical systems 
$\lr{\mc{A},V}{\upeta}$
such that $\mc{A}$ is unital,
while for any 
$\mf{A},\mf{B}\in Obj(\ms{C}(V))$
we define $Mor_{\ms{C}(V)}(\mf{A},\mf{B})$
the set of the surjective $(\mf{A},\mf{B})-$equivariant
morphisms with law of composition the map composition
and if $\mf{A}=\mf{B}$ the identity map as the identity morphism. 
Let $\ms{C}_{u}(V)$ denote the full subcategory of $\ms{C}(V)$ 
whose object set is the set of the dynamical systems $\mf{A}=\lr{\mc{A},V}{\upeta}$ such that $\mc{A}$ 
is a von Neumann algebra in its canonical standard form.
Let $\ms{v}^{\mf{A}}$ or $\ms{v}^{\upeta}$ denote the unique group morphism of $H$ into $\mc{U}(\mc{A})$ 
unitarily implementing 
and associated with $\upeta$ and to the canonical standard form of $\mc{A}$ according to \cite[Thm $9.1.15$]{tak2}. 
\end{definition}
\begin{definition}
[The category of nucleon systems]
\label{05301823}
Define 
$Obj(\mf{G}(G,F,\uprho))$ 
the set of the 
$\mc{G}=
\lr{\pf{T},\ms{I},\upbeta_{c},\ms{P},\mf{a}}
{\mf{e},\ps{\upvarphi},\ms{A},\uppsi,\mf{b},\ms{m},\mf{E}}$
such that
\begin{enumerate}
\item 
$\pf{T}$ is a set;
\item
$\ms{I}:\pf{T}\to\ms{set}$;
\item
$\upbeta_{c}\in\prod_{\mf{Q}\in\pf{T}}\ms{I}^{\mf{Q}}$;
\item
$\ms{P}\in\prod_{\mf{Q}\in\pf{T}}\mathscr{P}(\ms{I}^{\mf{Q}})$;
\item
$\mf{a}\in
\prod_{\mf{Q}\in\pf{T}}
\prod_{\alpha\in\ms{I}^{\mf{Q}}}
\ms{C}(H)
$;
\item
$\mf{e}\in\prod_{\mf{Q}\in\pf{T}}
\prod_{\alpha\in\ms{I}^{\mf{Q}}}
\ms{C}(\R)$;
\item
$\ps{\upvarphi}\in
\prod_{\mf{Q}\in\pf{T}}
\prod_{\alpha\in\ms{I}^{\mf{Q}}}
\ms{E}_{\mc{A}_{\alpha}^{\mf{Q}}}$;
\item
$\ms{A}\in\ms{Ab}$
\item
$\uppsi\in Mor_{\ms{Gr}}(H,Aut_{\ms{Ab}}(\ms{A}))$;
\label{08261212}
\item
$\mf{b}\in Mor_{\ms{Gr}}(H,Aut_{\ms{set}}(\pf{T}))$;
\label{08261213}
\item
$\ms{m}\in
Mor_{\ms{Ab}}(\ms{A},\prod_{\mf{Q}\in\pf{T}}
Mor_{\ms{set}}(\ms{P}^{\mf{Q}},\R))$.
\end{enumerate}
Here 
$\mf{a}_{\alpha}^{\mf{T}}=\lr{\mc{A}_{\alpha}^{\mf{T}},H}{\upeta_{\alpha}^{\mf{T}}}$, 
in addition let 
$\ms{F}_{\ps{\upvarphi}_{\alpha}^{\mf{T}}}$ denote $\ms{F}_{\ps{\upvarphi}_{\alpha}^{\mf{T}}}(\mf{a}_{\alpha}^{\mf{T}})$
for all $\mf{T}\in\pf{T}$ and $\alpha\in\ms{I}^{\mf{T}}$, and let $\mf{T}^{l}$ denote $\mf{b}(l)(\mf{T})$ for all $l\in H$.
Then we require for all $\mf{T}\in\pf{T}$ and $\alpha\in\mf{T}$
\begin{equation*}
\ps{\upvarphi}_{\alpha}^{\mf{T}}\in
\ms{E}_{\mc{A}_{\alpha}^{\mf{T}}}^{G}
(\uptau_{\upeta_{\alpha}^{\mf{T}}}),
\end{equation*}
that
$\ms{I}^{\mf{T}^{l}}=\ms{I}^{\mf{T}}$
and
\begin{equation}
\label{12051548}
\beta_{c}^{\mf{T}^{l}}=\beta_{c}^{\mf{T}},
\quad
\ms{ad}(\Pr_{2}(l))
(\ms{F}_{\ps{\upvarphi}_{\alpha}^{\mf{T}}})
=
\ms{F}_{\ps{\upvarphi}_{\alpha}^{\mf{T}^{l}}}.
\end{equation}
Moreover
\begin{thermal}
for all $\mf{T}\in\pf{T}$
\begin{equation}
\label{eqtherm}
\ms{P}^{\mf{T}}
=\{
\alpha\in\ms{I}^{\mf{T}}
\mid
\ms{F}_{\ps{\upvarphi}_{\alpha}^{\mf{T}}}
\supseteq
\ms{F}_{\ps{\upvarphi}_{\upbeta_{c}^{\mf{T}}}^{\mf{T}}}
\}.
\end{equation}
\end{thermal}
\begin{equivar}
For any $l\in H$ and 
$\ms{f}\in\ms{A}$
\footnote{
\eqref{12051444} is well-set since Rmk. \ref{01081803}}
\begin{equation}
\label{12051444}
\ms{ev}_{\ms{f}}
(\ms{m}\circ\uppsi(l))
=
\ms{ev}_{\ms{f}}(\ms{m})
\circ\mf{b}(l^{-1}).
\end{equation}
\end{equivar}
\begin{dyn}
For any $\mf{T}\in\pf{T}$ and $\alpha\in\ms{I}^{\mf{T}}$
we have
$\mf{e}_{\alpha}^{\mf{T}}=\lr{\mc{A}_{\alpha}^{\mf{T}},\R}{\ep_{\alpha}^{\mf{T}}}$
and
\begin{equation}
\label{eqdyn}
\begin{aligned}
\ps{\upvarphi}_{\alpha}^{\mf{T}}
\in
\ep_{\alpha}^{\mf{T}}-KMS,
\\
\ep_{\alpha}^{\mf{T}^{l}}
=
\ms{ad}(\eta_{\alpha}^{\mf{T}}(l))\circ\ep_{\alpha}^{\mf{T}},
\forall l\in H;
\end{aligned}
\end{equation}
\end{dyn}
\begin{integ}
$\ms{m}(\ms{f})(\mf{T})$
is a $\Z-$valued map,
for all 
$(\ms{f},\mf{T})\in\ms{A}\times\pf{T}$.
\end{integ}
\begin{stb}
$\mf{E}=\lr{\upmu}{\mf{u},\pf{H},\ms{D},\Upgamma,
\mf{v},\mf{w},\mf{z}}$
such that 
\begin{enumerate}
\item
$\upmu
\in
\prod_{\mf{Q}\in\pf{T}}\prod_{\beta\in\ms{P}^{\mf{Q}}}
\mc{H}(\ms{S}_{\beta}^{\mf{Q}}(\mc{G}))$,
\item
$\mf{u}
\in
\prod_{\mf{Q}\in\pf{T}}\prod_{\beta\in\ms{P}^{\mf{Q}}}
\in 
Mor_{\ms{Ab}}(\ms{A},\ms{K}_{0}(\mc{B}_{\beta}^{\mf{Q}}(\mc{G})^{+}))$,
\item
$\pf{H}
\in
\prod_{\mf{Q}\in\pf{T}}\prod_{\beta\in\ms{P}^{\mf{Q}}}
Rep_{c}(\mc{A}_{\beta}^{\mf{Q}})$
\item
$\pf{H}_{\alpha}^{\mf{T}}=
\lr{\mf{H}_{\alpha}^{\mf{T}}}{\uppi_{\alpha}^{\mf{T}},\Upomega_{\alpha}^{\mf{T}}}$ 
is a cyclic representation of $\mc{A}_{\alpha}^{\mf{T}}$
associated with $\ps{\upvarphi}_{\alpha}^{\mf{T}}$,
for all $\mf{T}\in\pf{T}$, $\alpha\in\ms{P}^{\mf{T}}$;
\item
$\ms{D},\Upgamma
\in
\prod_{\mf{Q}\in\pf{T}}\prod_{\beta\in\ms{P}^{\mf{Q}}}
Sd(\mf{H}_{\beta}^{\mf{Q}})$;
\item
$\mf{v}
\in 
\prod_{\mf{Q}\in\pf{T}}\prod_{\beta\in\ms{P}^{\mf{Q}}}
\prod_{l\in H}
U(\mf{H}_{\beta}^{\mf{Q}},\mf{H}_{\beta}^{\mf{Q}^{l}})$;
\item
$\mf{w}
\in 
\prod_{\mf{Q}\in\pf{T}}\prod_{\beta\in\ms{P}^{\mf{Q}}}
\prod_{l\in H}
Mor_{\ms{CA}^{\ast}}(\mc{B}_{\beta}^{\mf{Q}}(\mc{G}),\mc{B}_{\beta}^{\mf{Q}^{l}}(\mc{G}))$,
\item
$\mf{z}
\in 
\prod_{\mf{Q}\in\pf{T}}\prod_{\beta\in\ms{P}^{\mf{Q}}}
\prod_{l\in H}
Mor_{\ms{CA}^{\ast}}(\mc{A}_{\beta}^{\mf{Q}},\mc{A}_{\beta}^{\mf{Q}^{l}})$,
\item
for all $\mf{T}\in\pf{T}$, $\alpha\in\ms{P}^{\mf{T}}$,
\begin{enumerate}
\item
$\lr{\tilde{\pf{R}}_{\alpha}^{\mf{T}}(\mc{G})}
{\ms{D}_{\alpha}^{\mf{T}},\Upgamma_{\alpha}^{\mf{T}}}$
is an even $\theta-$summable $K-$cycle;
\item
for all $l,h\in H$ and 
$\mf{g}\in\{\mf{v},\mf{w},\mf{z}\}$ 
\begin{equation}
\label{01031206}
\begin{aligned}
\upeta_{\alpha}^{\mf{T}^{l}}(h)
\circ
\mf{z}_{\alpha}^{\mf{T}}(l)
&=
\mf{z}_{\alpha}^{\mf{T}}(l)
\circ
\upeta_{\alpha}^{\mf{T}}(h),
\\
\mf{g}_{\alpha}^{\mf{T}}(h\cdot l)
&=
\mf{g}_{\alpha}^{\mf{T}^{l}}(h)
\circ
\mf{g}_{\alpha}^{\mf{T}}(l)
\\
\mf{g}_{\alpha}^{\mf{T}}(\un)
&=
Id.
\end{aligned}
\end{equation}
\item
we have
\begin{equation}
\label{01031145d}
\begin{aligned}
\ms{D}_{\alpha}^{\mf{T}^{l}}
=
\mf{v}_{\alpha}^{\mf{T}}(l)\,
\ms{D}_{\alpha}^{\mf{T}}\,
\mf{v}_{\alpha}^{\mf{T}}(l)^{-1},
\\
\Upgamma_{\alpha}^{\mf{T}^{l}}
=
\mf{v}_{\alpha}^{\mf{T}}(l)\,
\Upgamma_{\alpha}^{\mf{T}}\,
\mf{v}_{\alpha}^{\mf{T}}(l)^{-1},
\end{aligned}
\end{equation}
while 
the following 
(\ref{01031145},\ref{01031145b},\ref{01031145c})
are commutative diagrams
\begin{equation}
\label{01031145}
\xymatrix{
\mc{B}_{\alpha}^{\mf{T}^{l}}(\mc{G})
\ar[rr]^{\mf{R}_{\alpha}^{\mf{T}^{l}}(\mc{G})}
& &
\mc{L}(\mf{H}_{\alpha}^{\mf{T}^{l}})
\\
& &
\\
\mc{B}_{\alpha}^{\mf{T}}(\mc{G})
\ar[uu]^{\mf{w}_{\alpha}^{\mf{T}}(l)}
\ar[rr]_{\mf{R}_{\alpha}^{\mf{T}}(\mc{G})}
& &
\mc{L}(\mf{H}_{\alpha}^{\mf{T}})
\ar[uu]_{\ms{ad}(\mf{v}_{\alpha}^{\mf{T}}(l))}}
\end{equation}
\begin{equation}
\label{01031145b}
\xymatrix{
\ms{A}
\ar[rr]^{\mf{u}_{\alpha}^{\mf{T}^{l}}}
& &
\ms{K}_{0}(\mc{B}_{\alpha}^{\mf{T}^{l}}(\mc{G})^{+})
\\
& &
\\
\ms{A}
\ar[uu]^{\uppsi(l)}
\ar[rr]_{\mf{u}_{\alpha}^{\mf{T}}}
& &
\ms{K}_{0}(\mc{B}_{\alpha}^{\mf{T}}(\mc{G})^{+})
\ar[uu]_{(\mf{w}_{\alpha}^{\mf{T}}(l)^{+})_{\ast}}}
\end{equation}
\begin{equation}
\label{01031145c}
\xymatrix{
&
\ms{M}(\mc{B}_{\alpha}^{\mf{T}^{l}}(\mc{G}))
& &
\mc{A}_{\alpha}^{\mf{T}^{l}}
\ar[ll]_{\mf{j}_{\alpha}^{\mf{T}^{l}}}
\\
& & &
\\
\ms{M}(\mc{B}_{\alpha}^{\mf{T}}(\mc{G}))
\ar[uur]^{
(\mf{i}_{\alpha}^{\mf{T}^{l}}
\circ
\mf{w}_{\alpha}^{\mf{T}}(l))^{-}}
& & &
\\
& & &
\\
& 
\mc{A}_{\alpha}^{\mf{T}}
\ar[uul]^{\mf{j}_{\alpha}^{\mf{T}}}
\ar[rr]_{\mf{z}_{\alpha}^{\mf{T}}(l)}
& & 
\mc{A}_{\alpha}^{\mf{T}^{l}}
\ar[uuuu]_{\upeta_{\alpha}^{\mf{T}^{l}}(l)}
} 
\end{equation}
\item
for all 
$\ms{f}\in\ms{A}$ 
we have
\begin{equation}
\label{12220022}
\ms{m}(\ms{f})(\mf{T},\alpha)
=
\lr{\mf{u}_{\alpha}^{\mf{T}}
(\ms{f})}
{\ms{ch}
\lr{\tilde{\pf{R}}_{\alpha}^{\mf{T}}(\mc{G})}
{\ms{D}_{\alpha}^{\mf{T}},\Upgamma_{\alpha}^{\mf{T}}}}_{(\mc{G},\mf{T},\alpha)}.
\end{equation}
\end{enumerate}
\end{enumerate}
\end{stb}
Here for any $\mf{T}\in\pf{T}$ and $\alpha\in\ms{P}^{\mf{T}}$
we set
\begin{equation}
\label{08040736}
\begin{aligned}
\ms{S}_{\alpha}^{\mf{T}}(\mc{G})
&\coloneqq
\s{\ps{\upvarphi}}{\alpha}{\mf{T}},
\\
\mc{B}_{\alpha}^{\mf{T}}(\mc{G})
&\coloneqq
\ms{B}_{\upmu_{\alpha}^{\mf{T}}}^{\ps{\upvarphi}_{\alpha}^{\mf{T}}}
(\mf{a}_{\alpha}^{\mf{T}}),
\\
\mf{i}_{\alpha}^{\mf{T}}
&\coloneqq
\mf{i}^{\mc{B}_{\alpha}^{\mf{T}}(\mc{G})},
\\
\mf{j}_{\alpha}^{\mf{T}}
&\coloneqq
\mf{j}_{\mc{A}_{\alpha}^{\mf{T}}}^
{\mc{B}_{\alpha}^{\mf{T}}(\mc{G})},
\\
\mf{R}_{\alpha}^{\mf{T}}(\mc{G})
&\coloneqq
\mf{R}_{\pf{H}_{\alpha}^{\mf{T}}}^{\upmu_{\alpha}^{\mf{T}}}
(\mf{a}_{\alpha}^{\mf{T}}),
\\
\pf{R}_{\alpha}^{\mf{T}}(\mc{G})
&\coloneqq
\pf{R}_{\pf{H}_{\alpha}^{\mf{T}}}^{\upmu_{\alpha}^{\mf{T}}}
(\mf{a}_{\alpha}^{\mf{T}}),
\\
\lr{\cdot}{\cdot}_{(\mc{G},\mf{T},\alpha)}
&\coloneqq
\lr{\cdot}
{\cdot}_{(\mc{B}_{\alpha}^{\mf{T}}(\mc{G}))^{+}}.
\end{aligned}
\end{equation}
We call 
any object $\mc{G}$ of $\mf{G}(H)$
a nucleon system with symmetry group $H$,
or simply nucleon system,
moreover we call
$\pf{T}$ and $\ms{m}$ 
the set of 
operations 
and
the mean value map associated with $\mc{G}$
respectively.
\end{definition}
In the following definition and remark let
$\mc{G}=\lr{\pf{T},\ms{I},\upbeta_{c},\ms{P},\mf{a}}
{\mf{e},\ps{\upvarphi},\ms{A},\uppsi,\mf{b},\ms{m},\mf{E}}$
be an object of $\mf{G}(G,F,\uprho)$,
$l\in H$, $\mf{T}\in\pf{T}$ and $\alpha\in\ms{P}^{\mf{T}}$,
where
$\mf{E}=\lr{\upmu}{\mf{u},\pf{H},\ms{D},\Upgamma,
\mf{v},\mf{w},\mf{z}}$.
\begin{definition}
\label{09051150}
Define
$\ms{V}(\mc{G})_{\alpha}^{\mf{T}}(l)
\in 
Mor_{\ms{CA}^{\ast}}(\mc{A}_{\alpha}^{\mf{T}},\mc{A}_{\alpha}^{\mf{T}^{l}})$
such that 
$\ms{V}(\mc{G})_{\alpha}^{\mf{T}}(l)
\coloneqq
\mf{z}_{\alpha}^{\mf{T}}(l)
\circ
\upeta_{\alpha}^{\mf{T}}(l)$.
\end{definition}
\begin{remark}
\label{01081803}
$\ms{P}^{\mf{T}}=\ms{P}^{\mf{T}^{l}}$
since \eqref{12051548}.
Let $\mf{g}\in\{\mf{v},\mf{w},\mf{z}\}$ then
by \eqref{01031206} 
we deduce that 
$\mf{g}_{\alpha}^{\mf{T}}(l)$ 
is bijective
and 
\begin{equation}
\label{09051254}
(\mf{g}_{\alpha}^{\mf{T}}(l)
)^{-1}=
\mf{g}_{\alpha}^{\mf{T}^{l}}(l^{-1}), 
\end{equation}
so
$\ms{ad}(\mf{z}_{\alpha}^{\mf{T}}(l))
\circ
\upeta_{\alpha}^{\mf{T}}
=
\upeta_{\alpha}^{\mf{T}^{l}}$,
then
\begin{equation}
\label{09051209}
\ms{V}(\mc{G})_{\alpha}^{\mf{T}}(h\cdot l)
=
\ms{V}(\mc{G})_{\alpha}^{\mf{T}^{l}}(h)
\circ
\ms{V}(\mc{G})_{\alpha}^{\mf{T}}(l),
\end{equation}
moreover 
$\ms{V}(\mc{G})_{\alpha}^{\mf{T}}(\un)
=
Id$
thus
\begin{equation}
\label{09051254b}
(\ms{V}(\mc{G})_{\alpha}^{\mf{T}}(l))^{-1}=
\ms{V}(\mc{G})_{\alpha}^{\mf{T}^{l}}(l^{-1}).
\end{equation}
Next 
$\mf{w}_{\alpha}^{\mf{T}}(l)$ 
is surjective being bijective
and
$\mf{i}_{\alpha}^{\mf{T}^{l}}$ is nondegenerate,
thus 
$\mf{i}_{\alpha}^{\mf{T}^{l}}
\circ
\mf{w}_{\alpha}^{\mf{T}}(l)$
is nondegenerate since Lemma \ref{01081736}. 
Therefore 
$(\mf{i}_{\alpha}^{\mf{T}^{l}}
\circ
\mf{w}_{\alpha}^{\mf{T}}(l))^{-}$
in \eqref{01031145c}
is well set and since \eqref{01081715}
satisfies
$
(\mf{i}_{\alpha}^{\mf{T}^{l}}
\circ
\mf{w}_{\alpha}^{\mf{T}}(l))^{-}
\circ
\mf{i}_{\alpha}^{\mf{T}}
=
\mf{i}_{\alpha}^{\mf{T}^{l}}
\circ
\mf{w}_{\alpha}^{\mf{T}}(l)$.
Finally
by an application of \eqref{01031145c}
we obtain
\begin{equation}
\label{01231919}
(\mf{i}_{\alpha}^{\mf{T}}
\circ
\mf{w}_{\alpha}^{\mf{T}^{l}}(l^{-1}))^{-}
\circ
\mf{j}_{\alpha}^{\mf{T}^{l}}
=
\mf{j}_{\alpha}^{\mf{T}}
\circ
\upeta_{\alpha}^{\mf{T}}(l^{-1})
\circ
\mf{z}_{\alpha}^{\mf{T}^{l}}(l^{-1}).
\end{equation}
\end{remark}
\begin{definition}
\label{11251056}
For any $i\in\{1,2,3\}$
let 
$\mc{G}^{i}=\lr{\pf{T}_{i},\ms{I}_{i},\upbeta_{c}^{i},
\ms{P}_{i},\mf{a}^{i}}
{\mf{e}^{i},\ps{\upvarphi}^{i},\ms{A}^{i},\uppsi^{i},\mf{b}^{i},
\ms{m}^{i},\mf{E}^{i}}
\in Obj(\mf{G}(G,F,\uprho))$.
Define
$Mor_{\mf{G}(G,F,\uprho)}(\mc{G}^{1},\mc{G}^{2})$
to be the set of the 
\begin{equation*}
(\mf{g},\mf{d})
\in
Mor_{\ms{Ab}}(\ms{A}^{1},\ms{A}^{2})
\times
Mor_{\ms{set}}(\pf{T}_{2},\pf{T}_{1})
\end{equation*}
such that 
$\ms{P}_{1}^{\mf{d}(\mf{T})}
=
\ms{P}_{2}^{\mf{T}}$
for all 
$\mf{T}\in\pf{T}_{2}$, 
and 
for all
$\ms{f}\in\ms{A}^{1}$
we have
\begin{equation*}
\ms{ev}_{\ms{f}}
(\ms{m}^{2}\circ\mf{g})
=
\ms{ev}_{\ms{f}}(\ms{m}^{1})
\circ\mf{d}.
\end{equation*}
Moreover for any 
$(\mf{g},\mf{d})\in 
Mor_{\mf{G}(G,F,\uprho)}(\mc{G}^{1},\mc{G}^{2})$
and
$(\mf{h},\mf{s})\in Mor_{\mf{G}(G,F,\uprho)}(\mc{G}^{2},\mc{G}^{3})$
define
\begin{equation}
\label{22052156}
(\mf{h},\mf{s})\circ(\mf{g},\mf{d})
\coloneqq
(\mf{h}\circ\mf{g},\mf{d}\circ\mf{s}).
\end{equation}
\end{definition}
Thus we have the following
\begin{proposition}
\label{25110856}
There exists a unique category $\mf{G}(G,F,\uprho)$
whose set of objects is $Obj(\mf{G}(G,F,\uprho))$
and for all objects $\mc{G}^{1}$ and $\mc{G}^{2}$
the set of the morphisms from $\mc{G}^{1}$ to $\mc{G}^{2}$
is
$Mor_{\mf{G}(G,F,\uprho)}(\mc{G}^{1},\mc{G}^{2})$
provided by the law of composition as defined in 
Def. \ref{11251056}.
\end{proposition}
\begin{proof}
For any $i\in\{1,2,3\}$ let $\mc{G}^{i}\in\mf{G}(G,F,\uprho)$,
moreover
$(\mf{g},\mf{d})
\in Mor_{\mf{G}(G,F,\uprho)}(\mc{G}^{1},\mc{G}^{2})$
and
$(\mf{h},\mf{s})
\in Mor_{\mf{G}(G,F,\uprho)}(\mc{G}^{2},\mc{G}^{3})$.
The identity morphism in 
$Mor_{\mf{G}(G,F,\uprho)}(\mc{G}^{1},\mc{G}^{1})$
is the couple composed by the identity maps on $\ms{A}^{1}$ and on $\pf{T}_{1}$. 
Let $\ms{f}\in\ms{A}^{1}$, then
\begin{equation*}
\begin{aligned}
\ms{ev}_{\ms{f}}
(\ms{m}^{3}\circ\mf{h}\circ\mf{g}
)
&=
\ms{ev}_{\mf{g}(\ms{f})}(\ms{m}^{3}
\circ\mf{h})
=
\ms{ev}_{\mf{g}(\ms{f})}(\ms{m}^{2})\circ\mf{s}
\\
&=
\ms{ev}_{\ms{f}}(\ms{m}^{2}\circ\mf{g})\circ\mf{s}
=
\ms{ev}_{\ms{f}}(\ms{m}^{1})\circ\mf{d}\circ\mf{s}.
\end{aligned}
\end{equation*}
Hence 
$(\mf{h},\mf{s})\circ(\mf{g},\mf{d})
\in Mor_{\mf{G}(G,F,\uprho)}(\mc{G}^{1},\mc{G}^{3})$,
it is easy to see that the composition is associative
hence the statement follows.
\end{proof}
\begin{convention}
\label{05061835}
Let $\mc{G}$ be an object of $\mf{G}(G,F,\uprho)$,
if the contrary is not stated, then we shall use the following notation
\begin{equation*}
\begin{aligned}
\mc{G}
&=
\lr{\pf{T}_{\mc{G}},
\ms{I}_{\mc{G}},\upbeta_{c}^{\mc{G}},
\ms{P}_{\mc{G}},\mf{a}_{\mc{G}}}
{\mf{e}_{\mc{G}},
\ps{\upvarphi}^{\mc{G}},\ms{A}_{\mc{G}},
\uppsi^{\mc{G}},\mf{b}^{\mc{G}},\ms{m}^{\mc{G}},
\mf{E}^{\mc{G}}},
\\
\mf{E}^{\mc{G}}
&=
\lr{\upmu^{\mc{G}},\mf{u}^{\mc{G}},\pf{H}^{\mc{G}},\ms{D}^{\mc{G}}}
{\Upgamma^{\mc{G}},\mf{v}^{\mc{G}},\mf{w}^{\mc{G}},\mf{z}^{\mc{G}}};
\end{aligned}
\end{equation*}
and for any $\mf{Q}\in\pf{T}_{\mc{G}}$ and $\beta\in\ms{P}_{\mc{G}}^{\mf{Q}}$.
\begin{equation*}
\begin{aligned}
(\mf{a}_{\mc{G}})_{\beta}^{\mf{Q}}
&=
\lr{\mc{A}(\mc{G})_{\beta}^{\mf{Q}},H}
{(\upeta^{\mc{G}})_{\beta}^{\mf{Q}}},
\\
(\mf{e}_{\mc{G}})_{\beta}^{\mf{Q}}
&=
\lr{\mc{A}(\mc{G})_{\beta}^{\mf{Q}},\R}{(\ep^{\mc{G}})_{\beta}^{\mf{Q}}},
\\
(\pf{H}^{\mc{G}})_{\beta}^{\mf{Q}}
&=
\lr{(\mf{H}^{\mc{G}})_{\beta}^{\mf{Q}}}
{(\uppi^{\mc{G}})_{\beta}^{\mf{Q}},(\Upomega^{\mc{G}})_{\beta}^{\mf{Q}}}.
\end{aligned}
\end{equation*}
\end{convention}
\begin{remark}
\label{06132143}
$H\ni l\mapsto(\uppsi^{\mc{N}}(l),\mf{b}^{\mc{N}}(l^{-1}))\in Aut_{\mf{G}(G,F,\uprho)}(\mc{N})$
is a morphism of groups for any object $\mc{N}$ of $\mf{G}(G,F,\uprho)$
since \eqref{22052156}, 
since $\uppsi^{\mc{N}}$ and $\mf{b}^{\mc{N}}$ are morphisms of groups 
and by \eqref{12051444}.
\end{remark}
\subsection{States originated via a phase}
\label{06171340}
The set $\mf{N}^{\mc{N}}(\mf{c})$ of states originated via a given phase $\mf{c}\in\ms{A}_{\mc{N}}^{\ast}$ 
is a primary concept needed 
in order to provide a reasonable physical interpretation of the structural data of any object 
$\mc{N}$ of $\mf{G}(G,F,\uprho)$.
Let us briefly describe and physically interpret this set, whose precise definition is given in 
Def. \ref{12192100} through Def.\ref{12161311}.
Let $Rep^{\mc{N}}(\mf{c})$ be the set of the tuples 
$\mf{r}=\lr{\mf{T},\alpha,\upnu}{\mf{q},\pf{K},\Upphi}$, 
called representations of $\mf{c}$, where $\mf{T}\in\pf{T}_{\mc{N}}$, $\alpha\in\ms{P}_{\mc{N}}^{\mf{T}}$,
$\upnu$ is a Haar measure on $\ms{S}_{\alpha}^{\mf{T}}(\mc{N})$,
$\pf{K}=\lr{\mf{K},\uptheta}{\Upomega}$ is a $GNS-$represetation associated with the
state $(\ps{\upvarphi}^{\mc{N}})_{\alpha}^{\mf{T}}$, $\mf{q}$ is a group morphism from $\ms{A}_{\mc{N}}$
to the $K_{0}-$theory of $\mc{B}_{\mf{r}}^{+}$, where 
$\mc{B}_{\mf{r}}=
\ms{B}_{\upnu}^{(\ps{\upvarphi}^{\mc{N}})_{\alpha}^{\mf{T}}}
((\mf{a}_{\mc{N}})_{\alpha}^{\mf{T}})$
and $\Upphi$ is an entire normalized even cocycle 
on $\mc{B}_{\mf{r}}^{+}$ such that 
$\mf{c}$ 
factorizes through 
the real part, via the standard duality, of the character generated by $\Upphi$
and the map $\mf{q}$, i.e.
$\mf{c}=\mf{v}_{\mf{r}}$ with
\begin{equation*}
\mf{v}_{\mf{r}}
\coloneqq
\Re \lr{\mf{q}(\cdot)}{[\Upphi]}.
\end{equation*}
Here $[\Upphi]$ is the entire cyclic cohomology class generated by $\Upphi$, 
$\lr{\cdot}{\cdot}$ is 
the standard duality between the entire cyclic cohomology and the $K_{0}-$theory of $\mc{B}_{\mf{r}}^{+}$, 
and $\Re$ is the real part.
We let $\uppi_{\mf{r}}$, $\Upphi^{\mf{r}}$, $\mc{A}_{\mf{r}}$, 
$\mf{R}_{\mf{r}}$, $\mf{T}_{\mf{r}}$, $\alpha_{\mf{r}}$,
$\ep_{\mf{r}}$ and $\ps{\upvarphi}_{\mf{r}}$
denote 
$(\uppi^{\mc{N}})_{\alpha}^{\mf{T}}$, 
$\Upphi$, $\mc{A}(\mc{N})_{\alpha}^{\mf{T}}$, 
$\uptheta\rtimes\ms{W}$, $\mf{T}$, $\alpha$, 
$(\ep^{\mc{N}})_{\alpha}^{\mf{T}}$
and
$(\ps{\upvarphi}^{\mc{N}})_{\alpha}^{\mf{T}}$
respectively, 
where $\ms{W}(h)\uptheta(a)\Upomega=\uptheta((\upeta^{\mc{N}})_{\alpha}^{\mf{T}}(h)(a))\Upomega$, 
for all $h\in\ms{S}_{\alpha}^{\mf{T}}(\mc{N})$ and $a\in\mc{A}(\mc{N})_{\alpha}^{\mf{T}}$.
\par
With any representation $\mf{r}$ of $\mf{c}$ we associate a state $\varrhoup_{\mf{r}}$
of $\mc{A}_{\mf{r}}$ in the following way. 
Firstly we get the state $(\Upphi_{0}^{\mf{r}})^{\natural}$ 
associated with the $0-$dimensional component of $\Upphi^{\mf{r}}$,
to do this we use Def. \ref{12241101} where we construct the state 
associated with any functional $\phi$ on a unital $C^{\ast}-$algebra.
Secondly we get the canonical extension
$((\Upphi_{0}^{\mf{r}})^{\natural}\up\mc{B}_{\mf{r}})^{-}$ 
of $(\Upphi_{0}^{\mf{r}})^{\natural}\up\mc{B}_{\mf{r}}$ 
to the multiplier algebra $\ms{M}(\mc{B}_{\mf{r}})$, 
according to the construction provided in
Lemma \ref{12171928}.
Finally we compose the extension so obtained with the 
canonical injection 
$\mf{j}_{\mf{r}}$
of 
$\mc{A}_{\mf{r}}$ into $\ms{M}(\mc{B}_{\mf{r}})$, 
by obtaining a state $\varrhoup_{\mf{r}}$ of $\mc{A}_{\mf{r}}$,
which is required to be $\uppi_{\mf{r}}-$normal by definition.
We are forced to use the multiplier algebra
because in general $\mc{A}_{\mf{r}}$ could not be 
injectively mapped into $\mc{B}_{\mf{r}}$.
Now the $\uppi_{\mf{r}}-$normality 
is required in order to interpret $\varrhoup_{\mf{r}}$
as a state obtained by performing 
an operation on $\ps{\upvarphi}_{\mf{r}}$.
Note that if $\alpha_{\mf{r}}\in\R_{0}^{+}$,
then $\ps{\upvarphi}_{\mf{r}}$ is an
$\alpha_{\mf{r}}-KMS$ state w.r.t. the dynamics
$\ep_{\mf{r}}(-\alpha_{\mf{r}}^{-1}(\cdot))$, 
therefore $\varrhoup_{\mf{r}}$
is a state obtained by perturbing a state of thermal equilibrium 
at the inverse temperature $\alpha_{\mf{r}}$ 
of the physical system evolving in time via the
dynamics $\ep_{\mf{r}}(-\alpha_{\mf{r}}^{-1}(\cdot))$.
By definition $\mf{N}^{\mc{N}}(\mf{c})$  
is the set of the states $\varrhoup_{\mf{s}}$ by ranging over all representations $\mf{s}$ of $\mf{c}$,
so summing up what said we have 
\begin{equation}
\label{03061620}
\begin{aligned}
\varrhoup_{\mf{r}}&\coloneqq((\Upphi_{0}^{\mf{r}})^{\natural}\up\mc{B}_{\mf{r}})^{-}\circ\mf{j}_{\mf{r}},
\\
\varrhoup_{\mf{r}}&\in\ms{N}_{\uppi_{\mf{r}}},
\\
\mf{N}^{\mc{N}}(\mf{c})&\coloneqq\left\{\varrhoup_{\mf{s}}\mid\mf{s}\in Rep^{\mc{N}}(\mf{c})\right\}.
\end{aligned}
\end{equation}
It is worthwhile noting
that the presence of cohomology classes in the definition 
of $Rep^{\mc{N}}(\mf{c})$ 
it is at the basis of the possible degeneration of
$\mf{N}^{\mc{N}}(\mf{c})$.
Indeed for any representation 
$\mf{r}=\lr{\mf{T},\alpha,\upnu}{\mf{q},\pf{K},\Upphi}$
of $\mf{c}$
it is so also 
$\tilde{\mf{r}}=
\lr{\mf{T},\alpha,\upnu}{\mf{q},\pf{K},\tilde{\Upphi}}$ 
whenever 
$\varrhoup_{\tilde{\mf{r}}}$ is $\uppi_{\tilde{\mf{r}}}-$normal
and
the cocycles $\Upphi$ and $\tilde{\Upphi}$ on $\mc{B}_{\mf{r}}$ 
belongs to the same cohomology class, i.e. 
$[\Upphi]=[\tilde{\Upphi}]$, 
hence $\mf{N}^{\mc{N}}(\mf{c})$ will be degenerate 
as soon as $\varrhoup_{\mf{r}}\neq\varrhoup_{\tilde{\mf{r}}}$.
\par
With the characterization of $\mf{N}^{\mc{N}}(\mf{c})$ in mind, 
we will propose in the assumption at page \pageref{Gstb} 
the existence of a physical system $\mc{N}$ 
and 
for any $\mf{T}\in\pf{T}_{\mc{N}}$ and $\alpha\in\ms{P}^{\mf{T}}$
of physical systems $\mc{O}_{\alpha}^{\mf{T}}$,
such that for any $\mf{r}\in Rep^{\mc{N}}(\mf{c})$ the occurrance of 
$\mc{O}_{\alpha_{\mf{r}}}^{\mf{T}_{\mf{r}}}$ in the state $\varrhoup_{\mf{r}}$ 
implies the previous occurrance of the system $\mc{N}$ in the phase $\mf{c}$.
It is exactly in this sense that has to be understood Def. \ref{12211430}\eqref{12211430d} in which we declare that
the state $\varrhoup_{\mf{r}}$ is originated via the phase $\mf{c}$,
for any $\mf{r}\in Rep^{\mc{N}}(\mf{c})$.
\par
A natural question arises when $\mf{c}$ 
is an integer phase of Chern-Connes type, i.e.
$\mf{c}(\ms{A}_{\mc{N}})\subseteq\Z$
and there exists a representation
$\mf{r}$ of $\mf{c}$
such that $\Upphi^{\mf{r}}$ 
is the JLO cocycle associated with a $\theta-$summable $K-$cycle 
$\lr{\mc{B}_{\mf{r}}^{+},\tilde{\mf{R}}_{\mf{r}}}{\ms{D},\Upgamma}$:
is it the state $\omega_{e^{-\ms{D}^{2}}}\circ\uppi_{\mf{r}}$ 
originated via $\mf{c}$ in the sense above specified, 
namely
$\omega_{e^{-\ms{D}^{2}}}\circ\uppi_{\mf{r}}
=
\varrhoup_{\mf{r}}$?
Lemma \ref{12201040} gives the answer in the positive 
essentially under 
the hypothesis that $\Upgamma$ is represented via 
$\tilde{\mf{R}}_{\mf{r}}$ by an element with norm less or equal to $1$.
This is an important result employed, 
in the construction in Thm. \ref{01151104} of an equivariant stability,
in order to prove the equivariance \eqref{01141641b}.
\begin{definition}
\label{12161259}
Let $\mc{M}\in\mf{G}(G,F,\uprho)$ set
\begin{equation*}
\ms{A}_{\mc{M}}^{\ast}\coloneqq 
Mor_{\ms{Ab}}(\ms{A}_{\mc{M}},\R),
\end{equation*}
$\mf{c}\in\ms{A}_{\mc{M}}^{\ast}$ 
is said to be integer
if $\mf{c}(\ms{A}_{\mc{M}})\subseteq\Z$,
moreover define
\begin{equation*}
\begin{aligned}
\ov{\ms{m}}^{\mc{M}}
&\in\prod_{\mf{Q}\in\pf{T}_{\mc{M}}}
Mor_{\ms{set}}(\ms{P}_{\mc{M}}^{\mf{Q}},\ms{A}_{\mc{M}}^{\ast}),
\\
\ov{\ms{m}}^{\mc{M}}(\mf{Q},\alpha)(\ms{f})
&\coloneqq
\ms{m}^{\mc{M}}(\ms{f})(\mf{Q},\alpha),
\\
\forall
\mf{Q}\in\pf{T}_{\mc{M}},
\alpha
&\in\ms{P}_{\mc{M}}^{\mf{Q}},
\ms{f}\in\ms{A}_{\mc{M}},
\end{aligned}
\end{equation*}
and 
\begin{equation*}
\uppsi_{\ast}^{\mc{M}}:H
\to 
Aut_{\ms{set}}(\ms{A}_{\mc{M}}^{\ast}),\,
l\mapsto(\upomega\mapsto\upomega\circ\uppsi^{\mc{M}}(l^{-1})).
\end{equation*}
\end{definition}
\begin{definition}
\label{18061441}
Define
\begin{equation*}
\begin{aligned}
\ms{gr}&\in\prod_{f\in Mor_{\ms{set}}}Mor_{\ms{set}}(d(f),d(f)\times c(f)),
\\
f&\mapsto(x \mapsto(x,f(x)))
\end{aligned}
\end{equation*}
\end{definition}
\begin{definition}
\label{12241101}
Let $\mc{B}$ be a unital 
$C^{\ast}-$algebra and $\phi\in\mc{B}^{\ast}-\{\ze\}$.
We call the state associated with $\phi$ 
the 
state defined as follows
\begin{equation*}
\phi^{\natural}\coloneqq
\frac{\phi_{1}+\phi_{2}}{\|\phi_{1}+\phi_{2}\|}.
\end{equation*}
Here
$(\phi_{1},\phi_{2})$ is the unique couple such that
\begin{enumerate}
\item
$\phi_{j}$ is a positive functional on $\mc{B}$, $j\in\{1,2\}$,
\item
$\frac{1}{2}(\phi+\phi^{\ast})
=\phi_{1}-\phi_{2}$,
\item
$\frac{1}{2}\|(\phi+\phi^{\ast})\|=\|\phi_{1}\|+\|\phi_{2}\|$,
\end{enumerate}
where $\phi^{\ast}\in\mc{B}^{\ast}$ 
such that $\phi^{\ast}(a)\coloneqq\ov{\phi(a^{\ast})}$,
for all $a\in\mc{B}$.
\end{definition}
\begin{remark}
The existence and uniqueness of the couple $(\phi_{1},\phi_{2})$ in Def. \ref{12241101} follows since 
\cite[Thm. $4.3.6$]{kad} applied to the hermitian and bounded functional $\frac{1}{2}(\phi+\phi^{\ast})$,
where a functional $\psi$ is hermitian if $\psi^{\ast}=\psi$.
\end{remark}
\begin{remark}
\label{12241631}
Under the notations of Def. \ref{12241101}
we have that
$\|\phi_{1}+\phi_{2}\|
=\|\phi_{1}\|+\|\phi_{2}\|
=
\frac{1}{2}\|(\phi+\phi^{\ast})\|$,
the first equality coming since 
\cite[Cor. $2.3.12$]{bra}.
\end{remark}
We recall that
for any unital $C^{\ast}-$algebra 
$\mc{B}$, 
$\lr{\cdot}{\cdot}_{\mc{B}}$
denotes the standard duality
between the $\ms{K}_{0}-$theory of $\mc{B}$
and the even entire cyclic cohomology of $\mc{B}$.
Moreover for any
entire normalized even 
cocycle $\Upphi$ on $\mc{B}$,
$[\Upphi]$ denotes the element in 
$H_{\ep}^{ev}(\mc{B})$ corresponding to $\Upphi$.
Note that $\Upphi_{0}\in\mc{B}^{\ast}$,
where
$\Upphi_{0}$ is the $0-$dimension component of $\Upphi$,
hence if $\Upphi_{0}\neq\ze$ then
$\Upphi_{0}^{\natural}$ is the state of $\mc{B}$
associated with $\Upphi_{0}$ according to Def. \ref{12241101}.
In what follows we use 
Def. \ref{11211142} 
moreover we recall that $\ms{N}_{\uppi}$ is the set of $\uppi-$normal states,
for any representation $\uppi$ of a $C^{\ast}-$algebra.
Up to the end of section \ref{06171340} 
we let $\mc{G}$ be a fixed but arbitrary object of $\mf{G}(G,F,\uprho)$.
\begin{convention}
\label{08040834}
For any $\mf{T}\in\pf{T}_{\mc{G}}$, 
$\alpha\in\ms{P}_{\mc{G}}^{\mf{T}}$
and
$\upnu\in\mc{H}(\ms{S}_{\alpha}^{\mf{T}}(\mc{G}))$
let 
\begin{equation*}
\begin{aligned}
\mc{B}_{\alpha,\upnu}^{\mf{T}}(\mc{G})
&\coloneqq
\ms{B}_{\upnu}^{(\ps{\upvarphi}^{\mc{G}})_{\alpha}^{\mf{T}}}
((\mf{a}_{\mc{G}})_{\alpha}^{\mf{T}})
\\
\mf{j}_{\alpha,\upnu}^{\mf{T}}
&\coloneqq
\mf{j}_{\mc{A(\mc{G})}_{\alpha}^{\mf{T}}}^
{\mc{B}_{\alpha,\upnu}^{\mf{T}}(\mc{G})},
\\
\lr{\cdot}{\cdot}_{(\mc{G},\mf{T},\alpha,\upnu)}
&\coloneqq
\lr{\cdot}{\cdot}_{\mc{B}_{\alpha,\upnu}^{\mf{T}}(\mc{G})^{+}}.
\end{aligned}
\end{equation*}
\end{convention}
The above convention generalizes \eqref{08040736}, indeed
$\mc{B}_{\alpha}^{\mf{T}}(\mc{G})
=
\mc{B}_{\alpha,(\upmu^{\mc{G}})_{\alpha}^{\mf{T}}}^{\mf{T}}(\mc{G})$,\,
$\mf{j}_{\alpha}^{\mf{T}}
=
\mf{j}_{\alpha,(\upmu^{\mc{G}})_{\alpha}^{\mf{T}}}^{\mf{T}}$
and
$\lr{\cdot}{\cdot}_{(\mc{G},\mf{T},\alpha)}
=
\lr{\cdot}{\cdot}_{(\mc{G},\mf{T},\alpha,(\upmu^{\mc{G}})_{\alpha}^{\mf{T}})}$.
\begin{definition}
[$\mc{G}-$representations of a phase]
\label{12161311}
Let $\mf{c}\in\ms{A}_{\mc{G}}^{\ast}$ define $Rep^{\mc{G}}(\mf{c})$
the set of the 
$\mf{r}=\lr{\mf{T},\alpha,\upnu}{\mf{u},\pf{K},\Upphi}$
such that 
\begin{enumerate}
\item 
$\mf{T}\in\pf{T}_{\mc{G}}$ and $\alpha\in\ms{P}_{\mc{G}}^{\mf{T}}$,
\label{12161311req1}
\item
$\upnu\in\mc{H}(\ms{S}_{\alpha}^{\mf{T}}(\mc{G}))$,
\label{12161311req2}
\item
$\mf{u}\in Mor_{\ms{Ab}}(\ms{A}_{\mc{G}},
\ms{K}_{0}(\mc{B}_{\alpha,\upnu}^{\mf{T}}(\mc{G})^{+}))$,
\label{12161311req3}
\item
$\pf{K}=\lr{\mf{K},\uppi}{\Upomega}$ 
is a cyclic representation of $\mc{A}(\mc{G})_{\alpha}^{\mf{T}}$
associated with $(\ps{\upvarphi}^{\mc{G}})_{\alpha}^{\mf{T}}$,
\label{12161311req4}
\item
$\Upphi$ is an 
entire normalized even cocycle on
$\mc{B}_{\alpha,\upnu}^{\mf{T}}(\mc{G})^{+}$
such that 
$\Upphi_{0}\up\mc{B}_{\alpha,\upnu}^{\mf{T}}(\mc{G})\neq\ze$,
\label{12161311req5}
\item
$(\Upphi_{0}^{\natural}\up\mc{B}_{\alpha,\upnu}^{\mf{T}}(\mc{G}))^{-}
\circ
\mf{j}_{\alpha,\upnu}^{\mf{T}}
\in\ms{N}_{\uppi}
$,
\label{12161311req6}
\item
$\mf{c}=v_{\mf{r}}$.
\label{12161311req7}
\end{enumerate}
Here
$\Upphi_{0}^{\natural}$ is the state
associated with $\Upphi_{0}$ according to Def. \ref{12241101},
$(\cdot)^{-}$ is the canonical extension defined in Def. \ref{30051132},
while
$v_{\mf{r}}\in\ms{A}_{\mc{G}}^{\ast}$
such that for all $\ms{f}\in\ms{A}_{\mc{G}}$ 
\begin{equation}
\label{12192032}
v_{\mf{r}}(\ms{f})
\coloneqq
\Re\lr{\mf{u}(\ms{f})}
{[\Upphi]}_{(\mc{G},\mf{T},\alpha,\upnu)},
\end{equation}
where $\Re\lambda$ is the real part of $\lambda\in\C$. 
We call $\Upphi$ the representative of $\mf{c}$ relative to $\mf{r}$
and call any element of $Rep^{\mc{G}}(\mf{c})$
a $\mc{G}-$representation of $\mf{c}$.
\end{definition}
\begin{remark}
\label{27071427}
$(\Upphi_{0}^{\natural}\up\mc{B}_{\alpha,\upnu}^{\mf{T}}(\mc{G}))^{-}
\circ\mf{j}_{\alpha,\upnu}^{\mf{T}}$
is a state of $\mc{A}(\mc{G})_{\alpha}^{\mf{T}}$ since Cor. \ref{12181715},
thus in Def. \ref{12161311}\eqref{12161311req6}
we require that this state belongs to $\ms{N}_{\uppi}$.
\end{remark}
\begin{definition}
\label{17051300}
Let 
$\mf{t}=\lr{\mf{T},\alpha,\upnu}{\mf{u},\pf{K},\ms{L},\Updelta}$
such that 
\begin{enumerate}
\item
Def. \ref{12161311}(\ref{12161311req1}-\ref{12161311req4}) 
hold,
\item
$\lr{\tilde{\pf{R}}_{\pf{K}}^{\upnu}
((\mf{a}_{\mc{G}})_{\alpha}^{\mf{T}})}
{\ms{L},\Updelta}$
is an even $\theta-$summable $K-$cycle,
\end{enumerate}
then we set
$w_{\mf{t}}\in\ms{A}_{\mc{G}}^{\ast}$
such that 
for all $\ms{f}\in\ms{A}_{\mc{G}}$ 
\begin{equation}
\label{12200900}
w_{\mf{t}}(\ms{f})=
\lr{\mf{u}(\ms{f})}
{\ms{ch}
\lr{\tilde{\pf{R}}_{\pf{K}}^{\upnu}((\mf{a}_{\mc{G}})_{\alpha}^{\mf{T}})}
{\ms{L},\Updelta}
}_{(\mc{G},\mf{T},\alpha,\upnu)}.
\end{equation}
\end{definition}
\begin{definition}
[$C_{0}-$representations of an integer phase]
\label{12200859c0}
Let $\mf{c}\in\ms{A}_{\mc{G}}^{\ast}$ be an integer phase,
define 
$C_{0}^{\mc{G}}(\mf{c})$
the set of the 
$\mf{t}=\lr{\mf{T},\alpha,\upnu}{\mf{u},\pf{K},\ms{L},\Updelta}$
such that 
\begin{enumerate}
\item
Def. \ref{12161311}(\ref{12161311req1}-\ref{12161311req4}) 
hold,
\item
$\lr{\tilde{\pf{R}}_{\pf{K}}^{\upnu}
((\mf{a}_{\mc{G}})_{\alpha}^{\mf{T}})}
{\ms{L},\Updelta}$
is an even $\theta-$summable $K-$cycle
\item
there exists an element 
$b\in\mc{B}_{\alpha,\upnu}^{\mf{T}}(\mc{G})^{+}$
such that
$\|b\|\leq 1$
and
$\tilde{\pf{R}}_{\pf{K}}^{\upnu}((\mf{a}_{\mc{G}})_{\alpha}^{\mf{T}})(b)=\Updelta$,
\item
$\mf{c}=w_{\mf{t}}$.
\end{enumerate}
We call 
$C_{0}-$representation
of $\mf{c}$
any element of 
$C_{0}^{\mc{G}}(\mf{c})$.
\end{definition}
\begin{definition}
[$C-$representations of an integer phase]
\label{12200859}
Let $\mf{c}\in\ms{A}_{\mc{G}}^{\ast}$ be an integer phase,
define 
$C^{\mc{G}}(\mf{c})$
the set of the 
$\mf{t}=\lr{\mf{T},\alpha,\upnu}{\mf{u},\pf{K},\ms{L},\Updelta}$
such that 
\begin{enumerate}
\item
Def. \ref{12161311}(\ref{12161311req1}-\ref{12161311req4}) 
hold,
\item
$\lr{\tilde{\pf{R}}_{\pf{K}}^{\upnu}
((\mf{a}_{\mc{G}})_{\alpha}^{\mf{T}})}
{\ms{L},\Updelta}$
is an even $\theta-$summable $K-$cycle
\item
Def. \ref{12161311}(\ref{12161311req5}-\ref{12161311req6}) hold
where $\Upphi$ is the JLO cocycle associated with 
$\lr{\tilde{\pf{R}}_{\pf{K}}^{\upnu}
((\mf{a}_{\mc{G}})_{\alpha}^{\mf{T}})
}{\ms{L},\Updelta}$,
\item
$\mf{c}=w_{\mf{t}}$.
\end{enumerate}
We call 
$C-$representation
of $\mf{c}$
any element of 
$C^{\mc{G}}(\mf{c})$.
\end{definition}
\begin{definition}
\label{01151659}
Define 
$\pf{V}_{\bullet}(\mc{G})$
the subset of the 
$\mf{T}\in\pf{T}_{\mc{G}}$ 
such that for any $\alpha\in\ms{P}_{\mc{G}}^{\mf{T}}$
there exists an element 
$b\in\mc{B}_{\alpha}^{\mf{T}}(\mc{G})^{+}$
such that
$\|b\|\leq 1$
and
$\tilde{\mf{R}}_{\alpha}^{\mf{T}}(\mc{G})(b)=
(\Upgamma^{\mc{G}})_{\alpha}^{\mf{T}}$.
\end{definition}
\begin{definition}
\label{12291116}
For any $\mf{T}\in\pf{T}_{\mc{G}}$ and
$\alpha\in\ms{P}_{\mc{G}}^{\mf{T}}$,
define
$\mf{t}^{\mc{G}}(\mf{T},\alpha)
\coloneqq
\lr{\mf{T},\alpha}
{(\upmu^{\mc{G}})_{\alpha}^{\mf{T}},
(\mf{u}^{\mc{G}})_{\alpha}^{\mf{T}},
(\pf{H}^{\mc{G}})_{\alpha}^{\mf{T}},
(\ms{D}^{\mc{G}})_{\alpha}^{\mf{T}},
(\Upgamma^{\mc{G}})_{\alpha}^{\mf{T}}}$.
\end{definition}
\begin{remark}
\label{12291844}
If
$\pf{V}_{\bullet}(\mc{G})\neq\varnothing$
then
$\mf{t}(\mf{T},\alpha)
\in 
C_{0}^{\mc{G}}(\ov{\ms{m}}^{\mc{G}}(\mf{T},\alpha)
)$
for all $\mf{T}\in\pf{V}_{\bullet}(\mc{G})$ 
and $\alpha\in\ms{P}_{\mc{G}}^{\mf{T}}$
since \eqref{12220022}.
\end{remark}
\begin{definition}
\label{1228937}
Let 
$\mf{c}\in\ms{A}_{\mc{G}}^{\ast}$ 
and
$\mf{r}\in Rep^{\mc{G}}(\mf{c})$
let us use the following convention
\begin{equation*}
\begin{aligned}
\mf{r}
&=
\lr{\mf{T}_{\mf{r}},\alpha_{\mf{r}},\upmu_{\mf{r}}}
{\mf{u}_{\mf{r}},\pf{H}_{\mf{r}},\Upphi^{\mf{r}}},
\\
\pf{H}_{\mf{r}}
&=
\lr{\mf{H}_{\mf{r}}}{\uppi_{\mf{r}},\Upomega_{\mf{r}}};
\end{aligned}
\end{equation*}
and set
\begin{equation*}
\begin{aligned}
\mc{A}_{\mf{r}}
&
\coloneqq
\mc{A}(\mc{G})_{\alpha_{\mf{r}}}^{\mf{T}_{\mf{r}}},
\\
\upeta_{\mf{r}}
&
\coloneqq
(\upeta^{\mc{G}})_{\alpha_{\mf{r}}}^{\mf{T}_{\mf{r}}},
\\
\mf{z}_{\mf{r}}
&
\coloneqq
(\mf{z}^{\mc{G}})_{\alpha_{\mf{r}}}^{\mf{T}_{\mf{r}}},
\\
\ep_{\mf{r}}
&
\coloneqq
(\ep^{\mc{G}})_{\alpha_{\mf{r}}}^{\mf{T}_{\mf{r}}},
\\
\ps{\upvarphi}_{\mf{r}}
&
\coloneqq
(\ps{\upvarphi}^{\mc{G}})_{\alpha_{\mf{r}}}^{\mf{T}_{\mf{r}}},
\\
\mc{B}_{\mf{r}}
&\coloneqq
\mc{B}_{\alpha_{\mf{r}},\upmu_{\mf{r}}}^{\mf{T}_{\mf{r}}}(\mc{G}),
\\
\mf{j}_{\mf{r}}
&\coloneqq
\mf{j}_{\mc{A}_{\mf{r}}}^{\mc{B}_{\mf{r}}},
\\
\Uppsi_{\mf{r}}
&\coloneqq
(\Upphi^{\mf{r}})_{0}^{\natural}
\up
\mc{B}_{\mf{r}},
\\
\lr{\cdot}{\cdot}_{\mf{r}}
&\coloneqq
\lr{\cdot}{\cdot}_{(\mc{G},\mf{T}_{\mf{r}},
\alpha_{\mf{r}},\upmu_{\mf{r}})}.
\end{aligned}
\end{equation*}
If $\mf{c}\in\ms{A}_{\mc{G}}^{\ast}$ is an integer phase
and
$\mf{t}\in C^{\mc{G}}(\mf{c})\cup C_{0}^{\mc{G}}(\mf{c})$,
let us use the following convention 
$\mf{t}=\lr{\mf{T}_{\mf{t}},\alpha_{\mf{t}},\upmu_{\mf{t}}}
{\mf{u}_{\mf{t}},
\pf{H}_{\mf{t}},\ms{D}_{\mf{t}},\Gamma_{\mf{t}}}$,
moreover
$\pf{H}_{\mf{t}}
=
\lr{\mf{H}_{\mf{t}},\uppi_{\mf{t}}}{\Upomega_{\mf{t}}}$.
\end{definition}
\begin{definition}
[States originated via a phase]
\label{12192100}
Let $\mf{c}\in\ms{A}_{\mc{G}}^{\ast}$ define
\begin{equation*}
\mf{N}^{\mc{G}}(\mf{c})
\coloneqq
\left\{
\Uppsi_{\mf{r}}^{-}\circ\mf{j}_{\mf{r}}
\mid
\mf{r}\in Rep^{\mc{G}}(\mf{c})
\right\},
\end{equation*}
if in addition $\mf{c}$ is an integer phase we set
\begin{equation*}
\mf{D}^{\mc{G}}(\mf{c})
\coloneqq
\left\{\omega_{e^{-\ms{D}_{\mf{t}}^{2}}}\circ\uppi_{\mf{t}}
\mid
\mf{t}\in C_{0}^{\mc{G}}(\mf{c})
\right\}.
\end{equation*}
\end{definition}
\begin{remark}
\label{09051341}
According to Def. \ref{12161311}\eqref{12161311req6}, see also Rmk. \ref{27071427}, 
for any $\mf{c}\in\ms{A}_{\mc{G}}^{\ast}$ and $\mf{r}\in Rep^{\mc{G}}(\mf{c})$ 
we have $\Uppsi_{\mf{r}}^{-}\circ\mf{j}_{\mf{r}}\in\ms{N}_{\uppi_{\mf{r}}}$, 
which is at the ground for the physical interpretation constructed in Def. \ref{12211430}.
\end{remark}
If $\mf{c}$ is an integer phase, the next result shows that $C_{0}^{\mc{G}}(\mf{c})\subseteq C^{\mc{G}}(\mf{c})$, 
one can associate a $\mc{G}-$representation of $\mf{c}$ with any element $\mf{t}$ in $C^{\mc{G}}(\mf{c})$, 
and in the $C_{0}^{\mc{G}}(\mf{c})$ case the state in $\mf{N}(\mf{c})$ 
associated with this representation assumes the simple form $\omega_{e^{-\ms{D}_{\mf{t}}^{2}}}\circ\uppi_{\mf{t}}$.
This result is at the basis of the proof of the equivariance \eqref{01141641b} in Thm. \ref{01151104}. 
It is in the proof of the next result that we use Cor. \ref{12181715}\eqref{12181715st1}.
\begin{lemma}
[States originated via the same integer phase]
\label{12201040}
Let $\mf{c}\in\ms{A}_{\mc{G}}^{\ast}$ be integer, then the following map
\begin{equation}
\label{12201040E}
\lr{\mf{T},\alpha,\upmu}{\mf{u},\pf{H},\ms{D},\Gamma}
\overset{\ms{e}^{\mc{G}}}{\mapsto}
\lr{\mf{T},\alpha,\upmu}{\mf{u},\pf{H},\Upphi},
\end{equation}
where $\Upphi$ is the JLO cocycle asssociated to 
$\lr{\tilde{\pf{R}}_{\pf{H}}^{\upmu}
(\mf{a}(\mc{G})_{\alpha}^{\mf{T}})}{\ms{D},\Gamma}$, 
is well-defined on $C^{\mc{G}}(\mf{c})\cup C_{0}^{\mc{G}}(\mf{c})$ and mapping $C^{\mc{G}}(\mf{c})$ 
into $Rep^{\mc{G}}(\mf{c})$.
Moreover for all $\mf{t}\in C_{0}^{\mc{G}}(\mf{c})$
\begin{equation}
\label{12291851}
\Uppsi_{\ms{e}^{\mc{G}}(\mf{t})}^{-}
\circ
\mf{j}_{\ms{e}^{\mc{G}}(\mf{t})}
=
\omega_{e^{-\ms{D}_{\mf{t}}^{2}}}\circ\uppi_{\mf{t}},
\end{equation}
in particular
\begin{equation}
\label{12291851b}
C_{0}^{\mc{G}}(\mf{c})
\subseteq
C^{\mc{G}}(\mf{c})
\text{ and }
\mf{D}^{\mc{G}}(\mf{c})
\subseteq
\mf{N}^{\mc{G}}(\mf{c}).
\end{equation}
\end{lemma}
\begin{remark}
\label{12291844bis}
In Assump. \ref{Gstb} 
and Def. \ref{12211430}
$\mc{G}$
will be interpreted as a physical system, 
$\ms{A}_{\mc{G}}^{\ast}$
as the set of the states of $\mc{G}$, called phases, 
and $\mf{N}^{\mc{G}}(\mf{c})$ as the set of the states, 
of suitable physical systems, originated via the phase $\mf{c}$ 
for any $\mf{c}\in\ms{A}_{\mc{G}}^{\ast}$.
Thus since Lemma \ref{12201040}
in case $\mf{c}$ is integer,
any  $C_{0}-$representation of $\mf{c}$
induces a state originated via $\mf{c}$
whose form is given in \eqref{12291851}.
In particular since Rmk. \ref{12291844} 
we have that
$\omega_{e^{-((\ms{D}^{\mc{G}})_{\alpha}^{\mf{T}})^{2}}}
\circ(\uppi^{\mc{G}})_{\alpha}^{\mf{T}}$
is a state originated via the integral phase 
$\ov{\ms{m}}^{\mc{G}}(\mf{T},\alpha)$
for all $\mf{T}\in\pf{V}_{\bullet}(\mc{G})$ and $\alpha\in\ms{P}_{\mc{G}}^{\mf{T}}$.
This fact is at the basis of the proof in Thm. \ref{01151104} that the map
$\mc{V}_{\bullet}$, defined in Def. \ref{01151759}, satisfies the $H$ equivariance \eqref{01141641b}.
\end{remark}
\begin{proposition}
\label{27072057}
Let $\mf{H}$ be a Hilbert space, $\Gamma$ a $\Z_{2}-$grading on $\mf{H}$, and $\ms{D}$ 
a possibly unbounded selfadjoint operator in $\mf{H}$. 
If $\Gamma\ms{D}\Gamma=-\ms{D}$ and $\ms{D}$ is positive, then $\ms{D}=\ze$.
\end{proposition}
\begin{proof}
If $\ms{D}$ is positive then $\ms{D}=\ms{D}^{\frac{1}{2}}\,\ms{D}^{\frac{1}{2}}$, 
where $\ms{D}^{\frac{1}{2}}$ is a positive,
in particular selfadjoint, operator in $\mf{H}$, thus 
\begin{equation*}
\begin{aligned}
\Gamma\ms{D}\Gamma&=\Gamma\ms{D}^{\frac{1}{2}}\,\ms{D}^{\frac{1}{2}}\Gamma\\
&=(\ms{D}^{\frac{1}{2}}\Gamma)^{\ast}(\ms{D}^{\frac{1}{2}}\Gamma),
\end{aligned}
\end{equation*}
where the second equality follows since a general rule, see for example \cite[Prp. $1.2.4.(4)$]{tra},
and since $Dom(\ms{D}^{\frac{1}{2}}\Gamma)=\Gamma Dom(\ms{D}^{\frac{1}{2}})$ is dense indeed 
$Dom(\ms{D}^{\frac{1}{2}})$ is dense and $\Gamma$ is unitary.
Therefore $\Gamma\ms{D}\Gamma$ is positive, and its spectrum is a subset of $\R^{+}$, while the spectrum of $-\ms{D}$ 
is a subset
of $\R^{-}$. Hence $\Gamma\ms{D}\Gamma=-\ms{D}$ implies that the spectrum of $-\ms{D}$ equals $\{0\}$ 
and the statement follows since
\eqref{18225} 
and the spectral theorem.
\end{proof}
\begin{proof}[Proof of Lemma \ref{12201040}.]
The first sentence of the statement is trivial, 
\eqref{12291851b} follows since 
$\omega_{e^{-\ms{D}_{\mf{t}}^{2}}}\circ\uppi_{\mf{t}}
\in\ms{N}_{\uppi_{\mf{t}}}$
and \eqref{12291851}, so let us prove \eqref{12291851}.
Let 
$\mf{t}=\lr{\mf{T},\alpha,\upmu}{\mf{u},\pf{H},\ms{D},\Gamma}$,
$\pf{H}=\lr{\mf{H},\uppi}{\Upomega}$,
$\rho=e^{-\ms{D}^{2}}$,
$\Upphi_{0}=\Upphi_{0}^{\ms{e}(\mf{t})}$,
$\mc{B}=\mc{B}_{\alpha,\upmu}^{\mf{T}}(\mc{G})$,
$\mc{B}_{1}^{+}$ the closed unit ball of $\mc{B}^{+}$,
while 
$\mc{S}=\tilde{\mf{R}}_{\pf{H}}^{\upmu}
((\mf{a}_{\mc{G}})_{\alpha}^{\mf{T}})$
and
$\mc{S}_{0}=\mf{R}_{\pf{H}}^{\upmu}
((\mf{a}_{\mc{G}})_{\alpha}^{\mf{T}})$
which is nondegenerate since $\pf{H}$ it is so.
We deduce that
\begin{equation}
\label{12241441}
\Upphi_{0}
=
Tr \circ L_{\Gamma\rho}\circ\mc{S}
=
Tr \circ L_{\rho}\circ R_{\Gamma}\circ\mc{S},
\end{equation}
the first equality follows
since \cite[$IV.8.\delta$]{connes},
the second since the commutativity property of the trace.
Next $\Gamma$ commutes with $\ms{D}^{2}$ since 
$\Gamma\ms{D}\Gamma=-\ms{D}$ by construction, and since 
$\Gamma\ms{D}^{2}\Gamma=\Gamma\ms{D}\Gamma\,\Gamma\ms{D}\Gamma$,
thus $\Gamma\in\ms{E}_{\ms{D}^{2}}(\mc{B}(\R))'$ since \cite[Cor. $18.2.4$]{ds} or \cite[$5.6.17$]{kad},
where $\ms{E}_{\ms{D}^{2}}$ is the resolution of the identity of the selfadjoint operator $\ms{D}^{2}$
and $\mc{B}(\R)$ is the set of the Borelian subsets of $\R$.
Moreover since the construction of the functional calculus and \cite[Thm. $4.10.8(f)$]{ds}, 
or by \cite[$5.6.26$]{kad} or \cite[Thm $1.5.5(vi)$]{tra}, 
the functional calculus of any possibly unbounded normal operator $A$ takes values in the algebra of the operators
affiliated to the von Neumann algebra $\ms{E}_{A}(\mc{B}(\R))''$, in particular the operator function belongs to
this algebra if it is bouded, therefore
\begin{equation}
\label{12241600}
[\Gamma,\rho]=\ze.
\end{equation}
It is worthwhile noting that since the map composition law of the functional calculus $(\ms{D}^{2})^{1/2}=|\ms{D}|$ 
which equals $\ms{D}$ if and only if $\ms{D}$ is positive, case excluded by Prp. \ref{27072057} and since 
the trivial case $\ms{D}=\ze$
has been excluded from the beginning.
Therefore the fact that  $\Gamma$ commutes with $e^{-\ms{D}^{2}}$ does not conflict with $\Gamma\ms{D}\Gamma=-\ms{D}$.
Next $\Gamma$ and $\rho$ are selfadjoint
then $\Upphi_{0}$ is hermitian 
since \eqref{12241600}
since $Tr$ is hermitian
and the commutativity preperty of the trace,
thus
\begin{equation}
\label{12241548}
\Upphi_{0}=\frac{1}{2}(\Upphi_{0}+\Upphi_{0}^{\ast}).
\end{equation}
Next let $b\in\mc{B}_{1}^{+}$ such that $\mc{S}(b)=\Gamma$
which exists by hypothesis, then 
since $\Gamma^{2}=\un$, Rmk. \ref{12211131}
and \eqref{12241441} we have
\begin{equation}
\label{12241501}
\begin{aligned}
\|\Upphi_{0}\|
&\geq
\sup
\{(Tr \circ L_{\rho}\circ R_{\Gamma}\circ\mc{S})(ab)\,
\mid\,
a\in\mc{B}_{1}^{+}
\}
\\
&=
\sup
\{(Tr \circ L_{\rho}\circ\mc{S})(a)\,
\mid\,
a\in\mc{B}_{1}^{+}
\}
\\
&=
Tr(\rho).
\end{aligned}
\end{equation}
Next
let $P_{\pm}$ be the projector associated with 
the Hilbert subspace $\mf{H}^{\pm}=\{v\in\mf{H}\mid\Gamma v=\pm v\}$,
thus
$P_{+}+P_{-}=\un$ 
and $\Gamma P_{j}=(-1)^{j}P_{j}$,
$j\in\{-1,1\}$,
hence
for all $a\in\mc{B}^{+}$
\begin{equation*}
\begin{aligned}
Tr(\Gamma\rho\mc{S}(a))
&=
\sum_{j=-1,1}
Tr\bigl(\Gamma P_{j}\rho\mc{S}(a)\bigr)
\\
&=
\sum_{j=-1,1}(-1)^{j}
Tr(P_{j}\rho\mc{S}(a)).
\end{aligned}
\end{equation*}
Therefore since \eqref{12241441}
\begin{equation}
\label{12241532}
\Upphi_{0}=
\sum_{j=-1,1}(-1)^{j}
Tr\circ L_{P_{j}\rho}\circ\mc{S},
\end{equation}
in particular by Rmk. \ref{12211131}
\begin{equation*}
\begin{aligned}
\|\Upphi_{0}\|
&\leq
\sum_{j=-1,1}
\|Tr\circ L_{P_{j}\rho}\circ\mc{S}\|
\\
&=
\sum_{j=-1,1}
Tr(P_{j}\rho)
\\
&=
Tr(\rho),
\end{aligned}
\end{equation*}
which together \eqref{12241501} implies
\begin{equation}
\label{12241543}
\|\Upphi_{0}\|
=
\sum_{j=-1,1}
\|Tr\circ L_{P_{j}\rho}\circ\mc{S}\|
=Tr(\rho).
\end{equation}
So
$\Upphi_{0}^{\natural}
=
\omega_{\rho}\circ\mc{S}$
since 
(\ref{12241548},\ref{12241532},\ref{12241543})
and Rmk. \ref{12241631},
hence
\begin{equation}
\label{01141843}
\Upphi_{0}^{\natural}\up\mc{B}
=\omega_{\rho}\circ\mc{S}_{0},
\end{equation}
and \eqref{12291851}
follows since 
Cor. \ref{12181715}\eqref{12181715st1}.
\end{proof}
\subsection{Physical Interpretation}
\label{06171244}
In Def. \ref{12211430} we define
the physical interpretation of the data of the category $\mf{G}(G,F,\uprho)$,
expecially we physically decode the states originated via a phase 
described in Def. \ref{12192100}.
What here introduced 
will be emploied in the next section to describe the physical invariances of a nucleon-fragment doublet.
\begin{remark}
\label{11051500}
If $\upphi$ is a $\beta-$KMS state for some one-parameter dynamics $\uptau$ with $\beta>0$,
i.e. a state of thermal equilibrium at the inverse temperature $\beta$ 
of the system whose dynamics is $\uptau$,
then $\upphi-$normal states usually are interpreted as states 
obtained by performing on $\upphi$ small perturbations or operations. 
Thus, according to Def. \ref{12161311}\eqref{12161311req6}, \eqref{eqdyn} and \cite[section 5.3.1]{bra},
for any
$\mf{c}\in\ms{A}^{\ast}$ and $\mf{r}\in Rep^{\mc{G}}(\mf{c})$
such that
$\alpha_{\mf{r}}\in\R_{0}^{+}$, 
it follows
that
$\Uppsi_{\mf{r}}^{-}\circ\mf{j}_{\mf{r}}$
is a state 
generated by an operation
performed on
the state of thermal equilibrium
$\ps{\upvarphi}_{\mf{r}}$
at the inverse temperature $\alpha_{\mf{r}}$ 
of the system whose dynamics is 
$\ep_{\mf{r}}(-\alpha_{\mf{r}}^{-1}(\cdot))$.
\end{remark}
\begin{remark}
[Noncommutative geometric nature of the 
degeneration of $\mf{N}^{\mc{G}}(\mf{c})$]
\label{stbgc}
Let 
$\mf{c}\in\ms{A}^{\ast}$
and
$\chi\in\mf{N}(\mf{c})$,
so there exists
$\mf{r}\in Rep^{\mc{G}}(\mf{c})$
such that 
\begin{equation}
\label{12301142a}
\chi=\Uppsi_{\mf{r}}^{-}\circ\mf{j}_{\mf{r}}.
\end{equation}
Next
$\Uppsi_{\mf{r}}^{-}$ 
is the canonical extension 
of $\Uppsi_{\mf{r}}$ 
to $\ms{M}(\mc{B}_{\mf{r}})$,
where
$\Uppsi_{\mf{r}}$ 
is
the restriction 
to $\mc{B}_{\mf{r}}$ of the state associated 
to
the $0-$dimensional component of
the entire normalized cyclic even cocycle  
$\Upphi^{\mf{r}}$
on $\mc{B}_{\mf{r}}^{+}$
such that
$\mf{c}=v_{\mf{r}}$
i.e. 
\begin{equation}
\label{12301142}
\begin{aligned}
\Uppsi_{\mf{r}}^{-}
\circ
i^{\mc{B}_{\mf{r}}}
&=
\Uppsi_{\mf{r}},
\\
\Uppsi_{\mf{r}}
&=
(\Upphi_{0}^{\mf{r}})^{\natural}\up\mc{B}_{\mf{r}},
\\
\mf{c}(\ms{f})
&=
\Re
\lr{\mf{u_{\mf{r}}}(\ms{f})}
{[\Upphi^{\mf{r}}]}_{\mf{r}}
\quad
\forall \ms{f}\in\ms{A}.
\end{aligned}
\end{equation}
\eqref{12301142a} and \eqref{12301142}
justify Assumtion \ref{Gstb}\eqref{Gstb7} where we propose 
to consider any element in $\mf{N}(\mf{c})$ as a state whose occurrence signales
the occurrence of the phase $\mf{c}$. If we get an 
entire normalized even cocyle $\tilde{\Upphi}$ on $\mc{B}_{\mf{r}}$ such that 
\begin{equation*}
\begin{cases}
\tilde{\chi}
\doteq
((\tilde{\Upphi}_{0})^{\natural}\up\mc{B}_{\mf{r}})^{-}
\circ\mf{j}_{\mf{r}}
\in\ms{N}_{\uppi_{\mf{r}}},
\\
[\Upphi^{\mf{r}}]
=[\tilde{\Upphi}],
\end{cases}
\end{equation*}
then 
$\tilde{\mf{r}}
\doteq
\lr{\mf{T},\alpha,\upmu,\mf{u},\pf{H}}{\tilde{\Upphi}}
\in\mf{N}(\mf{c})$.
Therefore $\tilde{\chi}\neq\chi$ implies that $\mf{N}(\mf{c})$ is degenerate 
with degeneration of noncommutative geometric nature.
\end{remark}
Next we think about the elements of $\mf{N}^{\mc{G}}(\mf{c})$ as those states
of suitable systems $\mc{O}$'s, signaling the occurrance of the phase $\mf{c}$ of $\mc{G}$, 
namely such that whenever one of the systems $\mc{O}$'s occurs in one of these states, 
then the physical system $\mc{G}$ previously occurred in the phase $\mf{c}$. 
According to Rmk. \ref{stbgc} the degeneration of the set of the states originated via the same phase 
can be of noncommutative geometric nature. 
We can consider $G$ as the group of spatiotemporal
translations and $F$ the group containing the gauge 
group $F_{0}$ and the remaining symmetries of the system, 
in such a case $\uprho$ restricted to $F_{0}$ needs to be trivial.
\begin{assumption}
\label{Gstb}
For any $\mc{N}\in Obj(\mf{G}(G,F,\uprho))$
there exists a unique 
physical system still denoted by
$\mc{N}$ such that 
\begin{enumerate}
\label{Gstbst1}
\item
$\ms{A}_{\mc{N}}$ is the set of observables of $\mc{N}$;
\item
$\ms{A}_{\mc{N}}^{\ast}$ 
is the set of states, said phases of $\mc{N}$, 
\label{Gstbst2}
\item
for any $\mf{T}\in\pf{T}_{\mc{N}}$ and $\alpha\in\ms{P}_{\mc{N}}^{\mf{T}}$ 
there exists a physical system $\mc{O}(\mc{N})_{\alpha}^{\mf{T}}$ 
such that
\begin{enumerate}
\item
$\mc{A}(\mc{N})_{\alpha}^{\mf{T}}$ and $\ms{E}_{\mc{A}(\mc{N})_{\alpha}^{\mf{T}}}$
are the algebra of observables and the set of states
of $\mc{O}(\mc{N})_{\alpha}^{\mf{T}}$ respectively.
If $\alpha\in\R_{0}^{+}$ then 
$\mc{O}(\mc{N})_{\alpha}^{\mf{T}}$ 
evolves in time through
$(\ep^{\mc{N}})_{\alpha}^{\mf{T}}(-\alpha^{-1}(\cdot))$,
\item
$\mf{T}$ acts as a special type of operation on 
$\ms{E}_{\mc{A}_{\alpha}^{\mf{T}}}$
producing states of $\mc{O}(\mc{N})_{\alpha}^{\mf{T}}$ 
via the intermediation of phases of $\mc{N}$;
\label{07061013}
\end{enumerate}
\item
for all $\mf{c}\in\ms{A}_{\mc{N}}^{\ast}$ 
and 
$\mf{r}\in Rep^{\mc{N}}(\mf{c})$
the following properties hold
\begin{enumerate}
\label{Gstb6}
\item
if the operation $\mf{T}_{\mf{r}}$ is performed 
on the system $\mc{O}(\mc{N})_{\alpha_{\mf{r}}}^{\mf{T}_{\mf{r}}}$ 
when occurring in the state $\ps{\upvarphi}_{\mf{r}}$,
then $\mc{O}(\mc{N})_{\alpha_{\mf{r}}}^{\mf{T}_{\mf{r}}}$ 
will occur in the state $\Uppsi_{\mf{r}}^{-}\circ\mf{j}_{\mf{r}}$,
\label{Gstb8}
\item
\textbf{if $\mc{O}(\mc{N})_{\alpha_{\mf{r}}}^{\mf{T}_{\mf{r}}}$
occurs in the state 
$\Uppsi_{\mf{r}}^{-}\circ\mf{j}_{\mf{r}}$
then $\mc{N}$ occurred in the phase $\mf{c}$};
\label{Gstb7}
\end{enumerate}
\item
the information about the set of operations $\pf{T}_{\mc{N}}$
are sufficient to provide the set of observables $\ms{A}_{\mc{N}}$
with the structure of a group but not an algebra;
\item
for any $l\in H$, $\ms{f}\in\ms{A}_{\mc{N}}$,
$\mf{T}\in\pf{T}_{\mc{N}}$ and $a\in\mc{A}(\mc{N})_{\alpha}^{\mf{T}}$ 
\begin{enumerate}
\item
$\uppsi^{\mc{N}}(l)(\ms{f})$ is the observable 
obtained by transforming $\ms{f}$
through $l$, 
\label{Gstb4}
\item
$\mf{b}^{\mc{N}}(l)(\mf{T})$ is the operation 
obtained by transforming $\mf{T}$ through $l$,
\label{Gstb5}
\item
$\ms{V}(\mc{N})_{\alpha}^{\mf{T}}(l)(a)$ is the 
observable of the system 
$\mc{O}(\mc{N})_{\alpha}^{\mf{b}^{\mc{N}}(l)(\mf{T})}$,
obtained by transforming $a$ through $l$,
\label{07061031}
\end{enumerate}
\item
for all $\mc{M}\in Obj(\mf{G}(G,F,\uprho))$,
$(\mf{g},\mf{d})\in
Mor_{\mf{G}(G,F,\uprho)}(\mc{N},\mc{M})$,
$\ms{f}\in\ms{A}_{\mc{N}}$
and $\mf{Q}\in\pf{T}_{\mc{M}}$
\begin{enumerate}
\item
$\mf{g}(\ms{f})$ is the observable of the system $\mc{M}$
obtained by 
transforming $\ms{f}$ through $\mf{g}$,
\item
$\mf{d}(\mf{Q})$ is the operation obtained by tranforming
$\mf{Q}$ through $\mf{d}$,
\end{enumerate}
\end{enumerate}
\end{assumption}
Now we fix those properties of any physical interpretation satisfying
Assumption \ref{Gstb} and Rmk. \ref{11051500}.
In what follows let $\equiv$ denote equivalence of propositions.
\begin{definition}
\label{12211430}
We call interpretation any couple of maps $(\mf{s},\mf{u})$ 
such that
if
$\mc{G},\mc{M}\in Obj(\mf{G}(G,F,\uprho))$,
$(\mf{g},\mf{f})\in Mor_{\mf{G}(G,F,\uprho)}(\mc{G},\mc{M})$,
$\mf{P}\in\pf{T}_{\mc{M}}$,
$\beta\in\ms{P}_{\mc{M}}^{\mf{P}}$,
$\mf{Q}\in\pf{T}_{\mc{G}}$,
$\alpha\in\ms{P}_{\mc{G}}^{\mf{Q}}$,
$T\in Mor_{\ms{CA}^{\ast}}
(\mc{A}(\mc{G})_{\beta}^{\mf{f}(\mf{P})},\mc{A}(\mc{M})_{\beta}^{\mf{P}})$,
$b\in\mc{A}(\mc{G})_{\beta}^{\mf{f}(\mf{P})}$ such that $b=b^{\ast}$,
$l\in H$, 
$\ms{f}\in\ms{A}_{\mc{G}}$,
$\mf{c}\in\ms{A}_{\mc{G}}^{\ast}$
such that $Rep^{\mc{G}}(\mf{c})\neq\varnothing$,
\footnote
{in particular $\mf{c}=\ov{\ms{m}}^{\mc{G}}(\mf{Q},\alpha)$}
$\mf{r}\in Rep^{\mc{G}}(\mf{c})$,
$\chi\in\ms{E}_{\mc{A}(\mc{G})_{\alpha}^{\mf{Q}}}$,
$a\in\mc{A}_{\alpha}^{\mf{Q}}$ such $a=a^{\ast}$,
$\mf{C}$ is a category, 
$\ms{a}\in\mf{C}$ and 
$\mc{F}$ is a functor from $\mf{C}$ to $\mf{G}(G,F,\uprho)$, 
then
\begin{enumerate}
\item
$\mf{s}(\mc{G})\equiv$ the nucleon system $\mf{u}(\mc{G})$,
\item
$\mf{u}(\mc{F}(\ms{a}))\equiv$ generated by the fissioning system $\mf{u}(\ms{a})$,
\item
$\mf{s}(\mc{O}(\mc{G})_{\alpha}^{\mf{Q}})\equiv$
the fragment system whose observable algebra is 
$\mc{A}(\mc{G})_{\alpha}^{\mf{Q}}$;
\item
$\mf{s}(\mc{O}(\mc{G})_{\alpha}^{\mf{Q}})\equiv$
the fragment system whose observable algebra is 
$\mc{A}(\mc{G})_{\alpha}^{\mf{Q}}$
and whose dynamics is $(\ep^{\mc{G}})_{\alpha}^{\mf{Q}}(-\alpha^{-1}(\cdot))$,
if $\alpha\in\R_{0}^{+}$, 
\label{12211430a}
\item
$\mf{s}(\chi)\equiv$
the state of $\mf{s}(\mc{O}(\mc{G})_{\alpha}^{\mf{Q}})$
$\mf{u}(\chi)$, 
\item
$\mf{s}(\chi)\equiv$
the state $\mf{u}(\chi)$
of $\mf{s}(\mc{O}(\mc{G})_{\alpha}^{\mf{Q}})$,
\item
$\mf{s}(a)\equiv$
the observable $\mf{u}(a)$ 
of $\mf{s}(\mc{O}(\mc{G})_{\alpha}^{\mf{Q}})$,
\item
$\mf{s}(a)\equiv$
the observable of $\mf{s}(\mc{O}(\mc{G})_{\alpha}^{\mf{Q}})$
$\mf{u}(a)$,
\label{12211430b}
\item
$\mf{u}(\ms{V}(\mc{G})_{\alpha}^{\mf{Q}}(l))\equiv l$,
\item
$\mf{s}(\mf{Q})\equiv$ the operation $\mf{u}(\mf{Q})$,
\label{12211430c}
\item
$\mf{u}(\mf{b}^{\mc{G}}(l))\equiv l$;
\item
$\mf{u}^{\mc{M}}(Tb)\equiv$
obtained by transforming through the action of $\mf{u}(T)$ $\mf{s}(b)$, 
\item
$\mf{u}(\mf{f}(\mf{P}))\equiv$
obtained by transforming through the action of $\mf{u}(\mf{f})$ $\mf{s}(\mf{P})$, 
\label{08111209b}
\item
$\chi(a)$
equals the mean value in $\mf{s}(\chi)$ of $\mf{s}(a)$, 
\label{12211430g}
\item
$\mf{u}((\ps{\upvarphi}^{\mc{G}})_{\alpha}^{\mf{Q}})\equiv$
of thermal equilibrium $(\ps{\upvarphi}^{\mc{G}})_{\alpha}^{\mf{Q}}$
at the inverse temperature $\alpha$, 
if $\alpha\in\R_{0}^{+}$, 
\label{12211430h}
\item
$\mf{s}(\mf{c})\equiv$
the phase of 
$\mf{s}(\mc{G})$
occurring by performing $\mf{s}(\mf{T}_{\mf{r}})$
on $\mf{s}(\ps{\upvarphi}_{\mf{r}})$;
\label{12211430l}
\item
$\mf{s}(\mf{c})\equiv$
the phase of 
$\mf{s}(\mc{G})$
occurring by performing on $\mf{s}(\ps{\upvarphi}_{\mf{r}})$
$\mf{s}(\mf{T}_{\mf{r}})$,
\label{12211430lb}
\item
$\mf{u}(\Uppsi_{\mf{r}}^{-}\circ\mf{j}_{\mf{r}})\equiv$
\textbf{originated via $\mf{s}(\mf{c})$},
\label{12211430d}
\item
$\mf{s}(\ms{f})\equiv$
the observable $\mf{u}(\ms{f})$ of $\mf{s}(\mc{G})$,
\label{12211430i}
\item
$\mf{u}(\uppsi^{\mc{G}}(l))\equiv l$,
\item
$\mf{u}^{\mc{M}}(\mf{g}(\ms{f}))\equiv$
obtained by transforming through the action of $\mf{u}(\mf{g})$ $\mf{s}(\ms{f})$, 
\label{08111209a}
\item
$\mf{c}(\ms{f})$
equals
the mean value in $\mf{s}(\mf{c})$
of $\mf{s}(\ms{f})$.
\end{enumerate}
\end{definition}
In order to avoid redundancies,
often we convein to remove the expression
``of $\mf{s}(\mc{O}_{\alpha}^{\mc{T}})$``
and ``of $\mf{s}(\mc{G})$''. 
According to Assumption \ref{Gstb}\eqref{Gstb7}, 
Def. \ref{12211430}\eqref{12211430d}
has to be understood as 
``whose occurrence follows the occurrence of $\mf{s}(\mf{c})$'',
moreover Def. \ref{12211430}\eqref{12211430l} follows 
since Rmk. \ref{11051500} and Assumption \ref{Gstb}\eqref{Gstb7}. 
The introduction of the map $\mf{u}$ permits to specify 
the semantics in actual models,
we shall use it in addressing in 
part \ref{07301107} 
the nuclear binary fission phenomenon.
This is why we use here in an abstract meaning the concepts of fragment and nucleon
systems.
In the remaining of the work let $(\mf{s},\mf{u})$ be a fixed interpretation.
\begin{remark}
Since Rmk. \ref{06132143} we obtain
\begin{enumerate}
\item
$\mf{u}(\ms{V}(\mc{G})_{\alpha}^{\mf{Q}}(l)(a))\equiv$
obtained by transforming through the action of $l$ $\mf{s}(a)$, 
\item
$\mf{u}(\mf{b}^{\mc{G}}(l)\mf{Q})\equiv$
obtained by transforming through the action of $l$ $\mf{s}(\mf{Q})$, 
\item
$\mf{u}(\uppsi^{\mc{G}}(l)\ms{f})\equiv$
obtained by transforming through the action of $l$ $\mf{s}(\ms{f})$.
\end{enumerate}
\end{remark}
Let $\mc{G}$ be an object of $\mf{G}(G,F,\uprho)$ and $\mf{c}\in\ms{A}_{\mc{G}}^{\ast}$,
then the next two propositions easily follow by Def. \ref{12211430}. 
\begin{proposition}
[Thermal nature of the nucleon phases 
and their stability 
under variation of the operations.]
\label{12051738}
$\mf{s}(\mf{c})\equiv$  
the phase of the nucleon system $\mf{u}(\mc{G})$
occurring by performing the operation $\mf{u}(\mf{T}_{\mf{p}})$ 
on the state of thermal equilibrium 
$\ps{\upvarphi}_{\mf{p}}$
at the inverse temperature $\alpha_{\mf{p}}$
of the fragment system whose observable algebra is 
$\mc{A}_{\mf{p}}$
and whose dynamics is 
$\ep_{\mf{p}}(-\alpha_{\mf{p}}^{-1}(\cdot))$,
for all $\mf{p}\in Rep^{\mc{G}}(\mf{c})$ such that $\alpha_{\mf{p}}>0$.
\end{proposition}
\begin{proposition}
[Noncommutative geometric and thermal origination of fragment states 
via a nucleon phase.]
\label{06251847}
\begin{equation*}
\Uppsi_{\mf{r}}^{-}\circ\mf{j}_{\mf{r}}
\end{equation*}
is the $\ps{\upvarphi}_{\mf{r}}-$normal state originated via $\mf{s}(\mf{c})$, 
of the fragment system whose observable algebra is 
$\mc{A}_{\mf{r}}$
and whose dynamics is 
$\ep_{\mf{r}}(-\alpha_{\mf{r}}^{-1}(\cdot))$,
for all $\mf{r}\in Rep^{\mc{G}}(\mf{c})$ such that $\alpha_{\mf{r}}>0$.
\end{proposition}
\subsection{Nucleon-fragment doublets on $\mf{C}$}
\label{06241544}
We define the auxiliary concepts of equivariant stability and 
extended $\mf{C}-$equivariant stability for a category $\mf{C}$, 
in Def. \ref{11051303} and Def. \ref{06201206}. 
Nucleon-fragment doublets are essential in order to solve the universality claim described 
in introduction \ref{introI},
as we shall establish in Cor. \ref{19061936}.
In Def. \ref{06161650} we define the structure of nucleon-fragment doublet $\mc{T}$ on an arbitrary category, 
in Def. \ref{07271553} we define the $\mc{T}-$nucleon phase and $\mc{T}-$fragment state
and their expanded equivariances which represent 
the $\mc{T}-$resolution of the equivariant form of the universality claim.
We describe the properties of equivariance of $\mc{T}$ in Prp. \ref{06211545},
the consequent properties of invariance and their physical interpretation
in Prp. \ref{20051810dbt} and Prp. \ref{12051206dbt} with the help of Prp. \ref{02202114dbt}. 
In Cor. \ref{06241610} we relate a nucleon-fragment doublet on $\mf{C}$ to any 
extended $\mf{C}-$equivariant stability.
The first step is to define equivariant phases in Def. \ref{01141121}
let us start with preparatory definitions.
\begin{definition}
\label{18061142}
Let $\pf{U}$ be defined pre $(G,F,\uprho)-$map 
if it is a map 
on $Dom(\pf{U})\subseteq Obj(\mf{G}(G,F,\uprho))$
such that $\pf{U}_{\mc{N}}\subseteq\pf{T}_{\mc{N}}$,
for all $\mc{N}\in Dom(\pf{U})$.
Define $\pf{U}$ to be a $(G,F,\uprho)-$map 
if it is a pre $(G,F,\uprho)-$map such that 
\begin{equation}
\label{18061152}
(\forall\mc{M}\in Dom(\pf{U}))
(\forall(\mf{h},\mf{f})\in Mor_{\mf{G}(G,F,\uprho)}
(\mc{N},\mc{M}))
(\mf{f}(\pf{U}_{\mc{M}})
\subseteq
\pf{U}_{\mc{N}}).
\end{equation}
Let 
$\lr{\ms{H}}{\pf{U}}$ be defined a pre $(G,F,\uprho)-$couple of maps 
if $\pf{U}$ is a pre $(G,F,\uprho)-$map 
and $\ms{H}$ is a map on $Dom(\pf{U})$
such that $\ms{H}_{\mc{N}}$ is a subgroup of $H$,
for all $\mc{N}\in Dom(\pf{U})$.
$\ms{H}$ is said to be full if it is the constant map with constant value equal to $H$.
Let $\lr{\ms{H}}{\pf{U}}$ be defined a $(G,F,\uprho)-$couple of maps 
if it is a pre $(G,F,\uprho)-$couple of maps such that 
for all $\mc{N}\in Dom(\pf{U})$ and $l\in\ms{H}_{\mc{N}}$
\begin{equation}
\label{18061148}
\mf{b}^{\mc{N}}(l)(\pf{U}_{\mc{N}})\subseteq\pf{U}_{\mc{N}}.
\end{equation}
\end{definition}
If $\pf{U}$ is a $(G,F,\uprho)-$map, then 
$\lr{\ms{H}}{\pf{U}}$ is a $(G,F,\uprho)-$couple of maps
with $\ms{H}$ full, since Rmk. \ref{06132143}.
\begin{definition}
\label{19061130}
Let 
$\lr{\ms{H},\pf{D}}{\pf{U}}$ 
be defined a $(G,F,\uprho)-$triplet of maps
if $\pf{U}$ is a $(G,F,\uprho)-$map
and $\lr{\ms{H}}{\pf{U}}$ is a pre $(G,F,\uprho)-$couple of maps,
$\pf{D}$ is a map such that
$Dom(\pf{D})\subseteq Dom(\pf{U})$,
$\pf{D}_{\mc{N}}\subseteq\pf{U}_{\mc{N}}$, for all $\mc{N}\in Dom(\pf{D})$,
and $\lr{\ms{H}\up Dom(\pf{D})}{\pf{D}}$ is a $(G,F,\uprho)-$couple of maps.
\end{definition}
If $\lr{\ms{H}}{\pf{U}}$ is a pre $(G,F,\uprho)-$couple of maps
such that $\pf{U}$ is a $(G,F,\uprho)-$map,
then
$\lr{\ms{H},\pf{U}}{\pf{T}}$ is a $(G,F,\uprho)-$triplet of maps. 
\begin{definition}
\label{06191427}
Let $\pf{D}$ be a pre $(G,F,\uprho)-$map,
define $Dom(\pf{D})^{0}$ to be the 
unique subcategory of $\mf{G}(G,F,\uprho)$
such that 
$Obj(Dom(\pf{D})^{0})=Dom(\pf{D})$, 
and
\begin{equation*}
Mor_{Dom(\pf{D})^{0}}(\mc{N},\mc{M})
=
\{
(\mf{h},\mf{f})\in Mor_{\mf{G}(G,F,\uprho)}(\mc{N},\mc{M})
\mid
\mf{f}(\pf{D}_{\mc{M}})\subseteq\pf{D}_{\mc{N}}
\},
\forall
\mc{N},\mc{M}\in Dom(\pf{D}).
\end{equation*}
\end{definition}
\begin{remark}
Let $\pf{D}$ be a pre $(G,F,\uprho)-$map, then
\begin{equation*}
(\pf{D},\Pr_{2})\in\ms{Fct}((Dom(\pf{D})^{0})^{op},\ms{set}),
\end{equation*}
where $\Pr_{2}$ is the map $(\mf{g},\mf{f})\mapsto\mf{f}$ 
defined on $Mor_{\mf{G}(G,F,\uprho)}$. 
If $\lr{\ms{H}}{\pf{D}}$ is a $(G,F,\uprho)-$couple of maps,
$\mc{N}\in Dom(\pf{D})$ and $l\in\ms{H}_{\mc{N}}$ then since Rmk. \ref{06132143}
\begin{equation*}
(\uppsi^{\mc{N}}(l),\mf{b}^{\mc{N}}(l^{-1}))\in Aut_{Dom(\pf{D})^{0}}(\mc{N}).
\end{equation*}
If $\pf{U}$ is a pre $(G,F,\uprho)-$map then
$Dom(\pf{U})^{0}$ is a subcategory of 
$Dom(\pf{U})$ the full subcategory of 
$\mf{G}(G,F,\uprho)$ whose object set equals 
$Dom(\pf{U})$;
however if 
$\pf{U}$ is a $(G,F,\uprho)-$map then $Dom(\pf{U})^{0}=Dom(\pf{U})$ as categories. 
\end{remark}
\begin{definition}
Let $\mc{R}$ be an object of $\mf{G}(G,F,\uprho)$ and $\mf{Q}\in\pf{T}_{\mc{R}}$,
set
$\ms{X}(\mc{R},\mf{Q})
\coloneqq
\prod_{\beta\in\ms{P}_{\mc{R}}^{\mf{Q}}}
\left(
\mc{A}(\mc{R})_{\beta}^{\mf{Q}}
\right)^{\ast}$.
\end{definition}
\begin{definition}
\label{20051117}
Let 
$\Updelta_{o}$,
$\ms{Z}$
and
$\ms{V}_{\vardiamondsuit}$
be maps on 
$Obj(\mf{G}(G,F,\uprho))$
while
$\Updelta_{m}$
be the map on
$ Mor_{Obj(\mf{G}(G,F,\uprho))}$
such that for all
$\mc{M},\mc{N}\in 
Obj(\mf{G}(G,F,\uprho))$
and 
$(\mf{h},\mf{f})\in Mor_{\mf{G}(G,F,\uprho)}(\mc{N},\mc{M})$,
we have
\begin{equation*}
\Updelta_{o}(\mc{M})
\coloneqq
\bigcup_{\mf{Q}\in\pf{T}_{\mc{M}}}
Mor_{\ms{set}}(\ms{P}_{\mc{M}}^{\mf{Q}},\ms{A}_{\mc{M}}^{\ast}),
\end{equation*}
while
\begin{equation*}
\Updelta_{m}(\mf{h},\mf{f}):
\Updelta_{o}(\mc{M})
\to
\Updelta_{o}(\mc{N}),
\quad
Mor_{\ms{set}}(\ms{P}_{\mc{M}}^{\mf{I}},\ms{A}_{\mc{M}}^{\ast})
\ni f
\mapsto
\left(
\ms{P}_{\mc{N}}^{\mf{f}(\mf{I})}
\ni\alpha\mapsto
f(\alpha)\circ\mf{h}
\right),
\,
\forall\mf{I}\in\pf{T}_{\mc{M}},
\end{equation*}
and
\begin{equation*}
\ms{Z}(\mc{M})\coloneqq
\coprod_{\mf{Q}\in\pf{T}_{\mc{M}}}
\ms{X}(\mc{M},\mf{Q}),
\end{equation*}
and
$\ms{V}_{\vardiamondsuit}(\mc{M}):H
\to
Aut_{\ms{set}}(\ms{Z}(\mc{M}))$
such that 
\begin{equation*}
\ms{V}_{\vardiamondsuit}(\mc{M})(l):
(\mf{T},f)
\mapsto
\left(
\mf{b}^{\mc{M}}(l)(\mf{T}),
\ms{P}_{\mc{M}}^{\mf{b}^{\mc{M}}(l)(\mf{T})}
\ni\alpha\mapsto
f(\alpha)\circ
\ms{V}(\mc{M})_{\alpha}^{\mf{b}^{\mc{M}}(l)(\mf{T})}(l^{-1})
\right).
\end{equation*}
\end{definition}
\begin{definition}
\label{18061456}
Let $\lr{\ms{H}}{\pf{D}}$ be a pre $(G,F,\uprho)-$couple of maps.
Define 
$\mf{Z}_{\ms{H}}$ and $\mf{O}_{\ms{H}}$ to be maps on 
$Dom(\pf{D})$
such that for all $\mc{N}\in Dom(\pf{D})$
\begin{equation*}
\begin{aligned}
\mf{Z}_{\ms{H}}^{\mc{N}}&\coloneqq
(\un\mapsto\ms{Z}(\mc{N}),
\ms{V}_{\vardiamondsuit}(\mc{N})\up\ms{H}_{\mc{N}}),
\\
\mf{O}_{\ms{H}}^{\mc{N}}&\coloneqq
\left(\un\mapsto\Updelta_{o}(\mc{N}),
\ms{H}_{\mc{N}}\ni l\mapsto\Updelta_{m}
(\uppsi^{\mc{N}}(l^{-1}),\mf{b}^{\mc{N}}(l))
\right).
\end{aligned}
\end{equation*}
If in addition $\lr{\ms{H}}{\pf{D}}$ is a $(G,F,\uprho)-$couple of maps
let us define $\mf{P}_{\pf{D},\ms{H}}$ to be the map on $Dom(\pf{D})$ such that
for all $\mc{O}\in Dom(\pf{D})$ 
\begin{equation*}
\mf{P}_{\pf{D},\ms{H}}^{\mc{O}}
\coloneqq
(\un_{\mc{O}}
\mapsto\pf{D}_{\mc{O}},\ms{H}_{\mc{O}}\ni l\mapsto\mf{b}^{\mc{O}}(l)
\up\pf{D}_{\mc{O}}).
\end{equation*}
If $\ms{H}$ is full we let 
$\mf{Z}$, $\mf{O}$ and $\mf{P}_{\pf{D}}$
denote
$\mf{Z}_{\ms{H}}$, $\mf{O}_{\ms{H}}$ and $\mf{P}_{\pf{D},\ms{H}}$
respectively.
Let 
$\mf{Q}^{\pf{D}}\coloneqq(\pf{D},\mf{Q}_{m}^{\pf{D}})$ 
where $\mf{Q}_{m}^{\pf{D}}$ is the map on 
$Mor_{Dom(\pf{D})^{0}}$
such that 
for all $\mc{N},\mc{M}\in Dom(\pf{D})$
and
$(\mf{h},\mf{f})\in Mor_{Dom(\pf{D})^{0}}(\mc{N},\mc{M})$
we have
\begin{equation*}
\mf{Q}_{m}^{\pf{D}}((\mf{h},\mf{f}))
\coloneqq
\mf{f}\up\pf{D}_{\mc{M}}.
\end{equation*}
Finally define
\begin{equation*}
\Updelta^{\pf{D}}\coloneqq 
(\Updelta_{o}\up Dom(\pf{D}),
\Updelta_{m}\up Mor_{Dom(\pf{D})}).
\end{equation*}
\end{definition}
\begin{remark}
\label{18061500}
Let $\lr{\ms{H}}{\pf{D},\pf{U}}$ be a $(G,F,\uprho)-$triplet of maps.
Thus
$\Updelta^{\pf{U}}\in\ms{Fct}(Dom(\pf{U})^{op},\ms{set})$
and
$\mf{Q}^{\pf{D}}\in\ms{Fct}( (Dom(\pf{D})^{0})^{op},\ms{set})$,
$\mf{Q}^{\pf{U}}\in\ms{Fct}(Dom(\pf{U})^{op},\ms{set})$, 
while 
$\mf{P}_{\pf{U}}^{\mc{N}}\in\ms{Fct}(H,\ms{set})$,
in particular
$\mf{P}_{\pf{U},\ms{H}}^{\mc{N}}\in\ms{Fct}(\ms{H}_{\mc{N}},\ms{set})$,
and
$\mf{P}_{\pf{D},\ms{H}}^{\mc{O}},
\mf{Z}_{\ms{H}}^{\mc{O}},
\mf{O}_{\ms{H}}^{\mc{O}}
\in\ms{Fct}(\ms{H}_{\mc{O}},\ms{set})$,
for any 
$\mc{N}\in Dom(\pf{U})$
and
$\mc{O}\in Dom(\pf{D})$.
\end{remark}
Next we define an equivariant phase as a section of 
maps valued in nucleon phase valued maps
contravariant under action of $Dom(\pf{U})$, and as a result 
the value of this section at any $\mc{N}$ induces a covariant section under action of $\ms{H}_{\mc{N}}$.
\begin{definition}
[Equivariant phases]
\label{01141121}
Let $\lr{\pf{U}}{\mf{m}}$ be defined equivariant phase
if 
\begin{enumerate}
\item
$\pf{U}$ is a $(G,F,\uprho)-$map;
\item
$\mf{m}\in
\prod_{\mc{N}\in Dom(\pf{U})}
\prod_{\mf{Q}\in\pf{U}_{\mc{N}}}
Mor_{\ms{set}}(\ms{P}_{\mc{N}}^{\mf{Q}},\ms{A}_{\mc{N}}^{\ast})$;
\item
$\mf{m}\in Mor_{\ms{Fct}(Dom(\pf{U})^{op},\ms{set})}(\mf{Q}^{\pf{U}},\Updelta^{\pf{U}})$.
\label{06191722}
\end{enumerate}
$\mf{m}$ is called integer 
if
$\mf{m}^{\mc{N}}(\mf{T},\alpha)$
is integer
for all $\mc{N}\in Dom(\pf{U})$,
$\mf{T}\in\pf{U}_{\mc{N}}$
and $\alpha\in\ms{P}_{\mc{N}}^{\mf{T}}$.
\end{definition}
\begin{remark}
\label{06151649}
Let $\lr{\pf{U}}{\mf{m}}$ be an equivariant phase
then Def. \ref{01141121}\eqref{06191722} is equivalent to state 
that for all
$\mc{M},\mc{N}\in Dom(\pf{U})$
and
$(\mf{h},\mf{f})\in Mor_{\mf{G}(G,F,\uprho)}(\mc{N},\mc{M})$
the following diagram is commutative 
\begin{equation}
\label{genequiv}
\xymatrix{
\pf{U}_{\mc{N}}
\ar[rr]^{\mf{m}^{\mc{N}}}
&&
\Updelta_{o}(\mc{N})
\\
&&
\\
\pf{U}_{\mc{M}}
\ar[uu]^{\mf{f}\up\pf{U}_{\mc{M}}
}
\ar[rr]_{\mf{m}^{\mc{M}}}
&&
\Updelta_{o}(\mc{M})
\ar[uu]_{\Updelta_{m}(\mf{h},\mf{f})}.}
\end{equation}
Therefore for all $\mc{N}\in Dom(\pf{U})$ and $l\in H$ 
the following diagram is commutative 
since Rmk. \ref{06132143} 
\begin{equation}
\label{19051804}
\xymatrix{
\pf{U}_{\mc{N}}
\ar[rr]^{\mf{m}^{\mc{N}}}
&&
\Updelta_{o}(\mc{N})
\\
&&
\\
\pf{U}_{\mc{N}}
\ar[uu]^{\mf{b}^{\mc{N}}(l)\up\pf{U}_{\mc{N}}}
\ar[rr]_{\mf{m}^{\mc{N}}}
&&
\Updelta_{o}(\mc{N})
\ar[uu]_{\Updelta_{m}(\uppsi^{\mc{N}}(l^{-1}),\mf{b}^{\mc{N}}(l))},}
\end{equation}
i.e.
\begin{equation*}
(\un\mapsto\mf{m}^{\mc{N}})\in
Mor_{\ms{Fct}(H,\ms{set})}
(\mf{P}_{\pf{U}}^{\mc{N}},\mf{O}^{\mc{N}}),
\forall\mc{N}\in Dom(\pf{U}).
\end{equation*}
\end{remark}
\begin{definition}
[Field determined by an equivariant phase]
\label{02021557}
Let $\lr{\pf{U}}{\mf{m}}$ be an equivariant phase,
we call field determined by $\mf{m}$
the map $\mf{S}$ defined on $Dom(\pf{U})$ such that for all $\mc{N}\in Dom(\pf{U})$
\begin{equation*}
\mf{S}^{\mc{N}}
\coloneqq
\{\mf{m}^{\mc{N}}(\mf{Q},\alpha)
\vert
\mf{Q}\in\pf{U}_{\mc{N}},
\alpha\in\ms{P}_{\mc{N}}^{\mf{Q}}
\}.
\end{equation*}
\end{definition}
\begin{proposition}
[Functoriality of $\mf{S}$]
\label{02021434}
Let $\lr{\pf{U}}{\mf{m}}$ be an equivariant phase, 
and $\mf{S}$ be the field determined by $\mf{m}$,
then for all $\mc{N}\in Dom(\pf{U})$,
$l\in H$, $\mf{Q}\in\pf{U}_{\mc{N}}$ and $\alpha\in\ms{P}_{\mc{N}}^{\mf{Q}}$
we have
$\uppsi_{\ast}^{\mc{N}}(l)(\mf{m}^{\mc{N}}(\mf{Q},\alpha))
=
\mf{m}^{\mc{N}}(\mf{b}^{\mc{N}}(l)\mf{Q},\alpha)$.
In particular we obtain
\begin{equation}
\label{02021442}
\uppsi_{\ast}^{\mc{N}}(l)\mf{S}^{\mc{N}}=\mf{S}^{\mc{N}},
\forall l\in H.
\end{equation}
Moreover for all $\mc{M}\in Dom(\pf{U})$,
$(\mf{h},\mf{f})\in Mor_{\mf{G}(G,F,\uprho)}(\mc{N},\mc{M})$,
$\mf{T}\in\pf{U}_{\mc{M}}$ and $\beta\in\ms{P}_{\mc{M}}^{\mf{T}}$
we have
$\mf{h}_{\dagger}(\mf{m}^{\mc{M}}(\mf{T},\beta)) 
=\mf{m}^{\mc{N}}(\mf{f}(\mf{T}),\beta)$.
In particular we obtain 
$\mf{h}_{\dagger}\mf{S}^{\mc{M}}\subseteq\mf{S}^{\mc{N}}$
hence 
\begin{equation}
\label{02021511}
(\mf{S},(\mf{h},\mf{f})\mapsto\mf{h}_{\dagger})\in\ms{Fct}(Dom(\pf{U})^{op},\ms{set}).
\end{equation}
\end{proposition}
\begin{proof}
Since Rmk. \ref{06151649}.
\end{proof}
Next we define sections of maps valued in
maps whose values are fragment states 
originated via the phases determined by an equivariant phase.
The value of this section at any $\mc{N}$ induces a covariant section under action of $\ms{H}_{\mc{N}}$.
Notice that we do not require the contravariance with respect to the action of some subcategory of $\mf{G}(G,F,\uprho)$.
A variant of this request is postponed when introducing extended $\mf{C}-$equivariant stabilities.
\begin{definition}
[Equivariant states associated with an equivariant phase]
\label{01141641}
$\mc{W}$ is a state associated 
with $\lr{\ms{H}}{\pf{U},\mf{m}}$ 
and equivariant on $\pf{D}$
if 
$\lr{\pf{U}}{\mf{m}}$ 
is an equivariant phase,
$\lr{\ms{H},\pf{D}}{\pf{U}}$ is a 
$(G,F,\uprho)-$triplet of maps and 
there exist $\mf{o}$ such that
\begin{equation}
\begin{aligned}
\label{01141641a}
\mc{W}
&\in
\prod_{\mc{N}\in Dom(\pf{D})}
\prod_{\mf{Q}\in\pf{D}_{\mc{N}}}
\prod_{\beta\in\ms{P}_{\mc{N}}^{\mf{Q}}}
\mf{N}^{\mc{N}}(\mf{m}^{\mc{N}}(\mf{Q},\beta)),
\\
\mf{o}
&\in
\prod_{\mc{N}\in Dom(\pf{D})}
\prod_{\mf{Q}\in\pf{D}_{\mc{N}}}
\prod_{\beta\in\ms{P}_{\mc{N}}^{\mf{Q}}}
Rep^{\mc{N}}(\mf{m}^{\mc{N}}(\mf{Q},\beta)),
\\
\mc{W}^{\mc{N}}(\mf{Q},\beta)
&=
\Uppsi_{\mf{o}^{\mc{N}}(\mf{Q},\beta)}^{-}\circ\mf{j}_{\mf{o}^{\mc{N}}(\mf{Q},\beta)},
\\
\mf{T}_{\mf{o}^{\mc{N}}(\mf{Q},\beta)}
&=
\mf{Q},
\\
\alpha_{\mf{o}^{\mc{N}}(\mf{Q},\beta)}
&=
\beta,
\\
\forall
\mc{N}
&\in Dom(\pf{D}),
\mf{Q}\in\pf{D}_{\mc{N}},
\beta\in\ms{P}_{\mc{N}}^{\mf{Q}};
\end{aligned}
\end{equation}
moreover
\begin{equation}
\label{06151653}
(\un\mapsto\ms{gr}(\mc{W}^{\mc{N}}))\in
 Mor_{\ms{Fct}(\ms{H}_{\mc{N}},\ms{set})}
(\mf{P}_{\pf{D},\ms{H}}^{\mc{N}},\mf{Z}_{\ms{H}}^{\mc{N}}),
\forall\mc{N}\in Dom(\pf{D}).
\end{equation}
We call $\mc{W}$ a state
associated with 
$\lr{\pf{U}}{\mf{m}}$ 
equivariant on $\pf{D}$
if it is a state
associated with
$\lr{\ms{H}}{\pf{U},\mf{m}}$ 
equivariant on $\pf{D}$
where $\ms{H}$ is full.
\end{definition}
\begin{remark}
Notice that
\begin{equation*}
\mc{W}
\in
\prod_{\mc{N}\in Dom(\pf{D})}
\prod_{\mf{Q}\in\pf{D}_{\mc{N}}}
\prod_{\beta\in\ms{P}_{\mc{N}}^{\mf{Q}}}
\ms{N}_{\mc{A}(\mc{N})_{\beta}^{\mf{Q}}},
\end{equation*}
in particular
$\ms{gr}(\mc{W}^{\mc{N}})$
is a section of the bundle $\lr{\ms{Z}(\mc{N}),\pf{D}_{\mc{N}}}{\Pr_{1}}$.
\end{remark}
\begin{remark}
\label{06151651}
\eqref{06151653} is equivalent to say that
for all $l\in\ms{H}_{\mc{N}}$ the following diagram is commutative
\begin{equation}
\label{01141641b}
\xymatrix{
\pf{D}_{\mc{N}}
\ar[rr]^{\ms{gr}(\mc{W}^{\mc{N}})}
&&
\ms{Z}(\mc{N})
\\
&&
\\
\pf{D}_{\mc{N}}
\ar[uu]^{\mf{b}^{\mc{N}}(l)\up\pf{D}_{\mc{N}}}
\ar[rr]_{\ms{gr}(\mc{W}^{\mc{N}})}
&&
\ms{Z}(\mc{N})
\ar[uu]_{\ms{V}_{\vardiamondsuit}(\mc{N})(l)}.}
\end{equation}
\end{remark}
\begin{definition}
[Field determined by an equivariant state]
\label{06131408}
Let $\mc{W}$ be a state associated 
to $\lr{\ms{H}}{\pf{U},\mf{m}}$ 
and equivariant on $\pf{D}$,
we call field determined by $\mc{W}$ 
the map $\mf{J}$ on $Dom(\pf{D})$ such that for all $\mc{N}\in Dom(\pf{D})$
\begin{equation*}
\mf{J}^{\mc{N}}\coloneqq\ms{gr}(\mc{W}^{\mc{N}})(\pf{D}_{\mc{N}}).
\end{equation*}
\end{definition}
\begin{proposition}
[Covariance of the field determined by an equivariant state]
\label{02031835}
Let $\mc{W}$ be a state associated with $\lr{\ms{H}}{\pf{U},\mf{m}}$ 
and equivariant on $\pf{D}$, and let $\mf{J}$ be the field determined by $\mc{W}$,
then for all $\mc{N}\in\ Dom(\pf{D})$ and $l\in\ms{H}_{\mc{N}}$,
\begin{equation*}
\ms{V}_{\vardiamondsuit}
(\mc{N})(l)
\mf{J}^{\mc{N}} 
=
\mf{J}^{\mc{N}},
\end{equation*} 
i.e.
$\mf{J}^{\mc{N}}$ is the space of the representation of $\ms{H}_{\mc{N}}$
via the action $\ms{V}_{\vardiamondsuit}(\mc{N})$.
In particular 
if we assume that there exists a von Neumann algebra $\mc{X}(\mc{N})$ 
and a map $\ms{X}(\mc{N}):H\to Aut_{\ms{CA}^{\ast}}(\mc{X}(\mc{N}))$
such that 
$\mc{X}(\mc{N})=\mc{A}(\mc{N})_{\alpha}^{\mf{Q}}$, 
and 
$\ms{X}(\mc{N})=\ms{V}(\mc{N})_{\alpha}^{\mf{Q}}$, 
for all 
$\mf{Q}\in\pf{D}_{\mc{N}}$ and $\alpha\in\ms{P}_{\mc{N}}^{\mf{Q}}$,
and define
$\ms{X}^{\mc{N}}_{\ast}:\ms{H}_{\mc{N}}\to
Aut_{\ms{set}}(\mc{X}(\mc{N})^{\ast})$
such that 
$l\mapsto(\upomega\mapsto\upomega\circ\ms{X}^{\mc{N}}(l^{-1}))$
and the map $\mf{W}$ on $Dom(\pf{D})$ such that 
\begin{equation*}
\mf{W}^{\mc{N}}\coloneqq\{\mc{W}^{\mc{N}}(\mf{Q},\alpha)
\vert\mf{Q}\in\pf{D}_{\mc{N}},\alpha\in\ms{P}_{\mc{N}}^{\mf{Q}} 
\},
\end{equation*}
then
\begin{equation}
\label{02031856}
\ms{X}_{\ast}^{\mc{N}}(h)\mf{W}^{\mc{N}}=\mf{W}^{\mc{N}},\forall h\in\ms{H}_{\mc{N}}.
\end{equation}
I.e. $\mf{W}^{\mc{N}}$ is the space of a representation of $\ms{H}_{\mc{N}}$
via the dual action of $\ms{X}^{\mc{N}}$.
\end{proposition}
\begin{proof}
Since \eqref{01141641b}.
\end{proof}
\begin{definition}
[Equivariant stabilities]
\label{11051303}
$\lr{\ms{H}}{\pf{U},\mf{m},\mc{W}}$ 
is an equivariant stability on $\pf{D}$
if $\mc{W}$ is a state
associated with $\lr{\ms{H}}{\pf{U},\mf{m}}$
and equivariant on $\pf{D}$.
Moreover 
$\lr{\pf{U}}{\mf{m},\mc{W}}$ 
is a full 
equivariant stability on $\pf{D}$
if
$\lr{\ms{H}}{\pf{U},\mf{m},\mc{W}}$ 
is an 
equivariant stability on $\pf{D}$
and $\ms{H}$ is full,
while it is integer if $\mf{m}$ it is so. 
\end{definition}
In order to introduce extended $\mf{C}-$equivariant stabilities 
we need some preparatory definitions.
As remarked in section \ref{not1}, 
for any nonempty subset $S$ of the object set of a category $A$
we let $S$ denote also the full subcategory of $A$ whose object set is $S$.
The general concept of enrichment is introduced for example in \cite[Def. $1.3.2$]{lei}.
For what concerns the following definition we need a particular case, 
and of this case only to know that if $A$ is enriched over $B$, then 
$Mor_{A}(x,y)$ is an object of $B$ for all $x,y$ objects of $A$.
\begin{definition}
\label{08141107}
Let $\mf{K}$ and $\mf{D}$ be categories,
$\mc{F}\in\ms{Fct}(\mf{K},\mf{D})$, $C$ be an object of $\mf{D}$
and $\mf{D}$ be enriched over itself 
in a way that there exists the map
\begin{equation*}
(\cdot)_{\dagger}^{\mc{F},C}\in\prod_{T\in Mor_{\mf{K}}}
Mor_{\mf{D}}(Mor_{\mf{D}}(\mc{F}_{o}(c(T)),C), Mor_{\mf{D}}(\mc{F}_{o}(d(T)),C)),
\end{equation*}
such that 
for any object $X,Y$ of $\mf{K}$ and $T\in Mor_{\mf{K}}(X,Y)$ 
we have
\begin{equation}
\label{06171458}
\begin{cases}
T_{\dagger}^{\mc{F},C}:Mor_{\mf{D}}(\mc{F}_{o}(Y),C)\to Mor_{\mf{D}}(\mc{F}_{o}(X),C),
\\
T_{\dagger}^{\mc{F},C}(g)\coloneqq g\circ\mc{F}_{m}(T),
\forall
g\in Mor_{\mf{D}}(\mc{F}_{o}(Y),C).
\end{cases}
\end{equation}
\end{definition}
\begin{remark}
Since $\mc{F}$ is a functor we deduce that 
\begin{equation*}
\left(X\mapsto Mor_{\mf{D}}(\mc{F}_{o}(X),C),T\mapsto T_{\dagger}^{\mc{F},C}\right)
\in\ms{Fct}(\mf{K}^{op},\mf{D}).
\end{equation*}
We shall use this notation mainly in the following cases,
(1) $\mf{K}=\ms{CA}^{\ast}$, $\mf{D}=\ms{BS}$, 
$\mc{F}=\ms{I}_{\ms{CA}^{\ast}\to\ms{BS}}$ 
and $C=\C$;
(2) $\mf{K}=\mf{D}=\ms{Ab}$, $\mc{F}=Id_{\ms{Ab}}$ and $C=\R$.
Thus consistently with the operation $(\cdot)_{\dagger}$ used in 
part \ref{07301047},
we let $T_{\dagger}$ denote $T_{\dagger}^{\mc{F},C}$
only in these two cases and whenever it is clear by the context which category $\mf{K}$ is involved.
\end{remark}
\begin{definition}
\label{06171611}
Let $\mf{K}$ be a category, $\mf{E}$ a subcategory of $\mf{G}(G,F,\uprho)$
and  $\mc{L}\in\ms{Fct}(\mf{K},\mf{E})$.
For $i\in\{1,2\}$ let $\mc{L}_{i}=\Pr_{i}\circ\mc{L}_{m}$ 
where $\Pr_{i}$ is the function
defined on $Mor_{\mf{G}(G,F,\uprho)}$ projecting over the $i-$th component.
Thus $\mc{L}_{1}$ and $\mc{L}_{2}$ are 
maps uniquely determined by 
\begin{equation}
\label{06222005}
\begin{aligned}
(\ms{A}\circ\mc{L}_{o},\mc{L}_{1})
&\in
\ms{Fct}(\mf{K},\ms{Ab}),
\\
(\pf{T}\circ\mc{L}_{o},\mc{L}_{2})
&\in
\ms{Fct}(\mf{K}^{op},\ms{set}),
\\
\mc{L}_{m}(\mf{t})&=(\mc{L}_{1}(\mf{t}),\mc{L}_{2}(\mf{t})),
\forall\mf{t}\in Mor_{\mf{K}}.
\end{aligned}
\end{equation}
$\ms{Ab}$ is enriched over itself,
since $Mor_{\ms{Ab}}(x,y)\in\ms{Ab}$ via the pointwise composition, inversion and identity
for any $x,y\in\ms{Ab}$, moreover 
\begin{equation*}
\begin{aligned}
(\mc{L}_{1})_{\dagger}&\in
\prod_{\mf{t}\in Mor_{\mf{K}}}
Mor_{\ms{Ab}}(\ms{A}_{\mc{L}(c(\mf{t}))}^{\ast},\ms{A}_{\mc{L}(d(\mf{t}))}^{\ast}),
\\
\mf{t}&\mapsto(\mc{L}_{1}(\mf{t}))_{\dagger}.
\end{aligned}
\end{equation*}
\end{definition}
\begin{remark}
\label{06221940}
Let $\mf{K}$ be a category, $\mf{E}$ a subcategory of $\mf{G}(G,F,\uprho)$
and  $\mc{L}\in\ms{Fct}(\mf{K},\mf{E})$, then
the two inclusions in \eqref{06222005}
mean the following
\begin{equation*}
\begin{aligned}
\mc{L}_{1}&\in
\prod_{\mf{t}\in Mor_{\mf{K}}}
Mor_{\ms{Ab}}(\ms{A}_{\mc{L}(d(\mf{t}))},\ms{A}_{\mc{L}(c(\mf{t}))}),
\\
\mc{L}_{2}&\in
\prod_{\mf{t}\in Mor_{\mf{K}}}
Mor_{\ms{set}}(\pf{T}_{\mc{L}(c(\mf{t}))},\pf{T}_{\mc{L}(d(\mf{t}))}),
\end{aligned}
\end{equation*}
and for all $\mf{x},\mf{y}\in Mor_{\mf{K}}$ such that $d(\mf{x})=c(\mf{y})$
\begin{equation*}
\begin{aligned}
\mc{L}_{1}(\mf{x}&\circ\mf{y})&=\mc{L}_{1}(\mf{x})\circ\mc{L}_{1}(\mf{y}),
\\
\mc{L}_{2}(\mf{x}&\circ\mf{y})&=\mc{L}_{2}(\mf{y})\circ\mc{L}_{2}(\mf{x}).
\end{aligned}
\end{equation*}
Moreover $(\ms{A},\Pr_{1})\in\ms{Fct}(\mf{G}(G,F,\uprho),\ms{Ab})$, so 
$(\ms{A},\Pr_{1})\circ\ms{I}_{\mf{E}\to\mf{G}(G,F,\uprho)}\circ\mc{L}\in\ms{Fct}(\mf{K},\ms{Ab})$,
let us denote it by $\mc{U}$ thus
\begin{equation*}
\begin{aligned}
\mc{U}&=(\ms{A}\circ\mc{L}_{o},\mc{L}_{1}),
\\
(\mc{L}_{1})_{\dagger}&=(\cdot)_{\dagger}^{\mc{U},\R},
\\
(\ms{A}^{\ast}\circ\mc{L}_{o},(\mc{L}_{1})_{\dagger})
&\in
\ms{Fct}(\mf{K}^{op},\ms{Ab}).
\end{aligned}
\end{equation*}
\end{remark}
\begin{definition}
\label{06152102}
Let $\pf{U}$ be a map such that $Dom(\pf{U})\subseteq Obj(\mf{G}(G,F,\uprho))$,
$\mf{C}$ a category and $\mc{F}$ a functor from $\mf{C}$ to $\mf{G}(G,F,\uprho)$.
Define 
$\Uptheta(\pf{U},\mc{F})\coloneqq\mc{F}^{-1}(Dom(\pf{U}))$ 
and 
$\mc{F}_{\pf{U}}\coloneqq\mc{F}\circ\ms{I}_{\Uptheta(\pf{U},\mc{F})\to\mf{C}}$.
Set 
$\mc{F}_{\pf{U}}^{o}\coloneqq(\mc{F}_{\pf{U}})_{o}$,
and
$\mc{F}_{\pf{U}}^{m}\coloneqq(\mc{F}_{\pf{U}})_{m}$.
$Dom(\pf{U})$ is a subcategory of $\mf{G}(G,F,\uprho)$,
so $\Uptheta(\pf{U},\mc{F})$ is a well-defined subcategory of $\mf{C}$,
and $\mc{F}_{\pf{U}}$ is a functor from 
$\Uptheta(\pf{U},\mc{F})$ to $Dom(\pf{U})$
since the convention established in section \ref{not1}.
Thus according to Def. \ref{06171611}
we can define the map 
$\mc{F}_{\pf{U}}^{\dagger}
\coloneqq
((\mc{F}_{\pf{U}})_{1})_{\dagger}$.
\end{definition}
Note that by \eqref{06231849} we have also that
\begin{equation*}
\mc{F}_{\pf{U}}\in
\ms{Fct}(\Uptheta(\pf{U},\mc{F})^{op},Dom(\pf{U})^{op}).
\end{equation*}
\begin{definition}
\label{06131837}
Let $\pf{D}$ be a pre $(G,F,\uprho)$ map,
$\mf{C}$ be a category and $\mc{F}$ a functor from $\mf{C}$ to $\mf{G}(G,F,\uprho)$.
Define 
\begin{equation*}
\begin{aligned}
\Upxi(\pf{D},\mc{F})
&\coloneqq\mc{F}^{-1}(Dom(\pf{D})^{0}),
\\
\mc{F}^{\pf{D}}&\coloneqq\mc{F}\circ\ms{I}_{\Upxi(\pf{D},\mc{F})\to\mf{C}},
\end{aligned}
\end{equation*}
and set 
$\mc{F}_{o}^{\pf{D}}\coloneqq(\mc{F}^{\pf{D}})_{o}$,
and
$\mc{F}_{m}^{\pf{D}}\coloneqq(\mc{F}^{\pf{D}})_{m}$.
\end{definition}
According to the convention in section \ref{not1}
we 
consider $\mc{F}^{\pf{D}}$ as a functor from $\Upxi(\pf{D},\mc{F})$ to $Dom(\pf{D})^{0}$,
hence by \eqref{06231849} we have also that
\begin{equation*}
\mc{F}^{\pf{D}}\in
\ms{Fct}(\Upxi(\pf{D},\mc{F})^{op},(Dom(\pf{D})^{0})^{op}).
\end{equation*}
Note that if $\pf{U}$ is a pre $(G,F,\uprho)-$map,
then $\Upxi(\pf{U},\mc{F})$ is a subcategory of $\Uptheta(\pf{U},\mc{F})$,
however 
if $\pf{U}$ is a $(G,F,\uprho)-$map then 
$\Upxi(\pf{U},\mc{F})=\Uptheta(\pf{U},\mc{F})$.
In the next definition we use $\ms{Z}$ introduced in Def. \ref{20051117},
\begin{definition}
Let $\mf{K}$ be a category, $\mf{E}$ a subcategory of $\mf{G}(G,F,\uprho)$, 
$\mc{E}\in\ms{Fct}(\mf{K},\mf{E})$
and 
$\mc{Z}\in\ms{Fct}(\mf{K}^{op},\ms{set})$ such that 
$\mc{Z}_{o}=\ms{Z}\circ\mc{E}_{o}$.
Define $\mc{Z}_{1}$ and $\mc{Z}_{2}$ maps on $Mor_{\mf{K}}$
such that 
\begin{equation*}
\begin{aligned}
&
\mc{Z}_{1}\in\prod_{\mf{t}\in Mor_{\mf{K}}}
Mor_{\ms{set}}(\pf{T}_{\mc{E}(c(\mf{t}))},\pf{T}_{\mc{E}(d(\mf{t}))}),
\\
&\mc{Z}_{2}\in\prod_{\mf{t}\in Mor_{\mf{K}}}
\prod_{\mf{Q}\in\pf{T}_{\mc{E}(c(\mf{t}))}}
Mor_{\ms{set}}\bigl(\ms{X}(\mc{E}(c(\mf{t})),\mf{Q}),
\ms{X}(\mc{E}(d(\mf{t})),\mc{Z}_{1}(\mf{t})\mf{Q})\bigr);
\end{aligned}
\end{equation*}
and for all $\mf{t}\in Mor_{\mf{K}}$
and $(\mf{Q},f)\in\ms{Z}(\mc{E}(c(\mf{t})))$
\begin{equation*}
\mc{Z}(\mf{t})(\mf{Q},f)
=(\mc{Z}_{1}(\mf{t})\mf{Q},\mc{Z}_{2}(\mf{t},\mf{Q})f).
\end{equation*}
\end{definition}
The concept below introduced is fundamental 
to integrate in the definition of extended $\mf{C}-$equivariant stabilities
the missing contravariance of $\mc{W}$ under action of $Dom(\pf{D})^{0}$
in a way that this action is induced by the conjugate of an action over observables. 
\begin{definition}
\label{0620115}
Let $\mf{H}$ be a category, $\mf{E}$ a subcategory of $\mf{G}(G,F,\uprho)$ 
and $\mc{E}\in\ms{Fct}(\mf{H},\mf{E})$.
Let $\lr{\mc{Z}}{\mc{S}}$ be defined a conjugate action via $\mc{E}$
if
\begin{enumerate}
\item
$\mc{Z}\in\ms{Fct}(\mf{H}^{op},\ms{set})$;
\label{0620115st1}
\item
$\mc{Z}_{o}=\ms{Z}\circ\mc{E}_{o}$;
\label{0620115st2}
\item
$\mc{Z}_{1}=\mc{E}_{2}$;
\label{0620115st3}
\item
\begin{equation*}
\mc{S}\in
\prod_{\mf{t}\in Mor_{\mf{H}}}
\prod_{\mf{Q}\in\pf{T}_{\mc{E}(c(\mf{t}))}}
\prod_{\beta\in\ms{P}_{\mc{E}(c(\mf{t}))}^{\mf{Q}}}
Mor_{\ms{CA}^{\ast}}
\left(\mc{A}(\mc{E}(d(\mf{t})))_{\beta}^{\mc{E}_{2}(\mf{t})\mf{Q}},
\mc{A}(\mc{E}(c(\mf{t})))_{\beta}^{\mf{Q}}\right);
\end{equation*}
\label{0620115st4}
\item
let 
$\mf{t}\in Mor_{\mf{H}}$,
$\mf{Q}\in\pf{T}_{\mc{E}(c(\mf{t}))}$
and
$\alpha\in\ms{P}_{\mc{E}(c(\mf{t}))}^{\mf{Q}}$,
thus
for any 
$f\in\ms{X}(\mc{E}(c(\mf{t})),\mf{Q})$
and
$\mf{s}\in Mor_{\mf{H}}$ such that $d(\mf{t})=c(\mf{s})$
the following diagram is commutative
\begin{equation*}
\xymatrix{
& &
\mc{A}(\mc{E}(c(\mf{t})))_{\alpha}^{\mf{Q}}
\ar[rr]^{f(\alpha)}
& &
\C
\\
& &
& &
\\
\mc{A}(\mc{E}(d(\mf{s})))_{\alpha}^{\mc{E}_{2}(\mf{t}\circ\mf{s})\mf{Q}}
\ar[rr]^{\mc{S}(\mf{s},\mc{E}_{2}(\mf{t})\mf{Q},\alpha)}
\ar[uurr]^{\mc{S}(\mf{t}\circ\mf{s},\mf{Q},\alpha)
}
& &
\mc{A}(\mc{E}(d(\mf{t})))_{\alpha}^{\mc{E}_{2}(\mf{t})\mf{Q}}
\ar[uu]^{\mc{S}(\mf{t},\mf{Q},\alpha)}
\ar[uurr]_{(\mc{Z}_{2}(\mf{t},\mf{Q})f)(\alpha)}.
& &
}
\end{equation*}
\label{0620115st5}
\end{enumerate}
\end{definition}
In particular 
$(\mc{Z}_{2}(\mf{t},\mf{Q})f)(\alpha)
=
\mc{S}(\mf{t},\mf{Q},\alpha)_{\dagger}
(f(\alpha))$,
moreover 
the left triangle 
in the diagram in Def. \ref{0620115}\eqref{0620115st5}
is well-set indeed $d(\mf{s})=d(\mf{t}\circ\mf{s})$, 
$c(\mf{t})=c(\mf{t}\circ\mf{s})$
and 
$\mc{E}_{2}(\mf{t}\circ\mf{s})=\mc{E}_{2}(\mf{s})\circ\mc{E}_{2}(\mf{t})$ 
since \eqref{06222005}.
\begin{definition}
[$\mf{C}-$equivariant stabilities]
\label{22050909}
$\mc{K}=\lr{\lr{\ms{H}}{\pf{U},\mf{m},\mc{W}}}{\mc{F}}$
is a $\mf{C}-$equivariant stability on $\pf{D}$
if
$\lr{\ms{H}}{\pf{U},\mf{m},\mc{W}}$ 
is an equivariant stability on $\pf{D}$,
$\mf{C}$ is a category
and
$\mc{F}\in\ms{Fct}(\mf{C},\mf{G}(G,F,\uprho))$.
$\lr{\lr{\pf{U}}{\mf{m},\mc{W}}}{\mc{F}}$
is a full $\mf{C}-$equivariant stability on $\pf{D}$
if $\mc{K}$ is a $\mf{C}-$equivariant stability on $\pf{D}$
and $\ms{H}$ is full.
\end{definition}
Now we are able to give the following 
\begin{definition}
[Extended $\mf{C}-$equivariant stabilities]
\label{06201206}
$\lr{\lr{\ms{H}}{\pf{U},\mf{m},\mc{W}}}{\mc{F}}$
is an extended $\mf{C}-$equivariant stability on $\pf{D}$ 
via $\mc{Z}$ and $\mc{S}$
if 
\begin{enumerate}
\item
$\lr{\lr{\ms{H}}{\pf{U},\mf{m},\mc{W}}}{\mc{F}}$
is a $\mf{C}-$equivariant stability on $\pf{D}$;
\label{06201206st1}
\item
$\lr{\mc{Z}}{\mc{S}}$ is a conjugate action via $\mc{F}^{\pf{D}}$;
\label{06201206st2}
\item
$\ms{gr}\circ\mc{W}\circ\mc{F}_{o}^{\pf{D}}
\in
Mor_{\ms{Fct}(\Upxi(\pf{D},\mc{F})^{op},\ms{set})}
(\mf{Q}^{\pf{D}}\circ\mc{F}^{\pf{D}},\mc{Z})$.
\label{06201206st3}
\end{enumerate}
\end{definition}
Notice that Def. \ref{06201206}\eqref{06201206st3}
integrates the missing symmetry of $\mc{W}$ 
w.r.t. the contravariant action of the category $Dom(\pf{D})^{0}$,
by displaying instead the contravariance of the section 
$\ms{gr}\circ\mc{W}\circ\mc{F}_{o}^{\pf{D}}$
with respect to the action of 
the inverse image $\Upxi(\pf{D},\mc{F})$ of $Dom(\pf{D})^{0}$ via the functor $\mc{F}$.
We shall see its meaning in the more general context of nucleon-fragment doublets in 
Prp. \ref{06211545} and Prp. \ref{20051810dbt} 
and its physical interpretation in Prp. \ref{12051206dbt}.
\begin{remark}
\label{06151915}
Let
$\lr{\lr{\ms{H}}{\pf{U},\mf{m},\mc{W}}}{\mc{F}}$
be an extended $\mf{C}-$equivariant stability on $\pf{D}$ via $\mc{Z}$ and $\mc{S}$,
then Def. \ref{06201206}\eqref{06201206st3} is equivalent to state that
for all $\ms{d},\ms{e}\in\Upxi(\pf{D},\mc{F})$ and 
$\mf{t}\in Mor_{\Upxi(\pf{D},\mc{F})}(\ms{e},\ms{d})$ 
the following is a commutative diagram
\begin{equation*}
\xymatrix{
\pf{D}_{\mc{F}(\ms{e})}
\ar[rr]^{\ms{gr}(\mc{W}^{\mc{F}(\ms{e})})}
&&
\ms{Z}(\mc{F}(\ms{e}))
\\
&&
\\
\pf{D}_{\mc{F}(\ms{d})}
\ar[uu]^{\mc{F}_{2}(\mf{t})\up\pf{D}_{\mc{F}(\ms{d})}}
\ar[rr]_{\ms{gr}(\mc{W}^{\mc{F}(\ms{d})})}
&&
\ms{Z}(\mc{F}(\ms{d}))
\ar[uu]_{\mc{Z}(\mf{t})}.}
\end{equation*}
\end{remark}
Next we state the symmetry properties corresponding to the data 
of an extended $\mf{C}-$equivariant stability.
\begin{proposition}
[Properties of equivariance related to a $\mf{C}-$equivariant stability]
\label{18061742}
Let $\mf{C}$ be a category and
$\mc{E}=\lr{\lr{\ms{H},\pf{U},\mf{m}}{\mc{W}}}{\mc{F}}$ be a 
$\mf{C}-$equivariant stability on $\pf{D}$.
Then $\mf{m}\circ\mc{F}_{\pf{U}}^{o}=\mf{m}\ast\un_{\mc{F}_{\pf{U}}}$
and
\begin{enumerate}
\item
$\mf{m}\circ\mc{F}_{\pf{U}}^{o}\in
Mor_{\ms{Fct}(\Uptheta(\pf{U},\mc{F})^{op},\ms{set})}
(\mf{Q}^{\pf{U}}\circ\mc{F}_{\pf{U}},
\Updelta^{\pf{U}}\circ\mc{F}_{\pf{U}})$, 
\label{18061742st1}
\item
$(\un\mapsto\mf{m}^{\mc{F}(\ms{a})})\in
Mor_{\ms{Fct}(\ms{H}_{\mc{F}(\ms{a})},\ms{set})}
(\mf{P}_{\pf{U}}^{\mc{F}(\ms{a})},\mf{O}_{\ms{H}}^{\mc{F}(\ms{a})})$,
for all
$\ms{a}\in\Uptheta(\pf{U},\mc{F})$,
\label{18061742st2}
\item 
$(\un\mapsto\ms{gr}(\mc{W}^{\mc{F}(\ms{b})}))
\in
Mor_{\ms{Fct}(\ms{H}_{\mc{F}(\ms{b})},\ms{set})}
(\mf{P}_{\pf{D},\ms{H}}^{\mc{F}(\ms{b})},
\mf{Z}_{\ms{H}}^{\mc{F}(\ms{b})})$,
for all
$\ms{b}\in\Upxi(\pf{D},\mc{F})$.
\label{18061742st3}
\end{enumerate} 
If $\mf{S}$ denotes the field determined by $\mf{m}$, 
then
\begin{equation}
\label{06152005}
\begin{aligned}
(\mf{S}\circ\mc{F}_{\pf{U}}^{o},\mc{F}_{\pf{U}}^{\dagger})
&\in\ms{Fct}(\Uptheta(\pf{U},\mc{F})^{op},\ms{set}),
\\
\uppsi_{\ast}^{\mc{F}(\ms{a})}(h)\mf{S}^{\mc{F}(\ms{a})}
&=\mf{S}^{\mc{F}(\ms{a})},
\forall\ms{a}\in\Uptheta(\pf{U},\mc{F}),
\forall h\in\ms{H}_{\mc{F}(\ms{a})}.
\end{aligned}
\end{equation}
If $\mc{E}$ is an extended $\mf{C}-$equivariant stability on $\pf{D}$ via $\mc{Z}$ and $\mc{S}$
and $\mf{J}$ denotes the field determined by $\mc{W}$,
then
\begin{equation}
\label{06152000}
\begin{aligned}
(\mf{J}\circ\mc{F}_{o}^{\pf{D}},\mc{Z}_{m})
&\in
\ms{Fct}(\Upxi(\pf{D},\mc{F})^{op},\ms{set}),
\\
\ms{V}_{\vardiamondsuit}(\mc{F}(\ms{b}))(l)
\mf{J}^{\mc{F}(\ms{b})}
&=
\mf{J}^{\mc{F}(\ms{b})},
\forall\ms{b}\in\Upxi(\pf{D},\mc{F}),
\forall l\in\ms{H}_{\mc{F}(\ms{b})}.
\end{aligned}
\end{equation}
\end{proposition}
\begin{proof}
The first sentence follows since \eqref{20061403}, 
then follows st. \eqref{18061742st1};
sts. (\ref{18061742st2},\ref{18061742st3}) are trivial.
\eqref{06152005} follows since Prp. \ref{02021434},
while
\eqref{06152000} follows since Rmk. \ref{06151915} and Prp. \ref{02031835}.
\end{proof}
\begin{remark}
\label{06152049}
The inclusion in \eqref{06152000} means 
for all $\ms{a},\ms{b}\in\Upxi(\pf{D},\mc{F})$ and 
$\mf{t}\in Mor_{\Upxi(\pf{D},\mc{F})}(\ms{b},\ms{a})$ 
that
\begin{equation*}
\mc{Z}(\mf{t})\mf{J}^{\mc{F}(\ms{a})}
\subseteq
\mf{J}^{\mc{F}(\ms{b})}.
\end{equation*}
\end{remark}
Part of the requests in the definition of the category $\mf{G}(G,F,\uprho)$ 
has been introduced to develop the semantics, i.e. the physical interpretation in section \ref{06171244},
while the remaining part to ensure the existence of an equivariant stability as we shall see in 
section \ref{07101559}. 
This because, in order to fulfill our aim to resolve the universality claim,
we are interested to the symmetry properties 
related to the category $\mf{C}$ stated in Prp. \ref{18061742}
and Def. \ref{06201206}\eqref{06201206st3},
but isolated from all what does not affect the physical meaning according to the semantics.
Prp. \ref{18061742} itself and Def. \ref{06201206}
exhibit the clues to extract the physically relevant information 
of an extended $\mf{C}-$equivariant stability,
in a context where the use of the functor $\mc{F}$ reduces at minimum,
i.e. where it suffices and operates only to apply the semantics.
Henceforth we are in the position to introduce the following 
definition where each item is titled with the description of 
the corresponding content.
\begin{definition}
[Nucleon-fragment doublets]
\label{06161650}
$\lr{S,J,\mc{Z},\mc{S},\mc{L}}{m,W,R,\pf{D},\pf{U}}$ 
is a nucleon-fragment doublet on $\mf{C}$ 
if
\begin{enumerate}
\item
$\mf{C}$ is a category;
\item
$\lr{R,\pf{D}}{\pf{U}}$ is a $(G,F,\uprho)-$triplet of maps;
 \item
\begin{equation*}
\mc{L}\in\ms{Fct}(\mf{C},\mf{G}(G,F,\uprho));
\end{equation*}
\label{06161650st1}
\item 
$S$ is a field of nucleon phases contravariant under action of $\Uptheta(\pf{U},\mc{L})$
via the conjugate of $\mc{L}_{1}$.
\begin{equation*}
\begin{aligned}
S&\in\ms{Fct}(\Uptheta(\pf{U},\mc{L})^{op},\ms{set}),
\\
S^{\ms{a}}&\subseteq\ms{A}_{\mc{L}(\ms{a})}^{\ast},
\forall\ms{a}\in\Uptheta(\pf{U},\mc{L}),
\\
S(\mf{t})&=(\mc{L}_{1}(\mf{t}))_{\dagger}\up S^{c(\mf{t})},
\forall\mf{t}\in Mor_{\Uptheta(\pf{U},\mc{L})};
\end{aligned}
\end{equation*}
\label{06161650st3}
\item
Each fiber of $S$ is covariant under action of $H$ via the dual of $\uppsi\circ\mc{L}_{o}$.
\begin{equation*}
\uppsi_{\ast}^{\mc{L}(\ms{a})}(h)S^{\ms{a}}=S^{\ms{a}},
\forall\ms{a}\in\Uptheta(\pf{U},\mc{L}),
\forall h\in R_{\mc{L}(\ms{a})};
\end{equation*}
\label{06161650st4}
\item
$m$ is a section of maps valued in $S-$valued maps,
contravariant under action of $\Uptheta(\pf{U},\mc{L})$ via $S_{m}$ and $\mc{L}_{2}$,
whose values 
induce covariant sections under action of $H$ 
via the dual of $\uppsi\circ\mc{L}_{o}$ and via $\mf{b}\circ\mc{L}_{o}$.
\begin{equation*}
m\in\prod_{\ms{a}\in\Uptheta(\pf{U},\mc{L})}
\prod_{\mf{Q}\in\pf{U}_{\mc{L}(\ms{a})}}
Mor_{\ms{set}}(\ms{P}_{\mc{L}(\ms{a})}^{\mf{Q}},S^{\ms{a}}),
\end{equation*}
such that 
\begin{equation}
\label{07271618a}
m\in Mor_{\ms{Fct}(\Uptheta(\pf{U},\mc{L})^{op},\ms{set})}
(\mf{Q}^{\pf{U}}\circ\mc{L}_{\pf{U}},\Updelta^{\pf{U}}\circ\mc{L}_{\pf{U}}),
\end{equation}
and 
\begin{equation}
\label{07271618b}
(\un\mapsto m^{\ms{a}})\in 
Mor_{\ms{Fct}(R_{\mc{L}(\ms{a})},\ms{set})}
(\mf{P}_{\pf{U}}^{\mc{L}(\ms{a})},
\mf{O}_{R}^{\mc{L}(\ms{a})}),
\forall\ms{a}\in\Uptheta(\pf{U},\mc{L}).
\end{equation}
\label{06191003}
\item
$J$ is a field of disjoint union over operations of set of maps 
-
with values fragment states originated via the nucleon phases 
determined by $m$
-
contravariant under action of $\Upxi(\pf{D},\mc{L})$
via $J_{m}$ where $J_{1}=\mc{L}_{2}$ and $J_{2}$ is induced by the conjugate of 
the evaluation of $\mc{S}$. 
\begin{equation}
\label{06211606}
J\in\ms{Fct}(\Upxi(\pf{D},\mc{L})^{op},\ms{set}),
\end{equation}
and $\lr{\mc{Z}}{\mc{S}}$ is a conjugate action via $\mc{L}^{\pf{D}}$
such that 
\begin{enumerate}
\item
$J^{\ms{b}}\subset\mc{Z}(\ms{b})$, for all $\ms{b}\in\Upxi(\pf{D},\mc{L})$,
\label{06181433st5}
\item
$J(\mf{t})=\mc{Z}(\mf{t})\up J^{c(\mf{t})}$
for all $\mf{t}\in Mor_{\Upxi(\pf{D},\mc{L})}$;
\label{06181433st6}
\end{enumerate}
and
\begin{equation}
\label{06181526}
J^{\ms{b}}
\subset
\coprod_{\mf{Q}\in\pf{D}_{\mc{L}(\ms{b})}}
\prod_{\beta\in\ms{P}_{\mc{L}(\ms{b})}^{\mf{Q}}}
\mf{N}^{\mc{L}(\ms{b})}(m^{\ms{b}}(\mf{Q},\beta)),
\forall\ms{b}\in\Upxi(\pf{D},\mc{L}).
\end{equation}
\label{06161650st6}
\item
Each fiber of $J$ is covariant under action of $H$
via a map induced by the dual of $\ms{V}\circ\mc{L}_{o}$.
\begin{equation*}
\ms{V}_{\vardiamondsuit}(\mc{L}(\ms{b}))(l)
J^{\ms{b}}
=
J^{\ms{b}},
\forall\ms{b}\in\Upxi(\pf{D},\mc{L}),
\forall l\in R_{\mc{L}(\ms{b})};
\end{equation*}
\label{06161650st7}
\item
$W$ is a map such that
$\ms{gr}\circ W$ is a section of $J$,
contravariant under action of $\Upxi(\pf{D},\mc{L})$ via $J_{m}$ and $\mc{L}_{2}$,
whose values induce sections covariant under action of $H$ 
via the dual of $\ms{V}\circ\mc{L}_{o}$ and via $\mf{b}\circ\mc{L}_{o}$.
There exists $\mf{o}$ such that
\begin{equation}
\begin{aligned}
\label{06232128}
W
&\in
\prod_{\ms{b}\in\Upxi(\pf{D},\mc{L})}
\prod_{\mf{Q}\in\pf{D}_{\mc{L}(\ms{b})}}
\prod_{\beta\in\ms{P}_{\mc{L}(\ms{b})}^{\mf{Q}}}
\mf{N}^{\mc{L}(\ms{b})}(m^{\ms{b}}(\mf{Q},\beta)),
\\
\mf{o}
&\in
\prod_{\ms{b}\in\Upxi(\pf{D},\mc{L})}
\prod_{\mf{Q}\in\pf{D}_{\mc{L}(\ms{b})}}
\prod_{\beta\in\ms{P}_{\mc{L}(\ms{b})}^{\mf{Q}}}
Rep^{\mc{L}(\ms{b})}(m^{\ms{b}}(\mf{Q},\beta)),
\\
W^{\ms{b}}(\mf{Q},\beta)
&=
\Uppsi_{\mf{o}^{\ms{b}}(\mf{Q},\beta)}^{-}
\circ
\mf{j}_{\mf{o}^{\ms{b}}(\mf{Q},\beta)},
\\
\mf{T}_{\mf{o}^{\ms{b}}(\mf{Q},\beta)}
&=
\mf{Q},
\\
\alpha_{\mf{o}^{\ms{b}}(\mf{Q},\beta)}
&=
\beta,
\\
\forall
\ms{b}&\in\Upxi(\pf{D},\mc{L}),
\mf{Q}\in\pf{D}_{\mc{L}(\ms{b})},
\beta\in\ms{P}_{\mc{L}(\ms{b})}^{\mf{Q}};
\end{aligned}
\end{equation}
\begin{equation}
\label{07271620a}
\ms{gr}\circ W\in
Mor_{\ms{Fct}(\Upxi(\pf{D},\mc{L})^{op},\ms{set})}
(\mf{Q}^{\pf{D}}\circ\mc{L}^{\pf{D}},J)
\end{equation}
and
\begin{equation}
\label{07271620b}
(\un\mapsto \ms{gr}(W^{\ms{b}}))
\in
Mor_{\ms{Fct}(R_{\mc{L}(\ms{b})},\ms{set})}
(\mf{P}_{\pf{D},R}^{\mc{L}(\ms{b})},\mf{Z}_{R}^{\mc{L}(\ms{b})}),
\forall\ms{b}\in\Upxi(\pf{D},\mc{L}).
\end{equation}
\label{06191405}
\end{enumerate}
\end{definition}
\begin{remark}
Notice that the first request in \eqref{06232128} implies 
the existence of an $\mf{o}$ satisfying the second request and 
for which the first equality in \eqref{06232128} holds.
However even \eqref{07271620a} does not imply the second and third equalities in 
\eqref{06232128}.
\end{remark}
Since Def. \ref{06161650} abstracts properties 
of an extended $\mf{C}-$equivariant stability, 
it is natural to expect that
\begin{corollary}
\label{06241610}
Let $\mf{C}$ be a category and
$\mc{E}=\lr{\lr{\ms{H},\pf{U},\mf{m}}{\mc{W}}}{\mc{F}}$ be an 
extended $\mf{C}-$equivariant stability 
on $\pf{D}$ via $\mc{Z}$ and $\mc{S}$,
and let $\mf{S}$ and $\mf{J}$ denote 
the fields determined by $\mf{m}$ and $\mc{W}$ respectively.
Then
$\lr{(\mf{S}\circ\mc{F}_{\pf{U}}^{o},\mc{F}_{\pf{U}}^{\dagger}),
(\mf{J}\circ\mc{F}_{o}^{\pf{D}},\mc{Z}_{m}),
\mc{Z},\mc{S},\mc{F}}
{\mf{m}\ast\un_{\mc{F}_{\mf{U}}},\mc{W}\circ\mc{F}_{0}^{\pf{D}},
\ms{H},\pf{D},\pf{U}}$
is a nucleon-fragment doublet on $\mf{C}$. 
\end{corollary}
\begin{proof}
Since Prp. \ref{18061742} and the definitions of $\mf{S}$ and $\mf{J}$.
\end{proof}
\begin{definition}
\label{06251940}
We call related to $\mc{E}$
the nucleon-fragment doublet constructed in Cor. \ref{06241610}.
\end{definition}
Up to the end of section \ref{06241544} we 
let $\mc{T}=\lr{S,J,\mc{Z},\mc{S},\mc{F}}{m,W,R,\pf{D},\pf{U}}$ 
be a fixed but arbitrary nucleon-fragment doublet on $\mf{C}$. 
Clearly the following results apply for the nucleon-fragment doublet 
related to any extended $\mf{C}-$equivariant stability 
on $\pf{D}$ via $\mc{Z}$ and $\mc{S}$. 
In Prp. \ref{06211545} we make explicit the symmetries of $S$, $m$, $J$ and $W$,
in this way clarifying the meaning of the titles in the items of 
Def. \ref{06161650}.
We emphatize 
that all the actions over functionals involved are conjugate of 
actions over observables. 
As a result Prp. \ref{20051810dbt} 
describes the properties of invariance of a nucleon-fragment doublet, 
physically interpeted as invariance of mean values in 
Prp. \ref{12051206dbt} with the help of Prp. \ref{02202114dbt} and Prp. \ref{07031844}. 
\begin{proposition}
[Properties of equivariance of a nucleon-fragment doublet]
\label{06211545}
Let 
$\mf{t}\in Mor_{\Uptheta(\pf{U},\mc{F})}$
and
$\mf{h}\in Mor_{\Upxi(\pf{D},\mc{F})}$. 
Thus
\begin{enumerate}
\item
$S(\mf{t}\circ\mf{l})=S(\mf{l})\circ S(\mf{t})$
for all 
$\mf{l}\in Mor_{\Uptheta(\pf{U},\mc{F})}$
such that 
$d(\mf{t})=c(\mf{l})$
and
\begin{equation*}
\begin{aligned}
S^{c(\mf{t})}&\subseteq\ms{A}_{\mc{F}(c(\mf{t}))}^{\ast},
\\
S(\mf{t})S^{c(\mf{t})}&\subseteq S^{d(\mf{t})},
\\
S(\mf{t})\mf{u}&=\mf{u}\circ\mc{F}_{1}(\mf{t}),
\forall\mf{u}\in S^{c(\mf{t})}.
\end{aligned}
\end{equation*}
\label{06211545st1}
\item
for all 
$\mf{Q}\in\pf{U}_{\mc{F}(c(\mf{t}))}$
and
$\alpha\in\ms{P}_{\mc{F}(c(\mf{t}))}^{\mf{Q}}$
\begin{equation*}
\begin{aligned}
m^{c(\mf{t})}(\mf{Q},\alpha)&\in S^{c(\mf{t})},
\\
m^{c(\mf{t})}(\mf{Q},\alpha)
\circ
\mc{F}_{1}(\mf{t})
&=
m^{d(\mf{t})}(\mc{F}_{2}(\mf{t})\mf{Q},\alpha),
\\
m^{c(\mf{t})}(\mf{Q},\alpha)
\circ
\uppsi^{\mc{F}(c(\mf{t}))}(l^{-1})
&=
m^{c(\mf{t})}(\mf{b}^{\mc{F}(c(\mf{t}))}(l)\mf{Q},\alpha),
\forall l\in R_{\mc{F}(c(\mf{t}))}.
\end{aligned}
\end{equation*}
\label{06211545st3}
\item
$J(\mf{h}\circ\mf{i})=J(\mf{i})\circ J(\mf{h})$,
for all 
$\mf{i}\in Mor_{\Upxi(\pf{D},\mc{F})}$
such that 
$d(\mf{h})=c(\mf{i})$,
\begin{equation*}
J(\mf{h})J^{c(\mf{h})}\subseteq J^{d(\mf{h})},
\end{equation*}
and for all $(\mf{I},f)\in J^{c(\mf{h})}$
\begin{equation}
\label{06231333}
\begin{aligned}
\mf{I}\in\pf{D}_{\mc{F}(c(\mf{h}))},
f&\in
\prod_{\beta\in\ms{P}_{\mc{F}(c(\mf{h}))}^{\mf{I}}}
\mf{N}^{\mc{F}(c(\mf{h}))}(m^{c(\mf{h})}(\mf{I},\beta))
\cap
\ms{N}_{\mc{A}(\mc{F}(c(\mf{h})))_{\beta}^{\mf{I}}},
\\
J(\mf{h})(\mf{I},f)&=(\mc{F}_{2}(\mf{h})\mf{I},f_{\mf{h}}),
\\
f_{\mf{h}}:
\ms{P}_{\mc{F}(d(\mf{h}))}^{\mc{F}_{2}(\mf{h})\mf{I}}
&\ni
\gamma
\mapsto
f(\gamma)\circ\mc{S}(\mf{h},\mf{I},\gamma).
\end{aligned}
\end{equation}
\label{06211545st6}
\item
For all 
$\mf{I}\in\pf{D}_{\mc{F}(c(\mf{h}))}$,
$\beta\in\ms{P}_{\mc{F}(c(\mf{h}))}^{\mf{I}}$
and 
$h\in R_{\mc{F}(c(\mf{h}))}$
\begin{equation*}
\begin{aligned}
(\mf{I},W^{c(\mf{h})}(\mf{I}))
&\in
J^{c(\mf{h})},
\\
W^{c(\mf{h})}(\mf{I},\beta)
\circ\mc{S}(\mf{h},\mf{I},\beta)
&=
W^{d(\mf{h})}(\mc{F}_{2}(\mf{h})\mf{I},\beta),
\\
W^{c(\mf{h})}(\mf{I},\beta)
\circ
\ms{V}(\mc{F}(c(\mf{h})))_{\beta}^{\mf{b}^{\mc{F}(c(\mf{h}))}(h)\mf{I}}(h^{-1})
&=
W^{c(\mf{h})}(\mf{b}^{\mc{F}(c(\mf{h}))}(h)\mf{I},\beta).
\end{aligned}
\end{equation*}
\label{06211545st8}
\end{enumerate}
\end{proposition}
\begin{proof}
St.\eqref{06211545st1} follows since Def. \ref{06161650}\eqref{06161650st3},
st.\eqref{06211545st3} by Def. \ref{06161650}\eqref{06191003},
the first item in \eqref{06231333} follows
since Def. \ref{06161650}\eqref{06181433st5} and \eqref{06181526},
the remaining of st.\eqref{06211545st6} by Def. \ref{06161650}\eqref{06161650st6}.
St.\eqref{06211545st8} follows since \eqref{07271620a} and \eqref{07271620b}.
\end{proof}
As often remarked we introduce the structure of nucleon-fragment doublet on a category
in order to resolve the universality claim.
The next definition specifies the characteristics of the doublet,
equivalently stated in Prp. \ref{06211545}, 
which primarly serve to this end, as we will show in Cor. \ref{19061936}.
\begin{definition}
[$\mc{T}-$nucleon phase and $\mc{T}-$fragment state]
\label{07271553}
We call $m$ and $W$ the $\mc{T}-$nucleon phase and $\mc{T}-$fragment state. 
Moreover 
let \eqref{07271618a}
be called 
the equivariance of the $\mc{T}-$nucleon phase under contravariant action of 
$\Uptheta(\pf{U},\mc{L})$,
and let 
\eqref{07271618b}
be called 
the equivariance of the $\mc{T}-$nucleon phase under action of $H$.
Let 
\eqref{07271620a}
be called 
the equivariance of the $\mc{T}-$fragment state under contravariant action of 
$\Upxi(\pf{D},\mc{L})$,
and let 
\eqref{07271620b}
be called 
the equivariance of the $\mc{T}-$fragment state under action of $H$.
Complexively we call all the previous properties the 
equivariances of the $\mc{T}-$nucleon phase and $\mc{T}-$fragment state,
expanded 
if we add also the property \eqref{06232128}.
Finally we call $\mc{T}-$resolution of the equivariant form of the universality claim,
the set of the expanded  
equivariances of the $\mc{T}-$nucleon phase and $\mc{T}-$fragment state.
\end{definition}
Since all the above introduced state-evaluated maps
are equivariant under actions implemented by conjugate of actions over observables,
easily Prp. \ref{06211545} gives rise to invariance of mean values as stated in the following
\begin{proposition}
[Properties of invariance of nucleon-fragment doublets]
\label{20051810dbt}
Let
\begin{enumerate}
\item
$\ms{a},\ms{b}\in\Uptheta(\pf{U},\mc{F})$,
$\mf{T}\in\pf{U}_{\mc{F}(\ms{b})}$,
$\beta\in\ms{P}_{\mc{F}(\ms{b})}^{\mf{T}}$,
\item
$\mf{t}\in Mor_{\Uptheta(\pf{U},\mc{F})}(\ms{a},\ms{b})$,
\item
$\mf{I}\in\pf{U}_{\mc{F}(\ms{a})}$,
$\alpha\in\ms{P}_{\mc{F}(\ms{a})}^{\mf{I}}$,
$\ms{f}\in\ms{A}_{\mc{F}(\ms{a})}$,
$l\in R_{\mc{F}(\ms{a})}$;
\end{enumerate}
then
\begin{equation*}
\begin{aligned}
m^{\ms{b}}(\mf{T},\beta)(\mc{F}_{1}(\mf{t})\ms{f})
&=
m^{\ms{a}}(\mc{F}_{2}(\mf{t})\mf{T},\beta)(\ms{f}),
\\
m^{\ms{a}}(\mf{I},\alpha)(\ms{f})
&=
m^{\ms{a}}(\mf{b}^{\mc{F}(\ms{a})}(l)\mf{I},\alpha)
(\uppsi^{\mc{F}(\ms{a})}(l)\ms{f}),
\end{aligned}
\end{equation*}
Moreover let
\begin{enumerate}
\item
$\ms{d},\ms{e}\in\Upxi(\pf{D},\mc{F})$,
$\mf{R}\in\pf{D}_{\mc{F}(\ms{e})}$,
$\delta\in\ms{P}_{\mc{F}(\ms{e})}^{\mf{R}}$,
$b\in\mc{A}(\mc{F}(\ms{d}))_{\delta}^{\mc{F}_{2}(\mf{p})\mf{R}}$,
\item
$\mf{p}\in Mor_{\Upxi(\pf{D},\mc{F})}(\ms{d},\ms{e})$,
\item
$\mf{O}\in\pf{D}_{\mc{F}(\ms{d})}$,
$\gamma\in\ms{P}_{\mc{F}(\ms{d})}^{\mf{O}}$,
$a\in\mc{A}(\mc{F}(\ms{d}))_{\gamma}^{\mf{O}}$,
$h\in R_{\mc{F}(\ms{d})}$;
\end{enumerate}
then 
\begin{equation*}
\begin{aligned}
W^{\ms{e}}(\mf{R},\delta)(\mc{S}(\mf{p},\mf{R},\delta)b)
&=
W^{\ms{d}}(\mc{F}_{2}(\mf{p})\mf{R},\delta)(b),
\\
W^{\ms{d}}(\mf{O},\gamma)(a)
&=
W^{\ms{d}}(\mf{b}^{\mc{F}(\ms{d})}(h)\mf{O},\gamma)
(\ms{V}(\mc{F}(\ms{d}))_{\gamma}^{\mf{O}}(h)a).
\end{aligned}
\end{equation*}
\end{proposition}
\begin{proof}
Since Prp. \ref{06211545} and \eqref{09051254b}.
\end{proof}
\begin{proposition}
\label{02202114dbt}
Under the hypothesis of Prp. \ref{20051810dbt} 
we have
\begin{enumerate}
\item
$\mf{s}(\mc{F}(\ms{a}))\equiv$
the nucleon system generated by the fissioning system $\mf{u}(\ms{a})$;
\item
$\mf{s}(\mc{O}(\mc{F}(\ms{a}))_{\alpha}^{\mf{I}})\equiv$
the fragment system whose observable algebra is $\mc{A}(\mc{F}(\ms{a}))_{\alpha}^{\mf{I}}$
and whose dynamics is $(\ep^{\mc{F}(\ms{a})})_{\alpha}^{\mf{I}}(-\alpha^{-1}(\cdot))$, if $\alpha>0$; 
\item
$\mf{s}((\ps{\upvarphi}^{\mc{F}(\ms{a})})_{\alpha}^{\mf{I}})\equiv$
the state of thermal equilibrium 
$(\ps{\upvarphi}^{\mc{F}(\ms{a})})_{\alpha}^{\mf{I}}$
at the inverse temperature $\alpha$
of 
$\mf{s}(\mc{O}(\mc{F}(\ms{a}))_{\alpha}^{\mf{I}})$, if $\alpha>0$; 
\item
$\mf{s}(m^{\ms{a}}(\mf{I},\alpha))\equiv$
the phase of $\mf{s}(\mc{F}(\ms{a}))$
occurring by performing the operation $\mf{u}(\mf{I})$ 
on 
$\mf{s}((\ps{\upvarphi}^{\mc{F}(\ms{a})})_{\alpha}^{\mf{I}})$;
\label{07031656}
\item
$\mf{s}(W^{\ms{e}}(\mf{R},\delta))\equiv$
the state of
$\mf{s}(\mc{O}\mc{F}(\ms{e})_{\delta}^{\mf{R}})$
originated via 
$\mf{s}(m^{\ms{e}}(\mf{R},\delta))$;
\label{07031655}
\item
$\mf{s}(\mc{O}(\mc{F}(\ms{a}))_{\alpha}^{\mf{b}^{\mc{F}(\ms{a})}(l)\mf{I}})\equiv$
the fragment system whose observable algebra is 
$\mc{A}(\mc{F}(\ms{a}))_{\alpha}^{\mf{b}^{\mc{F}(\ms{a})}(l)\mf{I}}$ 
and whose dynamics is 
$\ms{ad}((\eta^{\mc{F}(\ms{a})})_{\alpha}^{\mf{I}}(l))
\circ(\ep^{\mc{F}(\ms{a})})_{\alpha}^{\mf{I}}(-\alpha^{-1}(\cdot))$, if $\alpha>0$;
\item
$\mf{s}((\ps{\upvarphi}^{\mc{F}(\ms{a})})_{\alpha}^{\mf{b}^{\mc{F}(\ms{a})}(l)\mf{I}})\equiv$
the state of thermal equilibrium 
$(\ps{\upvarphi}^{\mc{F}(\ms{a})})_{\alpha}^{\mf{b}^{\mc{F}(\ms{a})}(l)\mf{I}}$
at the inverse temperature $\alpha$
of 
$\mf{s}(\mc{O}(\mc{F}(\ms{a}))_{\alpha}^{\mf{b}^{\mc{F}(\ms{a})}(l)\mf{I}})$;
\item
$\mf{s}(m^{\ms{a}}(\mf{b}^{\mc{F}(\ms{a})}(l)\mf{I},\alpha))\equiv$
the phase of $\mf{s}(\mc{F}(\ms{a}))$ occurring by performing 
on
$\mf{s}((\ps{\upvarphi}^{\mc{F}(\ms{a})})_{\alpha}^{\mf{b}^{\mc{F}(\ms{a})}(l)\mf{I}})$
the operation obtained by transforming through the action of $l$ the operation $\mf{u}(\mf{I})$;
\label{07031808}
\item
$\mf{s}(W^{\ms{d}}(\mf{b}^{\mc{F}(\ms{d})}(h)\mf{O},\gamma))\equiv$
the state of $\mf{s}(\mc{O}(\mc{F}(\ms{d}))_{\gamma}^{\mf{b}^{\mc{F}(\ms{d})}(h)\mf{O}})$
originated via 
$\mf{s}(m^{\ms{d}}(\mf{b}^{\mc{F}(\ms{d})}(h)\mf{O},\gamma))$;
\label{07031809}
\item
$\mf{s}(m^{\ms{b}}(\mf{T},\beta))\equiv$ the phase of $\mf{s}(\mc{F}(\ms{b}))$
occurring by performing the operation $\mf{u}(\mf{T})$
on $\mf{s}((\ps{\upvarphi}^{\mc{F}(\ms{b})})_{\beta}^{\mf{T}})$;
\label{07031810}
\item
$\mf{s}(m^{\ms{a}}(\mc{F}_{2}(\mf{t})(\mf{T}),\beta))\equiv$
the phase of $\mf{s}(\mc{F}(\ms{a}))$ occurring by performing 
on
$\mf{s}((\ps{\upvarphi}^{\mc{F}(\ms{a})})_{\beta}^{\mc{F}_{2}(\mf{t})(\mf{T})})$
the operation obtained by transforming through the action of $\mf{u}(\mc{F}_{2}(\mf{t}))$ 
the operation $\mf{u}(\mf{T})$;
\label{07031811}
\item
$\mf{s}(W^{\ms{d}}(\mc{F}_{2}(\mf{p})\mf{R},\delta))\equiv$
the state of
$\mf{s}(\mc{O}(\mc{F}(\ms{d}))_{\delta}^{\mc{F}_{2}(\mf{p})\mf{R}})$
originated via 
$\mf{s}(m^{\ms{d}}(\mc{F}_{2}(\mf{p})\mf{R},\delta))$.
\label{07031812}
\end{enumerate}
\end{proposition}
\begin{proof}
Sts. (\ref{07031656},\ref{07031655}) 
follow by \eqref{06232128},
sts. (\ref{07031808}-\ref{07031812}) follow by 
sts. (\ref{07031656},\ref{07031655}),
the remaining statements are trivial. 
\end{proof}
\begin{convention}
Let
$\ms{a}\in\Upxi(\pf{D},\mc{F})$,
$\mf{I}\in\pf{D}_{\mc{F}(\ms{a})}$,
$\delta\in\ms{P}_{\mc{F}(\ms{a})}^{\mf{I}}$,
set
$\ms{F}(\ms{a})_{\delta}^{\mf{I}}
\doteq
\ms{F}_{(\ps{\upvarphi}^{\mc{F}(\ms{a})})_{\delta}^{\mf{I}}}
((\mf{a}_{\mc{F}(\ms{a})})_{\delta}^{\mf{I}})$,
i.e.
\begin{equation*}
\ms{F}(\ms{a})_{\delta}^{\mf{I}}
=
\left\{h\in F\mid
\left((\ps{\upvarphi}^{\mc{F}(\ms{a})})_{\delta}^{\mf{I}}
\circ(\upeta^{\mc{F}(\ms{a})})_{\delta}^{\mf{I}}\circ 
j_{2}\right)(h)
=(\ps{\upvarphi}^{\mc{F}(\ms{a})})_{\delta}^{\mf{I}}\right\}.
\end{equation*}
\end{convention}
\begin{proposition}
[Thermal nature-fragment state NCG origination-Phase transition
of the nucleon phase $m^{\ms{a}}$]
\label{07031844}
Under the hypothesis of Prp. \ref{20051810dbt} we obtain
\begin{enumerate}
\item
\emph
{Thermal nature of the nucleon phase $m^{\ms{a}}(\mf{I},\alpha)$ 
and stability under variation of the operations $\mf{T}_{\mf{p}}$.}
$\mf{s}(m^{\ms{a}}(\mf{I},\alpha))\equiv$  
the phase of the nucleon system generated by the fissioning system $\mf{u}(\ms{a})$,
occurring by performing the operation $\mf{u}(\mf{T}_{\mf{p}})$ 
on the state of thermal equilibrium 
$\ps{\upvarphi}_{\mf{p}}$
at the inverse temperature $\alpha_{\mf{p}}$
of the fragment system whose observable algebra is 
$\mc{A}_{\mf{p}}$
and whose dynamics is 
$\ep_{\mf{p}}(-\alpha_{\mf{p}}^{-1}(\cdot))$,
for all $\mf{p}\in Rep^{\mc{F}(\ms{a})}(\mf{c})$ such that $\alpha_{\mf{p}}>0$.
\item
\emph
{Noncommutative geometric and thermal origination 
of fragment states via the nucleon phase $m^{\ms{a}}(\mf{I},\alpha)$.}
\begin{equation*}
\Uppsi_{\mf{r}}^{-}\circ\mf{j}_{\mf{r}},
\end{equation*}
is the $\ps{\upvarphi}_{\mf{r}}-$normal state 
originated via $\mf{s}(m^{\ms{a}}(\mf{I},\alpha))$,
of the fragment system whose observable algebra is 
$\mc{A}_{\mf{r}}$
and whose dynamics is 
$\ep_{\mf{r}}(-\alpha_{\mf{r}}^{-1}(\cdot))$,
for all $\mf{r}\in Rep^{\mc{F}(\ms{a})}(m^{\ms{a}}(\mf{I},\alpha))$
such that $\alpha_{\mf{r}}>0$.
\item
\emph
{Phase transition of $m^{\ms{a}}$ via symmetry breaking.}
$\mf{s}(m^{\ms{a}}(\mf{I},\xi))$
exists only 
for those
$\xi\in\ms{I}_{\mc{F}(\ms{a})}^{\mf{I}}$
such that 
$\ms{F}(\ms{a})_{\xi}^{\mf{I}}
\supseteq
\ms{F}(\ms{a})_{(\beta_{c}^{\mc{F}(\ms{a})})^{\mf{I}}}^{\mf{I}}$.
\end{enumerate}
\end{proposition}
\begin{proof}
Since the definition of $\ms{P}_{\mc{N}}$ for any $\mc{N}$ object of $\mf{G}(G,F,\uprho)$,
Prp. \ref{12051738} and Prp. \ref{06251847}.
\end{proof}
Jointly Prp. \ref{02202114dbt} the next result physically interpret 
Prp. \ref{20051810dbt} 
\begin{proposition}
[Mean values invariance related to a nucleon-fragment doublet]
\label{12051206dbt}
Under the hypothesis of Prp. \ref{20051810dbt} we obtain
\begin{enumerate}
\item
\emph
{$\Upxi(\pf{D},\mc{F})-$invariance of the mean value in $m$.}
The following values are equal
\begin{itemize}
\item
the mean value in $\mf{s}(m^{\ms{b}}(\mf{T},\beta))$
of the observable
obtained by transforming through the action of $\mf{u}(\mc{F}_{1}(\mf{t}))$
the observable $\mf{u}(\ms{f})$, 
\item
the mean value in 
$\mf{s}(m^{\ms{a}}(\mc{F}_{2}(\mf{t})(\mf{T}),\beta))$
of the observable $\mf{u}(\ms{f})$.
\end{itemize}
\label{21050538dbt}
\item
\emph
{$H-$invariance of the mean value in $m^{\ms{a}}$.}
The following values are equal
\begin{itemize}
\item
the mean value in $\mf{s}(m^{\ms{a}}(\mf{I},\alpha))$
of the observable $\mf{u}(\ms{f})$, 
\item
the mean value 
in $\mf{s}(m^{\ms{a}}(\mf{b}^{\mc{F}(\ms{a})}(l)(\mf{I}),\alpha))$
of the observable obtained by transforming 
through the action of $l$ the observable $\mf{u}(\ms{f})$,
\end{itemize}
\label{12051206st5dbt}
\item
\emph
{$\Upxi(\pf{D},\mc{F})-$invariance of the mean value in $\mc{W}$.}
The following values are equal
\begin{itemize}
\item
the mean value in
$\mf{s}(W^{\ms{e}}(\mf{R},\delta))$
of the observable obtained by transforming through the action of
$\mf{u}(\mc{S}(\mf{p},\mf{R},\delta))$
the observable $\mf{u}(b)$,
\item
the mean value in
$\mf{s}(W^{\ms{d}}(\mc{F}_{2}(\mf{p})\mf{R},\delta))$
of the observable $\mf{u}(b)$;
\end{itemize}
\label{06201815dbt} 
\item
\emph
{$H-$invariance of the mean value in $W^{\ms{d}}$.}
The following values are equal
\begin{itemize}
\item
the mean value in $\mf{s}(W^{\ms{d}}(\mf{O},\gamma))$
of the observable $\mf{u}(a)$,
\item
the mean value 
in $\mf{s}(W^{\ms{d}}(\mf{b}^{\mc{F}(\ms{d})}(h)(\mf{O}),\gamma))$
of the observable obtained by transforming 
through the action of $h$ the observable $\mf{u}(a)$.
\end{itemize}
\label{12051206st4dbt}
\end{enumerate}
\end{proposition}
\begin{proof}
Since Prp. \ref{20051810dbt}. 
\end{proof}
\subsection{Construction of an equivariant stability}
\label{07101559}
In this section we construct an equivariant stability in Thm. \ref{01151104}.
\begin{definition}
\label{01151759}
Define
\begin{equation*}
\begin{aligned}
\pf{T}_{\bullet}
&:Obj(\mf{G}(G,F,\uprho))\ni\mc{M}
\mapsto
\pf{T}_{\mc{M}},
\\
\ov{\ms{m}}_{\bullet}
&:Obj(\mf{G}(G,F,\uprho))\ni\mc{M}
\mapsto
\ov{\ms{m}}^{\mc{M}},
\\
Dom(\pf{V}_{\bullet})
&\coloneqq
\{\mc{M}\in Obj(\mf{G}(G,F,\uprho))
\mid
\pf{V}_{\bullet}(\mc{M})\neq\varnothing\},
\end{aligned}
\end{equation*}
moreover set the map 
$\mc{V}_{\bullet}$ defined on $Dom(\pf{V}_{\bullet})$ such that 
for all $\mc{M}\in Dom(\pf{V}_{\bullet})$,
$\mf{T}\in\pf{V}_{\bullet}(\mc{M})$ 
and 
$\alpha\in\ms{P}_{\mc{M}}^{\mf{T}}$
\begin{equation*}
\begin{aligned}
\mc{V}_{\bullet}(\mc{M})
&\in\prod_{\mf{Q}\in\pf{V}_{\bullet}(\mc{M})}
\prod_{\beta\in\ms{P}_{\mc{M}}^{\mf{Q}}}
\mf{N}^{\mc{M}}(\ov{\ms{m}}^{\mc{M}}(\mf{Q},\beta)),
\\
\mc{V}_{\bullet}(\mc{M})(\mf{\mf{T},\alpha})
&\coloneqq
\Uppsi_{(\ms{e}^{\mc{M}}\circ\mf{t}^{\mc{M}})(\mf{T},\alpha)}^{-}
\circ
\mf{j}_{(\ms{e}^{\mc{M}}\circ\mf{t}^{\mc{M}})(\mf{T},\alpha)}.
\end{aligned}
\end{equation*}
\end{definition}
\begin{remark}
\label{01151800}
$(\ms{e}^{\mc{M}}\circ\mf{t}^{\mc{M}})
(\mf{T},\alpha)\in 
Rep^{\mc{M}}(\ov{\ms{m}}^{\mc{M}}(\mf{T},\alpha))$
for all
$\mc{M}\in Dom(\pf{V}_{\bullet})$,
$\mf{T}\in\pf{V}_{\bullet}(\mc{M})$ and 
$\alpha\in\ms{P}_{\mc{M}}^{\mf{T}}$
since Rmk. \ref{12291844} and Lemma \ref{12201040},
so $\mc{V}_{\bullet}$ is well-defined. 
\end{remark}
\begin{lemma}
\label{10311019}
Let $\mc{B}$, $\mc{C}$ and $\mc{D}$ be three $\ast-$algebras,
where $\mc{D}$ is unital with unit $\un$.
$U\in\mc{U}(\mc{D})$, $\mf{R}\in Mor_{\ms{CA}^{\ast}}(\mc{C},\mc{D})$
and $\upeta\in Mor_{\ms{CA}^{\ast}}(\mc{B},\mc{C})$.
Then 
$(\mf{R}\circ\upeta)^{\tilde{}}
=
\tilde{\mf{R}}\circ\upeta^{+}$,
and
$(ad(U)\circ\mf{R})^{\tilde{}}
=
ad(U)\circ\tilde{\mf{R}}$.
\end{lemma}
\begin{proof}
Let $b\in\mc{B}$, $c\in\mc{C}$ and $\lambda\in\C$, then
\begin{equation*}
\begin{aligned}
(\mf{R}\circ\upeta)^{\tilde{}}(b,\lambda)
&
=
(\mf{R}\circ\upeta)(b)+\lambda\un
\\
&
=
\tilde{\mf{R}}(\upeta(b),\lambda)
=
(\tilde{\mf{R}}\circ\upeta^{+})(b,\lambda),
\end{aligned}
\end{equation*}
and
\begin{equation*}
\begin{aligned}
(ad(U)\circ\mf{R})^{\tilde{}}(c,\lambda)
&
=
(ad(U)\circ\mf{R})(c)+\lambda\un
\\
&
=
ad(U)(\mf{R}(c)+\lambda\un)
=
(ad(U)\circ\tilde{\mf{R}})(c,\lambda).
\end{aligned}
\end{equation*}
\end{proof}
\begin{lemma}
\label{12051430}
Let $\mc{N}\in\mf{G}(G,F,\uprho)$,
$\mf{T}\in\pf{T}_{\mc{N}}$ 
and
$\alpha\in\ms{P}_{\mc{N}}^{\mf{T}}$
such that 
there exists an element 
$b\in(\mc{B}_{\alpha}^{\mf{T}}(\mc{N}))^{+}$
for which 
$\|b\|\leq 1$
and
$\tilde{\mf{R}}_{\alpha}^{\mf{T}}(\mc{N})(b)=
(\Upgamma^{\mc{N}})_{\alpha}^{\mf{T}}$.
If $l\in H$ set
$b^{l}=((\mf{w}^{\mc{N}})_{\alpha}^{\mf{T}})^{+}(l)(b)$
then
$b^{l}\in(\mc{B}_{\alpha}^{\mf{T}^{l}}(\mc{N}))^{+}$
such that 
$\|b^{l}\|\leq 1$
and
$\tilde{\mf{R}}_{\alpha}^{\mf{T}^{l}}(\mc{N})(b^{l})=
\Upgamma_{\alpha}^{\mf{T}^{l}}$.
In particular
if $\pf{V}_{\bullet}(\mc{N})\neq\varnothing$ 
then 
$\mf{b}(l)(\pf{V}_{\bullet}(\mc{N}))\subseteq
\pf{V}_{\bullet}(\mc{N})$.
\end{lemma}
\begin{proof}
The inequality 
follows since
$((\mf{w}^{\mc{N}})_{\alpha}^{\mf{T}})^{+}(l)$
is an isometry being $(\mf{w}^{\mc{N}})_{\alpha}^{\mf{T}}(l)$ so,
while the equality follows since (\ref{01031145},\ref{01031145d})
and Lemma \ref{10311019}.
\end{proof}
The following important result ensures the existence of an equivariant stability.
Here the requests (\ref{01031206}, \ref{01031145d}, \ref{01031145}, \ref{01031145c}),
in the definition of the category $\mf{G}(G,F,\uprho)$,
find their justification since are used in its proof directly and via Lemma \ref{12051430}.
\begin{theorem}
[Existence of a full integer equivariant stability
on $\pf{V}_{\bullet}$]
\label{01151104}
$\lr{\pf{T}_{\bullet},\ov{\ms{m}}_{\bullet}}{\mc{V}_{\bullet}}$
is a full integer equivariant stability on $\pf{V}_{\bullet}$.
Moreover 
for all $\mc{N}\in Dom(\pf{V}_{\bullet})$,
$\mf{T}\in\pf{V}_{\bullet}(\mc{N})$ and 
$\alpha\in\ms{P}_{\mc{N}}^{\mf{T}}$
\begin{equation}
\label{14051532}
\mc{V}_{\bullet}(\mc{N})(\mf{T},\alpha)
=
\omega_{\exp\left(
-((\ms{D}^{\mc{N}})_{\alpha}^{\mf{T}})^{2}\right)}
\circ
(\uppi^{\mc{N}})_{\alpha}^{\mf{T}}.
\end{equation}
\end{theorem}
\begin{proof}
In this proof
let $\mc{G}$ be an arbitrary but fixed object of $\mf{G}(G,F,\uprho)$
and let us use \eqref{08040736} and convention \ref{05061835} 
by removing the index $\mc{G}$.
\eqref{14051532} follows since \eqref{12291851}. 
$\lr{\pf{T}_{\bullet}}{\ov{\ms{m}_{\bullet}}}$ is a full equivariant phase according 
to \eqref{12051444} Def. \ref{11251056},
while $\mc{V}_{\bullet}$ 
satisfies \eqref{01141641a} since construction,
$\pf{V}_{\bullet}$ satisfies \eqref{18061148} since Lemma \ref{12051430}.
Let us show that $\mc{V}_{\bullet}$ satisfies \eqref{01141641b}.
Let $\mf{T}\in\pf{V}_{\bullet}(\mc{G})$, $\alpha\in\ms{P}^{\mf{T}}$, $l\in H$, $h\in\{\un,l\}$ and set temporarily 
\begin{equation*}
\begin{aligned}
\rho_{h}
&=
\exp(-(\ms{D}_{\alpha}^{\mf{T}^{h}})^{2})
& 
\qquad
\upeta_{l^{-1}}
&
=
\upeta_{\alpha}^{\mf{T}}(l^{-1})
\\
\Uppsi_{h}
&=
\Uppsi_{(\ms{e}\circ\mf{t})(\mf{T}^{h},\alpha)}
&
\qquad
\mf{z}_{l^{-1}}
&
=
\mf{z}_{\alpha}^{\mf{T}^{l}}(l^{-1})
\\
\mf{R}_{h}
&=
\mf{R}_{\alpha}^{\mf{T}^{h}}(\mc{G})
&
\qquad
\mf{w}_{l^{-1}}
&
=
\mf{w}_{\alpha}^{\mf{T}^{l}}(l^{-1})
\\
\mf{j}_{h}
&=
\mf{j}_{\alpha}^{\mf{T}^{h}}
&
\qquad
\mf{v}_{l}
&=
\mf{v}_{\alpha}^{\mf{T}}(l)
\\
\uppi
&=
\uppi_{\alpha}^{\mf{T}}
&
\qquad
\mf{i}
&=
\mf{i}_{\alpha}^{\mf{T}}
\\
\mf{H}_{h}
&=
\mf{H}_{\alpha}^{\mf{T}^{h}}.
\end{aligned}
\end{equation*}
We remove the index $h$ whenever it equals $\un$.
Since \eqref{01141843} and Lemma \ref{12171928}\eqref{12171928st3}
we deduce 
\begin{equation}
\label{01231813}
\Uppsi_{l}^{-}=\omega_{\rho_{l}}\circ\mf{R}_{l}^{-}.
\end{equation}
Next 
$\mf{R}_{l}=
\ms{ad}(\mf{v}_{l})\circ\mf{R}\circ\mf{w}_{l^{-1}}$
by \eqref{09051254} and \eqref{01031145},
hence
\begin{equation}
\label{01231906}
\begin{aligned}
\mf{R}_{l}^{-}
&=
(\ms{ad}(\mf{v}_{l})\circ\mf{R})^{-}
\circ
(\mf{i}\circ\mf{w}_{l^{-1}})^{-}
\\
&=
\ms{ad}(\mf{v}_{l})\circ\mf{R}^{-}
\circ
(\mf{i}\circ\mf{w}_{l^{-1}})^{-},
\end{aligned}
\end{equation}
where the first equality follows since $\mf{w}_{l^{-1}}$
is surjective, $\mf{R}_{l}$ is nondegenerate and 
Cor. \ref{01101930}, while the second one follows since $\mf{R}$
and $\ms{ad}(\mf{v}_{l})\circ\mf{R}$ are nondegenerate and
Rmk. \ref{01101755}.
Next 
$\mf{R}_{l}^{-}\circ\mf{j}_{l}
=
\ms{ad}(\mf{v}_{l})\circ\mf{R}^{-}\circ\mf{j}\circ\upeta_{l^{-1}}
\circ\mf{z}_{l^{-1}}$
since \eqref{01231919} and \eqref{01231906},
thus 
$\mf{R}_{l}^{-}\circ\mf{j}_{l}
=
\ms{ad}(\mf{v}_{l})
\circ\uppi
\circ\upeta_{l^{-1}}
\circ\mf{z}_{l^{-1}}$
since \eqref{12191452}, 
which together \eqref{01231813} yields
\begin{equation}
\label{01231910}
\Uppsi_{l}^{-}\circ\mf{j}_{l}
=
\omega_{\rho_{l}}\circ\ms{ad}(\mf{v}_{l})
\circ\uppi\circ\upeta_{l^{-1}}
\circ\mf{z}_{l^{-1}}.
\end{equation}
Next $\rho_{l}=\ms{ad}(\mf{v}_{l})(\rho)$ since
\eqref{01031145d} and \eqref{11251436}, thus for all 
$a\in\mc{L}(\mf{H})$
\begin{equation*}
Tr_{\mf{H}_{l}}(\rho_{l}\ms{ad}(\mf{v}_{l})(a))
=
Tr_{\mf{H}_{l}}(\ms{ad}(\mf{v}_{l})(\rho a))
=
Tr_{\mf{H}}(\rho a),
\end{equation*}
in particular 
$Tr_{\mf{H}_{l}}(\rho_{l})
=
Tr_{\mf{H}}(\rho)$
so
$\omega_{\rho_{l}}\circ\ms{ad}(\mf{v}_{l})=\omega_{\rho}$,
therefore since \eqref{01231910} and \eqref{09051254}
\begin{equation*}
\begin{aligned}
\Uppsi_{l}^{-}\circ\mf{j}_{l}
&=
\omega_{\rho}\circ\uppi\circ\upeta_{l^{-1}}\circ\mf{z}_{l^{-1}}
\\
&=
\omega_{\rho}\circ\uppi\circ
(\mf{z}_{\alpha}^{\mf{T}}(l)
\circ
\upeta_{\alpha}^{\mf{T}}(l))^{-1},
\end{aligned}
\end{equation*}
and \eqref{01141641b} follows by \eqref{09051254b} and \eqref{12291851}.
\end{proof}
\section{The canonical nucleon-fragment doublet $\mc{T}_{\bullet}$}
\label{func}
In section \ref{func1} we construct the object part of a functor $\mf{G}_{\vartriangle}^{H}$
from $\ms{C}_{u}(H)$ to $\mf{G}(G,F,\uprho)$.
In section \ref{func2} we complete the construction of $\mf{G}_{\vartriangle}^{H}$,
and employ this result to establish in our main Thm. \ref{06242121}
the existence of a 
extended full integer $\ms{C}_{u}(H)-$equivariant stability on $\pf{V}_{\bullet}$,
and consequently of a nucleon-fragment doublet on $\ms{C}_{u}(H)$.
Finally as a consequence of the main theorem we exhibit
the resolution of the equivariant form of the universality claim.
In the present section \ref{func} we assume fixed two locally compact topological groups $G$ and $F$,
a group homomorphism $\uprho:F\to Aut_{\ms{Gr}}(G)$  such that the map $(g,f)\mapsto\uprho_{f}(g)$
on $G\times F$ at values in $G$, is continuous. Let $H$ denote $G\rtimes_{\uprho}F$.
\subsection{The object part $\mf{G}^{H}$}
\label{func1}
The main result in this section is 
Cor. \ref{11271221}
were we construct the object part of a functor
from $\ms{C}_{u}(H)$ to $\mf{G}(G,F,\uprho)$.
Auxiliary important results in this directions are 
Thm. \ref{09191747}  
and \eqref{10121958}, 
Thm. \ref{11260828}, Thm. \ref{main} and Thm. \ref{01141808}.
In the present section \ref{func1}
we fix a dynamical system $\mf{A}=\lr{\mc{A}}{H,\upsigma}$, 
while starting from Def. \ref{10291125} 
and except Thm. \ref{01141808},
$\mf{A}$ is inner and we fix a group morphism $\ms{v}:H\to\mc{U}(\mc{A})$ such that $\mf{A}$ is implemented by $\ms{v}$.
We let $\K^{X}$ denote $Mor_{\ms{set}}(X,\K)$ for any set $X$ and $\K\in\{\R,\C\}$.
By taking into account \eqref{fc}, Def. \ref{01262145} 
and notations after \eqref{prod} 
we can give the following
\begin{definition}
\label{08192224}
Let $\upomega\in\ms{E}_{\mc{A}}^{G}(\uptau)$.
We say that
$\lr{\pf{H}}{\mu,\upzeta,f}$
is a 
$\lr{\mf{A}}{\upomega}-$selfadjoint system
if
\begin{enumerate}
\item
$\pf{H}
\coloneqq
\lr{\mf{H},\uppi}{\Upomega}$ 
is a cyclic representation 
of $\mc{A}$ associated with $\upomega$;
\item
$\upmu\in\mc{H}(\ms{S}_{\ms{F}_{\upomega}}^{G})$;
\item
$\upzeta:\R^{X}\to\ms{S}_{\ms{F}_{\upomega}}^{G}$ 
is a continuous group morphism, where
$X$ is a nonempty set;
\item
$f$ is a $\mc{C}_{\sigma}(\C^{X})-$measurable map
such that there exists an
$A\in\mathscr{P}_{\omega}(X)$
satisfying 
\begin{equation*}
\ov{f(\mc{C}(A,\prod_{x\in A}supp(\ms{E}_{T_{x}})))}\subseteq\R.
\end{equation*}
\end{enumerate}
Here
$\ms{T}\doteq\{T_{x}\}_{x\in X}$ 
is such that 
$iT_{x}$ is the infinitesimal generator of the
strongly continuous one-parameter semigroup 
of unitarities 
$\ms{U}_{\pf{H}}\circ\upzeta\circ\mf{i}_{x}\up\R^{+}$
on $\mf{H}$,
where
$\mf{i}_{x}:\R\to\R^{X}$ is such that
$\Pr_{y}\circ\mf{i}_{x}=\delta_{x,y} Id_{\R}$, for all $x,y\in X$.
Set 
\begin{equation}
\label{01281854}
\ms{D}_{\pf{H}}^{\upzeta,f}(\mf{A})
\coloneqq
f(\ms{E}_{\mc{E}_{\ms{T}}}).
\end{equation}
We convein to remove $(\mf{A})$, whenever it is clear by the context which dynamical system is involved.
\end{definition}
\begin{remark}
\label{02010723}
Since $\R^{X}$ is an abelian group, 
we deduce by Prp. \ref{11270642}
that $\mc{E}_{\ms{T}}$
is a family of commuting Borel RI's in $\mf{H}$.
Then we can apply 
Thm. \ref{11241732}\eqref{32st3}
to $\mc{E}_{\ms{T}}$ and state that 
$\ms{D}_{\pf{H}}^{\upzeta,f}(\mf{A})$
is a well-defined selfadjoint operator in $\mf{H}$.
\end{remark}
\begin{definition}
[Pre thermal phases]
\label{30081041}
We say that 
$\mc{T}=\lr{h,\upxi,\beta_{c},I}{\ps{\upomega}}$ 
is a pre thermal phase associated with $\mf{A}$,
if 
$h\in H$,
$\upxi:\R\to G$ is a continuous group morphism,
$\beta_{c}\in\widetilde{\R}$, 
$I$ is a
neighbourhood of $\beta_{c}$,
and
\begin{equation}
\label{08291407}
\ps{\upomega}\in\prod_{\beta\in I}
\ms{E}_{\mc{A}}^{G}(\uptau)
\cap
\ms{K}_{\beta}^{\uptau^{(h,\upxi)}},
\end{equation}
such that for all $\alpha\in I$
\begin{equation}
\label{08301407}
\begin{cases}
\alpha\leq\beta_{c}
\Rightarrow
\ms{F}_{\ps{\upomega}_{\alpha}}
\supseteq
\ms{F}_{\ps{\upomega}_{\beta_{c}}},
\\
\beta_{c}<\alpha
\Rightarrow
\ms{F}_{\ps{\upomega}_{\beta_{c}}}
\cap
\complement
\ms{F}_{\ps{\upomega}_{\alpha}}
\neq\varnothing.
\end{cases}
\end{equation}
\end{definition}
\begin{definition}
\label{08311553}
Let
\begin{itemize}
\item
$\mc{T}=\lr{h,\upxi,\beta_{c},I}{\ps{\upomega}}$ be a 
pre thermal phase associated with $\mf{A}$;
\item
$\ps{\upmu}$ be a Haar system associated with $\ps{\upomega}$
and $\mf{A}$;
\item
$\pf{H}:I\to Rep_{c}(\mc{A})$ 
be such that 
$\pf{H}_{\beta}=
\lr{\mf{H}_{\beta},\uppi_{\beta}}{\Upomega_{\beta}}$ 
is a cyclic representation of $\mc{A}$
associated with $\ps{\upomega}_{\beta}$, for all $\beta\in I$;
\item
$\upzeta:\R^{X}\to\ms{S}_{\ms{F}_{\ps{\upomega}_{\beta_{c}}}}^{G}$ 
is a continuous group morphism, where
$X$ is a nonempty set;
\item
$\Upgamma\in
\prod_{\alpha\in I\cap]-\infty,\beta_{c}]}
\mc{L}(\mf{H}_{\alpha})$;
\item
$l\in H$,
$\beta\in I$
and
$\alpha\in I\cap ]-\infty,\beta_{c}]$.
\end{itemize}
Define
\begin{itemize}
\item
$\upzeta^{l}\coloneqq\ms{ad}(l)\circ\upzeta$
\item
$
\ps{\upmu}^{l}:I\times H
\ni
(\gamma,h)
\mapsto
\ps{\upmu}_{(\gamma,h)}^{l}
\coloneqq
\ps{\upmu}_{(\gamma,h\cdot_{\uprho} l)}$.
\end{itemize}
If there exist  
$A\in\mathscr{P}_{\omega}(X)$
and a
$\mc{C}_{\sigma}(\C^{X})-$measurable map $f$ 
satisfying 
\begin{equation}
\label{01291801}
\ov{f(\mc{C}(A,\prod_{x\in A}
supp(\ms{E}_{T_{x}^{\alpha}})))}
\subseteq\R,
\end{equation}
where
$iT_{y}^{\alpha}$ is the infinitesimal generator of 
the semigroup 
$\ms{U}_{\pf{H}_{\alpha}}\circ\upzeta\circ\mf{i}_{y}\up\R^{+}$
on $\mf{H}_{\alpha}$, for all $y\in X$,
we can set
\begin{itemize}
\item 
$\ms{D}_{\pf{H},\alpha}^{\upzeta,f}(\mf{A})
\coloneqq
\ms{D}_{\pf{H}_{\alpha}}
^{\upzeta,f}(\mf{A})$,
\item
$\ps{R}_{\pf{H},\Upgamma,\alpha}^{\ps{\upmu},\upzeta,f}(\mf{A})
\coloneqq
(\tilde{\pf{R}}_{\pf{H},\alpha}^{\ps{\upmu}}(\mf{A}),
\ms{D}_{\pf{H},\alpha}^{\upzeta,f}(\mf{A}),
\Upgamma_{\alpha})$.
\end{itemize}
If $\mf{A}$ is inner implemented by $\ms{v}$ we can define
\begin{equation*}
\Upgamma_{\pf{H},\ms{v}}^{l}:
I\cap]-\infty,\beta_{c}]
\ni\delta\mapsto
\Upgamma_{\pf{H},\ms{v},\delta}^{l}
\coloneqq
ad(\uppi_{\delta}(\ms{v}(l)))(\Upgamma_{\delta}),
\end{equation*}
and if in addition 
\eqref{01291801} holds 
we can set
\begin{enumerate}
\item 
$\ms{D}_{\pf{H},\ms{v},\alpha,l}^{\upzeta,f}(\mf{A})
\coloneqq
\ms{D}_{\pf{H}_{\alpha}^{(\ms{v},l)}}
^{\upzeta^{l},f}(\mf{A})$,
\label{10091152d}
\item
$\ps{R}_{\pf{H},\ms{v},\Upgamma,\alpha,l}^{\ps{\upmu},\upzeta,f}
(\mf{A})
\coloneqq
(\tilde{\pf{R}}_{\pf{H},\ms{v},\alpha,l}^{\ps{\upmu}}(\mf{A}),
\ms{D}_{\pf{H},\ms{v},\alpha,l}^{\upzeta,f}(\mf{A}),
\Upgamma_{\pf{H},\ms{v},\alpha}^{l})$.
\label{10091152e}
\end{enumerate}
We convein to remove $\mf{A}$ whenever it is clear 
the dynamical system involved,
in addition
to remove both the indices $\ms{v}$ and $l$ 
whenever $l$ equals the unit.
\end{definition}
The next result shows that Def. \ref{08311553} is well-set and the equivariance of the 
operator $\ms{D}_{\pf{H},\ms{v},\alpha,l}^{\upzeta,f}$ under action of $H$,
a step toward the construction in Cor. \ref{11271221}
of the object part of a functor
from $\ms{C}_{u}(H)$ to $\mf{G}(G,F,\uprho)$.
Here we use the covariance of the functional
calculus relative to a commuting set of resolutions of the identity
stated in 
Thm. \ref{11252000}. 
\begin{proposition}
\label{10111026}
Def. \ref{08311553}\eqref{10091152d}\,\&\,\eqref{10091152e} are well-defined and we have
\begin{equation}
\label{10121958} 
\ms{D}_{\pf{H},\ms{v},\alpha,l}^{\upzeta,f}
=
\uppi_{\alpha}(\ms{v}(l))\,
\ms{D}_{\pf{H},\alpha}^{\upzeta,f}\,
\uppi_{\alpha}(\ms{v}(l^{-1})).
\end{equation}
\end{proposition}
\begin{proof}
Since \eqref{08301407}, 
Lemma \ref{09102032}\eqref{09101945pre} 
and the fact that any bijective map on a set $Y$ induces an isomorphism of 
the order by inclusion on the power set of $Y$, we have
$\ms{S}_{\ms{F}_{\ps{\upomega}_{\beta_{c}}^{l}}}^{G}\subseteq\ms{S}_{\ms{F}_{\ps{\upomega}_{\alpha}^{l}}}^{G}$,
so 
\begin{equation*}
\upzeta^{l}:\R^{X}\to
\ms{S}_{\ms{F}_{\ps{\upomega}_{\alpha}^{l}}}^{G}.
\end{equation*}
Therefore
$\ms{U}_{\pf{H},\ms{v},\alpha,l}\circ\upzeta^{l}$ is well-set
and
Cor. \ref{09131225}
yields
\begin{equation}
\label{01291726}
\ms{U}_{\pf{H},\ms{v},\alpha,l}\circ\upzeta^{l}
=
\ms{ad}(\uppi_{\alpha}(\ms{v}(l)))
\circ
\ms{U}_{\pf{H},\alpha}\circ\upzeta.
\end{equation}
Hence letting $T_{x}^{\alpha,l}$ be such that $iT_{x}^{\alpha,l}$ is the infinitesimal generator of the semigroup 
$\ms{U}_{\pf{H},\ms{v},\alpha,l}\circ\upzeta^{l}\circ\mf{i}_{x}\up\R^{+}$
we obtain for all $x\in X$
\begin{equation}
\label{10121958pre} 
T_{x}^{\alpha,l}
=
\uppi_{\alpha}(\ms{v}(l))\,
T_{x}^{\alpha}\,
\uppi_{\alpha}(\ms{v}(l^{-1})),
\end{equation}
thus by 
Cor. \ref{10301830}\eqref{10301830st2}, 
\eqref{18225} 
and \eqref{01291801}, 
\begin{equation}
\label{01310756}
\ov{f(\mc{C}(A,\prod_{x\in A}
supp(\ms{E}_{T_{x}^{\alpha,l}})))}
\subseteq\R.
\end{equation}
Therefore
according to Def. \ref{08192224},
the objects in Def. \ref{08311553}\eqref{10091152d}\,\&\,\eqref{10091152e} are well-defined.
Finally \eqref{10121958} follows since \eqref{10121958pre} and 
Thm. \ref{11252000}\eqref{11252000st2}.
\end{proof}
\begin{corollary}
\label{09101654}
Let
\begin{enumerate}
\item
$\mc{T}=\lr{h,\upxi,\beta_{c},I}{\ps{\upomega}}$ be a 
pre thermal phase associated with $\mf{A}$;
\item
$\ps{\upmu}$ be a Haar system associated with $\ps{\upomega}$
and $\mf{A}$;
\item
$\pf{H}:I\to Rep_{c}(\mc{A})$ be such that 
$\pf{H}_{\beta}=
\lr{\mf{H}_{\beta},\uppi_{\beta}}{\Upomega_{\beta}}$ 
is a cyclic representation of $\mc{A}$
associated with $\ps{\upomega}_{\beta}$, for all $\beta\in I$;
\item
$\upzeta:\R^{X}\to\ms{S}_{\ms{F}_{\ps{\upomega}_{\beta_{c}}}}^{G}$ 
is a continuous group morphism, where
$X$ is a nonempty set;
\item
$\Upgamma\in
\prod_{\alpha\in I\cap]-\infty,\beta_{c}]}
\mc{L}(\mf{H}_{\alpha})$;
\item
$l,u\in H$,
$\beta\in I$
and
$\alpha\in]-\infty,\beta_{c}]\cap I$.
\end{enumerate}
Then
\begin{enumerate}
\item
$\ms{F}_{\upsigma^{\ast}(l)(\ps{\upomega}_{\beta})}
=
\ms{ad}(\Pr_{2}(l))(\ms{F}_{\ps{\upomega}_{\beta}})$;
\label{09101724}
\item
$\ms{S}_{\ms{F}_{\upsigma^{\ast}(l)(\ps{\upomega}_{\beta})}}^{G}
=
\ms{ad}(l)(\ms{S}_{\ms{F}_{\ps{\upomega}_{\beta}}}^{G})$;
\label{09101945}
\item
$\upzeta^{l}:\R^{X}\to
\ms{S}_{\ms{F}_{\ps{\upomega}_{\beta_{c}}^{l}}}^{G}$
is a continuous group morphism;
\label{10131910}
\item
$\ps{\upmu}^{l}$ is a Haar system associated with 
$\ps{\upomega}^{l}$;
\label{10291000}
\item
$\lr{l\cdot_{\uprho}h,\upxi,\beta_{c},I}
{\ps{\upomega}^{l}}$ 
is a pre thermal phase asociated to $\mf{A}$;
\label{09101941}
\item
if
$\mf{A}$ is inner implemented by $\ms{v}$
then
$\pf{R}_{\pf{H}^{(\ms{v},l)},\ms{v},\alpha,u}
^{\ps{\upmu}^{l}}
=
\pf{R}_{\pf{H},\ms{v},\alpha,u\cdot_{\uprho}l}^{\ps{\upmu}}$
moreover if 
there exists an
$A\in\mathscr{P}_{\omega}(X)$
and a $\mc{C}_{\sigma}(\C^{X})-$measurable map $f$
satisfying \eqref{01291801},
then
\begin{equation}
\label{10131207p}
\ms{D}_{\pf{H}^{(\ms{v},l)},\ms{v},\alpha,u}^{\upzeta^{l},f} 
=
\ms{D}_{\pf{H},\ms{v},\alpha,(u\cdot_{\uprho}l)}^{\upzeta,f}
\end{equation}
in particular
\begin{equation}
\label{10131207}
\ps{R}_{\pf{H}^{(\ms{v},l)},\Upgamma_{\pf{H},\ms{v}}^{l},\alpha}
^{\ps{\upmu}^{l},\upzeta^{l},f}
=
\ps{R}_{\pf{H},\ms{v},\Upgamma,\alpha,l}^{\ps{\upmu},\upzeta,f}.
\end{equation}
\label{10131207st}
\end{enumerate}
\end{corollary}
\begin{proof}
St. \eqref{09101724}\,\&\,\eqref{09101945} follow by  $\ps{\upomega}_{\beta}\in\ms{E}_{\mc{A}}^{G}(\uptau)$, 
and by Lemma \ref{09102032}.
St.\eqref{10131910} follows by st.\eqref{09101945}, while st.\eqref{10291000} since 
Lemma \ref{09171541}\eqref{10041833c}.
St.\eqref{09101941} follows since 
Cor. \ref{09081449},
st.\eqref{09101724} and since any bijective map on a set induces an isomorphism 
of the order by inclusion on its power set.
$\pf{R}_{\pf{H}^{(\ms{v},l)},\ms{v},\alpha,u}^{\ps{\upmu}^{l}}$ is well-set since st.\eqref{10291000},
$\ms{D}_{\pf{H}^{(\ms{v},l)},\ms{v},\alpha,u}^{\upzeta^{l},f}$
is well-set since st. (\ref{10131910},\ref{10291000},\ref{09101941}), 
Rmk. \ref{09151131} 
and \eqref{01310756},
thus \eqref{10131207p} and the first equality in st.\eqref{10131207st} follow since
\begin{equation}
\label{05201352}
(\upzeta^{l})^{u}=\upzeta^{u\cdot_{\uprho}l}
\quad
(\ps{\upmu}^{l})^{u}=\ps{\upmu}^{u\cdot_{\uprho}l}
\quad
(\pf{H}^{(\ms{v},l)})^{(\ms{v},u)}
=
\pf{H}^{(\ms{v},u\cdot_{\uprho}l)}.
\end{equation}
\end{proof}
\begin{definition}
\label{10291045}
If 
$\mc{T}=\lr{h,\upxi,\beta_{c},I}{\ps{\upomega}}$ 
is a pre thermal phase associated with $\mf{A}$
and $l\in H$,
then we define
$\mc{T}^{l}\coloneqq
\lr{l\cdot_{\uprho}h,\upxi,\beta_{c},I}
{\ps{\upomega}^{l}}$ 
which is a pre thermal phase associated with $\mf{A}$ 
according to Cor. \ref{09101654}.
\end{definition}
\begin{definition}
\label{08261134}
Let
$\pf{T}_{\mf{A}}$ 
be the set of the
$\lr{\mc{T},\ps{\upmu},\pf{H}}{\upzeta,f,\Upgamma}$
such that
\begin{enumerate}
\item
$\mc{T}=\lr{h,\upxi,\beta_{c},I}{\ps{\upomega}}$ 
is a 
pre thermal phase associated with $\mf{A}$;
\label{08261134I}
\item
$\ps{\upmu}$ is a Haar system associated with $\ps{\upomega}$
and $\mf{A}$;
\label{08261134III}
\item
$\pf{H}:I\to Rep_{c}(\mc{A})$ is such that 
$\pf{H}_{\beta}=
\lr{\mf{H}_{\beta},\uppi_{\beta}}{\Upomega_{\beta}}$ 
is a cyclic representation of $\mc{A}$
associated with $\ps{\upomega}_{\beta}$, for all $\beta\in I$;
\label{08261134IV}
\item
$\upzeta:
\R^{X}\to\ms{S}_{\ms{F}_{\ps{\upomega}_{\beta_{c}}}}^{G}$ 
is a continuous group morphism, where
$X$ is a nonempty set;
\label{08261134V}
\item
$\Upgamma\in
\prod_{\alpha\in I\cap]-\infty,\beta_{c}]}
\mc{L}(\mf{H}_{\alpha})$;
\label{08261134VII}
\item
$f$ is a $\mc{C}_{\sigma}(\C^{X})-$measurable map
such that there exists an
$A\in\mathscr{P}_{\omega}(X)$ satisfying
\eqref{01291801}
for all $\alpha\in I\cap]-\infty,\beta_{c}]$, 
\label{08261134VI}
\item
$\ps{R}(\mf{T},\alpha)
\coloneqq
\ps{R}_{\pf{H},\Upgamma,\alpha}^{\ps{\upmu},\upzeta,f}
$ 
is an even $\theta-$summable $K-$cycle
for all $\alpha\in I\cap]-\infty,\beta_{c}]$. 
\label{08261134VIII}
\end{enumerate}
\end{definition}
\begin{remark}
\label{02041855}
We deduce since 
\cite[$IV.2.\gamma$  Def. $11$ and $IV.8$ Def. $1$]{connes}
that Def. \ref{08261134}\eqref{08261134VIII}
is equivalent to the following two requests,
\begin{enumerate}
\item
\eqref{10311100} for $h=Id$;
\item
for all $\alpha\in I\cap]-\infty,\beta_{c}]$,
$\Upgamma_{\alpha}$
is a unitary, selfadjoint operator on $\mf{H}_{\alpha}$
(a $\Z_{2}-$grading on $\mf{H}_{\alpha}$),
such 
that for all
$a\in\ms{B}_{\ps{\upmu}}^{\ps{\upomega},\alpha,+}$
\begin{enumerate}
\item
$
[\Upgamma_{\alpha},
\tilde{\mf{R}}_{\pf{H},\alpha}^{\ps{\upmu}}(a)
]=\ze$,
\item
$
\Upgamma_{\alpha}
\ms{D}_{\pf{H},\alpha}^{\upzeta,f}
\Upgamma_{\alpha}
=
-
\ms{D}_{\pf{H},\alpha}^{\upzeta,f}$.
\end{enumerate}
\end{enumerate}
\end{remark}
\begin{definition}
\label{05121538}
Let 
$\mf{T}=\lr{\mc{T},\ps{\upmu},\pf{H}}{\upzeta,f,\Upgamma}
\in\pf{T}_{\mf{A}}$,
where
$\mc{T}=\lr{h,\upxi,\beta_{c},I}{\ps{\upomega}}$, 
define
$\ms{P}_{\mf{A}}^{\mf{T}}\coloneqq I\cap]-\infty,\beta_{c}]$
and
for all $\alpha\in\ms{P}_{\mf{A}}^{\mf{T}}$
\begin{equation*}
\ms{K}_{\alpha}^{\mf{T}}(\mf{A})
\coloneqq
\ms{K}_{0}(\ms{B}_{\ps{\upmu}}^{\ps{\upomega},\alpha,+}).
\end{equation*}
Moreover set
\begin{equation*}
\mc{K}^{\mf{A}}
\coloneqq
\bigcup_{\mf{Q}\in\pf{T}_{\mf{A}}}
\bigcup_{\alpha\in\ms{P}_{\mf{A}}^{\mf{Q}}}
\ms{K}_{\alpha}^{\mf{Q}}(\mf{A}),
\end{equation*}
and
\begin{equation*}
\ov{\mc{K}}^{\mf{A}}
\coloneqq
\bigcup_{\mf{Q}\in\pf{T}_{\mf{A}}}
\prod_{\alpha\in\ms{P}_{\mf{A}}^{\mf{Q}}}
\ms{K}_{\alpha}^{\mf{Q}}(\mf{A}).
\end{equation*}
Finally set $\mf{c}^{\mf{A}}(l)$ as the map defined on 
$\mc{K}^{\mf{A}}$ 
such that
for any 
$\mf{T}\in\pf{T}_{\mf{A}}$
and
$\alpha\in\ms{P}_{\mf{A}}^{\mf{T}}$
\begin{equation*}
\mf{c}^{\mf{A}}(l)
\up
\ms{K}_{\alpha}^{\mf{T}}(\mf{A})
\coloneqq
\bigr((\ps{\upsigma}^{(\ps{\upomega}_{\alpha},l)})^{+}\bigl)_{\ast}.
\end{equation*}
\end{definition}
\begin{convention}
If 
$\ps{\upomega}:A\to\ms{E}_{\mc{A}}^{G}(\uptau)$ 
with $A$ a nonempty set
and
$\ps{\upmu}$ is a Haar system associated with $\ps{\upomega}$ and $\mf{A}$,
then for all $l\in H$ and $\alpha\in A$
we convein to denote 
the pairing 
$\lr{\cdot}{\cdot}_{\ms{B}_{\ps{\upmu}}^{\ps{\upomega},\alpha,l,+}}$
by
$\lr{\cdot}{\cdot}_{\ps{\upmu},\ps{\upomega},\alpha,l}$.
Moreover we remove the index $l$ if it equals the identity.
\end{convention}
\begin{definition}
\label{05131816}
Let $\ms{A}_{\mf{A}}$ be the group whose underlying
set is 
\begin{equation*}
\ms{A}_{\mf{A}}
\coloneqq
\prod_{\mf{Q}\in\pf{T}_{\mf{A}}}
\prod_{\beta\in\ms{P}_{\mf{A}}^{\mf{Q}}}
\ms{K}_{\beta}^{\mf{Q}}(\mf{A}),
\end{equation*}
and whose composition, inversion and identity $\ze$
are the pointwise composition, inversion and identity,
namely
$\ms{f}\cdot\ms{g}\in\ms{A}_{\mf{A}}$ such that
$(\ms{f}\cdot\ms{g})
(\mf{T})(\alpha)
\coloneqq
\ms{f}(\mf{T})(\alpha)\cdot\ms{g}(\mf{T})(\alpha)$,
$\ms{f}^{-1}(\mf{T})(\alpha)
\coloneqq\ms{f}(\mf{T})(\alpha)^{-1}$
and
$\ze(\mf{T})(\alpha)$ is the identity of 
$\ms{K}_{\alpha}^{\mf{T}}(\mf{A})$,
for all 
$\ms{f},\ms{g}\in\ms{A}_{\mf{A}}$,
$\mf{T}\in\pf{T}_{\mf{A}}$ and $\alpha\in\ms{P}_{\mf{A}}^{\mf{T}}$.
Moreover set
\begin{equation*}
\ms{m}^{\mf{A}}:
\ms{A}_{\mf{A}}\to
\prod_{\mf{Q}\in\pf{T}_{\mf{A}}}
Mor_{\ms{set}}(\ms{P}_{\mf{A}}^{\mf{Q}},\R),
\end{equation*}
such that for any 
$\ms{f}\in\ms{A}_{\mf{A}}$
and 
$\mf{T}\in\pf{T}_{\mf{A}}$
\begin{equation*}
\ms{m}^{\mf{A}}(\ms{f})(\mf{T}):
\ms{P}_{\mf{A}}^{\mf{T}}\to\R,
\quad
\alpha
\mapsto
\lr{\ms{f}(\mf{T})(\alpha)}
{\ms{ch}(\ps{R}(\mf{T},\alpha))}_{\ps{\upmu},\ps{\upomega},\alpha},
\end{equation*}
where
$\mf{T}=\lr{\mc{T},\ps{\upmu},\pf{H}}{\upzeta,f,\Upgamma}$
and
$\mc{T}=\lr{h,\upxi,\beta_{c},I}{\ps{\upomega}}$.
$\ms{m}^{\mf{A}}$ is called 
mean value map associated with $\mf{A}$.
\end{definition}
\begin{definition}
\label{10291125}
Let $l\in H$, define
\begin{equation*}
\mf{b}^{\mf{A},\ms{v}}(l):
\pf{T}_{\mf{A}}\ni
\lr{\mc{T},\ps{\upmu},\pf{H}}{\upzeta,f,\Upgamma}
\mapsto
\lr{\mc{T}^{l},\ps{\upmu}^{l},\pf{H}^{(\ms{v},l)}}
{\upzeta^{l},f,\Upgamma_{\pf{H},\ms{v}}^{l}}.
\end{equation*}
\end{definition}
\begin{convention}
Often in the proof of the statements and only when it is not cause of confusion
we convein to remove $\mf{A}$ from 
$\ms{A}_{\mf{A}}$, $\ms{m}^{\mf{A}}$, 
$\mf{c}^{\mf{A}}$,
$\mc{K}^{\mf{A}}$,
$\ov{\mc{K}}^{\mf{A}}$
and
$\ms{K}_{\alpha}^{\mf{T}}(\mf{A})$, 
for any $\mf{T}\in\pf{T}_{\mf{A}}$ and $\alpha\in\ms{P}_{\mf{A}}^{\mf{T}}$,
moreover we remove 
$\mf{A}$ and $\ms{v}$
from $\mf{b}^{\mf{A},\ms{v}}$.
\end{convention}
\begin{proposition}
\label{01010828}
Let $l\in H$,
thus
$\mf{c}^{\mf{A}}(l)$ is well-defined 
and
$\mf{c}^{\mf{A}}(l)(\ms{K}_{\alpha}^{\mf{T}}(\mf{A}))
\subseteq
\ms{K}_{0}(\ms{B}_{\ps{\upmu}}^{\ps{\upomega},\alpha,l,+})$,
for any 
$\mf{T}\in\pf{T}_{\mf{A}}$
and
$\alpha\in\ms{P}_{\mf{A}}^{\mf{T}}$.
\end{proposition}
\begin{proof}
Let $\mf{Q},\mf{T}\in\pf{T}_{\mf{A}}$ and 
$\alpha\in\ms{P}_{\mf{A}}^{\mf{T}}$, $\beta\in\ms{P}_{\mf{A}}^{\mf{Q}}$
such that 
$
\ms{K}_{\alpha}^{\mf{T}}
\cap
\ms{K}_{\beta}^{\mf{Q}}
\neq\varnothing$
thus
$
\mc{B}\doteq
\ms{B}_{\ps{\upmu}}^{\ps{\upomega},\alpha}
\cap
\ms{B}_{\ps{\upmu}}^{\ps{\upomega},\beta}
\neq\varnothing$.
Now 
$\ms{B}_{\ps{\upmu}}^{\ps{\upomega},\alpha}$
is the completion
of 
$\pc{C}_{c}^{\ps{\upmu}_{(\alpha,\un)}}
(\s{\upomega}{\alpha}{},\mc{A})$,
thus its underlying set is the set of all
minimal Cauchy filters of
$\pc{C}_{c}^{\ps{\upmu}_{(\alpha,\un)}}
(\s{\upomega}{\alpha}{},\mc{A})$,
see for example \cite[$II.21$]{BourGT},
therefore 
$\mc{B}\neq\varnothing$
implies
\footnote{This irrespectively by the 
fact that 
there could be 
$\delta\in\ms{P}_{\mf{A}}^{\mf{T}}$,
$\ep\in\ms{P}_{\mf{A}}^{\mf{Q}}$,
a $C^{\ast}-$algebra $\mc{D}$
and $\ast-$embeddings
$\uplambda_{\delta}:
\ms{B}_{\ps{\upmu}}^{\ps{\upomega},\delta}
\to\mc{D}$
and
$\uplambda_{\ep}:
\ms{B}_{\ps{\upmu}}^{\ps{\upomega},\ep}
\to\mc{D}$
such that 
$\uplambda_{\delta}(\mc{C}_{c}(\s{\upomega}{\delta}{},\mc{A}))
\cap
\uplambda_{\ep}(\mc{C}_{c}(\s{\upomega}{\ep}{},\mc{A}))
=
\varnothing$
although
$\uplambda_{\delta}(\ms{B}_{\ps{\upmu}}^{\ps{\upomega},\delta})
\cap
\uplambda_{\ep}(\ms{B}_{\ps{\upmu}}^{\ps{\upomega},\ep})
\neq
\varnothing$.}
$\mc{C}_{c}(\s{\upomega}{\alpha}{},\mc{A})
\cap
\mc{C}_{c}(\s{\upomega}{\beta}{},\mc{A})
\neq\varnothing$.
Hence a fortiori 
$\s{\upomega}{\alpha}{}=\s{\upomega}{\beta}{}$
so
there exists a constant $C$
such that 
$\|\cdot\|^{\ps{\upmu}_{(\alpha,\un)}}
=
C\|\cdot\|^{\ps{\upmu}_{(\beta,\un)}}$
then
$\ms{B}_{\ps{\upmu}}^{\ps{\upomega},\alpha,+}
=
\ms{B}_{\ps{\upmu}}^{\ps{\upomega},\beta,+}$
and
$\ps{\upsigma}^{\ps{\upomega}_{\alpha},l}
=
\ps{\upsigma}^{\ps{\upomega}_{\beta},l}$
since 
Cor. \ref{09191329}.
Thus
$\ms{K}_{\alpha}^{\mf{T}}
=
\ms{K}_{\beta}^{\mf{Q}}$
and
$\bigr((\ps{\upsigma}^{(\ps{\upomega}_{\alpha},l)})^{+}\bigl)_{\ast}
=
\bigr((\ps{\upsigma}^{(\ps{\upomega}_{\beta},l)})^{+}\bigl)_{\ast}$
therefore the statement follows
since 
\eqref{10291116},
\eqref{10291120},
Cor. \ref{09191329}
and the standard picture of the functor $\ms{K}_{0}$.
\end{proof}
\begin{remark}
[Integrality]
\label{02061932}
Since the general result
\cite[$IV.8.\delta$, Prp. $18$, Thm. $19$, and $IV.8.\epsilon$ Thm. $22$]{connes}
we deduce that
$\ms{m}^{\mf{A}}(\ms{f})(\mf{T})$
is a $\Z-$valued map
for any 
$\ms{f}\in\ms{A}_{\mf{A}}$
and
$\mf{T}\in\pf{T}_{\mf{A}}$.
\end{remark}
The following Thm. \ref{11260828}, Thm. \ref{main} and Thm. \ref{01141808}
are important steps towards the proof of Cor. \ref{11271221}
were we construct the object part of a functor from $\ms{C}_{u}(H)$ to $\mf{G}(G,F,\uprho)$.
\begin{theorem}
\label{11260828}
$\mf{b}^{\mf{A},\ms{v}}\in Mor_{\ms{Gr}}(H, Aut_{\ms{set}}(\pf{T}_{\mf{A}}))$
and
$\mf{c}^{\mf{A}}\in Mor_{\ms{Gr}}(H, Aut_{\ms{set}}(\mc{K}^{\mf{A}}))$.
Moreover for any $\mf{T}\in\pf{T}_{\mf{A}}$, $\alpha\in\ms{P}_{\mf{A}}^{\mf{T}}$ 
and $l\in H$
\begin{equation*}
\mf{c}^{\mf{A}}(l)(\ms{K}_{\alpha}^{\mf{T}}(\mf{A}))
=
\ms{K}_{\alpha}^{\mf{b}^{\mf{A},\ms{v}}(l)(\mf{T})}(\mf{A}).
\end{equation*}
\end{theorem}
\begin{proof}
Let 
$\mf{T}=\lr{\mc{T},\ps{\upmu},\pf{H}}
{\upzeta,f,\Upgamma}\in\pf{T}_{\mf{A}}$,
$l\in H$
and
$\alpha\in\ms{P}_{\mf{A}}^{\mf{T}}$.
In this proof for any $h\in\{Id,l\}$
we use the following notations
\begin{equation*}
\begin{cases}
\mf{H}\doteq\mf{H}_{\alpha},\,\un\doteq\un_{\mf{H}},
\text{ and }Tr\doteq Tr_{\mf{H}},
\\
\ms{D}_{h}\doteq\ms{D}_{\pf{H},\ms{v},\alpha,h}^{\upzeta,f},
\\
R(\lambda,\ms{D}_{h})\doteq(\lambda\un-\ms{D}_{h})^{-1},
\forall\lambda\in\rho(\ms{D}_{h}),
\\
U\doteq\uppi_{\alpha}(\ms{v}(l)),
\\
\ms{B}_{h}\doteq\ms{B}_{\ps{\upmu}}^{\ps{\upomega},\alpha,h},
\\
\mf{R}_{h}\doteq\mf{R}_{\pf{H},\ms{v},\alpha,h}^{\ps{\upmu}},
\\
\ps{\upsigma}_{l}\doteq\ps{\upsigma}^{(\ps{\upomega}_{\alpha},l)},
\\
\mf{T}^{l}
\doteq
\mf{b}(l)(\mf{T}),
\\
\ms{x}^{l}
\doteq
\mf{c}(l)(\ms{x}),
\forall\ms{x}\in
\ms{K}_{\alpha}^{\mf{T}},
\\
\ps{R}_{h}
\doteq
\ps{R}(\mf{T}^{h},\alpha).
\end{cases}
\end{equation*}
Here 
$\un\doteq\un_{\mf{H}_{\alpha}}$
and
$\rho(T)$
is the resolvent set of any selfadjoint operator $T$ in $\mf{H}$.
We convein to remove the index $h$ 
whenever it equals the unit.
\eqref{10131207} yields
\begin{equation}
\label{02042235}
\ps{R}_{l}
=
\ps{R}_{\pf{H},\ms{v},\Upgamma,\alpha,l}^{\ps{\upmu},\upzeta,f}.
\end{equation}
Since Cor. \ref{09101654} and \eqref{01310756} 
to prove $\mf{T}^{l}\in\pf{T}_{\mf{A}}$ it is sufficient to show that
$\ps{R}_{l}$ is an even $\theta-$summable $K-$cycle.
Since \eqref{10121958} and 
Cor. \ref{10301830}\eqref{10301830st2}
\begin{equation}
\label{02010753}
\rho(\ms{D}_{l})=\rho(\ms{D}),
\end{equation}
moreover 
since Lemma \ref{10311019} and 
Thm. \ref{09191747}
we obtain
\begin{equation}
\label{10311020}
\tilde{\mf{R}}_{l}\circ\ps{\upsigma}_{l}^{+}
= ad(U)\circ\tilde{\mf{R}}.
\end{equation}
Let us consider the following set of statements
for $h\in\{Id,l\}$ 
\begin{equation}
\label{10311100}
\begin{cases}
R(\lambda,\ms{D}_{h})
\text{ is a compact operator on $\mf{H}$},
\forall\lambda\in\rho(\ms{D}_{h}),
\\
Dom([\ms{D}_{h},\tilde{\mf{R}_{h}}(a)])=Dom(\ms{D}_{h}),
\forall a\in\ms{B}_{h}^{+}
\\
[\ms{D}_{h},\tilde{\mf{R}_{h}}(a)]
\in\mc{L}(Dom(\ms{D}_{h}),\mf{H}),
\forall a\in\ms{B}_{h}^{+}
\\
Tr(\exp(-\ms{D}_{h}^{2}))<\infty
\end{cases}
\end{equation}
then it holds by hypothesis for $h=Id$, we claim to show it 
for $h=l$. Let $\lambda\in\rho(\ms{D})$, 
thus since 
\eqref{10121958}, 
we have 
$\lambda\un-\ms{D}_{l}
=
U(\lambda\un-\ms{D})U^{-1}$,
and
$Dom(\ms{D}_{l})=U Dom(\ms{D})$,
moreover $Dom(R(\lambda,\ms{D}))=\mf{H}$,
hence 
$(\lambda\un-\ms{D}_{l})
ad(U)(R(\lambda,\ms{D}))
=\un$
and
$ad(U)(R(\lambda,\ms{D}))(\lambda\un-\ms{D}_{l})
=
Id_{Dom(\ms{D}_{l})}$.
Therefore
\begin{equation}
\label{10311055}
R(\lambda,\ms{D}_{l})
=
ad(U)(R(\lambda,\ms{D})).
\end{equation}
Since \eqref{10311055}, \eqref{10311100} for $h=Id$ and since the set of compact operators on $\mf{H}$ 
is a two-sided ideal of $\mc{L}(\mf{H})$, we obtain
\begin{equation}
\label{10311100a}
R(\lambda,\ms{D}_{l})
\text{ is a compact operator.}
\end{equation}
Let $a\in\ms{B}^{+}$, 
thus
$
Dom([\ms{D}_{l},\tilde{\mf{R}}_{l}(\ps{\upsigma}_{l}^{+}(a))])
=
UDom(\ms{D})$
since \eqref{10311100}, \eqref{10311020} and \eqref{10121958},
moreover
\begin{equation}
\label{02010706}
\ms{D}_{l}
\tilde{\mf{R}}_{l}(\ps{\upsigma}_{l}^{+}(a))
-
\tilde{\mf{R}}_{l}(\ps{\upsigma}_{l}^{+}(a))
\ms{D}_{l}
=
U
[\ms{D},\tilde{\mf{R}}(a)]
U^{-1},
\end{equation}
hence for all $v\in Dom(\ms{D})$
\begin{equation*}
\begin{aligned}
\|[\ms{D}_{l},\tilde{\mf{R}}_{l}(\ps{\upsigma}_{l}^{+}(a))]Uv\|
&
=
\|[\ms{D},\tilde{\mf{R}}(a)]v\|
\leq
\|[\ms{D},\tilde{\mf{R}}(a)]\|_{\mc{L}(Dom(\ms{D}),\mf{H})}
\|Uv\|.
\end{aligned}
\end{equation*}
Next
$\ps{\upsigma}_{l}^{+}(\ms{B}^{+})=\ms{B}_{l}^{+}$
since Cor. \ref{09191329}, 
therefore we can state for
all $b\in\ms{B}_{l}^{+}$
\begin{equation}
\label{10311100b}
\begin{cases}
Dom([\ms{D}_{l},\tilde{\mf{R}}_{l}(b)])
=
Dom(\ms{D}_{l}),
\\
[\ms{D}_{l},\tilde{\mf{R}}_{l}(b)]
\in
\mc{L}(Dom(\ms{D}_{l}),\mf{H}),
\\
\|
[\ms{D}_{l},\tilde{\mf{R}}_{l}(b)]
\|_{\mc{L}(Dom(\ms{D}_{l}),\mf{H})}
=
\|
[\ms{D},\tilde{\mf{R}}((\ps{\upsigma}_{l}^{+})^{-1}(b))]
\|_{\mc{L}(Dom(\ms{D}),\mf{H})}.
\end{cases}
\end{equation}
For any $h\in\{Id,l\}$, 
$sp(\ms{D}_{h})\subseteq\R$, since $\ms{D}_{h}$
is a selfadjoint operator by Rmk. \ref{02010723}, 
therefore 
$\exp(-\ms{D}_{h}^{2})\in\mc{L}(\mf{H})$,
since the spectral theorem, see for example
\cite[Thm. $18.2.11(c)$]{ds},
and the fact that 
$sp(\ms{D}_{l})\ni\lambda\mapsto\exp(-\lambda^{2})$, is bounded.
Therefore $Tr(\exp(-\ms{D}_{h}^{2}))$
is a well-set element of $\tilde{\R}$.
Since 
Cor. \ref{10301830}\eqref{10301830st3} 
and \eqref{10121958} 
\begin{equation}
\label{02042212}
\exp(-\ms{D}_{l}^{2})
=
ad(U)(\exp(-\ms{D}^{2})),
\end{equation}
hence 
$Tr(\exp(-\ms{D}_{l}^{2}))
=
Tr(\exp(-\ms{D}^{2}))$, then by 
\eqref{10311100} for $h=Id$ we obtain
\begin{equation}
\label{10311100c}
Tr(\exp(-\ms{D}_{l}^{2}))<\infty.
\end{equation}
\eqref{10311100} for $h=l$ 
follows by
\eqref{02010753}, \eqref{10311100a}, \eqref{10311100b} 
and \eqref{10311100c}.
Next $\Upgamma_{\alpha}^{l}\doteq
\Upgamma_{\pf{H},\ms{v},\alpha}^{l}
=ad(U)(\Upgamma_{\alpha})$
so is a $\Z_{2}-$grading on $\mf{H}$,
moreover since \eqref{10311020}, \eqref{10121958},
the bijectivity of $\ps{\upsigma}_{l}^{+}$
and Rmk. \ref{02041855} 
we obtain for all
$b\in\ms{B}_{l}$
\begin{equation}
\label{02042032}
\begin{cases}
[\Upgamma_{\alpha}^{l},\tilde{\mf{R}}_{l}(b)]=\ze,
\\
\Upgamma_{\alpha}^{l}
\ms{D}_{l}
\Upgamma_{\alpha}^{l}
= -\ms{D}_{l}.
\end{cases}
\end{equation}
Thus
\eqref{10311100} for $h=l$ 
and
\eqref{02042032}
yields that
$\ps{R}_{l}$ is an even $\theta-$summable $K-$cycle,
thus
$\mf{T}^{l}\in\pf{T}_{\mf{A}}$.
So $\mf{b}(l)$ maps $\pf{T}_{\mf{A}}$ into $\pf{T}_{\mf{A}}$
and
$\mf{b}$ is a $H-$action on $\pf{T}_{\mf{A}}$ 
since
\eqref{05201352}. 
Let 
$\ms{x}\in\ms{K}_{\alpha}^{\mf{T}}$
thus
$\ms{x}^{l}\in\ms{K}_{\alpha}^{\mf{T}^{l}}$
since $\mf{T}^{l}\in\pf{T}_{\mf{A}}$ and Prp. \ref{01010828}.
Thus $\mf{c}(l)$ maps $\mc{K}^{\mf{A}}$ into itself since 
Cor. \ref{09191329}
and $\mf{c}$ is a $H-$ action on $\mc{K}^{\mf{A}}$ since 
Cor. \ref{01031207}
\end{proof}
Thm. \ref{11260828} permits the following
\begin{definition}
\label{01011514}
Define $\ov{\mf{c}}^{\mf{A}}:H\to Mor_{\ms{set}}(\ov{\mc{K}}^{\mf{A}},\ov{\mc{K}}^{\mf{A}})$
such that 
$\ov{\mf{c}}^{\mf{A}}(l)(g)\coloneqq\mf{c}^{\mf{A}}(l)\circ g$
for all $l\in H$
and $g\in\ov{\mc{K}}^{\mf{A}}$.
\end{definition}
\begin{proposition}
\label{01011526}
$\ov{\mf{c}}^{\mf{A}}\in Mor_{\ms{Gr}}(H, Aut_{\ms{set}}(\ov{\mc{K}}^{\mf{A}}))$.
\end{proposition}
\begin{proof}
Since Thm. \ref{11260828}.
\end{proof}
Since $\ms{A}_{\mf{A}}\subset Mor_{\ms{set}}(\pf{T}_{\mf{A}},\ov{\mc{K}}^{\mf{A}})$,
Thm. \ref{11260828} and Prp. \ref{01011526} permit the following
\begin{definition}
\label{11261006}
For any $l\in H$ define the map
$\uppsi^{\mf{A},\ms{v}}(l)$ on $\ms{A}_{\mf{A}}$
such that for all $\ms{f}\in\ms{A}_{\mf{A}}$
\begin{equation*}
\uppsi^{\mf{A},\ms{v}}(l)(\ms{f})
\coloneqq
\ov{\mf{c}}^{\mf{A}}(l)
\circ\ms{f}\circ\mf{b}^{\mf{A},\ms{v}}(l^{-1}).
\end{equation*}
\end{definition}
\begin{theorem}
\label{main}
We have
\begin{enumerate}
\item
$\uppsi^{\mf{A},\ms{v}}\in Mor_{\ms{Gr}}(H, Aut_{\ms{Ab}}(\ms{A}_{\mf{A}}))$,
\label{mainst1}
\item
for all $l\in H$ and $\ms{f}\in\ms{A}_{\mf{A}}$ we have
\begin{equation*}
\ms{ev}_{\ms{f}}(\ms{m}^{\mf{A}}\circ\uppsi^{\mf{A},\ms{v}}(l))\circ\mf{b}^{\mf{A},\ms{v}}(l)
=\ms{ev}_{\ms{f}}(\ms{m}^{\mf{A}}).
\end{equation*}
\label{mainst2}
\end{enumerate}
\end{theorem}
\begin{proof}
In this proof we convein to denote $\uppsi^{\mf{A},\ms{v}}$ by $\uppsi$
and $\ov{\mf{c}}^{\mf{A}}$ by $\ov{\mf{c}}$.
$\uppsi(l)$ maps $\ms{A}_{\mf{A}}$ into itself, 
in addition $\uppsi$ is a $H-$action since $\mf{b}$ and $\ov{\mf{c}}$
are $H-$actions by Thm. \ref{11260828} and Prp. \ref{01011526}, 
finally $\uppsi(l)$ is a group morphism since the second inclusion in \eqref{10291120} and since the 
standard picture used for the $\ms{K}_{0}-$groups,
hence st.\eqref{mainst1} follows.
Next let us adopt the notations in proof of Thm. \ref{11260828}, 
let $l\in H$,
$\ms{f}\in\ms{A}_{\mf{A}}$, $\mf{T}\in\pf{T}_{\mf{A}}$ and $\alpha\in\ms{P}_{\mf{A}}^{\mf{T}}$
then
\begin{equation*}
\begin{aligned}
\bigl(\ms{ev}_{\ms{f}}(\ms{m}\circ\uppsi(l))\circ\mf{b}(l)\bigr)
(\mf{T})(\alpha)
&=
\ms{m}(\uppsi(l)(\ms{f}))(\mf{T}^{l})(\alpha)
\\
&=
\lr{\uppsi(l)(\ms{f})(\mf{T}^{l})(\alpha)}{\ms{ch}(
\ps{R}(\mf{T}^{l},\alpha))}_{\ps{\upmu}^{l},
\ps{\upomega}^{l},\alpha}
\\
&=
\lr{\mf{c}(l)\bigl(\ms{f}(\mf{T})(\alpha)\bigr)}
{\ms{ch}(\ps{R}_{l})
}_{\ps{\upmu}^{l},\ps{\upomega}^{l},\alpha}
\\
&=
\lr{\bigl(\ms{f}(\mf{T})(\alpha)\bigr)^{l}}
{\ms{ch}(\ps{R}_{l})}_{\ps{\upmu}^{l},\ps{\upomega}^{l},\alpha}.
\end{aligned}
\end{equation*}
Hence st.\eqref{mainst2} follows if we show that
for all $\ms{x}\in\ms{K}_{\alpha}^{\mf{T}}$
\begin{equation}
\label{01310850}
\lr{\ms{x}^{l}}{\ms{ch}
(\ps{R}_{l})}_{\ps{\upmu}^{l},\ps{\upomega}^{l},\alpha}
=
\lr{\ms{x}}{\ms{ch}
(\ps{R})}_{\ps{\upmu},\ps{\upomega},\alpha}.
\end{equation}
Let us prove \eqref{01310850}.
By \eqref{10311020}, \eqref{10121958} and \eqref{02042212}
we obtain for all $s_{0},\dots,s_{2n}\in\R$
and
$a_{0},\dots,a_{2n}\in\ms{B}^{+}$
\begin{multline}
\label{02051920}
Tr\bigl(
\Upgamma_{\alpha}^{l}\,
\tilde{\mf{R}}_{l}(\ps{\upsigma}_{l}^{+}(a_{0}))
\exp(-s_{0}\ms{D}_{l}^{2})
[\ms{D}_{l},\tilde{\mf{R}}_{l}(\ps{\upsigma}_{l}^{+}(a_{1}))]
\exp(-s_{1}\ms{D}_{l}^{2})
\dotsc
\\
[\ms{D}_{l},\tilde{\mf{R}}_{l}(\ps{\upsigma}_{l}^{+}(a_{2n-1}))]
\exp(-s_{2n-1}\ms{D}_{l}^{2})
[\ms{D}_{l},\tilde{\mf{R}}_{l}(\ps{\upsigma}_{l}^{+}(a_{2n}))]
\exp(-s_{2n}\ms{D}_{l}^{2})
\bigr)
=
\\
(Tr\circ\ms{ad}(U))
\bigl(
\Upgamma_{\alpha}\,
\tilde{\mf{R}}(a_{0})
\exp(-s_{0}\ms{D}^{2})
[\ms{D},\tilde{\mf{R}}(a_{1})]
\exp(-s_{1}\ms{D}^{2})
\dots
\\
[\ms{D},\tilde{\mf{R}}(a_{2n-1})]
\exp(-s_{2n-1}\ms{D}^{2})
[\ms{D},\tilde{\mf{R}}(a_{2n})]
\exp(-s_{2n}\ms{D}^{2})\bigr)
=
\\
Tr\bigl(
\Upgamma_{\alpha}\,
\tilde{\mf{R}}(a_{0})
\exp(-s_{0}\ms{D}^{2})
[\ms{D},\tilde{\mf{R}}(a_{1})]
\exp(-s_{1}\ms{D}^{2})
\dots
\\
[\ms{D},\tilde{\mf{R}}(a_{2n-1})]
\exp(-s_{2n-1}\ms{D}^{2})
[\ms{D},\tilde{\mf{R}}(a_{2n})]
\exp(-s_{2n}\ms{D}^{2})\bigr).
\end{multline}
\eqref{01310850} 
and then 
st.\eqref{mainst2}
follows since
\eqref{02051920}, \eqref{02042235} and 
\cite[$IV.8.\epsilon$ Thm $22$, Thm. $21$, and $IV.7.\delta$ Thm $21$]{connes}.
\end{proof}
\begin{remark}
[Odd case]
Let 
$\pf{T}_{\mf{A}}^{1}$ 
be defined as 
$\pf{T}_{\mf{A}}$ 
by replacing $\ms{K}_{0}$ by $\ms{K}_{1}$
and
setting
$\Upgamma_{\alpha}=\un_{\alpha}$
for all $\alpha\in I\cap]-\infty,\beta_{c}]$.
Thus it is easy to show that Thm. \ref{main} 
still holds
with
$\pf{T}_{\mf{A}}^{1}$ in place of $\pf{T}_{\mf{A}}$
and $\ms{K}_{1}$ 
in place of  $\ms{K}_{0}$. 
\end{remark}
According to the definition of $\ms{v}^{\mf{A}}$ in Def. \ref{11211655} 
and the construction of $\ps{\upsigma}^{(\ps{\upomega}_{\alpha},l)}$ in
Cor. \ref{09191329} we set the following
\begin{definition}
\label{06200822}
Let $\mf{A}=\lr{\mc{A},H}{\upsigma}\in Obj(\ms{C}_{u}(H))$ define
\begin{enumerate}
\item
$\ms{I}_{\mf{A}}:\pf{T}_{\mf{A}}\to\ms{set}$
\item
$(\upbeta_{c}^{\mf{A}})\in\prod_{\mf{Q}\in\pf{T}_{\mf{A}}}
\ms{I}_{\mf{A}}^{\mf{Q}}$;
\item
$\mf{a}_{\mf{A}}\in
\prod_{\mf{Q}\in\pf{T}_{\mf{A}}}
\prod_{\beta\in\ms{I}_{\mf{A}}^{\mf{Q}}}
\ms{C}(H)$;
\item
$\mf{e}_{\mf{A}}\in\prod_{\mf{Q}\in\pf{T}_{\mf{A}}}
\prod_{\beta\in\ms{I}_{\mf{A}}^{\mf{Q}}}
\ms{C}(\R)$;
\item
$\ps{\upvarphi}^{\mf{A}}\in
\prod_{\mf{Q}\in\pf{T}_{\mf{A}}}
\prod_{\beta\in\ms{I}_{\mf{A}}^{\mf{Q}}}
\ms{E}_{\mc{A}}$;
\item
$\uppsi^{\mf{A}}\coloneqq\uppsi^{\mf{A},\ms{v}^{\mf{A}}}$;
\item
$\mf{b}^{\mf{A}}\coloneqq\mf{b}^{\mf{A},\ms{v}^{\mf{A}}}$;
\item
$\mf{H}^{\mf{A}}
\in
\prod_{\mf{Q}\in\pf{T}_{\mf{A}}}
\prod_{\beta\in\ms{I}_{\mf{A}}^{\mf{Q}}}
HS$;
\item
$\uppi^{\mf{A}}
\in
\prod_{\mf{Q}\in\pf{T}_{\mf{A}}}
\prod_{\beta\in\ms{I}_{\mf{A}}^{\mf{Q}}}
Mor_{\ms{CA}^{\ast}}(\mc{A},\mc{L}((\mf{H}^{\mf{A}})_{\beta}^{\mf{Q}}))$;
\item
$\Upomega^{\mf{A}}
\in
\prod_{\mf{Q}\in\pf{T}_{\mf{A}}}
\prod_{\beta\in\ms{I}_{\mf{A}}^{\mf{Q}}}
(\mf{H}^{\mf{A}})_{\beta}^{\mf{Q}}$;
\item
$\mf{E}^{\mf{A}}
=\lr{\upmu^{\mf{A}}}{\mf{u}^{\mf{A}},\pf{H}^{\mf{A}},
\ms{D}^{\mf{A}},\Upgamma^{\mf{A}},\mf{v}^{\mf{A}},\mf{w}^{\mf{A}},\mf{z}^{\mf{A}}}$
an $8-$tuple of elements of 
$\prod_{\mf{Q}\in\pf{T}_{\mf{A}}}\prod_{\beta\in\ms{P}_{\mf{A}}^{\mf{Q}}}\ms{set}$;
\end{enumerate}
where if 
\begin{equation*}
\begin{aligned}
\mf{T}&\in\pf{T}_{\mf{A}},
\\
\mf{T}&=\lr{\mc{T},\ps{\upmu},\pf{H}}{\upzeta,f,\Upgamma},
\\
\mc{T}&=\lr{h,\upxi,\beta_{c},I}{\ps{\upomega}},
\\
\pf{H}_{\gamma}&=\lr{\mf{H}_{\gamma},\uppi_{\gamma}}{\Upomega_{\gamma}},
\forall\gamma\in I;
\end{aligned}
\end{equation*}
then
\begin{enumerate}
\item
$\ms{I}_{\mf{A}}^{\mf{T}}=I$;
\item
$(\upbeta_{c}^{\mf{A}})^{\mf{T}}=\beta_{c}$;
\item
for all $\alpha\in I$
\begin{enumerate}
\item 
$(\mf{a}_{\mf{A}})_{\alpha}^{\mf{T}}=\mf{A}$,
\item
$(\mf{e}_{\mf{A}})_{\alpha}^{\mf{T}}=
\lr{\mc{A},\R}{\uptau_{\upsigma}^{(h,\upxi)}(-\alpha(\cdot))}$,
\item
$(\ps{\upvarphi}^{\mf{A}})_{\alpha}^{\mf{T}}=\ps{\upomega}_{\alpha}$,
\item
$\lr{(\mf{H}^{\mf{A}})_{\alpha}^{\mf{T}}}
{(\uppi^{\mf{A}})_{\alpha}^{\mf{T}},
(\Upomega^{\mf{A}})_{\alpha}^{\mf{T}}}
=
\lr{\mf{H}_{\alpha},\uppi_{\alpha}}{\Upomega_{\alpha}}$,
\item 
$(\upmu^{\mf{A}})_{\alpha}^{\mf{T}}=\ps{\upmu}_{(\alpha,\un)}$,
\item
$(\mf{u}^{\mf{A}})_{\alpha}^{\mf{T}}=
\ms{ev}_{\alpha}\circ\ms{ev}_{\mf{T}}\up\ms{A}_{\mf{A}}$,
\item
$(\pf{H}^{\mf{A}})_{\alpha}^{\mf{T}}
=
\pf{H}_{\alpha}$,
\item
$(\ms{D}^{\mf{A}})_{\alpha}^{\mf{T}}
=
\ms{D}_{\pf{H},\alpha}^{\upzeta,f}(\mf{A})$,
\item
$(\Upgamma^{\mf{A}})_{\alpha}^{\mf{T}}=\Upgamma_{\alpha}$,
\item
for all $l\in H$
\begin{enumerate}
\item
$(\mf{v}^{\mf{A}})_{\alpha}^{\mf{T}}(l)
=
\uppi_{\alpha}(\ms{v}^{\mf{A}}(l))$,
\item
$(\mf{w}^{\mf{A}})_{\alpha}^{\mf{T}}(l)=
\ps{\upsigma}^{(\ps{\upomega}_{\alpha},l)}$,
\item
$(\mf{z}^{\mf{A}})_{\alpha}^{\mf{T}}(l)=Id_{\mc{A}}$;
\end{enumerate}
\end{enumerate}
\end{enumerate}
in addition for all $\alpha\in I$ we define
\begin{equation*}
\begin{aligned}
\ms{S}_{\alpha}^{\mf{T}}(\mf{A})
&\coloneqq
\s{\ps{\upomega}}{\alpha}{}(\mf{A}),
\\
\mc{B}_{\alpha}^{\mf{T}}(\mf{A})
&\coloneqq
\ms{B}_{\ps{\upmu}}^{\ps{\upomega},\alpha}
(\mf{A}),
\\
\mf{i}_{\alpha}^{\mf{T}}
&\coloneqq
\mf{i}^{\mc{B}_{\alpha}^{\mf{T}}(\mf{A})},
\\
\mf{j}_{\alpha}^{\mf{T}}
&\coloneqq
\mf{j}_{\mc{A}}^{\mc{B}_{\alpha}^{\mf{T}}(\mf{A})},
\\
\mf{R}_{\alpha}^{\mf{T}}(\mf{A})
&\coloneqq
\mf{R}_{\pf{H},\alpha}^{\ps{\upmu}}(\mf{A}),
\\
\pf{R}_{\alpha}^{\mf{T}}(\mf{A})
&\coloneqq
\pf{R}_{\pf{H},\alpha}^{\ps{\upmu}}(\mf{A}),
\\
\lr{\cdot}{\cdot}_{(\mf{A},\mf{T},\alpha)}
&\coloneqq
\lr{\cdot}
{\cdot}_{\mc{B}_{\alpha}^{\mf{T}}(\mf{A})^{+}}.
\end{aligned}
\end{equation*}
\end{definition}
Often and only if it will not cause confusion, we convein to remove $(\mf{A})$ 
from $\ms{S}_{\alpha}^{\mf{T}}(\mf{A})$ and $\mc{B}_{\alpha}^{\mf{T}}(\mf{A})$.
According to 
Def. \ref{08261134},
Def. \ref{05121538},
Def. \ref{05131816}
and
Def. \ref{06200822}
we can set the following
\begin{definition}
\label{11261911}
Define $\mf{G}^{H}$ to be the map on $Obj(\ms{C}_{u}(H))$ such that if 
$\mf{A}\in Obj(\ms{C}_{u}(H))$ then
\begin{equation}
\label{11262132}
\mf{G}^{H}(\mf{A})
\coloneqq
\lr{\pf{T}_{\mf{A}},\ms{I}_{\mf{A}},\upbeta_{c}^{\mf{A}},\ms{P}_{\mf{A}},
\mf{a}_{\mf{A}}}
{\mf{e}_{\mf{A}},\ps{\upvarphi}^{\mf{A}},\ms{A}_{\mf{A}},\uppsi^{\mf{A}},
\mf{b}^{\mf{A}},\ms{m}^{\mf{A}},\mf{E}^{\mf{A}}}.
\end{equation}
\end{definition}
\begin{theorem}
\label{01141808}
Let $\mf{T}\in\pf{T}_{\mf{A}}$, 
$\alpha\in\ms{P}_{\mf{A}}^{\mf{T}}$ and $l\in H$,
thus
\begin{equation*}
(\mf{i}_{\alpha}^
{\mf{b}^{\mf{A}}(l)\mf{T}}
\circ
\ps{\upsigma}^{(\ps{\upomega}_{\alpha},l)})^{-}
\circ
\mf{j}_{\mc{A}}^{\mf{T}}
=
\mf{j}_{\mc{A}}^{\mf{b}^{\mf{A}}(l)\mf{T}}\circ\upsigma(l).
\end{equation*}
\end{theorem}
\begin{proof}
In this proof let $\mf{T}^{l}$ denote $\mf{b}^{\mf{A}}(l)\mf{T}$.
Let $f\in\mc{C}_{c}(\ms{S}_{\alpha}^{\mf{T}},\mc{A})$ 
and $a\in\mc{A}$
then
\begin{equation}
\label{01091144}
\begin{aligned}
(\mf{i}_{\alpha}^{\mf{T}^{l}}
\circ
\ps{\upsigma}^{(\ps{\upomega}_{\alpha},l)})^{-}
(\mf{j}_{\mc{A}}^{\mf{T}}(a))
\circ
(\mf{i}_{\alpha}^{\mf{T}^{l}}
\circ
\ps{\upsigma}^{(\ps{\upomega}_{\alpha},l)})(f)
&=
\\
(\mf{i}_{\alpha}^{\mf{T}^{l}}
\circ
\ps{\upsigma}^{(\ps{\upomega}_{\alpha},l)})^{-}
((\mf{j}_{\mc{A}}^{\mf{T}}(a))\circ\mf{i}_{\alpha}^{\mf{T}}(f))
&=
\\
(\mf{i}_{\alpha}^{\mf{T}^{l}}
\circ
\ps{\upsigma}^{(\ps{\upomega}_{\alpha},l)})^{-}
(\mf{i}_{\alpha}^{\mf{T}}(\mf{j}_{\mc{A}}^{\mf{T}}(a)(f)))
&=
\\
(\mf{i}_{\alpha}^{\mf{T}^{l}}
\circ
\ps{\upsigma}^{(\ps{\upomega}_{\alpha},l)})
(\mf{j}_{\mc{A}}^{\mf{T}}(a)(f))
&=
\mf{i}_{\alpha}^{\mf{T}^{l}}
\left(\upsigma(l)\circ\mf{j}_{\alpha}^{\mf{T}}(a)(f)
\circ\ms{ad}(l^{-1})\up\ms{S}_{\alpha}^{\mf{T}^{l}}\right),
\end{aligned}
\end{equation}
where the first and third equalities follow since \eqref{01081715},
the second one by \eqref{01081112II}, the fourth by construction.
Next for all $h\in\ms{S}_{\alpha}^{\mf{T}^{l}}$
\begin{equation*}
\begin{aligned}
\left(\upsigma(l)\circ\mf{j}_{\alpha}^{\mf{T}}(a)(f)
\circ\ms{ad}(l^{-1})\up\ms{S}_{\alpha}^{\mf{T}^{l}}\right)(h)
&=
\\
\upsigma(l)(a f(\ms{ad}(l^{-1})(h)))
&=
\\
\upsigma(l)(a)\ps{\upsigma}^{(\upomega_{\alpha},l)}(f)(h)
&=
\left(
\mf{j}_{\mc{A}}^{\mf{T}^{l}}(\upsigma(l)(a))
\circ
\ps{\upsigma}^{(\ps{\upomega}_{\alpha},l)}
\right)(f)(h),
\end{aligned}
\end{equation*}
hence by \eqref{01091144} we have
\begin{equation}
\label{01091414}
\begin{aligned}
(\mf{i}_{\alpha}^{\mf{T}^{l}}
\circ
\ps{\upsigma}^{(\ps{\upomega}_{\alpha},l)})^{-}
(\mf{j}_{\mc{A}}^{\mf{T}}(a))
\circ
(\mf{i}_{\alpha}^{\mf{T}^{l}}
\circ
\ps{\upsigma}^{(\ps{\upomega}_{\alpha},l)})(f)
&=
\\
\left(
\mf{i}_{\alpha}^{\mf{T}^{l}}
\circ
\mf{j}_{\mc{A}}^{\mf{T}^{l}}(\upsigma(l)(a))
\circ
\ps{\upsigma}^{(\ps{\upomega}_{\alpha},l)}
\right)(f)
&=
\mf{j}_{\mc{A}}^{\mf{T}^{l}}(\upsigma(l)(a))
\circ
(\mf{i}_{\alpha}^{\mf{T}^{l}}
\circ
\ps{\upsigma}^{(\ps{\upomega}_{\alpha},l)})(f),
\end{aligned}
\end{equation}
where the last equality follows since \eqref{01081112II}.
Next $\mc{C}_{c}(\ms{S}_{\alpha}^{\mf{T}},\mc{A})$ 
is dense in $\mc{B}_{\alpha}^{\mf{T}}$
moreover 
$\ps{\upsigma}^{(\ps{\upomega}_{\alpha},l)}$
is an isometry since 
Cor. \ref{09191329},
thus by \eqref{01091414} we deduce 
\begin{equation*}
\mf{i}^{\ms{M}(\mc{B}_{\alpha}^{\mf{T}^{l}})}
((\mf{i}_{\alpha}^{\mf{T}^{l}}
\circ
\ps{\upsigma}^{(\ps{\upomega}_{\alpha},l)})^{-}
(\mf{j}_{\mc{A}}^{\mf{T}}(a)))
\up\mc{K}(\mc{B}_{\alpha}^{\mf{T}^{l}})
=
\mf{i}^{\ms{M}(\mc{B}_{\alpha}^{\mf{T}^{l}})}
(\mf{j}_{\mc{A}}^{\mf{T}^{l}}(\upsigma(l)(a)))
\up\mc{K}(\mc{B}_{\alpha}^{\mf{T}^{l}}),
\end{equation*}
therefore the statement follows since Lemma \ref{01091430}.
\end{proof}
Now we are able to state the following
\begin{corollary}
\label{11271221}
$\mf{G}^{H}$ 
maps 
$Obj(\ms{C}_{u}(H))$ into $Obj(\mf{G}(G,F,\uprho))$.
\end{corollary}
\begin{proof}
Def. \ref{05301823}(\ref{08261212},\ref{08261213})
follow since Thm. \ref{main}\eqref{mainst1}
and Thm. \ref{11260828}. 
The additivity of $\ms{m}^{\mf{A}}$ follows by the 
additivity of the Chern-Connes character,
\eqref{12051548} follows since 
Lemma \ref{09102032}\eqref{09101724pre},
\eqref{eqtherm} by \eqref{08301407}, 
\eqref{12051444} since Thm. \ref{main}\eqref{mainst2},
\eqref{eqdyn} by \eqref{08291407} and 
\eqref{09071855},
the integrality
since Rmk. \ref{02061932},
$(\mf{u}^{\mf{A}})_{\alpha}^{\mf{T}}$ is by construction a group
morphism while
\eqref{12220022} follows since the construction
of $\ms{m}^{\mf{A}}$.
Next \eqref{01031206}, \eqref{01031145d} and \eqref{01031145} 
since 
Cor. \ref{01031207}, 
\eqref{10121958} 
and 
Thm. \ref{09191747} 
respectively, while \eqref{01031145b} follows by the 
construction of $\uppsi^{\mf{A}}$. Finally \eqref{01031145c} follows since 
Thm. \ref{01141808}.
\end{proof}
Notice that $\ms{V}(\mf{G}^{H}(\mf{A}))_{\alpha}^{\mf{T}}=\upsigma$ for all $\mf{A}\in Obj(\ms{C}_{u}(H))$,
$\mf{T}\in\pf{T}_{\mf{A}}$ and $\alpha\in\ms{P}_{\mf{A}}^{\mf{T}}$, where $\upsigma$ is the dynamics underlying $\mf{A}$.
\subsection{The functor $\mf{G}_{\vartriangle}^{H}$}
\label{func2}
Thm. \ref{12071120} completes the result in Cor. \ref{11271221}
by stating that $\mf{G}^{H}$ is the object part 
of a functor from $\ms{C}_{u}(H)$ to $\mf{G}(G,F,\uprho)$.
Then we utilize this result in our main theorem
Thm. \ref{06242121}
to exhibit the existence of 
an extended full integer $\ms{C}_{u}(H)-$equivariant stability,
and consequently the existence of a 
nucleon-fragment doublet on $\ms{C}_{u}(H)$. 
As a consequence we obtain in Cor. \ref{19061936}
the resolution of the equivariant form of the universality claim.
\begin{definition}
\label{11210748}
Let $T:\mc{A}\to\mc{B}$ be a $\ast-$homomorphism between
$C^{\ast}-$algebras
and
$\pf{H}:A\ni\alpha\mapsto\lr{\mf{H}_{\alpha},
\uppi_{\alpha}}{\Upomega_{\alpha}}\in Rep_{c}(\mc{B})$, 
define the map $\pf{H}^{T}$ on $A$ such that 
$\pf{H}_{\beta}^{T}\coloneqq
\lr{\mf{H}_{\beta},\uppi_{\beta}\circ T}{\Upomega_{\beta}}$,
for all $\beta\in A$. 
\end{definition}
\begin{lemma}
\label{11200732}
Let $T:\mc{A}\to\mc{B}$ be a surjective $\ast-$homomorphism between $C^{\ast}-$algebras,
$\upomega\in\ms{E}_{\mc{B}}$ and $\pf{H}$ be a cyclic representation of $\mc{B}$ associated with $\upomega$.
Then $\pf{H}^{T}$ is a cyclic representation of $\mc{A}$ associated with $T_{\dagger}(\upomega)$.
\end{lemma}
\begin{proof}
$T_{\dagger}(\upomega)\in\ms{E}_{\mc{B}}$ since 
Lemma \ref{11191814} 
so the statement is well-set,
moreover $\pf{H}^{T}$ is cyclic since $T(\mc{A})=\mc{B}$.
\end{proof}
Unless differently specified in the present subsection
let 
$\mf{A}=\lr{\mc{A},H}{\upeta}$,
$\mf{B}=\lr{\mc{B},H}{\uptheta}$
and
$\mf{C}=\lr{\mc{C},H}{\updelta}$
be objects of $\ms{C}(H)$,
$T\in Mor_{\ms{C}(H)}(\mf{A},\mf{B})$ 
and
$S\in Mor_{\ms{C}(H)}(\mf{B},\mf{C})$. 
\begin{corollary}
\label{11201755}
Let 
$\upphi\in\ms{E}_{\mc{B}}^{G}(\uptau_{\uptheta})$
and $\pf{H}$ 
be a cyclic representation 
of $\mc{B}$ associated with $\upphi$,
then
$T_{\dagger}(\upphi)\in\ms{E}_{\mc{A}}^{G}(\uptau_{\upeta})$
and 
$\ms{U}_{\pf{H}^{T}}^{\upeta}
=
\ms{U}_{\pf{H}}^{\uptheta}$.
\end{corollary}
\begin{remark}
\label{11201312}
$\ms{U}_{\pf{H}_{T}}^{\upeta}$ makes sense
since 
Def. \ref{08192224pre},
the first sentence of the statement of Cor. \ref{11201755}
and Lemma \ref{11200732}.
The equality in Cor. \ref{11201755} is well-set
since
Lemma \ref{11131734}. 
\end{remark}
\begin{proof}[Proof of Cor. \ref{11201755}]
$\upphi\circ T\circ\upeta(j_{1}(g))=
\upphi\circ\uptheta(j_{1}(g))\circ T
=
\upphi\circ T$ for all $g\in G$,
so the first sentence of the statement follows since
Lemma \ref{11191814}.
Let $\pf{H}=\lr{\mf{H},\uppi}{\Upomega}$
and $l\in\ms{S}_{\ms{F}_{T_{\dagger}(\upphi)}}^{G}$
then for all $a\in\mc{A}$
\begin{equation*}
\begin{aligned}
\ms{U}_{\pf{H}_{T}}^{\upeta}(l)
(\uppi\circ T)
(a)\Upomega
&=
(\uppi\circ T)(\upeta(l)a)\Upomega
\\
&=
(\uppi\circ\uptheta(l)\circ T)(a)\Upomega
=
\ms{U}_{\pf{H}}^{\uptheta}(l)
(\uppi\circ T)(a)\Upomega,
\end{aligned}
\end{equation*}
where the last equality follows since 
Lemma \ref{11131734}.
Thus the equality in the statement follows since $\pf{H}_{T}$ is cyclic
by Lemma \ref{11200732}.
\end{proof}
\begin{proposition}
\label{11211031}
Let $A$ be a nonempty set and 
$\ps{\upomega}:A\to\ms{E}_{\mc{B}}^{G}(\uptau_{\uptheta})$.
Then
$\mc{H}(\ps{\upomega},\mf{B})=
\mc{H}(T_{\dagger}\circ\ps{\upomega},\mf{A})$.
\end{proposition}
\begin{proof}
Since Cor. \ref{11201755} the statement is well-set.
Let $l\in H$ and $\alpha\in A$ then
\begin{equation*}
\upeta^{\ast}(l)(T_{\dagger}(\ps{\upomega}_{\alpha}))
=
\ps{\upomega}_{\alpha}\circ\uptheta(l^{-1})\circ T
=
T_{\dagger}(\uptheta^{\ast}(l)(\ps{\upomega}_{\alpha})),
\end{equation*}
hence by 
Lemma \ref{11131734}
\begin{equation}
\label{27061857}
\ms{S}_{\ms{F}_{\upeta^{\ast}(l)(T_{\dagger}(\ps{\upomega}_{\alpha}))}}^{G}(\mf{A})
=
\ms{S}_{\ms{F}_{\uptheta^{\ast}(l)(\ps{\upomega}_{\alpha})}}^{G}(\mf{B}),
\end{equation}
and the statement follows.
\end{proof}
\begin{remark}
\label{29061357}
Let $A$ be a nonempty set, 
$\ps{\upomega}:A\to\ms{E}_{\mc{B}}^{G}(\uptau_{\uptheta})$
and $\ps{\upmu}\in\mc{H}(\ps{\upomega},\mf{B})$.
Thus for all $\alpha\in A$ and $l\in H$ let 
$\ms{ad}_{\star}^{(\ps{\upomega}_{\alpha},\uptheta)}(l)$
on
$\mc{C}_{c}(\s{\upomega}{\alpha}{},\ms{A})$
at values in
$\mc{C}_{c}
(\ms{S}_{\ms{F}_{\uptheta^{\ast}(l)(\ps{\upomega}_{\alpha})}}^{G}
\mc{A})$
such that 
\begin{equation*}
f
\mapsto
f\circ\ms{ad}(l^{-1})
\up
\ms{S}_{\ms{F}_{\uptheta^{\ast}(l)(\ps{\upomega}_{\alpha})}}^{G},
\end{equation*}
this map is continuous w.r.t. the inductive limit topology 
following the line in the proof of
Lemma \ref{09171541},
hence there exists according to \cite[Cor. $2.47$]{will}
a unique extension on 
$\ms{B}_{\ps{\upmu}}^{\ps{\upomega},\alpha}(\mf{B})$
at values in
$\ms{B}_{\ps{\upmu}}^{\ps{\upomega},\alpha,l}(\mf{B})$
which will be denoted again by the symbol
$\ms{ad}_{\star}^{(\ps{\upomega}_{\alpha},\uptheta)}(l)$.
Next it is easy to see that
\begin{equation*}
\begin{aligned}
\ps{\uptheta}^{(\ps{\upomega}_{\alpha},l)}
&=
\ms{c}_{\ps{\upmu}_{(\alpha,l)}}(\uptheta(l))
\circ
\ms{ad}_{\star}^{(\ps{\upomega}_{\alpha},\uptheta)}(l)
\\
&=
\ms{ad}_{\star}^{(\ps{\upomega}_{\alpha},\uptheta)}(l)
\circ
\ms{c}_{\ps{\upmu}_{(\alpha,\un)}}(\uptheta(l)).
\end{aligned}
\end{equation*}
\end{remark}
\begin{remark}
\label{11271617}
Let $A$ be a nonempty set, 
$\ps{\upomega}:A\to\ms{E}_{\mc{B}}^{G}(\uptau_{\uptheta})$
and $\ps{\upmu}\in\mc{H}(\ps{\upomega},\mf{B})$.
Thus for all $\alpha\in A$ and $l\in H$
since 
\eqref{12010348}, 
Lemma \ref{11191814} 
and 
Lemma \ref{11131734}
we have
\begin{equation*}
\ms{c}_{\ps{\upmu}_{(\alpha,l)}}(T)
\in Mor_{\ms{CA}^{\ast}}
(\ms{B}_{\ps{\upmu}}^{T_{\dagger}\circ\ps{\upomega},\alpha,l}(\mf{A}),
\ms{B}_{\ps{\upmu}}^{\ps{\upomega},\alpha,l}(\mf{B})),
\end{equation*}
where $\ms{B}_{\ps{\upmu}}^{T_{\dagger}\circ\ps{\upomega},\alpha,l}(\mf{A})$ is well-set since Prp. \ref{11211031}.
In addition since \eqref{11291802} we have 
\begin{equation*}
\ms{k}_{\ps{\upmu}_{(\alpha,l)}}(T):
\ms{K}_{0}
(\ms{B}_{\ps{\upmu}}^{T_{\dagger}\circ\ps{\upomega},
\alpha,l}(\mf{A}))
\to
\ms{K}_{0}(\ms{B}_{\ps{\upmu}}^{\ps{\upomega},\alpha,l}(\mf{B})).
\end{equation*}
Moreover 
$T_{\dagger}\circ\ps{\upomega}:A\to\ms{E}_{\mc{A}}^{G}
(\uptau_{\upeta})$
and $\ps{\upmu}\in\mc{H}(T_{\dagger}\circ\ps{\upomega},\mf{A})$
since Cor. \ref{11201755} and Prp. \ref{11211031}.
Finally since \eqref{27061857} we deduce that
\begin{equation}
\label{29061509}
\ms{ad}_{\star}^{(\ps{\upomega}_{\alpha},\uptheta)}(l)
\circ
\ms{c}_{\ps{\upmu}_{(\alpha,\un)}}(T)
=
\ms{c}_{\ps{\upmu}_{(\alpha,l)}}(T)
\circ
\ms{ad}_{\star}^{(T_{\dagger}(\ps{\upomega}_{\alpha}),\upeta)}(l).
\end{equation}
\end{remark}
\begin{definition}
\label{11271658}
Define $\mf{d}^{H}$ the map on $Mor_{\ms{C}(H)}$ such that 
\begin{equation*}
\mf{d}^{H}(T):\pf{T}_{\mf{B}}\ni
\lr{\mc{T},\ps{\upmu},\pf{H}}{\upzeta,f,\Upgamma}
\mapsto
\lr{\mc{T}^{T},\ps{\upmu},\pf{H}^{T}}
{\upzeta,f,\Upgamma},
\end{equation*}
where
if
$\mc{T}=\lr{h,\upxi,\beta_{c},I}{\ps{\upomega}}$ 
then
$\mc{T}^{T}=
\lr{h,\upxi,\beta_{c},I}{T_{\dagger}\circ\ps{\upomega}}$. 
\end{definition}
In the remaining of this subsection we let $\mf{d}$ denote $\mf{d}^{H}$.
In the following result we shall use 
the equivariance of the $KMS-$states under the action 
of surjective equivariant maps stated in 
Thm. \ref{11141627}.
\begin{lemma}
\label{11271700} 
$\mf{d}(T):\pf{T}_{\mf{B}}\to\pf{T}_{\mf{A}}$ and $\mf{d}(S\circ T)=\mf{d}(T)\circ\mf{d}(S)$,
as a result $(\pf{T}_{\bullet}\circ\mf{G}^{H},\mf{d}\up Mor_{\ms{C}_{u}(H)})$ is
a functor from $\ms{C}_{u}(H)^{op}$ to $\ms{set}$.
\end{lemma}
\begin{proof}
Let 
$\mf{T}=\lr{\mc{T},\ps{\upmu},\pf{H}}{\upzeta,f,\Upgamma}
\in\pf{T}_{\mf{B}}$ 
with 
$\mc{T}=\lr{h,\upxi,\beta_{c},I}{\ps{\upomega}}$,
we claim to show that $\mf{d}(T)(\mf{T})$ satisfies 
the requests in Def. \ref{08261134}.
$\mc{T}^{T}$ is a pre thermal phase associated with $\mf{A}$
since Cor. \ref{11201755} and 
Thm. \ref{11141627}.
Therefore we obtain
Def. \ref{08261134}
(\ref{08261134III},\ref{08261134IV},\ref{08261134V},\ref{08261134VI})
since Prp. \ref{11211031}, Lemma \ref{11200732}, 
Lemma \ref{11131734} 
and Cor. \ref{11201755}
respectively.
Next let 
$\alpha\in\ms{P}_{\mf{B}}^{\mf{T}}$ 
and $\pf{H}_{\alpha}=
\lr{\uppi,\mf{H}_{\alpha}}{\Upomega_{\alpha}}$,
then
$\pf{R}_{\pf{H}^{T},\alpha}^{\ps{\upmu}}(\mf{A})
=
(\ms{B}_{\ps{\upmu}}^{T_{\dagger}\circ\ps{\upomega},\alpha}(\mf{A}),
\mf{R}_{\pf{H}^{T},\alpha}^{\ps{\upmu}}(\mf{A}))$,
where
\begin{equation}
\label{11301824}
\begin{aligned}
\mf{R}_{\pf{H}^{T},\alpha}^{\ps{\upmu}}(\mf{A})
&=
(\uppi_{\alpha}\circ T)\rtimes^{\ps{\upmu}_{(\alpha,\un)}}
\ms{U}_{\pf{H}_{\alpha}^{T}}^{\upeta}
\\
&=
(\uppi_{\alpha}\circ T)\rtimes^{\ps{\upmu}_{(\alpha,\un)}}
\ms{U}_{\pf{H}_{\alpha}}^{\uptheta}
\\
&=
(\uppi_{\alpha}\rtimes^{\ps{\upmu}_{(\alpha,\un)}}
\ms{U}_{\pf{H}_{\alpha}}^{\uptheta})
\circ\ms{c}_{\ps{\upmu}_{(\alpha,\un)}}(T)
\\
&=
\mf{R}_{\pf{H},\alpha}^{\ps{\upmu}}(\mf{B})
\circ\ms{c}_{\ps{\upmu}_{(\alpha,\un)}}(T),
\end{aligned}
\end{equation}
where the second equality follows by 
Cor. \ref{11201755}.
Next
let $X$ be the nonempty set such that 
$\R^{X}$ is the domain of $\upzeta$,
thus
$\ms{D}_{\pf{H}^{T},\alpha}^{\upzeta,f}(\mf{A})
=
\ms{D}_{\pf{H}_{\alpha}^{T}}^{\upzeta,f}(\mf{A})
=
f(\ms{E}_{\mc{E}_{\ms{L}}})$,
where
$\ms{L}=\{L_{x}\}_{x\in X}$ such that $iL_{x}$ is 
the infinitesimal generator
of the strongly continuous one-parameter semigroup 
$\ms{U}_{\pf{H}_{\alpha}^{T}}^{\upeta}
\circ\upzeta\circ\mf{i}_{x}\up\R^{+}$ on $\pf{H}_{\alpha}$,
for all $x\in X$.
Thus since Cor. \ref{11201755} we deduce that
\begin{equation}
\label{11302054}
\ms{D}_{\pf{H}^{T},\alpha}^{\upzeta,f}(\mf{A})
=
\ms{D}_{\pf{H},\alpha}^{\upzeta,f}(\mf{B}).
\end{equation}
Def. \eqref{08261134}\eqref{08261134VIII}
follows since
\eqref{11301824}, \eqref{11302054} and Rmk. \ref{02041855}
and our claim is proved so 
$\mf{d}(T)(\mf{T})\in\pf{T}_{\mf{A}}$.
The remaining part of the statement is easy to show.
\end{proof}
\begin{definition}
\label{11271820}
Define
\begin{equation*}
\mc{K}^{\mf{A}}(T)
\coloneqq
\bigcup_{\mf{T}\in\pf{T}_{\mf{B}}}
\bigcup_{\alpha\in\ms{P}_{\mf{B}}^{\mf{T}}}
\ms{K}_{\alpha}^{\mf{d}(T)(\mf{T})}(\mf{A}),
\end{equation*}
and
\begin{equation*}
\ov{\mc{K}}^{\mf{A}}(T)
\coloneqq
\bigcup_{\mf{T}\in\pf{T}_{\mf{B}}}
\prod_{\alpha\in\ms{P}_{\mf{B}}^{\mf{T}}}
\ms{K}_{\alpha}^{\mf{d}(T)(\mf{T})}(\mf{A}).
\end{equation*}
\end{definition}
$\mc{K}^{\mf{A}}(T)$ is a well-defined subset of $\mc{K}^{\mf{A}}$
since Lemma \ref{11271700}.
\begin{definition}
\label{11271828}
Define $\mf{h}^{H}(T):\mc{K}^{\mf{A}}(T)\to\mc{K}^{\mf{B}}$
such that for any 
$\mf{T}
=
\lr{\mc{T},\ps{\upmu},\pf{H}}{\upzeta,f,\Upgamma}
\in\pf{T}_{\mf{B}}$
and
$\alpha\in\ms{P}_{\mf{B}}^{\mf{T}}$ 
\begin{equation*}
\mf{h}^{H}(T)
\up
\ms{K}_{\alpha}^{\mf{d}(T)(\mf{T})}(\mf{A})
\coloneqq
\ms{k}_{\ps{\upmu}_{(\alpha,\un)}}(T).
\end{equation*}
\end{definition}
In the remaining of this subsection
we let $\mf{h}$ denote $\mf{h}^{H}$. 
\begin{lemma}
\label{11272138}
We have
\begin{enumerate}
\item
$\mf{h}(T)$ is well-defined and
$\mf{h}(T)(\ms{K}_{\alpha}^{\mf{d}(T)(\mf{T})}(\mf{A}))
\subseteq
\ms{K}_{\alpha}^{\mf{T}}(\mf{B})$,
for any $\mf{T}\in\pf{T}_{\mf{B}}$
and $\alpha\in\ms{P}_{\mf{B}}^{\mf{T}}$;
\label{11272138st1}
\item
$\mc{K}^{\mf{A}}(S\circ T)\subseteq\mc{K}^{\mf{A}}(T)$;
\label{11272138st2}
\item
Let $\imath$ be the identity
map from $\mc{K}^{\mf{A}}(S\circ T)$ to $\mc{K}^{\mf{A}}(T)$,
then the following diagram is commutative 
\begin{equation}
\label{12010253}
\xymatrix{
\mc{K}^{\mf{B}}(S) \ar[r]^{\mf{h}(S)} & \mc{K}^{\mf{C}}
\\
\ar[u]^{\mf{h}(T)\circ\imath}
\mc{K}^{\mf{A}}(S\circ T) 
\ar[ru]_{\mf{h}(S\circ T)} 
&
}
\end{equation}
\end{enumerate}
\end{lemma}
\begin{proof}
$\mf{h}(T)$ is well-defined
since the same argument used in Lemma \ref{01010828},
while the inclusion in
st.\eqref{11272138st1} follows by Rmk \ref{11271617}.
St.\eqref{11272138st2} follows by Lemma \ref{11271700}.
$\mf{h}(T)\circ\imath$ is a well-set map since 
st.\eqref{11272138st2},
with values in $\mc{K}^{\mf{B}}(S)$ 
since Lemma \ref{11271700} 
and st.\eqref{11272138st1}. 
The diagram is commutative since 
$\ms{k}_{\ps{\upmu}_{(\alpha,\un)}}$
is the morphism map of
a functor from $\ms{C}_{0}(H)$ to $\ms{Ab}$.
\end{proof}
Since Lemma \ref{11272138}\eqref{11272138st1} we can give the following
\begin{definition}
\label{01020953}
Define 
$\ov{\mf{h}}^{H}(T):
\ov{\mc{K}}^{\mf{A}}(T)
\to\ov{\mc{K}}^{\mf{B}}$
such that 
$\ov{\mf{h}}^{H}(T)(g)\coloneqq
\mf{h}(T)\circ g$
for all $g\in \ov{\mc{K}}^{\mf{A}}(T)$.
Moreover
set
$\ov{\imath}:
\ov{\mc{K}}^{\mf{A}}(S\circ T)
\to
\ov{\mc{K}}^{\mf{A}}(T)$
such that 
$\ov{\imath}(g)=\imath\circ g$
for all
$g\in\ov{\mc{K}}^{\mf{A}}(S\circ T)$.
\end{definition}
\begin{proposition}
The following diagram is commutative 
\begin{equation}
\label{12010253bis}
\xymatrix{
\ov{\mc{K}}^{\mf{B}}(S) \ar[r]^{\ov{\mf{h}}(S)} 
& \ov{\mc{K}}^{\mf{C}}
\\
\ar[u]^{\ov{\mf{h}}(T)\circ\ov{\imath}}
\ov{\mc{K}}^{\mf{A}}(S\circ T) 
\ar[ru]_{\ov{\mf{h}}(S\circ T)}
&
}
\end{equation}
\end{proposition}
\begin{proof}
Since \eqref{12010253}.
\end{proof}
We recall that according to 
our notation $(f\times g):X\to A\times B$ such that 
$(f\times g)(x)=(f(x),g(x))$ for any $x\in X$, where
$X,A,B$ are sets while $f:X\to A$ and $g:X\to B$.
For any $\ms{f}\in\ms{A}_{\mf{A}}$
the map
$\ms{f}\circ\mf{d}(T)$
is well-set
since Lemma \ref{11271700},
while
$\ms{f}\circ\mf{d}(T):\pf{T}_{\mf{B}}\to\ov{\mc{K}}^{\mf{A}}(T)$,
thus since Def. \ref{01020953}
and Lemma \ref{11272138}\eqref{11272138st1},
we can give the following
\begin{definition}
\label{11272234}
Define the map $\mf{g}^{H}$ on $Mor_{\ms{C}(H)}$ 
such that  
\begin{equation*}
\mf{g}^{H}(T):
\ms{A}_{\mf{A}}\to\ms{A}_{\mf{B}},
\end{equation*}
and for any $\ms{f}\in\ms{A}_{\mf{A}}$
\begin{equation*}
\mf{g}^{H}(T)(\ms{f})
\coloneqq
\ov{\mf{h}}^{H}(T)\circ\ms{f}\circ\mf{d}^{H}(T).
\end{equation*}
Finally define
\begin{equation*}
\mf{G}_{\vartriangle}^{H}
\coloneqq
(\mf{G}^{H},(\mf{g}^{H}\times\mf{d}^{H})\up Mor_{\ms{C}_{u}(H)}).
\end{equation*}
\end{definition}
Now we are in the position of stating that $\mf{G}_{\vartriangle}^{H}$ 
is a functor namely
\begin{theorem}
\label{12071120}
$\mf{G}_{\vartriangle}^{H}\in\ms{Fct}(\ms{C}_{u}(H),\mf{G}(G,F,\uprho))$.
\end{theorem}
\begin{proof}
The part of the statement concerning $\mf{G}^{H}$ 
follows since Cor. \ref{11271221}. 
Let $\mf{A}$, $\mf{B}$ and $\mf{C}$
be objects of $\ms{C}_{u}(H)$, 
$T\in Mor_{\ms{C}_{u}(H)}(\mf{A},\mf{B})$ 
and
$S\in Mor_{\ms{C}_{u}(H)}(\mf{B},\mf{C})$. 
Let $\mf{g}$ denote $\mf{g}^{H}$, thus
$\mf{g}(T)$ is a group morphism since 
the second line in \eqref{10291120} and since the 
standard picture used for the $\ms{K}_{0}-$groups.
Next we claim to show that
\begin{equation}
\label{12011707}
\begin{aligned}
\mf{g}(S\circ T)
&=\mf{g}(S)\circ\mf{g}(T),
\\
\ms{ev}_{\ms{f}}(\ms{m}^{\mf{B}}\circ\mf{g}(T))
&=
\ms{ev}_{\ms{f}}(\ms{m}^{\mf{A}})\circ\mf{d}(T),
\forall\ms{f}\in\ms{A}_{\mf{A}}.
\end{aligned}
\end{equation}
The first equality follows since \eqref{12010253bis} 
and 
Lemma \ref{11271700},
let us prove the second equality of \eqref{12011707}.
Let $\ms{f}\in\ms{A}_{\mf{A}}$, 
$\mf{T}\in\pf{T}_{\mf{B}}$,
and $\alpha\in\ms{P}_{\mf{B}}^{\mf{T}}$,
moreover let
$\mf{T}=
\lr{\mc{T},\ps{\upmu},\pf{H}}{\upzeta,f,\Upgamma}$
with
$\mc{T}=\lr{h,\upxi,\beta_{c},I}{\ps{\upomega}}$,
and let $\mf{T}^{T}$ denote $\mf{d}(T)(\mf{T})$, then
\begin{equation}
\label{12011644a}
\begin{aligned}
\ms{ev}_{\ms{f}}(\ms{m}^{\mf{B}}\circ\mf{g}(T))(\mf{T})(\alpha)
&=
\ms{m}^{\mf{B}}(\mf{g}(T)(\ms{f}))(\mf{T})(\alpha)
\\
&=
\lr{\mf{g}(T)(\ms{f})(\mf{T})(\alpha)}
{\ms{ch}(\ps{R}(\mf{T},\alpha))}_{\ps{\upmu},\ps{\upomega},\alpha}
\\
&=
\lr{(\ms{c}_{\ps{\upmu}_{(\alpha,\un)}}^{+}(T))_{\ast}\,
\bigl(\ms{f}(\mf{T}^{T})(\alpha)\bigr)}
{\ms{ch}(\ps{R}(\mf{T},\alpha))}_{\ps{\upmu},\ps{\upomega},\alpha}
\\
&=
\lr{\ms{f}(\mf{T}^{T})(\alpha)}
{(\ms{c}_{\ps{\upmu}_{(\alpha,\un)}}^{+}(T))_{\dagger}
\bigl(\ms{ch}(\ps{R}(\mf{T},\alpha))\bigr)}_{\ps{\upmu},T_{\dagger}\circ\ps{\upomega},\alpha},
\end{aligned}
\end{equation}
where the last equality follows since Rmk. \ref{11271617}
and
\eqref{12011443}.
Next 
\begin{equation}
\label{12011644b}
\begin{aligned}
\bigl(\ms{ev}_{\ms{f}}(\ms{m}^{\mf{A}})\circ\mf{d}(T)\bigr)
(\mf{T})(\alpha)
&=
\ms{m}^{\mf{A}}(\ms{f})(\mf{T}^{T})(\alpha)
\\
&=
\lr{\ms{f}(\mf{T}^{T})(\alpha)}
{\ms{ch}(\ps{R}(\mf{T}^{T},\alpha))}_{\ps{\upmu},T_{\dagger}\circ\ps{\upomega},\alpha}.
\end{aligned}
\end{equation}
Moreover by construction
\begin{equation*}
\ps{R}(\mf{T}^{T},\alpha)
=\bigl(
\ms{B}_{\ps{\upmu}}^{T_{\dagger}\circ\ps{\upomega},\alpha,+}(\mf{A}),\,
\tilde{\mf{R}}_{\pf{H}^{T},\alpha}^{\ps{\upmu}}(\mf{A}),\,
\ms{D}_{\pf{H}^{T},\alpha}^{\upzeta,f}(\mf{A}),\,
\Upgamma_{\alpha}\bigr),
\end{equation*}
thus we obtain since 
\eqref{11301824}\,\&\,\eqref{11302054} and Lemma \ref{10311019}
\begin{equation*}
\ps{R}(\mf{T}^{T},\alpha)
=\bigl(\ms{B}_{\ps{\upmu}}^{T_{\dagger}\circ\ps{\upomega},\alpha,+}(\mf{A}),\,
\tilde{\mf{R}}_{\pf{H},\alpha}^{\ps{\upmu}}(\mf{B})\circ
\ms{c}_{\ps{\upmu}_{(\alpha,\un)}}^{+}(T),\,
\ms{D}_{\pf{H},\alpha}^{\upzeta,f}(\mf{B}),\,
\Upgamma_{\alpha}\bigr).
\end{equation*}
Hence we deduce 
since 
\cite[Ch $IV.8.\epsilon$, Thm. $22$ and Thm. $21$]{connes}
and
\cite[Ch $IV.7.\delta$, Thm. $21$ and Lemma $20$]{connes}
that
\begin{equation}
\label{12011644c}
\ms{ch}(\ps{R}(\mf{T}^{T},\alpha))
=
(\ms{c}_{\ps{\upmu}_{(\alpha,\un)}}^{+}(T))_{\dagger}\,
(\ms{ch}(\ps{R}(\mf{T},\alpha))).
\end{equation}
\eqref{12011644a}\,\&\,\eqref{12011644b}\,\&\,\eqref{12011644c} imply the claimed second equality in \eqref{12011707}.
Next since \eqref{12011707} and Lemma \ref{11271700}
\begin{equation*}
(\mf{g}\times\mf{d})(T)
\in Mor_{\mf{G}(G,F,\uprho)}(\mf{G}^{H}(\mf{A}),\mf{G}^{H}(\mf{B})),
\end{equation*}
while since the first equality in \eqref{12011707} 
and Lemma \ref{11271700}
\begin{equation*}
(\mf{g}\times\mf{d})(S\circ T)
=
(\mf{g}\times\mf{d})(S)
\circ(\mf{g}\times\mf{d})(T),
\end{equation*}
where $\circ$ in the right side of the equality 
is the law of composition in the set of morphisms
of $\mf{G}(G,F,\uprho)$, and the statement follows.
\end{proof}
\begin{remark}
\label{06260916}
According to Cor. \ref{11271221}, 
Def. \ref{05301823}, Cnv. \ref{05061835}, 
Def. \ref{06200822} and Def. \ref{11261911},
we have for any $\mf{D}\in\ms{C}_{u}(H)$ 
the following connection between notations
\begin{equation*}
\begin{aligned}
(\pf{T}_{\mf{D}},
\ms{I}_{\mf{D}},\upbeta_{c}^{\mf{D}},
\ms{P}_{\mf{D}},\mf{a}_{\mf{D}},
\mf{e}_{\mf{D}})
&=
(\pf{T}_{\mf{G}^{H}(\mf{D})},
\ms{I}_{\mf{G}^{H}(\mf{D})},\upbeta_{c}^{\mf{G}^{H}(\mf{D})},
\ms{P}_{\mf{G}^{H}(\mf{D})},\mf{a}_{\mf{G}^{H}(\mf{D})},
\mf{e}_{\mf{G}^{H}(\mf{D})})
\\
(\ps{\upvarphi}^{\mf{D}},\ms{A}_{\mf{D}},
\uppsi^{\mf{D}},\mf{b}^{\mf{D}},\ms{m}^{\mf{D}},
\mf{E}^{\mf{D}})
&=
(\ps{\upvarphi}^{\mf{G}^{H}(\mf{D})},\ms{A}_{\mf{G}^{H}(\mf{D})},
\uppsi^{\mf{G}^{H}(\mf{D})},\mf{b}^{\mf{G}^{H}(\mf{D})},\ms{m}^{\mf{G}^{H}(\mf{D})},
\mf{E}^{\mf{G}^{H}(\mf{D})}),
\\
(\upmu^{\mf{D}},\mf{u}^{\mf{D}},
\pf{H}^{\mf{D}},\ms{D}^{\mf{D}})
&=
(\upmu^{\mf{G}^{H}(\mf{D})},\mf{u}^{\mf{G}^{H}(\mf{D})},
\pf{H}^{\mf{G}^{H}(\mf{D})},
\ms{D}^{\mf{G}^{H}(\mf{D})})
\\
(\Upgamma^{\mf{D}},\mf{v}^{\mf{D}},\mf{w}^{\mf{D}},\mf{z}^{\mf{D}})
&=
(\Upgamma^{\mf{G}^{H}(\mf{D})},\mf{v}^{\mf{G}^{H}(\mf{D})},
\mf{w}^{\mf{G}^{H}(\mf{D})},\mf{z}^{\mf{G}^{H}(\mf{D})}),
\\
(\mf{H}^{\mf{D}},\uppi^{\mf{D}},\Upomega^{\mf{D}})
&=
(\mf{H}^{\mf{G}^{H}(\mf{D})},\uppi^{\mf{G}^{H}(\mf{D})},\Upomega^{\mf{G}^{H}(\mf{D})}).
\end{aligned}
\end{equation*}
Moreover
for all
$\mf{T}\in\pf{T}_{\mf{D}}$
and
$\alpha\in\ms{P}_{\mf{D}}^{\mf{T}}$
\begin{equation*}
\begin{aligned}
\ms{S}_{\alpha}^{\mf{T}}(\mf{D})
&=
\ms{S}_{\alpha}^{\mf{T}}(\mf{G}^{H}(\mf{D})),
\\
\mc{B}_{\alpha}^{\mf{T}}(\mf{D})
&=
\mc{B}_{\alpha}^{\mf{T}}(\mf{G}^{H}(\mf{D})),
\\
\mf{R}_{\alpha}^{\mf{T}}(\mf{D})
&=
\mf{R}_{\alpha}^{\mf{T}}(\mf{G}^{H}(\mf{D})),
\\
\pf{R}_{\alpha}^{\mf{T}}(\mf{D})
&=
\pf{R}_{\alpha}^{\mf{T}}(\mf{G}^{H}(\mf{D})),
\\
\lr{\cdot}{\cdot}_{(\mf{D},\mf{T},\alpha)}
&=
\lr{\cdot}{\cdot}_{(\mf{G}^{H}(\mf{D}),\mf{T},\alpha)}.
\end{aligned}
\end{equation*}
Finally 
$\mf{i}_{\alpha}^{\mf{T}}$ and $\mf{j}_{\alpha}^{\mf{T}}$
in 
Def. \ref{05301823} 
applied to $\mc{G}=\mf{G}^{H}(\mf{D})$
equal
$\mf{i}_{\alpha}^{\mf{T}}$ and $\mf{j}_{\alpha}^{\mf{T}}$
in 
Def. \ref{06200822}
applied to $\mf{A}=\mf{D}$
respectively.
\end{remark}
According to 
Def. \ref{12161311}, Def. \ref{01151659}, Def. \ref{12291116},
Def. \ref{12192100}, Lemma \ref{12201040} 
and Def. \ref{01151759}
we set
\begin{convention}
\label{06301240}
Let $\mf{D}\in\ms{C}_{u}(H)$ and $\mf{B}\in\ms{C}_{u}^{0}(H)$
set 
\begin{equation*}
\begin{aligned}
\mc{A}(\mf{D})_{\alpha}^{\mf{T}}
&\coloneqq
\mc{A}(\mf{G}^{H}(\mf{D}))_{\alpha}^{\mf{T}},
\forall
\mf{T}\in\pf{T}_{\mf{D}},
\alpha\in\ms{P}_{\mf{D}}^{\mf{T}},
\\
\mf{P}_{\pf{T}_{\bullet}}^{\mf{D}}
&\coloneqq 
\mf{P}_{\pf{T}_{\bullet}}^{\mf{G}^{H}(\mf{D})},
\\
\mf{O}^{\mf{D}}
&\coloneqq 
\mf{O}^{\mf{G}^{H}(\mf{D})},
\\
Rep^{\mf{D}}
&\coloneqq 
Rep^{\mf{G}^{H}(\mf{D})},
\\
\pf{V}(\mf{D})
&\coloneqq
\pf{V}_{\bullet}(\mf{G}^{H}(\mf{D})),
\\
\mf{t}^{\mf{D}}
&\coloneqq 
\mf{t}^{\mf{G}^{H}(\mf{D})},
\\
\mf{N}^{\mf{D}}
&\coloneqq
\mf{N}^{\mf{G}^{H}(\mf{D})},
\\
\ms{e}^{\mf{D}}
&\coloneqq 
\ms{e}^{\mf{G}^{H}(\mf{D})},
\\
\mf{P}_{\pf{V}_{\bullet}}^{\mf{B}}
&\coloneqq 
\mf{P}_{\pf{V}_{\bullet}}^{\mf{G}^{H}(\mf{B})},
\\
\mf{Z}^{\mf{B}}
&\coloneqq 
\mf{Z}^{\mf{G}^{H}(\mf{B})},
\\
\mc{V}^{\mf{B}}
&\coloneqq
\mc{V}_{\bullet}(\mf{G}^{H}(\mf{B})).
\end{aligned}
\end{equation*}
\end{convention}
Clearly $\mc{A}(\mf{D})_{\alpha}^{\mf{T}}$ equals the $C^{\ast}-$algebra underlying $\mf{D}$.
\begin{definition}
\label{06242034}
\begin{equation*}
\begin{aligned}
\ms{C}_{u}^{0}(H)
&\coloneqq
\Upxi(\pf{V}_{\bullet},\mf{G}_{\vartriangle}^{H}),
\\
\tilde{\mf{G}}_{\vartriangle}^{H}
&\coloneqq
(\mf{G}_{\vartriangle}^{H})^{\pf{V}_{\bullet}}.
\end{aligned}
\end{equation*}
\end{definition}
\begin{proposition}
\label{19061753}
\begin{equation*}
\begin{aligned}
\tilde{\mf{G}}^{H}
&=
\mf{G}^{H}
\up Obj(\ms{C}_{u}^{0}(H))
\\
\tilde{\mf{G}}_{\vartriangle}^{H}
&=
\left(\tilde{\mf{G}}^{H},(\mf{g}^{H}\times\mf{d}^{H})
\up Mor_{\ms{C}_{u}^{0}(H)}\right),
\end{aligned}
\end{equation*}
moreover 
\begin{equation*}
Obj(\ms{C}_{u}^{0}(H))
=\{\mf{A}\in\ms{C}_{u}(H)
\mid
\pf{V}(\mf{A})\neq\varnothing\},
\end{equation*}
and for any $\mf{A},\mf{B}\in\ms{C}_{u}^{0}(H)$
we have 
\begin{equation*}
Mor_{\ms{C}_{u}^{0}(H)}(\mf{B},\mf{A})
=
\{T\in Mor_{\ms{C}_{u}(H)}(\mf{B},\mf{A})
\mid
(\forall\mf{T}\in\pf{V}(\mf{A}))
(\mf{d}^{H}(T)(\mf{T})\in\pf{V}(\mf{B}))
\}.
\end{equation*}
\end{proposition}
\begin{definition}
\label{21051632}
Set
$\mc{E}_{\bullet}^{H}
\coloneqq
\lr{\lr{\pf{T}_{\bullet},\ov{\ms{m}}_{\bullet}}{\mc{V}_{\bullet}}}{\mf{G}_{\vartriangle}^{H}}$.
\end{definition}
Whenever it is clear by the context 
which group $H$ is involved we use the convention
to denote $\mc{E}_{\bullet}^{H}$ by $\mc{E}_{\bullet}$.
Now we can state
\begin{theorem}
\label{12051936}
\begin{enumerate}
\item
$\mc{E}_{\bullet}$
is a full integer $\ms{C}_{u}(H)-$equivariant stability on $\pf{V}_{\bullet}$;
\label{12051936st1}
\item
let $\mf{A}\in\ms{C}_{u}(H)$, $\mf{T}\in\pf{T}_{\mf{A}}$ and $\alpha\in\ms{P}_{\mf{A}}^{\mf{T}}$,
thus
\begin{enumerate}
\item
for all $\ms{f}\in\ms{A}_{\mf{A}}$
\begin{equation*}
\ov{\ms{m}}^{\mf{A}}(\mf{T},\alpha)(\ms{f})
=
\lr{\ms{f}(\mf{T})(\alpha)}
{\ms{ch}(\ps{R}(\mf{T},\alpha))}_{(\mf{A},\mf{T},\alpha)};
\end{equation*}
\label{05061911}
\item
if $\mf{A}\in\ms{C}_{u}^{0}(H)$ and $\mf{T}\in\pf{V}(\mf{A})$ then
\begin{equation*}
\mc{V}^{\mf{A}}(\mf{T},\alpha)
=
\omega_{\exp(-((\ms{D}^{\mf{A}})_{\alpha}^{\mf{T}})^{2})}
\circ
(\uppi^{\mf{A}})_{\alpha}^{\mf{T}},
\end{equation*}
\label{12051936st3}
\item
let $\mf{B}\in\ms{C}_{u}(H)$ and $T\in Mor_{\ms{C}_{u}(H)}(\mf{B},\mf{A})$. 
If 
$\mf{A}\in\ms{C}_{u}^{0}(H)$, $\mf{T}\in\pf{V}(\mf{A})$ and 
\begin{enumerate}
\item
if 
$\mf{B}\in\ms{C}_{u}^{0}(H)$ and $\mf{d}^{H}(T)(\mf{T})\in\pf{V}(\mf{B})$, 
then 
\begin{equation*}
\mc{V}^{\mf{B}}\left(\mf{d}^{H}(T)(\mf{T}),\alpha\right)
=
\mc{V}^{\mf{A}}(\mf{T},\alpha)\circ T,
\end{equation*}
\label{11061441}
moreover 
for all $l\in H$ and $a\in\mc{B}$ 
\begin{equation*}
\mc{V}^{\mf{B}}\left(\mf{b}^{\mf{B}}(l)(\mf{d}^{H}(T)(\mf{T})),\alpha\right)(\uptheta(l)(a))
=
\mc{V}^{\mf{A}}(\mf{T},\alpha)(T(a));
\end{equation*}
\item
if for all $\beta\in\ms{P}_{\mf{A}}^{\mf{T}}$
there exist
$b\in\mc{B}_{\beta}^{\mf{T}}(\mf{A})^{+}$ 
and
$\tilde{b}\in\mc{B}_{\beta}^{\mf{d}^{H}(T)(\mf{T})}(\mf{B})^{+}$ 
such that 
$\tilde{\mf{R}}_{\beta}^{\mf{T}}(\mf{A})(b)=(\Upgamma^{\mf{A}})_{\beta}^{\mf{T}}$,
$\|\tilde{b}\|\leq 1$ and 
$b=\ms{c}_{(\upmu^{\mf{A}})_{\beta}^{\mf{T}}}^{+}(T)(\tilde{b})$,
then 
$\mf{d}^{H}(T)(\mf{T})\in\pf{V}(\mf{B})$.
\label{12051936st4}
\end{enumerate}
\end{enumerate}
\end{enumerate}
\end{theorem}
\begin{proof}
St.\eqref{12051936st1} follows since Thm. \ref{01151104} and Thm. \ref{12071120}, 
st.\eqref{05061911} follows since the construction of $\mf{G}^{H}$, while st.\eqref{12051936st3} since \eqref{14051532}.
The first equality in
st.\eqref{11061441} follows since st.\eqref{12051936st3} applied to $\mf{B}$ and $\mf{A}$, and by \eqref{11302054}
switching $\mf{A}$ with $\mf{B}$. 
The second equality follows since the first one,
st.\eqref{12051936st1}
and Prp. \ref{18061742}\eqref{18061742st3}. 
St.\eqref{12051936st4} follows since \eqref{11301824} 
switching $\mf{A}$ with $\mf{B}$ and by Lemma  \ref{10311019}.
\end{proof}
\begin{remark}
\label{31072045}
Under the conditions in Thm. \ref{12051936}\eqref{12051936st4} 
we deduce that $\mf{B}\in\ms{C}_{u}^{0}(H)$,
moreover if these conditions are true for all $\mf{T}\in\pf{V}(\mf{A})$
then we deduce that 
$T\in Mor_{\ms{C}_{u}^{0}(H)}(\mf{B},\mf{A})$.
\end{remark}
According to the definition of $\ms{Z}$ 
in Def. \ref{20051117} we can set
\begin{definition}
\label{19061841}
Define $\nabla\coloneqq(\nabla_{o},\nabla_{m})$ and $\mc{X}_{H}$
such that 
\begin{equation*}
\begin{aligned}
\nabla_{o}&\coloneqq\ms{Z}\circ\tilde{\mf{G}}^{H},
\\
\nabla_{m}
&\in\prod_{T\in Mor_{\ms{C}_{u}^{0}(H)}}
Mor_{\ms{set}}(\nabla_{o}(c(T)),\nabla_{o}(d(T)))
\\
\nabla_{m}(T)(\mf{T},f)
&\coloneqq
(\mf{d}^{H}(T)\mf{T},\alpha\mapsto f(\alpha)\circ T),
\\
\forall T&\in Mor_{\ms{C}_{u}^{0}(H)},
\\
\mc{X}_{H}
&\in
\prod_{T\in Mor_{\ms{C}_{u}^{0}(H)}}
\prod_{\mf{Q}\in\pf{T}_{c(T)}}
\prod_{\beta\in\ms{P}_{c(T)}^{\mf{Q}}}
Mor_{\ms{CA}^{\ast}}(\mc{A}(d(T))_{\beta}^{\mf{d}^{H}(T)\mf{Q}},\mc{A}(c(T))_{\beta}^{\mf{Q}}),
\\
\mc{X}_{H}(T,\mf{Q},\beta)
&
\coloneqq T,
\\
\forall
T&\in Mor_{\ms{C}_{u}^{0}(H)},
\mf{Q}\in\pf{T}_{c(T)},
\beta\in\ms{P}_{c(T)}^{\mf{Q}}.
\end{aligned}
\end{equation*}
\end{definition}
\begin{proposition}
\label{19061754}
We have that
\begin{enumerate}
\item
$\nabla\in\ms{Fct}(\ms{C}_{u}^{0}(H)^{op},\ms{set})$,
\label{19061754st2}
\item
$\lr{\nabla}{\mc{X}_{H}}$ 
is a conjugate action via $\tilde{\mf{G}}_{\vartriangle}^{H}$.
\label{06252015}
\end{enumerate}
\end{proposition}
\begin{proof}
St.\eqref{19061754st2} follows since Lemma \ref{11271700}, 
st.\eqref{06252015} follows since st.\eqref{19061754st2}
and the construction of $\tilde{\mf{G}}_{\vartriangle}^{H}$.
\end{proof}
Now we are in the position of stating our
\begin{main}
[canonical nucleon-fragment doublet on $\ms{C}_{u}(H)$]
\label{06242121}
$\mc{E}_{\bullet}$
is an extended full integer $\ms{C}_{u}(H)-$equivariant stability on $\pf{V}_{\bullet}$
via $\nabla$ and $\mc{X}_{H}$.
In particular there exists a nucleon-fragment doublet on $\ms{C}_{u}(H)$
namely the nucleon-fragment doublet $\mc{T}_{\bullet}$ related to $\mc{E}_{\bullet}$. 
\end{main}
\begin{proof}
Since
Thm. \ref{12051936}\eqref{12051936st1},
Prp. \ref{19061754}\eqref{06252015},
Rmk. \ref{06151915}
and
Thm. \ref{12051936}\eqref{11061441}.
\end{proof}
\begin{definition}
Let $\mc{E}_{\bullet}$ be called the canonical extended $\ms{C}_{u}(H)-$equivariant stability,
and let $\mc{T}_{\bullet}$ be called the canonical nucleon-fragment doublet on $\ms{C}_{u}(H)$.
\end{definition}
\begin{definition}
\label{07281627}
Let $\mf{I}_{\bullet}$ denote the field determined by $\mc{V}_{\bullet}$,
set
$J_{\bullet}\coloneqq(\mf{I}_{\bullet}\circ\tilde{\mf{G}}^{H},\nabla_{m})$.
\end{definition}
As a consequence of the main theorem we obtain
the properties of equivariance and invariance of
$\mc{T}_{\bullet}$
\begin{corollary}
\label{06242107}
Prp. \ref{06211545}, Prp. \ref{20051810dbt} and Prp. \ref{12051206dbt}
hold true for $\mc{T}=\mc{T}_{\bullet}$.
\end{corollary}
In particular the next result
resolves the equivariant form of the universality claim 
described in introduction \ref{introI}.
\begin{corollary}
[Equivariant form of the universality claim]
\label{19061936}
We have
\begin{enumerate}
\item
$\ov{\ms{m}}\circ\mf{G}^{H}
\in
 Mor_{\ms{Fct}(\ms{C}_{u}(H)^{op},\ms{set})}
(\mf{Q}^{\pf{T}_{\bullet}}
\circ
\mf{G}_{\vartriangle}^{H},
\Updelta^{\pf{T}_{\bullet}}
\circ
\mf{G}_{\vartriangle}^{H})$;
\label{19061936st1}
\item
$(\un\mapsto\ov{\ms{m}}^{\mf{A}})
\in
 Mor_{\ms{Fct}(H,\ms{set})}
(\mf{P}_{\pf{T}_{\bullet}}^{\mf{A}},
\mf{O}^{\mf{A}})$,
for all
$\mf{A}\in\ms{C}_{u}(H)$;
\label{19061936st2}
\item
$\ms{gr}\circ\mc{V}_{\bullet}\circ\tilde{\mf{G}}^{H}
\in
 Mor_{\ms{Fct}(\ms{C}_{u}^{0}(H)^{op},\ms{set})}
(\mf{Q}^{\pf{V}_{\bullet}}
\circ
\tilde{\mf{G}}_{\vartriangle}^{H},J_{\bullet})$;
\label{19061936st4}
\item
$(\un\mapsto\ms{gr}(\mc{V}^{\mf{B}}))
\in
 Mor_{\ms{Fct}(H,\ms{set})}
(\mf{P}_{\pf{V}_{\bullet}}^{\mf{B}},\mf{Z}^{\mf{B}})$,
for all
$\mf{B}\in\ms{C}_{u}^{0}(H)$;
\label{19061936st3}
\item
\begin{equation*}
\mc{V}_{\bullet}\circ\tilde{\mf{G}}^{H}
\in
\prod_{\mf{D}\in\ms{C}_{u}^{0}(H)}
\prod_{\mf{Q}\in\pf{V}(\mf{D})}
\prod_{\beta\in\ms{P}_{\pf{D}}^{\mf{Q}}}
\mf{N}^{\mf{D}}(\ov{\ms{m}}^{\mf{D}}(\mf{Q},\beta))
\cap
\ms{N}_{\mc{A}(\mf{D})_{\beta}^{\mf{Q}}},
\end{equation*}
and by setting 
$\mf{o}^{\mf{D}}=\ms{e}^{\mf{D}}\circ\mf{t}^{\mf{D}}$
for all $\mf{D}\in\ms{C}_{u}^{0}(H)$,
we obtain
\begin{equation*}
\begin{aligned}
\mf{o}
&\in
\prod_{\mf{D}\in\ms{C}_{u}^{0}(H)}
\prod_{\mf{Q}\in\pf{V}(\mf{D})}
\prod_{\beta\in\ms{P}_{\pf{D}}^{\mf{Q}}}
Rep^{\mf{D}}(\ov{\ms{m}}^{\mf{D}}(\mf{Q},\beta))
\\
\mc{V}^{\mf{D}}(\mf{Q},\beta)
&=
\Uppsi_{\mf{o}^{\mf{D}}(\mf{Q},\beta)}^{-}
\circ
\mf{j}_{\mf{o}^{\mf{D}}(\mf{Q},\beta)},
\\
\mf{T}_{\mf{o}^{\mf{D}}(\mf{Q},\beta)}
&=
\mf{Q},
\\
\alpha_{\mf{o}^{\mf{D}}(\mf{Q},\beta)}
&=
\beta,
\\
\forall
\mf{D}&\in\ms{C}_{u}^{0}(H),
\mf{Q}\in\pf{V}(\mf{D}),
\beta\in\ms{P}_{\pf{D}}^{\mf{Q}}.
\end{aligned}
\end{equation*}
\label{19061936st5}
\end{enumerate}
\end{corollary}
\begin{proof}
Since Thm. \ref{06242121}, Cor. \ref{06241610} and Def. \ref{06161650}.
\end{proof}
In part \ref{07301107} 
we shall establish the compact equivariant and invariant forms of the universality claim.
\part{Universality}
\label{07301107}
\section{Introduction}
\label{introIII}
In this part we complete our task started in part \ref{07301106}
of resolving in diverse formulations the universality claim 
described in introduction \ref{introI},
here we provide the compact equivariant and invariant forms.
\par
The part is organized as follows.
We start section \ref{21071150} by displaying
the solution of the equivariant form of the universality claim,
i.e. the expanded equivariances of the $\mc{T}_{\bullet}-$nucleon phase and $\mc{T}_{\bullet}-$fragment state
established at the end of 
part \ref{07301106}.
Then in the second main result of the entire work,
we resolve the compact equivariant form of the claim
by encoding the above properties into two natural transformations.
Namely $\mf{m}_{\star}$ between functors from $\ms{C}_{u}(H)^{op}$ to $\ms{Fct}(H,\ms{set})$,
and $\mf{v}_{\natural}$ between functors from $\ms{C}_{u}^{0}(H)^{op}$ to $\ms{Fct}(H,\ms{set})$,
provided that the hypothesis $\ms{E}$
ensuring the functoriality of the source and target of $\mf{v}_{\natural}$
holds true.
In section \ref{07011744} we associate nucleon and fragment masses maps $\upzeta_{j}^{\mc{T}}$
and $\upkappa_{j}^{\mc{T}}$,
and a Terrell-like law $\upnu^{\mc{T}}$ with any 
nucleon-fragment doublet $\mc{T}$ on a category $\mf{C}$, and show how 
the invariance of these maps follows by the
equivariances of the $\mc{T}-$nucleon phase and $\mc{T}-$fragment state.
In section \ref{07170749} we establish
the $\mc{T}-$resolution of the invariant form of the universality claim,
representing an equivariant formulation of the result of the previous section.
Indeed we introduce the $\mc{T}-$nucleon and fragment masses $\muup_{j}^{\mc{T}}$ 
and $\uplambda_{j}^{\mc{T}}$,
and the $\mc{T}-$Terrell law $\uptheta^{\mc{T}}$ as set-valued extensions of 
$\upzeta_{j}^{\mc{T}}$, $\upkappa_{j}^{\mc{T}}$ and $\upnu^{\mc{T}}$ respectively.
Then we establish their universality by showing that 
$\muup_{j}^{\mc{T}}$, $\uplambda_{j}^{\mc{T}}$ and $\uptheta^{\mc{T}}$ 
are invariant under contravariant action of $\Upxi(\pf{D},\mc{F})$ and under action of $H$.
In section \ref{25061121} we apply the results of section \ref{07170749}
when $\mc{T}$ is the canonical nucleon-fragment doublet on $\ms{C}_{u}^{0}(H)$.
As a result we obtain, in the third main result of the entire work,
the univarsality of the global nucleon and fragment masses and the universality of the global Terrell law,
thus resolving the invariant form of the universality claim.
\section{Natural transformations related to $\mc{T}_{\bullet}$}
\label{21071150}
In 
Cor. \ref{19061936} 
we stated 
the expanded equivariances of the $\mc{T}_{\bullet}-$nucleon phase and $\mc{T}_{\bullet}-$fragment state,
resolving the equivariant form of the universality claim described in introduction \ref{introI}.
In Thm. \ref{10071930} the main result of this section and the second main result of the entire work,
we shall encode these properties in a unique fashion.
Shortly we show, modulo a suitable equivalence relation, 
that the maps $\ov{\ms{m}}$ and $\mc{V}_{\bullet}$ realize natural transformations between 
functors from the category $\ms{C}_{u}(H)^{op}$ to the category $\ms{Fct}(H,\ms{set})$
and from $\ms{C}_{u}^{0}(H)^{op}$ to $\ms{Fct}(H,\ms{set})$ respectively.
In the present section 
we assume fixed two locally compact topological groups $G$ and $F$,
a group homomorphism $\uprho:F\to Aut_{\ms{Gr}}(G)$  such that the map $(g,f)\mapsto\uprho_{f}(g)$
on $G\times F$ at values in $G$, is continuous, let $H$ denote $G\rtimes_{\uprho}F$.
For the sake of completeness let us restate 
Cor. \ref{19061936} 
\begin{enumerate}
\item
$\ov{\ms{m}}\circ\mf{G}^{H}
\in
 Mor_{\ms{Fct}(\ms{C}_{u}(H)^{op},\ms{set})}
(\mf{Q}^{\pf{T}_{\bullet}}
\circ
\mf{G}_{\vartriangle}^{H},
\Updelta^{\pf{T}_{\bullet}}
\circ
\mf{G}_{\vartriangle}^{H})$;
\item
$(\un\mapsto\ov{\ms{m}}^{\mf{A}})
\in
 Mor_{\ms{Fct}(H,\ms{set})}
(\mf{P}_{\pf{T}_{\bullet}}^{\mf{A}},
\mf{O}^{\mf{A}})$,
for all
$\mf{A}\in\ms{C}_{u}(H)$;
\item
$\ms{gr}\circ\mc{V}_{\bullet}\circ\tilde{\mf{G}}^{H}
\in
 Mor_{\ms{Fct}(\ms{C}_{u}^{0}(H)^{op},\ms{set})}
(\mf{Q}^{\pf{V}_{\bullet}}
\circ
\tilde{\mf{G}}_{\vartriangle}^{H},J_{\bullet})$;
\item
$(\un\mapsto\ms{gr}(\mc{V}^{\mf{B}}))
\in
 Mor_{\ms{Fct}(H,\ms{set})}
(\mf{P}_{\pf{V}_{\bullet}}^{\mf{B}},\mf{Z}^{\mf{B}})$,
for all
$\mf{B}\in\ms{C}_{u}^{0}(H)$;
\item
\begin{equation*}
\mc{V}_{\bullet}\circ\tilde{\mf{G}}^{H}
\in
\prod_{\mf{D}\in\ms{C}_{u}^{0}(H)}
\prod_{\mf{Q}\in\pf{V}(\mf{D})}
\prod_{\beta\in\ms{P}_{\pf{D}}^{\mf{Q}}}
\mf{N}^{\mf{D}}(\ov{\ms{m}}^{\mf{D}}(\mf{Q},\beta))
\cap
\ms{N}_{\mc{A}(\mf{D})_{\beta}^{\mf{Q}}},
\end{equation*}
and by setting 
$\mf{o}^{\mf{D}}=\ms{e}^{\mf{D}}\circ\mf{t}^{\mf{D}}$
for all $\mf{D}\in\ms{C}_{u}^{0}(H)$,
we obtain
\begin{equation*}
\begin{aligned}
\mf{o}
&\in
\prod_{\mf{D}\in\ms{C}_{u}^{0}(H)}
\prod_{\mf{Q}\in\pf{V}(\mf{D})}
\prod_{\beta\in\ms{P}_{\pf{D}}^{\mf{Q}}}
Rep^{\mf{D}}(\ov{\ms{m}}^{\mf{D}}(\mf{Q},\beta))
\\
\mc{V}^{\mf{D}}(\mf{Q},\beta)
&=
\Uppsi_{\mf{o}^{\mf{D}}(\mf{Q},\beta)}^{-}
\circ
\mf{j}_{\mf{o}^{\mf{D}}(\mf{Q},\beta)},
\\
\mf{T}_{\mf{o}^{\mf{D}}(\mf{Q},\beta)}
&=
\mf{Q},
\\
\alpha_{\mf{o}^{\mf{D}}(\mf{Q},\beta)}
&=
\beta,
\\
\forall
\mf{D}&\in\ms{C}_{u}^{0}(H),
\mf{Q}\in\pf{V}(\mf{D}),
\beta\in\ms{P}_{\pf{D}}^{\mf{Q}}.
\end{aligned}
\end{equation*}
\end{enumerate}
Thm. \ref{10071930} organizes in a more concise and elegant form the above properties,
but at cost of rearranging the functors modulo a suitable equivalence relation on a subset of $\pf{T}_{\mf{A}}$.
More exactly, since 
$\mf{A}\mapsto\mf{P}_{\pf{T}_{\bullet}}^{\mf{A}}$ 
and 
$\mf{A}\mapsto\mf{O}^{\mf{A}}$ 
are maps from $Obj(\ms{C}_{u}(H))$ 
to $Obj(\ms{Fct}(H,\ms{set}))$ 
it is natural to ask if we can arrange them to form the object part of two functors, 
say $\ms{M}$ and $\ms{N}$, 
from the category
$\ms{C}_{u}(H)^{op}$ 
to the category $\ms{Fct}(H,\ms{set})$, 
and then to verify if the following claim holds true:
$\mf{A}\mapsto(\un\mapsto\ov{\ms{m}}^{\mf{A}})$
realizes a natural transformation between $\ms{M}$ and $\ms{N}$.
Similarly since 
$\mf{B}\mapsto\mf{P}_{\pf{V}_{\bullet}}^{\mf{B}}$ and 
$\mf{B}\mapsto\mf{Z}^{\mf{B}}$
are maps from $Obj(\ms{C}_{u}^{0}(H))$ to $Obj(\ms{Fct}(H,\ms{set}))$) 
it is natural to ask if we can arrange them to form the object part of two functors, 
say $\ms{M}'$ and $\ms{N}'$
from the category $\ms{C}_{u}^{0}(H)^{op}$
to the category $\ms{Fct}(H,\ms{set})$, 
and then to verify if the following claim holds true:
$\mf{B}\mapsto(\un\mapsto\ms{gr}(\mc{V}^{\mf{B}}))$
realizes a natural transformation between the functors $\ms{M}'$ and $\ms{N}'$.
Now the request of the existence of the aforementioned functor $\ms{M}$ is equivalent to require that
for all $\mf{A},\mf{B}\in\ms{C}_{u}(H)$, $T\in Mor_{\ms{C}_{u}(H)}(\mf{B},\mf{A})$, 
$l\in H$ and $\mf{T}\in\pf{T}_{\mf{A}}$
\begin{equation*}
(\mf{b}^{\mf{B}}(l) 
\circ
\mf{d}^{H}(T))(\mf{T})
=
(\mf{d}^{H}(T)
\circ
\mf{b}^{\mf{A}}(l))(\mf{T}),
\end{equation*}
while we shall prove in
Lemma \ref{07071822}\eqref{07071822st1}
that 
for all $\mf{T}\in\pf{T}_{\mf{A}}^{\diamond}$
\begin{equation*}
(\mf{b}^{\mf{B}}(l) 
\circ
\mf{d}^{H}(T))(\mf{T})
\af{B}
(\mf{d}^{H}(T)
\circ
\mf{b}^{\mf{A}}(l))(\mf{T}),
\end{equation*}
where $\af{B}$ is a suitable equivalence relation on a subset $\pf{T}_{\mf{B}}^{\diamond}$ of $\pf{T}_{\mf{B}}$.
Hence it is clear that in order to prove our 
claim we need to pass in a convenient sense to the quotient all the involved functors w.r.t.
the relations $\af{A}$'s. The construction culminate in Def. \ref{10071700}, while in Cor. \ref{10071703} 
we prove that the constructed structures realize functors from the categories $\ms{C}_{u}(H)^{op}$ and 
$\ms{C}_{u}^{0}(H)^{op}$ to the category $\ms{Fct}(H,\ms{set})$.
Finally we succeed in proving our claim by stating in Thm. \ref{10071930} that
$\mf{A}\mapsto(\un\mapsto\ov{\ms{m}}_{\star}^{\mf{A}})$ 
and $\mf{B}\mapsto(\un\mapsto\ms{gr}(\mc{V}_{\natural}^{\mf{B}}))$
are natural transformations between the 
constructed functors where $\ov{\ms{m}}_{\star}$ and $\mc{V}_{\natural}$ are $\ov{\ms{m}}$ and 
$\mc{V}$ after passing in a convenient sense to the quotient w.r.t. the respective equivalence relations.
\begin{convention}
In the remaining of this section let
$\mf{A}=\lr{\mc{A},H}{\upeta}$ 
such that $\mf{A}\in\ms{C}(H)$,
while by starting from Def. \ref{05071126}, 
let
$\mf{B}=\lr{\mc{B},H}{\uptheta}$ 
and assume 
that $\mf{A},\mf{B}\in\ms{C}_{u}(H)$.
If
$\mf{T},\mf{Q}\in\pf{T}_{\mf{A}}$
we convein to use the following notation
whenever it does not cause confusion
$\mf{T}=\lr{\mc{T},\ps{\upmu},\pf{H}}{\upzeta,f,\Upgamma}$
and
$\mf{Q}=\lr{\mc{T}',\ps{\upmu}',\pf{K}}{\upzeta',f',\Updelta}$,
where
$\mc{T}=\lr{h,\upxi,\beta_{c},I}{\ps{\upomega}}$
and
$\mc{T}'=\lr{h',\upxi',\beta_{c}',I'}{\ps{\upomega}'}$,
$X$ and $X'$
are the sets such that $\R^{X}$ and $\R^{X'}$ 
are the domains of the maps $\upzeta$ and $\upzeta'$
respectively,
while
$\pf{H}_{\alpha}=
\lr{\mf{H}_{\alpha},\uppi_{\alpha}}{\Upomega_{\alpha}}$ 
and
$\pf{K}_{\beta}=
\lr{\mf{K}_{\beta},\upupsilon_{\beta}}{\Uppsi_{\beta}}$ 
for any $\alpha\in\ms{P}_{\mf{A}}^{\mf{T}}$ 
and $\beta\in\ms{P}_{\mf{A}}^{\mf{Q}}$.
\end{convention}
\begin{definition}
\label{12071151}
Let $\mathscr{A}$ and $\mathbb{A}$ be the maps on $Obj(\ms{C}(H))$ such that 
$\mathscr{A}(\mf{A}),\mathbb{A}(\mf{A})\in\prod_{\mf{I}\in\pf{T}_{\mf{A}}}
\prod_{\alpha\in\ms{P}_{\mf{A}}^{\mf{I}}}\mathscr{P}(\mc{L}(\mf{H}_{\alpha}))$ 
and for all 
$\mf{T}\in\pf{T}_{\mf{A}}$ and $\alpha\in\ms{P}_{\mf{A}}^{\mf{T}}$ we have
\begin{equation*}
\mathscr{A}(\mf{A})_{\alpha}^{\mf{T}}
\coloneqq
\uppi_{\alpha}(\mc{A})\cup\ms{U}_{\pf{H}_{\alpha}}(H),
\qquad
\mathbb{A}(\mf{A})_{\alpha}^{\mf{T}}
\coloneqq(\mathscr{A}(\mf{A})_{\alpha}^{\mf{T}})''.
\end{equation*}
\end{definition}
\begin{definition}
\label{05071015}
Let $\sm{A}$ be the relation 
such that 
\begin{multline*}
\sm{A}
\coloneqq
\bigl\{
(\mf{T},\mf{Q})\in\pf{T}_{\mf{A}}\times\pf{T}_{\mf{A}}
\mid
(\mc{T},\ps{\upmu},\upzeta,f)
=
(\mc{T}',\ps{\upmu}',\upzeta',f')\wedge
(\forall\alpha\in\ms{P}_{\mf{A}}^{\mf{T}})\\
(\exists\,V_{\alpha}:\mf{H}_{\alpha}\to\mf{K}_{\alpha}\text{ unitary})
\left(
\upupsilon_{\alpha}=\ms{ad}(V_{\alpha})\circ\uppi_{\alpha},\,
\Uppsi_{\alpha}=V_{\alpha}\Upomega_{\alpha},\,
\Updelta_{\alpha}=\ms{ad}(V_{\alpha})(\Upgamma_{\alpha})
\right)
\bigr\}.
\end{multline*}
Set $[\mf{T}]_{\sm{A}}
\coloneqq\{\mf{Q}\mid\mf{T}\sm{A}\mf{Q}\}$
and 
$\pf{T}_{\mf{A}}^{\dv}\coloneqq
\{[\mf{T}]_{\sm{A}}\mid\mf{T}\in\pf{T}_{\mf{A}}\}$.
\end{definition}
\begin{remark}
\label{13071106}
Let $\mf{T}\in\pf{T}_{\mf{A}}$ and $\alpha\in\ms{P}_{\mf{A}}^{\mf{T}}$ then since the bicommutant theorem 
$\mathbb{A}(\mf{A})_{\alpha}^{\mf{T}}$ 
is the von Neumann algebra generated by the set $\mathscr{A}(\mf{A})_{\alpha}^{\mf{T}}$.
Thus $\mf{T}\in\pf{V}(\mf{A})$ implies $\Upgamma_{\alpha}\in\mathbb{A}(\mf{A})_{\alpha}^{\mf{T}}$,
since the integration in 
\eqref{05171307} 
is w.r.t. the strong operator topology. Moreover the position $f=f'$ is well-set since 
agrees with \eqref{06071411}.
\end{remark}
\begin{definition}
\label{09070857}
Define 
\begin{equation*}
\begin{aligned}
\pf{T}_{\mf{A}}^{\diamond}
&\coloneqq
\{\mf{T}\in\pf{T}_{\mf{A}}
\mid
(\forall\alpha\in\ms{P}_{\mf{A}}^{\mf{T}})
(\Upgamma_{\alpha}\in\uppi_{\alpha}(\mc{A})'')
\},
\\
\af{A}
&\coloneqq
(\sm{A})
\cap
(\pf{T}_{\mf{A}}^{\diamond}\times\pf{T}_{\mf{A}}^{\diamond}),
\\
[\mf{T}]_{\af{A}}
&\coloneqq
\{\mf{Q}\mid\mf{T}\af{A}\mf{Q}\},\quad
\mf{T}\in\pf{T}_{\mf{A}}^{\diamond},
\\
\pf{T}_{\mf{A}}^{\star}
&\coloneqq
\{
[\mf{T}]_{\af{A}}
\mid
\mf{T}\in\pf{T}_{\mf{A}}^{\diamond}
\},
\end{aligned}
\end{equation*} 
moreover
\begin{equation*}
\begin{aligned}
\pf{T}_{\mf{A}}^{\vc}
&\coloneqq
\{\mf{T}\in\pf{T}_{\mf{A}}
\mid
(\forall\alpha\in\ms{P}_{\mf{A}}^{\mf{T}})
(\Upgamma_{\alpha}\in\mathbb{A}(\mf{A})_{\alpha}^{\mf{T}})
\},
\\
\dm{A}
&\coloneqq
(\sm{A})
\cap
(\pf{T}_{\mf{A}}^{\vc}\times\pf{T}_{\mf{A}}^{\vc}),
\\
[\mf{T}]_{\dm{A}}
&\coloneqq
\{\mf{Q}\mid\mf{T}\dm{A}\mf{Q}\},\quad
\mf{T}\in\pf{T}_{\mf{A}}^{\vc},
\\
\pf{T}_{\mf{A}}^{\vm}
&\coloneqq
\{
[\mf{T}]_{\dm{A}}
\mid
\mf{T}\in\pf{T}_{\mf{A}}^{\vc}
\},
\end{aligned}
\end{equation*} 
and if $\mf{A}\in\ms{C}_{u}^{0}(H)$ set
\begin{equation*}
\begin{aligned}
\tm{A}
&\coloneqq
(\sm{A})
\cap
(\pf{V}(\mf{A})\times\pf{V}(\mf{A})),
\\
[\mf{T}]_{\tm{A}}
&\coloneqq
\{\mf{Q}\mid\mf{T}\tm{A}\mf{Q}\},\quad
\mf{T}\in\pf{V}(\mf{A})
\\
\pf{V}_{\mf{A}}^{\natural}
&\coloneqq
\{
[\mf{T}]_{\tm{A}}\mid\mf{T}\in
\pf{V}(\mf{A})
\}.
\end{aligned}
\end{equation*} 
\end{definition}
\begin{lemma}
\label{12071157}
Let $\mf{T}\sm{A}\mf{Q}$ and $\alpha\in\ms{P}_{\mf{A}}^{\mf{T}}$
then
$\mathscr{A}(\mf{A})_{\alpha}^{\mf{Q}}
=\ms{ad}(V_{\alpha})(\mathscr{A}(\mf{A})_{\alpha}^{\mf{T}})$
and
$\mathbb{A}(\mf{A})_{\alpha}^{\mf{Q}}
=
\ms{ad}(V_{\alpha})(\mathbb{A}(\mf{A})_{\alpha}^{\mf{T}})$.
\end{lemma}
\begin{proof}
Let $\mf{T}\sm{A}\mf{Q}$ and $\alpha\in\ms{P}_{\mf{A}}^{\mf{T}}$.
Thus for any $h\in\s{\upomega}{\alpha}{}$
and $a\in\mc{A}$
\begin{equation*}
\begin{aligned}
\ms{U}_{\pf{K}_{\alpha}}(h)
\upupsilon_{\alpha}(a)\Uppsi_{\alpha}
&=\upupsilon_{\alpha}(\upeta(h)a)\Uppsi_{\alpha}
=V_{\alpha}\uppi_{\alpha}(\upeta(h)a)\Upomega_{\alpha}
\\
&=V_{\alpha}\ms{U}_{\pf{H}_{\alpha}}(h)
\uppi_{\alpha}(a)\Upomega_{\alpha}
=V_{\alpha}
\ms{U}_{\pf{H}_{\alpha}}(h)
V_{\alpha}^{\ast}
\upupsilon_{\alpha}(a)\Uppsi_{\alpha}
\\
&=
(\ms{ad}(V_{\alpha})\circ
\ms{U}_{\pf{H}_{\alpha}})(h)
\upupsilon_{\alpha}(a)\Uppsi_{\alpha},
\end{aligned}
\end{equation*}
so we obtain since the cyclicity of $\pf{K}_{\alpha}$ 
\begin{equation}
\label{05072220}
\ms{U}_{\pf{K}_{\alpha}}
=
\ms{ad}(V_{\alpha})
\circ
\ms{U}_{\pf{H}_{\alpha}},
\end{equation}
and the first equality of the statement follows. The second equality follows since the first one, since 
the bicommutant theorem and since the continuity of $\ms{ad}(W)$ w.r.t. 
the weak operator topology on $\mc{L}(\mf{H}_{\alpha})$
for any unitary operator $W$ on $\mf{H}_{\alpha}$.
\end{proof}
\begin{proposition}
\label{09070854}
$\sm{A}$, $\af{A}$, $\dm{A}$ and $\tm{A}$
are equivalence relations, 
moreover 
$[\mf{T}]_{\sm{A}}
=
[\mf{T}]_{\af{A}}$
and
$[\mf{I}]_{\sm{A}}
=
[\mf{I}]_{\dm{A}}$,
for any 
$\mf{T}\in\pf{T}_{\mf{A}}^{\diamond}$
and
$\mf{I}\in\pf{T}_{\mf{A}}^{\vc}$,
in particular
$\pf{T}_{\mf{A}}^{\star}
\subseteq
\pf{T}_{\mf{A}}^{\dv}$
and
$\pf{T}_{\mf{A}}^{\vm}
\subseteq
\pf{T}_{\mf{A}}^{\dv}$.
Finally if $\mf{A}\in\ms{C}_{u}^{0}(H)$ then
$\pf{V}(\mf{A})\subseteq\pf{T}_{\mf{A}}^{\vc}$ 
and $[\mf{Y}]_{\tm{A}}=[\mf{Y}]_{\sm{A}}$ for any $\mf{Y}\in\pf{V}(\mf{A})$,
in particular $\pf{V}_{\mf{A}}^{\natural}\subseteq\pf{T}_{\mf{A}}^{\vm}$.
\end{proposition}
\begin{proof}
The first sentence is easy to show, while the first equality follows since the bicommutant theorem.
Next for any 
$\mf{T}\in\pf{T}_{\mf{A}}^{\vc}$ 
clearly 
$[\mf{T}]_{\sm{A}}
\supseteq
[\mf{T}]_{\dm{A}}$,
while
$[\mf{T}]_{\sm{A}}
\subseteq
[\mf{T}]_{\dm{A}}$
follows since Lemma \ref{12071157}.
$\pf{V}(\mf{A})\subseteq\pf{T}_{\mf{A}}^{\vc}$ follows since Rmk. \ref{13071106},
while
$[\mf{Y}]_{\sm{A}}\subseteq[\mf{Y}]_{\tm{A}}$ for any $\mf{Y}\in\pf{V}(\mf{A})$ follows since \eqref{05072302}.
\end{proof}
\begin{convention}
For any $\mf{T}\in\pf{T}_{\mf{A}}$
we 
let $[\mf{T}]$ denote $[\mf{T}]_{\sm{A}}$ and often 
when it does not cause confusion we let
$\mf{T}\simeq\mf{Q}$ denote $\mf{T}\sm{A}\mf{Q}$.
\end{convention}
\begin{remark}
\label{06070900}
By using a line similar to the one in the proof of 
Thm. \ref{11260828}, 
by \eqref{06071411}
and taking into account that two unitary
equivalent cyclic representations of a $C^{\ast}-$algebra are
associated with the same state, 
we deduce that $[\mf{T}]$ holds more than one element
for any $\mf{T}\in\pf{T}_{\mf{A}}$.
\end{remark}
\begin{lemma}
\label{18071115}
If $\mf{A},\mf{B}\in\ms{C}_{u}(H)$, $T\in Mor_{\ms{C}_{u}(H)}(\mf{B},\mf{A})$, $\mf{T}\in\pf{T}_{\mf{A}}$ and 
$\alpha\in\ms{P}_{\mf{A}}^{\mf{T}}$, 
then $\mathbb{A}(\mf{B})_{\alpha}^{\mf{d}^{H}(T)(\mf{T})}=\mathbb{A}(\mf{A})_{\alpha}^{\mf{T}}$.
\end{lemma}
\begin{proof}
Since the bicommutant theorem we deduce that
$\mathbb{A}(\mf{B})_{\alpha}^{\mf{d}^{H}(T)(\mf{T})}=\ov{\uppi_{\alpha}(T(\mc{B}))\cup\ms{U}_{\pf{H}_{\alpha}^{T}}(H)}^{w}$
and
$\mathbb{A}(\mf{A})_{\alpha}^{\mf{T}}=\ov{\uppi_{\alpha}(\mc{A})\cup\ms{U}_{\pf{H}_{\alpha}}(H)}^{w}$,
where $\ov{S}^{w}$ is the closure w.r.t. the weak operator topology of any subset $S$ of $\mc{L}(\mf{H}_{\alpha})$.
Thus the statement follows since $T$ is surjective and 
Cor. \ref{11201755}. 
\end{proof}
\begin{proposition}
\label{05071049}
If $\mf{A}$ is unitarily implemented by $\ms{v}$
and $l\in H$, 
then
$\mf{b}^{\mf{A},\ms{v}}(l)$ is $(\sm{A},\sm{A})-$compatible,
moreover 
$\mf{b}^{\mf{A},\ms{v}}(l)(\pf{T}_{\mf{A}}^{\diamond})
\subseteq\pf{T}_{\mf{A}}^{\diamond}$
and
$\mf{b}^{\mf{A},\ms{v}}(l)(\pf{T}_{\mf{A}}^{\vc})
\subseteq\pf{T}_{\mf{A}}^{\vc}$
hence
$\mf{b}^{\mf{A},\ms{v}}(l)\up\pf{T}_{\mf{A}}^{\diamond}$
is $(\af{A},\af{A})-$compatible
and
$\mf{b}^{\mf{A},\ms{v}}(l)\up\pf{T}_{\mf{A}}^{\vc}$
is $(\dm{A},\dm{A})-$compatible.
If in addition $\mf{A},\mf{B}\in\ms{C}_{u}(H)$
and 
$T\in Mor_{\ms{C}_{u}(H)}(\mf{B},\mf{A})$,
then 
$\mf{d}^{H}(T)$ 
is $(\sm{A},\sm{B})-$compatible,
moreover
$\mf{d}^{H}(T)(\pf{T}_{\mf{A}}^{\diamond})
\subseteq
\pf{T}_{\mf{B}}^{\diamond}$
and
$\mf{d}^{H}(T)(\pf{T}_{\mf{A}}^{\vc})
\subseteq
\pf{T}_{\mf{B}}^{\vc}$
hence
$\mf{d}^{H}(T)\up\pf{T}_{\mf{A}}^{\diamond}$ 
is $(\af{A},\af{B})-$compatible
and
$\mf{d}^{H}(T)\up\pf{T}_{\mf{A}}^{\vc}$ 
is $(\dm{A},\dm{B})-$compatible.
If $\mf{A}\in\ms{C}_{u}^{0}(H)$ then $\mf{b}^{\mf{A},\ms{v}}(l)\up\pf{V}(\mf{A})$ is $(\tm{A},\tm{A})-$compatible,
while whenever $\mf{B}\in\ms{C}_{u}^{0}(H)$ and $T\in Mor_{\ms{C}_{u}^{0}(H)}(\mf{B},\mf{A})$ then 
$\mf{d}^{H}(T)\up\pf{V}_{\mf{A}}$ is $(\tm{A},\tm{B})-$compatible.
\end{proposition}
\begin{proof}
Let $\mf{T}\sm{A}\mf{Q}$
and $l\in H$
we claim to show that
$\mf{b}^{\mf{A},\ms{v}}(l)(\mf{T})
\sm{A}
\mf{b}^{\mf{A},\ms{v}}(l)(\mf{Q})$
and 
if in addition $\mf{A},\mf{B}\in\ms{C}_{u}(H)$
that
$\mf{d}^{H}(T)(\mf{T})
\sm{B}
\mf{d}^{H}(T)(\mf{Q})$.
Let
$\alpha\in\ms{P}_{\mf{A}}^{\mf{T}}$
then
$V_{\alpha}\uppi_{\alpha}(\ms{v}(l))\Upomega_{\alpha}
=\upupsilon_{\alpha}(\ms{v}(l))\Uppsi_{\alpha}$
and
$(\ms{ad}(V_{\alpha})
\circ
\ms{ad}(\uppi_{\alpha}(\ms{v}(l))))
(\Upgamma_{\alpha})
=
(\ms{ad}(\upupsilon_{\alpha}(\ms{v}(l)))
\circ
\ms{ad}(V))
(\Upgamma_{\alpha})
=
\ms{ad}(\upupsilon_{\alpha}(\ms{v}(l)))
(\Updelta_{\alpha})
$,
while clearly
$\ms{ad}(V_{\alpha})\circ\uppi_{\alpha}\circ T
=
\upupsilon_{\alpha}\circ T$, and our claim is proved.
Next assume
$\mf{T}\in\pf{T}_{\mf{A}}^{\vc}$, the case $\mf{T}\in\pf{T}_{\mf{A}}^{\diamond}$ follows similarly.
Clearly $\ms{ad}(\uppi_{\alpha}(\ms{v}(l)))\circ\uppi_{\alpha}=\uppi_{\alpha}\circ\ms{ad}(\ms{v}(l))$, hence 
$\ms{ad}(\uppi_{\alpha}(\ms{v}(l)))(\mathscr{A}(\mf{A})_{\alpha}^{\mf{T}})=
\mathscr{A}(\mf{A})_{\alpha}^{\mf{b}^{\mf{A},\ms{v}}(l)
(\mf{T})}$
since 
Cor. \ref{09131225}
thus 
$\ms{ad}(\uppi_{\alpha}(\ms{v}(l)))(\mathbb{A}(\mf{A})_{\alpha}^{\mf{T}})
=\mathbb{A}(\mf{A})_{\alpha}^{\mf{b}^{\mf{A},\ms{v}}(l)(\mf{T})}$
since the bicommutant theorem and since $\ms{ad}(\uppi_{\alpha}(\ms{v}(l)))$ is weakly continuous.
Therefore 
$\mf{b}^{\mf{A},\ms{v}}(l)(\mf{T})
\in\pf{T}_{\mf{A}}^{\vc}$.
Next $\mf{d}^{H}(T)(\mf{T})\in\pf{T}_{\mf{B}}^{\vc}$ 
since Lemma \ref{18071115}. The last sentence of the statement follows since 
the first one and 
Lemma \ref{12051430}, 
and since the second one and the definition of $ Mor_{\ms{C}_{u}^{0}(H)}$.
\end{proof}
The following result is fundamental in order to define in 
Def. \ref{06071514} the maps $\mf{m}_{\#}$ and $\mf{v}_{\natural}$,
where $\#\in\{\dv,\star,\vm\}$.
\begin{lemma}
\label{06071046}
We have 
\begin{enumerate}
\item
$\ov{\ms{m}}^{\mf{A}}$ is $(\sm{A},=)-$compatible,
i.e.
$\mf{T}\sm{A}\mf{Q}\Rightarrow
\ov{\ms{m}}^{\mf{A}}(\mf{T})=\ov{\ms{m}}^{\mf{A}}(\mf{Q})$,
\label{06071046st1}
\item
if $\mf{A}\in\ms{C}_{u}^{0}(H)$ then
$\mc{V}^{\mf{A}}$ is $(\tm{A},=)-$compatible.
\label{06071046st2}
\end{enumerate}
\end{lemma}
\begin{proof}
Let $\mf{T}\sm{A}\mf{Q}$ and $\alpha\in\ms{P}_{\mf{A}}^{\mf{T}}$.
Since \eqref{05072220} and following the arguments used 
in the proof of 
Thm. \ref{09191747} 
we deduce that
the next is a commutative diagram
\begin{equation*}
\xymatrix{
\mc{L}(\mf{H}_{\alpha})
\ar[r]^{\ms{ad}(V_{\alpha})}
&
\mc{L}(\mf{K}_{\alpha})
\\
\ms{B}_{\ps{\upmu}}^{\ps{\upomega},\alpha}
\ar[u]^{\mf{R}_{\pf{H},\alpha}^{\ps{\upmu}}}
\ar[ur]_{\mf{R}_{\pf{K},\alpha}^{\ps{\upmu}}},}
\end{equation*}
hence since 
Lemma \ref{10311019} 
we obtain 
\begin{equation}
\label{05072302}
\tilde{\mf{R}}_{\pf{K},\alpha}^{\ps{\upmu}}
=
\ms{ad}(V_{\alpha})
\circ
\tilde{\mf{R}}_{\pf{H},\alpha}^{\ps{\upmu}}.
\end{equation}
Next for any $x\in X$
if $S_{x}^{\alpha}$ and $T_{x}^{\alpha}$ 
is such that $iS_{x}^{\alpha}$ is the 
infinitesimal generator of the semigroup
$\ms{U}_{\pf{K}_{\alpha}}\circ\upzeta\circ\mf{i}_{x}\up\R^{+}$
and
$\ms{U}_{\pf{H}_{\alpha}}\circ\upzeta\circ\mf{i}_{x}\up\R^{+}$
respectively
then since \eqref{05072220} we obtain
\begin{equation}
\label{06071413}
S_{x}^{\alpha}=V_{\alpha}T_{x}^{\alpha}V_{\alpha}^{\ast},
\end{equation}
thus by 
Cor \ref{10301830}\eqref{10301830st2},
\eqref{18225} 
and 
\eqref{01291801}, 
\begin{equation}
\label{06071411}
\ov{f(\mc{C}(A,\prod_{x\in A}
supp(\ms{E}_{S_{x}^{\alpha}})))}
\subseteq\R,
\end{equation}
where 
$A\in\mathscr{P}_{\omega}(X)$ satisfying
Def. \ref{08261134}\eqref{08261134VI}
applied to $\mf{T}$.
Since \eqref{06071413} and 
Thm. \ref{11252000}\eqref{11252000st2}
we deduce that
\begin{equation}
\label{05072253}
\ms{D}_{\pf{K},\alpha}^{\upzeta,f}
=
V_{\alpha}
\ms{D}_{\pf{H},\alpha}^{\upzeta,f}
V_{\alpha}^{\ast}.
\end{equation}
By construction we have
\begin{equation}
\label{05072310}
\begin{aligned}
\ps{R}(\mf{T},\alpha)
&=
\left(
(\ms{B}_{\ps{\upmu}}^{\ps{\upomega},\alpha,+},
\tilde{\mf{R}}_{\pf{H},\alpha}^{\ps{\upmu}}),
\ms{D}_{\pf{H},\alpha}^{\upzeta,f},
\Upgamma_{\alpha}\right),
\\
\ps{R}(\mf{Q},\alpha)
&=
\left(
(\ms{B}_{\ps{\upmu}}^{\ps{\upomega},\alpha,+},
\tilde{\mf{R}}_{\pf{K},\alpha}^{\ps{\upmu}}),
\ms{D}_{\pf{K},\alpha}^{\upzeta,f},
\Updelta_{\alpha}\right).
\end{aligned}
\end{equation}
By \eqref{05072253} and 
Cor. \ref{10301830}\eqref{10301830st3}
we have 
\begin{equation}
\label{05072305}
\exp(-(\ms{D}_{\pf{K},\alpha}^{\upzeta,f})^{2})
=
\ms{ad}(V_{\alpha})
(\exp(-(\ms{D}_{\pf{H},\alpha}^{\upzeta,f})^{2})).
\end{equation}
Next by abuse of language let 
$V$,
$\Updelta$, $\Upgamma$, 
$\tilde{\mf{R}}_{\pf{K}}$, $\tilde{\mf{R}}_{\pf{H}}$,
$\ms{D}_{\pf{K}}$ and $\ms{D}_{\pf{H}}$
denote
$V_{\alpha}$,
$\Updelta_{\alpha}$, $\Upgamma_{\alpha}$, 
$\tilde{\mf{R}}_{\pf{K},\alpha}^{\ps{\upmu}}$,
$\tilde{\mf{R}}_{\pf{H},\alpha}^{\ps{\upmu}}$,
$\ms{D}_{\pf{K},\alpha}^{\upzeta,f}$
and 
$\ms{D}_{\pf{H},\alpha}^{\upzeta,f}$
respectively, 
then since \eqref{05072302}\,\&\,\eqref{05072253}\,\&\,\eqref{05072305} we obtain for all $s_{0},\dots,s_{2n}\in\R$ and
$a_{0},\dots,a_{2n}\in\ms{B}_{\ps{\upmu}}^{\ps{\upomega},\alpha,+}$
\begin{multline*}
Tr\bigl(
\Updelta\,
\tilde{\mf{R}}_{\pf{K}}(a_{0})
\exp(-s_{0}\ms{D}_{\pf{K}}^{2})
[\ms{D}_{\pf{K}},
\tilde{\mf{R}}_{\pf{K}}(a_{1})]
\exp(-s_{1}\ms{D}_{\pf{K}}^{2})
\dotsc
\\
[\ms{D}_{\pf{K}},
\tilde{\mf{R}}_{\pf{K}}(a_{2n-1})]
\exp(-s_{2n-1}\ms{D}_{\pf{K}}^{2})
[\ms{D}_{\pf{K}},
\tilde{\mf{R}}_{\pf{K}}(a_{2n})]
\exp(-s_{2n}\ms{D}_{\pf{K}}^{2})
\bigr)
=
\\
(Tr\circ\ms{ad}(V))
\bigl(
\Upgamma\,
\tilde{\mf{R}}_{\pf{H}}(a_{0})
\exp(-s_{0}\ms{D}_{\pf{H}}^{2})
[\ms{D}_{\pf{H}},\tilde{\mf{R}}_{\pf{H}}(a_{1})]
\exp(-s_{1}\ms{D}_{\pf{H}}^{2})
\dots
\\
[\ms{D}_{\pf{H}},\tilde{\mf{R}}_{\pf{H}}(a_{2n-1})]
\exp(-s_{2n-1}\ms{D}_{\pf{H}}^{2})
[\ms{D}_{\pf{H}},\tilde{\mf{R}}_{\pf{H}}(a_{2n})]
\exp(-s_{2n}\ms{D}_{\pf{H}}^{2})\bigr)
=
\\
Tr\bigl(
\Upgamma\,
\tilde{\mf{R}}_{\pf{H}}(a_{0})
\exp(-s_{0}\ms{D}_{\pf{H}}^{2})
[\ms{D}_{\pf{H}},\tilde{\mf{R}}_{\pf{H}}(a_{1})]
\exp(-s_{1}\ms{D}_{\pf{H}}^{2})
\dots
\\
[\ms{D}_{\pf{H}},\tilde{\mf{R}}_{\pf{H}}(a_{2n-1})]
\exp(-s_{2n-1}\ms{D}_{\pf{H}}^{2})
[\ms{D}_{\pf{H}},\tilde{\mf{R}}_{\pf{H}}(a_{2n})]
\exp(-s_{2n}\ms{D}_{\pf{H}}^{2})\bigr).
\end{multline*}
Therefore since \eqref{05072310} and \cite[$IV.8.\epsilon$ Thm $22$, Thm. $21$, and $IV.7.\delta$ Thm $21$]{connes}
we have 
\begin{equation*}
\lr{\cdot}
{\ms{ch}(\ps{R}(\mf{Q},\alpha))}_{\ps{\upmu},\ps{\upomega},\alpha}
=
\lr{\cdot}
{\ms{ch}(\ps{R}(\mf{T},\alpha))}_{\ps{\upmu},\ps{\upomega},\alpha},
\end{equation*}
and st.\eqref{06071046st1} follows. St.\eqref{06071046st2} follows since \eqref{05072305}, the hypothesis 
$\upupsilon_{\alpha}=\ms{ad}(V_{\alpha})\circ\uppi_{\alpha}$ and 
Thm. \ref{12051936}\eqref{12051936st3}\,\&\,\eqref{05061911}.
\end{proof}
Prp. \ref{05071049} permits to set the following
\begin{definition}
\label{05071126}
Let $\#\in\{\dv,\star,\vm\}$, define $\mf{b}_{\#}^{\mf{A}}$ to be the map on $H$ and $\mf{d}_{\#}^{H}$ to be the map on 
$ Mor_{\ms{C}_{u}(H)}$ such that for all $l\in H$ and $T\in Mor_{\ms{C}_{u}(H)}(\mf{B},\mf{A})$
\begin{equation*}
\begin{aligned}
\mf{b}_{\dv}^{\mf{A}}(l):
\pf{T}_{\mf{A}}^{\dv}
&\to\pf{T}_{\mf{A}}^{\dv},\,
[\mf{T}]_{\sm{A}}
\mapsto[\mf{b}^{\mf{A}}(l)(\mf{T})]_{\sm{A}},
\\
\mf{b}_{\star}^{\mf{A}}(l):
\pf{T}_{\mf{A}}^{\star}
&\to\pf{T}_{\mf{A}}^{\star},\,
[\mf{T}]_{\af{A}}
\mapsto[\mf{b}^{\mf{A}}(l)(\mf{T})]_{\af{A}},
\\
\mf{b}_{\vm}^{\mf{A}}(l):
\pf{T}_{\mf{A}}^{\vm}
&\to\pf{T}_{\mf{A}}^{\vm},\,
[\mf{T}]_{\dm{A}}
\mapsto[\mf{b}^{\mf{A}}(l)(\mf{T})]_{\dm{A}},
\end{aligned}
\hspace{15pt}
\begin{aligned}
\mf{d}_{\dv}^{H}(T):
\pf{T}_{\mf{A}}^{\dv}
&\to
\pf{T}_{\mf{B}}^{\dv},\,
[\mf{T}]_{\sm{A}}\mapsto[\mf{d}^{H}(T)(\mf{T})]_{\sm{B}},
\\
\mf{d}_{\star}^{H}(T):
\pf{T}_{\mf{A}}^{\star}
&\to
\pf{T}_{\mf{B}}^{\star},\,
[\mf{T}]_{\af{A}}\mapsto[\mf{d}^{H}(T)(\mf{T})]_{\af{B}},
\\
\mf{d}_{\vm}^{H}(T):
\pf{T}_{\mf{A}}^{\vm}
&\to
\pf{T}_{\mf{B}}^{\vm},\,
[\mf{T}]_{\dm{A}}\mapsto[\mf{d}^{H}(T)(\mf{T})]_{\dm{B}},
\end{aligned}
\end{equation*}
if in addition $\mf{A},\mf{B}\in\ms{C}_{u}^{0}(H)$ we can set
for all $l\in H$ and $T\in Mor_{\ms{C}_{u}^{0}(H)}(\mf{B},\mf{A})$
\begin{equation*}
\mf{b}_{\natural}^{\mf{A}}(l):
\pf{V}_{\mf{A}}^{\natural}
\to\pf{V}_{\mf{A}}^{\natural},\,
[\mf{T}]_{\tm{A}}
\mapsto[\mf{b}^{\mf{A}}(l)(\mf{T})]_{\tm{A}},
\hspace{15pt}
\mf{d}_{\natural}^{H}(T):
\pf{V}_{\mf{A}}^{\natural}
\to
\pf{V}_{\mf{B}}^{\natural},\,
[\mf{T}]_{\tm{A}}\mapsto[\mf{d}^{H}(T)(\mf{T})]_{\tm{B}}.
\end{equation*}
\end{definition}
\begin{remark}
\label{09071514}
Since Prp. \ref{09070854} we have that
$\mf{b}_{\star}^{\mf{A}}(l)
=
\mf{b}_{\dv}^{\mf{A}}(l)
\up
\pf{T}_{\mf{A}}^{\star}$
and
$\mf{d}_{\star}^{H}(T)
=
\mf{d}_{\dv}^{H}(T)
\up
\pf{T}_{\mf{A}}^{\star}$,
while
$\mf{b}_{\vm}^{\mf{A}}(l)
=
\mf{b}_{\dv}^{\mf{A}}(l)
\up
\pf{T}_{\mf{A}}^{\vm}$
and
$\mf{d}_{\vm}^{H}(T)
=
\mf{d}_{\dv}^{H}(T)
\up
\pf{T}_{\mf{A}}^{\vm}$,
finally if $\mf{A}\in\ms{C}_{u}^{0}(H)$ then
$\mf{b}_{\natural}^{\mf{A}}(l)
=
\mf{b}_{\vm}^{\mf{A}}(l)
\up
\pf{V}_{\mf{A}}^{\natural}$
and
$\mf{d}_{\natural}^{H}(T)
=
\mf{d}_{\vm}^{H}(T)
\up
\pf{V}_{\mf{A}}^{\natural}$.
\end{remark}
\begin{definition}
\label{07070947}
For any $\mf{q}\in\pf{T}_{\mf{A}}^{\dv}$ set 
$\ms{P}_{\mf{A}}^{\mf{q}}$ 
to be the set $\ms{P}_{\mf{A}}^{\mf{Q}}$,
where $\mf{Q}\in\mf{q}$.
Since  
$\mf{T}\sm{A}\mf{Q}\Rightarrow(\forall\alpha\in\ms{P}_{\mf{A}}^{\mf{T}})
(\ms{K}_{\alpha}^{\mf{T}}(\mf{A})=\ms{K}_{\alpha}^{\mf{Q}}(\mf{A}))$
we can set for all $\mf{q}\in\pf{T}_{\mf{A}}^{\dv}$ and 
$\alpha\in\ms{P}_{\mf{A}}^{\mf{q}}$
\begin{equation*}
\ms{K}_{\alpha}^{\mf{q}}(\mf{A})
\coloneqq
\ms{K}_{\alpha}^{\mf{Q}}(\mf{A}),\,
\mf{Q}\in\mf{q}.
\end{equation*}
Set $s(\dv)=\varnothing$, $s(\star)=\diamond$ and $s(\vm)=\vc$.
Let $\#\in\{\dv,\star,\vm\}$, 
define $\ms{A}_{\#}:Obj(\ms{C}_{u}(H))\to Obj(\ms{Ab})$
such that 
\begin{equation*}
\ms{A}_{\#,\mf{A}}
\coloneqq
\prod_{\mf{q}\in\mf{T}_{\mf{A}}^{\#}}
\prod_{\alpha\in\ms{P}_{\mf{A}}^{\mf{q}}}
\ms{K}_{\alpha}^{\mf{q}}(\mf{A}),
\end{equation*}
provided by the pointwise composition, inversion and identity.
By abuse of language we shall use the same symbol $\ms{A}_{\#,\mf{A}}$ 
to denote the set underlying the group $\ms{A}_{\#,\mf{A}}$.
Let $\ms{r}_{\#}$ be the map on $Obj(\ms{C}_{u}(H))$ such that 
$\ms{r}_{\#}^{\mf{A}}:\ms{A}_{\#,\mf{A}}\to\ms{A}_{\mf{A}}$
and for all $\ms{f}\in\ms{A}_{\#,\mf{A}}$, $\mf{Q}\in\pf{T}_{\mf{A}}^{s(\#)}$ and $\alpha\in\ms{P}_{\mf{A}}^{\mf{Q}}$ 
\begin{equation*}
\ms{r}_{\#}^{\mf{A}}(\ms{f})(\mf{Q})(\alpha)
\coloneqq
\ms{f}([\mf{Q}]_{\sm{A}})(\alpha).
\end{equation*}
Define 
\begin{equation*}
\uppsi_{\#}\in\prod_{\mf{D}\in\ms{C}_{u}(H)}Mor_{\ms{set}}(H, Aut_{\ms{Ab}}(\ms{A}_{\#,\mf{D}})),
\end{equation*}
such that for all $l\in H$ and $\ms{f}\in\ms{A}_{\#,\mf{A}}$ 
\begin{equation*}
\uppsi_{\#}^{\mf{A}}(l)(\ms{f})
\coloneqq
\ov{\mf{c}}^{\mf{A}}(l)\circ\ms{f}\circ\mf{b}_{\#}^{\mf{A}}(l^{-1}),
\end{equation*}
Finally 
\begin{equation*}
\mf{g}_{\#}^{H}\in\prod_{L\in Mor_{\ms{C}_{u}(H)}} Mor_{\ms{Ab}}(\ms{A}_{\#,\,d(L)},\ms{A}_{\#,\,c(L)}),
\end{equation*}
such that for all $W\in Mor_{\ms{C}_{u}(H)}(\mf{A},\mf{B})$ and $\ms{f}\in\ms{A}_{\#,\mf{A}}$
\begin{equation*}
\mf{g}_{\#}^{H}(W)(\ms{f})\coloneqq\ov{\mf{h}}^{H}(W)\circ\ms{f}\circ\mf{d}_{\#}^{H}(W).
\end{equation*}
\end{definition}
Note that $\ms{r}_{\#}$ is well-defined since Prp. \ref{09070854},
while $\uppsi_{\#}$ and $\mf{g}_{\#}^{H}$ are well-defined since 
the second inclusion in \eqref{10291120}
and the 
standard picture used for the $\ms{K}_{0}-$groups.
\begin{definition}
\label{14071100}
Define $\ms{Z}_{\natural}:Obj(\ms{C}_{u}^{0}(H))\to Obj(\ms{set})$ 
and
$\ms{Z}_{\natural}^{m}\in\prod_{T\in Mor_{\ms{C}_{u}^{0}(H)}}
Mor_{\ms{set}}(\ms{Z}_{\natural}(c(T)),\ms{Z}_{\natural}(d(T)))$,
such that 
\begin{equation*}
\ms{Z}_{\natural}:\mf{D}\mapsto\coprod_{\mf{p}\in
\pf{V}_{\mf{D}}^{\natural}}
Mor_{\ms{set}}(\ms{P}_{\mf{D}}^{\mf{p}},\mc{A}_{\mf{D}}^{\ast}),
\end{equation*}
where $\mc{A}_{\mf{D}}$ denotes the $C^{\ast}-$algebra underlying $\mf{D}$,
and if $\mf{A},\mf{B}\in\ms{C}_{u}^{0}(H)$ and $T\in Mor_{\ms{C}_{u}^{0}(H)}(\mf{B},\mf{A})$ then
\begin{equation*}
\begin{aligned}
\ms{Z}_{\natural}^{m}(T):\ms{Z}_{\natural}(\mf{A})&\to\ms{Z}_{\natural}(\mf{B}),
\\
(\mf{p},g)&\mapsto\left(
\mf{d}_{\natural}^{H}(T)(\mf{p}),\ms{P}_{\mf{B}}^{\mf{d}_{\natural}^{H}(T)(\mf{p})}\ni\alpha\mapsto g(\alpha)\circ T\right).
\end{aligned}
\end{equation*}
Moreover define 
\begin{equation*}
\ms{V}_{\natural}\in\prod_{\mf{D}\in\ms{C}_{u}^{0}(H)}
Mor_{\ms{Gr}}(H,Aut_{\ms{set}}(\ms{Z}_{\natural}(\mf{D}))),
\end{equation*}
such that if $\mf{A}\in\ms{C}_{u}^{0}(H)$ and $l\in H$ then
\begin{equation*}
\ms{V}_{\natural}(\mf{A})(l):
(\mf{p},g)\mapsto\left(\mf{b}_{\natural}^{\mf{A}}(l)(\mf{p}),
\ms{P}_{\mf{A}}^{\mf{b}_{\natural}^{\mf{A}}(l)(\mf{p})}\ni\alpha\mapsto g(\alpha)\circ\upeta(l^{-1})\right).
\end{equation*}
\end{definition}
For the remaining of this section let $\un$ denote the unit 
element of $H$.
\begin{definition}
\label{07071022}
Let $\#\in\{\dv,\star,\vm\}$.
Define 
$\mf{P}_{\#}$ and $\mf{Q}_{o}^{\#}$ be the maps on $Obj(\ms{C}_{u}(H))$, 
and $\mf{Q}_{m}^{\#}$ be the map on $ Mor_{\ms{C}_{u}(H)}$
such that 
\begin{equation*}
\begin{aligned}
\mf{P}_{\#}^{\mf{A}}
&\coloneqq
(\un\mapsto\pf{T}_{\mf{A}}^{\#},\mf{b}_{\#}^{\mf{A}}),
\\
\mf{Q}_{o}^{\#}(\mf{A})
&\coloneqq
\pf{T}_{\mf{A}}^{\#},
\\
\mf{Q}_{m}^{\#}
&\coloneqq
\mf{d}_{\#}^{H}.
\end{aligned}
\end{equation*}
Moreover define 
$\mf{P}_{\natural}$, $\mf{Z}_{\natural}$, and 
$\mf{Q}_{o}^{\natural}$ be the maps on $Obj(\ms{C}_{u}^{0}(H))$, 
and $\mf{Q}_{m}^{\natural}$ be the map on $ Mor_{\ms{C}_{u}^{0}(H)}$
such that if $\mf{A}\in\ms{C}_{u}^{0}(H)$ then 
\begin{equation*}
\begin{aligned}
\mf{P}_{\natural}^{\mf{A}}
&\coloneqq
(\un\mapsto\pf{V}_{\mf{A}}^{\natural},\mf{b}_{\natural}^{\mf{A}}),
\\
\mf{Z}_{\natural}^{\mf{A}}
&\coloneqq
\left(\un\mapsto\ms{Z}_{\natural}(\mf{A}),\ms{V}_{\natural}(\mf{A})\right),
\\
\mf{Q}_{o}^{\natural}(\mf{A})
&\coloneqq
\pf{V}_{\mf{A}}^{\natural},
\\
\mf{Q}_{m}^{\natural}
&\coloneqq
\mf{d}_{\natural}^{H}.
\end{aligned}
\end{equation*}
\end{definition}
\begin{lemma}
\label{07071040}
We have that
\begin{enumerate}
\item
$(\mf{Q}_{o}^{\dv},\mf{Q}_{m}^{\dv})\in
\ms{Fct}(\ms{C}_{u}(H)^{op},\ms{set})$,
\label{07071040st1}
\item
$(\ms{Z}_{\natural},\ms{Z}_{\natural}^{m})\in\ms{Fct}(\ms{C}_{u}^{0}(H)^{op},\ms{set})$,
\label{07071040st1z}
\item
$\mf{P}_{\dv}^{\mf{A}}\in\ms{Fct}(H,\ms{set})$,
\label{07071040st2}
\item
$\mf{Z}_{\natural}^{\mf{A}}\in\ms{Fct}(H,\ms{set})$, if $\mf{A}\in\ms{C}_{u}^{0}(H)$,
\label{07071040st2z}
\item
$\uppsi_{\dv}^{\mf{A}}\in Mor_{\ms{Gr}}(H,Aut_{\ms{Ab}}(\ms{A}_{\dv,\mf{A}}))$, 
\label{07071040st3}
\item
$(\ms{A}_{\dv},\mf{g}_{\dv}^{H})\in\ms{Fct}(\ms{C}_{u}(H),\ms{Ab})$,
\label{07071040st4}
\item
$\ms{r}_{\dv}^{\mf{A}}$ is an injective group morphism,
\label{07071040st5}
\item
$\ms{r}_{\dv}^{\mf{A}}
\circ
\uppsi_{\dv}^{\mf{A}}(l)
=
\uppsi^{\mf{A}}(l)
\circ
\ms{r}_{\dv}^{\mf{A}}$, for all $l\in H$,
\label{07071040st6}
\item
$\mf{g}^{H}(T)\circ\ms{r}_{\dv}^{\mf{B}} 
=
\ms{r}_{\dv}^{\mf{A}}\circ\mf{g}_{\dv}^{H}(T)$,
for all $T\in Mor_{\ms{C}_{u}(H)}(\mf{B},\mf{A})$.
\label{07071040st7}
\end{enumerate}
\end{lemma}
\begin{proof}
st.\eqref{07071040st1}\,\&\,\eqref{07071040st1z} follow since 
Lemma \ref{11271700}, 
st.\eqref{07071040st2}\,\&\,\eqref{07071040st2z}
by 
Thm. \ref{11260828}.
For any $l\in H$ the
$\uppsi_{\dv}^{\mf{A}}(l)$ 
maps 
$\ms{A}_{\dv,\mf{A}}$ into itself
since the construction of $\mf{b}_{\dv}^{\mf{A}}$
and by 
Thm. \ref{11260828}.
$\uppsi_{\dv}^{\mf{A}}$ 
is a $H-$action since $\mf{b}_{\dv}^{\mf{A}}$
and $\ov{\mf{c}}^{\mf{A}}$
are $H-$actions by st.\eqref{07071040st2}
and
Prp. \ref{01011526} 
respectively,
so st.\eqref{07071040st3} follows.
St.\eqref{07071040st4} follows since 
\eqref{12010253bis} and st.\eqref{07071040st1}. 
St.\eqref{07071040st5} is immediate, while
st.\eqref{07071040st6}\,\&\,\eqref{07071040st7}  
follow by direct computation.
\end{proof}
Since Rmk. \ref{09071514} and then following the same 
argument used in the proof of Lemma \ref{07071040}
we obtain the following
\begin{lemma}
\label{07071040b}
For all $\#\in\{\star,\vm\}$ we obtain
\begin{enumerate}
\item
$(\mf{Q}_{o}^{\#},\mf{Q}_{m}^{\#})\in
\ms{Fct}(\ms{C}_{u}(H)^{op},\ms{set})$,
\label{07071040bst1}
\item
$\mf{P}_{\#}^{\mf{A}}\in\ms{Fct}(H,\ms{set})$,
\label{07071040bst2}
\item
$\uppsi_{\#}^{\mf{A}}\in Mor_{\ms{Gr}}(H,Aut_{\ms{Ab}}(\ms{A}_{\#,\mf{A}}))$, 
\label{07071040bst3}
\item
$(\ms{A}_{\#},\mf{g}_{\#}^{H})
\in\ms{Fct}(\ms{C}_{u}(H),\ms{Ab})$,
\label{07071040bst4}
\item
$\ms{r}_{\#}^{\mf{A}}$ is an injective group morphism,
\label{07071040bst5}
\item
$\ms{r}_{\#}^{\mf{A}}
\circ
\uppsi_{\#}^{\mf{A}}(l)
=
\uppsi^{\mf{A}}(l)
\circ
\ms{r}_{\#}^{\mf{A}}$, for all $l\in H$,
\label{07071040bst6}
\item
$\mf{g}^{H}(T)\circ\ms{r}_{\#}^{\mf{B}} 
=
\ms{r}_{\#}^{\mf{A}}\circ\mf{g}_{\#}^{H}(T)$,
for all $T\in Mor_{\ms{C}_{u}(H)}(\mf{B},\mf{A})$,
\label{07071040bst7}
\item
$(\mf{Q}_{o}^{\natural},\mf{Q}_{m}^{\natural})\in
\ms{Fct}(\ms{C}_{u}^{0}(H)^{op},\ms{set})$,
\label{13071854a}
\item
$\mf{P}_{\natural}^{\mf{A}}\in\ms{Fct}(H,\ms{set})$, if $\mf{A}\in\ms{C}_{u}^{0}(H)$.
\label{13071854b}
\end{enumerate}
\end{lemma}
According to Lemma \ref{06071046}, Lemmas \ref{07071040}\eqref{07071040st5} and \ref{07071040b}\eqref{07071040bst5},
we can set the following 
\begin{definition}
\label{06071514}
Let $\#\in\{\dv,\star,\vm\}$, define $\ov{\ms{m}}_{\#}$ and $\mf{m}_{\#}$ to be the maps on $Obj(\ms{C}_{u}(H))$
such that 
\begin{equation*}
\begin{aligned}
\ov{\ms{m}}_{\#}^{\mf{A}}
&\in\prod_{\mf{p}\in\pf{T}_{\mf{A}}^{\#}}
Mor_{\ms{set}}(\ms{P}_{\mf{A}}^{\mf{p}},\ms{A}_{\#,\mf{A}}^{\ast}),
\\
\ov{\ms{m}}_{\#}^{\mf{A}}(\mf{p})(\beta)
&\coloneqq
\ov{\ms{m}}^{\mf{A}}(\mf{I})(\beta)\circ\ms{r}_{\#}^{\mf{A}},
\qquad
\mf{p}\in\pf{T}_{\mf{A}}^{\#},
\mf{I}\in\mf{p},
\beta\in\ms{P}_{\mf{A}}^{\mf{p}};
\\
\mf{m}_{\#}^{\mf{A}}
&\coloneqq(\un\mapsto\ov{\ms{m}}_{\#}^{\mf{A}}).
\end{aligned}
\end{equation*}
Moreover define $\mc{V}_{\natural}$ and $\mf{v}_{\natural}$ to be the maps on $Obj(\ms{C}_{u}^{0}(H))$ 
such that 
if $\mf{A}\in\ms{C}_{u}^{0}(H)$ then for any section $\mf{s}$ 
of the fibered
space $\lr{\pf{V}_{\mf{A}}}{[\cdot]_{\tm{A}},\pf{V}_{\mf{A}}^{\natural}}$
\begin{equation*}
\begin{aligned}
\mc{V}_{\natural}^{\mf{A}}&\in\prod_{\mf{p}\in\pf{V}_{\mf{A}}^{\natural}}
\prod_{\beta\in\ms{P}_{\mf{A}}^{\mf{p}}}
\mf{N}^{\mf{A}}(\ov{\ms{m}}^{\mf{A}}(\mf{s}_{\mf{p}})(\beta)),
\\
\mc{V}_{\natural}^{\mf{A}}(\mf{p})(\beta)
&\coloneqq
\mc{V}^{\mf{A}}(\mf{s}_{\mf{p}})(\beta),
\quad
\mf{p}\in\pf{V}_{\mf{A}}^{\natural},
\beta\in\ms{P}_{\mf{A}}^{\mf{p}};
\\
\mf{v}_{\natural}^{\mf{A}}
&\coloneqq(\un\mapsto\ms{gr}(\mc{V}_{\natural}^{\mf{A}})).
\end{aligned}
\end{equation*}
\end{definition}
According to Lemma \ref{07071040}(\ref{07071040st3}\,\&\,\ref{07071040st4}) and 
Lemma \ref{07071040b}(\ref{07071040bst3}\,\&\,\ref{07071040bst4}) we can set the following
\begin{definition}
\label{06071546}
Let $\#\in\{\dv,\star,\vm\}$, define $\Updelta_{o}^{\#}$
and 
$\mf{O}_{\#}$ 
be the maps on $Obj(\ms{C}_{u}(H))$, 
and
$\Updelta_{m}^{\#}$ 
be the map on
$ Mor_{\ms{C}_{u}(H)}$
such that
\begin{equation*}
\Updelta_{o}^{\#}(\mf{A})
\coloneqq
\bigcup_{\mf{p}\in\pf{T}_{\mf{A}}^{\#}}
Mor_{\ms{set}}(\ms{P}_{\mf{A}}^{\mf{p}},\ms{A}_{\#,\mf{A}}^{\ast}),
\end{equation*}
while for any $T\in Mor_{\ms{C}_{u}(H)}(\mf{B},\mf{A})$ 
\begin{equation*}
\begin{aligned}
\Updelta_{m}^{\#}(T)&:
\Updelta_{o}^{\#}(\mf{A})
\to
\Updelta_{o}^{\#}(\mf{B}),
\\
Mor_{\ms{set}}(\ms{P}_{\mf{A}}^{\mf{p}},\ms{A}_{\#,\mf{A}}^{\ast})
&\ni g
\mapsto
\left(
\ms{P}_{\mf{B}}^{\mf{d}_{\#}^{H}(T)(\mf{p})}
\ni\beta\mapsto
g(\beta)\circ\mf{g}_{\#}^{H}(T)
\right),
\,
\forall\mf{p}\in\pf{T}_{\mf{A}}^{\#}.
\end{aligned}
\end{equation*}
Moreover define 
\begin{equation*}
\begin{aligned}
\uppsi_{\#,\rhd}&\in\prod_{\mf{D}\in\ms{C}_{u}(H)}
Mor_{\ms{Gr}}(H, Aut_{\ms{set}}(\Updelta_{o}^{\#}(\mf{D})))
\\
\uppsi_{\#,\rhd}^{\mf{A}}(l)&: 
\Updelta_{o}^{\#}(\mf{A})
\to
\Updelta_{o}^{\#}(\mf{A}),
\\
Mor_{\ms{set}}(\ms{P}_{\mf{A}}^{\mf{p}},\ms{A}_{\#,\mf{A}}^{\ast})
&\ni g
\mapsto
\left(
\ms{P}_{\mf{A}}^{\mf{p}}
\ni\beta\mapsto
g(\beta)\circ\uppsi_{\#}^{\mf{A}}(l^{-1})
\right),
\,
\forall\mf{p}\in\pf{T}_{\mf{A}}^{\#},
l\in H;
\end{aligned}
\end{equation*}
finally
\begin{equation*}
\mf{O}_{\#}^{\mf{A}}
\coloneqq
(\un\mapsto\Updelta_{o}^{\#}(\mf{A}),
\uppsi_{\#,\rhd}^{\mf{A}}).
\end{equation*}
\end{definition}
\begin{lemma}
\label{06071817}
For all $\#\in\{\dv,\star,\vm\}$ we obtain that $\uppsi_{\#,\rhd}$ is well-defined and 
\begin{enumerate}
\item
$\mf{O}_{\#}^{\mf{A}}\in\ms{Fct}(H,\ms{set})$,
\label{06071817st1b}
\item
$(\Updelta_{o}^{\#},\Updelta_{m}^{\#})\in
\ms{Fct}(\ms{C}_{u}(H)^{op},\ms{set})$.
\label{06071817st2b}
\end{enumerate}
\end{lemma}
\begin{proof}
The case $\#=\dv$ follows since Lemma \ref{07071040}(\ref{07071040st3}\,\&\,\ref{07071040st4}),
while the cases $\#\in\{\star,\vm\}$ follow since Lemma \ref{07071040b}(\ref{07071040bst3}\,\&\,\ref{07071040bst4}).
\end{proof}
\begin{definition}
Let $T\in Mor_{\ms{C}_{u}(H)}(\mf{B},\mf{A})$, $l\in H$, $\mf{T}\in\pf{T}_{\mf{A}}$ and 
$\alpha\in\ms{P}_{\mf{A}}^{\mf{T}}$.
We say that 
the hypothesis $\ms{E}(T,l,\mf{T},\alpha)$ holds true 
if
$(\ms{ad}\circ\uppi_{\alpha}\circ T\circ\ms{v}^{\mf{B}})(l)\up\mc{U}(\mf{H}_{\alpha})
=
(\ms{ad}\circ\uppi_{\alpha}\circ\ms{v}^{\mf{A}})(l)
\up\mc{U}(\mf{H}_{\alpha})$,
while 
the hypothesis $\ms{E}(T,l)$ holds true if the hypothesis 
$\ms{E}(T,l,\mf{I},\beta)$ holds true for all $\mf{I}\in\pf{T}_{\mf{A}}$ 
and $\beta\in\ms{P}_{\mf{A}}^{\mf{I}}$,
finally we say that the hypothesis $\ms{E}$ holds true if the hypothesis $\ms{E}(T,l)$ holds true for all $l\in H$ and
$T\in Mor_{\ms{C}_{u}^{0}(H)}$.
\end{definition}
\begin{lemma}
\label{07071822}
For all $l\in H$ and $T\in Mor_{\ms{C}_{u}(H)}(\mf{B},\mf{A})$
we have
\begin{enumerate}
\item
$\mf{b}_{\star}^{\mf{B}}(l) 
\circ
\mf{d}_{\star}^{H}(T)
=
\mf{d}_{\star}^{H}(T)
\circ
\mf{b}_{\star}^{\mf{A}}(l)$,
\label{07071822st1}
\item
$\ov{\mf{h}}^{H}(T)
\circ
\ov{\mf{c}}^{\mf{B}}(l)
=
\ov{\mf{c}}^{\mf{A}}(l)
\circ
\ov{\mf{h}}^{H}(T)$,
\label{07071822st2}
\item
$\mf{g}_{\star}^{H}(T)
\circ
\uppsi_{\star}^{\mf{B}}(l) 
=
\uppsi_{\star}^{\mf{A}}(l) 
\circ
\mf{g}_{\star}^{H}(T)$,
\label{07071822st3}
\item
$\uppsi_{\star,\rhd}^{\mf{B}}(l)
\circ
\Updelta_{m}^{\star}(T)
=
\Updelta_{m}^{\star}(T)
\circ
\uppsi_{\star,\rhd}^{\mf{A}}(l)$,
\label{07071822st4}
\item
if in addition the hypothesis $\ms{E}(T,l)$ holds true then we obtain
\begin{enumerate}
\item
$\mf{b}_{\vm}^{\mf{B}}(l) 
\circ
\mf{d}_{\vm}^{H}(T)
=
\mf{d}_{\vm}^{H}(T)
\circ
\mf{b}_{\vm}^{\mf{A}}(l)$,
\label{07071822st51}
\item
$\mf{g}_{\vm}^{H}(T)
\circ
\uppsi_{\vm}^{\mf{B}}(l) 
=
\uppsi_{\vm}^{\mf{A}}(l) 
\circ
\mf{g}_{\vm}^{H}(T)$,
\label{07071822st53}
\item
$\uppsi_{\vm,\rhd}^{\mf{B}}(l)
\circ
\Updelta_{m}^{\vm}(T)
=
\Updelta_{m}^{\vm}(T)
\circ
\uppsi_{\vm,\rhd}^{\mf{A}}(l)$.
\label{07071822st54}
\end{enumerate}
\label{07071822st5}
\end{enumerate}
\end{lemma}
\begin{proof}
Let $l\in H$, $T\in Mor_{\ms{C}_{u}(H)}(\mf{B},\mf{A})$
and $\mf{T}\in\pf{T}_{\mf{A}}^{\diamond}$
then
\begin{equation}
\label{08071534a}
\begin{aligned}
(\mf{d}^{H}(T)\circ\mf{b}^{\mf{A}}(l))(\mf{T})
&=
\lr{(\mc{T}^{l})^{T},\ps{\upmu}^{l},
(\pf{H}^{(\ms{v}^{\mf{A}},l)})^{T},\upzeta^{l}}
{f,\Upgamma_{\pf{H},\ms{v}^{\mf{A}}}^{l}}   
\\
(\mc{T}^{l})^{T}
&=
\lr{l\cdot h,\upxi,\beta_{c}}
{I,T_{\dagger}\circ\upeta^{\ast}(l)\circ\ps{\upomega}}
\\
(\pf{H}^{(\ms{v}^{\mf{A}},l)})^{T}&:
\ms{P}_{\mf{A}}^{\mf{T}}\ni\alpha\mapsto
\lr{\mf{H}_{\alpha},\uppi_{\alpha}\circ T}
{\uppi_{\alpha}(\ms{v}^{\mf{A}}(l))\Upomega_{\alpha}},
\\
\Upgamma_{\pf{H},\ms{v}^{\mf{A}}}^{l}&:
\ms{P}_{\mf{A}}^{\mf{T}}\ni\alpha\mapsto
\ms{ad}\left(\uppi_{\alpha}(\ms{v}^{\mf{A}}(l))\right)
(\Upgamma_{\alpha}),
\end{aligned}
\end{equation}
while
\begin{equation}
\label{08071534b}
\begin{aligned}
(\mf{b}^{\mf{B}}(l)\circ\mf{d}^{H}(T)
)(\mf{T})
&=
\lr{(\mc{T}^{T})^{l},\ps{\upmu}^{l},
(\pf{H}^{T})^{(\ms{v}^{\mf{B}},l)},\upzeta^{l}}
{f,\Upgamma_{\pf{H}^{T},\ms{v}^{\mf{B}}}^{l}}   
\\
(\mc{T}^{T})^{l}
&=
\lr{l\cdot h,\upxi,\beta_{c}}
{I,\uptheta^{\ast}(l)\circ T_{\dagger}\circ\ps{\upomega}}
\\
(\pf{H}^{T})^{(\ms{v}^{\mf{B}},l)}&:
\ms{P}_{\mf{A}}^{\mf{T}}\ni\alpha\mapsto
\lr{\mf{H}_{\alpha},\uppi_{\alpha}\circ T}
{(\uppi_{\alpha}\circ T)(\ms{v}^{\mf{B}}(l))\Upomega_{\alpha}},
\\
\Upgamma_{\pf{H}^{T},\ms{v}^{\mf{B}}}^{l}
&:
\ms{P}_{\mf{A}}^{\mf{T}}\ni\alpha\mapsto
\ms{ad}\left((\uppi_{\alpha}\circ T)
(\ms{v}^{\mf{B}}(l))\right)
(\Upgamma_{\alpha}),
\end{aligned}
\end{equation}
moreover $T$ is equivariant by hypothesis so 
$\uptheta^{\ast}(l)\circ T_{\dagger}=
T_{\dagger}\circ\upeta^{\ast}(l)$, thus
\begin{equation}
\label{08070838}
(\mc{T}^{T})^{l}=(\mc{T}^{l})^{T}.
\end{equation}
Let $\alpha\in\ms{P}_{\mf{A}}^{\mf{T}}$, thus
$(\pf{H}^{(\ms{v}^{\mf{A}},l)})_{\alpha}^{T}$
is a cyclic representation associated with 
$(T_{\dagger}\circ\upeta^{\ast}(l))(\ps{\upomega}_{\alpha})$
and
$(\pf{H}^{T})^{(\ms{v}^{\mf{B}},l)}_{\alpha}$
is a cyclic representation associated with 
the state
$(\uptheta^{\ast}(l)\circ T_{\dagger})
(\ps{\upomega}_{\alpha})
=
(T_{\dagger}\circ\upeta^{\ast}(l))
(\ps{\upomega}_{\alpha})$,
so there exists a unitary operator $V_{\alpha}$
on $\mf{H}_{\alpha}$
such that 
\begin{equation}
\label{08070852}
\begin{aligned}
\uppi_{\alpha}\circ T&=
\ms{ad}(V_{\alpha})
\circ(\uppi_{\alpha}\circ T),
\\
V_{\alpha}\uppi_{\alpha}(\ms{v}^{\mf{A}}(l))\Upomega_{\alpha}
&=
(\uppi_{\alpha}\circ T)(\ms{v}^{\mf{B}}(l))\Upomega_{\alpha}.
\end{aligned}
\end{equation}
Next for all $b\in\mc{B}$ we have
\begin{equation*}
\begin{aligned}
\ms{ad}\left(T(\ms{v}^{\mf{B}}(l))\right)
(T(b))
&=
(T\circ\uptheta(l))(b)
\\
&=
(\upeta(l)\circ T)(b)
=
(\ms{ad}(\ms{v}^{\mf{A}}(l))(T(b)),
\end{aligned}
\end{equation*}
but $T$ is surjective so
\begin{equation}
\label{16071848}
\ms{ad}\left(T(\ms{v}^{\mf{B}}(l))\right)
=
\ms{ad}(\ms{v}^{\mf{A}}(l)),
\end{equation}
in particular
\begin{equation}
\label{13071611}
\ms{ad}\left((\uppi_{\alpha}\circ T)(\ms{v}^{\mf{B}}(l))\right)
\up\uppi_{\alpha}(\mc{A})
=
\ms{ad}\left(\uppi_{\alpha}(\ms{v}^{\mf{A}}(l))
\right)\up\uppi_{\alpha}(\mc{A}),
\end{equation}
so since $\ms{ad}(W)$ is weakly continuous for any unitary 
operator $W$ on $\mf{H}_{\alpha}$ and 
$\uppi(\mc{A})''=\ov{\uppi_{\alpha}(\mc{A})}^{w}$
since the bicommutant theorem , we obtain
\begin{equation*}
\ms{ad}\left((\uppi_{\alpha}\circ T)(\ms{v}^{\mf{B}}(l))\right)
\up\uppi_{\alpha}(\mc{A})''
=
\ms{ad}\left(\uppi_{\alpha}(\ms{v}^{\mf{A}}(l))
\right)\up\uppi_{\alpha}(\mc{A})''.
\end{equation*}
Moreover $\Upgamma_{\alpha}\in\uppi_{\alpha}(\mc{A})''$
by hypothesis so since \eqref{08071534a}\,\&\,\eqref{08071534b}
\begin{equation}
\label{08071519}
\Upgamma_{\pf{H}^{T},\ms{v}^{\mf{B}},\alpha}^{l}
=
\Upgamma_{\pf{H},\ms{v}^{\mf{A}},\alpha}^{l}.
\end{equation}
Next $V_{\alpha}\in\uppi_{\alpha}(\mc{A})'$ since $T$ is surjective
and the first equality of \eqref{08070852}, therefore
\begin{equation}
\label{08071527}
\Upgamma_{\pf{H},\ms{v}^{\mf{A}},\alpha}^{l}
=
\ms{ad}(V_{\alpha})(\Upgamma_{\pf{H},\ms{v}^{\mf{A}},\alpha}^{l}).
\end{equation}
Finally since 
\eqref{08071534a}\,\&\,\eqref{08071534b}\,\&\,\eqref{08070838}\,\&\,
\eqref{08070852}\,\&\,\eqref{08071519}\,\&\,\eqref{08071527}
we obtain
\begin{equation*}
(\mf{d}^{H}(T)
\circ
\mf{b}^{\mf{A}}(l))(\mf{T})
\sm{B}
(\mf{b}^{\mf{B}}(l) 
\circ
\mf{d}^{H}(T))
(\mf{T}),
\end{equation*}
then st.\eqref{07071822st1} follows since 
Prp. \ref{09070854}\,\&\,\ref{05071049}.
St.\eqref{07071822st2} follows since 
Rmk. \ref{29061357},
\eqref{29061509} 
and 
Prp. \ref{11211031},
and the fact that $(\ms{K}_{0}(\cdot),(\cdot)_{\ast})\circ((\cdot)^{+},(\cdot)^{+})$ 
is a functor from $\ms{CA}^{\ast}$ to $\ms{Ab}$. 
In conclusion st.\eqref{07071822st3}
follows since st.\eqref{07071822st1}\,\&\,\eqref{07071822st2},
and st.\eqref{07071822st4} since st.\eqref{07071822st3}.
If the hypothesis $\ms{E}(T,l)$ holds true then since \eqref{13071611} and the bicommutant theorem we obtain
\begin{equation}
\label{13071611b}
\ms{ad}\left((\uppi_{\alpha}\circ T)(\ms{v}^{\mf{B}}(l))\right)
\up\mathbb{A}(\mf{A})_{\alpha}^{\mf{T}}
=
\ms{ad}\left(\uppi_{\alpha}(\ms{v}^{\mf{A}}(l))\right)
\up\mathbb{A}(\mf{A})_{\alpha}^{\mf{T}},
\end{equation}
then st.\eqref{07071822st51} follows under the same argument used to prove st.\eqref{07071822st1}.
St.\eqref{07071822st53}\,\&\,\eqref{07071822st54} follow since st.\eqref{07071822st51}\,\&\,\eqref{07071822st2}.
\end{proof}
\begin{definition}
\label{10071700}
Define
\begin{equation*}
\begin{aligned}
\ps{P}^{H}
&\coloneqq
\left(\mf{P}_{\star},\,
 Mor_{\ms{C}_{u}(H)^{op}}\ni T\mapsto(\un\mapsto\mf{Q}_{m}^{\star}(T))\right),
\\
\ps{O}^{H}
&\coloneqq
\left(\mf{O}_{\star},\,
 Mor_{\ms{C}_{u}(H)^{op}}\ni T\mapsto(\un\mapsto\Updelta_{m}^{\star}(T))\right),
\\
\ps{P}_{\natural}^{H}
&\coloneqq
\left(\mf{P}_{\natural},\,
 Mor_{\ms{C}_{u}^{0}(H)^{op}}\ni T\mapsto(\un\mapsto\mf{Q}_{m}^{\natural}(T))\right),
\\
\ps{Z}_{\natural}^{H}
&\coloneqq
\left(\mf{Z}_{\natural},\,
 Mor_{\ms{C}_{u}^{0}(H)^{op}}\ni T\mapsto(\un\mapsto\ms{Z}_{\natural}^{m}(T))\right).
\end{aligned}
\end{equation*}
\end{definition}
\begin{corollary}
\label{10071703}
We have 
\begin{enumerate}
\item
$\ps{P}^{H}\in\ms{Fct}(\ms{C}_{u}(H)^{op},\ms{Fct}(H,\ms{set}))$;
\label{10071703st1}
\item
$\ps{O}^{H}\in\ms{Fct}(\ms{C}_{u}(H)^{op},\ms{Fct}(H,\ms{set}))$;
\label{10071703st2}
\item
If the hypothesis $\ms{E}$ holds true then
\begin{enumerate}
\item
$\ps{P}_{\natural}^{H}\in\ms{Fct}(\ms{C}_{u}^{0}(H)^{op},\ms{Fct}(H,\ms{set}))$,
\label{10071703st3}
\item
$\ps{Z}_{\natural}^{H}\in\ms{Fct}(\ms{C}_{u}^{0}(H)^{op},\ms{Fct}(H,\ms{set}))$.
\label{10071703st4}
\end{enumerate}
\end{enumerate}
\end{corollary}
\begin{proof}
Let 
$T\in Mor_{\ms{C}_{u}(H)}(\mf{B},\mf{A})$.
Since Lemma \ref{07071040b}\eqref{07071040bst1}\,\&\,\eqref{07071040bst2}
to prove st.\eqref{10071703st1} 
it is sufficient that
\begin{equation*}
(\un\mapsto\mf{Q}_{m}^{\star}(T))
\in Mor_{\ms{Fct}(H,\ms{set})}
(\mf{P}_{\star}^{\mf{A}},\mf{P}_{\star}^{\mf{B}}),
\end{equation*}
i.e. $\un\mapsto\mf{Q}_{m}^{\star}(T)$ is 
a natural transformation from the functor 
$\mf{P}_{\star}^{\mf{A}}$ to $\mf{P}_{\star}^{\mf{B}}$,
which is true since 
Lemma \ref{07071822}\eqref{07071822st1}.
Since Lemma \ref{06071817}\eqref{06071817st1b}\,\&\,\eqref{06071817st2b}
to prove st.\eqref{10071703st2} 
it is sufficient that
\begin{equation*}
(\un\mapsto\Updelta_{m}^{\star}(T))
\in Mor_{\ms{Fct}(H,\ms{set})}
(\mf{O}_{\star}^{\mf{A}},\mf{O}_{\star}^{\mf{B}}),
\end{equation*}
which is true since Lemma \ref{07071822}\eqref{07071822st4}.
Assume $\mf{A},\mf{B}\in\ms{C}_{u}^{0}(H)$ and $T\in Mor_{\ms{C}_{u}^{0}(H)}(\mf{B},\mf{A})$.
Since Lemma \ref{07071040b}\eqref{13071854a}\,\&\,\eqref{13071854b} 
to prove st.\eqref{10071703st3} it is sufficient that
\begin{equation*}
(\un\mapsto\mf{Q}_{m}^{\natural}(T))
\in Mor_{\ms{Fct}(H,\ms{set})}
(\mf{P}_{\natural}^{\mf{A}},\mf{P}_{\natural}^{\mf{B}}),
\end{equation*}
which is true since Lemma \ref{07071822}\eqref{07071822st51} and Rmk. \ref{09071514}.
Since Lemma \ref{07071040}\eqref{07071040st1z}\,\&\,\eqref{07071040st2z} 
to prove st.\eqref{10071703st4} it is sufficient that
\begin{equation*}
(\un\mapsto\ms{Z}_{\natural}^{m}(T))
\in Mor_{\ms{Fct}(H,\ms{set})}
(\mf{Z}_{\natural}^{\mf{A}},\mf{Z}_{\natural}^{\mf{B}}),
\end{equation*}
which is true since Lemma \ref{07071822}\eqref{07071822st51} and Rmk. \ref{09071514},
and since $T$ is equivariant i.e. $T\circ\uptheta(l)=\upeta(l)\circ T$, for all $l\in H$.
\end{proof}
\begin{remark}
\label{07281923}
Let $\Gamma_{\mf{D}}$ denote the space of sections of the fibered space 
$\lr{\pf{V}_{\mf{D}}}{[\cdot]_{\tm{D}},\pf{V}_{\mf{D}}^{\natural}}$, for any 
$\mf{D}\in\ms{C}_{u}^{0}(H)$, 
then for any $\mf{s}\in\prod_{\mf{D}\in\ms{C}_{u}^{0}(H)}\Gamma_{\mf{D}}$
since Def. \ref{06071514} we have that
\begin{equation*}
\begin{aligned}
\mc{V}_{\natural}
&\in
\prod_{\mf{D}\in\ms{C}_{u}^{0}(H)}
\prod_{\mf{p}\in\pf{V}_{\mf{D}}^{\natural}}
\prod_{\beta\in\ms{P}_{\mf{D}}^{\mf{p}}}
\mf{N}^{\mf{D}}(\ov{\ms{m}}^{\mf{D}}(\mf{s}_{\mf{p}}^{\mf{D}})(\beta)),
\\
\mf{v}_{\natural}^{\mf{D}}
&=
(\un\mapsto\ms{gr}(\mc{V}_{\natural}^{\mf{D}})),\forall\mf{D}\in\ms{C}_{u}^{0}(H).
\end{aligned}
\end{equation*}
\end{remark}
Now we are able to state in the second main result of this work
that $\mf{m}_{\star}$ and, under the hypothesis $\ms{E}$, also $\mf{v}_{\natural}$ 
are natural transformations.
Since Rmk. \ref{07281923} 
this result represents the compact equivariant form of the universality claim 
described in introduction \ref{introI},
whose invariant form will be established in Thm. \ref{10061606ante}.
\begin{theorem}
[Compact equivariant form of the universality claim]
\label{10071930}
We have
\begin{enumerate}
\item
\begin{equation*}
\mf{m}_{\star}\in 
Mor_{\ms{Fct}(\ms{C}_{u}(H)^{op},\ms{Fct}(H,\ms{set}))}
(\ps{P}^{H},\ps{O}^{H}),
\end{equation*}
\label{10071930st1}
\item
if hypothesis $\ms{E}$ holds true then
\begin{equation*}
\mf{v}_{\natural}\in
 Mor_{\ms{Fct}(\ms{C}_{u}^{0}(H)^{op},\ms{Fct}(H,\ms{set}))}
(\ps{P}_{\natural}^{H},\ps{Z}_{\natural}^{H})
\end{equation*}
\label{10071930st2}
\end{enumerate}
\end{theorem}
\begin{proof}
St.\eqref{10071930st1} is well-set since Cor. \ref{10071703},
moreover it ammounts to be equivalent to the claimed statements \eqref{10071931a}\,\&\,\eqref{11070655a}, where
\begin{equation}
\label{10071931a}
(\un\mapsto\ov{\ms{m}}_{\star}^{\mf{A}})
\in Mor_{\ms{Fct}(H,\ms{set})}(\mf{P}_{\star}^{\mf{A}},\mf{O}_{\star}^{\mf{A}}),
\end{equation}
and for all $T\in Mor_{\ms{C}_{u}(H)}(\mf{B},\mf{A})$ 
the following is a commutative diagram in the category $\ms{Fct}(H,\ms{set})$
\begin{equation}
\label{11070655a}
\xymatrix{
\mf{P}_{\star}^{\mf{A}}
\ar[rr]^{\un\mapsto\ov{\ms{m}}_{\star}^{\mf{A}}}
\ar[dd]_{\un\mapsto\mf{Q}_{m}^{\star}(T)}
&
&
\mf{O}_{\star}^{\mf{A}}
\ar[dd]^{\un\mapsto\Updelta_{m}^{\star}(T)}
\\
&&
\\
\mf{P}_{\star}^{\mf{B}}
\ar[rr]_{\un\mapsto\ov{\ms{m}}_{\star}^{\mf{B}}}
&&
\mf{O}_{\star}^{\mf{B}}.}
\end{equation}
Next \eqref{10071931a} equivales to the commutativity in $\ms{set}$ of the following diagram for all $l\in H$
\begin{equation}
\label{10071931b}
\xymatrix{
\pf{T}_{\mf{A}}^{\star}
\ar[rr]^{\ov{\ms{m}}_{\star}^{\mf{A}}}
\ar[dd]_{\mf{b}_{\star}^{\mf{A}}(l)}
&&
\Updelta_{o}^{\star}(\mf{A})
\ar[dd]^{\uppsi_{\star,\rhd}^{\mf{A}}(l)}
\\
&&
\\
\pf{T}_{\mf{A}}^{\star}
\ar[rr]_{\ov{\ms{m}}_{\star}^{\mf{A}}}
&&
\Updelta_{o}^{\star}(\mf{A}).}
\end{equation}
Let $\mf{p}\in\pf{T}_{\mf{A}}^{\star}$, $\mf{I}\in\mf{p}$ and $\beta\in\ms{P}_{\mf{A}}^{\mf{p}}$.
Then
\begin{equation*}
\begin{aligned}
(\uppsi_{\star,\rhd}^{\mf{A}}(l)\circ\ov{\ms{m}}_{\star}^{\mf{A}})(\mf{p})(\beta)
&=\ov{\ms{m}}_{\star}^{\mf{A}}(\mf{p})(\beta)\circ\uppsi_{\star}^{\mf{A}}(l^{-1})\\
&=\ov{\ms{m}}^{\mf{A}}(\mf{I})(\beta)\circ\ms{r}_{\star}^{\mf{A}}\circ\uppsi_{\star}^{\mf{A}}(l^{-1})\\
&=\ov{\ms{m}}^{\mf{A}}(\mf{I})(\beta)\circ\uppsi^{\mf{A}}(l^{-1})\circ\ms{r}_{\star}^{\mf{A}}\\ 
&=\ov{\ms{m}}^{\mf{A}}(\mf{b}^{\mf{A}}(l)\mf{I})(\beta)\circ\ms{r}_{\star}^{\mf{A}}\\
&=(\ov{\ms{m}}_{\star}^{\mf{A}}\circ\mf{b}_{\star}^{\mf{A}}(l))(\mf{p})(\beta),
\end{aligned}
\end{equation*}
where the third equality follows since Lemma \ref{07071040b}\eqref{07071040bst6} 
and the fourth one since 
Thm. \ref{main}\eqref{mainst2}.
Hence \eqref{10071931b} and so also \eqref{10071931a} follow.
Next \eqref{11070655a} equivales to the commutativity in $\ms{set}$ of the following diagram
\begin{equation}
\label{11070655b}
\xymatrix{
\pf{T}_{\mf{A}}^{\star}
\ar[rr]^{\ov{\ms{m}}_{\star}^{\mf{A}}}
\ar[dd]_{\mf{d}_{\star}^{H}(T)}
&&
\Updelta_{o}^{\star}(\mf{A})
\ar[dd]^{\Updelta_{m}^{\star}(T)}
\\
&&
\\
\pf{T}_{\mf{B}}^{\star}
\ar[rr]^{\ov{\ms{m}}_{\star}^{\mf{B}}}
&&
\Updelta_{o}^{\star}(\mf{B}).}
\end{equation}
Next 
\begin{equation*}
\begin{aligned}
(\ov{\ms{m}}_{\star}^{\mf{B}}\circ\mf{d}_{\star}^{H}(T))(\mf{p})(\beta)&=\ov{\ms{m}}_{\star}^{\mf{B}}([\mf{d}^{H}(T)\mf{I}]_{\af{B}})(\beta)
\\
&=\ov{\ms{m}}^{\mf{B}}(\mf{d}^{H}(T)\mf{I})(\beta)\circ\ms{r}_{\star}^{\mf{B}}\\
&=\ov{\ms{m}}^{\mf{A}}(\mf{I})(\beta)\circ\mf{g}^{H}(T)\circ\ms{r}_{\star}^{\mf{B}}\\
&=\ov{\ms{m}}^{\mf{A}}(\mf{I})(\beta)\circ\ms{r}_{\star}^{\mf{A}}\circ\mf{g}_{\star}^{H}(T)\\
&=\ov{\ms{m}}_{\star}^{\mf{A}}(\mf{p})(\beta)\circ\mf{g}_{\star}^{H}(T)\\
&=(\Updelta_{m}^{\star}(T)\circ\ov{\ms{m}}_{\star}^{\mf{A}})(\mf{p})(\beta),
\end{aligned}
\end{equation*}
where the third equality follows since the second equality in 
\eqref{12011707} 
(in switching $\mf{A}$ with $\mf{B}$),
while the fourth one since Lemma \ref{07071040b}\eqref{07071040bst7}. 
Therefore \eqref{11070655b} and \eqref{11070655a} and then st.\eqref{10071930st1} follow. 
Up to the end of the present proof we assume that 
$\mf{A},\mf{B}\in\ms{C}_{u}^{0}(H)$.
St.\eqref{10071930st2} is well-set since Cor. \ref{10071703}, moreover it ammounts to be equivalent to the 
claimed statements \eqref{10071931az}\,\&\,\eqref{11070655az}, where
\begin{equation}
\label{10071931az}
(\un\mapsto\ms{gr}(\mc{V}_{\natural}^{\mf{A}}))
\in Mor_{\ms{Fct}(H,\ms{set})}(\mf{P}_{\natural}^{\mf{A}},\mf{Z}_{\natural}^{\mf{A}}),
\end{equation}
and for all 
$T\in Mor_{\ms{C}_{u}^{0}(H)}(\mf{B},\mf{A})$ the following is a commutative diagram in the category $\ms{Fct}(H,\ms{set})$
\begin{equation}
\label{11070655az}
\xymatrix{
\mf{P}_{\natural}^{\mf{A}}
\ar[rr]^{\un\mapsto\ms{gr}(\mc{V}_{\natural}^{\mf{A}})}
\ar[dd]_{\un\mapsto\mf{Q}_{m}^{\natural}(T)}
&
&
\mf{Z}_{\natural}^{\mf{A}}
\ar[dd]^{\un\mapsto\ms{Z}_{\natural}^{m}(T)}
\\
&&
\\
\mf{P}_{\natural}^{\mf{B}}
\ar[rr]_{\un\mapsto\ms{gr}(\mc{V}_{\natural}^{\mf{B}})}
&&
\mf{Z}_{\natural}^{\mf{B}}.}
\end{equation}
Next \eqref{10071931az} equivales to the commutativity in $\ms{set}$ of the following diagram for all $l\in H$
\begin{equation}
\label{10071931bz}
\xymatrix{
\pf{V}_{\mf{A}}^{\natural}
\ar[rr]^{\ms{gr}(\mc{V}_{\natural}^{\mf{A}})}
\ar[dd]_{\mf{b}_{\natural}^{\mf{A}}(l)}
&&
\ms{Z}_{\natural}(\mf{A})
\ar[dd]^{\ms{V}_{\natural}(\mf{A})(l)}
\\
&&
\\
\pf{V}_{\mf{A}}^{\natural}
\ar[rr]_{\ms{gr}(\mc{V}_{\natural}^{\mf{A}})}
&&
\ms{Z}_{\natural}(\mf{A}),}
\end{equation}
i.e. 
$\upeta^{\ast}(l)\circ\mc{V}_{\natural}^{\mf{A}}(\mf{p})=(\mc{V}_{\natural}^{\mf{A}}\circ\mf{b}_{\natural}^{\mf{A}}(l))(\mf{p})$,
for all $\mf{p}\in\pf{V}_{\mf{A}}^{\natural}$,
which follows since 
Cor. \ref{19061936}\eqref{19061936st3}, 
hence \eqref{10071931az} is proved.
Next \eqref{11070655az} equivales to the commutativity in $\ms{set}$ of the following diagram
\begin{equation}
\label{11070655bz}
\xymatrix{
\pf{V}_{\mf{A}}^{\natural}
\ar[rr]^{\ms{gr}(\mc{V}_{\natural}^{\mf{A}})}
\ar[dd]_{\mf{d}_{\natural}^{H}(T)}
&&
\ms{Z}_{\natural}(\mf{A})
\ar[dd]^{\ms{Z}_{\natural}^{m}(T)}
\\
&&
\\
\pf{V}_{\mf{B}}^{\natural}
\ar[rr]_{\ms{gr}(\mc{V}_{\natural}^{\mf{B}})}
&&
\ms{Z}_{\natural}(\mf{B}),}
\end{equation}
which follows since 
Thm. \ref{12051936}\eqref{11061441}, 
therefore \eqref{11070655az} as well st.\eqref{10071930st2} follow.
\end{proof}
\begin{remark}
\label{11071826}
Notice that Thm. \ref{10071930}\eqref{10071930st1}
is equivalent to 
Cor. \ref{19061936}(\ref{19061936st1},\ref{19061936st2}),
while
under the hypothesis $\ms{E}$ 
Thm. \ref{10071930}\eqref{10071930st2}
is equivalent to 
Cor. \ref{19061936}(\ref{19061936st4},\ref{19061936st3}).
Indeed Thm. \ref{10071930}\eqref{10071930st1} 
is equivalent to \eqref{10071931b} and \eqref{11070655b},
for all $\mf{A},\mf{B}\in\ms{C}_{u}(H)$, $l\in H$ and $T\in Mor_{\ms{C}_{u}(H)}(\mf{B},\mf{A})$,
while, in case the hypothesis $\ms{E}$ holds true,
Thm. \ref{10071930}\eqref{10071930st2} 
is equivalent to \eqref{10071931bz} and \eqref{11070655bz}
for all $\mf{A},\mf{B}\in\ms{C}_{u}^{0}(H)$, 
$l\in H$ and 
$T\in Mor_{\ms{C}_{u}^{0}(H)}(\mf{B},\mf{A})$.
\end{remark}
\begin{remark}
\label{16071741}
Let $\mc{X}=\lr{\mc{D},\mf{H}}{J,\mf{P}}$ be the canonical standard form of $\mc{A}$, 
$\upphi$ a faithful semi-finite normal weight on $\mc{A}$, 
$\lr{\uppi_{\upphi},\mf{H}_{\upphi}}{\eta_{\upphi}}$ 
the semi-cyclic representation of $\mc{A}$ associated with $\upphi$, 
\cite[Def. $7.1.5$]{tak2}, and 
$\mc{Y}=\lr{\uppi_{\upphi}(\mc{A}),\mf{H}_{\upphi}}{J_{\upphi},\mf{P}_{\upphi}}$ the standard form associated 
to $\upphi$ realized on $\mf{H}$, 
thus $\mc{X}=\mc{Y}$ \cite[p. $153$]{tak2}; 
finally let $\Delta$ denote the modular operator associated with $\upphi$.
Since the construction of $\ms{v}^{\mf{A}}$ and since 
\cite[Lemma $6.1.5(vi)$, Thm. $7.2.6$ and Thm. $9.1.15$]{tak2} we deduce that the 
following holds
\begin{itemize}
\item
$\ms{v}^{\mf{A}}(l)\mf{P}=\mf{P}$;
\item
$\ms{v}^{\mf{A}}(l)\,J\,\ms{v}^{\mf{A}}(l^{-1})\Delta^{\frac{1}{2}}\eta_{\upphi}(x)=\eta_{\upphi}(x^{\ast})$,
for all $x\in\mf{n}_{\upphi}\cap\mf{n}_{\upphi}^{\ast}$.
\end{itemize}
Next let $\mf{A},\mf{B}\in\ms{C}_{u}(H)$ and $T\in Mor_{\ms{C}_{u}(H)}(\mf{B},\mf{A})$
and let $\mc{A}$ and $\mc{B}$ denote the underlying von Neumann algebras of 
$\mf{A}$ and $\mf{B}$ respectively, 
thus the hypothesis $\ms{E}(T,l)$ holds true if 
\begin{enumerate}
\item
$T(\ms{v}^{\mf{B}}(l))\mf{P}=\mf{P}$;
\label{16071741st1b}
\item
$T(\ms{v}^{\mf{B}}(l))\,J\,T(\ms{v}^{\mf{B}}(l^{-1}))\Delta^{\frac{1}{2}}\eta_{\upphi}(x)=\eta_{\upphi}(x^{\ast})$,
for all $x\in\mf{n}_{\upphi}\cap\mf{n}_{\upphi}^{\ast}$.
\label{16071741st2b}
\end{enumerate}
Note that the requests 
$\ms{ad}(\ms{v}^{\mf{A}}(l))\mc{A}=\mc{A}$ and $\ms{ad}(T(\ms{v}^{\mf{B}}(l)))\mc{A}=\mc{A}$ 
are automatically satisfied 
since $\ms{v}^{\mf{A}}(l)\in\mc{U}(\mc{A})$ and $\ms{v}^{\mf{B}}(l)\in\mc{U}(\mc{B})$ by construction.
Now \eqref{16071741st1b} should be not difficult to show, while with the help of \eqref{16071848} 
it should be possible to prove
\eqref{16071741st2b}, but up to now we did not succeed. 
\end{remark}
\section{Universality of the global Terrell law}
\label{binfis0}
Any fissioning system 
$U^{233}+n_{th}$, $U^{235}+n_{th}$, $Pu^{239}+n_{th}$ and $Cf^{252}$, 
below referred as $as-$type, 
exhibits an asymmetric binary fission consisting in obtaining two asymmetric final fragments. 
Here the asymmetry of the fragments is 
with respect to their mass numbers in unified atomic mass units.
Let $A_{H}$ and $A_{L}$ be the mass numbers of the heavy and light fragments respectively.
In addition any fissioning system up to $Cf^{252}$ for example $Fm^{258}$ and $Hs^{266}$, below referred as $s-$type, 
exhibits a symmetric binary fission consisting in obtaining two symmetric final fragments.
The nucleon phase hypothesis advanced by Mouze and Ythier 
states the following, see \cite{mhy1,mhy2,ric} 
for details. 
The reaction-time of any binary fission process is $1.77\,10^{-25}s$ 
thus occurring at temperatures of the order of 
$10^{13}K$, according to the energy-time uncertainty relation.
Hence the distinction between the proton and neutron phase disappears and a new nucleon phase occurs.
More exactly whenever a fission process involves an $as-$type fissioning system $\varsigma$, 
two nucleon cores come into existence, 
one of mass number $82$ and the other of mass number $126$.
The two final asymmetric fragments consist by these two cores surrounded by their valence nucleons, 
and the closure of the shells 
at $82$ and $126$ explains the following Terrell law \cite[eq. $(1)$]{ric}
\begin{equation}
\label{26061108}
\ov{\nu}=0.08\,(A_{L}-82)+0.1\,(A_{H}-126),
\end{equation}
where $\ov{\nu}$ is the 
mean value of the prompt-neutron yield in the state describing the fragments
obtained next the fission process occurs to $\varsigma$,
in what follows simply called prompt-neutron yield of the fission process occurring to $\varsigma$.
Similarly the reaction involving a $s-$type fissioning system generates two nucleon cores each one of mass number $126$, 
and the symmetric final fragments consist of such a cores surrounded by their valence nucleons, \cite[II.a]{ric}.
\par
From the information extracted by the experiment we deduce the following stability,
the values $82$ and $126$ 
remain unchanged under variation of the generating fissioning system.
Moreover the hypothesis of existence of nucleon cores implies the
thermal nature and phase transition of the nucleon phase: 
the fission process activates and the cores are generated 
at temperatures higher than $10^{13}K$.
\par
The universality claim described in introduction \ref{introI}
has beeen resolved in its equivariant form in
Cor. \ref{19061936}.
Here we shall resolve the invariant form of the claim which establishes
the \emph{invariance} of  
the global light and heavy nucleon and fragment masses, 
and the global Terrell law 
under contravariant action over the field of fission processes
of the subcategory $\ms{C}_{u}^{0}(H)$ of fissioning systems and their 
transformations, and under covariant action over the field of fission processes
of the symmetry group $H$.
As a result the (restricted global) 
light and heavy nucleon masses as well the prompt-neutron yield 
are invariant
under essential contravariant action of $\ms{C}_{u}^{0}(H)$ 
and under covariant action of $H$.
Finally under the revised nucleon phase hypothesis the stability
of the nucleon masses at the values $82$ and $126$ and the invariance
of \eqref{26061108} under action of the above transformations 
follow.
It ammounts to apply the below
$\mc{T}-$resolution of the invariant form of the universality claim,
to the case where $\mc{T}$ equals $\mc{T}_{\bullet}$.
Let $\#_{m}=$light and $\#_{w}=$heavy.
\par
\emph
{For any nucleon-fragment doublet $\mc{T}$ on $\mf{C}$,
there exist 
sections 
$\muup_{j}^{\mc{T}}$, $\uplambda_{j}^{\mc{T}}$, $j\in\{m,w\}$ and $\uptheta^{\mc{T}}$
all invariant under contravariant action of $\Upxi(\pf{D},\mc{F})$
and such that each of their values induces 
an invariant section under covariant action of $H$.
There exists a field 
$\mc{P}^{\mc{T}}=(\mc{P}_{o}^{\mc{T}},\mc{P}_{m}^{\mc{T}})$ 
contravariant under action of 
the subcategory $\Upxi(\pf{D},\mc{F})$ of fissioning systems and their transformations,
where
$\mc{P}_{o}^{\mc{T}}(\ms{a})$
is the set of subsets 
of fission processes occurring to the fissioning system $\ms{a}$,
for any object $\ms{a}$ of $\Upxi(\pf{D},\mc{F})$.
The maps
$(\muup_{j}^{\mc{T}})_{\ms{a}}$,
$(\uplambda_{j}^{\mc{T}})_{\ms{a}}$
and
$\uptheta_{\ms{a}}^{\mc{T}}$ 
associate with each 
$Y$ in $\mc{P}_{o}^{\mc{T}}(\ms{a})$
the set of the 
$\mc{T}-\#_{j}$ nucleon masses,  
$\mc{T}-\#_{j}$ fragment masses,  
and
the prompt-neutron yields of the fission precesses belonging to $Y$
respectively.
As a result of 
Prp. \ref{20051810dbt}
the universality of 
the $\mc{T}-\#_{j}$ nucleon and fragment masses,  
and the universality of the $\mc{T}-$Terrell law 
establishes what follows.
$\muup_{j}^{\mc{T}}$,
$\uplambda_{j}^{\mc{T}}$
and
$\uptheta^{\mc{T}}$ 
are natural transformations whose source functor is 
$\mc{P}^{\mc{T}}$ and the morphism map of their common target functor is the constant map with constant value 
equal to the identity map on the power set of $\R$,
moreover
$\un\mapsto(\muup_{j}^{\mc{T}})_{\ms{a}}$,
$\un\mapsto(\uplambda_{j}^{\mc{T}})_{\ms{a}}$
and 
$\un\mapsto\uptheta_{\ms{a}}^{\mc{T}}$ 
are morphisms in the category of functors from $H$
to $\ms{set}$. 
In other words 
\begin{equation}
\label{07191046}
\begin{aligned}
\uptheta_{d(T)}^{\mc{T}}
\circ
\mc{P}_{m}^{\mc{T}}(T)
&=
\uptheta_{c(T)}^{\mc{T}},
\forall T\in Mor_{\Upxi(\pf{D},\mc{F})},
\\
\uptheta_{\ms{a}}^{\mc{T}}
\circ
\uptau_{\ms{a}}^{\mc{T}}(l)
&=
\uptheta_{\ms{a}}^{\mc{T}},
\forall\ms{a}\in\Upxi(\pf{D},\mc{F}),\, l\in H,
\end{aligned}
\end{equation}
similarly for 
$\muup_{j}^{\mc{T}}$
and
$\uplambda_{j}^{\mc{T}}$.
The contravariance of $\mc{P}^{\mc{T}}$ under action of $\Upxi(\pf{D},\mc{F})$ means that
for all $T,S\in Mor_{\Upxi(\pf{D},\mc{F})}$ such that $d(T)=c(S)$ 
\begin{equation*}
\begin{aligned}
\mc{P}_{m}^{\mc{T}}(T)
\mc{P}_{o}^{\mc{T}}(c(T))
&\subseteq
\mc{P}_{o}^{\mc{T}}(d(T)),
\\
\mc{P}_{m}^{\mc{T}}(T\circ S)
&=
\mc{P}_{m}^{\mc{T}}(S)
\circ
\mc{P}_{m}^{\mc{T}}(T).
\end{aligned}
\end{equation*}}
\par
Next we briefly sketch
for the canonical case what above outlined. 
Let $\ms{N}_{as}^{\mc{T}_{\bullet}}$ be the set of the
\begin{equation*}
\mf{n}=\lr{\mf{A}}{\mf{T},\alpha,\{\ms{f}_{j},N_{j}\}_{j\in\{m,w\}}},
\end{equation*}
such that 
$\mf{A}=\lr{\mc{A},H}{\upeta}$ is an object of the category $\ms{C}_{u}^{0}(H)$,
$\mf{T}\in\pf{V}(\mf{A})$, $\alpha\in\ms{P}_{\mc{A}}^{\mf{T}}\cap\R_{0}^{+}$, 
$N_{j}\in\mc{A}_{+}$ and $\ms{f}_{j}\in\ms{A}_{\mf{A}}$, satisfying consistent physical requests.
$\ms{N}_{as}^{\mc{T}_{\bullet}}$ models the set of all the possible asymmetric fission 
processes, later called simply fission processes, in the following sense.
For any $\mf{n}\in\ms{N}_{as}^{\mc{T}_{\bullet}}$ as above we have that
\begin{enumerate}
\item
$\mf{A}$ is the fissioning system to which the asymmetric fission process $\mf{n}$ occurs,
\item
$\mf{G}^{H}(\mf{A})$ 
is the nucleon system generated by the fissioning system $\mf{A}$,
\item
$\mc{O}(\mf{A})_{\alpha}^{\mf{T}}$ 
is the fragment system 
whose observable algebra 
is $\mc{A}$ and whose dynamics is $(\ep^{\mf{A}})_{\alpha}^{\mf{T}}(-\alpha^{-1}(\cdot))$,
\item
$\ps{\upvarphi}_{\alpha}^{\mf{T}}$ is the state 
of thermal equilibrium at inverse temperature $\alpha$
of $\mc{O}(\mf{A})_{\alpha}^{\mf{T}}$,
\item
$\mf{T}$ is the operation realizing the fission process $\mf{n}$
whenever performed on $\ps{\upvarphi}_{\alpha}^{\mf{T}}$,
\item 
$\ov{\ms{m}}^{\mf{A}}(\mf{T},\alpha)$ is the phase of $\mf{G}^{H}(\mf{A})$,
occurring by performing $\mf{T}$ on $\ps{\upvarphi}_{\alpha}^{\mf{T}}$,
\item $\mc{V}^{\mf{A}}(\mf{T},\alpha)$ is the state of $\mc{O}(\mf{A})_{\alpha}^{\mf{T}}$ 
\emph{originated} via $\ov{\ms{m}}^{\mf{A}}(\mf{T},\alpha)$, 
\item 
$N_{j}$ is the observable of $\mc{O}(\mf{A})_{\alpha}^{\mf{T}}$ relative to
the mass of the $\#_{j}$ fragment,
\item
$\ms{f}_{j}$ is the observable of $\mf{G}^{H}(\mf{A})$ relative to
the mass of the $\#_{j}$ nucleon core,
\item
$\upkappa_{j}^{\mc{T}_{\bullet}}(\mf{n})
\coloneqq
\mc{V}^{\mf{A}}(\mf{T},\alpha)(N_{j})$
is the mean value in $\mc{V}^{\mf{A}}(\mf{T},\alpha)$
of the mass of the $\#_{j}$ fragment,
\item 
$\upzeta_{j}^{\mc{T}_{\bullet}}(\mf{n})
\coloneqq
\ov{\ms{m}}^{\mf{A}}(\mf{T},\alpha)(\ms{f}_{j})$
is the mean value in $\ov{\ms{m}}^{\mf{A}}(\mf{T},\alpha)$
of the mass of the $\#_{j}$ nucleon core.
\end{enumerate}
It is worthy remarking that for each fissioning system  
the nucleon phase as well the fragment state is one, 
while the light and heavy cases are ascribable to the observables 
$\ms{f}_{m}$ and $\ms{f}_{w}$ for the nucleon cores 
and $N_{m}$ and $N_{w}$ for the fragments.
Now the restricted global Terrell law 
generalizing \eqref{26061108} can be defined in the following way.
For any  $\mf{n}\in\ms{N}_{as}^{\mc{T}_{\bullet}}$ 
the 
\emph
{mean value of the prompt-neutron yield in $\mc{V}^{\mf{A}}(\mf{T},\alpha)$}, 
said also prompt-neutron yield of the fission process $\mf{n}$,
equals $\nuup^{\mc{T}_{\bullet}}(\mf{n})$ where
\begin{equation}
\label{08142057}
\nuup^{\mc{T}_{\bullet}}
\coloneqq
0.08(\upkappa_{m}^{\mc{T}_{\bullet}}-\upzeta_{m}^{\mc{T}_{\bullet}})
+
0.1(\upkappa_{w}^{\mc{T}_{\bullet}}-\upzeta_{w}^{\mc{T}_{\bullet}}).
\end{equation}
Since by construction fission processes depend by fissioning systems and operations,
it is natural to expect that the category $\ms{C}_{u}^{0}(H)^{op}$ and the group $H$
act on the set $\ms{N}_{as}^{\mc{T}_{\bullet}}$.
Indeed for any $l\in H$ let
$\mf{T}^{l}=\mf{b}^{\mf{A}}(l)(\mf{T})$, 
$\ms{f}_{j}^{l}=\uppsi^{\mf{A}}(l)(\ms{f}_{j})$ and $N_{j}^{l}=\upeta(l)(N_{j})$.
Moreover let $\mf{B}$ be a fissioning system in $\ms{C}_{u}^{0}(H)$,
$T\in Mor_{\ms{C}_{u}^{0}(H)}(\mf{B},\mf{A})$ such that there exists 
$\mf{x}=\{\ms{f}_{j}',N_{j}'\}_{j\in\{m,w\}}$,
where $\ms{f}_{j}'\in\ms{A}_{\mf{B}}$ and $N_{j}'\in\mc{B}_{+}$ 
with $\mc{B}$ the $C^{\ast}-$algebra underlying $\mf{B}$,
satisfying 
$\ms{f}_{j}=\mf{g}^{H}(T)(\ms{f}_{j}')$ and $N_{j}=T(N_{j}')$.
Let $\mf{T}^{T}=\mf{d}^{H}(T)\mf{T}$, define
\begin{equation}
\label{07191139}
\begin{aligned}
\mf{n}^{(T,\mf{x})}
&\coloneqq
\lr{\mf{B}}{\mf{T}^{T},\alpha,\{\ms{f}_{j}',N_{j}'\}_{j\in\{m,w\}}},
\\
\mf{n}^{l}
&\coloneqq
\lr{\mf{A}}{\mf{T}^{l},\alpha,\{\ms{f}_{j}^{l},N_{j}^{l}\}_{j\in\{m,w\}}}.
\end{aligned}
\end{equation}
Thus 
$\ms{C}_{u}^{0}(H)$ acts ``essentially'' contravariantly
and $H$ acts covariantly on the set of fission processes since 
\begin{equation*}
\mf{n}^{(T,\mf{x})}\in\ms{N}_{as}^{\mc{T}_{\bullet}},
\quad 
\mf{n}^{l}\in\ms{N}_{as}^{\mc{T}_{\bullet}}.
\end{equation*}
As we shall see later, 
in order to have a ``true'' contravariant action of $\ms{C}_{u}^{0}(H)$ we need to 
implement a power-set extension of the above transformations. 
Notice that
\begin{enumerate}
\item
$\mf{T}^{l}$ is the operation obtained by transforming $\mf{T}$ through the action of $l$, 
\item
$\ms{f}_{j}^{l}$ is the observable of $\mf{G}^{H}(\mf{A})$,
obtained by transforming through the action of $l$ the observable $\ms{f}_{j}$,
\item
$N_{j}^{l}$ is the observable of $\mc{O}(\mf{A})_{\alpha}^{\mf{T}}$ 
obtained by transforming through the action of $l$ the observable $N_{j}$,
\item
$\mf{T}^{T}$ is the operation obtained by transforming $\mf{T}$ through the action of $T$, 
\item
$\ms{f}_{j}$ is the observable of $\mf{G}^{H}(\mf{A})$,
obtained by transforming through the action of $T$ the observable $\ms{f}_{j}'$
of $\mf{G}^{H}(\mf{B})$,
\item
$N_{j}$ is the observable of $\mc{O}(\mf{A})_{\alpha}^{\mf{T}}$ 
obtained by transforming through the action of $T$ the observable $N_{j}'$
of $\mc{O}(\mf{B})_{\alpha}^{\mf{T}^{T}}$.
\end{enumerate}
Now by applying to $\mc{T}_{\bullet}$
the properties of invariance of doublets stated in 
Prp. \ref{20051810dbt}
we establish for any $j\in\{m,w\}$
\emph
{the invariance of the restricted global nucleon mass 
$\upzeta_{j}^{\mc{T}_{\bullet}}$,
the invariance of the restricted global fragment mass 
$\upkappa_{j}^{\mc{T}_{\bullet}}$
and the invariance of the restricted global Terrell law 
$\nuup^{\mc{T}_{\bullet}}$
under essential contravariant action of the category $\ms{C}_{u}^{0}(H)$
and under action of the group $H$}. 
More exactly 
\begin{equation}
\label{07251125}
\begin{aligned}
\upkappa_{j}^{\mc{T}_{\bullet}}(\mf{n}^{(T,\mf{x})})
&=
\upkappa_{j}^{\mc{T}_{\bullet}}(\mf{n}),
\\
\upkappa_{j}^{\mc{T}_{\bullet}}(\mf{n}^{l})
&=
\upkappa_{j}^{\mc{T}_{\bullet}}(\mf{n});
\\
\upzeta_{j}^{\mc{T}_{\bullet}}(\mf{n}^{(T,\mf{x})})
&=
\upzeta_{j}^{\mc{T}_{\bullet}}(\mf{n}),
\\
\upzeta_{j}^{\mc{T}_{\bullet}}(\mf{n}^{l})
&=
\upzeta_{j}^{\mc{T}_{\bullet}}(\mf{n}),
\end{aligned}
\end{equation}
hence
\begin{equation}
\label{07251126}
\begin{aligned}
\nuup^{\mc{T}_{\bullet}}(\mf{n}^{(T,\mf{x})})
&=
\nuup^{\mc{T}_{\bullet}}(\mf{n}),
\\
\nuup^{\mc{T}_{\bullet}}(\mf{n}^{l})
&=
\nuup^{\mc{T}_{\bullet}}(\mf{n}).
\end{aligned}
\end{equation}
The equalities in 
\eqref{07251125}
and
\eqref{07251126}
are equivalent to the universality of the
$\mc{T}-$nucleon and fragment masses,
and the $\mc{T}-$Terrell law 
respectively
stated in \eqref{07191046}
for the case $\mc{T}=\mc{T}_{\bullet}$,
if $\mc{P}_{o}^{\mc{T}_{\bullet}}(\mf{A})$ 
is the power set of 
$\ms{N}_{as}^{\mc{T}_{\bullet}}(\mf{A})$
the set of the fission processes whose underlying 
fissioning system is $\mf{A}$, 
$\mc{P}_{m}^{\mc{T}_{\bullet}}(T)$ 
extends and modifies the first map defined in \eqref{07191139}
by transforming the dependence by $\mf{x}$ into a set-valed map,
$(\muup_{j}^{\mc{T}_{\bullet}})_{\mf{A}}$,
$(\uplambda_{j}^{\mc{T}_{\bullet}})_{\mf{A}}$
and
$\uptheta_{\mf{A}}^{\mc{T}_{\bullet}}$
are the   $\mathscr{P}(\R)-$valued
extensions to 
$\mc{P}_{o}^{\mc{T}_{\bullet}}(\mf{A})$ of the 
restrictions of 
$\upzeta_{j}^{\mc{T}_{\bullet}}$,
$\upkappa_{j}^{\mc{T}_{\bullet}}$
and
$\upnu^{\mc{T}_{\bullet}}$
to 
$\ms{N}_{as}^{\mc{T}_{\bullet}}(\mf{A})$
respectively;
finally $\uptau_{\mf{A}}^{\mc{T}_{\bullet}}(l)$ 
is the extension
of the restriction of the second map in \eqref{07191139}
to 
$\ms{N}_{as}^{\mc{T}_{\bullet}}(\mf{A})$.
We remark that provided the hypothesis $\ms{E}$, 
the universality of the global nucleon phase 
and global fragment state
(namely the $\mc{T}_{\bullet}-$nucleon phase and $\mc{T}_{\bullet}-$fragment state)
and hence the universality of the 
global Terrell law follow
by the naturality of the transformations $\mf{m}_{\star}$ and $\mf{v}_{\natural}$
stated in Thm. \ref{10071930}.
\par
Now we call $\#_{j}$ nucleon and fragment masses, and Terrell law 
the restriction of 
$\upzeta_{j}^{\mc{T}_{\bullet}}$,
$\upkappa_{j}^{\mc{T}_{\bullet}}$
and
$\nuup^{\mc{T}_{\bullet}}$
respectively
to the set of all the fission processes $\mf{n}$ satisfying the 
revised nucleon phase hypothesis
requiring 
$\upzeta_{m}^{\mc{T}_{\bullet}}(\mf{n})=82$,
$\upzeta_{w}^{\mc{T}_{\bullet}}(\mf{n})=126$
and that $H$ would contain as a subgroup 
the direct product of the universal covering group of the Poincar\'{e} 
group with the gauge group of the standard model.
Thus under the revised nucleon phase hypothesis 
as a result of \eqref{07251125} and \eqref{07251126} we can state what follows.
The light and heavy nucleon masses are invariant with constant values $82$ and $126$
and the prompt-neutron yield \eqref{26061108} is invariant 
under essential contravariant action of 
$\ms{C}_{u}^{0}(H)$
and under action of relativistic transformations of reference frames
over the field of fission processes.
\par
In section \ref{binfis0} we assume fixed two locally compact topological groups $G$ and $F$,
a group homomorphism $\uprho:F\to Aut_{\ms{Gr}}(G)$  such that the map $(g,f)\mapsto\uprho_{f}(g)$
on $G\times F$ at values in $G$, is continuous, let $H$ denote $G\rtimes_{\uprho}F$.
\subsection{Invariance of the restricted $\mc{T}-$Terrell law}
\label{07011744}
In Def. \ref{q08062037} we define the restricted $\mc{T}-$nucleon and fragment masses
and the restricted $\mc{T}-$Terrell law 
for any nucleon-fragment doublet $\mc{T}$ on a category $\mf{C}$, 
and provide their invariance in Prp. \ref{11061659}, physically interpreted in Prp. \ref{11061851}.
In the present section \ref{07011744} let $\mf{C}$ be a category, 
and $\mc{T}$ be $\lr{S,J,\mc{Z},\mc{S},\mc{F}}{\mf{m},\mc{W},R,\pf{D},\pf{U}}$ 
a fixed but arbitrary nucleon-fragment doublet on $\mf{C}$ 
where for simplicity we assume $R$ to be the constant map with constant value equal to $H$.
\begin{convention}
\label{07161437}
For any $x\in\Upxi(\pf{D},\mc{F})$, $\mf{T}\in\pf{D}_{\mc{F}(x)}$, 
$\alpha\in\ms{P}_{\mc{F}(x)}^{\mf{T}}$ 
set $\mc{A}(x,\mf{T},\alpha)\coloneqq\mc{A}(\mc{F}(x))_{\alpha}^{\mf{T}}$,
and for any 
$T\in Mor_{\Upxi(\pf{D},\mc{F})}$ and $\mf{Q}\in\pf{D}_{\mc{F}(c(T))}$, 
set
$\mf{Q}^{T}\coloneqq\mc{F}_{2}(T)\mf{Q}$.
\end{convention}
\begin{definition}
[Fission processes relative to $\mc{T}$]
\label{qnuclph1}
Let $\ms{N}_{as}^{\mc{T}}$ be 
the set of the asymmetric 
fission processes
relative to 
$\mc{T}$ defined as the set of $4-$tuples 
\begin{equation*}
\lr{\ms{a}}{\mf{T},\alpha,
\{\ms{f}_{j},N_{j}\}_{j\in\{m,w\}}},
\end{equation*}
such that for all $j\in\{m,w\}$
\begin{enumerate}
\item
$\ms{a}\in\Upxi(\pf{D},\mc{F})$,
\label{qnuclphrq2as}
\item
$\mf{T}\in\pf{D}_{\mc{F}(\ms{a})}$ 
and 
$\alpha\in\ms{P}_{\mc{F}(\ms{a})}^{\mf{T}}\cap\R_{0}^{+}$,
\label{qnuclphrq3as}
\item
$\ms{f}_{j}\in\ms{A}_{\mc{F}(\ms{a})}$,
\item
$N_{j}\in\mc{A}(\ms{a},\mf{T},\alpha)_{+}$,
\item
\begin{equation*}
\begin{aligned}
\mc{W}^{\ms{a}}(\mf{T},\alpha)(N_{j})
&\geq
\mf{m}^{\ms{a}}(\mf{T},\alpha)(\ms{f}_{j}),
\\
\mf{m}^{\ms{a}}(\mf{T},\alpha)(\ms{f}_{w})
&\geq
\mf{m}^{\ms{a}}(\mf{T},\alpha)(\ms{f}_{m})\geq 0,
\end{aligned}
\end{equation*}
\label{q08061831}
\item
for any $T\in Mor_{\Upxi(\pf{D},\mc{F})}$ such that $c(T)=\ms{a}$
\begin{equation*}
\begin{aligned}
\mc{F}_{1}(T)^{-1}(\ms{f}_{j})
&\neq
\varnothing,
\\
\mc{S}(T,\mf{T},\alpha)^{-1}(N_{j})
\cap
\mc{A}(d(T),\mf{T}^{T},\alpha)_{+}
&\neq
\varnothing.
\end{aligned}
\end{equation*}
\label{07201029}
\end{enumerate}
Let $\ms{N}^{\mc{T}}$ be the set of the 
fission processes 
relative to $\mc{T}$ defined as the set 
of the $3-$tuples 
\begin{equation*}
\lr{\ms{a}}{
\{\mf{T}_{i},\alpha_{i}\}_{i\in\{s,as\}},
\{\ms{f}_{j},N_{j}\}_{j\in\{m,w\}}},
\end{equation*}
such that
$\lr{\ms{a}}{\mf{T}_{as},\alpha_{as},
\{\ms{f}_{j},N_{j}\}_{j\in\{m,w\}}}\in\ms{N}_{as}^{\mc{T}}$,
$\mf{T}_{s}\in\pf{D}_{\mc{F}(\ms{a})}$, 
$\alpha_{s}\in\ms{P}_{\mc{F}(\ms{a})}^{\mf{T}_{s}}$
satisfying 
$\mc{A}(\ms{a},\mf{T}_{s},\alpha_{s})=\mc{A}(\ms{a},\mf{T}_{as},\alpha_{as})$
and for all $j\in\{m,w\}$
\begin{equation*}
\begin{aligned}
\mc{W}^{\ms{a}}(\mf{T}_{s},\alpha_{s})(N_{j})
&\geq
\mf{m}^{\ms{a}}(\mf{T}_{s},\alpha_{s})(\ms{f}_{j}),
\\
\mf{m}^{\ms{a}}(\mf{T}_{s},\alpha_{s})(\ms{f}_{m})
&=
\mf{m}^{\ms{a}}(\mf{T}_{s},\alpha_{s})(\ms{f}_{w})
\\
&=
\mf{m}^{\ms{a}}(\mf{T}_{as},\alpha_{as})(\ms{f}_{w}).
\end{aligned}
\end{equation*}
\end{definition}
Note that by construction with any fission process
we can associate the asymmetric fission process
extracted from it.
\begin{definition}
[Restricted $\mc{T}-$nucleon and fragment masses and $\mc{T}-$Terrell law]
\label{q08062037}
Let $\#_{m}=$light and $\#_{w}=$heavy.
For any $j\in\{m,w\}$ 
let $\upzeta_{j}^{\mc{T}}:\ms{N}_{as}^{\mc{T}}\to\R$
and 
$\upkappa_{j}^{\mc{T}}:\ms{N}_{as}^{\mc{T}}\to\R$
be the 
restricted $\mc{T}-\#_{j}$ nucleon mass
and the
restricted $\mc{T}-\#_{j}$ fragment mass
respectively,
defined as the maps such that for all
$\mf{n}=\lr{\ms{a}}{\mf{T},\alpha,
\{\ms{f}_{j},N_{j}\}_{j\in\{m,w\}}}
\in\ms{N}_{as}^{\mc{T}}$
we have
\begin{equation*}
\begin{aligned}
\upzeta_{j}^{\mc{T}}(\mf{n})
&\coloneqq 
\mf{m}^{\ms{a}}(\mf{T},\alpha)(\ms{f}_{j}),
\\
\upkappa_{j}^{\mc{T}}(\mf{n})
&\coloneqq 
\mc{W}^{\ms{a}}(\mf{T},\alpha)(N_{j}).
\end{aligned}
\end{equation*}
Let
$\nuup^{\mc{T}}:\ms{N}_{as}^{\mc{T}}\to\R$
be 
the restricted $\mc{T}-$Terrell law
such that 
\begin{equation*}
\nuup^{\mc{T}}
\coloneqq
0.08(\upkappa_{m}^{\mc{T}}-\upzeta_{m}^{\mc{T}})
+
0.1(\upkappa_{w}^{\mc{T}}-\upzeta_{w}^{\mc{T}}).
\end{equation*}
\end{definition}
\begin{remark}
For any
$\mf{n}=\lr{\ms{a}}{\mf{T},\alpha,
\{\ms{f}_{j},N_{j}\}_{j\in\{m,w\}}}
\in\ms{N}_{as}^{\mc{T}}$
we have
\begin{equation*}
\begin{aligned}
\nuup^{\mc{T}}(\mf{n})
&=
0.08\left(   
\mc{W}^{\ms{a}}(\mf{T},\alpha)(N_{m})
-
\mf{m}^{\ms{a}}(\mf{T},\alpha)(\ms{f}_{m})\right)+
\\
&
0.1
\left(\mc{W}^{\ms{a}}(\mf{T},\alpha)(N_{w})
-
\mf{m}^{\ms{a}}(\mf{T},\alpha)(\ms{f}_{w})
\right).
\end{aligned}
\end{equation*}
\end{remark}
\begin{definition}
\label{11061646}
Define $\mf{k}^{\mc{T}}$ 
as the map on $H$ such that 
$\mf{k}^{\mc{T}}(l)$ is a map on $\ms{N}_{as}^{\mc{T}}$ 
such that for all 
$\mf{n}=\lr{\ms{a}}{\mf{T},\alpha,
\{\ms{f}_{j},N_{j}\}_{j\in\{m,w\}}}
\in\ms{N}_{as}^{\mc{T}}$
we have 
\begin{equation*}
\mf{k}^{\mc{T}}(l)(\mf{n})
\coloneqq
\lr{\ms{a}}{\mf{b}^{\mc{F}(\ms{a})}(l)(\mf{T}),\,\alpha,\,
\{\,\uppsi^{\mc{F}(\ms{a})}(l)(\ms{f}_{j}),\,
\ms{V}(\mc{F}(\ms{a}))_{\alpha}^{\mf{T}}(l)(N_{j})\,
\}_{j\in\{m,w\}}}.
\end{equation*}
\end{definition}
\begin{convention}
If
$\mf{n}=\lr{\ms{a}}{\mf{T},\alpha,
\{\ms{f}_{j},N_{j}\}_{j\in\{m,w\}}}
\in\ms{N}_{as}^{\mc{T}}$,
then for any $j\in\{m,w\}$ and $l\in H$ 
whenever it is not cause of confusion,
we let 
$\mf{T}^{l}$,
$\ms{f}_{j}^{l}$,
$N_{j}^{l}$
and
$\mf{n}^{l}$
denote
$\mf{b}^{\mc{F}(\ms{a})}(l)(\mf{T})$,
$\uppsi^{\mc{F}(\ms{a})}(l)(\ms{f}_{j})$,
$\ms{V}(\mc{F}(\ms{a}))_{\alpha}^{\mf{T}}(l)(N_{j})$
and
$\mf{k}^{\mc{T}}(l)(\mf{n})$
respectively.
\end{convention}
Up to the end of section \ref{07011744} let
$\mf{n}=\lr{\ms{a}}{\mf{T},\alpha,\{\ms{f}_{j},N_{j}\}_{j\in\{m,w\}}}\in\ms{N}_{as}^{\mc{T}}$.
\begin{remark}
For all $l\in H$ we have
\begin{equation*}
\begin{aligned}
(\nuup^{\mc{T}}\circ\mf{k}^{\mc{T}}(l))
(\mf{n})
&=
0.08\left(   
\mc{W}^{\ms{a}}(\mf{T}^{l},\alpha)(N_{m}^{l})
-
\mf{m}^{\ms{a}}(\mf{T}^{l},\alpha)(\ms{f}_{m}^{l})
\right)+
\\
&
0.1
\left(
\mc{W}^{\ms{a}}
(\mf{T}^{l},\alpha)(N_{w}^{l})
-
\mf{m}^{\ms{a}}(\mf{T}^{l},\alpha)(\ms{f}_{w}^{l})
\right).
\end{aligned}
\end{equation*}
\end{remark}
\begin{definition}
\label{07011614}
We say that $(\mf{n},\ms{b},T,\mf{x})$ satisfies the hypothesis $\ms{S}$ w.r.t. $\mc{T}$, 
if
$\ms{b}\in\Upxi(\pf{D},\mc{F})$,
$T\in Mor_{\Upxi(\pf{D},\mc{F})}(\ms{b},\ms{a})$
and
$\mf{x}=\{\ms{f}_{j}',N_{j}'\}_{j\in\{m,w\}}$
such that for any $j\in\{m,w\}$
\begin{enumerate}
\item
$\ms{f}_{j}'\in\mc{F}_{1}(T)^{-1}(\ms{f}_{j})$,
\item
$N_{j}'\in\mc{S}(T,\mf{T},\alpha)^{-1}(N_{j})\cap\mc{A}(\ms{b},\mf{T}^{T},\alpha)_{+}$.
\end{enumerate}
Let $(\mf{n},\ms{b},T,\mf{x})$ satisfy the hypothesis $\ms{S}$ w.r.t. $\mc{T}$, 
define
\begin{equation}
\label{07171703}
\mf{n}^{(T,\mf{x})}
\coloneqq
\lr{\ms{b}}{\mf{T}^{T},\alpha,\mf{x}}.
\end{equation}
\end{definition}
\begin{proposition}
\label{07011654}
We have that 
\begin{enumerate}
\item
$\nuup^{\mc{T}}$ is a nonnegative map,
\label{07011654st1}
\item
$\mf{k}^{\mc{T}}\in Mor_{\ms{Gr}}(H,Aut_{\ms{set}}(\ms{N}_{as}^{\mc{T}}))$;
\label{07011654st2}
\item
for all $j\in\{m,w\}$ we have
\begin{equation*}
\begin{aligned}
\mf{m}^{\ms{a}}(\mf{T}^{l},\alpha)(\ms{f}_{j}^{l})
&=
\mf{m}^{\ms{a}}(\mf{T},\alpha)(\ms{f}_{j})
\\
\mc{W}^{\ms{a}}(\mf{T}^{l},\alpha)(N_{j}^{l})
&=
\mc{W}^{\ms{a}}(\mf{T},\alpha)(N_{j});
\end{aligned}
\end{equation*}
\label{07011847}
\item
let $(\mf{n},\ms{b},T,\mf{x})$ satisfy the hypothesis $\ms{S}$ w.r.t. $\mc{T}$, 
where $\mf{x}=\{\ms{f}_{j}',N_{j}'\}_{j\in\{m,w\}}$, 
then for all $j\in\{m,w\}$
\begin{equation*}
\begin{aligned}
\mf{m}^{\ms{b}}\left(\mf{T}^{T},\alpha\right)(\ms{f}_{j}')
&=
\mf{m}^{\ms{a}}(\mf{T},\alpha)(\ms{f}_{j}),
\\
\mc{W}^{\ms{b}}\left(\mf{T}^{T},\alpha\right)(N_{j}')
&=
\mc{W}^{\ms{a}}(\mf{T},\alpha)(N_{j}),
\end{aligned}
\end{equation*}
moreover 
$\mf{n}^{(T,\mf{x})}\in\ms{N}_{as}^{\mc{T}}$,
and
\begin{equation*}
\begin{aligned}
\nuup^{\mc{T}}
(\mf{n}^{(T,\mf{x})})
&=
0.08
\left(
\mc{W}^{\ms{b}}
\left(\mf{T}^{T},\alpha\right)
(N_{m}')
-
\mf{m}^{\ms{b}}
\left(\mf{T}^{T},\alpha\right)(\ms{f}_{m}')
\right)
+
\\
&
0.1
\left(
\mc{W}^{\ms{b}}
\left(\mf{T}^{T},\alpha\right)
(N_{w}')
-
\mf{m}^{\ms{b}}
\left(\mf{T}^{T},\alpha\right)(\ms{f}_{w}')
\right).
\end{aligned}
\end{equation*}
\label{24051704pre}
\end{enumerate}
\end{proposition}
\begin{proof}
St.\eqref{07011654st1} follows by Def. \ref{qnuclph1}\eqref{q08061831}, 
the equalities in st.\eqref{07011847} and the first two equalities in st.\eqref{24051704pre} follow by 
Prp. \ref{20051810dbt};
st.\eqref{07011654st2} follows by st.\eqref{07011847} and \eqref{09051209}.
$\mf{n}^{(T,\mf{x})}\in\ms{N}_{as}^{\mc{T}}$ 
follows since the definition of the morphism set of $\Upxi(\pf{D},\mc{F})$
and since the equalities in st.\eqref{24051704pre}.
\end{proof}
\begin{proposition}
[Invariance of the restricted $\mc{T}-$nucleon masses,
restricted $\mc{T}-$fragment masses and restricted $\mc{T}-$Terrell law]
\label{11061659}
Let $(\mf{n},\ms{b},T,\mf{x})$ satisfy the hypothesis $\ms{S}$ w.r.t. $\mc{T}$,
$l\in H$ and $j\in\{m,w\}$ then
\begin{enumerate}
\item
Invariance of the restricted $\mc{T}-$nucleon masses
and restricted $\mc{T}-$fragment masses 
under essential contravariant action of 
$\Upxi(\pf{D},\mc{F})$
\begin{equation*}
\begin{aligned}
\upzeta_{j}^{\mc{T}}(\mf{n}^{(T,\mf{x})})
&=
\upzeta_{j}^{\mc{T}}(\mf{n}),
\\
\upkappa_{j}^{\mc{T}}(\mf{n}^{(T,\mf{x})})
&=
\upkappa_{j}^{\mc{T}}(\mf{n}).
\end{aligned}
\end{equation*}
\label{11061659st2}
\item
Invariance of the restricted $\mc{T}-$nucleon masses
and restricted $\mc{T}-$fragment masses 
under action of $H$
\begin{equation*}
\begin{aligned}
\upzeta_{j}^{\mc{T}}\circ\mf{k}^{\mc{T}}(l)
&=
\upzeta_{j}^{\mc{T}},
\\
\upkappa_{j}^{\mc{T}}\circ\mf{k}^{\mc{T}}(l)
&=
\upkappa_{j}^{\mc{T}}.
\end{aligned}
\end{equation*}
\label{11061659st1}
\item
Invariance of the restricted $\mc{T}-$Terrell law under essential contravariant action of 
$\Upxi(\pf{D},\mc{F})$ and under action of $H$
\begin{equation*}
\begin{aligned}
\nuup^{\mc{T}}(\mf{n}^{(T,\mf{x})})
&=
\nuup^{\mc{T}}(\mf{n}),
\\
\nuup^{\mc{T}}\circ\mf{k}^{\mc{T}}(l)
&=
\nuup^{\mc{T}}.
\end{aligned}
\end{equation*}
\label{11061659st3}
\end{enumerate}
\end{proposition}
\begin{proof}
Sts.(\ref{11061659st2},\ref{11061659st1}) follow since Prp. \ref{07011654}(\ref{24051704pre},\ref{07011847})
respectively,
st.\eqref{11061659st3} since sts.(\ref{11061659st2},\ref{11061659st1}). 
\end{proof}
\begin{definition}
\label{q12290857}
We say that $(\mf{s},\mf{u})$ is a semantics for $\mc{T}$
if $(\mf{s},\mf{u})$ is an interpretation such that for any 
$\lr{\ms{a}}{\{\mf{T}_{\#},\alpha_{\#}\}_{\#\in\{s,as\}},
\{\ms{f}_{j},N_{j}\}_{j\in\{m,w\}}}\in\ms{N}^{\mc{T}}$,
$\#\in\{s,as\}$,
$\ms{a},\ms{b}\in\Upxi(\pf{D},\mc{F})$,
$\mf{T}\in\pf{D}_{\mc{F}(\ms{a})}$, 
$\alpha\in\ms{P}_{\mc{F}(\ms{a})}^{\mf{T}}$
and 
$T\in Mor_{\Upxi(\pf{D},\mc{F})}(\ms{b},\ms{a})$
we have
\begin{enumerate}
\item 
$\mf{u}(\mf{T}_{\#})\equiv$
realizing the $\#$ymmetric fission process of the fissioning system $\mf{u}(\ms{a})$, 
\label{q12290857st1}
\item
$\mf{u}(N_{m})\equiv$
relative to the light fragment mass,
\label{q12290857st2}
\item
$\mf{u}(N_{w})\equiv$
relative to the heavy fragment mass,
\label{q12290857st3}
\item
$\mf{u}(\ms{f}_{m})\equiv$ 
relative to the light nucleon core mass, 
\label{q12290857st4}
\item
$\mf{u}(\ms{f}_{w})\equiv$ 
relative to the heavy nucleon core mass,
\label{q12290857st5}
\item
$\mf{u}(\mc{S}(T,\mf{T},\alpha))\equiv T$,
\label{q12290857st6}
\item
$\mf{u}(\mc{F}_{1}(T))\equiv T$,
\label{q12290857st7a}
\item
$\mf{u}(\mc{F}_{2}(T))\equiv T$.
\label{q12290857st7}
\end{enumerate}
\end{definition}
In the present section, 
let $(\mf{s},\mf{u})$ be a semantics for $\mc{T}$. 
Prp. \ref{02202114dbt} 
and Def. \ref{q12290857} justify the following 
\begin{definition}
\label{12060815}
$\nuup^{\mc{T}}(\mf{n})$ equals the mean value of the prompt 
neutron-yield in $\mf{s}(\mc{W}^{\ms{a}}(\mf{T},\alpha))$.
\end{definition}
\begin{proposition}
\label{11061851}
Let $(\mf{n},\ms{b},T,\mf{x})$ satisfy the hypothesis $\ms{S}$ w.r.t. $\mc{T}$
and $l\in H$ thus for any $j\in\{m,w\}$
\begin{enumerate}
\item
The mean value of the observable relative to the $\#_{j}$ nucleon core mass
of the nucleon system generated by the fissioning system $\mf{u}(\ms{b})$, 
in 
the phase occurring by performing 
on 
$\mf{s}((\ps{\upvarphi}^{\mc{F}(\ms{b})})_{\alpha}^{\mc{F}_{2}(T)\mf{T}})$
the operation obtained by transforming through the action of $T$
the operation realizing the asymmetric fission process of the fissioning system $\mf{u}(\ms{a})$, 
equals
the mean value of the observable, 
obtained by transforming through the action of $T$
the observable relative to the $\#_{j}$ nucleon core mass
of the nucleon system generated by the fissioning system $\mf{u}(\ms{b})$,
in
the phase of the nucleon system generated by the fissioning system $\mf{u}(\ms{a})$, 
occurring by performing 
on 
$\mf{s}((\ps{\upvarphi}^{\mc{F}(\ms{a})})_{\alpha}^{\mf{T}})$
the operation realizing the 
asymmetric fission process of the fissioning system $\mf{u}(\ms{a})$.
\label{11061851st1}
\item
The mean value of the observable obtained by transforming through the action of $l$
the observable relative to the $\#_{j}$ nucleon core mass,
in 
the phase of 
the nucleon system generated by
the fissioning system $\mf{u}(\ms{a})$, 
occurring by performing 
on 
$\mf{s}((\ps{\upvarphi}^{\mc{F}(\ms{a})})_{\alpha}^{\mf{T}^{l}})$
the operation obtained by transforming through the action of $l$
the operation realizing the asymmetric fission process 
of the fissioning system $\mf{u}(\ms{a})$, 
equals
the mean value of the observable relative to the $\#_{j}$ nucleon core mass
in
the phase of the nucleon system 
generated by the fissioning system $\mf{u}(\ms{a})$, 
occurring by performing 
on 
$\mf{s}((\ps{\upvarphi}^{\mc{F}(\ms{a})})_{\alpha}^{\mf{T}})$
the operation realizing the asymmetric fission process of 
the fissioning system $\mf{u}(\ms{a})$.
\label{11061851st2}
\item
The mean value $\nuup^{\mc{T}}(\mf{n}^{(T,\mf{x})})$ of the prompt neutron-yield in 
the state 
originated via 
the phase of the nucleon system generated by the fissioning system $\mf{u}(\ms{b})$, 
occurring by performing 
on 
$\mf{s}((\ps{\upvarphi}^{\mc{F}(\ms{b})})_{\alpha}^{\mc{F}_{2}(T)\mf{T}})$
the operation obtained by transforming through the action of $T$
the operation realizing the asymmetric fission process of the fissioning system $\mf{u}(\ms{a})$, 
equals
the mean value $\nuup^{\mc{T}}(\mf{n})$ of the prompt neutron-yield in 
the state 
originated via 
the phase of the nucleon system 
generated by the fissioning system $\mf{u}(\ms{a})$, 
occurring by performing 
on 
$\mf{s}((\ps{\upvarphi}^{\mc{F}(\ms{a})})_{\alpha}^{\mf{T}})$
the operation realizing the 
asymmetric fission process of the fissioning system $\mf{u}(\ms{a})$.
\item
The mean value $\nuup^{\mc{T}}(\mf{n}^{l})$ of the prompt neutron-yield in 
the state 
originated via 
the phase of 
the nucleon system generated by
the fissioning system $\mf{u}(\ms{a})$, 
occurring by performing 
on 
$\mf{s}((\ps{\upvarphi}^{\mc{F}(\ms{a})})_{\alpha}^{\mf{T}^{l}})$
the operation obtained by transforming through the action of $l$
the operation realizing the asymmetric fission process 
of the fissioning system $\mf{u}(\ms{a})$, 
equals
the mean value $\nuup^{\mc{T}}(\mf{n})$ of the prompt neutron-yield in 
the state originated via 
the phase of the nucleon system 
generated by the fissioning system $\mf{u}(\ms{a})$, 
occurring by performing 
on 
$\mf{s}((\ps{\upvarphi}^{\mc{F}(\ms{a})})_{\alpha}^{\mf{T}})$
the operation realizing the asymmetric fission process of 
the fissioning system $\mf{u}(\ms{a})$.
\end{enumerate}
Here 
$\mf{s}((\ps{\upvarphi}^{\mc{F}(\ms{a})})_{\alpha}^{\mf{T}})\equiv$
the state of thermal equilibrium 
$(\ps{\upvarphi}^{\mc{F}(\ms{a})})_{\alpha}^{\mf{T}}$
at the inverse temperature $\alpha$
of 
the fragment system whose observable algebra is $\mc{A}(\ms{a},\mf{T},\alpha)$
and whose dynamics is $(\ep^{\mc{F}(\ms{a})})_{\alpha}^{\mf{T}}(-\alpha^{-1}(\cdot))$.
\end{proposition}
\begin{proof}
Since Prp. \ref{11061659},
Prp. \ref{02202114dbt} 
and rearrangements to avoid redundancies,
alternatively sts. (\ref{11061851st1},\ref{11061851st2})
follow since 
Prp. \ref{12051206dbt}.
\end{proof}
The next definition will be used in formulating the nucleon phase hypothesis in section \ref{01231337}.
\begin{definition}
\label{13061939}
$\mf{l}$ is said to be nucleon equivalent with $\mf{n}$ 
if $\mf{l}\in\ms{N}_{as}^{\mc{T}}$ and 
$\upzeta_{j}^{\mc{T}}(\mf{l})=\upzeta_{j}^{\mc{T}}(\mf{n})$
for all $j\in\{m,w\}$.
\end{definition}
\begin{remark}
\label{26061121}
For any $\mf{l}=\lr{\ms{d},\mf{Q},\beta}{\{\ms{g}_{j},M_{j}\}_{j\in\{m,w\}}}$ 
nucleon equivalent with $\mf{n}$ 
we have 
\begin{equation*}
\begin{aligned}
\nuup^{\mc{T}}(\mf{l})
&=
0.08\left(\mc{W}^{\ms{d}}(\mf{Q},\beta)(M_{m})
-
\mf{m}^{\ms{a}}(\mf{T},\alpha)(\ms{f}_{m})
\right)+
\\
&
0.1
\left(\mc{W}^{\ms{d}}(\mf{Q},\beta)(M_{w})
-
\mf{m}^{\ms{a}}(\mf{T},\alpha)(\ms{f}_{w})
\right).
\end{aligned}
\end{equation*}
\end{remark}
\subsection{Universality of the $\mc{T}-$Terrell law}
\label{07170749}
In order to maintain the present section reasonably independent by the previous one,
when feasible we shall not use the results in section \ref{07011744}.
In Def. \ref{07252801} and Def. \ref{07141845} we define 
the $\mc{T}-$nucleon and fragment masses,
and the $\mc{T}-$Terrell law,
and prove in Thm. \ref{07141857} their invariance
under contravariant action of $\Upxi(\pf{D},\mc{F})$
and under action of $H$,
by providing the equivariant form of Prp. \ref{11061659}.
This is an essential result representing the
$\mc{T}-$resolution of the invariant form of the universality claim,
since when applied to $\mc{T}_{\bullet}$ furnishes
the invariant form of the claim.
Thm. \ref{07141857} establishes the existence
of invariant sections $\muup_{j}^{\mc{T}}$, $\uplambda_{j}^{\mc{T}}$ 
and $\uptheta^{\mc{T}}$ 
under contravariant action of $\Upxi(\pf{D},\mc{F})$
via the extension of the map defined in \eqref{07171703} to the power set of $\ms{N}_{as}^{\mc{T}}$.
Moreover each value of these sections 
induces an invariant section under covariant action of $H$.
More exactly we show in Cor. \ref{07141823} that there exists 
a field $\mc{P}^{\mc{T}}$ contravariant under action of 
$\Upxi(\pf{D},\mc{F})$ such that its fiber in each $\ms{a}$
is the power set of the subset $\ms{N}_{as}^{\mc{T}}(\ms{a})$ 
of the asymmetric fission processes relative to $\mc{T}$
whose underlying fissioning system equals $\ms{a}$;
while its morphism map extends and modifies the map defined in \eqref{07171703} 
by transforming the dependence by $\mf{x}$ into a set-valued map. 
Thus $\muup_{j}^{\mc{T}}$, $\uplambda_{j}^{\mc{T}}$ and $\uptheta^{\mc{T}}$
result to be morphisms of the category 
$\ms{Fct}(\Upxi(\pf{D},\mc{F})^{op},\ms{set})$
whose source functor is $\mc{P}^{\mc{T}}$ and such that 
$(\muup_{j}^{\mc{T}})_{\ms{a}}$, $(\uplambda_{j}^{\mc{T}})_{\ms{a}}$ 
and $\uptheta_{\ms{a}}^{\mc{T}}$ 
are the power set extensions 
of the restrictions of $\upzeta_{j}^{\mc{T}}$, $\upkappa_{j}^{\mc{T}}$,    
and $\upnu^{\mc{T}}$ to $\ms{N}_{as}^{\mc{T}}(\ms{a})$
respectively, and inducing morphisms of $\ms{Fct}(H,\ms{set})$.
In the present section \ref{07170749} let $\mf{C}$ be a category, 
and $\mc{T}=\lr{S,J,\mc{Z},\mc{S},\mc{F}}{\mf{m},\mc{W},R,\pf{D},\pf{U}}$ be 
a nucleon-fragment doublet 
on $\mf{C}$ 
where for simplicity we assume that $R$ is the constant map with constant value equal to $H$.
Finally up to the end of this part
since we will deal only with asymmetric processes, whenever we mention 
fission processes without any specification we always refer to the asymmetric ones.
\begin{definition}
\label{07141809}
Define 
$\ms{N}_{as}^{\mc{T}}(\cdot):
Obj(\Upxi(\pf{D},\mc{F}))\to\mathscr{P}(\ms{N}_{as}^{\mc{T}})$
such that for all $\ms{a}\in\Upxi(\pf{D},\mc{F})$
\begin{equation*}
\ms{N}_{as}^{\mc{T}}(\ms{a})
\coloneqq
\{
\mf{n}
\in
\ms{N}_{as}^{\mc{T}}
\,\vert\,
\Pr_{1}(\mf{n})=\ms{a}
\},
\end{equation*}
where $\Pr_{1}$ is the map on $\ms{N}_{as}^{\mc{T}}$ mapping any 
$4-$tupla into its first element.
Moreover define 
\begin{equation*}
\mc{P}_{o}^{\mc{T}}:Obj(\Upxi(\pf{D},\mc{F}))\ni\ms{a}
\mapsto
\mathscr{P}(\ms{N}_{as}^{\mc{T}}(\ms{a})),
\end{equation*}
and
\begin{equation*}
\uprho^{\mc{T}}\in
\prod_{T\in Mor_{\Upxi(\pf{D},\mc{F})}}
Mor_{\ms{set}}(\ms{N}_{as}^{\mc{T}}(c(T)),\mc{P}_{o}^{\mc{T}}(d(T))),
\end{equation*}
such that 
for any $T\in Mor_{\Upxi(\pf{D},\mc{F})}$ and 
$\fp{c(T)}{\mf{T}}{\alpha}{\ms{f}}{N}\in\ms{N}_{as}^{\mc{T}}(c(T))$
\begin{multline*}
\uprho^{\mc{T}}(T)(\fp{c(T)}{\mf{T}}{\alpha}{\ms{f}}{N})
\coloneqq
\\
\big\{\fp{d(T)}{\mf{T}^{T}}{\alpha}{x}{y}
\,\vert\,
x\in
\prod_{j\in\{m,w\}}\mc{F}_{1}(T)^{-1}(\ms{f}_{j}),
\\
y\in
\prod_{j\in\{m,w\}}
\mc{S}(T,\mf{T},\alpha)^{-1}(N_{j})
\cap
\mc{A}(d(T),\mf{T}^{T},\alpha)_{+}
\big\}.
\end{multline*}
Finally define
\begin{equation*}
\mc{P}_{m}^{\mc{T}}
\in
\prod_{T\in Mor_{\Upxi(\pf{D},\mc{F})}}
Mor_{\ms{set}}(\mc{P}_{o}^{\mc{T}}(c(T)),\mc{P}_{o}^{\mc{T}}(d(T))),
\end{equation*}
such that for all $T\in Mor_{\Upxi(\pf{D},\mc{F})}$ and $Y\in\mc{P}_{o}^{\mc{T}}(c(T))$
\begin{equation*}
\mc{P}_{m}^{\mc{T}}(T)Y
\coloneqq
\bigcup
\left\{\uprho^{\mc{T}}(T)y\,\vert\, y\in Y
\right\},
\end{equation*}
set $\mc{P}^{T}\coloneqq(\mc{P}_{o}^{\mc{T}},\mc{P}_{m}^{\mc{T}})$.
\end{definition}
In order to prove that $\mc{P}^{T}$ is a functor we need the following
\begin{lemma}
\label{07161534}
\begin{enumerate}
\item
$\uprho^{\mc{T}}$ and $\mc{P}^{\mc{T}}$ are well-defined and nonempty set-valued maps;
\label{07161534st1}
\item
$\uprho^{\mc{T}}(T\circ S)=\mc{P}_{m}^{\mc{T}}(S)\circ\uprho^{\mc{T}}(T)$,
for all $S,T\in Mor_{\Upxi(\pf{D},\mc{F})}$ such that $d(T)=c(S)$.
\label{07161534st2}
\end{enumerate}
\end{lemma}
\begin{proof}
St. \eqref{07161534st1}
follows since 
Prp. \ref{20051810dbt}, 
and Def. \ref{qnuclph1}\eqref{07201029}.
Let $S,T\in Mor_{\Upxi(\pf{D},\mc{F})}$ such that $d(T)=c(S)$ and let 
$\fp{c(T)}{\mf{T}}{\alpha}{\ms{f}}{N}
\in\ms{N}_{as}^{\mc{T}}(c(T))$
then 
since
Def. \ref{0620115}\eqref{0620115st5}, \eqref{06222005} and $(f\circ g)^{-1}=g^{-1}\circ f^{-1}$
\begin{multline}
\label{07181102}
\uprho^{\mc{T}}(T\circ S)(\fp{c(T)}{\mf{T}}{\alpha}{\ms{f}}{N})
=
\\
\big\{\fp{d(S)}{(\mf{T}^{T})^{S}}{\alpha}{z}{q}
\,\vert\,
z\in\prod_{j\in\{m,w\}}
\left(\mc{F}_{1}(S)^{-1}\circ\mc{F}_{1}(T)^{-1}\right)(\ms{f}_{j}),
\\
q\in\prod_{j\in\{m,w\}}
\left(\mc{S}(S,\mf{T}^{T},\alpha)^{-1}\circ\mc{S}(T,\mf{T},\alpha)^{-1}\right)
(N_{j})
\cap
\mc{A}(d(S),(\mf{T}^{T})^{S},\alpha)_{+}
\big\}.
\end{multline}
Next
\begin{multline}
\label{07161528}
(\mc{P}_{m}^{\mc{T}}(S)\circ\uprho^{\mc{T}}(T))(\fp{c(T)}{\mf{T}}{\alpha}{\ms{f}}{N})
=
\\
\bigcup
\big\{
\uprho^{\mc{T}}(S)(\fp{d(T)}{\mf{T}^{T}}{\alpha}{x}{y})
\,\vert\,
x\in\prod_{j\in\{m,w\}}
\mc{F}_{1}(T)^{-1}(\ms{f}_{j}),
\\
y\in
\prod_{j\in\{m,w\}}
\mc{S}(T,\mf{T},\alpha)^{-1}(N_{j})
\cap
\mc{A}(d(T),\mf{T}^{T},\alpha)_{+}
\big\}=
\\
\bigcup
\bigg\{
\big\{
\fp{d(S)}{(\mf{T}^{T})^{S}}{\alpha}{z}{q}
\,\vert\,
z\in
\prod_{j\in\{m,w\}}
\mc{F}_{1}(S)^{-1}(x_{j}),
\\
q\in\prod_{j\in\{m,w\}}
\mc{S}(S,\mf{T}^{T},\alpha)^{-1}(y_{j})
\cap
\mc{A}(d(S),(\mf{T}^{T})^{S},\alpha)_{+}\big\}
\\
\,\big\vert\,
x\in
\prod_{j\in\{m,w\}}
\mc{F}_{1}(T)^{-1}(\ms{f}_{j}),\,
\\
y\in\prod_{j\in\{m,w\}}
\mc{S}(T,\mf{T},\alpha)^{-1}(N_{j})
\cap
\mc{A}(d(T),\mf{T}^{T},\alpha)_{+}
\bigg\}.
\end{multline}
Next for any $j\in\{m,w\}$
\begin{equation}
\label{07181145}
\bigcup_{x_{j}\in\mc{F}_{1}(T)^{-1}(\ms{f}_{j})}
\mc{F}_{1}(S)^{-1}(x_{j})
=
(\mc{F}_{1}(S)^{-1}\circ\mc{F}_{1}(T)^{-1})(\ms{f}_{j}),
\end{equation}
while 
\begin{multline}
\label{07181146}
\bigcup_{y_{j}\in\mc{S}(T,\mf{T},\alpha)^{-1}(N_{j})
\cap
\mc{A}(d(T),\mf{T}^{T},\alpha)_{+}}
\mc{S}(S,\mf{T}^{T},\alpha)^{-1}(y_{j})
\cap
\mc{A}(d(S),(\mf{T}^{T})^{S},\alpha)_{+}
=\\
\mc{S}(S,\mf{T}^{T},\alpha)^{-1}
\left(\mc{S}(T,\mf{T},\alpha)^{-1}(N_{j})
\cap 
\mc{A}(d(T),\mf{T}^{T},\alpha)_{+}
\right)
\cap
\mc{A}(d(S),(\mf{T}^{T})^{S},\alpha)_{+}
=
\\
(\mc{S}(S,\mf{T}^{T},\alpha)^{-1}
\circ\mc{S}(T,\mf{T},\alpha)^{-1})(N_{j})
\cap
\mc{S}(S,\mf{T}^{T},\alpha)^{-1}
(\mc{A}(d(T),\mf{T}^{T},\alpha)_{+})
\cap
\mc{A}(d(S),(\mf{T}^{T})^{S},\alpha)_{+}
=
\\
(\mc{S}(S,\mf{T}^{T},\alpha)^{-1}
\circ\mc{S}(T,\mf{T},\alpha)^{-1})(N_{j})
\cap
\mc{A}(d(S),(\mf{T}^{T})^{S},\alpha)_{+},
\end{multline}
where we considered that $L^{-1}(\mc{B}_{+})\supseteq\mc{A}_{+}$ for any 
$L\in Mor_{\ms{CA}^{\ast}}(\mc{A},\mc{B})$, 
that
$f^{-1}(\cup_{x\in X}a_{x})=\cup_{x\in X}f^{-1}(a_{x})$,
$f^{-1}(\cap_{x\in X}a_{x})=\cap_{x\in X}f^{-1}(a_{x})$, 
and
$\cup_{x\in X}(a_{x}\cap b)=(\cup_{x\in X}a_{x})\cap b$ 
for any map $f:A\to B$,
subset $\{a_{x}\,\vert\, x\in X\}$ of $\mathscr{P}(B)$
and $b\in\mathscr{P}(B)$.
St. \eqref{07161534st2} follows since
\eqref{07181102}, \eqref{07161528}
\cite[(35) Fascicule de resultats; Pr. 9, II.36]{BourEns}, 
and
\eqref{07181145}, \eqref{07181146}.
\end{proof}
\begin{corollary}
\label{07141823}
$\mc{P}^{\mc{T}}\in\ms{Fct}(\Upxi(\pf{D},\mc{F})^{op},\ms{set})$
whose object map takes nonempty set values.
\end{corollary}
\begin{proof}
Let $S,T\in Mor_{\Upxi(\pf{D},\mc{F})}$ such that $d(T)=c(S)$ and 
$Y\in\mc{P}_{o}^{\mc{T}}(c(T))$, then since 
Lemma \ref{07161534}\eqref{07161534st2} 
we have 
\begin{equation*}
\begin{aligned}
(\mc{P}_{m}^{\mc{T}}(S)\circ\mc{P}_{m}^{\mc{T}}(T))Y
&=
\mc{P}_{m}^{\mc{T}}(S)\left(\bigcup\{\uprho^{\mc{T}}(T)(y)\,\vert\,y\in Y\}\right)
\\
&=
\bigcup\{(\mc{P}_{m}^{\mc{T}}(S)\circ\uprho^{\mc{T}}(T))(y)\,\vert\,y\in Y\}
\\
&=\bigcup\{\uprho^{\mc{T}}(T\circ S)(y)\,\vert\,y\in Y\}
=\mc{P}_{m}^{\mc{T}}(T\circ S)Y.
\end{aligned}
\end{equation*}
hence the statement follows since the second assertion of
Lemma \ref{07161534}\eqref{07161534st1}. 
\end{proof}
Prp. \ref{07011654} permits to provide the following
\begin{definition}
\label{07141844}
Define 
\begin{equation*}
\begin{aligned}
\uptau^{\mc{T}} 
&\in
\prod_{\ms{a}
\in\Upxi(\pf{D},\mc{F})}
Mor_{\ms{Gr}}(H,Aut_{\ms{set}}(\mc{P}^{\mc{T}}(\ms{a}))),
\\
\uptau^{\mc{T}}_{\ms{a}}(l)Y
&\coloneqq
\{\mf{k}^{\mc{T}}(l)(y)\,\vert\,y\in Y\},
\forall 
\ms{a}\in\Upxi(\pf{D},\mc{F}),
l\in H,
Y\in\mc{P}^{\mc{T}}(\ms{a}).
\end{aligned}
\end{equation*}
Set for all $\ms{a}\in\Upxi(\pf{D},\mc{F})$
\begin{equation*}
\ms{Q}_{\ms{a}}^{\mc{T}}\coloneqq(\un\mapsto\mc{P}^{\mc{T}}(\ms{a}),\uptau_{\ms{a}}^{\mc{T}}).
\end{equation*}
\end{definition}
\begin{definition}
\label{07141849}
Let $\ms{R}^{\mc{T}}$ be the unique element of $\ms{Fct}(\Upxi(\pf{D},\mc{F})^{op},\ms{set})$
such that $\ms{R}_{o}^{\mc{T}}$ is the constant map with constant value equal to $\mathscr{P}(\R)$, and 
$\ms{R}_{m}^{\mc{T}}$ is the constant map with constant value equal to $Id_{\mathscr{P}(\R)}$.
Let $\ms{R}^{H}$ be the unique element in $\ms{Fct}(H,\ms{set})$ such that 
$\ms{R}_{o}^{H}=\un\mapsto\mathscr{P}(\R)$ and $\ms{R}_{m}^{H}$ is the constant map 
with constant value equal to $Id_{\mathscr{P}(\R)}$.
\end{definition}
\begin{definition}
[$\mc{T}-$nucleon masses and $\mc{T}-$fragment masses]
\label{07252801}
For any $j\in\{m,w\}$ let 
\begin{equation*}
\begin{aligned}
\ms{M}_{j}^{\mc{T}}:
\mathscr{P}(\ms{N}_{as}^{\mc{T}})
&\to
\mathscr{P}(\R),
\\
Y
&\mapsto
\{\upzeta_{j}^{\mc{T}}(y)\,\vert\,y\in Y\};
\\
\ms{L}_{j}^{\mc{T}}:
\mathscr{P}(\ms{N}_{as}^{\mc{T}})
&\to
\mathscr{P}(\R),
\\
Y
&\mapsto
\{\upkappa_{j}^{\mc{T}}(y)\,\vert\,y\in Y\}.
\end{aligned}
\end{equation*}
Define the $\mc{T}-\#_{j}$ nucleon mass to be the map
\begin{equation*}
\begin{aligned}
\muup_{j}^{\mc{T}}
&\in
\prod_{\ms{a}\in\Upxi(\pf{D},\mc{F})}
Mor_{\ms{set}}(\mc{P}^{\mc{T}}(\ms{a}),\ms{R}^{\mc{T}}(\ms{a})),
\\
(\muup_{j}^{\mc{T}})_{\ms{a}}
&\coloneqq
\ms{M}_{j}^{\mc{T}}\up\mc{P}^{\mc{T}}(\ms{a}),
\forall\ms{a}\in\Upxi(\pf{D},\mc{F});
\end{aligned}
\end{equation*}
and define the $\mc{T}-\#_{j}$ fragment mass to be the map
\begin{equation*}
\begin{aligned}
\uplambda_{j}^{\mc{T}}
&\in
\prod_{\ms{a}\in\Upxi(\pf{D},\mc{F})}
Mor_{\ms{set}}(\mc{P}^{\mc{T}}(\ms{a}),\ms{R}^{\mc{T}}(\ms{a})),
\\
(\uplambda_{j}^{\mc{T}})_{\ms{a}}
&\coloneqq
\ms{L}_{j}^{\mc{T}}\up\mc{P}^{\mc{T}}(\ms{a}),
\forall\ms{a}\in\Upxi(\pf{D},\mc{F}).
\end{aligned}
\end{equation*}
\end{definition}
\begin{definition}
[$\mc{T}-$Terrell law]
\label{07141845}
Let
\begin{equation*}
\begin{aligned}
\Uptheta^{\mc{T}}:
\mathscr{P}(\ms{N}_{as}^{\mc{T}})
&\to
\mathscr{P}(\R),
\\
Y
&\mapsto
\{\nuup^{\mc{T}}(y)\,\vert\,y\in Y\}.
\end{aligned}
\end{equation*}
Define the $\mc{T}-$Terrell law to be the map
\begin{equation*}
\begin{aligned}
\uptheta^{\mc{T}}
&\in
\prod_{\ms{a}\in\Upxi(\pf{D},\mc{F})}
Mor_{\ms{set}}(\mc{P}^{\mc{T}}(\ms{a}),\ms{R}^{\mc{T}}(\ms{a})),
\\
\uptheta_{\ms{a}}^{\mc{T}}
&\coloneqq
\Uptheta^{\mc{T}}\up\mc{P}^{\mc{T}}(\ms{a}),
\forall\ms{a}\in\Upxi(\pf{D},\mc{F}).
\end{aligned}
\end{equation*}
\end{definition}
Now we are in the position of stating
the $\mc{T}-$resolution of the invariant form of the universality claim.
\begin{theorem}
[Universality of the $\mc{T}-$nucleon masses,
$\mc{T}-$fragment masses and $\mc{T}-$Terrell law]
\label{07141857}
For any $j\in\{m,w\}$ we have 
\begin{enumerate}
\item
Invariance of the $\mc{T}-$nucleon masses and $\mc{T}-$fragment masses
under contravariant action of $\Upxi(\pf{D},\mc{F})$
\begin{equation*}
\muup_{j}^{\mc{T}},\,
\uplambda_{j}^{\mc{T}}
\in
Mor_{\ms{Fct}(\Upxi(\pf{D},\mc{F})^{op},\ms{set})}(\mc{P}^{\mc{T}},\ms{R}^{\mc{T}}).
\end{equation*}
\item
Invariance of the $\mc{T}-$nucleon masses and $\mc{T}-$fragment masses
under action of $H$.
For all $\ms{a}\in\Upxi(\pf{D},\mc{F})$ 
\begin{equation*}
(\un\mapsto(\muup_{j}^{\mc{T}})_{\ms{a}}),\,
(\un\mapsto(\uplambda_{j}^{\mc{T}})_{\ms{a}})
\in
Mor_{\ms{Fct}(H,\ms{set})}(\ms{Q}_{\ms{a}}^{\mc{T}},\ms{R}^{H}).
\end{equation*}
\item
Invariance of the $\mc{T}-$Terrell law under contravariant action of 
$\Upxi(\pf{D},\mc{F})$
\begin{equation*}
\uptheta^{\mc{T}}
\in
Mor_{\ms{Fct}(\Upxi(\pf{D},\mc{F})^{op},\ms{set})}(\mc{P}^{\mc{T}},\ms{R}^{\mc{T}}),
\end{equation*}
and under action of $H$.
For all $\ms{a}\in\Upxi(\pf{D},\mc{F})$ 
\begin{equation*}
(\un\mapsto\uptheta_{\ms{a}}^{\mc{T}})
\in
Mor_{\ms{Fct}(H,\ms{set})}(\ms{Q}_{\ms{a}}^{\mc{T}},\ms{R}^{H}).
\end{equation*}
\end{enumerate}
\end{theorem}
\begin{remark}
\label{07161606}
Since Cor. \ref{07141823} the statements of Thm. \ref{07141857} are well-set.
Moreover they are equivalent to what follows,
for all 
$T\in Mor_{\Upxi(\pf{D},\mc{F})}$, 
$\ms{a}\in\Upxi(\pf{D},\mc{F})$, 
$l\in H$ 
and
$j\in\{m,w\}$
we have 
\begin{equation*}
\begin{aligned}
(\muup_{j}^{\mc{T}})_{d(T)}
\circ
\mc{P}_{m}^{\mc{T}}(T)
&=
(\muup_{j}^{\mc{T}})_{c(T)},
\\
(\muup_{j}^{\mc{T}})_{\ms{a}}
\circ
\uptau_{\ms{a}}^{\mc{T}}(l)
&=
(\muup_{j}^{\mc{T}})_{\ms{a}},
\end{aligned}
\end{equation*}
\begin{equation*}
\begin{aligned}
(\uplambda_{j}^{\mc{T}})_{d(T)}
\circ
\mc{P}_{m}^{\mc{T}}(T)
&=
(\uplambda_{j}^{\mc{T}})_{c(T)},
\\
(\uplambda_{j}^{\mc{T}})_{\ms{a}}
\circ
\uptau_{\ms{a}}^{\mc{T}}(l)
&=
(\uplambda_{j}^{\mc{T}})_{\ms{a}},
\end{aligned}
\end{equation*}
and 
\begin{equation*}
\begin{aligned}
\uptheta_{d(T)}^{\mc{T}}
\circ
\mc{P}_{m}^{\mc{T}}(T)
&=
\uptheta_{c(T)}^{\mc{T}},
\\
\uptheta_{\ms{a}}^{\mc{T}}
\circ
\uptau_{\ms{a}}^{\mc{T}}(l)
&=
\uptheta_{\ms{a}}^{\mc{T}}.
\end{aligned}
\end{equation*}
\end{remark}
\begin{proof}[Proof of Thm. \ref{07141857}]
Since Prp. \ref{20051810dbt} 
and Lemma \ref{07161534}\eqref{07161534st1}
we obtain
for all $T\in Mor_{\Upxi(\pf{D},\mc{F})}$,
$l\in H$ 
and $j\in\{m,w\}$
\begin{equation*}
\begin{aligned}
(\muup_{j}^{\mc{T}})_{d(T)}
\circ
\uprho^{\mc{T}}(T)
&=
\{\cdot\}
\circ
\upzeta_{j}^{\mc{T}}\up\ms{N}_{as}^{\mc{T}}(c(T)),
\\
\upzeta_{j}^{\mc{T}}
\circ
\mf{k}^{\mc{T}}(l)
&=
\upzeta_{j}^{\mc{T}},
\end{aligned}
\end{equation*}
\begin{equation*}
\begin{aligned}
(\uplambda_{j}^{\mc{T}})_{d(T)}
\circ
\uprho^{\mc{T}}(T)
&=
\{\cdot\}
\circ
\upkappa_{j}^{\mc{T}}\up\ms{N}_{as}^{\mc{T}}(c(T)),
\\
\upkappa_{j}^{\mc{T}}
\circ
\mf{k}^{\mc{T}}(l)
&=
\upkappa_{j}^{\mc{T}},
\\
\uptheta_{d(T)}^{\mc{T}}
\circ
\uprho^{\mc{T}}(T)
&=
\{\cdot\}
\circ
\upnu^{\mc{T}}\up\ms{N}_{as}^{\mc{T}}(c(T)),
\\
\upnu^{\mc{T}}\circ\mf{k}^{\mc{T}}(l)
&=
\upnu^{\mc{T}}.
\end{aligned}
\end{equation*}
Then for any $Y\in\mc{P}_{o}^{\mc{T}}(c(T))$
we have
\begin{equation*}
\begin{aligned}
(\uptheta_{d(T)}^{\mc{T}}\circ\mc{P}_{m}^{\mc{T}}(T))(Y)
&=
\bigcup\{
(\uptheta_{d(T)}^{\mc{T}}\circ\uprho^{\mc{T}}(T))(y)
\,\vert\,y\in Y\}
\\
&=
\bigcup\{\{\upnu^{\mc{T}}(y)\}\,\vert\,y\in Y\}
\\
&=
\{\upnu^{\mc{T}}(y)\,\vert\,y\in Y\}
=
\uptheta_{c(T)}^{\mc{T}}(Y);
\end{aligned}
\end{equation*}
moreover
for all $\ms{a}\in\Upxi(\pf{D},\mc{F})$ and $Z\in\mc{P}^{\mc{T}}(\ms{a})$
\begin{equation*}
\begin{aligned}
(\uptheta_{\ms{a}}^{\mc{T}}\circ\uptau_{\ms{a}}^{\mc{T}}(l))(Z)
&=
\{
(\upnu^{\mc{T}}\circ\mf{k}^{\mc{T}}(l))(z)
\,\vert\,z\in Z\}
\\
&=
\{\upnu^{\mc{T}}(z)\,\vert\,z\in Z\}
=
\uptheta_{\ms{a}}^{\mc{T}}(Z),
\end{aligned}
\end{equation*}
similarly for $\muup$ and $\uplambda$.
Hence our claim follows since Rmk. \ref{07161606}.
\end{proof}
\subsection{Universality of the global Terrell law}
\label{25061121}
The resolution of the equivariant form of the universality 
claim has been established in 
Cor. \ref{19061936} 
whose compact form has been stated in Thm. \ref{10071930}.
Here we apply the results of section \ref{07170749}
to the canonical nucleon-fragment doublet $\mc{T}_{\bullet}$ 
in order to establish in Thm. \ref{10061606ante}
the invariant form of the universality claim, 
the main result of this section and the final one of the entire project.
It states the invariance of the global nucleon and fragment masses
and the invariance of the global Terrell law
under contravariant action of $\ms{C}_{u}^{0}(H)$ and under covariant 
action of $H$, over the field of fission processes.
As a result we obtain the invariance of the restricted global Terrell law stated in Cor. \ref{10061606},
whose physical interpretation 
as invariance of the prompt-neutron yield is established in Cor. \ref{14061529};
result that can be obtained in a more direct way 
as consequence of the corresponding ones in section \ref{07011744}. 
Let $(\mf{s},\mf{u})$ be a semantics for $\mc{T}_{\bullet}$.
\begin{definition}
[Global nucleon masses, global fragment masses and global Terrell law]
Let the global nucleon phase and the global fragment state
be the $\mc{T}_{\bullet}-$nucleon phase 
and $\mc{T}_{\bullet}-$fragment state respectively.
Moreover
let 
the (restricted) global $\#_{j}$ nucleon mass
be 
the (restricted) $\mc{T}_{\bullet}-\#_{j}$ nucleon mass,
let 
the (restricted) global $\#_{j}$ fragment mass
be 
the (restricted) $\mc{T}_{\bullet}-\#_{j}$ fragment mass
for any $j\in\{m,w\}$.
Let the (restricted) global Terrell law
be
the (restricted) $\mc{T}_{\bullet}-$Terrell law.  
\end{definition}
Up to the end of section \ref{25061121}
let
$\mf{n}=\lr{\mf{A}}{\mf{T},\alpha,\{\ms{f}_{j},N_{j}\}_{j\in\{m,w\}}}\in\ms{N}_{as}^{\mc{T}_{\bullet}}$.
Since section \ref{07011744} for all $l\in H$ we have
\begin{equation*}
\begin{aligned}
(\nuup^{\mc{T}_{\bullet}}\circ\mf{k}^{\mc{T}_{\bullet}}(l))
(\mf{n})
&=
0.08\left(   
\mc{V}^{\mf{A}}(\mf{T}^{l},\alpha)(N_{m}^{l})
-
\ov{\ms{m}}^{\mf{A}}(\mf{T}^{l},\alpha)(\ms{f}_{m}^{l})
\right)+
\\
&
0.1
\left(\mc{V}^{\mf{A}}(\mf{T}^{l},\alpha)(N_{w}^{l})
-
\ov{\ms{m}}^{\mf{A}}
(\mf{T}^{l},\alpha)(\ms{f}_{w}^{l})
\right).
\end{aligned}
\end{equation*}
Moreover 
if $(\mf{n},\mf{B},T,\mf{x})$ satisfies the hypothesis $\ms{S}$ 
w.r.t. $\mc{T}_{\bullet}$, 
then
$\mf{B}\in\ms{C}_{u}^{0}(H)$,
$T\in Mor_{\ms{C}_{u}^{0}(H)}(\mf{B},\mf{A})$
and
$\mf{x}=\{\ms{f}_{j}',N_{j}'\}_{j\in\{m,w\}}$
such that 
$\ms{f}_{j}'\in\ms{A}_{\mf{B}}$, 
$N_{j}'\in\mc{B}_{+}$, 
$\ms{f}_{j}=\mf{g}^{H}(T)(\ms{f}_{j}')$,
and
$N_{j}=T(N_{j}')$,
for any $j\in\{m,w\}$, 
where $\mc{B}$ is the $C^{\ast}-$algebra underlying $\mf{B}$.
Let $(\mf{n},\mf{B},T,\mf{x})$ satisfy the hypothesis $\ms{S}$ 
w.r.t. $\mc{T}_{\bullet}$,
where $\mf{x}=\{\ms{f}_{j}',N_{j}'\}_{j\in\{m,w\}}$
then
\begin{equation*}
\mf{n}^{(T,\mf{x})}
=
\lr{\mf{B}}{\mf{T}^{T},\alpha,\mf{x}},
\end{equation*}
where $\mf{T}^{T}=\mf{d}^{H}(T)(\mf{T})$,
moreover
$\mf{n}^{(T,\mf{x})}\in\ms{N}_{as}^{\mc{T}_{\bullet}}$
and
\begin{equation*}
\begin{aligned}
\nuup^{\mc{T}_{\bullet}}
(\mf{n}^{(T,\mf{x})})
&=
0.08\left(
\mc{V}^{\mf{B}}\left(\mf{T}^{T},\alpha\right)
(N_{m}')
-
\ov{\ms{m}}^{\mf{B}}
\left(\mf{T}^{T},\alpha\right)(\ms{f}_{m}')
\right)+
\\
&
0.1
\left(
\mc{V}^{\mf{B}}
\left(\mf{T}^{T},\alpha\right)(N_{w}')
-
\ov{\ms{m}}^{\mf{B}}
\left(\mf{T}^{T},\alpha\right)(\ms{f}_{w}')
\right).
\end{aligned}
\end{equation*}
The next is the third main result of this work,
jointly Cor. \ref{19061936}
resolves the invariant form of the universality claim,
in particular it presents the stability of the global nucleon and fragment masses
and the invariance of the global Terrell law.
As an application in the next section we will obtain the invariance of the Terrell law \eqref{26061108}
and the stability of the nucleon masses at the values $82$ and $126$.
\begin{theorem}
[Universality of the global nucleon and fragment masses and the global Terrell law]
\label{10061606ante}
For any $j\in\{m,w\}$ we have 
\begin{enumerate}
\item
Invariance of the global nucleon masses and global fragment masses 
under contravariant action of $\ms{C}_{u}^{0}(H)$
\begin{equation*}
\muup_{j}^{\mc{T}_{\bullet}},\,
\uplambda_{j}^{\mc{T}_{\bullet}}
\in
Mor_{\ms{Fct}(\ms{C}_{u}^{0}(H)^{op},\ms{set})}(\mc{P}^{\mc{T}_{\bullet}},\ms{R}^{\mc{T}_{\bullet}}).
\end{equation*}
\item
Invariance of the global nucleon masses and global fragment masses
under action of $H$.
For all $\mf{D}\in\ms{C}_{u}^{0}(H)$ 
\begin{equation*}
(\un\mapsto(\muup_{j}^{\mc{T}_{\bullet}})_{\mf{D}}),\,
(\un\mapsto(\uplambda_{j}^{\mc{T}_{\bullet}})_{\mf{D}})
\in
Mor_{\ms{Fct}(H,\ms{set})}(\ms{Q}_{\mf{D}}^{\mc{T}_{\bullet}},\ms{R}^{H})
\end{equation*}
\item
Invariance of the global Terrell law under contravariant action of 
$\ms{C}_{u}^{0}(H)$
\begin{equation*}
\uptheta^{\mc{T}_{\bullet}}
\in
Mor_{\ms{Fct}(\ms{C}_{u}^{0}(H)^{op},\ms{set})}(\mc{P}^{\mc{T}_{\bullet}},\ms{R}^{\mc{T}_{\bullet}}),
\end{equation*}
and under action of $H$.
For all $\mf{D}\in\ms{C}_{u}^{0}(H)$ 
\begin{equation*}
(\un\mapsto\uptheta_{\mf{D}}^{\mc{T}_{\bullet}})
\in
Mor_{\ms{Fct}(H,\ms{set})}(\ms{Q}_{\mf{D}}^{\mc{T}_{\bullet}},\ms{R}^{H}).
\end{equation*}
\end{enumerate}
\end{theorem}
\begin{proof}
Since 
Thm. \ref{06242121}
and Thm. \ref{07141857} applied for $\mc{T}=\mc{T}_{\bullet}$.
\end{proof}
\begin{proof}[Alternative proof of Thm. \ref{10061606ante}]
If hypothesis $\ms{E}$ holds true then it 
follows since Thm. \ref{10071930} and Rmk. \ref{11071826}.
\end{proof}
As a result we can state the following invariance
whose physical interpretation is explained in Cor. \ref{14061529}.
\begin{corollary}
[Invariance of the restricted global nucleon masses, global fragment masses and global Terrell law]
\label{10061606}
Let $(\mf{n},\mf{B},T,\mf{x})$ satisfy the hypothesis $\ms{S}$ w.r.t. $\mc{T}_{\bullet}$, 
$l\in H$ and $j\in\{m,w\}$ then 
\begin{enumerate}
\item
Invariance of the restricted global nucleon masses and 
restricted global fragment masses 
under essential contravariant action of $\ms{C}_{u}^{0}(H)$
\begin{equation*}
\begin{aligned}
\upzeta_{j}^{\mc{T}_{\bullet}}
(\mf{n}^{(T,\mf{x})})
&=
\upzeta_{j}^{\mc{T}_{\bullet}}(\mf{n}),
\\
\upkappa_{j}^{\mc{T}_{\bullet}}
(\mf{n}^{(T,\mf{x})})
&=
\upkappa_{j}^{\mc{T}_{\bullet}}(\mf{n}).
\end{aligned}
\end{equation*}
\item
Invariance of the restricted global nucleon masses
and restricted global fragment masses under action of $H$
\begin{equation*}
\begin{aligned}
\upzeta_{j}^{\mc{T}_{\bullet}}
\circ\mf{k}^{\mc{T}_{\bullet}}(l)
&=
\upzeta_{j}^{\mc{T}_{\bullet}},
\\
\upkappa_{j}^{\mc{T}_{\bullet}}
\circ\mf{k}^{\mc{T}_{\bullet}}(l)
&=
\upkappa_{j}^{\mc{T}_{\bullet}}.
\end{aligned}
\end{equation*}
\item
Invariance of the restricted global Terrell law under essential contravariant action of 
$\ms{C}_{u}^{0}(H)$ and action of $H$
\begin{equation*}
\begin{aligned}
\nuup^{\mc{T}_{\bullet}}
(\mf{n}^{(T,\mf{x})})
&=
\nuup^{\mc{T}_{\bullet}}(\mf{n}),
\\
\nuup^{\mc{T}_{\bullet}}\circ\mf{k}^{\mc{T}_{\bullet}}(l)
&=
\nuup^{\mc{T}_{\bullet}}.
\end{aligned}
\end{equation*}
\end{enumerate}
\end{corollary}
\begin{proof}
Since Thm. \ref{10061606ante},
alternatively by Thm. \ref{06242121}
and Prp. \ref{11061659}.
\end{proof}
Next we clarify the physical meaning of Cor. \ref{10061606} concerning the gobal Terrell law,
the part concerning the restricted nucleon and fragment masses follows
similarly by applying Prp. \ref{11061851}.
\begin{corollary}
\label{14061529}
Let $(\mf{n},\mf{B},T,\mf{x})$ satisfy the hypothesis $\ms{S}$ w.r.t. $\mc{T}_{\bullet}$
and $l\in H$,
then
\begin{enumerate}
\item
The mean value of the prompt neutron-yield in
the state originated via the phase of the nucleon system 
generated by the fissioning system $\mf{u}(\mf{B})$, 
occurring by performing 
on 
$\mf{s}((\ps{\upvarphi}^{\mf{B}})_{\alpha}^{\mf{T}^{T}})$
the operation obtained by transforming through the action of $T$
the operation realizing the asymmetric fission process of 
the fissioning system $\mf{u}(\mf{A})$, 
equals
the mean value of the prompt neutron-yield in
the state originated via the phase of the nucleon system 
generated by the fissioning system $\mf{u}(\mf{A})$, 
occurring by performing 
on $\mf{s}((\ps{\upvarphi}^{\mf{A}})_{\alpha}^{\mf{T}})$
the operation realizing the asymmetric fission process of 
the fissioning system $\mf{u}(\mf{A})$.
\label{14061529b}
\item
\emph
The mean value of the prompt neutron-yield in 
the state originated via 
the phase of the nucleon system 
generated by the fissioning system $\mf{u}(\mf{A})$, 
occurring by performing 
on 
$\mf{s}((\ps{\upvarphi}^{\mf{A}})_{\alpha}^{\mf{T}^{l}})$
the operation obtained by transforming through the action of $l$
the operation realizing the asymmetric fission process of 
the fissioning system $\mf{u}(\mf{A})$, 
equals
the mean value of the prompt neutron-yield in
the state 
originated via the phase of the nucleon system 
generated by the fissioning system $\mf{u}(\mf{A})$, 
occurring by performing 
on
$\mf{s}((\ps{\upvarphi}^{\mf{A}})_{\alpha}^{\mf{T}})$
the operation realizing the asymmetric fission process of 
the fissioning system $\mf{u}(\mf{A})$. 
\label{14061529a}
\end{enumerate}
Here 
$\mf{s}((\ps{\upvarphi}^{\mf{A}})_{\alpha}^{\mf{T}})\equiv$
the state of thermal equilibrium 
$(\ps{\upvarphi}^{\mf{A}})_{\alpha}^{\mf{T}}$
at the inverse temperature $\alpha$
of 
the fragment system whose observable algebra is $\mc{A}$
and whose dynamics is $(\ep^{\mf{A}})_{\alpha}^{\mf{T}}(-\alpha^{-1}(\cdot))$,
where $\mc{A}$ is the underlying $C^{\ast}-$algebra of $\mf{A}$.
\end{corollary}
\begin{proof}
Since 
Cor. \ref{10061606} 
and Prp. \ref{02202114dbt}; 
alternatively since 
Prp. \ref{11061851} 
applied for $\mc{T}=\mc{T}_{\bullet}$.
\end{proof}
\subsection{Nucleon phase hypothesis revised}
\label{01231337}
In this section we provide the revised version in our setting of the nucleon phase hypothesis advanced 
by G. Mouze and C. Ythier see \cite{mhy1,mhy2,ric}. 
It ammounts to specify firstly the group $H$ as the semidirect product of $\R^{4}$ with a group containing
the direct product of the universal covering group of the proper Lorentz group on $\R^{4}$ 
with the gauge group of the standard model.
Secondly to hypothize the existence 
of a suitable element $\mf{n}\in\ms{N}_{as}^{\mc{T}_{\bullet}}$ displaying the values $82$ and $126$.
We apply Cor. \ref{10061606} to the nucleon phase hypothesis
to obtain in Cor. \ref{08161246} the invariance of the Terrell law and the stability of the nucleon masses
at the above values.
\par
To any asymmetric fission process
$\mf{n}=\lr{\mf{A}}{\mf{T},\alpha,\{\ms{f}_{j},N_{j}\}_{j\in\{m,w\}}}$ relative to $\mc{T}_{\bullet}$
two asymmetric fragments correspond with mass numbers
$\mc{V}^{\mf{A}}(\mf{T},\alpha)(N_{m})$
and
$\mc{V}^{\mf{A}}(\mf{T},\alpha)(N_{w})$.
However the Terrell law \eqref{26061108}
describes the behaviour of a family
of couples of fragments each of them arising 
by a fissioning system belonging to the set
$\mathbb{B}\coloneqq\{U^{233}+n_{th}, U^{235}+n_{th}, Pu^{239}+n_{th}, Cf^{252}\}$,
moreover the values $82$ and $126$ appears in it.
The hypothesis below formulated fixes the group $H$,
assumes the existence of a fission process $\mf{n}$ relative to $\mc{T}_{\bullet}$
such that 
$\ov{\ms{m}}^{\mf{A}}(\mf{T},\alpha)(\ms{f}_{m})=82$
and
$\ov{\ms{m}}^{\mf{A}}(\mf{T},\alpha)(\ms{f}_{w})=126$,
and for any $\varsigma\in\mathbb{B}$ supposes 
the existence of a fission process relative to $\mc{T}_{\bullet}$
nucleon equivalent with $\mf{n}$.
\par
$SU(2,\C)$ 
is the 
universal covering group
of the proper Lorentz group $\mf{L}_{+}^{\uparrow}$ on $\R^{4}$,
see for example \cite[p.121]{blot},
while $F^{0}=U(1)\times SU(2)\times SU(3)$
is the gauge group of the standard model.
Moreover let 
the standard action of $SU(2,\C)$ on $\R^{4}$ 
denote
the action defined in \cite[eqs. $(3.39)-(3.33a)$]{blot},
and let 
$\ms{g}$ denote the Lorentz metric tensor on $\R^{4}$.
Finally 
let $\imath_{1}$ and $\imath_{2}$ 
be the canonical injections of $SL(2,\C)$ and 
$F^{0}$
in $SL(2,\C)\times F^{0}$ respectively,
and
$\Pr_{\mu}(\lambda)=\lambda_{\mu}$ for any $\lambda\in\C^{4}$
and $\mu\in\{1,\dots, 4\}$.
\begin{definition}
\label{nuclsys}
We call $\lr{R,\upeta}{K}$ a Poincar\'{e} triplet if
\begin{enumerate}
\item
$R$ is a locally compact group;
\label{binfis1a}
\item
$SL(2,\C)\times F^{0}$ is a topological subgroup of $R$;
\label{binfis1b}
\item
$\upeta:R\to Aut_{\ms{Gr}}(\R^{4})$ is a group homomorphism;
\label{binfis2a}
\item
the map $(g,f)\mapsto\upeta_{f}(g)$
on $\R^{4}\times R$ at values in $\R^{4}$, is continuous;
\label{binfis2}
\item
$\upeta\circ\imath_{1}$
and
$\upeta\circ\imath_{2}$
are 
the standard action of $SL(2,\C)$ 
and the trivial action of $F^{0}$
on $\R^{4}$
respectively;
\label{binfis1}
\item
$K=\R^{4}\rtimes_{\upeta}R$.
\label{binfis3}
\end{enumerate}
\end{definition}
\begin{definition}
$\mathbb{B}\coloneqq\{U^{233}+n_{th}, U^{235}+n_{th}, Pu^{239}+n_{th},Cf^{252}\}$.
\end{definition}
\begin{hypothesis}
[Revised nucleon phase hypothesis]
\label{25061633}
Let $\lr{F,\uprho}{K}$ be a Poincar\'{e} triplet,
we say that $\mf{n}$ is associated with $\lr{F,\uprho}{K}$ 
if 
\begin{enumerate}
\item
$\mf{n}\in\ms{N}_{as}^{\mc{T}_{\bullet}^{K}}$; 
\label{25061633st1}
\item
$(\beta_{c}^{\mf{A}})^{\mf{T}}\leq 10^{-13}$; 
\item
$\upzeta_{m}^{\mc{T}_{\bullet}^{K}}(\mf{n})=82$
and
$\upzeta_{w}^{\mc{T}_{\bullet}^{K}}(\mf{n})=126$;
\item
for any $\varsigma\in\mathbb{B}$ there exists a fission process relative to $\mc{T}_{\bullet}^{K}$
nucleon equivalent with $\mf{n}$.
\end{enumerate}
Here $\mf{A}$ and $\mf{T}$ are the dynamical system and the operation underlying $\mf{n}$
respectively.
We say that the revised nucleon phase hypothesis holds true if 
there exists a Poincar\'{e} triplet $\lr{F,\uprho}{K}$ 
such that the set of the fission processes associated with it is not empty.
In such a case we say that $\lr{F,\uprho}{K}$ satisfies the revised nucleon phase hypothesis.
\end{hypothesis}
If $\lr{F,\uprho}{K}$ satisfies the revised nucleon phase hypothesis
we call 
$\#_{j}$ nucleon mass, $\#_{j}$ fragment mass and Terrell law 
(associated with the triplet)  
the restriction
of 
$\upzeta_{j}^{\mc{T}_{\bullet}^{K}}$,
$\upkappa_{j}^{\mc{T}_{\bullet}^{K}}$
and
$\nuup^{\mc{T}_{\bullet}^{K}}$
to the set of the fission processes associated with $\lr{F,\uprho}{K}$
respectively.
Thus we have
\begin{corollary}
[Invariance of the nucleon masses and Terrell law]
\label{08161246}
Let $\lr{F,\uprho}{H}$ satisfy the revised nucleon phase hypothesis,
$\mf{n}$ be associated with $\lr{F,\uprho}{H}$ and let 
$\mc{T}_{\bullet}$ denote $\mc{T}_{\bullet}^{H}$.
Thus for all $(\mf{B},T,\mf{x})$ such that $(\mf{n},\mf{B},T,\mf{x})$ 
satisfies the hypothesis $\ms{S}$ w.r.t. $\mc{T}_{\bullet}$,
$j\in\{m,w\}$ and all $l\in H$ we obtain 
\begin{equation}
\label{07271358}
\begin{aligned}
\upzeta_{j}^{\mc{T}_{\bullet}}(\mf{k}^{\mc{T}_{\bullet}}(l)\mf{n})
&=
\upzeta_{j}^{\mc{T}_{\bullet}}(\mf{n})
=
\upzeta_{j}^{\mc{T}_{\bullet}}(\mf{n}^{(T,\mf{x})})
\\
&=
Z_{j},
\end{aligned}
\end{equation}
where $Z_{m}=82$ and $Z_{w}=126$, 
and 
\begin{equation}
\label{07271359}
\begin{aligned}
\nuup^{\mc{T}_{\bullet}}(\mf{k}^{\mc{T}_{\bullet}}(l)\mf{n})
&=
\nuup^{\mc{T}_{\bullet}}(\mf{n})
=
\nuup^{\mc{T}_{\bullet}}(\mf{n}^{(T,\mf{x})})
\\
&=
0.08\left(   
\mc{V}^{\mf{A}}(\mf{T},\alpha)(N_{m})
-
82
\right)+
0.1
\left(\mc{V}^{\mf{A}}(\mf{T},\alpha)(N_{w})
- 126
\right).
\end{aligned}
\end{equation}
\end{corollary}
\begin{proof}
Since Cor. \ref{10061606}. 
\end{proof}
Clearly similar result follows for any fission process nucleon equivalent with $\mf{n}$.
\begin{remark}
\eqref{07271358} in particular 
establishes the stability of the light and heavy nucleon masses 
at the values $82$ and $126$ under action of the universal covering group of the 
Poincar\'{e} group
over the field of fission processes, 
implemented as an action over the operations realizing these processes.
As a result these values remain unchanged under relativistic transformations of the reference frame.
Moreover the values $82$ and $126$ result to be stable under 
essential contravariant action over the field of fission processes
of the fissioning system transformations belonging to
$\ms{C}_{u}^{0}(H)$. 
\eqref{07271359} establishes the invariance of the prompt-neutron yield under the 
above described actions.
\end{remark}
\begin{remark}
Although the global nucleon phase is integer, 
we emphasize that our goal is \emph{not} to provide an explanation 
of why the light and heavy nucleon masses hold exactly the values $82$ and $126$
respectively.
Rather we focus our attention, on how and in this way also why,
the universality of the global nucleon phase and global fragment state
-
namely their reciprocal relashionship and their 
equivariance under contravariant action of $\ms{C}_{u}^{0}(H)$ and under covariant action of $H$
precisely established in the resolution of the
equivariant form of the universality claim in
Cor. \ref{19061936} 
-
renders the values $82$ and $126$ in the Terrell law, and the Terrell law itself
\emph{invariant} under essential contravariant action 
of suitable perturbations of fissioning systems
and under action of relativistic transformations of reference frames,
over the field of fission processes.
What is fundamental in our context is the existence of 
what we called the nucleon system generated by the fissioning system, 
whose nucleon phases are, as they should be, relativistically \emph{covariant}, 
contravariant under action of suitable perturbations,
and \emph{originating} states of fragment systems,
\emph{regardless} of the exact values suitable observables assume when measured in them.
\end{remark}
\begin{remark}
[Experimental testability of the universality of the global nucleon phase]
The universality of the global nucleon masses and global Terrell law
-
in the special form of the stability of the values $82$ and $126$ occurring in the Terrell law
and the invariance of the Terrell law itself mainly under the relativistic transformations
described in Cor. \ref{08161246}
-
is experimentally testable and can provide
an empirical evidence of the existence and universality of the 
global nucleon phase $\ov{\ms{m}}$.
\end{remark}
\begin{remark}
Let $r\in\{s,as\}$
then provisionally
we can set
$\mf{T}_{r}=\lr{\mc{T},\ps{\upmu},\pf{H}}{\upzeta_{r},f,\Upgamma}
\in\pf{T}_{\mf{A}}$ 
with
$\mc{T}=\lr{h,\upxi,\beta_{c},I}{\ps{\upomega}}$,
satisfying
\begin{equation}
\label{fragb}
\begin{cases}
Dom(\upzeta_{r})=\R^{4},
\\
f=
\sum_{\mu,\nu=1}^{4}
\Pr_{\mu}\,
\ms{g}_{\mu,\nu}\,
\Pr_{\nu}.
\end{cases}
\end{equation}
$f$ satisfies 
\eqref{01291801} 
for the position
$A=X=\{1,\dots, 4\}$ for any $\alpha\in\ms{P}_{\mf{A}}^{\mf{T}}$,
since $f$ maps $\R^{4}$ into $\R$
and the support of the resolution of the identity
of any selfadjoint operator in a Hilbert space
is a subset of $\R$. 
Hence the request \eqref{fragb} is well-set.
Next
$\upzeta_{r}:
\R^{4}\to
\ms{S}_{\ms{F}_{\ps{\upomega}_{\beta_{c}}}}^{G}$
since 
Def. \ref{08261134}.
Set
$\ms{T}_{r}^{\alpha}=\{T_{\nu,r}^{\alpha}\mid\nu\in X\}$,
where
$i T_{\mu,r}^{\alpha}$ is the infinitesimal generator of the
one-parameter unitary group  
$\ms{U}_{\pf{H}_{\alpha}}\circ\upzeta_{r}\circ\mf{i}_{\upmu}$
on $\pf{H}_{\alpha}$
and 
$\mf{i}_{\mu}:\R\to\R^{4}$ such that 
$\Pr_{\nu}\circ\mf{i}_{\mu}=\delta_{\nu,\mu} Id_{\R}$, for all 
$\mu,\nu\in X$.
Since
$\Pr_{\mu}(\ms{E}_{\mc{E}_{\ms{T}_{r}^{\alpha}}})
=T_{\mu,r}^{\alpha}$,
\eqref{fragb}
and an application of the functional calculus,
see for example \cite[Thm. $5.6.19$]{kad}
we deduce that
\begin{equation}
\label{12091819}
\ms{D}_{\pf{H},\alpha}^{\upzeta_{r},f}(\mf{A})
=
\ov{
\sum_{\mu,\nu=1}^{4}
T_{\mu,r}^{\alpha}
\,\ms{g}_{\mu,\nu}\,
T_{\nu,r}^{\alpha}},
\end{equation}
where $\ov{S}$ is the closure of any closable operator $S$ in a Hilbert space.
Thus in this case the only difference between the phases
$\ov{\ms{m}}^{\mf{A}}(\mf{T}_{s},\alpha)$
and
$\ov{\ms{m}}^{\mf{A}}(\mf{T}_{as},\alpha)$
results by the fact that we select, 
via $\upzeta_{s}$ and $\upzeta_{as}$,
two different sets of infinitesimal generators of
commutative subgroups of 
$\ms{S}_{\ms{F}_{\ps{\upomega}_{\beta_{c}}}}^{G}$.
Finally it is worthwhile noting that the selfadjoint operator $\ms{D}_{\pf{H},\alpha}^{\upzeta_{r},f}(\mf{A})$ 
in general cannot be considered
the mass operator even when $G=\R^{4}$ and
$\upzeta_{r}(x)=(x,\un_{F})$ for any $x\in\R^{4}$, indeed it needs not to be positive.
\end{remark}
\begin{remark}
Since 
\eqref{12220022} 
it is clear that 
$\ov{\ms{m}}^{\mc{G}}(\mf{T},\alpha)$ 
is represented by the Chern-Connes character of an entire cocycle depending by the element $(\mf{T},\alpha)$,
while it is preferable to have a unique character associated with the map $\ov{\ms{m}}^{\mc{G}}$. 
We shall analyze in a future work the possibility of 
constructing a functor $\mc{L}$ on $\mf{G}(G,F,\uprho)$, 
encoding for any object $\mc{G}$ the data
$\{(\mf{T},\alpha)\mid\mf{T}\in\pf{T}_{\mc{G}},
\alpha\in\ms{P}_{\mc{G}}^{\mf{T}}\}$
into three $C^{\ast}-$algebras 
and relating the map $\ov{\ms{m}}^{\mc{G}}$ 
to a bivariant Chern-Connes
character as defined in \cite{nis}, whose notation
we use in what follows.
We require $\mc{L}$ assigning to any object 
$\mc{G}$ of $\mf{G}(G,F,\uprho)$ at least the following data:
\begin{itemize}
\item
$C^{\ast}-$algebras $\mc{D}$ and $\mc{R}$ 
and a smooth 
subalgebra $\mc{R}^{\infty}$ of $\mc{R}$, 
\item
a group morphism $\mf{v}:\ms{A}_{\mc{G}}\to\ms{K}_{0}(\mc{D})$, 
\item
$i\in\{0,1\}$, $p,n\in\N$ such that $(2-i)n\geq p-1$,
\item
$\ms{z}\in\prod_{\mf{I}\in\pf{T}_{\mc{G}}}\prod_{\beta\in\ms{P}_{\mc{G}}^{\mf{I}}}
\mc{E}_{p}^{i}(\mc{D};\mc{R}^{\infty},\mc{R}^{\infty\,+})$,
\item
$\phiup\in Ext_{\Lambda}^{2n+i}(\mc{R}^{\infty\,\natural},
\C^{\natural})$,
\end{itemize}
such that for all $\mf{T}\in\pf{T}_{\mc{G}}$, 
$\alpha\in\ms{P}_{\mc{G}}^{\mf{T}}$ and 
$\ms{f}\in\ms{A}_{\mc{G}}$
we obtain
\begin{equation*}
\ov{\ms{m}}^{\mc{G}}(\mf{T},\alpha)(\ms{f})
=
\phiup_{\star}(\mf{v}(\ms{f})\otimes_{\mc{D}}\ms{z}(\mf{T},\alpha)),
\end{equation*}
well-set since the left side is integer by construction, 
the right side is integer since the index theorem \cite[Thm. $6.4.$]{nis}.
\end{remark}
\section*{Acknowledgments}
I wish to thank Mauro Cappelli who signaled me the paper \cite{ric} of Renato Angelo Ricci on the nucleon phase 
hypothesis of binary fission advanced by G. Mouze and C. Ythier. 


\bibliography{references}

\providecommand{\bysame}{\leavevmode\hbox to3em{\hrulefill}\thinspace}
\providecommand{\MR}{\relax\ifhmode\unskip\space\fi MR }
\providecommand{\MRhref}[2]{%
  \href{http://www.ams.org/mathscinet-getitem?mr=#1}{#2}
}
\providecommand{\href}[2]{#2}
\begin{thebibliography}{10}

\bibitem{tra}
J.~P. Antoine, A.~Inoue, and C.~Trapani, \emph{Partial $\ast-$algebras and
  their operator realizations}, Kluwer Academic Publishers, 2002.

\bibitem{sga4}
M.~Artin, A.~Grothendieck, and J.~L. Verdier, \emph{Theorie des topos et
  cohomologie etale des schemas. tome i theorie des topos}, Seminaire Geometrie
  Algebrique du Bois-Marie, 1963-1964.

\bibitem{ber}
Y.~M. Berezanskii, \emph{Selfadjoint operators in spaces of functions of
  infinitely many variables}, American Mathematical Society, 1986.

\bibitem{bla}
B.~Blackadar, \emph{$k-$theory for operator algebras}, Cambridge Press, 1998.

\bibitem{bla2}
\bysame, \emph{Operator algebras}, Springer, 2006.

\bibitem{blot}
N.~N. Bogolubov, A.~A. Logunov, A.~I. Oksak, and I.~T. Todorov, \emph{General
  principles of quantum field theory}, Kluwer Academic Publishers, 1990.

\bibitem{bor}
F.~Borceux, \emph{Handbook of categorical algebra 1}, Cambridge University
  Press, 1994.

\bibitem{BourEns}
N.~Bourbaki, \emph{Theorie des ensambles}, Diffusion C.C.L.S, 1970.

\bibitem{BourGT}
\bysame, \emph{Topologie generale}, CCLS, 1974.

\bibitem{BourTVS}
\bysame, \emph{Topological vector spaces}, Springer, 1989.

\bibitem{IntBourb}
\bysame, \emph{Integration}, Springer, 2003.

\bibitem{bra}
O.~Bratteli and D.W. Robinson, \emph{Operator algebras and quantum statistical
  mechanics i, ii}, Springer-Verlag, 1987, 1996.

\bibitem{connes}
A.~Connes, \emph{Noncommutative geometry}, Academic Press, 1994.

\bibitem{ds}
N.~Dunford and J.T. Schwartz, \emph{Linear operators 1,2,3}, Wiley
  Interscience, 1967, 1964, 1971.

\bibitem{ght}
E.~Guentner, N.~Higson, and J.~Trout, \emph{Equivariant $e-$theory for
  $c^{\ast}-$algebras}, Memoirs of the American Mathematical Society, vol. 703,
  American Mathematical Society, 2000.

\bibitem{jlo}
A.~Jaffe, A.~Lesniewski, and K.~Osterwalder, \emph{Quantum k-theory: the chern
  character}, Comm. Math. Phys. \textbf{118} (1988).

\bibitem{kad}
R.V. Kadison and J.R. Ringrose, \emph{Fundamentals of the theory of operator
  algebras. volume 1,2}, American Mathematical Society, 1997.

\bibitem{ks}
M.~Kashiwara and P.~Schapira, \emph{Categories and sheaves}, Grundlehren der
  mathematischen Wissenschaften, vol. 332, Springer, 2006.

\bibitem{mcl}
S.~Mac Lane, \emph{Categories for the working mathematician}, Springer, 1998.

\bibitem{lei}
T.~Leinster, \emph{Higher operads, higher categories}, ArXiv.

\bibitem{mhy2}
G.~Mouze, S.~Hachem, and C.~Ythier, \emph{The nucleon phase hypothesis of
  binary fission}, ArXiv.

\bibitem{mhy1}
\bysame, Intern. Journal of Modern Physics E \textbf{17} (2008), 2240.

\bibitem{nai}
M.~A. Naimark, \emph{Normed algebras}, Wolters-Noordhoff Publishing Groningen,
  1972.

\bibitem{nis}
V.~Nistor, \emph{A bivariant chern-connes character}, Ann. Math. \textbf{138}
  (1993).

\bibitem{palmer}
T.W. Palmer, \emph{Banach algebras and the general theory of $*$algebras 1,2},
  Cambridge Universiry Press, 1994,2001.

\bibitem{ped}
G.~K. Pedersen, \emph{$c^{\ast}-$algebras and their automorphism groups},
  Academic Press, 1979.

\bibitem{rw}
I.~Raeburn and D.~P. Williams, \emph{Morita equivalence and continuous-trace
  $c^{\ast}-$algebras}, Mathematical Surveys and Monographs, vol.~60, Americal
  Mathematical Society, 1998.

\bibitem{ric}
R.~A. Ricci, \emph{A daring interpretation of binary fission}, Europhysics News
  \textbf{40} (2009), 13--16.

\bibitem{Rudin2}
W.~Rudin, \emph{Real and complex analysis}, McGraw-Hill, 1987.

\bibitem{SilInt}
Benedetto Silvestri, \emph{Integral equalities for functions of unbounded
  spectral operators in banach spaces}, Dissertationes Math. \textbf{464}
  (2009), 60.

\bibitem{tak2}
M.~Takesaki, \emph{Theory of operator algebras ii}, Springer-Verlag, 2003.

\bibitem{weg}
N.~E. Wegg-Olsen, \emph{$k-$theory and $c^{\ast}-$algebras}.

\bibitem{will}
D.~P. Williams, \emph{Crossed products of $c^{\ast}-$algebras}, Mathematical
  Surveys and Monographs, vol. 134, Americal Mathematical Society, 2007.

\end{thebibliography}
\bibliographystyle{amsplain}

\end{document}